\documentclass[reqno,11pt,a4paper]{amsart}
\usepackage{amsmath}
\usepackage{amssymb}
\usepackage{graphicx}
\usepackage{mathrsfs}
\usepackage{amssymb}
\usepackage{amsfonts,bm,bbm}
\usepackage{amsthm}
\usepackage[colorlinks,linkcolor=blue,anchorcolor=blue,citecolor=blue,urlcolor=purple]{hyperref}
\usepackage{orcidlink}
\usepackage{geometry}
\usepackage{fullpage}
\usepackage{color}
\usepackage{enumitem}
\usepackage{underscore}
\usepackage{cite}
\usepackage{doi}

\usepackage[open,openlevel=2]{bookmark}
\newtheorem{Thm}{Theorem}[section]

\newtheorem{Lem}[Thm]{Lemma}

\newtheorem{Rem}[Thm]{Remark}
\newtheorem{Prop}[Thm]{Proposition}

\numberwithin{equation}{section}

\newcommand{\1}{\mathbf{1}}
\newcommand{\R}{\mathbb{R}}

\newcommand{\Rd}{{\mathbb{R}^3}}

\renewcommand{\Re}{\text{Re}}

\newcommand{\F}{\mathcal{F}}
\newcommand{\N}{\mathbb{N}}

\renewcommand{\L}{\mathcal{L}}

\newcommand{\E}{\mathcal{E}}
\newcommand{\DD}{\mathcal{D}}

\newcommand{\G}{\mathcal{G}}
\renewcommand{\F}{\mathcal{F}}
\newcommand{\ve}{\varepsilon}
\newcommand{\vt}{\vartheta}

\newcommand{\vp}{\varphi}
\newcommand{\vpi}{\varpi}
\newcommand{\ga}{\gamma}
\newcommand{\Ga}{\Gamma}
\renewcommand{\th}{\theta}
\newcommand{\wt}{\widetilde}
\newcommand{\ti}{\tilde}
\newcommand{\ol}{\overline}

\newcommand{\pa}{\partial}
\newcommand{\na}{\nabla}
\newcommand{\de}{\delta}
\newcommand{\De}{\Delta}
\newcommand{\al}{\alpha}
\newcommand{\ka}{\kappa}
\newcommand{\si}{\sigma}
\newcommand{\Si}{\Sigma}
\newcommand{\lam}{\lambda}

\newcommand{\wh}{\widehat}
\renewcommand{\S}{\mathbb{S}}

\renewcommand{\b}{\mathbf{b}}
\newcommand{\<}{\langle}
\renewcommand{\>}{\rangle}
\newcommand{\I}{{I}}

\renewcommand{\P}{{P}}

\newcommand{\vertiii}[1]{{\left\vert\kern-0.25ex\left\vert\kern-0.25ex\left\vert #1 \right\vert\kern-0.25ex\right\vert\kern-0.25ex\right\vert}}

\usepackage{tikz,pgfplots}
\usepackage{tikz-cd}
\usepackage{tikz-3dplot}
\usepackage{amsmath,amssymb,amsthm,amsfonts,bm}
\usepackage{color}
\usepackage{tcolorbox}
\usepackage{pdfpages}
\usepackage{textcomp}
\usetikzlibrary{patterns}
\usetikzlibrary{shapes.geometric}
\usetikzlibrary{arrows.meta,arrows}
\usetikzlibrary{decorations.pathreplacing,decorations.pathmorphing}
\usetikzlibrary{hobby}
\usetikzlibrary{math}
\usepgfplotslibrary{fillbetween}
\pgfplotsset{compat=1.18}
\usepackage{graphicx}
\usepackage{multicol}
\usepackage{float}
\usepackage{subfigure}
\usepackage{subcaption}

\allowdisplaybreaks

\makeatletter
\def\@tocline#1#2#3#4#5#6#7{\relax
	\ifnum #1>\c@tocdepth 
	\else
	\par \addpenalty\@secpenalty\addvspace{#2}%
	\begingroup \hyphenpenalty\@M
	\@ifempty{#4}{%
		\@tempdima\csname r@tocindent\number#1\endcsname\relax
	}{%
		\@tempdima#4\relax
	}%
	\parindent\z@ \leftskip#3\relax \advance\leftskip\@tempdima\relax
	\rightskip\@pnumwidth plus4em \parfillskip-\@pnumwidth
	#5\leavevmode\hskip-\@tempdima
	\ifcase #1
	\or\or \hskip 1em \or \hskip 2em \else \hskip 3em \fi%
	#6\nobreak\relax
	\hfill\hbox to\@pnumwidth{\@tocpagenum{#7}}\par
	\nobreak
	\endgroup
	\fi}
\makeatother

\begin{document}

\title[Non-cutoff Boltzmann equation in bounded domains]{The Non-cutoff Boltzmann Equation in Bounded Domains}
\author{Dingqun Deng
			\,\orcidlink{0000-0001-9678-314X}
}
\address{Department of Mathematics, Pohang University of Science and Technology, Pohang, Republic of Korea (South), ORCID: \href{https://orcid.org/0000-0001-9678-314X}{0000-0001-9678-314X}}
\curraddr{}
\email{dingqun.deng@postech.ac.kr}
\email{dingqun.deng@gmail.com}
\date{\today}
\begin{abstract}
	The initial-boundary value problem for the inhomogeneous non-cutoff Boltzmann equation is a challenging open problem. In this paper, we study the stability and long-time dynamics of the Boltzmann equation near a global Maxwellian without angular cutoff assumption in a general $C^3$ bounded domain $\Omega$ (including convex and non-convex cases) with physical boundary conditions: inflow boundary and Maxwell-reflection boundary with accommodation coefficient $\al\in(0,1)$. We obtain the global-in-time existence, which has an exponential decay rate towards the global Maxwellian for both hard and soft potentials. The crucial methods are the forward-backward extension of the boundary problem to the whole space by Vlasov-type equations, a level-function trace lemma, an improved velocity averaging lemma with less regularity but without cutoff in velocity, and an extra damping provided by the advection operator, followed by the De Giorgi iteration and the $L^2$--$L^\infty$ energy method. 
\end{abstract}
\subjclass[2020]{Primary: 35Q20; Secondary 76P05, 35B40, 76N15, 82C40}

\keywords{Boltzmann equation, general bounded domain, extension, global existence, De Giorgi iteration}

\maketitle
\tableofcontents


\section{Introduction}
In 1989, R. J. Diperna and P. L. Lions \cite{Diperna1989} obtained the well-known global existence of renormalized solutions to the Cauchy problem for the Boltzmann equation. Subsequently, in the early 1990s, several researchers investigated the initial-boundary value problem describing the time evolution of rarefied gas in a bounded domain $\Omega$; e.g. 
\cite{Hamdache1992,Cercignani1992,Arkeryd1993}. In the last two decades, many researchers have studied various significant topics of the Boltzmann equation in a domain with boundary, including global existence and uniqueness, the hydrodynamic limit, the boundary layer, the wave solutions, the $L^2$--$L^\infty$ method, the regularity, and the Hilbert expansion, etc.; see Section \ref{sec161} for more discussion. 

\smallskip
However, the above-mentioned literature is limited to the case of Grad's angular cutoff assumption introduced by H. Grad \cite{Grad1963}. It's a challenging open problem to consider a physically reasonable assumption, the angular non-cutoff assumption, in a domain with a boundary. In this paper, we will study the Boltzmann equation without the angular cutoff assumption in a general 3-dimensional open bounded $C^3$ domain $\Omega$.

\subsection{Model and bounded domain}
Let $\Omega$ be an open bounded subset of $\R^3_x$. The \emph{Boltzmann equation} for a particle distribution function $F(t,x,v):[0,\infty)\times\Omega\times\R^3_v$ at time $t\ge 0$, position $x\in\Omega$ and velocity $v\in\R^3$, takes the form
\begin{align}
	\label{B}
	\pa_tF(t,x,v) + v\cdot\na_xF(t,x,v)=Q(F,F)(t,x,v), \quad F(0,x,v)=F_0(x,v),
\end{align}
where $Q$ is the \emph{Boltzmann collision operator} given by
\begin{align*}
	Q(F,G)=\int_{\R^3}\int_{\S^2}B(v-v_*,\sigma)\big(F(v'_*)G(v')-F(v_*)G(v)\big)\,d\sigma dv_*,
\end{align*}
where $(v,v_*)$ and $(v',v_*')$ are pre-post velocities in a collision given by
\begin{align}\label{vprime}
	\begin{aligned}
		v'=\frac{v+v_*}{2}+\frac{|v-v_*|\sigma}{2}, \\
		v'_*=\frac{v+v_*}{2}-\frac{|v-v_*|\sigma}{2},
	\end{aligned}
\end{align}
with $\sigma\in\S^2$. Here $\S^2$ is the $2$-sphere, i.e. $\S^2=\{x\in\R^3\,:\,|x|=1\}$. The pre-post velocities satisfy the conservation laws of momentum and energy for elastic collisions:
\begin{align*}\begin{aligned}
		v+v_* & =v'+v'_*, \\
		|v|^2+|v_*|^2 & =|v'|^2+|v'_*|^2.
	\end{aligned}
\end{align*}
The \emph{cross-section} $B(v-v_*,\sigma)$ is a measure of the probability for the event of a collision (or scattering) given by
\begin{align*}
	B(v-v_*,\sigma)=|v-v_*|^\ga b(\cos\th),
\end{align*}
which depends only on the relative velocity $|v-v_*|$ and the deviation angle $\th$ through $\cos\th=\mathbf{k}\cdot\sigma$, where $\mathbf{k}=\frac{v-v_*}{|v-v_*|}$ and $\cdot$ is the usual scalar product in $\R^3$. Without loss of generality, we assume that $B(v-v_*,\sigma)$ is supported on $\mathbf{k}\cdot\sigma\ge 0$, i.e. $0\le\th\le\frac{\pi}{2}$, since one can reduce to this situation with standard symmetrization: $\ol B(v-v_*,\si)=[B(v-v_*,\si)+B(v-v_*,\si)]\1_{\mathbf{k}\cdot\si\ge 0}$. Moreover, we assume the angular non-cutoff assumption:
\begin{align}
	\label{ths}
	\frac{1}{C_b}\th^{-1-2s}\le\sin\th b(\cos\th)\le C_b \th^{-1-2s},
\end{align}
for some $C_b>0$, which is derived from the inverse power law (for long-range interactions) according to a spherical intermolecular repulsive potential $\phi(r)=r^{-(p-1)}$ $(p>2)$ with $\ga=\frac{p-5}{p-1}$ and $s=\frac{1}{p-1}$; see for instance \cite{Alexandre2000,Maxwell1867,Cercignani1988,Cercignani1994,Villani2002}. The factor $\sin \th$ corresponds to the Jacobian factor for integration in spherical coordinates. Thus, the function $\sin\th b(\cos\th)$ represents a non-integrable singularity as $\th\to 0$ (the grazing collisions), which is an essential difficulty compared to the angular cutoff case. In this paper, we assume that the indices $(\ga,s)$ satisfy
\begin{align*}
	-\frac{3}{2}<\ga\le 2, \quad s\in(0,1).
\end{align*}
It's convenient to call them hard potentials when $\ga+2s\ge 0$ and soft potentials when $\ga+2s<0$.
We remark that the upper bound of $\ga$ doesn't play an essential role, and can be relaxed to any fixed constant.
It's straightforward to check that the global Maxwellian
\begin{align}\label{mu}
	\mu(v):=(2\pi)^{-\frac{3}{2}}e^{-\frac{|v|^2}{2}}
\end{align}
is a steady solution to equation \eqref{B}.
Throughout the paper, we always assume a positive initial datum $F_0\ge 0$. 

{{\smallskip}} The equation \eqref{B} describes the dynamics of moving particles with the collision in a domain $\Omega$. To make equation \eqref{B} mathematically and physically reasonable, we need to introduce the boundary condition. Let $\Omega$ be a open bounded subset of $\R^3_x$, given by
\begin{align}\label{Omega}
	\Omega=\{x\in\R^3\,:\,\xi(x)<0\}\ \text{ with }\ \xi\in C^3(\R^3_x).
\end{align}
(In this work, we only require a $C^3$ boundary.)
Then $\pa\Omega=\{x\in\R^3\,:\,\xi(x)=0\}$. We also assume that $\na\xi(x)\neq 0$ on $\pa\Omega$. 
To use the method of straightening out the boundary, we further assume that there exists an open cover $\{B_k\}_{k=1}^N$ and $C^3$ functions $\rho_k:\R^2\to\R$ with $1\le k\le N<+\infty$ such that (upon relabeling and reorienting the coordinates axes if necessary) 
\begin{align}\label{Bnboundary}
	\Omega\cap B_k =\{x\in B_k\,:\,x_3<\rho_k(x_1,x_2)\}.
\end{align}
This can be easily deduced if $\Omega$ is bounded by using the Heine-Borel theorem of finite open covers. 

\smallskip 
For the regularity of the boundary, we further assume that there exists a vector field
\begin{align}\label{naxi1}
	n(x)\in W^{2,\infty}(\R^3_x;\R^3).
\end{align}
such that $n(x)$ coincides with the outward unit normal vector $\frac{\na\xi}{|\na\xi|}$ at $x\in\pa\Omega$. This can be easily obtained if $\Omega$ is bounded and $\na\xi\neq 0$ on $\pa\Omega$. In fact, in this case, for any sufficiently small $\de>0$, we have
$\na\xi(x)\neq 0$ for any $x$ satisfying $d(x,\pa\Omega)\le \de$.
Here $d(x,\pa\Omega)$ is the distance function between $x$ and $\pa\Omega$.
Then $n(x)=\frac{\na\xi}{|\na\xi|}$ is well defined for $x\in\Omega_\de\equiv\{x\in\R^3_x\,:\,d(x,\pa\Omega)\le\de\}$, which is a compact subset of $\R^3_x$. Moreover, since $\xi\in C^3(\R^3_x)$, we know that $n(x)\in W^{2,\infty}(\{x\in\R^3_x\,:\,d(x,\pa\Omega)\le\de\})$ and hence possesses an Sobolev extension such that \eqref{naxi1} is valid.

%

{{\smallskip}}
We then decompose the boundary of the phase space $\Si:=\pa\Omega\times\R^3_v$ as
\begin{align*}
	\begin{aligned}
		\Si_+ & =\{(x,v)\in\pa\Omega\times\R^3_v\,:\, v\cdot n(x)>0\}, \\
		\Si_- & =\{(x,v)\in\pa\Omega\times\R^3_v\,:\, v\cdot n(x)<0\}, \\
		\Si_0 & =\{(x,v)\in\pa\Omega\times\R^3_v\,:\, v\cdot n(x)=0\},
	\end{aligned}
\end{align*}
to represent the outgoing $(\Si_+)$, incoming $(\Si_-)$ and grazing $(\Si_0)$ sets.

{{\smallskip}} The boundary condition takes into account how the particles are reflected with the wall and thus takes the form of incoming velocity being represented by the outgoing velocity. When the incoming velocity is given by a known time-dependent function, the solution satisfies the inflow boundary condition, On the other hand, J. C. Maxwell \cite[Appendix]{Maxwell1879} introduced a phenomenological law of splitting the reflection operator into a local-in-velocity reflection operator and a diffuse reflection operator (which is nonlocal in velocity); see also \cite{Hamdache1992,Mischler2010}. They are given by:
\begin{enumerate}
	\item The inflow-boundary condition: for $(t,x,v)\in[0,\infty)\times\Si_-$,
	\begin{align*}
		F(t,x,v)|_{\Si_-}=G(t,x,v),
	\end{align*}
	for some given function $G$ on $[0,\infty)\times\Si_-$.
	
	{{\smallskip}}\item The \emph{Maxwell} reflection boundary condition: for $(t,x,v)\in[0,\infty)\times\Si_-$,
	\begin{align*}
		F(t,x,v)|_{\Si_-}=(1-\al)F(x,R_L(x)v)+\al c_\mu\mu(v)\int_{v'\cdot n(x)>0}\{v'\cdot n(x)\}F(t,x,v')(v')\,dv'.
	\end{align*}
	where we use the natural normalization with some constant $c_\mu>0$ satisfying
	\begin{align}\label{cmu}
		c_\mu\int_{v\cdot n(x)>0}\mu(v)|v\cdot n|\,dv=1.
	\end{align}
	Here $\al\in(0,1)$ is a constant, called the \emph{accommodation coefficient}.
	The local reflection operator $R_L(x)$ is given by
	\begin{align}\label{RL1}\begin{aligned}
		& R_L(x)=-v\ (\text{bounce-back reflection}) \\
		\text{ or } & R_L(x)=v-2(n(x)\cdot v)n(x)\ (\text{specular reflection}),
	\end{aligned}
\end{align}
	Note that there always exists a diffuse-reflection part since $\al>0$, and one can consider the pure \emph{diffuse} reflection boundary condition by letting $\al=1$.
	In this work, the condition $\al>0$ is essential since we need to utilize the diffuse effect on the boundary. 
\end{enumerate}

\subsection{Reformulations and boundary condition}
In this work, we use the exponential perturbation $f$ such that
\begin{align*}
	F=\mu+\mu^{\frac{1}{2}}f.
\end{align*}
Then the Boltzmann equation can be rewritten as
\begin{align}
	\label{B1}
	\left\{
	\begin{aligned}
		& \pa_tf+ v\cdot\na_xf = \Gamma(\mu^{\frac{1}{2}}+f,f)+\Gamma(f,\mu^{\frac{1}{2}})\quad \text{ in } (0,\infty)\times\Omega\times\R^3_v, \\
		& f(0,x,v)=f_0\quad \text{ in }\Omega\times\R^3_v,
	\end{aligned}\right.
\end{align}
where the standard Boltzmann collision operator is given by
\begin{align}\label{Ga}
	\Gamma(f,g)=\mu^{-\frac{1}{2}}Q(\mu^{\frac{1}{2}}f,\mu^{\frac{1}{2}}g)=\int_{\R^3}\int_{\S^2}B(v-v_*,\sigma)\mu^{\frac{1}{2}}(v_*)\big(f'_*g'-f_*g\big)\,d\sigma dv_*.
\end{align}
For convenience, we also denote the linear Boltzmann collision operator as
\begin{align}\label{L}
	Lf = \Gamma(\mu^{\frac{1}{2}},f)+\Gamma(f,\mu^{\frac{1}{2}}).
\end{align}
The corresponding boundary conditions are:
\begin{enumerate}
	\item The \emph{inflow}-boundary condition: for $(t,x,v)\in[0,\infty)\times\Si_-$,
	\begin{align}
		\label{inflow}
		f(t,x,v)|_{\Si_-}=g(t,x,v),
	\end{align}
	for some given function $g=\mu^{-\frac{1}{2}}(G-\mu)$ on $[0,\infty)\times\Si_-$.
	
	{{\smallskip}}\item The \emph{Maxwell} reflection boundary condition: for $(t,x,v)\in[0,\infty)\times\Si_-$,
	\begin{align}\label{Maxwell}
		f(t,x,v)|_{\Si_-}=Rf.
	\end{align}
	The \emph{Maxwell} reflection boundary operator is given by
	\begin{align}\label{reflect}
		Rf(x,v)=(1-\al)f(x,R_L(x)v)+\al R_Df(x,v),
	\end{align}
	for any $(x,v)\in\Si_-$. Here $\al\in(0,1)$ is the \emph{accommodation coefficient}.
	The local reflection operator $R_L(x)$ is given by \eqref{RL1}
	and the diffuse reflection operator $R_D(x)$ according to global Maxwellian $\mu$ is defined at the boundary point $x\in\Si_-$ by
	\begin{align}\label{RD1}
		R_D(x)f=c_\mu\mu^{\frac{1}{2}}(v)\int_{v'\cdot n(x)>0}\{v'\cdot n(x)\}f(t,x,v')\mu^{\frac{1}{2}}(v')\,dv'.
	\end{align}
	The dual reflection operator is given by
	\begin{align}\label{reflectdual}\notag
		R^*\Phi(x,v) & =(1-\al)\Phi(x,R_L(x)v) \\&\qquad+\al c_\mu\mu^{\frac{1}{2}}(v)\int_{v'\cdot n(x)<0}\{v'\cdot n(x)\}\Phi(t,x,v')\mu^{\frac{1}{2}}(v')\,dv',
	\end{align}
\end{enumerate}

%
%

\subsection{Notations, weight function and function spaces}\label{SecNotation}
\subsubsection{Miscellany}
Through this paper, $C$ denotes some positive constant (generally large) that may take different values in different lines. $[a_1,a_2,\dots,a_n]$ denotes any $n$-dimensional vector. $\N=\{1,2,3,\dots\}$ is the set of positive natural numbers. 
Write $x_+=\max\{x,0\}$ and $x_-=\max\{-x,0\}$ to be the positive and negative parts respectively.
Let $\1_A$ be the indicator function on the set $A$, $\ol\Omega$ be the closure of domain $\Omega$, and $dS(x)$ be the spherical measure on $\pa\Omega$.
The notation $a\approx b$ (resp. $a\gtrsim b$, $a\lesssim b$) for positive real functions $a,b$ is equivalent to $C^{-1}a\le b\le Ca$ (resp. $a\ge C^{-1}b$, $a\le Cb$) on their domains where $C>0$ is a constant not depending on possible free parameters. 
A constant $C=C(a_1,a_2,\dots)$ means that $C$ depends on $a_1,a_2,\dots$.
The set $C^\infty_c$ consists of smooth functions with compact support.

\subsubsection{Weight function}
Let $\<v\>=(1+|v|^2)^{\frac{1}{2}}$ be the Japanese bracket. To deal with level-function estimates on the boundary, we design a specific weight function. 
To split the grazing and non-grazing set in the trace lemmas \ref{diffboundLem} and \ref{LemR} later, we fix a small constant $\de\in(0,1)$. 
Let $\chi_{|v\cdot n(x)|\le 2\de^{-\frac{1}{4}}}\equiv\chi(v\cdot n(x))$ be a smooth cutoff function with argument $v\cdot n(x)$ satisfying 
\begin{align}
	\label{chivn}
	\1_{|v\cdot n(x)|\le \de^{-\frac{1}{4}}}\le \chi_{|v\cdot n(x)|\le 2\de^{-\frac{1}{4}}}\le \1_{|v\cdot n(x)|\le 2\de^{-\frac{1}{4}}}. 
\end{align}
Then we consider a modified weight function $\<v\>^{l}_{\de}$ given by 
\begin{align}\label{weightvde}
	\<v\>^{l}_{\de}=\frac{\<v\>^l}{(\de^2+\<v\>^{-2}(v\cdot n(x))^2\chi_{|v\cdot n(x)|\le 2\de^{-\frac{1}{4}}})^{\frac{1}{2}}}, 
\end{align}
where $n(x)\in W^{2,\infty}(\R^3_x)$ is the extended ``normal vector'' given by \eqref{naxi1}. 
In Lemmas \ref{vldeLem} and \ref{vldeLem1}, we will show that $\<v\>^l_\de$ has similar properties as $\<v\>^l$ and, for brevity of notations, we denote 
	\begin{align}
		\label{weight}
		\<v\>^l_\de=
		\begin{cases}
			\frac{\<v\>^l}{\big(\de^2+\<v\>^{-2}(v\cdot n(x))^2\chi_{|v\cdot n(x)|\le 2\de^{-\frac{1}{4}}}\big)^{\frac{1}{2}}} &\text{ if }\de\in(0,1),\\
			\<v\>^l &\text{ if }\de=1.
		\end{cases}
	\end{align}
We will use $\<v\>^l$ for the inflow case and $\<v\>^l_\de$, $\de\in(0,1)$, for the Maxwell-boundary case.

\subsubsection{$L^p$ spaces}
We denote $\|\cdot\|_{L^p}$ the Lebesgue $L^p$ norm and write $L^p_tL^q_xL^r_v=L^p_t(L^q_x(L^r_v))$ in short. The underlying domain will be specified in corresponding Sections. For instance,
\begin{align*}
	\|f\|_{L^p_tL^q_xL^r_v([0,T]\times\Omega\times\R^3_v)}:=\Big\{\int_0^T\Big[\int_\Omega\Big(\int_{\R^3_v}|f|^r\,dv\Big)^{\frac{q}{r}}\,dx\Big]^{\frac{p}{q}}\,dt\Big\}^{\frac{1}{p}}.
\end{align*}
The Sobolev space $W^{k,p}(\R^3)$ denotes the set of the tempered distribution such that its $W^{k,p}$ norm is finite:
\begin{align*}
	\|f\|_{W^{k,p}(\R^3)}:=\sum_{j=0}^k\|\na^jf\|_{L^p(\R^3)}<\infty.
\end{align*}
For convenience, we also denote the inner product on the boundary:
\begin{align}\label{boundarySi}
	(f,g)_{L^2(\Si_\pm)}=\int_{\Si_\pm}|v\cdot n|f\cdot g\,dS(x)dv,\quad \|f\|_{L^2(\Si_\pm)}^2=(f,f)_{L^2(\Si_\pm)}.
\end{align}

\subsubsection{Bessel potential}
The Bessel potential operator of order $\ka$ ($0<\Re\ka<\infty$) in $\R^d_x$ is
\begin{align*}
	(I-\De_x)^{-\frac{\ka}{2}}.
\end{align*}
Applying this to any suitable function $f$, it can be represented by the Bessel potential:
\begin{align*}
	(I-\De_x)^{-\frac{\ka}{2}}f&=f*G_\ka,\\
	\text{where} \ \ G_\ka(x) &= \big((1+4\pi^2|\xi|^2)^{-\frac{\ka}{2}}\big)^\vee(x),
\end{align*}
and $(\cdot)^\vee$ is the inverse Fourier transform.
By \cite[Proposition 1.2.5, pp. 13]{Grafakos2014a}, we know that
\begin{align}\label{BesselReal}
	\text{if $\ka>0$ is real, then $G_\ka$ is strictly positive.}
\end{align}
Furthermore, 
\begin{align}\label{GsL1}
	\begin{cases}
		\|G_\ka\|_{L^1_x(\R^d)}=1,&\\
		G_\ka(x)\le C_{\ka,d} e^{-\frac{|x|}{2}}\ &\text{ when }|x|\ge 2,\\
		\frac{1}{C_{\ka,d}}\le \frac{G_{\ka,d}(x)}{H_s(x)}\le C_\ka,\ &\text{ when }|x|\le 2,
	\end{cases}
\end{align}
where 
\begin{align*}
	H_\ka(x)=\begin{cases}
		|x|^{\ka-d}+1+O(|x|^{\ka-d+2}),&\text{if }0<\ka<d,\\
		\ln\frac{2}{|x|}+1+O(|x|^{2}),&\text{if }\ka=d,\\
		1+O(|x|^{\ka-d}),&\text{if }\ka>d.
	\end{cases}
\end{align*}
It follows from \eqref{GsL1} and Young's convolution inequality that for $1\le p\le \infty$,
\begin{align}\label{Besselbound}
	\|(I-\De_x)^{-\frac{\ka}{2}}f\|_{L^p_x(\R^d)}\le\|G_\ka\|_{L^1_x(\R^d)}\|f\|_{L^p_x(\R^d)}=\|f\|_{L^p_x(\R^d)}.
\end{align}
Moreover, we have the equivalent relation: for
$m,n\in\R$ and $1\le p\le \infty$,
\begin{align}
	\label{equiv}\|\<v\>^m\<D_v\>^nf\|_{L^p(\R^d)}\approx \|\<D_v\>^n\<v\>^mf\|_{L^p(\R^d)}.
\end{align}
This follows from \cite[Lemma 2.1 and its proof in Section 8]{Alonso2022} for the case $1\le p\le\infty$ and \cite[Proposition 5.5 and its Corollary, pp. 251--pp. 252]{Stein1993} for the case $1<p<\infty$.
%
Furthermore, we can define the Sobolev space $H^{s}_p(\R^d)$ of order $s\in\R$ by 
\begin{align*}
	H^s_p(\R^d)&=\big\{f\,:\,f \text{ is tempered distribution},\ \|f\|_{H^s_p(\R^d)}<\infty\big\},\\
		&\text{ where }\|f\|_{H^s_p(\R^d)}=\|((1+|\xi|^2)^{\frac{s}{2}}\wh f)^\vee\|_{L^p(\R^d)}. 
\end{align*}
Here, $\wh{f}$ and $f^\vee$ are the Fourier transform and inverse Fourier transform in the sense of tempered distribution, respectively.

\subsubsection{Besov space}
For $\al\in\R$, we denote by $B^{\al,q}_p(\R^{1+d}_{t,x})$ the Besov space about $(t,x)$ as follows. Let $d\ge 2$ be the spatial dimension (which is $3$ is our main result). 
Denote by $\wh\vp$ the Fourier transform of a function $\vp$ in $(t,x)$. 
We fix Schwartz functions $\Psi,\Psi_0$ on $\R^{1+d}_{t,x}$ such that their Fourier transform about $(t,x)$, $\wh\Psi,\wh\Psi_0\in C^\infty_c(\R^{1+d}_{t, x})$, are spherically symmetric, supported in $\{1\le |[\tau,\xi]|\le 3\}$ and $\{|[\tau,\xi]|\le 3\}$ respectively, and satisfy 
\begin{align*}
	\big(\wh\Psi_0(\tau,\xi)\big)^2+\sum^\infty_{j=1}\big(\wh\Psi(2^{-j}\tau,2^{-j}\xi)\big)^2=1\quad \text{ for all }(\tau,\xi)\in\R^{1+d}. 
\end{align*}
We denote by $\De_j$ the corresponding convolution operator defined by 
\begin{align*}
	\De_j\vp = \Psi_{2^{-j}}*_{t,x}\vp\ \ (\forall j\ge 1),\quad \De_0\vp=\Psi_0*_{t,x}\vp, 
\end{align*}
where $\Psi_{2^{-j}}(t,x)=2^{(1+d)j}\Psi(2^jt,2^jx)$. 
Define the inhomogeneous Besov space $B^{\al,q}_p(\R^{1+d}_{t,x})$ $(0<p,q\le\infty)$ to be the space of all tempered distributions $f$ for which the quantity
\begin{align}
	\label{Besov}
	\|f\|_{B^{\al,q}_p(\R^{1+d}_{t,x})}=\|\De_0^2f\|_{L^p_{t,x}}
	+\Big(\sum^\infty_{j=1}(2^{j\al}\|\De_j^2f\|_{L^p_{t,x}})^q\Big)^{1/q}
\end{align}
is finite. To universalize the arguments below, we use $\De^2_j$ instead of $\De_j$ in our definition (but they are equivalent). 
Similarly, one can define the Besov space $B^{\al,q}_p(\R^d_x)$ with respect to $x$. 
Moreover, one can define the Besov space by real interpolation
\begin{align}
	\label{interBesov}
	\big(H^{s_0}_p,H^{s_1}_p\big)_{\th,q}=B^{s,q}_p,
\end{align}
where $s=(1-\th)s_0+\th s_1$, $1\le p,q\le\infty$, $0<\th<1$; see \cite[Theorem 6.2.4]{Bergh1976} for more details. 

\subsubsection{Lorentz space}
The Lorentz space $L^{p,q}(\R^d)$ $(d\ge 1)$ is the space of measurable functions $f$ on $\R^d$ such that the following quasinorm is finite
$$\|f\|_{L^{p,q}(\R^d)}=p^{\frac {1}{q}}\left(\int _{0}^{\infty }t^{q}\big|\big\{x:|f(x)|\geq t\big\}\big|^{\frac {q}{p}}\,{\frac {dt}{t}}\right)^{\frac {1}{q}},$$
where $0<p,q<\infty$ and $|A|$ is the Lebesgue measure of the set $A$. 
Then for $p\ge 1$, by Cavalieri's principle, one has $L^{p,p}=L^p$. Also, one has Hölder's inequality
\begin{align*}
	\|fg\|_{L^{p,q}}\leq C\|f\|_{L^{p_{1},q_{1}}}\|g\|_{L^{p_{2},q_{2}}}, 
\end{align*}
where $0<p,p_{1},p_{2},q,q_{1},q_{2}<\infty$, $\frac{1}{p}=\frac{1}{p_1}+\frac{1}{p_2}$, and $\frac{1}{q}=\frac{1}{q_1}+\frac{1}{q_2}$. 
Moreover, one can define the Lorentz space by real interpolation
\begin{align*}
	\big(L^{p_0},L^{p_1}\big)_{\th,q}=L^{p,q},
\end{align*} 
where $0<p_0<p_1\le\infty$, $p_0<q\le\infty$, $\frac{1}{p}=\frac{1-\th}{p_0}+\frac{\th}{p_1}$, $0<\th<1$; see \cite[Theorem 5.2.1]{Bergh1976} for more details. 
The main embedding between Besov and Lebesgue spaces (equivalent to some Triebel-Lizorkin space) that we will use in this work is from \cite[Theorem 1.1]{Seeger2018}:
\begin{align}\label{embeddBesov}
	\|f\|_{L^{q,r}(\R^{d})}
	\le C\|f\|_{B^{s,2}_p(\R^d)},
\end{align}
where $\frac{s}{d}>\frac{1}{p}-\frac{1}{q}>0$. 

\subsubsection{Kernel of $L$}
The kernel of $L$ in $L^2_v$ is the span of $\{\mu^{\frac{1}{2}},v_i\mu^{\frac{1}{2}}$ $(1\le i\le 3),|v|^2\mu^{\frac{1}{2}}\}$ which follows from collision invariant; see for instance \cite{Cercignani1994}. Then we denote $\P$ the projection onto $\ker L\subset\R^3_v$:
\begin{align}\label{Pf}
	\P f(t,x,v) = (a(t,x)+b(t,x)\cdot v+\frac{|v|^2-3}{6}c(t,x))\mu^{\frac{1}{2}}(v),
\end{align}
where $[a,b,c]$ is given by
\begin{align*}
		[a(t,x),b(t,x),c(t,x)]=\int_{\R^3_v}\Big[1,v,\frac{|v|^2-3}{6}\Big]f(t,x,v)\mu^{\frac{1}{2}}\,dv.
\end{align*}

\subsubsection{Dissipation norm}
To describe the behavior of Boltzmann collision operator, \cite{Alexandre2012} introduces the norm $\vertiii{f}$:
\begin{align*}
	\vertiii{f}^2: & =\int B(v-v_*,\sigma)\Big(\mu_*(f'-f)^2+f^2_*((\mu')^{1/2}-\mu^{1/2})^2\Big)\,d\sigma dv_*dv,
\end{align*}
while \cite{Gressman2011} introduces the anisotropic norm $N^{s,\gamma}$:
\begin{align}\label{Nsga}
	\|f\|^2_{N^{s,\gamma}}: & =\|\<v\>^{\gamma/2+s}f\|^2_{L^2_v}+\int_{\R^6}(\<v\>\<v'\>)^{\frac{\gamma+2s+1}{2}}\frac{(f'-f)^2}{d(v,v')^{3+2s}}\1_{d(v,v')\le 1}\,dvdv',
\end{align}
where $d(v,v'):=\sqrt{|v-v'|^2+\frac{1}{4}(|v|^2-|v'|^2)^2}$.
Moreover, \cite{Global2019} and \cite{Deng2020} use the pseudo-differential-type norm
\begin{align}\label{tiaa}
	\|(\tilde{a}^{1/2})^wf\|_{L^2_v}, \quad \ti a(v,\eta)=\<v\>^{\ga}\big(1+|v|^2+|v\wedge\eta|^2+|\eta|^2\big)^s+K_0\<v\>^{\gamma+2s},
\end{align}
where $(\cdot)^w$ is the Weyl quantization, $K_0>0$ is a sufficiently large constant and $\wedge$ is the wedge product in three dimension; see \cite{Lerner2010,Deng2020a} for more details. We then define the dissipation norm $\|\cdot\|_{L^2_D}$ by
\begin{align}
	\label{L2D}
	\|f\|_{L^2_D}:=\|(\ti a^{\frac{1}{2}})^wf\|_{L^2_v},\quad (f,g)_{L^2_D}=((\ti a^{\frac{1}{2}})^wf,(\ti a^{\frac{1}{2}})^wf)_{L^2_v}.
\end{align}
Then one has from \cite[Eq. (2.13), (2.15)]{Gressman2011}, \cite[Proposition 2.1]{Alexandre2012} and \cite[Theorem 1.2]{Global2019} that these dissipation norms are all equivalent:
\begin{align}\label{tia}
	\|f\|_{L^2_D}\equiv\|(\ti a^{\frac{1}{2}})^wf\|_{L^2_v}\approx \|f\|^2_{N^{s,\gamma}}\approx\vertiii{f}^2.
\end{align}
One also has from \cite[Lemma 2.4]{Deng2020a} that $\|\<v\>^l(\ti a^{\frac{1}{2}})^wf\|_{L^2_v}\approx\|(\ti a^{\frac{1}{2}})^w(\<v\>^lf)\|_{L^2_v}$ for any $l\in\R$.
%

\subsubsection{Inflow and outflow regions}
Inspired by the extension of normal vector in \eqref{naxi1}, we can define the ``inflow'' and ``outflow'' regions in $\ol\Omega^c$ as 
\begin{align}\label{Dinoutwt}
	\begin{aligned}
		&D_{in}=\{(x,v)\in\ol\Omega^c\times\R^3\,:\,
		v\cdot n(x)<0
		\},& \text{ inflow},\\
		&D_{out}=\{(x,v)\in\ol\Omega^c\times\R^3\,:\,
		v\cdot n(x)>0
		\},& \text{ outflow}.
	\end{aligned}
\end{align}
Then their spatial and velocity boundaries can be given by 
\begin{align}\label{paxvDinwt}
	\begin{aligned}
		&\pa_xD_{in}=\Si_-\cup\{(x,v)\in\ol{D_{in}}:v\cdot n(x)=0\},\\
		&\pa_xD_{out}=\Si_+\cup\{(x,v)\in\ol{D_{out}}:v\cdot n(x)=0\},\\
		&\pa_vD_{in}=\{(x,v)\in\ol{D_{in}}\,:\,
		v\cdot n(x)=0
		\},\\
		&\pa_vD_{out}=\{(x,v)\in\ol{D_{out}}\,:\,
		v\cdot n(x)=0\}.
	\end{aligned}
\end{align}
Here, the spatial boundary and velocity boundary can be defined by, for instance, 
\begin{align*}
	&\pa_xD_{out}=\big\{(x,v)\,:\,v\in\R^3\text{ and }x\in\pa\{x\,:\,(x,v)\in D_{out}\}\big\},\\
	&\pa_vD_{out}=\big\{(x,v)\,:\,x\in\R^3\text{ and }v\in\pa\{v\,:\,(x,v)\in D_{out}\}\big\}.
\end{align*}

\subsection{Main result: inflow boundary}
We first present the main results and later list some important observations.
\begin{Thm}[Stability of Boltzmann equation with inflow boundary condition]
	\label{ThmInflowMain}
	Let $\Omega\subset\R^3_x$ be a bounded domain satisfying \eqref{Omega}, \eqref{Bnboundary} and \eqref{naxi1}. 
	Let $-\frac{3}{2}<\ga\le 2$, and $s\in(0,1)$. Fix any 
	$l\ge \ga+10$. Let $l_0=l_0(l,s)$ be a large constant and fix any $\ti C>0$.  
Then there exists a generic constant $c_0=c_0(\ga,s)>0$ and sufficiently small constants $\ve_\infty,\ve_1>0$ (depending on $\ga,s,l,\ti C$ such that if $f_0$ and $g$ satisfy $F_0=\mu+\mu^{\frac{1}{2}}f_0\ge 0$ and 
	\begin{align}\begin{aligned}\label{ve0p}
			\|\<v\>^lg\|_{L^\infty_{t,x,v}([0,\infty)\times\Si_-)}+\|\<v\>^lf_0\|_{L^\infty_{x,v}(\Omega\times\R^3_v)} & =\ve_\infty, \\
			\|e^{c_0t}\<v\>^{l-2}g\|^2_{L^2_{t,x,v}([0,\infty)\times\Si_-)}+\|\<v\>^{l-2}f_0\|^2_{L^2_x(\Omega)L^2_v}&=\ve_1,\\
			\|\<v\>^{l_0}g\|^2_{L^2_{t,x,v}([0,\infty)\times\Si_-)}+\|\<v\>^{l_0}f_0\|^2_{L^2_x(\Omega)L^2_v}&=\ti C, 
		\end{aligned}
	\end{align}
	then there exist a global-in-time solution $f(t)$ $(t\ge 0)$ to the Boltzmann equation \eqref{B1} with inflow boundary condition \eqref{inflow}, i.e. 
\begin{align}\label{non1p}
	\left\{
	\begin{aligned}
		& \pa_tf+ v\cdot\na_xf = \Gamma(\mu^{\frac{1}{2}}+f,f)+\Gamma(f,\mu^{\frac{1}{2}})\quad \text{ in } (0,T]\times\Omega\times\R^3_v, \\
		& f|_{\Si_-}=g\quad \text{ on }[0,T]\times\Si_-, \\
		& f(0,x,v)=f_0\quad \text{ in }\Omega\times\R^3_v,
	\end{aligned}\right.
\end{align}			
		 satisfying $F=\mu+\mu^{\frac{1}{2}}f\ge 0$ and, for any $T>0$ and $k\in[0,l_0]$, we have $L^2$--$L^\infty$ energy estimates:
		\begin{align}\label{L2es1b1}\notag
			&\|\<v\>^kf\|^2_{L^\infty_t([0,T])L^2_x(\Omega)L^2_v}
			+\|\<v\>^kf\|^2_{L^2_t([0,T])L^2_{x,v}(\Si_+)}+c_0\|\<v\>^kf\|_{L^2_t([0,T])L^2_x(\Omega)L^2_D}^2
			\\&\quad+c_0\|\<v\>^kf\|_{L^2_t([0,T])L^2_x(\Omega)L^2_v}^2\le C\|\<v\>^{k}f_0\|^2_{L^2_x(\Omega)L^2_v}+C\|\<v\>^kg\|^2_{L^2_t([0,T])L^2_{x,v}(\Si_-)},
		\end{align}
		and 
		\begin{align}\label{L2es1b2}
			\sup_{t\ge 0}\|\<v\>^lf\|_{L^\infty_{x,v}(\ol\Omega\times\R^3_v)} & \le 
			\ve_\infty+C\ve_1^\zeta,
		\end{align}
		with some constants $C=C(\ga,s,l)>0$ and $\zeta=\zeta(\ga,s)>0$ that are independent of $T$. 
		Moreover, one has large-time asymptotic $L^2$ behavior: 
		\begin{align}\notag\label{L2es1b}
			&e^{c_0t}\|\<v\>^{k}f(t)\|^2_{L^2_x(\Omega)L^2_v}+\|e^{c_0s}\<v\>^kf\|_{L^2_s([0,t])L^2_{x,v}(\Si_+)}^2\\
			&\quad\quad\le C\big(\|\<v\>^{k}f_0\|^2_{L^2_x(\Omega)L^2_v}+\|e^{c_0s}\<v\>^kg\|_{L^2_s([0,t])L^2_{x,v}(\Si_-)}^2\big). 
		\end{align}
\end{Thm}

\smallskip
Here, the solution is in the standard weak sense (see also Theorem \ref{weakfThm}): 
for any $\Phi\in C^\infty_c(\R_t\times\R^3_x\times\R^3_v)$ and $T>0$,
\begin{multline*}
	(f(T),\Phi(T))_{L^2_x(\Omega)L^2_v}-(f,(\pa_t+v\cdot\na_x)\Phi)_{L^2_{t,x,v}([0,T]\times\Omega\times\R^3_v)}+
(f,\Phi)_{L^2_{t,x,v}([0,T]\times\Si_+)}\\
	=(f_0,\Phi(0))_{L^2_x(\Omega)L^2_v}+
	(g,\Phi)_{L^2_{t,x,v}([0,T]\times\Si_-)}
	+\big(\Gamma(\mu^{\frac{1}{2}}+f,f)+\Gamma(f,\mu^{\frac{1}{2}}),\Phi\big)_{L^2_{t,x,v}([0,T]\times\Omega\times\R^3_v)}.
\end{multline*}
The proof will be given in Theorem \ref{ThmInflow} with the global $L^2$ energy estimate in Section \ref{Sec12}. 

\subsection{Main result: Maxwell boundary}
For the case of Maxwell reflection boundary, we further assume that the initial datum $f_0$ satisfies the conservation law in mass:
\begin{align}\label{conservationlaw}
	\int_{\Omega\times\R^3_v}f_0(x,v)\mu^{\frac{1}{2}}\,dxdv=0.
\end{align}
Then the solution $f$ to equation \eqref{B1} also satisfies the mass conservation
\begin{align*}
	\int_{\Omega\times\R^3_v}f(x,v)\mu^{\frac{1}{2}}\,dxdv=0.
\end{align*}
This is for the derivation of the global $L^2$ estimate.

\begin{Thm}[Stability of Boltzmann equation with Maxwell boundary condition]
	\label{ThmReflectionMain}
	Let $\Omega\subset\R^3_x$ be a bounded domain satisfying \eqref{Omega}, \eqref{Bnboundary} and \eqref{naxi1}. 
	Let $\al\in(0,1)$ (accommodation coefficient), $-\frac{3}{2}<\ga\le 2$, and $s\in(0,1)$, $l\ge \ga+10$, $\ti C>0$, and let $l_0=l_0(s,l)>0$ be a large constant. Fix a small $\de=\de(\al)>0$. 
	Suppose the initial data $f_0$ satisfies conservation law \eqref{conservationlaw}, $F_0=\mu+\mu^{\frac{1}{2}}f_0\ge 0$, and 
	\begin{align}
		\label{ve42y}
		\begin{aligned}
			\|\<v\>^{l_0}f_0\|_{L^2_x(\Omega)L^2_v}=\ti C,\quad 
			\|\<v\>^l_\de f_0\|_{L^\infty_x(\Omega)L^\infty_v}=\ve_\infty,\quad\|\<v\>^{l-2}f_0\|_{L^2_x(\Omega)L^2_v}= \ve_1,
		\end{aligned}
	\end{align}
	with sufficiently small $\ve_1,\ve_\infty\in(0,1)$ depending on $\al,\ga,s,l,\ti C$.
	Then there exists a global-in-time solution $f(t)$ $(t\ge 0)$ to the Boltzmann equation \eqref{B1} with Maxwell boundary condition \eqref{Maxwell}, i.e. 
\begin{align*}
	\left\{
	\begin{aligned}
		& \pa_tf+ v\cdot\na_xf = \Gamma(\mu^{\frac{1}{2}}+f,f)+\Gamma(f,\mu^{\frac{1}{2}})\quad \text{ in } (0,T]\times\Omega\times\R^3_v, \\
		& f(t,x,v)|_{\Si_-}=Rf\quad \text{ on }[0,T]\times\Si_-, \\
		& f(0,x,v)=f_0\quad \text{ in }\Omega\times\R^3_v,
	\end{aligned}\right.
\end{align*}		
		such that $F=\mu+\mu^{\frac{1}{2}}f\ge 0$. Moreover, for any $k,T\ge 0$, one has the $L^2$ energy estimate 
		\begin{multline}\label{maines144}
			\sup_{0\le t\le T}\|\<v\>^kf(t)\|^2_{L^2_x(\Omega)L^2_v}+c_{\al}\|\<v\>^kf\|_{L^2_{t}([0,T])L^2_{x,v}(\Si_+)}
		+c_0\|\<v\>^kf(t)\|_{L^2_{t}([0,T])L^2_x(\Omega)L^2_D}^2\\
		+c_0\|\<v\>^kf(t)\|_{L^2_{t}([0,T])L^2_x(\Omega)L^2_v}^2
		\le C\|\<v\>^kf_0\|^2_{L^2_x(\Omega)L^2_v}, 
		\end{multline}
		and $L^\infty$ estimate
		\begin{align}\label{maines144a}
			\sup_{t\ge 0}\|\<v\>^l_\de f\|_{L^\infty_{x,v}(\R^6_{x,v})}
		\le \ve_\infty+C(1+\ti C)^C(\ve_1)^{\zeta},
		\end{align}
	and the large-time $L^2$ decay  
\begin{align}\label{maines157}
	\begin{aligned}
		&e^{c_0 t}\|\<v\>^kf(t)\|^2_{L^2_x(\Omega)L^2_v}+\|e^{c_0s}\<v\>^kf\|^2_{L^2_s([0,t])L^2_{x,v}(\Si_+)} \le \|\<v\>^kf_0\|^2_{L^2_x(\Omega)L^2_v},
	\end{aligned}
\end{align}
whenever the right-hand sides are well-defined, for any $t\ge 0$. 
Here, the constants are $C=C(\al,\ga,s,l)>0$, $\zeta=\zeta(\ga,s)>0$, and $c_0=c_0(\ga,s)>0$.
\end{Thm}

Here, the solution is in the standard weak sense (see also Theorem \ref{weakfThm}): 
for any function $\Phi\in C^\infty_c(\R_t\times\R^3_x\times\R^3_v)$ satisfying $\Phi|_{\Si_+}=R^*\Phi$ with dual reflection operator $R^*$ given by \eqref{reflectdual}, 
\begin{multline*}
	(f(T),\Phi(T))_{L^2_x(\Omega)L^2_v}-(f,(\pa_t+v\cdot\na_x)\Phi)_{L^2_{t,x,v}([0,T]\times\Omega\times\R^3_v)}\\
	=(f_{0},\Phi(0))_{L^2_x(\Omega)L^2_v}
	+\big(\Gamma(\mu^{\frac{1}{2}}+f,f)+\Gamma(f,\mu^{\frac{1}{2}}),\Phi\big)_{L^2_{t,x,v}([0,T]\times\Omega\times\R^3_v)}.
\end{multline*}
The proof of Theorem \ref{ThmReflectionMain} will be given in Section \ref{SecMainMaxwell}. 

\subsection{Main difficulties and idea of the proof}
For the boundary problem of the non-cutoff Boltzmann equation, the main difficulty is to deal with the boundary effect and non-cutoff collision operator involving velocity diffusion. In the cutoff case, the cutoff Boltzmann collision operator obeys the $L^1_{x,v}$ estimate (\cite{Diperna1989}), and $L^2_{x,v}$--$L^\infty_{x,v}$ estimate (\cite{Guo2009}). In the former case, one can use the trace theorem from \cite{Ukai1986} to derive the $L^1_{x,v}$ control on the boundary, and then apply the Diperna-Lions convergence argument \cite{Diperna1989} to obtain the global existence for renormalized solution; see for instance \cite{Hamdache1992,Cercignani1992,Mischler2000}. In the latter case, one can use the decomposition $L=-\nu+K$, where $\nu\approx\<v\>^\ga$ is a positive function and $K$ is a compact operator, and apply the semigroup method and the Duhamel principle, which is roughly
\begin{align*}
	f(t)=e^{-\nu t}f_0+\int^t_{\max\{0,t-t_\b\}}e^{-\nu(t-s)}(Kf+\Gamma(f,f))(s)\,ds,
\end{align*}
where $t_\b(x,v)= \min\big\{\tau>0\,:\,x-v\tau\in\pa\Omega\big\}$ is the backward exit time.
With a delicate study of the backward characteristic line, one can obtain the global existence for the cutoff Boltzmann equation; e.g. \cite{Guo2009,Cao2019}.

{{\smallskip}} In the non-cutoff case, however, the simple $L^1_{x,v}$ estimate and the delicate decomposition $L=-\nu+K$ do not hold. (The semigroup method and Duhamel principle are still available for the non-cutoff case in the whole space, as mentioned in \cite{Deng2020}, but the corresponding semigroup $e^{-tL}$ doesn't keep the nice structure of the semigroup $e^{-\nu t}$ for the cutoff case, where the latter one is an exponentially-decay function. This may create barriers between the semigroup Duhamel method and the $L^2_{x,v}$--$L^\infty_{x,v}$ approach.) In such a situation, the analysis within the domain $\Omega$ could be an essentially difficult task.

\subsubsection{Difficulties for boundary problem and extension to the whole space}\label{Secexten}
To overcome the difficulties arising from the boundary conditions, the crucial method is to extend the boundary problem to a whole-space problem. For this purpose, we consider the Vlasov-type equation with some dissipation in the region $\ol\Omega^c\times\R^3_v$; cf. \cite{Golse1988}. 
That is, we consider the extension in $\ol\Omega^c\times\R^3_v$:
\begin{align}\label{omegac}
\left\{
\begin{aligned}
	&\pa_tf+v\cdot\na_xf+E\cdot\na_vf=P^2f& \text{ in } [T_1,T_2]\times D_{in},\\
	&\pa_tf+v\cdot\na_xf+E\cdot\na_vf=-P^2f& \text{ in } [T_1,T_2]\times D_{out},\\
	& f|_{\pa\Omega}=g & \text{ on }[T_1,T_2]\times\pa\Omega\times\R^3_v,\\
	&f(T_1,x,v)=0& \text{ in }D_{out},\\
	&f(T_2,x,v)=0& \text{ in }D_{in}, 
\end{aligned}\right.
\end{align}
where $E(x,v),P(x,v)$ are functions that will be given in \eqref{EP}. 
We call this the forward-backward extension method; see figure \ref{fig2} for its simplified graphs in two-dimensional spacetime $(x_1,t)$ and in the spatial region. Here the \emph{inflow} region $D_{in}$ and \emph{outflow} region $D_{out}$ are defined in \eqref{Dinoutwt}, corresponding to the domains in $\Omega^c\times\R^3_v$ consisting of the ``inflow particles" (satisfying $v\cdot n(x)<0$) and ``outflow particles" (satisfying $v\cdot n(x)>0$), respectively. The vector $n(x)\in W^{2,\infty}(\R^3_x)$ is given in \eqref{naxi1}. 

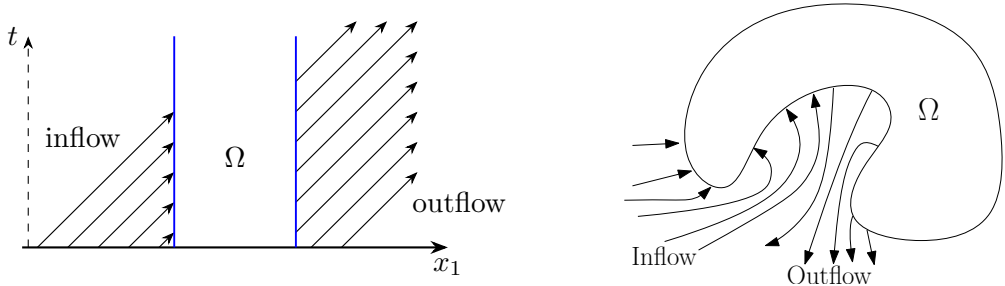
\begin{figure}[htbp]
	\centering
	\begin{tikzpicture}[>=Stealth,scale=0.4]
		\draw[dashed,->] (-6.8,0)--(-6.8,7)node[left]{$t$};
		\draw[thick,->] (-7,0)--(7,0)node[below]{$x_1$};
		\draw[line width=0.15mm,->] (-6.5,0)--(-2,4.5);
		\draw[line width=0.15mm,->] (-5.5,0)--(-2,3.5);
		\draw[line width=0.15mm,->] (-4.5,0)--(-2,2.5);
		\draw[line width=0.15mm,->] (-3.5,0)--(-2,1.5);
		\draw[line width=0.15mm,->] (-2.5,0)--(-2,0.5);
		\draw (0,3) node {$\Omega$};
		\draw (-5,3) node[above] {inflow};
		\draw (5.5,1.5) node[right]{outflow};
		\draw[line width=0.15mm,->] (3.5,0)--(6,2.5);
		\draw[line width=0.15mm,->] (2.5,0)--(6,3.5);
		\draw[line width=0.15mm,->] (2,0.5)--(6,4.5);
		\draw[line width=0.15mm,->] (2,1.5)--(6,5.5);
		\draw[line width=0.15mm,->] (2,2.5)--(6,6.5);
		\draw[line width=0.15mm,->] (2,3.5)--(6,7.5);
		\draw[line width=0.15mm,->] (2,4.5)--(5,7.5);
		\draw[line width=0.15mm,->] (2,5.5)--(4,7.5);
		\draw[line width=0.25mm,blue] (-2,0)--(-2,7);
		\draw[line width=0.25mm,blue] (2,0)--(2,7);
	\end{tikzpicture}
	\hspace{3em}
	\begin{subfigure}
		\centering
		\includegraphics[width=13em]{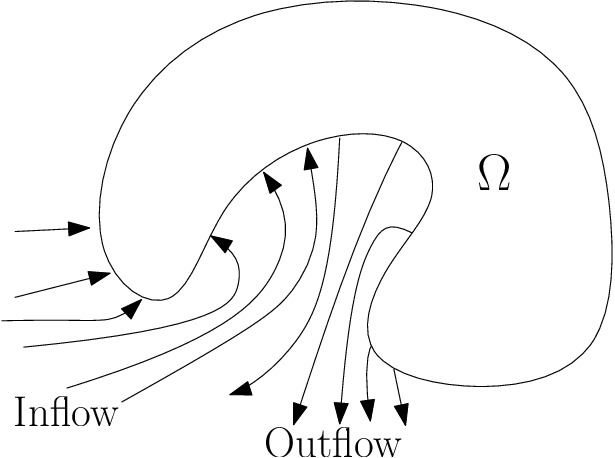}
	\end{subfigure}
	\centering
	\caption{Inflow and outflow regions}
	\label{fig2}
\end{figure}

\smallskip We add the force term $E$ and the dissipation term $P$ to ensure the particle trajectory is curved and to extract the confined effect. That is, all particles starting at $(t,x,v)\in [T_1, T_2]\times D_{in}$ (or $D_{out}$) will be confined in the region $D_{in}$ (or $D_{out}$) along the forward (or backward) characteristic curve. Then the particles will be confined in their own region, i.e. $D_{in}$ or $D_{out}$, and these two parts will not interrupt each other. 

\begin{Rem}
	For the strictly convex domain, one can simply use the transport equation without any forces. 
\end{Rem}

\subsubsection{Velocity averaging lemma}
To obtain the $L^\infty$ estimate, we will use the De Giorgi method, for which we require the time-space-velocity regularity and the embedding theorem. While the velocity regularity is natural for the non-cutoff Boltzmann operator, we will use the velocity averaging lemma to obtain the time-space regularity in some Besov space, which will be given in Lemma \ref{averThma} later. (The averaging lemma is an important analysis tool widely used in kinetic theory and fluid dynamics; see for example \cite{Diperna1989a, Guo2009, Arsenio2011, Jabin2004, DeVore2000}). 

\smallskip 
Furthermore, we provide an enhanced version of the velocity averaging lemma by replacing the smooth cutoff function in the velocity averages with a general regular function $\phi$ without compact support, such as $\phi=\<v\>^{-\rho}$ for some $\rho$. To generalize this assumption, we lose a small amount of regularity compared to \cite{Diperna1991a}. Therefore, the interpolation Lemma \ref{interpoLem} will help to establish the $L^p$ ($p>2$ but close to $2$) energy estimate and finally the $L^\infty$ estimate.

\subsubsection{Difficulties for nonlinear problem}\label{Sec1non}
In the cutoff case, one can obtain strong convergence in some function space (see for instance \cite{Cao2019}) and then pass the limit from the linearized equation to the nonlinear equation. That is, one can use iteration sequence $f^{n+1}$ with $f^0=0$ by 
\begin{align*}
		& \pa_tf^{n+1}+ v\cdot\na_xf^{n+1} = \Gamma(\mu^{\frac{1}{2}}+f^n,f^{n+1})+\Gamma(f^n,\mu^{\frac{1}{2}})\quad \text{ in } (0,T]\times\Omega\times\R^3_v, 
\end{align*}
Using the strong convergence $\|f^{n+1}-f^n\|\to 0$, as $n\to\infty$ in some Lebesgue space, one can find that $f^{n+1}$ and $f^n$ converge to the same limit. 
Thus, passing the limit $n\to\infty$, one obtains the solution $f$ to equation 
\begin{align}\label{sec1non1}
		& \pa_tf+ v\cdot\na_xf = \Gamma(\mu^{\frac{1}{2}}+f,f)+\Gamma(f,\mu^{\frac{1}{2}})\quad \text{ in } (0,T]\times\Omega\times\R^3_v, 
\end{align}
However, in the non-cutoff case, it's hard to obtain strong convergence of $\{f^n\}$ but merely the weak-$*$ convergence of a subsequence of $\{f^{n_k}\}$. In this case, $\{f^{n_k+1}\}$ and $\{f^{n_k}\}$ may not converge to the same limit. In order to overcome this difficulty from nonlinearity, we regularize the equation \eqref{sec1non1} by adding a vanishing regularizing term as in \cite{Alonso2022}. That is, we aim at solving 
\begin{align}\label{sec1noneq}
		& \pa_tf+ v\cdot\na_xf = \vpi Vf+\Gamma(\mu^{\frac{1}{2}}+f,f)+\Gamma(f,\mu^{\frac{1}{2}})\quad \text{ in } (0,T]\times\Omega\times\R^3_v,
\end{align}
for any $\vpi>0$, 
where 
\begin{align}
	\label{Vf}
	Vf=-2\wh{C}^2_0\<v\>^{{8}}f+2\na_v\cdot(\<v\>^{{4}}\na_v)f, 
\end{align}
with some large constant $\wh{C}_0=\wh{C}(\ga,s,l,\al)>0$ to be chosen.
Note that $(\cdot)_{K,+}$ is Lipschitz continuous and one can apply the first-order derivative to it.
By adding a regularizing term, one can easily obtain the strong convergence as in the cutoff case in the short time $T_\vpi>0$ (which depends on $\vpi$). After proving that the existence time $T>0$ for the nonlinear equation doesn't depend on $\vpi$, we can pass the limit $\vpi\to0$ to deduce the existence of equation \eqref{non1p}. 

\smallskip Moreover, as in \cite[Section 7]{Alonso2022}, we need to consider the cases $s\in(0,\frac{1}{2})$ and $s\in[\frac{1}{2},1)$ separately. On one hand, if $s\in(0,\frac{1}{2})$, the regularizing norm $\|\<v\>^2\<D_v\>f\|_{L^2_v}$ in \eqref{Vf} is enough to control the norms $\|\<v\>^2f\|_{H^{2s}_v}$ arising from collision term. On the other hand, if $s\in[\frac{1}{2},1)$, we need to truncate the collision kernel as in \cite{Alonso2022} as follows. For any $s\in[\frac{1}{2},1)$, we fix $s_*\in(0,\frac{1}{2})$ such that 
\begin{align*}
2s-2s_*<1.
\end{align*}
Since the collision kernel satisfies $b(\cos\th)\approx\th^{-2-2s}$ as in \eqref{ths}, for any $\eta\in(0,1)$, we denote 
\begin{align}\label{beta}
b_\eta(\cos\th):=\frac{b(\cos\th)\th^{2+2s}}{\th^{2+2s_*}(\th+\eta)^{2s-2s_*}}.
\end{align}
Then we denote the corresponding collision operator $\Ga_\eta$ as 
\begin{align}\label{Gaeta}
\Ga_\eta(f,g)=\int_{\R^3}\int_{\S^2}|v-v_*|^\ga b_\eta(\cos\th)\mu^{\frac{1}{2}}(v_*)\big(f'_*g'-f_*g\big)\,d\sigma dv_*.
\end{align}
Notice that, whenever $\th\in(0,\frac{\pi}{2})$ and $\eta\in(0,1)$, we have 
\begin{align*}
\frac{1}{(\th+\eta)^{2s-2s_*}}\ge \frac{1}{(\pi+1)^{2s-2s_*}}. 
\end{align*}
Thus, there exists a constant $\al_0>0$ that is independent of $\eta$ and a constant $C_\eta>0$ such that 
\begin{align*}
\frac{\al_0}{\th^{2+2s_*}}\le b_\eta(\cos\th)\le \frac{C_\eta}{\th^{2+2s_*}}.
\end{align*}
This implies that, for any fixed $\eta>0$, $b_\eta(\cos\th)$ can be regarded as a collision kernel with weak singularity $s_*\in(0,\frac{1}{2})$, and hence, the calculations for the case of weak singularity can be applied. Once we obtain the local-in-time solution for the nonlinear problem, we can take the limit $\eta\to 0$ to obtain the solution for strong singularity. 
Moreover, since 
\begin{align}\label{betab}
b_\eta(\cos\th)\le b(\cos\th),
\end{align}
roughly speaking, the calculation for the upper bound of $\Ga_\eta$ is uniform in $\eta$ and hence, the limit $\eta\to 0$ can be taken. See the basic estimates for the case of strong singularity in Section \ref{secStrong}.

\subsubsection{Initial $L^\infty$ bound}
As a quick note, to proceed with the De Giorgi iteration, we add a vanishing dissipation $-\eta\<v\>^lf$, with any $\eta>0$, and use a simple level function argument to obtain an initial $L^\infty$ bound:
 $$\|\<v\>^lf\|_{L^\infty_{t,x,v}}<\infty,$$
whose bound depends on $\eta$ and grows with time; but it won't blow up for a finite time. 
Otherwise if it's infinite, all computations will fail.
In the end, we will obtain an improved $L^\infty$ estimate that is independent of $\eta$, and let $\eta\to0$. 
(Note that we used notation $\eta$ twice, but the vanishing dissipation will be used in Sections \ref{SecLinfty} and \ref{Sec10}, and the ``cut-off" $\Ga_\eta$ will be used in Sections \ref{Sec8} and \ref{Sec11}.)

\subsubsection{The De Giorgi method}\label{sec144}
To obtain the $L^\infty$ estimate of the solution $f$, we utilize the De Giorgi method \cite{DeGiorgi1957}; see also its application to kinetic equations \cite{Guerand2022,Alonso2022}. The De Giorgi method provides an approach to obtain the $L^\infty$ estimate from the $L^2$ (or $L^p$) estimate. To do this, we use level functions with polynomial weight: (for inflow $\de=1$ while for Maxwell $\de\in(0,1)$)
\begin{align}\label{flK}
	f^{(l)}_K := f-K\<v\>^{-l}_\de,\quad f^{(l)}_{K,+}=f^{(l)}_K\1_{f^{(l)}_K\ge 0}. 
\end{align} 
Our goal is, by using the $L^2$ estimate of level functions $f^{(l)}_{K,+}$ and $(-f)^{(l)}_{K,+}$, to deduce
\begin{align*}
\|f^{(l)}_{K,+}(t)\|_{L^2_{x,v}}=0,\quad\text{ or }\quad \|(-f)^{(l)}_{K,+}(t)\|_{L^2_{x,v}}=0,
\end{align*}
with well-chosen $K>0$. Then one has $L^\infty$ estimate $\|\<v\>^l_\de f(t)\|_{L^\infty_{x,v}}\le K$.
\begin{Rem}
	Since $f$ is merely a weak solution, to obtain the $L^2$ energy estimate of $f^{(l)}_{K,+}$ (or even $f$), one may need the chain rule for weak solutions. For details, one can refer to \cite[Lemma 5.6]{Silvestre2022} and \cite[Lemma 2.6]{Zhu2022} for the case of kinetic Fokker-Planck equations, and \cite[Lemma 4.5]{Ouyang2023a} for the case of non-cutoff Boltzmann equation. Such a technique may be used from time to time. 
\end{Rem}

In detail, for the {\bf inflow case}, we split the equation \eqref{sec1noneq} into two equations to obtain its $L^2$--$L^\infty$ estimate, i.e. $f=f_1+f_2$. One has the vanishing initial-boundary value with an artificial dissipation and the other one has non-vanishing initial-boundary value:
\begin{align*}
\left\{
\begin{aligned}
	& \pa_tf_1+ v\cdot\na_xf_1 = \vpi Vf_1+ \Gamma(\mu^{\frac{1}{2}}+f,f_1)+\Gamma(f,\mu^{\frac{1}{2}})-N\<v\>^{l-2}f_1-\eta\<v\>^lf_1\quad \text{ in } \Omega, \\
	& f_1|_{\Si_-}=g\qquad\quad\text{ on } \pa\Omega,\\
	& f_1(0,x,v)=f_0\quad \text{ in }\Omega,
\end{aligned}\right.
\end{align*}
and
\begin{align*}
\left\{
\begin{aligned}
	& (\pa_t+ v\cdot\na_x)f_2 = 
		\vpi Vf_2+\Gamma(\mu^{\frac{1}{2}}+f,f_2)+N\<v\>^{l-2}(f-f_2)-\eta\<v\>^lf_2 \quad \text{ in } \Omega, \\
	& f_2|_{\Si_-}=0\qquad\quad\text{ on } \pa\Omega, \\
	& f_2(0,x,v)=0\quad \text{ in }\Omega\times\R^3_v,
\end{aligned}\right.
\end{align*}
respectively, with large $N>0$. Here we add the term $N\<v\>^{l-2}f_1$ to obtain a good dissipation. Then one can easily obtain the $L^2$ estimate for the level function $(f_1)^{(l)}_{K_1,+}$:
\begin{align*}
	\|(f_1)^{(l)}_{K_1,+}\|_{L^\infty_tL^2_x(\Omega)L^2_v}^2+\|(f_1)^{(l)}_{K_1,+}\|^2_{L^2_tL^2_{x,v}(\Si_+)}
	\le 2\|f^{(l)}_{K_1,+}(0)\|_{L^2_x(\Omega)L^2_v}^2
	+\|g^{(l)}_{K_1,+}(t)\|^2_{L^2_tL^2_{x,v}(\Si_-)}.  
\end{align*}
Setting $K_1$ to be greater than the initial-inflow boundary data yields the $L^\infty$ estimate of $f_1$. 

\smallskip 
For $f_2$, which has vanishing initial data, we will use the forward-backward extension method as in \eqref{omegac} to extend it to the whole space, followed by the De Giorgi method with a more delicate calculation.
By extension to the whole space and using the velocity averaging lemma, we can obtain the time-space-velocity regularity in the sense of energy functional 
\begin{align}\label{EpIntro}\notag
	\E_p(K):&=\|f^{(l)}_{K,+}\|^2_{L^\infty_tL^2_{x,v}([T_1, T_2]\times\R^3_x\times\R^3_v)}
	+\|f^{(l)}_{K,+}\|_{L^2_{t,x}L^2_D([T_1, T_2]\times\Omega\times\R^3_v)}^2\\
	&\notag\quad
	+\vpi\|[\wh{C}^{}_0\<v\>^{{4}}f^{(l)}_{K,+},\<v\>^{{2}}\na_vf^{(l)}_{K,+}]\|^2_{L^2_{t,x,v}([T_1, T_2]\times\Omega\times\R^3_v)}\\
	&\quad+\frac{1
		}{C_0\max\{C_\infty^{2p-2},1\}}\Big\|\int_{\R^3_v}\1_{[T_1,T_2]}\<v\>^{-10}(f^{(l)}_{K,+})^2\,dv\Big\|_{B^{s',2}_p(\R^{1+3}_{t,x})}^p. 
\end{align}
with some parameters $p\in(1,2)$ close to $1$, and $0<s',s<1$. 
For this functional, note that 
\begin{itemize}[leftmargin=2em]
	\item since $f$ has the initial $L^\infty$ bound, we write 
	\begin{align*}
		C_\infty:=\|\<v\>^{l}_{\de}f\|_{L^\infty_{t,x,v}([T_1,T_2]\times\ol\Omega\times\R^3_v)}, 
	\end{align*}
	which is {\bf finite} for any $\eta>0$. Also, the exponent of $C_\infty$ will be essential in the following analysis, and will be canceled by the left-hand $C_\infty$ at the end;
	\item the parameters $C_0,s',p$ will be chosen in Lemmas \ref{interLem}, \ref{energyinterLem}, \ref{LinftyLemVanish} and \ref{energyinterLem1} (which is independent of those mollified parameters such as $\vpi,\eta$);. Moreover, $p>1$ is a constant sufficiently close to $1$ chosen in \eqref{ppsharp}; 
	\item the last term has {\bf exponent} $p$, which gives better energy estimate, for instance, in \eqref{736};
	\item the term $\<v\>^{-10}$ is to capture a good large velocity averages for convenience. For example, $\|\<v\>^{-10}(\cdot)\|_{L^2_v}\le C\|\<v\>^{-8}(\cdot)\|_{L^2_D}$. 
\end{itemize}
Next, one can use the interpolation Lemma \ref{interpoLem} to control the term $\|\<v\>^n((f_2)_{K,+}^{(l)})^2\|_{L^{r}_{x,v}(\Omega\times\R^3_v)}$ for some any $r\in[1,2p]$ and any $n\ge0$,  which allows us to control the extra terms in energy estimates. 
 By delicate analysis of the level functions and the Boltzmann collision operator, we will obtain 
 \begin{align*}
 \E_p(M_{k+1})\le C\frac{2^{k\al}\E_p(M_{k})^{r}}{K_0^{\xi}},
 \end{align*}
 where $r>1$, $M_k:=K_0\big(1-\frac{1}{2^k}\big)$ for $k\ge 0$. The power $r>1$ will induce an exponential decay and suppress the polynomial growth in $2^{k\al}$. 
 Then choosing $K_0>0$ large enough (as a function of all source terms) one has $\E_k(K_0)\le\E_k(M_k)\to 0$ as $k\to\infty$, which implies the upper bound of $f_2$ in $\Omega$, i.e. $f_2\le K_0\<v\>^{-l}_\de$, while the lower bound can be deduced similarly. 
 
 \smallskip Combining the $L^\infty$ estimates for $f_1$ and $f_2$, one can derive the $L^\infty$ estimate for $f=f_1+f_2$. The standard $L^2$--$L^\infty$ method is followed. 
 \begin{Rem}\label{coeff1}
	 Such a splitting allows us to obtain: 
	 \begin{itemize}
		 \item the initial $L^\infty$ estimate for vanishing initial(-inflow) data;
		 \item a good coefficient in \eqref{Linfty1}, which is $1$ and is essential for the global estimate of the continuity arguments for time intervals $[0,1]$, $[1,2]$, $\dots$ (or $[j\de^3,(j+1)\de^3]$ for Maxwell case). 
	 \end{itemize} 
 \end{Rem}

The {\bf Maxwell boundary} case can be done similarly by using the functional $\E_p$ in \eqref{EpIntro}, while the boundary term is analyzed in Subsection \ref{subsecDiff}.

\subsubsection{The diffuse boundary term}\label{subsecDiff}
For the Maxwell reflection boundary, the diffuse reflection part provides advantages in obtaining the boundary effect, but disadvantages in estimating the level function, resulting in our constraint $\al\in(0,1)$ for the accommodation coefficient. 
On the other hand, the inflow boundary condition can be regarded as a ``nice'' feature in the estimation, since the boundary is fixed.

\smallskip 
For instance, to use the forward-backward extension method in \eqref{omegac}, we need to obtain dissipation properties in both the inflow and outflow regions. However, since one cannot absorb all the boundary energy arising from the extended region $\ol\Omega^c\times\R^3_v$, we need to assume the accommodation coefficient $\al\in(0,1)$ in \eqref{reflect} to utilize both the local-reflection and the diffuse-boundary effect. Then it is possible to obtain a few $L^2$ dissipation boundary energy on $\Si_+$ (which depends on the value of $\al\in(0,1)$). 

\smallskip 
To deal with the diffuse boundary term, we use the classic Ukai's trace lemma \ref{diffboundLem} and provide a level-function trace lemma \ref{LemR}. The classic trace lemma \ref{diffboundLem} provides a method to control the boundary energy by interior energy on the non-grazing set: 
\begin{multline*}
	\int_{T}^{s}\int_{\pa\Omega}\int_{v\cdot n(x)>0}|v\cdot n(x)|\chi^{+}_{\de}(t,x,v;{T})|f(v)|^2\,dvdS(x)dt
	\le\|f(T)\|^2_{L^2_x(\Omega)L^2_v}\\
	\quad
	+\int^s_T\int_{\Omega\times\R^3_v}\chi^{+}_{\de}\big(\pa_t|f|^2+v\cdot\na_x|f|^2\big)\,dvdxdt, 
\end{multline*}
with a well-chosen cutoff function $\chi^+_\de$; cf. \cite{Ukai1986,Mischler2010,Guo2021b}. 
In the new trace lemma \ref{LemR}, besides the standard diffuse trace estimate, we provide a level-function trace estimate on the non-grazing set:
\begin{align*}
	\|(R_Df)^{(l)}_{K,+}\|^2_{L^2_t(s,T)L^2_{x,v}(\Si_-)}
	&\le 
	C\de^2\|f^{(l)}_{K,+}\|_{L^2_t(s,T)L^2(\Si_-)}
	+\|f^{(l)}_{K,+}(T)\|^2_{L^2_x(\Omega)L^2_v}\\
	&\quad
	-\int_s^T\int_{\Omega\times\R^3_v}\chi^{-}_{\de}\big(\pa_t|f^{(l)}_{K,+}|^2+v\cdot\na_x|f^{(l)}_{K,+}|^2\big)\,dvdxdt.  
\end{align*}
This utilize the essential property of the weight functio n$\<v\>^{-l}_\de$ given in \eqref{weightvde}, which allows us to control $\mu^{\frac{1}{2}}$ by $\<v\>^{-l}_\de$ on non-grazing set. (The grazing set is small in the sense of energy.)

\subsubsection{Recovering spectral gap and global energy estimate}
The inflow and Maxwell boundary conditions possess a delicate structure for hard and (even) soft potentials: a ``spectral'' gap for the absorbing boundary condition ($g=0$), and exponential time decay. This was observed in an earlier work \cite{Deng2022}.
The crucial idea to recover the spectral gap in the case of soft potentials is to introduce a weight function in the phase variable $(x,v)\in \Omega\times \R^3$ that involves the scalar product of $x$ and $v$. In fact, fixing any positive constant $q>0$, we define the weight function
\begin{align}\label{W}
W = W(x,v) = \exp\Big(-q\frac{x\cdot v}{\<v\>}\Big).
\end{align}
It is straightforward to calculate
\begin{equation}\label{W1}
-v\cdot\na_xW =q\frac{|v|^2}{\<v\>}W.
\end{equation}
and
\begin{align}\label{Wbound}
\begin{aligned}
	& e^{-qC}\leq W\leq e^{qC},\quad
	|\pa_{x_i}W|\leq C qW, \\
	& |\pa_{v_i}W|\leq \frac{Cq}{\<v\>}W,\quad
	|\pa_{v_iv_j}W|\leq \frac{C(q+q^2)}{\<v\>^2}W,
\end{aligned}
\end{align}
for a generic constant $C\geq 1$ depending only on the size of $\Omega$ but not on $q$.
One can choose the weight function of the more general form $W=\exp\big\{-q\langle v\rangle^\vartheta (1-\epsilon \frac{x\cdot v}{\langle v\rangle})\big\}$ to generate higher velocity weight with parameters $q>0$, $0\leq\vt\leq 2$ and $\epsilon>0$. To keep our arguments concise, we will only use the weight $W$ in \eqref{W} due to the fact that $W\approx 1$.

\smallskip Using the above weight $W$ for the inflow and Maxwell conditions, one can derive the global $L^2$ estimate with exponential time decay for both hard and soft potentials.
Using the standard macro-micro decomposition $f=\P f+\{\I-\P\}f$, we derive the microscopic energy from the dissipation property of $Lf$ and then control the macroscopic energy by the microscopic energy; cf. \cite{Guo2002,Guo2009,Guo2021b}. However, for the non-cutoff Boltzmann case, we will carefully choose a smooth cutoff function (instead of an indicator function) for time interval $[T,T+\de^3]$ depending on small fixed $\de>0$. 

\smallskip 
Meanwhile, one also needs the following global $L^2$ weighted and non-weighted estimates for the Maxwell boundary case.  
\begin{itemize}[leftmargin=2em]
	\item Non-weighted $L^2$ estimate with vanishing boundary on both $\Si_+$ and $\Si_-$ (they cancel each other); 
	\item Non-weighted $L^2$ estimate with non-vanishing boundary on $\Si_+$, while the part $\Si_-$ will be controlled by the interior energy with the classic trace lemma;
	\item Weighted $L^2$ estimate estimate.
\end{itemize}

\subsection{Related works} \label{sec161}
Ludwig Boltzmann introduced the Boltzmann equation in 1872 as a significant model for describing the motion of gas particles. The Boltzmann equation is one of the notable nonlinear partial differential equations in mathematical physics and has various applications in statistical physics, plasma physics, special and general relativity, and quantum physics. Since Carleman \cite{Carleman1957} solved the global existence in the spatially homogeneous case for the first time, many important works have explored various topics on the Boltzmann equation. We refer to the following significant works in this regard.

\subsubsection{The $L^1$ existence theory for cutoff Boltzmann equation with boundary}
Since Diperna-Lions (1989) \cite{Diperna1989} established the global weak solution for the cutoff Boltzmann equation in the whole space in the $L^1_{x,v}$ framework, many authors used the $L^1_{x,v}$ framework to solve the existence of the cutoff Boltzmann equation in the bounded domain. For example, Hamdache \cite{Hamdache1992} considered the initial boundary value problem in $\Omega$ whose boundaries are kept at a constant temperature with the linear boundary condition of the form
\begin{align}\label{Kphi}
f|_{\Si_-}=(1-\al)K(f|_{\Si_+})+\al\phi,
\end{align}
where $K$ is a mass-preserving scattering operator and $\phi$ is a given inflow for the case $\al\in(0,1)$. Shortly after, Cercignani \cite{Cercignani1992} considered the same problem for the case $\al=0$, which is a more realistic situation. Later, Arkeryd-Cercignani \cite{Arkeryd1993}, Arkeryd-Maslova \cite{Arkeryd1994} and Arkeryd-Nouri \cite{Arkeryd1997} considered the initial boundary value problem with non-constant boundary temperature. In the new century, Mischler considered the initial-boundary value problem for the Vlasov-Poisson-Boltzmann system with linear boundary \eqref{Kphi} with $\al\in[0,1]$ in \cite{Mischler2000} and various kinetic equations with Maxwell reflection boundary \eqref{Maxwell} with $\al\in(0,1]$ in \cite{Mischler2010}.

\subsubsection{The $L^2$ existence theory for cutoff Boltzmann equation near Maxwellian with boundary}

For the earliest global existence in the $L^2$ framework without boundary, one may refer to \cite{Grad1965,Ukai1974,Ukai1982,Caflisch1980,Caflisch1980a}.

\smallskip
For the boundary problem, Ukai-Asano \cite{Ukai1983} considered the steady solution for a gas flow past an obstacle with several types of reflection boundary conditions, while Asano \cite{Asano1984} established the local existence in a bounded domain with bounce-back and specular reflection boundary condition by using the semigroup method and Duhamel principle. In \cite{Shizuta1977}, it was announced that Boltzmann solution admits a global stable solution near a Maxwellian in a smooth bounded convex domain with specular reflection boundary conditions, but unfortunately, we are not aware of any complete proof for such a result.

\smallskip
In 2005, Yang-Zhao \cite{Yang2005} considered the stability of the one-dimensional Boltzmann equation in a half-space with the specular boundary condition. Later, by the idea of \cite{Vidav1970}, Guo \cite{Guo2009} introduced the $L^2$--$L^\infty$ method to study the time decay and continuity for hard potential, while Liu-Yang \cite{Liu2016} studied this problem for soft potential. Shortly after, Kim-Lee \cite{Kim2017} established stability for the Boltzmann equation with an external potential with specular boundary condition in a $C^3$ convex domain, while the result \cite{Guo2009} was restricted to analytic convex domains. Briant-Guo \cite{Briant2016} investigated the stability in $C^1$ domain for Maxwell boundary condition with an accommodation coefficient $\al\in(\sqrt{2/3},1)$. Guo-Liu \cite{Guo2017} proved the global-in-time existence and uniqueness in a smoothly bounded convex domain with rotational symmetry and the specular reflection boundary condition. Cao-Kim-Lee \cite{Cao2019} obtained the global strong solutions of the Vlasov-Poisson-Boltzmann equation with the diffuse boundary condition.

\smallskip
For the non-convex domain, Kim-Lee \cite{Kim2018} considered the global stability in a periodic-in-$x_2$ cylindrical domain with a non-convex analytic cross-section with specular reflection boundary condition, and the recent preprint by Ko-Kim-Lee \cite{Ko2023} considered the non-convex 3D toroidal domain.

\smallskip
For the wave solutions with boundary, Liu-Yu \cite{Liu2006} considered the coupling of different localized wave solutions of the Boltzmann equation, such as the stationary, non-Maxwellian boundary layers and the interior fluid waves; see also \cite{Yu2008}.

\smallskip
For the non-isothermal boundary, Esposito-Guo-Kim-Marra \cite{Esposito2013} constructed a small-amplitude solution to the steady Boltzmann equation. With these motivations, Duan-Huang-Wang-Zhang \cite{Duan2019} studied the existence and long-time dynamics of the steady Boltzmann equation with soft interaction and non-isothermal boundary in a new mild formulation; see also \cite{Duan2022}.

\smallskip For the large-amplitude initial data, Duan-Wang \cite{Duan2019b} proved the stability for diffuse reflection boundary.

\smallskip For the exterior problem, Ukai-Asano \cite{Ukai1980} considered the linearized Boltzmann equation in an exterior domain, and recently, Dong-Yang-Zhong \cite{Dong2019} considered a flow under the effect of a self-induced electric field past an obstacle (exterior problem) governed by the linearized Vlasov-Poisson-Boltzmann equation.


\subsubsection{Related models in bounded domain}
For the Landau equation, Guo-Hwang-Jang-Ouyang \cite{Guo2020} established the global stability for the specular boundary condition by using a flattening extension near the boundary; see also the correction \cite{Guo2021b}. Then Dong-Guo-Ouyang \cite{Dong2022} extended this result to the Vlasov-Poisson-Landau system.

\smallskip
For the linear Fokker-Planck (Kolmogorov) equation, Hwang-Jang-Vel\'{a}zquez \cite{Hwang2014} established the one-dimensional Fokker-Planck equation in an interval with absorbing boundary conditions, and Hwang-Jang-Jung \cite{Hwang2018} generalized this result to multiple dimension. Recently, Zhu \cite{Zhu2022} established the H\"{o}lder regularity for general linear Fokker-Planck equation with inflow, diffuse, and specular reflection boundary conditions.

\subsubsection{The regularity theory for cutoff Boltzmann equation with boundary}
For the singularity theory, Kim \cite{Kim2011} studied the formation of singularities (non-continuity) at non-convex points of the boundary that propagate along characteristics and the regularity outside an identified set related to these characteristics. This shows that singularities occur at the grazing sets.

\smallskip
For the regularity theory, Guo-Kim-Tonon-Trescases \cite{Guo2015} established the optimal BV estimates in a general non-convex domain with diffuse boundary condition by using a new $W^{1,1}$-trace estimate, while \cite{Guo2016} established the regularity $C^1$ away from the grazing set in a bounded domain with typical reflection boundary conditions, and showed by examples the blow-up of the second derivatives. Chen-Kim \cite{Chen2022} constructed $C^{1,\beta}$ solutions away from the grazing boundary, for any $\beta<1$, to the stationary Boltzmann equation with the non-isothermal diffuse boundary condition in a strictly convex domain, confirming the conjecture in \cite{Guo2016}; see also preprint \cite{Chen2023} for the most recent work.

\smallskip
We also mention that Briant \cite{Briant2015} proved the immediate appearance of the lower bound of mild solutions to the Boltzmann equation in the torus or a $C^2$ convex domain with specular boundary conditions.

\subsubsection{Hydrodynamic/Diffusive limit for cutoff Boltzmann equation with boundary}

By taking the hydrodynamic (or diffusive) limit for the solutions of the Boltzmann equation, it will converge Euler equations (or Navier-Stokes-Fourier equations). For some classic results without boundary, one may refer to \cite{Caflisch1980b,Nishida1978,Golse2003,Levermore2009}. When the boundary is present, the boundary layer effect is non-negligible; cf. \cite{Golse1988a}.

\smallskip
In the $L^1$ framework of renormalized solutions (Diperna-Lions), Masmoudi and Saint-Raymond \cite{Masmoudi2003} considered the Stokes-Fourier fluid dynamic limit in a smooth bounded domain for the Boltzmann equation with Maxwell boundary condition, while Jiang-Masmoudi \cite{Jiang2016} established the incompressible Navier-Stokes-Fourier limit.

\smallskip In the $L^2$ framework, Esposito-Guo-Kim-Marra \cite{Esposito2017} used the $L^2$--$L^\infty$ approach to derive the steady incompressible Navier-Stokes-Fourier limit for the steady Boltzmann equation in a bounded domain with diffuse boundary condition, while \cite{Esposito2018} studied the steady case of flow past an obstacle.
Jang-Kim \cite{Jang2021a} established a rigorous derivation of the incompressible Euler equations with the no-penetration boundary condition from the Boltzmann equation with the diffuse reflection boundary condition. In the recent preprint by Ouyang-Wu \cite{Ouyang2023}, they considered the more challenging inflow boundary condition for the incompressible Navier-Stokes-Fourier limit in $L^2$.

\smallskip For the Hilbert expansion or asymptotic analysis, one may refer to Guo-Huang-Wang \cite{Guo2021a} and the preprint Wu-Ouyang \cite{Wu2020}.

\smallskip For the boundary layer problem, Golse-Benoît-Catherine \cite{Golse1988a} studied the Knudsen layer described by the one-dimensional nonlinear Boltzmann equation in half-space with a boundary condition of a slightly perturbed specular reflection. Ukai-Yang-Yu \cite{Ukai2003,Ukai2004} discussed the boundary layer depending on the Mach number in a one-dimensional half-space with inflow boundary; cf. \cite{Chen2004,Sun2011,Yang2011}. Sakamoto-Suzuki-Zhang \cite{Sakamoto2022} considered the nonlinear boundary layer on a three-dimensional half-space by perturbing around a Maxwellian.


\subsubsection{The non-cutoff Boltzmann equation with boundary}
For the non-cutoff Boltzmann equation in a domain without boundary, one may refer to some early work \cite{Pao1974,Alexandre2000,Alexandre2001,Gressman2011}.

\smallskip
For the non-cutoff Boltzmann equation in a domain with boundary, very recently, Duan-Liu-Sakamoto-Strain \cite{Duan2020} first considered Boltzmann and Landau equations in the finite channel with specular boundary condition. Later, Deng-Duan \cite{Deng2021c} generalized this technique to the Vlasov-Poisson-Boltzmann/Landau system in the finite channel, and Deng \cite{Deng2021,Deng2021e} considered the Boltzmann/Landau equation and Vlasov-Poisson-Boltzmann/Landau system in the union of cubes which has a ``flat" boundary. We also mention the work of Deng \cite{Deng2023} for the stability of rarefaction waves to the Vlasov-Poisson-Boltzmann system in a rectangular duct (with ``flat" boundary) with specular reflection boundary condition.

\smallskip However, these results require that the boundary is flat, which will lead to a great advantage in obtaining the high-order Sobolev regularity. In the general bounded domain, one cannot expect such high-order Sobolev regularity as shown in \cite{Kim2011,Guo2016}.

\smallskip For the $L^\infty$ estimate, Ouyang-Silvestre \cite{Ouyang2023a} obtain an estimate of the conditional $L^\infty$ estimate of the solution depending only on the macroscopic bounds on mass, energy, and entropy for hard potentials in general $C^{1,1}$ bounded domains.

\subsubsection{The De Giorgi method}
The De Giorgi method is an iteration scheme introduced by E. De Giorgi \cite{DeGiorgi1957}. Caffarelli-Vasseur \cite{Caffarelli2011} applied this method to elliptic and parabolic equations with some applications to the quasi-geostrophic equation; see also the lecture note by Vasseur \cite{Vasseur2016}.

\smallskip For the Boltzmann equation, Alonso \cite{Alonso2019} applied the De Giorgi method to spatially homogeneous Boltzmann equation without angular cutoff and obtained the $L^p$ estimate for $p\in[1,\infty]$. Later, Alonso-Morimoto-Sun-Yang \cite{Alonso2022} applied the De Giorgi method to spatially inhomogeneous Boltzmann equation with polynomial perturbation $F=\mu+f$ and without angular cutoff to obtain the $L^\infty$ estimate and established the global solution in the $L^2$--$L^\infty$ framework. Cao \cite{Cao2022b} used simplified arguments and generalized this result to the case of soft potential.

\subsection{Discussions}

\subsubsection{Related models} Using the methods presented here, we expect to be able to study the problem in the presence of a (self-consistent or external) \emph{electric} or \emph{electromagnetic field}. In addition, the collision operator commonly used \emph{Landau collision operator} in plasma physics is of interest. Our study is expected to provide insights into the study of (general) relativistic Boltzmann equation and quantum Boltzmann equation. There is also interest in investigating other nonlinear kinetic collision operators such as Lenard-Balescu and Fokker-Planck collision operators.

\subsubsection{Applications} We believe that our work will provide robust applications to various important topics in kinetic theory. The problems mentioned in \eqref{sec161} for the cutoff Boltzmann equation can now be carried out in the presence of the angular non-cutoff assumption; for example, the existence of wave solutions, non-isothermal boundary problem, (ir-)regularity theory, and fluid dynamic limit.

\subsection{Outline of the paper}
The remainder of the paper is organized as follows. 
\begin{itemize}[leftmargin=2em]
\item In Section \ref{Sec2}, we illustrate some tools in our analysis such as the collision operator estimates, the interpolation inequality, the velocity averaging Lemma, the energy functional interpolation, and two trace lemmes.
\item In Section \ref{Sec3}, we give the $L^2$ estimates of the Boltzmann collision operator, the regular change of variable, and the non-negativity of the solution.
\item In Section \ref{Sec4}, we present the forward-backward extension method for extending the boundary-value problem to the whole-space problem, and the local-in-time $L^2$ existence of the inflow boundary and Maxwell boundary problems. 
\item In Section \ref{Sec5}, for level functions, we establish some $L^p$ $(p=1,2)$ estimate of the collision terms, and the Besov regularity by using the velocity averaging lemma.
\item In Section \ref{SecLinfty}, we establish the $L^\infty$ for the linear equation with inflow boundary condition locally in time.
\item In Section \ref{Sec8}, we prove by the $L^2$--$L^\infty$ method the global existence of the Boltzmann equation with inflow boundary condition.
\item In Section \ref{Sec10}, we establish the $L^\infty$ for the linear equation with the Maxwell reflection boundary condition locally in time.
\item In Section \ref{Sec11}, we prove by the $L^2$--$L^\infty$ method the global existence of the Boltzmann equation with Maxwell reflection boundary condition.
\item In Section \ref{Sec12}, we prove the global a priori $L^2$ decay estimate. This Section is self-consistent in the sense that we don't need the $L^\infty$ estimate in the previous Sections.
\item In Appendix \ref{AppAver}, we give the proof of velocity averaging lemma. 
\end{itemize}

\section{Toolbox}\label{Sec2}
In this Section, we give some basic estimates and tools that are useful in our calculations.

\subsection{The weight function}
Recall the weight function $\<v\>^l$ and $\<v\>^{l}_\de$ given in \eqref{weight}. That is, by fixing a small constant $\de\in(0,1)$ (to be used in trace lemmas \ref{diffboundLem} and \ref{LemR}), we denote 
\begin{align*}
	\<v\>^l_\de=
	\begin{cases}
		\frac{\<v\>^l}{\big(\de^2+\<v\>^{-2}(v\cdot n(x))^2\chi_{|v\cdot n(x)|\le 2\de^{-\frac{1}{4}}}\big)^{\frac{1}{2}}} &\text{ if }\de\in(0,1),\\
		\<v\>^l &\text{ if }\de=1, 
	\end{cases}
\end{align*}
where $\chi_{|v\cdot n(x)|\le 2\de^{-\frac{1}{4}}}\equiv\chi(v\cdot n(x))$ is a smooth cutoff function with argument $v\cdot n(x)$ satisfying $\1_{|v\cdot n(x)|\le \de^{-\frac{1}{4}}}\le \chi_{|v\cdot n(x)|\le 2\de^{-\frac{1}{4}}}\le \1_{|v\cdot n(x)|\le 2\de^{-\frac{1}{4}}}$. 

\smallskip 
Here we list some basic properties of weight function $\<v\>^l$ and modified weight function $\<v\>^l_\de$. 
\begin{Lem}\label{vldeLem1}
	Let $l\ge 0$. Then 
\begin{itemize}[leftmargin=2em]
	\item By mean value theorem and \eqref{vpriminv}, we have 
	\begin{align*}
		\big|\<v'\>^{-l}-\<v\>^{-l}\big|\le C|v'-v|\le C|v-v_*|\sin\frac{\th}{2},  
	\end{align*}
	\item It follows from $|v'|\le |v|+|v_*|$ that 
	\begin{align*}
		\<v\>^{-l}=\frac{\<v_*\>^{l}}{\<v\>^{l}\<v_*\>^{l}}
		\le\frac{\<v_*\>^{l}}{\<v'\>^{l}}.
	\end{align*}
	\item To use Taylor expansion up to second order, we calculate 
	\begin{align*}
		\pa_{v_i}\<v\>^{-l} = -lv_i\<v\>^{-l-2}\quad\text{and}\quad \pa_{v_iv_j}\<v\>^{-l} = -l\de_{ij}\<v\>^{-l-2}+l(l+2)v_iv_j\<v\>^{-l-4}. 
	\end{align*}
	\item Moreover, we have the formula for the difference of square
	\begin{align*}
		\<v'\>^{-l}-\<v\>^{-l}=\big(\<v'\>^{-\frac{l}{2}}-\<v\>^{-\frac{l}{2}}\big)\big(\<v'\>^{-\frac{l}{2}}+\<v\>^{-\frac{l}{2}}\big).  
	\end{align*}
\end{itemize}
\end{Lem}

\begin{Lem}\label{vldeLem}
	Let $l\in\R$ and $\de\in(0,1)$. Then 
	\begin{align*}
		(1)\ &\frac{\<v\>^l}{C_{\|n\|_{L^\infty}}}\le\<v\>^l_{\de}\le C_\de\<v\>^l,\\
		(2)\ &|v\cdot\na_x\<v\>^l_\de|\le C_{\de,\|n\|_{W^{1,\infty}}}\<v\>^l,\\
		(3)\ &\<v\>^{-l}_\de\le\frac{C_{\|n\|_{L^\infty}}\<v_*\>^{l}}{\<v'\>^l},\ \ \text{ if }l>0,\\
		(4)\ &\na_v\<v\>^l_\de=O(\de,l,\|n\|_{L^\infty})\<v\>^{l-2}v+O(\de,l,\|n\|_{L^\infty})\<v\>^{l-2}n(x),\quad |\na^2_v\<v\>^l_\de|\le C_{\|n\|_{L^\infty},\de,l}\<v\>^{l-2},\\
		(5)\ &\<v\>^{l}_\de-\<u\>^{l}_\de=\<v\>^{\frac{l}{2}}_\de(\<v\>^{\frac{l}{2}}-\<u\>^{\frac{l}{2}})+(\<v\>^{\frac{l}{2}}_\de-\<u\>^{\frac{l}{2}}_\de)\<u\>^{\frac{l}{2}}, 
	\end{align*}
	where $O(\de,l,\|n\|_{L^\infty})$ is a function of $(x,v)$ that bounded above by a constant depending on $\de,l,\|n\|_{L^\infty}$. 
	for some constants $C=C_{(\cdot)}>0$ depending only on their arguments that is singular only when $\de\to0$. 
\end{Lem}
\begin{Rem}
	According to Lemma \ref{vldeLem}, the constant in this work may depend on the fixed $\|n\|_{W^{1,\infty}}$ without further notice, while $\de>0$ is a constant that will be used and fixed later in Section \ref{Sec10}. 
\end{Rem}
\begin{proof}
The denominator in \eqref{weightvde} can be estimated as 
\begin{align}\label{eq85}
	\de\le (\de^2+\<v\>^{-2}(v\cdot n(x))^2\chi_{|v\cdot n(x)|\le 2\de^{-\frac{1}{4}}})^{\frac{1}{2}}\le (\de^2+\|n\|_{L^\infty(\R^3_x)}^2)^{\frac{1}{2}}\le (1+\|n\|_{L^\infty(\R^3_x)}^2)^{\frac{1}{2}}, 
\end{align}
which implies (1). 
Moreover, the spatial derivative can be calculated as 
\begin{align*}
	v\cdot\na_x\<v\>^l_{\de}
	=-\frac{\<v\>^{l-2}\big(2v\cdot n(x)\chi_{|v\cdot n(x)|\le 2\de^{-\frac{1}{4}}}+(v\cdot n(x))^2\chi'_{|v\cdot n(x)|\le 2\de^{-\frac{1}{4}}}\big)v_i\pa_{x_j}n_i(x)v_j}{2(\de^2+\<v\>^{-2}(v\cdot n(x))^2\chi_{|v\cdot n(x)|\le 2\de^{-\frac{1}{4}}})^{\frac{3}{2}}}, 
\end{align*}
where repeated indices are summed implicitly. Thus, (2) follows from \eqref{eq85} and the support $\chi_{|v\cdot n(x)|\le 2\de^{-\frac{1}{4}}}$. 
Also, for estimate (3), by using $|v'|\le|v|+|v_*|$ and \eqref{eq85}, we can obtain 
\begin{align*}
	\<v\>^{-l}_\de\le\frac{C_{\|n\|_{L^\infty}}\<v_*\>^{l}}{\<v\>^l\<v_*\>^l}\le \frac{C_{\|n\|_{L^\infty}}\<v_*\>^{l}}{\<v'\>^l}. 
\end{align*}
The velocity derivative is
\begin{align*}
	\na_v\<v\>^l_\de
	&=\frac{l\<v\>^{l-2}v}{(\de^2+\<v\>^{-2}(v\cdot n(x))^2\chi_{|v\cdot n(x)|\le 2\de^{-\frac{1}{4}}})^{\frac{1}{2}}}
	-\frac{\<v\>^l\big(-2v\<v\>^{-4}(v\cdot n(x))^2\chi_{|v\cdot n(x)|\le 2\de^{-\frac{1}{4}}}\big)}{2(\de^2+\<v\>^{-2}(v\cdot n(x))^2\chi_{|v\cdot n(x)|\le 2\de^{-\frac{1}{4}}})^{\frac{3}{2}}}
	\\&\quad-\frac{\<v\>^l\big(\<v\>^{-2}2n(x)v\cdot n(x)\chi_{|v\cdot n(x)|\le 2\de^{-\frac{1}{4}}}
	+\<v\>^{-2}(v\cdot n(x))^2n(x)\chi'_{|v\cdot n(x)|\le 2\de^{-\frac{1}{4}}}\big)}{2(\de^2+\<v\>^{-2}(v\cdot n(x))^2\chi_{|v\cdot n(x)|\le 2\de^{-\frac{1}{4}}})^{\frac{3}{2}}},  
\end{align*}
which implies 
\begin{align*}
	\na_v\<v\>^l_\de=O(\de,l,\|n\|_{L^\infty})\<v\>^{l-2}v+O(\de,l,\|n\|_{L^\infty})\<v\>^{l-2}n(x).
\end{align*}
The second order derivative can be obtained similarly: $|\na^2_v\<v\>^l_\de|\le C(\|n\|_{L^\infty},\de,l)\<v\>^{l-2}$. Then we obtain (4). 
For estimate (5), by using definition \eqref{weightvde}, we have difference-of-square-type formula
\begin{align*}
	\<v\>^{l}_\de-\<u\>^{l}_\de
	&=\frac{\<v\>^\frac{l}{2}(\<v\>^{\frac{l}{2}}-\<u\>^{\frac{l}{2}})+\<v\>^\frac{l}{2}\<u\>^{\frac{l}{2}}}{(\de^2+\<v\>^{-2}(v\cdot n(x))^2\chi_{|v\cdot n(x)|\le 2\de^{-\frac{1}{4}}})^{\frac{1}{2}}}
	-\frac{\<u\>^{\frac{l}{2}}\<u\>^{\frac{l}{2}}}{(\de^2+\<u\>^{-2}(u\cdot n(x))^2\chi_{|u\cdot n(x)|\le 2\de^{-\frac{1}{4}}})^{\frac{1}{2}}}\\
	&=\<v\>^{\frac{l}{2}}_\de(\<v\>^{\frac{l}{2}}-\<u\>^{\frac{l}{2}})+(\<v\>^{\frac{l}{2}}_\de-\<u\>^{\frac{l}{2}}_\de)\<u\>^{\frac{l}{2}}.
\end{align*}
This completes the proof of Lemma \ref{vldeLem}. 
\end{proof}

\subsection{Collision operator}
We will frequently use these embeddings without further notice:
\begin{align*}
	\|f\|_{L^1_v}\le C\|\<v\>^4f\|_{L^\infty_v}, \\
	\|f\|_{L^2_v}\le C\|\<v\>^2f\|_{L^\infty_v}.
\end{align*}
Some standard estimates of collision operators $L$ and $\Gamma$ (given by \eqref{L} and \eqref{Ga}) can be obtained from \cite{Gressman2011,Alexandre2012,He2018}, and we list them below.
From \cite[Proposition 2.2]{Alexandre2012} or \cite[Eq. (2.15)]{Gressman2011}, we have
\begin{align}
	\label{esD}
	\|\<v\>^{\frac{\ga+2s}{2}}f\|_{L^2_v}+\|\<v\>^{\frac{\ga}{2}}\<D_v\>^sf\|_{L^2_v}\lesssim \|f\|_{L^2_D}\lesssim \|\<v\>^{\frac{\ga+2s}{2}}\<D_v\>^sf\|_{L^2_v}.
\end{align}
This can also be derived from \eqref{tiaa} by using pseudo-differential estimates; see \cite{Lerner2010,Deng2020a}.
By \cite[Eq. (6.6), pp. 817]{Gressman2011}, for $\ga+2s>-\frac{3}{2}$,
\begin{align}\label{Gammaes}
	|(\Gamma(f,g),h)_{L^2_v}|\le C\|f\|_{L^2_v}\|g\|_{L^2_D}\|h\|_{L^2_D}.
\end{align}
By \cite[Proposition 3.13, pp. 967]{Alexandre2012}, for any $l\ge 0$, we have
\begin{multline}\label{25}
	|(\<v\>^l\Gamma(f,g)-\Gamma(f,\<v\>^lg),h)_{L^2_v}|
	\le C_l\Big(\|\<v\>^{\frac{\ga+2s}{2}}f\|_{L^2_v}\|\<v\>^{l+\frac{\ga}{2}}g\|_{L^2_v}\\
	+\min\big\{\|f\|_{L^2_v}\|\<v\>^{l+\frac{\ga}{2}}g\|_{L^2_v},\|\<v\>^{s+\frac{\ga}{2}}f\|_{L^2_v}\|\<v\>^{l-s}g\|_{L^2_v}\big\}\Big)\|h\|_{L^2_D}. 
\end{multline}
By \cite[Eq. (2.13), pp. 784]{Gressman2011} or \cite[Proposition 2.1]{Alexandre2012}, for $\ga>-3$ and $s\in(0,1)$, one has
\begin{align}
	\label{micro}
	(Lf,f)_{L^2_v}\le -c_0\|\{\I-\P\}f\|_{L^2_D}^2,
\end{align}
for some $c_0>0$, where $\P f$ is given in \eqref{Pf}.
By \cite[Proposition 4.8, pp. 983]{Alexandre2012} and \eqref{esD}, for $\ga>-3$ and $s\in(0,1)$, one has
\begin{align}\notag
	\label{L1}
	\big(\Gamma(\mu^{\frac{1}{2}},f),\<v\>^{2l}f\big)_{L^2_v} & \le -2c_0\|\<v\>^lf\|_{L^2_D}^2+C\|\<v\>^{l+\frac{\ga}{2}}f\|_{L^2_v}^2 \\
	& \le -c_0\|\<v\>^lf\|_{L^2_D}^2+C_{l}\|\1_{|v|\le R_0}f\|_{L^2_v}^2,
\end{align}
for some constant $c_0>0$ and large $R_0>0$, where we used interpolation in $v$ and $\1_{|v|\le R_0}$ is the indicator function of $B(0,R_0)$. By \cite[Proposition 4.5]{Alexandre2012}, we have
\begin{align}
	\label{L2}
	\big|\big(\Gamma(\psi,\mu^{\frac{1}{2}}),\<v\>^{2l}f\big)_{L^2_v}\big|
	\le C\|\mu^{\frac{1}{10^4}}\psi\|_{L^2_v}\|\mu^{\frac{1}{10^4}}f\|_{L^2_v}.
\end{align}

\begin{Lem}\label{LemGa}
	Let $l\ge 0$, $\ga>\max\{-3,-\frac{3}{2}-2s\}$, and $\ga+2s<4$. We have
	\begin{align}\label{Gaesweight}
		|(\Gamma(f,g),\<v\>^{2l}h)_{L^2_v}|
		\le C_l\|\<v\>^{2}f\|_{L^2_v}\|\<v\>^{l}g\|_{L^2_D}\|\<v\>^{l}h\|_{L^2_D},
	\end{align}
	and
	\begin{align}\label{GaesweightDvs}
		\|\<v\>^{l-\frac{\ga+2s}{2}}\<D_v\>^{-s}\Gamma(f,g)\|_{L^2_v}\le C_l\|\<v\>^{2}f\|_{L^2_v}\|\<v\>^{l}g\|_{L^2_D}.
	\end{align}
	Consequently, if we let $\Psi=\mu^{\frac{1}{2}}+\psi$, then
	\begin{align}\label{56}
		\begin{aligned}
			\big(\Gamma(\Psi,f),\<v\>^{2l}f\big)_{L^2_v}
			& \le \big(-c_0+C\|\<v\>^{4}\psi\|_{L^\infty_v}\big)\|\<v\>^lf\|_{L^2_D}^2+C\|\1_{|v|\le R_0}f\|_{L^2_v}^2, \\
			\big(\Gamma(\vp,\mu^{\frac{1}{2}}),\<v\>^{2l}f\big)_{L^2_v} & \le C\|\mu^{\frac{1}{10^4}}\vp\|_{L^2_v}\|\mu^{\frac{1}{10^4}}f\|_{L^2_v},
		\end{aligned}
	\end{align}
	where $R_0>0$ is some large constant.
\end{Lem}
\begin{proof}
	By \eqref{esD}, \eqref{Gammaes} and \eqref{25},
	\begin{align}\label{214}\notag
		|(\Gamma(f,g),\<v\>^{2l}h)_{L^2_v}| & \le|(\Gamma(f,\<v\>^{l}g),\<v\>^{l}h)_{L^2_v}|+|(\<v\>^l\Gamma(f,g)-\Gamma(f,\<v\>^lg),h)_{L^2_v}| \\
		& \le C_l\big(\|\<v\>^{\frac{\ga+2s}{2}}f\|_{L^2_v}+\|f\|_{L^2_v}\big)\|\<v\>^{l}g\|_{L^2_D}\|\<v\>^{l}h\|_{L^2_D}.
	\end{align}
	This implies \eqref{Gaesweight}. 
	By \eqref{esD} and \eqref{equiv}, we have
	\begin{align*}
		|(\Gamma(f,g),\<v\>^{2l}h)_{L^2_v}|
		\le C_l\|\<v\>^{\max\{\frac{\ga+2s}{2},0\}}f\|_{L^2_v}\|\<v\>^{l}g\|_{L^2_D}\|\<v\>^{l+\frac{\ga+2s}{2}}\<D_v\>^sh\|_{L^2_v},
	\end{align*}
	and hence, by duality,
	\begin{align*}
		\|\<v\>^{l-\frac{\ga+2s}{2}}\<D_v\>^{-s}\Gamma(f,g)\|_{L^2_v}\le C_l\|\<v\>^{\max\{\frac{\ga+2s}{2},0\}}f\|_{L^2_v}\|\<v\>^{l}g\|_{L^2_D}.
	\end{align*}
		This implies \eqref{GaesweightDvs}.
		Applying \eqref{L1}, \eqref{L2} and \eqref{Gaesweight}, we can obtain \eqref{56}.
		This completes the proof of Lemma \ref{LemGa}.
	\end{proof}
	\begin{Rem}
		One can have a more accurate dual estimate than \eqref{GaesweightDvs} by using pseudo-differential calculus. The left-hand side of \eqref{tia} defines the norm of Sobolev space $H(\ti a^{\frac{1}{2}})$; see \cite[Section 2]{Lerner2010}. By using the dual property $\big(H(\ti a^{\frac{1}{2}})\big)^*=H(\ti a^{-\frac{1}{2}})$, one can deduce from \eqref{Gaesweight} that
		\begin{align*}
			\|\<v\>^{-l}(\ti a^{-\frac{1}{2}})^w\Gamma(f,g)\|_{L^2_v}\le C\|\<v\>^{4}f\|_{L^\infty_v}\|\<v\>^{l}g\|_{L^2_D}.
		\end{align*}
		In this work, we only use the less accurate dual estimate \eqref{GaesweightDvs}.
	\end{Rem}
	
	\smallskip We then give some facts about the pre-post velocities.
	Let
	\begin{align}\label{bfk}
		\mathbf{k}=\frac{v-v_*}{|v-v_*|}\ \text{ if }v\ne v_*;\quad \mathbf{k}=(1,0,0)\ \text{ if }v=v_*.
	\end{align}
	Under the spherical coordinate, we write
	\begin{align}\label{sigma}
		\sigma=\cos\th\,\mathbf{k}+\sin\th\,\omega,
	\end{align}
	with $\th\in[0,\pi/2]$, $\omega\in\S^1(\mathbf{k})$, where
	\begin{align}\label{S1k}
		\S^1(\mathbf{k})=\{\omega\in\S^2\,:\,\omega\cdot\mathbf{k}=0\}.
	\end{align}
	Then we have
	\begin{align}\label{vprimeth}
		\begin{aligned}
			v' & =\cos^2\frac{\th}{2}v+\sin^2\frac{\th}{2}v_*+\frac{1}{2}|v-v_*|\sin\th\omega, \\
			v'_* & =\sin^2\frac{\th}{2}v+\cos^2\frac{\th}{2}v_*-\frac{1}{2}|v-v_*|\sin\th\omega, 
		\end{aligned}
	\end{align}
	and
	\begin{align}\label{vpriminv}
		\begin{aligned}
			& |v'-v|=|v'_*-v_*|=|v-v_*|\sin\frac{\th}{2}, \\
			& |v'-v_*|=|v'_*-v|=|v-v_*|\cos\frac{\th}{2}, \\
			& |v'_*-v'|=|v_*-v|.
		\end{aligned}
	\end{align}
	Then it follows from $\th\in\big(0,\frac{\pi}{2}\big]$ that
	\begin{align}\label{vstar1}\notag
		|v_*|^2 & \le \big(|v_*-v|+|v|\big)^2
		\le \big(\frac{|v'_*-v|}{\cos\frac{\pi}{2}}+|v|\big)^2 \\
		& \le \big(\sqrt{2}|v'_*|+(\sqrt{2}+1)|v|\big)^2
		\le 4|v'_*|^2+18|v|^2,
	\end{align}
	and similarly,
	\begin{align}\label{vstar2}
		|v'|^2\le 4|v|^2+18|v'_*|^2.
	\end{align}
	Also, 
	\begin{align}\label{vstar3}
		|v'|^2\le \big(|v'-v_*|+|v_*|\big)^2\le \big(|v-v_*|+|v_*|\big)^2\le 2|v|^2+8|v_*|^2.
	\end{align}	
				
\subsection{Interpolation inequality}
In this Subsection, we introduce the crucial interpolation inequality to be combined with the velocity averaging lemma. This is the crucial idea to obtain the $L^\infty$ estimate of the solution. 
\begin{Lem}
	\label{interpoLem}
	Let $0\le T_1<T_2<\infty$, 
	$p\in(1,2)$, $\eta,\eta'\in(0,1)$ satisfying $0<\eta'<(d+1)/p$, and $d\ge 2$ be the dimension. 
	Assume that $\Omega\subset\R^d$ is bounded and let $\psi\in L^{2}(\R^d_v)$ be any (averaging function) satisfying $\|\psi\|_{L^\infty_v}\le 1$ (this $1$ can be any fixed constant).
	 Then there exists $r=r(\eta,\eta',p,d)>2$ and $\si=\si(\eta,\eta',p,d)\in (0,1)$ such that for any suitable function $\vp:\R^{1+2d}\to \R$ (such that the right-hand side of \eqref{inter} is finite), 
	\begin{align}\label{inter}
		\|\vp\psi\|_{L^r_{t,x,v}([T_1,T_2]\times\Omega\times\R^d_v)}\le C\|(I-\De_v)^{\frac{\eta}{2}}\vp\|_{L^2_{t,x,v}([T_1,T_2]\times\Omega\times\R^d)}^{\si}\,\Big\|\int_{\R^d_v}\1_{[T_1,T_2]}(\vp\psi)^2\,dv\Big\|_{B^{\eta',2}_p(\R^{1+d}_{t,x})}^{\frac{1-\si}{2}}, 
	\end{align}
	where $C=C(\eta,\eta',p,d)>0$ is a constant. 
	Moreover, $r=r(\eta,\eta',p,d)>2$ is a non-decreasing function with respect to $p$ for fixed $\eta,d$. Also, $r(\eta,\eta',p,d)$ is continuous with respect to $\eta$, $\eta'$ and $p$, and satisfies
	\begin{align*}
			\lim_{\eta'\to0,\,p\to 1}r(\eta,\eta',p,d)=2,&\quad\lim_{p\to 1}r(\eta,\eta',p,d)>2.
	\end{align*}
	and 
	\begin{align}\label{almr}
%
		\frac{\si}{2}+\frac{1-\si}{2p}>\frac{1}{r},\quad \frac{(1-\si)r}{2}<1. 
	\end{align}
\end{Lem}
\begin{proof}
	By Sobolev embedding on $\R^d$ for Bessel potential (\cite[Theorem 1.3.5]{Grafakos2014a}) and the embedding theorem for Besov space \eqref{embeddBesov}, by noting $\|\cdot\|_{L^{n}}=\|\phi\|_{L^{n,n}}$, we have
	\begin{align}\label{371}
		\begin{aligned}
			\Big(\int_{\R^d_v}|\vp(x,v)|^m\,dv\Big)^{\frac{2}{m}}&\le C\|(I-\De_v)^{\frac{\eta}{2}}\vp(x,\cdot)\|_{L^2_v(\R^d)}^2, \\
			\Big\|\int_{\R^d_v}\vp^2(\cdot,v)\,dv\Big\|_{L^{n}(\R^{1+d}_{t,x})}
			&\le C\Big\|\int_{\R^d_v}\vp^2(\cdot,v)\,dv\Big\|_{B^{\eta',2}_p(\R^{1+d}_{t,x})},
		\end{aligned}
	\end{align}
	where $C=C(\eta,\eta',p,d)>0$ is a generic constant, and we choose $m>2,n>p$ by 
	\begin{align}\label{pqmd}
		\frac{1}{m}=\frac{1}{2}-\frac{\eta}{d},\quad 0<\frac{1}{p}-\frac{1}{n}=\frac{\eta'}{2(1+d)}<\frac{\eta'}{1+d}.
	\end{align} 
	That is $m=\frac{2d}{d-\eta}$ and $n=\frac{2p(1+d)}{2(1+d)-p\eta'}$. 
	Set constants $\si_1,\si_2,r$ by 
	\begin{align}\label{pqmdr2}
		\frac{1-\si_1}{1-\si_2}=n,\quad \frac{\si_1}{\si_2}=\frac{2}{m},\quad r=m\si_1+2(1-\si_1)=2\si_2+2n(1-\si_2),
	\end{align}
	which, together with $n>p>1$ and $m>2$, implies 
	\begin{align}\label{pqmdr}
		\si_2=\frac{n-1}{n-\frac{2}{m}}\in(0,1),\quad \si_1=\frac{2}{m}\si_2=\frac{n-1}{\frac{m}{2}n-1}\in(0,1).
	\end{align}
	Thus, using $L^{\frac{1}{\si_1}}_v-L^{\frac{1}{1-\si_1}}_v$ and $L^{\frac{1}{\si_2}}_{t,x}-L^{\frac{1}{1-\si_2}}_{t,x}$ H\"{o}lder's inequality, and \eqref{371}, we obtain
	\begin{align*}
		& \|\vp\psi\|_{L^r_{t,x,v}([T_1,T_2]\times\Omega\times\R^d_v)}^r = \int_{[T_1,T_2]\times\Omega\times\R^d_v}|\vp(t,x,v)\psi(v)|^{m\si_1}|\vp(t,x,v)\psi(v)|^{2(1-\si_1)}\,dvdxdt\\
		&\quad \le\int_{[T_1,T_2]\times\Omega}\Big(\int_{\R^d}|\vp|^{m}\,dv\Big)^{\si_1}\,\Big(\int_{\R^d}|\vp\psi|^{2}\,dv\Big)^{1-\si_1}\,dxdt\\
		&\quad \le\Big(\int_{[T_1,T_2]\times\Omega}\Big(\int_{\R^d}|\vp|^{m}\,dv\Big)^{\frac{\si_1}{\si_2}}\,dxdt\Big)^{\si_2}\,\Big(\int_{[T_1,T_2]\times\Omega}\Big(\int_{\R^d}|\vp\psi|^{2}\,dv\Big)^{\frac{1-\si_1}{1-\si_2}}\,dxdt\Big)^{1-\si_2}\\
		&\quad =\Big(\int_{[T_1,T_2]\times\Omega}\Big(\int_{\R^d}|\vp|^{m}\,dv\Big)^{\frac{2}{m}}\,dxdt\Big)^{\si_2}\,\Big(\int_{[T_1,T_2]\times\Omega}\Big(\int_{\R^d}|\vp\psi|^{2}\,dv\Big)^{n}\,dxdt\Big)^{1-\si_2}\\
		&\quad\le C\|(I-\De_v)^{\frac{\eta}{2}}\vp\|_{L^2_{t,x,v}([T_1,T_2]\times\Omega\times\R^d)}^{2\si_2}\,\Big\|\int_{\R^d_v}\1_{[T_1,T_2]}(\vp\psi)^2(\cdot,v)\,dv\Big\|_{B^{\eta',2}_p(\R^{1+d}_{t,x})}^{n(1-\si_2)},
	\end{align*}
	This implies \eqref{inter} by taking power $(\cdot)^{\frac{1}{r}}$ and letting 
	\begin{align}\label{si000}
		\si:=\frac{2\si_2}{r},\quad\text{ and }\quad 1-\si=\frac{2n(1-\si_2)}{r}. 
	\end{align}
	Notice from \eqref{pqmd}, \eqref{pqmdr2} and \eqref{pqmdr} that $m=\frac{2d}{d-2\eta}$ and
	\begin{align}\notag\label{rDef}
		r &=m\si_1+2(1-\si_1)=\frac{2mn-m-2n}{\frac{m}{2}n-1}\\
		&\notag=\frac{4nd-2d-2n(d-2\eta)}{nd-d+2\eta}
		=\frac{2nd-2d+4n\eta}{nd-d+2\eta}\\
		&=2+\frac{4(n-1)\eta}{(n-1)d+2\eta}>2. 
	\end{align}
	From this and $n=\frac{2p(1+d)}{2(1+d)-p\eta'}$, we know that $r$ is continuous with respect to $\eta$, $\eta'$ and $p$. Moreover, $r=r(\eta,\eta',p,d)$ is a non-decreasing function with respect to $p$ for fixed $\eta,\eta',d$. Also, from \eqref{pqmd} and \eqref{rDef}, we have $\lim_{p\to 1}r(\eta,\eta',p,d)>2$, and 
	\begin{align*}
		\lim_{\eta'\to0}n=p,\quad \text{ and }\quad
		\lim_{\eta'\to0,\,p\to 1}r(\eta,\eta',p,d)=2. 
	\end{align*}
	To prove \eqref{almr}, 
	we apply \eqref{si000} and \eqref{pqmdr} to deduce
	\begin{align*}
		\frac{\si}{2}+\frac{1-\si}{2p}=\frac{\si_2}{r}+\frac{n(1-\si_2)}{pr}
		>\frac{1}{r}, 
	\end{align*}
	and 
	\begin{align*}
		\frac{(1-\si)r}{2}=n(1-\si_2)=\frac{1-\frac{2}{m}}{1-\frac{2}{mn}}<1. 
	\end{align*}
	This completes the proof of Lemma \ref{interpoLem}.
\end{proof}

\subsection{Velocity Averaging Lemma}
We would like to use the $L^p$ velocity averaging Lemma in \cite[Theorem 5]{Diperna1991a}, which considers a smooth cutoff function in the velocity averages and gives a general statement without energy estimates and without time-regularizing on the right-hand side. Here we present a more precise statement with a general averaging function $\psi$ without compact support. 

\begin{Thm}\label{averThma}
	Let $d\ge 2$ be the dimension, $0\le T_1<T_2$, $\ka\in[0,1)$, $m\ge 0$, $p\in(1,\infty)$, $n>0$. Suppose $G$ and $\psi$ are such that the following right-hand sides are well-defined.
	 Denote by $B^{\al,q}_p$ the Besov space given by \eqref{Besov}.
	
	\begin{enumerate}[leftmargin=*]
	\item Let $f$ be the solution 
	\begin{align}\label{transG}
		\pa_tf+v\cdot\na_xf=(I-\De_{t,x})^{\ka/2}(I-\De_v)^{m/2}G\ \text{ in }\R_t\times\R^d_x\times\R^d_v,
	\end{align}
	in the sense of distribution. 
		If $p\in(1,\infty)$, then 
		\begin{align}
			\label{avereq}
			\Big\|\int_{\R^d}f(v)\psi(v)\,dv\Big\|_{B^{\al,2}_p(\R^{1+d}_{t,x})}
			\le C_{d,p}\|\<v\>^{n}\<D_v\>^{m+1}\psi\|_{L^{2}_{v}}\Big(\|f\|_{L^p(\R^{1+2d}_{t,x,v})}+\|G\|_{L^p(\R^{1+2d}_{t,x,v})}\Big), 
		\end{align}
		where 
		$\al=\frac{n(1-\ka)}{(1+2n)(1+m)\max\{p,p'\}}$. 

	\smallskip 
\item Assume $p\in(1,2]$ and $\ka\in[0,\frac{1}{p})$. If $f\in L^p([T_1,T_2]\times\R^d_x\times\R^d_v)$ is the solution 
\begin{align}\label{transG1}
	\pa_tf+v\cdot\na_xf=G\ \text{ in }[T_1,T_2]\times\R^d_x\times\R^d_v,
\end{align}
in the sense of distribution, then 
$\int_{\R^d}\1_{[T_1,T_2]}(t)f(v)\psi(v)\,dv\in B^{\al,2}_p(\R^{1+d}_{t,x})$ 
\begin{align}\notag
	\label{avereq1}
	&\ \Big\|\int_{\R^d}\1_{[T_1,T_2]}(t)f(v)\psi(v)\,dv\Big\|_{B^{\al,2}_p(\R^{1+d}_{t,x})}\\
	&\quad\notag\le C_{d,p}\|\<v\>^{n}\<D_v\>^{m+1}\psi\|_{L^{2}_{v}}
	\Big(\|(I-\De_{x})^{-\frac{\ka}{2}}(I-\De_v)^{-\frac{m}{2}}f(T_1)\|_{L^p(\R^{2d}_{x,v})}\\
	&\qquad \notag+\|(I-\De_{x})^{-\frac{\ka}{2}}(I-\De_v)^{-\frac{m}{2}}f(T_2)\|_{L^p(\R^{2d}_{x,v})}\\
	&\qquad+\|\1_{[T_1,T_2]}f\|_{L^p(\R^{1+2d}_{t,x,v})}+\|(I-\De_{t,x})^{-\frac{\ka}{2}}(I-\De_v)^{-\frac{m}{2}}(\1_{[T_1,T_2]}G)\|_{L^p(\R^{1+2d}_{t,x,v})}\Big),
\end{align}
with regularity 
\begin{align}
	\label{lam1}
	\al=\frac{n(1-\ka)}{(1+2n)(1+m)}\big(1-\frac{1}{p}\big). 
\end{align}
\end{enumerate}
\end{Thm}
We will put the proof in Appendix \ref{AppAver}.

\subsection{Energy functional interpolation}
We will apply the De Giorgi method to deduce the $L^\infty$ estimate of the equation and prepare the following notations. We write $\<v\>^{-l}_\de$ as in \eqref{weight} and use the (polynomial) level functions as in \eqref{flK}: 
\begin{align*}
	f^{(l)}_{K,+} := (f-K\<v\>^{-l}_\de)_+\ \text{ with constant }K>0.
\end{align*}
Then for any $K>M$ and $k>1$, by Lemma \ref{vldeLem}, (we simply write $C=C_{\|n\|_{L^\infty}}$ later on)
\begin{align}
\label{fKM}
f^{(l)}_{K,+}\le \frac{f^{(l)}_{K,+}(f^{(l)}_{M,+})^{k-1}}{(K-M)^{k-1}(\<v\>^{-l}_\de)^{k-1}}\le \frac{C_{\|n\|_{L^\infty}}\<v\>^{l(k-1)}(f^{(l)}_{M,+})^k}{(K-M)^{k-1}}.
\end{align}
For any $-\infty<M<K<\infty$, $0<s',s<1$, $l\ge 0$, and $p>1$, we introduce the energy functional as in Subsection \ref{sec144}: 
\begin{align}\label{Ep}\notag
	\E_p(K):&=\|f^{(l)}_{K,+}\|^2_{L^\infty_tL^2_{x,v}([T_1, T_2]\times\R^3_x\times\R^3_v)}
	+\|f^{(l)}_{K,+}\|_{L^2_{t,x}L^2_D([T_1, T_2]\times\Omega\times\R^3_v)}^2\\
	&\notag\quad
	+\vpi\|[\wh{C}^{}_0\<v\>^{{4}}f^{(l)}_{K,+},\<v\>^{{2}}\na_vf^{(l)}_{K,+}]\|^2_{L^2_{t,x,v}([T_1, T_2]\times\Omega\times\R^3_v)}\\
	&\quad+\frac{1
		}{C_0\max\{C_\infty^{2p-2},1\}}\Big\|\int_{\R^3_v}\1_{[T_1,T_2]}\<v\>^{-10}(f^{(l)}_{K,+})^2\,dv\Big\|_{B^{s',2}_p(\R^{1+3}_{t,x})}^p,
\end{align}
where $C_0=C_0(l,\ga,s,p)>0$ is a large constant to be chosen. 
This is the main energy functional that will appear in the energy inequality. 
To apply iteration on the level energy $\E_p(K)$, we also denote the zeroth level energy $\E_0$ as $\E_0:=\E_p(0)$. 
The main difficulty in closing the level-function energy is to control the extra weighted $L^2$ and $L^1$ norms, i.e. $\|\cdot\|_{L^2}^2$ and $\|\cdot\|_{L^1}$. Thus, we apply De Giorgi's level-function arguments to raise the exponent and embed it back into $L^2$ energy in \eqref{Ep}; cf. \cite{Alonso2022,Vasseur2016,Guerand2022}. 

\medskip \noindent{\bf Choice of $p$.}
We begin with introducing some parameters $p,r(1),r(p),p^\#$ as follows. For any $p>1$, 
we choose $r(1)$ and $r(p)$ be the parameters given in Lemma \ref{interpoLem} such that 
\begin{align*}
	r(1) = r(s,s',1,3)>2,\quad r(p)=r(s,s',p,3)>r(1), 
\end{align*}
which implies $\frac{r(1)-2}{r(\frac{3}{2})}>0$. Note from Lemma \ref{interpoLem} that
$r(\cdot)$ is an increasing function, which implies 
\begin{align}\label{629}
	r(1)<r(p)\le r(p'),
\end{align}
for any $1<p<p'$. 
Let 
\begin{align}\label{qthstar}
	q_* := 1+\frac{1}{2}\frac{r(1)-2}{r(p^\#)}.
\end{align}
Then, noticing $q_*-\frac{r(1)}{2}=\frac{(2-r(1))(r(p^\#)-2)}{2r(p^\#)}<0$, we have 
\begin{align}\label{78}
	& 1<q_*<\frac{r(1)}{2},\quad
	\frac{r(p^\#)}{2}\frac{2q_*-2}{r(1)-2}=\frac{1}{2}<1.
\end{align}
On the other hand, 
by the non-decreasing property of $r(p)$, for any $p\in(1,\frac{3}{2})$, we have 
\begin{align*}
	\frac{r(1)-2}{r(p)}>\frac{r(1)-2}{r(\frac{3}{2})}>0.
\end{align*}
Thus, by continuity of $r(\cdot)$, there exists $p^\#>1$ such that
\begin{align}\label{ppsharp}
	\begin{aligned}
		\text{ for }1<p\le p^\#,\ 
		\text{ one has }\ &p<\min\{\frac{3}{2},q_*\},\\
		 \quad \text{ and }\ &2p-2<\min\Big\{1,\,\frac{r(1)-2}{r(p^\#)}\Big\}.
	\end{aligned}
\end{align}
For the extra $L^q$ ($q=1,2$) norms in energy estimate, we will use regularity in $(t,x,v)$ and embedding theory to control. The following energy functional interpolation is motivated by \cite[Lemma 3.8]{Alonso2022}; \cite{Alonso2022} used the hypoelliptic property of kinetic equations. However, in our case, the velocity diffusion is presented only within $\Omega$, and we can only use the weaker regularity, the velocity averaging lemma, to obtain the time-space regularity with respect to Besov space. Therefore, the norm on the left-hand side of \eqref{77} below must be within $\Omega$. Moreover, Lemma \ref{interLem} is a functional inequality that holds independently of the equation $f$. 
\begin{Lem}[Energy functional interpolation]
	\label{interLem}
	Let $0\le T_1<T_2<\infty$, $C_0>0$, and $0<s',s<1$, and $l,m\ge 0$. Denote $p^\#$ as in \eqref{ppsharp} and fix $1<p\le p^\#$. 
	Then there exist parameters $r_*,\xi_*$, given in \eqref{rstar} and \eqref{xistar}, depending only on $(s,s',p)$, satisfying
	\begin{align}\label{rxistar}
		r_*>1, \quad \xi_*>2, 
	\end{align}	
	such that the following holds. 

\medskip 
	Let $q\in [1,2p]$, and 
	let $l_0>0$ be a sufficiently large constant to be determined in \eqref{l0} 
	that depends only on $m,s,s',p,l$. Suppose $f$ satisfies
	\begin{align}\label{L2a}
		\|\<v\>^{l_0+l-2}f\|^2_{L^2_{t,x,v}([T_1,T_2]\times\Omega\times\R^3_v)}\le C_1,\quad \|\<v\>^lf\|_{L^\infty_{t,x,v}([T_1,T_2]\times\Omega\times\R^3_v)}=C_\infty, 
	\end{align}
	with some $C_1,C_\infty>0$. 
	Then for any $0\le M< K$,
	\begin{align}\label{77}
		\|\<v\>^{\frac{m}{q}}f^{(l)}_{K,+}\|^{q}_{L^{q}_{t,x}([T_1,T_2]\times\Omega)L^2_v}&\le \frac{C\big(C_0\max\{C_\infty^{2p-2},1\}\big)^{\frac{(1-\si)\beta_*\xi_*}{2p}}C_1^{\frac{(1-\beta_*)\xi_*}{4}}(\E_p(M))^{r_*}}{(K-M)^{\xi_*-q}}.
	\end{align}
	where $\E_p$ is given by \eqref{Ep}, $C=C(s,s',p)>0$ is independent of $C_1,f,l,T,m$.
	Moreover, we have
	\begin{align}\label{sibexi1}
		\frac{(1-\si)\beta_*\xi_*}{2p}<1,\quad \xi_*>2+\frac{r(1)-2}{r(p^\#)}. 
	\end{align}
	Furthermore, the estimates \eqref{77} holds for $-f$, with $f^{(l)}_{K,+}$ replaced by $(-f)^{(l)}_{K,+}$ (also in $\E_p(M)$).
\end{Lem}
\begin{proof}
	Let $q\in[1,2p]$. In this proof, we will use the interpolation lemma \ref{interpoLem} to control the $L^{2q}$ norm. 
To proceed, we first select the parameters $q_*$, $r_*$, and $\xi_*$ to satisfy some H\"older's indices as follows. 
Let $\si=\si(s,s',p,3)\in(0,1)$ be the coefficient given by Lemma \ref{interpoLem}.
Then we determine $\xi_*\in(2,r(p))$ and $\beta_*\in(0,1)$ by
\begin{align}\label{79a}
	\frac{1}{\xi_*} & =\frac{1-\beta_*}{2}+\frac{\beta_*}{r(p)}, \\ 1&=\si\frac{\beta_*\xi_*}{2}+(1-\si)\frac{\beta_*\xi_*}{2p}.
	\label{79b}
\end{align}
That is, 
\begin{align}\label{xistar}\begin{aligned}
	\xi_*&=\frac{r(p)-2}{r(p)}\frac{1}{\frac{\si}{2}+\frac{1-\si}{2p}}+2,\\
	\beta_*&=\frac{1}{\xi_*}\frac{1}{\frac{\si}{2}+\frac{1-\si}{2p}}=\frac{1}{\frac{r(p)-2}{r(p)}+2\big(\frac{\si}{2}+\frac{1-\si}{2p}\big)}.
	\end{aligned}
\end{align}
Moreover, the first half part of \eqref{sibexi1} follows from \eqref{79b}.
By \eqref{almr}, we have
\begin{align}\label{834}
	\frac{1}{2}>\frac{\si}{2}+\frac{1-\si}{2p}>\frac{1}{r(p)}, \
\end{align}
and hence, $\beta_*
<\frac{1}{\frac{r(p)-2}{r(p)}+\frac{2}{r(p)}}
= 1$, which implies $\beta_*\in(0,1)$.
Further, by \eqref{629}, \eqref{qthstar}, \eqref{78} and \eqref{834},
\begin{align}
			\label{xi2q}
	\xi_*-2q_* & >\frac{2(r(p)-2)}{r(p)}+2-2q_*
	>\frac{2(r(1)-2)}{r(p)}-\frac{r(1)-2}{r(p^\#)}>0,
\end{align}
which implies the second part of \eqref{sibexi1}. 

\smallskip 
With the above parameters, we can now calculate the norm $\|\<v\>^{\frac{m}{2}}f^{(l)}_{K,+}\|^{2q}_{L^{2q}_{t,x}([T_1,T_2]\times\Omega)L^2_v}$. 
First, it follows from \eqref{ppsharp} and \eqref{xi2q} that $\xi_*>2q_*\ge 2p^\#\ge 2p\ge 2q$. 
Thus, by \eqref{fKM} and H\"{o}lder inequality about $(t,x,v)$ with indices \eqref{79a},
\begin{align}\label{714}\notag
	&\|\<v\>^{\frac{m}{q}}f^{(l)}_{K,+}\|^{q}_{L^{q}_{t,x}([T_1,T_2]\times\Omega)L^2_v}
	\le\frac{C}{(K-M)^{\xi_*-q}}\|\<v\>^{\frac{m+l(\xi_*-q)}{\xi_*}}f^{(l)}_{M,+}\|^{\xi_*}_{L^{\xi_*}_{x,v}([T_1,T_2]\times\Omega\times\R^3_v)}\\
	&\notag\quad\le\frac{C}{(K-M)^{\xi_*-q}}\|\<v\>^{\frac{m+l(\xi_*-q)}{\xi_*}+5\beta_*}(f^{(l)}_{M,+})^{1-\beta_*}\<v\>^{-5\beta_*}(f^{(l)}_{M,+})^{\beta_*}\|^{\xi_*}_{L^{\xi_*}_{x,v}([T_1,T_2]\times\Omega\times\R^3_v)}\\
	&\quad\le\frac{C}{(K-M)^{\xi_*-q}}\|\<v\>^{\frac{l_0-4}{2}}f^{(l)}_{M,+}\|^{(1-\beta_*)\xi_*}_{L^{2}_{t,x,v}([T_1,T_2]\times\Omega\times\R^3_v)}\|\<v\>^{-5}f^{(l)}_{M,+}\|^{\beta_*\xi_*}_{L^{r(p)}_{t,x,v}([T_1,T_2]\times\Omega\times\R^3_v)},
\end{align}
where we put all the extra velocity weight in the second factor, and choose 
\begin{align}
	\label{l0}
	l_0\ge\frac{2}{1-\beta_*}\big(\frac{m+l(\xi_*-q)}{\xi_*}+5\beta_*\big)+4
\end{align} 
as a constant depending on $m,s,s',p,l$ which is independent of $T_1,T_2,f$.
For the second right-hand factor of \eqref{714}, by \eqref{L2a} and definition of $\E_p$ in \eqref{Ep}, we have
\begin{align}\label{639}\notag
	\|\<v\>^{\frac{l_0-4}{2}}f^{(l)}_{M,+}\|^{(1-\beta_*)\xi_*}_{L^{2}_{t,x}([T_1,T_2]\times\Omega)L^2_v}
	&\le\|\<v\>^{l_0-2}f^{(l)}_{M,+}\|^{\frac{(1-\beta_*)\xi_*}{2}}_{L^{2}_{t,x}([T_1,T_2]\times\Omega)L^2_v}\|\<v\>^{-2}f^{(l)}_{M,+}\|^{\frac{(1-\beta_*)\xi_*}{2}}_{L^{2}_{t,x}([T_1,T_2]\times\Omega)L^2_v}\\
	&\le C_1^{\frac{(1-\beta_*)\xi_*}{4}}(\E_p(M))^{\frac{(1-\beta_*)\xi_*}{4}}.
\end{align}
For the last factor of \eqref{714}, applying Lemma \ref{interpoLem} with $(r,\eta,\eta',p)=(r(p),s,s',p)$ and $\psi = \<v\>^{-5}$ therein, 
we obtain
\begin{align}\label{715a}
	&\notag\|\<v\>^{-5}f^{(l)}_{M,+}\|^{\beta_*\xi_*}_{L^{r(p)}_{t,x,v}([T_1,T_2]\times\Omega\times\R^3_v)}\\
	&\quad\notag\le C\|(I-\De_v)^{\frac{s}{2}}(\<v\>^{-5}f^{(l)}_{M,+})\|_{L^2_{t,x,v}([T_1,T_2]\times\Omega\times\R^3)}^{\si\beta_*\xi_*}\,\Big\|\int_{\R^3_v}\1_{[T_1,T_2]}(\<v\>^{-5}f^{(l)}_{M,+})^2\,dv\Big\|_{B^{s',2}_p(\R^{1+3}_{t,x})}^{\frac{(1-\si)\beta_*\xi_*}{2}}\\
	&\quad\le C\big(C_0\max\{C_\infty^{2p-2},1\}\big)^{\frac{(1-\si)\beta_*\xi_*}{2p}}\E_p(M)^{\frac{\si\beta_*\xi_*}{2}+\frac{(1-\si)\beta_*\xi_*}{2p}},
\end{align}
where 
$C=C(s,s',p)>0$, and we used 
\eqref{esD}, \eqref{equiv} and the functional $\E_p$ in \eqref{Ep}.

\medskip
Substituting \eqref{639} and \eqref{715a} into \eqref{714}, we have
\begin{multline*}
	\|\<v\>^{\frac{m}{2}}f^{(l)}_{M,+}\|^2_{L^2_{t,x}([T_1,T_2]\times\Omega)L^2_v}
	\le \frac{C\big(C_0\max\{C_\infty^{2p-2},1\}\big)^{\frac{(1-\si)\beta_*\xi_*}{2p}}C_1^{\frac{(1-\beta_*)\xi_*}{4}}}{(K-M)^{\xi_*-q}}
	\\\times\E_p(M)^{\frac{(1-\beta_*)\xi_*}{4}+\frac{\si\beta_*\xi_*}{2}+\frac{(1-\si)\beta_*\xi_*}{2p}}.
\end{multline*}
For the exponent of $\E_p(M)$, we have from 
\eqref{79b},
\eqref{xistar}, 
\eqref{834} and \eqref{ppsharp} that 
\begin{align}\label{rstar}\notag
	r_*: & =\frac{(1-\beta_*)\xi_*}{4}+\frac{\si\beta_*\xi_*}{2}+\frac{(1-\si)\beta_*\xi_*}{2p}\\
	&\ge \frac{(1-\beta_*)\xi_*}{4}+1>1,
\end{align}
which concludes \eqref{rxistar} and \eqref{77}. 
Since the estimate \eqref{77} is only a functional inequality on the function $f$ itself,
similar arguments and estimates can be carried out on the term $(-f)^{(l)}_{K,+}$ instead of $f^{(l)}_{K,+}$.
This concludes Lemma \ref{interLem}.
\end{proof}

\subsection{The classic trace lemma}\label{Sec103}
The main difficulty in solving the Boltzmann equation with \emph{Maxwell} reflection boundary condition is to deal with the diffuse boundary term $R_Df$ in \eqref{RD1}. 
In this Subsection, we begin with providing a classic trace lemma that controls the boundary energy by the interior energy, as seen in \cite{Ukai1986,Guo2009,Guo2021b}, with some modification. Another new trace lemma will be given in Section \ref{SecNewTrace}. First, we denote some useful cutoff notations. 
For any $\de>0$, denote $\chi_1:\R\to[0,1]$ and $\chi_2:\R^3\to[0,1]$ be smooth cutoff functions satisfying
\begin{align*}
	\chi_1(r)=\left\{\begin{aligned}
		 & 1\quad\text{ if } r\ge3\de^2, \\
		 & 0\quad\text{ if } r<2\de^2,
	 \end{aligned}\right.
	\quad \chi_2(v) = \left\{\begin{aligned}
 & 1\quad\text{ if } |v|\le 2\de^{-\frac{1}{4}}, \\
 & 0\quad\text{ if } |v|>4\de^{-\frac{1}{4}},
	 \end{aligned}\right.
\end{align*}
and
\begin{align}\label{nachi12}
	|\chi_1'(r)|\le C\de^{-2},\quad |\chi_1''(r)|\le C\de^{-4}, \quad |\na_v\chi_2(v)|\le C\de^{\frac{1}{4}},\quad |\na^2_v\chi_2(v)|\le C\de^{\frac{1}{2}}.
\end{align}
We also need an extension for the normal outward vector $n=n(x)$ of $\Omega$. Recall that we assume that the outward unit normal vector $n=n(x)$ has an extension to $\R^3_x$ in \eqref{naxi1} such that
\begin{align}\label{nW2}
	n(x)\in W^{2,\infty}(\R^3_x).
\end{align}
For any $T\in[T_1,T_2]$, we construct the backward smooth cutoff function supported on the outflow region:
\begin{align}\label{chideY}
	\chi_\de(t,x,v;{T})=\chi_1\big(v\cdot n(x-v\{t-{T}\})\big)\chi_2(v). 
\end{align}
Assuming $t\in [{T},{T}+\de^3]$, such a cutoff functions satisfies 
\begin{itemize}[leftmargin=2em]
	\item 
	$\pa_t\chi_\de+v\cdot\na_x\chi_\de=0$,
and
\begin{align*}
	\chi_\de(t,x,v;{T})=\left\{\begin{aligned}
 & 1,\quad\text{ if } v\cdot n(x-v\{t-{T}\})\ge3\de^2\text{ and }|v|\le 2\de^{-\frac{1}{4}}, \\
 & 0,\quad\text{ if } v\cdot n(x-v\{t-{T}\})<2\de^2\text{ or }|v|>4\de^{-\frac{1}{4}}.
	 \end{aligned}\right.
\end{align*}
\item 
If $\chi_\de(t,x,v;{T})>0$, then $v\cdot n(x-v\{t-{T}\})\ge 2\de^2$ and $|v|\le 4\de^{-\frac{1}{4}}$. Thus, if $\chi_\de(t,x,v;{T})>0$ then 
\begin{align*}\notag
	v\cdot n(x) & =v\cdot n(x-v\{t-{T}\})+v\cdot\big(n(x)-n(x-v\{t-{T}\})\big) \\
	 & \ge 2\de^2 - |v|^2(t-{T})\|\na n\|_{L^\infty}
	\ge 2\de^2 -\de^{\frac{5}{2}} C_n>\de^2,
\end{align*}
where $C_n>0$ is a constant depending only $n(x)$ given in \eqref{nW2}, and we choose $\de>0$ sufficiently small.
That is $\chi_\de =\chi_\de \1_{v\cdot n(x)>\de^2}$.
\item 
Similarly, 
if $v\cdot n(x)>4\de^2$ and $\chi_2(v)>0$, then $|v|\le4\de^{-\frac{1}{4}}$, and hence,
\begin{align*}
	v\cdot n(x-v\{t-{T}\}) & =v\cdot n(x)+v\cdot\big(n(x-v\{t-{T}\})-n(x)\big) \\
	 & \ge 4\de^2-|v|^2(t-{T})\|\na n\|_{L^\infty}>3\de^2,
\end{align*}
which implies $\chi_1\big(v\cdot n(x-v\{t-{T}\})=1$.
We next derive the support of $1_{v\cdot n(x)\ge 0}-\chi_\de$. Write
\begin{align*}
	\qquad1_{v\cdot n(x)\ge 0}-\chi_\de =1_{v\cdot n(x)\ge 0}\big(1-\chi_2(v)\big)+\Big[1-\chi_1\big(v\cdot n(x-v\{t-{T}\})\big)\Big]\chi_2(v)1_{v\cdot n(x)\ge 0}.
\end{align*}
Here, $1-\chi_2(v)=0$ for any $|v|\le2\de^{-\frac{1}{4}}$, and $\big[1-\chi_1\big(v\cdot n(x-v\{t-{T}\})\big)\big]\chi_2(v)=0$ for any $v\cdot n(x)>4\de^2$. Consequently,
\begin{align*}
	1_{v\cdot n(x)>0}-\chi_\de=0\ \ \text{ if }|v|\le2\de^{-\frac{1}{4}}\text{ and }v\cdot n(x)>4\de^2.
\end{align*}
\item Later, we also need to control the velocity derivative of $\chi_\de$. It's direct from \eqref{chideY} that
\begin{align*}
	\pa_{v_j}\chi_\de(t,x,v;{T})&=\big\{n_j(x-v\{t-{T}\})-\{t-{T}\}\sum_{i=1}^3v_i\pa_{x_j}n_i(x-v\{t-{T}\})\big\}\\
	&\quad\times(\chi_1)'\big(v\cdot n(x-v\{t-{T}\})\big)\chi_2(v)+\chi_1\big(v\cdot n(x-v\{t-{T}\})\big)\pa_{v_j}\chi_2(v),
\end{align*}
and
\begin{align*}
	&\pa_{v_jv_k}\chi_\de(t,x,v;{T})\\&=\Big\{-2\{t-{T}\}\pa_{x_k}n_j(x-v\{t-{T}\})+\{t-{T}\}^2\sum_{i=1}^3v_i\pa_{x_jx_k}n_i(x-v\{t-{T}\})\Big\}(\chi_1)'\chi_2(v)\\
	&\quad+\Big\{n_j(x-v\{t-{T}\})-\{t-{T}\}\sum_{i=1}^3v_i\pa_{x_j}n_i(x-v\{t-{T}\})\Big\}^2(\chi_1)''\chi_2(v)\\
	&\quad+\chi_1\big(v\cdot n(x-v\{t-{T}\})\big)\pa_{v_jv_k}\chi_2(v).
\end{align*}
Therefore, for $t\in[{T},{T}+\de^3]$, by \eqref{nW2}, \eqref{nachi12} and $|v|\le 4\de^{-\frac{1}{4}}$ from the support of $\chi_2$, we have
\begin{align*}
	\|[\na_v\chi_\de,\na_v^2\chi_\de]\|_{L^\infty_v}
	\le C\de^{-4}.
\end{align*}
where $C=C(\|n\|_{W^{2,\infty}})>0$ depends only on $n$ given in \eqref{nW2} and is independent of $x,t,{T}$.
\end{itemize}
Similar estimates hold for $t\in[T-\de^3,T]$. 
We can also define the forward smooth cutoff function supported on the inflow region. Thus, we let 
\begin{align}\label{chide}
	\begin{cases}
		\chi^{+}_{\de}(t,x,v;T)&=\chi_\de(t,x,v;{T})=\chi_1\big(v\cdot n(x-v\{t-{T}\})\big)\chi_2(v)\ge 0,\\
		\chi^{-}_{\de}(t,x,v;T)&=\chi_1\big(-v\cdot n(x-v\{t-{T}\})\big)\chi_2(v)\ge 0. 
	\end{cases}
\end{align}
In summary, we have
\begin{Lem}
	Let $t\in[T-\de^3,{T}+\de^3]$ and denote $\chi^{\pm}_{\de}$ as in \eqref{chide}. Then $\pa_t\chi^{\pm}_{\de}+v\cdot\na_x\chi^{\pm}_{\de}=0$ and 
\begin{align}
	\label{chide01}
	\begin{cases}
	\pm v\cdot n(x)>\de^2,&\text{ if $\chi^{\pm}_{\de}(t,x,v;{T})>0$}\\
	1_{\pm v\cdot n(x)>0}-\chi^{\pm}_{\de}=0, &\text{ if }|v|\le2\de^{-\frac{1}{4}}\text{ and }\pm v\cdot n(x)>4\de^2.	
	\end{cases}
\end{align}
Moreover, 
\begin{align}\label{chideesti}
	\|[\na_v\chi^{\pm}_{\de},\na_v^2\chi^{\pm}_{\de}]\|_{L^\infty_v}
	\le C(\|n\|_{W^{2,\infty}})\de^{-4}.
\end{align}
\end{Lem}
We are now ready to estimate the diffuse boundary
\begin{align*}
	\int_{\Si_-}|v\cdot n(x)||R_Df(v)|^2\,dvdS(x).
\end{align*}
For the part $1-\chi^{\pm}_{\de}$, we find its smallness from $\mu^{\frac{1}{2}}$ and integration, while for the part $\chi^{\pm}_{\de}$, we provide control it by the interior energy; cf. \cite{Ukai1986,Guo2009,Guo2021b}.

\begin{Lem}\label{diffboundLem}
	Let $T>T_1>0$ and $\chi^{\pm}_{\de}(t,x,v;{T})$ be a smooth cutoff function given in \eqref{chide} with sufficiently small $\de>0$.
	Assume that $f$ is any suitable function.
\begin{itemize}[leftmargin=2em]
	\item On the grazing part, we only consider the outflow region and $1-\chi^{+}_{\de}$. For any $s\in[T,T+\de^3]$, we have 
	\begin{multline}\label{diffbound1}
		\int_{T}^{s}\int_{\pa\Omega}c_\mu\Big|\int_{v\cdot n(x)>0}\{v\cdot n(x)\}(1-\chi^{+}_{\de}(t,x,v;{T}))f(v)\mu^{\frac{1}{2}}(v)\,dv\Big|^2\,dS(x)dt\\
		\le C(\de^4+e^{-\de^{-1/2}})\int_{T}^{s}\int_{\Si_+}\{v\cdot n(x)\}|f|^2\,dvdS(x)dt,
	\end{multline}
with some constant $C>0$ independent of ${T}$ and $\de$.
\item On the non-grazing part, 
 for any $s\in[T,T+\de^3]$, 
\begin{multline}\label{diffbound2j}
	\int_{T}^{s}\int_{\pa\Omega}\int_{v\cdot n(x)>0}|v\cdot n(x)|\chi^{+}_{\de}(t,x,v;{T})|f(v)|^2\,dvdS(x)dt
	\le\|f(T)\|^2_{L^2_x(\Omega)L^2_v}\\
	\quad
	+\int^s_T\int_{\Omega\times\R^3_v}\chi^{+}_{\de}\big(\pa_t|f|^2+v\cdot\na_x|f|^2\big)\,dvdxdt, 
\end{multline}
while for any $s\in[T-\de^3,T]$, 
\begin{multline}\label{diffbound2i}
	\int_{s}^{T}\int_{\pa\Omega}\int_{v\cdot n(x)<0}|v\cdot n(x)|\chi^{-}_{\de}(t,x,v;{T})|f(v)|^2\,dvdS(x)dt
	\le\|f(T)\|^2_{L^2_x(\Omega)L^2_v}\\
	\quad
	-\int_s^T\int_{\Omega\times\R^3_v}\chi^{-}_{\de}\big(\pa_t|f|^2+v\cdot\na_x|f|^2\big)\,dvdxdt. 
\end{multline}
\item For the instand energy at ``initial'' time $s$, if $\|Rf\|_{L^2_{x,v}(\Si_-)}\le \|f\|_{L^2_{x,v}(\Si_+)}$, we have
\begin{align}\label{273}
	\|f({T})\|^2_{L^2_x(\Omega)L^2_v} \le\|f(s)\|^2_{L^2_x(\Omega)L^2_v}
	+\int_{[s,T]\times\Omega\times\R^3_v}\big(\pa_t|f|^2+v\cdot\na_x|f|^2\big)\,dvdxdt.
\end{align}	
\end{itemize}	
\end{Lem}
\begin{proof}
	For the estimate \eqref{diffbound1}, we have from \eqref{chide01} and Cauchy-Schwarz inequality that for any $t\in[T,T+\de^3]$, 
	\begin{align}\label{940}\notag
		 &\int_{\pa\Omega}c_\mu\Big|\int_{v\cdot n(x)>0}\{v\cdot n(x)\}(1-\chi^{+}_{\de}(t,x,v;{T}))f(v)\mu^{\frac{1}{2}}(v)\,dv\Big|^2\,dS(x)dt \\
		 & \quad\notag\le\int_{\pa\Omega}c_\mu\Big|\int_{v\cdot n(x)>0}\big(\1_{v\cdot n(x)\le 4\de^2}+\1_{|v|>2\de^{-\frac{1}{4}}}\big)\{v\cdot n(x)\}f\mu^{\frac{1}{2}}\,dv\Big|^2\,dS(x)dt \\
		 & \quad\notag\le 2\int_{\pa\Omega}c_\mu\Big(\int_{v\cdot n(x)>0}\{v\cdot n(x)\}|f|^2\,dv\Big) \\
		 & \quad\quad\quad\times\Big(\int_{v\cdot n(x)>0}\big(\1_{v\cdot n(x)\le 4\de^2}+\1_{|v|>2\de^{-\frac{1}{4}}}\big)\{v\cdot n(x)\}\mu\,dv\Big)\,dS(x)dt.
	\end{align}
	Noticing $\1_{|v|>2\de^{-\frac{1}{4}}}\mu\le C\mu^{\frac{1}{2}}e^{-\de^{-1/2}}$, we have
	\begin{align}\label{desmall1}
		\int_{v\cdot n(x)>0}\1_{|v|>2\de^{-\frac{1}{4}}}\{v\cdot n(x)\}\mu\,dv
		\le e^{-\de^{-1/2}}C\int_{v\cdot n(x)>0}\{v\cdot n(x)\}\mu^{\frac{1}{2}}\,dv
		\le Ce^{-\de^{-1/2}}.
	\end{align}
	Also, by using rotation $v\mapsto \ti Rv$ with $\ti R^Tn=(1,0,0)$ and $|v|=|\ti Rv|$ ($\ti R$ is a orthogonal matrix and $\ti R^T$ is the transpose of $\ti R$), we have
	\begin{align}\label{desmall2}\notag
		\int_{v\cdot n(x)>0}\1_{v\cdot n(x)\le 4\de^2}\{v\cdot n(x)\}\mu\,dv
		 & =\int_{0<v_1\le 4\de^2}(2\pi)^{\frac{3}{2}}v_1e^{-\frac{|v|^2}{2}}\,dv \\
		 & =(2\pi)^{\frac{1}{2}}\big(1-e^{-8\de^4}\big)
		\le C\de^4.
	\end{align}
	Substituting the above two estimates into \eqref{940} implies \eqref{diffbound1}.

	\medskip Next, we consider estimates \eqref{diffbound2j} and \eqref{diffbound-} separately. For the part $\chi^+_\de$, let $s\in[T,T+\de^3]$, multiply $\chi^{\pm}_{\de}(t,x,v;{T})$ to 
	\begin{align}\label{diff4}
		\pa_t|f|^2+ v\cdot\na_x|f|^2=\pa_t|f|^2+ v\cdot\na_x|f|^2,
	\end{align}
	and then integrate over $[{T},s]\times\Omega\times\R^3_v$. Note that $\pa_t\chi^{\pm}_{\de}+v\cdot\na_x\chi^{\pm}_{\de}=0$. We have
\begin{multline}\label{941a}
	\int_{\Omega\times\R^3_v}\chi^{+}_{\de}|f(s)|^2\,dvdx+ \int_{[{T},s]\times\pa\Omega\times\R^3_v}v\cdot n(x)\chi^{+}_{\de}|f|^2\,dvdxdt\\
	=\int_{\Omega\times\R^3_v}\chi^{+}_{\de}|f({T})|^2\,dvdx+\int_{[{T},s]\times\Omega\times\R^3_v}\chi^{+}_{\de}\big(\pa_t|f|^2+v\cdot\na_x|f|^2\big)\,dvdxdt. 
\end{multline}
It follows from \eqref{chide01} that $\chi^{+}_{\de}(t,x,v;{T})=0$ for any $v\cdot n(x)\le \de$ and $t\in[{T},{T}+\de^3]$. Also, the first term of \eqref{941a} is non-negative, which yields \eqref{diffbound2j}. Similarly, for the part $\chi^-_\de$, letting $s\in[T-\de^3,T]$ and integrating \eqref{diff4} on $[s,T]$ instead, we have 
\begin{multline*}
	\int_{\Omega\times\R^3_v}\chi^{-}_{\de}|f(T)|^2\,dvdx+\int_{[s,T]\times\pa\Omega\times\R^3_v}v\cdot n(x)\chi^{-}_{\de}|f|^2\,dvdxdt\\
	=\int_{\Omega\times\R^3_v}\chi^{-}_{\de}|f({s})|^2\,dvdx+\int_{[s,T]\times\Omega\times\R^3_v}\chi^{-}_{\de}\big(\pa_t|f|^2+v\cdot\na_x|f|^2\big)\,dvdxdt.
\end{multline*}
Also, it follows from \eqref{chide01} that $\chi^{-}_{\de}(t,x,v;{T})=0$ for any $v\cdot n(x)\ge -\de$ and $t\in[T-\de^3,T]$. Then we obtain 
\begin{multline}\label{diffbound-}
	\int^T_s\int_{\pa\Omega}\int_{v\cdot n(x)<0}|v\cdot n(x)|\chi^{-}_{\de}|f|^2\,dvdxdt\le\int_{\Omega\times\R^3_v}\chi^{-}_{\de}|f(T)|^2\,dvdxdt\\
	-\int_{[s,T]\times\Omega\times\R^3_v}\chi^{-}_{\de}\big(\pa_t|f|^2+v\cdot\na_x|f|^2\big)\,dvdxdt.
\end{multline}
For the terms involving ``initial'' value at $s$, since we assume $\|Rf\|_{L^2_{x,v}(\Si_-)}\le \|f\|_{L^2_{x,v}(\Si_+)}$, by integrating \eqref{diff4} over $[s,T]\times\Omega\times\R^3_v$, we obtain \eqref{273}. 
This reduces arbitrary time $T$ to a fixed time $s\le T$. 
This completes the proof of Lemma \ref{diffboundLem}.
\end{proof}

\subsection{The new trace lemma for level functions}\label{SecNewTrace}
In this Section, we will derive the $L^2$ existence, boundary estimates, and some collision estimates for the linear Boltzmann equation in $\Omega$ with \emph{Maxwell} reflection boundary conditions.
To do this, let $\al>0$ and denote the reflection operator $R$ by
\begin{align}\label{reflectR}
	Rf(x,v)=(1-\al)f(x,R_L(x)v)+\al R_Df(x,v),
\end{align}
for any $(x,v)\in\Si_-$,
where $R_L(x)$ is the local reflection operator given by \eqref{RL1}
and $R_D(x)$ is the diffuse reflection operator given by \eqref{RD1}: 
\begin{align*}
	R_Df(v)=c_\mu\mu^{\frac{1}{2}}(v)\int_{v'\cdot n(x)>0}\{v'\cdot n(x)\}f(v')\mu^{\frac{1}{2}}(v')\,dv'.
\end{align*}
Here, we consider weight $\<v\>^{l}_\de$ with $\de\in(0,1)$ given by \eqref{weight} and the level functions
\begin{align}\label{levelmuii}
	f^{(l)}_K :=f-K\<v\>^{-l}_\de,\quad f^{(l)}_{K,+}=f^{(l)}_K\1_{f^{(l)}_K\ge 0}. 
\end{align}
We also have \eqref{fKM}: for any $K>M\ge 0$ and $p>q>0$, 
\begin{align*}
	(f^{(l)}_{K,+})^q\le \frac{(f^{(l)}_{K,+})^q(f^{(l)}_{M,+})^{p-q}}{(K\<v\>^{-l}_\de-M\<v\>^{-l}_\de)^{p-q}}
	\le \frac{C\<v\>^{l(p-q)/2}(f^{(l)}_{M,+})^{p}}{(K-M)^{p-q}},
\end{align*}
for some $C=C_{\|n\|_{L^\infty}}>0$. 

\begin{Lem}\label{LemR}
Let $T>0$, $p\ge 2$ be integer, $R$ be the reflection operator given in \eqref{reflectR} and denote the boundary norm $L^2(\Si_\pm)$ by \eqref{boundarySi}. For any suitable function $f$, we have 
	\begin{itemize}[leftmargin=2em]
		\item The standard $L^2$ diffuse-type boundary estimate:
		\begin{align}\label{101}
			\|Rf\|^2_{L^2_{x,v}(\Si_-)}
			&=\|f\|^2_{L^2_{x,v}(\Si_+)}-\al\|f-R_Df\|^2_{L^2_{x,v}(\Si_+)}.
		\end{align}
\item The $L^p$ diffuse-type boundary estimate:
\begin{align*}
	\|Rf\|_{L^p(\Si_-)}
	&\le (1-\beta)\|f\|_{L^p(\Si_+)}+\beta\|R_Df\|_{L^p(\Si_+)},
\end{align*}
for some $\beta=\beta(\al)\in(0,1]$. 
\item The polynomial-weight estimate:
\begin{align}\label{101w}
	\|\<v\>^{k}Rf\|^2_{L^2_{x,v}(\Si_-)}\le(1-\al)^2\|\<v\>^{k}f\|^2_{L^2_{x,v}(\Si_+)} + C_k\|f\|^2_{L^2_{x,v}(\Si_+)}.
\end{align}
The local-reflection part is preserved, while the diffuse-reflection part is controlled by an upper bound. 
\end{itemize}
\begin{itemize}[leftmargin=2em]
\item The level-function estimate. Denote cutoff $\chi^-_\de$ as in \eqref{chide}. If we choose a sufficiently small $\de=\de(l,\|n\|_{L^\infty})\in(0,1)$, then for any $K\ge 0$, 
\begin{align}\label{101a}
	\|(Rf)^{(l)}_{K,+}\|^2_{L^2_{x,v}(\Si_-)}&\le 
	(1-\al)\|f^{(l)}_{K,+}\|^2_{L^2_{x,v}(\Si_+)}
	+\al\|(R_Df)^{(l)}_{K,+}\|^2_{L^2_{x,v}(\Si_-)},
\end{align}
where, for any $s\in [T-\de^3,T]$, 
\begin{align}\label{533}\notag
	\|(R_Df)^{(l)}_{K,+}\|^2_{L^2_t(s,T)L^2_{x,v}(\Si_-)}
	&\le 
	C\de^2\|f^{(l)}_{K,+}\|_{L^2_t(s,T)L^2(\Si_-)}
	+\|f^{(l)}_{K,+}(T)\|^2_{L^2_x(\Omega)L^2_v}\\
	&\quad
	-\int_s^T\int_{\Omega\times\R^3_v}\chi^{-}_{\de}\big(\pa_t|f^{(l)}_{K,+}|^2+v\cdot\na_x|f^{(l)}_{K,+}|^2\big)\,dvdxdt. 
\end{align}
Similar estimates are valid for $(-f)^{(l)}_{K,+}$ instead of $f^{(l)}_{K,+}$.
\end{itemize}
\end{Lem}
If exponential decay is used in \eqref{levelmuii}, a better estimate can be derived. However, due to the lack of an exponent-weighted $L^2$ estimate for the collision term, we will use polynomial decay as in \eqref{levelmuii} in this work. 
\begin{proof}
{\bf Estimate of function $f$.}
By \eqref{reflectR}, for any even $p\ge 2$, we write $\|Rf\|^p_{L^p(\Si_-)}$ as 
\begin{align*}
	\|Rf\|^p_{L^p(\Si_-)}
	&=\int_{\Si_-}|v\cdot n|\sum_{k=0}^p{\binom {p}{k}}\big((1-\al)f(x,R_L(x)v)\big)^k\big(\al R_Df(x,v)\big)^{p-k}\,dS(x)dv. 
\end{align*}
By change of variable $v\mapsto R_L(x)v\,:\,\Si_-\to \Si_+$, which preserves $|v\cdot n|$, $|v|$ and hence $R_D(x)f$, 
\begin{align*}\notag
	\|Rf\|^p_{L^p(\Si_-)}
	&=\int_{\Si_+}|v\cdot n|\sum_{k=0}^p{\binom {p}{k}}\big((1-\al)f(x,v)\big)^k\big(\al R_Df(x,v)\big)^{p-k}\,dS(x)dv\\
	&=\|(1-\al)f+\al R_Df\|^p_{L^p(\Si_+)}\le \big((1-\al)\|f\|_{L^p(\Si_+)}+\al\|R_Df\|_{L^p(\Si_+)}\big)^p. 
\end{align*}
By binomial expansion and Young's inequality, 
\begin{align*}
	\|Rf\|^p_{L^p(\Si_-)}&\le \sum_{k=0}^p{\binom {p}{k}}\big((1-\al)\|f\|_{L^p(\Si_+)}\big)^k\big(\al\|R_Df\|_{L^p(\Si_+)}\big)^{p-k}\\
	&\le \sum_{k=0}^p{\binom {p}{k}}(1-\al)^k\al^{p-k}\Big(\frac{k\|f\|_{L^p(\Si_+)}^p}{p}+\frac{(p-k)\|R_Df\|_{L^p(\Si_+)}^p}{p}\Big)\\
	&=:(1-\beta)\|f\|_{L^p(\Si_+)}^p+\beta\|R_Df\|_{L^p(\Si_+)}^p, 
\end{align*}
where 
\begin{align*}
	\sum_{k=0}^p{\binom {p}{k}}(1-\al)^k\al^{p-k}\big(\frac{k}{p}+\frac{p-k}{p}\big)=1. 
\end{align*}
When $p=2$, we have $\beta=\al$,
\begin{align*}
	\|Rf\|^2_{L^2_{x,v}(\Si_-)}
	&\notag=(1-\al)\|f\|^2_{L^2_{x,v}(\Si_+)}
	+\al\|R_D f\|^2_{L^2_{x,v}(\Si_+)}, 
\end{align*}
 and 
\begin{align*}\notag
	\|R_D f\|^2_{L^2_{x,v}(\Si_+)}&=\int_{\pa\Omega}c_\mu\Big|\int_{v'\cdot n(x)>0}\{v'\cdot n(x)\}f(v')\mu^{\frac{1}{2}}(v')\,dv'\Big|^2\,dS(x)\\
	&=\int_{\Si_+}|v\cdot n|fR_Df\,dS(x)dv, 
\end{align*}
Consequently, 
\begin{align*}
	\|f\|^2_{L^2_{x,v}(\Si_+)}-\|Rf\|^2_{L^2_{x,v}(\Si_-)}
	&\notag=\al\big(\|f\|^2_{L^2_{x,v}(\Si_+)}-\|R_D f\|^2_{L^2_{x,v}(\Si_+)}\big)\\
	&\notag=\al\|f-R_Df\|^2_{L^2_{x,v}(\Si_+)}.
\end{align*}
This implies \eqref{101}. 
For the weighted estimate, by \eqref{reflectR} and change of variable $v\mapsto R_L(x)v$, 
\begin{align*}
	&\int_{\Si_-}|v\cdot n|\<v\>^{2k}|Rf|^2\,dS(x)dv
	=(1-\al)^2\int_{\Si_+}|v\cdot n|\<v\>^{2k}|f(v)|^2\,dS(x)dv\\
		&\notag\quad+2(1-\al)\al\int_{\Si_+}|v\cdot n| f(v)c_\mu\<v\>^{2k}\mu^{\frac{1}{2}}(v)\int_{v'\cdot n(x)>0}\{v'\cdot n(x)\}f(v')\mu^{\frac{1}{2}}(v')\,dv'dS(x)\\
		&\notag\quad
		+\al^2\int_{\si_+}(c_\mu)^2|v\cdot n|\<v\>^{2k}\mu(v)\Big|\int_{v'\cdot n(x)>0}\{v'\cdot n(x)\}f(v')\mu^{\frac{1}{2}}(v')\,dv'\Big|^2\,dS(x)\\
		&\le (1-\al)^2\|\<v\>^{k}f\|^2_{L^2_{x,v}(\Si_+)} + C_k\|f\|^2_{L^2_{x,v}(\Si_+)}.
\end{align*}

\smallskip\noindent
{\bf Estimate of level-function $f^{(l)}_{K,+}$.}
By change of variable $v\mapsto R_L(x)v\,:\,\Si_-\to \Si_+$, one has
\begin{align}\label{299}\notag
	&\int_{\Si_-}|v\cdot n||(Rf)^{(l)}_{K,+}|^2\,dS(x)dv\\
	&\notag\quad\le\int_{\Si_-}|v\cdot n|\big|(1-\al)(f(R_L(x)v)-K\<v\>^{-l})_++\al (R_Df(v)-K\<v\>^{-l})_+\big|^2\,dS(x)dv\\ 
	&\quad\le(1-\al)\int_{\Si_+}|v\cdot n||f^{(l)}_{K,+}|^2\,dS(x)dv
	+\al\int_{\Si_-}|v\cdot n|\big|(R_Df(v)-K\<v\>^{-l})_+\big|^2\,dS(x)dv.
\end{align}
To estimate $\|f^{(l)}_{K,+}\|_{L^2(\Si_-)}$, we will apply the trace lemma for the non-grazing part and utilize the designed weight $\<v\>^{l}_\de$ for the grazing part. Let $\chi^-_\de(t,x,v;T)$ be the smooth cutoff function given by \eqref{chide}. Then the non-grazing (and small velocity) part can be estimated by \eqref{diffbound2i}: for any $T\in[T_1,T_2]$, $\de>0$, and $s\in[T-\de^3,T]$, 
\begin{multline}\label{diffbound2ixx}
	\|(\chi^-_\de)^{\frac{1}{2}}(R_Df)^{(l)}_{K,+}\|_{L^2_t(s,T)L^2(\Si_-)}^2
	\le\|f^{(l)}_{K,+}(T)\|^2_{L^2_x(\Omega)L^2_v}\\
	\quad
	-\int_s^T\int_{\Omega\times\R^3_v}\chi^{-}_{\de}\big(\pa_t|f^{(l)}_{K,+}|^2+v\cdot\na_x|f^{(l)}_{K,+}|^2\big)\,dvdxdt. 
\end{multline}
For the grazing part or large velocity part, i.e. on the support of $1-\chi^-_\de(t,x,v;{T})$, we have from \eqref{chide01} that for any $(t,x,v)\in[T-\de^3,T]\times\R^3_x\times\R^3_v$, 
\begin{align}\label{eq511}
	\1_{v\cdot n(x)<0}(1-\chi^-_\de(t,x,v;T)) \le \1_{v\cdot n(x)<0}\big(\1_{|v|>2\de^{-\frac{1}{4}}}+\1_{v\cdot n(x)>-4\de^2}\big). 
\end{align}
Thus, by \eqref{RD1} and \eqref{levelmuii}, for any $s\in[T-\de^3,T]$, we have 
\begin{align}\label{eq512}\notag
	&\|(1-\chi^-_\de)^{\frac{1}{2}}(R_Df)^{(l)}_{K,+}\|_{L^2_t(s,T)L^2(\Si_-)}\\
	&\le 
	\Big\|\Big((1-\chi^-_\de)^{\frac{1}{2}}c_\mu\mu^{\frac{1}{2}}(v)\int_{v'\cdot n>0}\{v'\cdot n\}\big(f^{(l)}_{K,+}(v')+K\<v'\>^{-l}_\de\big)\mu^{\frac{1}{2}}(v')\,dv'-K\<v\>^{-l}_\de\Big)_+\Big\|_{L^2_t(s,T)L^2(\Si_-)}. 
\end{align}
For the $f^{(l)}_{K,+}$ part, we apply \eqref{desmall1} and \eqref{desmall2} to deduce 
\begin{align}\label{eq513a}\notag
	&\Big\|(1-\chi^-_\de)^{\frac{1}{2}}c_\mu\mu^{\frac{1}{2}}(v)\int_{v'\cdot n>0}\{v'\cdot n\}f^{(l)}_{K,+}(v')\mu^{\frac{1}{2}}(v')\,dv'\Big\|_{L^2_t(s,T)L^2(\Si_-)}\\
	&\quad\notag\le \Big\|(1-\chi^-_\de)^{\frac{1}{2}}\mu^{\frac{1}{2}}(v)\int_{v'\cdot n>0}\{v'\cdot n\}\big(f^{(l)}_{K,+}(v')\big)^2\,dv'\Big\|_{L^2_t(s,T)L^2(\Si_-)}\\
	&\quad\le C\de^2\|f^{(l)}_{K,+}\|_{L^2_t(s,T)L^2(\Si_+)}. 
\end{align}
By \eqref{eq511}, for the weight function $\<v\>^{-l}_\de$ part, it remains to show that 
\begin{align}\label{eq513}
	&\Big\|\big(\1_{|v|>2\de^{-\frac{1}{4}}}+\1_{v\cdot n(x)>-4\de^2}\big)\Big(c_\mu\mu^{\frac{1}{2}}(v)\int_{v'\cdot n>0}\{v'\cdot n\}\<v'\>^{-l}_\de\mu^{\frac{1}{2}}(v')\,dv'-\<v\>^{-l}_\de\Big)_+\Big\|_{L^2_t(s,T)L^2(\Si_-)}=0. 
\end{align}
If $|v|>2\de^{-\frac{1}{4}}$, then by \eqref{weightvde}, i.e. $\<v\>^{-l}_{\de}=\frac{\<v\>^{-l}}{(\de^2+\<v\>^{-2}(v\cdot n(x))^2\chi_{|v\cdot n(x)|\le 2\de^{-1/4}})^{1/2}}$, 
\begin{align*}
	&c_\mu\mu^{\frac{1}{2}}(v)\int_{v'\cdot n>0}\{v'\cdot n\}\<v'\>^{-l}_\de\mu^{\frac{1}{2}}(v')\,dv'-\<v\>^{-l}_\de\\
	&\quad\le c_\mu\mu^{\frac{1}{2}}(v)\int_{v'\cdot n>0}\{v'\cdot n\}\frac{\<v'\>^{-l}}{\de}\mu^{\frac{1}{2}}(v')\,dv'-\frac{\<v\>^{-l}}{(1+\|n\|_{L^\infty_x}^2)^{1/2}}\\
	&\quad\le Ce^{-\frac{|v|^{2}}{4}}\de^{-1}-C_{\|n\|_{L^\infty_x}}\<v\>^{-l}<0,
\end{align*}
if we choose a sufficiently small $\de=\de(l,\|n\|_{L^\infty_x}^2)>0$. 
On the other hand, if $0>v\cdot n(x)>-4\de^2$, recalling \eqref{chivn} that $\1_{|v\cdot n(x)|\le \de^{-\frac{1}{4}}}\le \chi_{|v\cdot n(x)|\le 2\de^{-\frac{1}{4}}}\le\1_{|v\cdot n(x)|\le 2\de^{-\frac{1}{4}}}$, we have 
\begin{align*}
	&c_\mu\mu^{\frac{1}{2}}(v)\int_{v'\cdot n>0}\{v'\cdot n\}\<v'\>^{-l}_\de\mu^{\frac{1}{2}}(v')\,dv'-\<v\>^{-l}_\de\\
	&\quad\le c_\mu\mu^{\frac{1}{2}}(v)\int_{v'\cdot n>0}\{v'\cdot n\}\frac{\big(\1_{|v'\cdot n(x)|\le \de^{-\frac{1}{4}}}+\1_{|v'\cdot n(x)|>\de^{-\frac{1}{4}}}\big)\<v'\>^{-l}}{(\de^2+\<v'\>^{-2}(v'\cdot n(x))^2\chi_{|v'\cdot n(x)|\le 2\de^{-1/4}})^{1/2}}\mu^{\frac{1}{2}}(v')\,dv'-\<v\>^{-l}_\de\\
	&\quad\le c_\mu\mu^{\frac{1}{2}}(v)\int_{v'\cdot n>0}{\1_{|v'\cdot n(x)|\le \de^{-\frac{1}{4}}}\<v'\>^{-l+1}}\mu^{\frac{1}{2}}(v')\,dv'
	\\&\qquad+c_\mu\mu^{\frac{1}{2}}(v)\int_{v'\cdot n>0}\{v'\cdot n\}\1_{|v'|>\de^{-\frac{1}{4}}}\<v'\>^{-l}\de^{-1}(2\pi)^{\frac{3}{8}}e^{-\frac{|\de|^{-1/2}}{8}}\mu^{\frac{1}{4}}(v')\,dv'
	-\frac{\<v\>^{-l}}{(\de^2+16\de^4\<v\>^{-2})^{1/2}}\\
	&\quad\le C\mu^{\frac{1}{2}}(v)+C\mu^{\frac{1}{2}}(v)e^{-\frac{|\de|^{-1/2}}{16}}-\frac{2\<v\>^{-l}}{\de}<0,
\end{align*}
if we further choose $\de=\de(l,\|n\|_{L^\infty_x}^2)>0$ sufficiently small. 
In the second inequality, we used the fact that $|v'|\ge |v'\cdot n(x)|>\de^{-\frac{1}{4}}$ on the boundary $\Si_-$ and the support of $\1_{|v'\cdot n(x)|>\de^{-\frac{1}{4}}}$. The above two cases imply \eqref{eq513}. Thus, combining \eqref{eq512} and \eqref{eq513a}, we obtain
\begin{align}\label{eq512a}
	&\|(1-\chi^-_\de)^{\frac{1}{2}}(R_Df)^{(l)}_{K,+}\|_{L^2_t(s,T)L^2(\Si_-)}
	\le C\de^2\|f^{(l)}_{K,+}\|_{L^2_t(s,T)L^2(\Si_-)}, 
\end{align}
if $\de=\de(l,\|n\|_{L^\infty_x}^2)>0$ is sufficiently small. 
Combining \eqref{299}, \eqref{diffbound2ixx} and \eqref{eq512a} yields \eqref{101a}. 
By the same calculations, similar estimates are valid for $(-f)^{(l)}_{K,+}$ instead of $f^{(l)}_{K,+}$. 
	This concludes the Lemma \ref{LemR}.
\end{proof}

\section{Basic \texorpdfstring{$L^2$}{L2} estimates}\label{Sec3}
In this Section, we will prepare some estimates of the collision terms and the level functions, which will be frequently used until the end of this paper.

\subsection{Preparation and regular change of variables}
We begin with a Lemma for the convergence of integrals over $b(\cos\th)$.
\begin{Lem}\label{bcossinLem}
	Assume $b(\cos\th)$ satisfies \eqref{ths}. Then
	\begin{align}\label{bcossin}
		\begin{aligned}
			\int_{\S^2}b(\cos\th)\sin^2\frac{\th}{2}\,d\sigma & \approx C_s, \\
			\int_{\S^2}b(\cos\th)\min\big\{\sin^2\frac{\th}{2}|v-v_*|^2,1\big\}\,d\sigma & \le C_s|v-v_*|^{2s}.
		\end{aligned}
	\end{align}
	For $k\ge 2$,
	\begin{align}\label{bcossink}
		\int_{\S^2}b(\cos\th)\Big(1-\cos^k\frac{\th}{2}\Big)\,d\sigma\approx C_{k,s},
	\end{align}
	and consequently, for $l\in\R$,
	\begin{align}
		\label{bcossink2}
		\Big|\int_{\S^2}b(\cos\th)\Big(1-\cos^l\frac{\th}{2}\Big)\,d\sigma\Big|\le C_{l,s}.
	\end{align}
\end{Lem}
\begin{proof}
	The first estimate of \eqref{bcossin} follows from \eqref{ths}.
	The second estimate of \eqref{bcossin} can be estimated as
	\begin{multline*}
		\int_{\S^2}b(\cos\th)\min\big\{|v-v_*|\sin\frac{\th}{2},1\big\}\,d\sigma
		\lesssim \int_{0}^{\frac{\pi}{2}}\th^{-1-2s}\min\big\{\sin^2\frac{\th}{2}|v-v_*|^2,1\big\}\,d\th\\
		\lesssim\int_{0}^{\min\{\frac{\pi}{2},|v-v_*|^{-1}\}}\th^{1-2s}|v-v_*|^2\,d\th+\int_{\min\{\frac{\pi}{2},|v-v_*|^{-1}\}}^{\frac{\pi}{2}}\th^{-1-2s}\,d\th
		\le C_s|v-v_*|^{2s}.
	\end{multline*}
	For \eqref{bcossink}, one may refer to \cite[Lemma 2.2]{Cao2022}, and we write its proof for the sake of completeness.
	By \eqref{ths}, change of variable $u=\sin\frac{\th}{2}$ with $du=\frac{1}{2}\cos\frac{\th}{2}d\th$, we have
	\begin{align*}
		\int_{\S^2}b(\cos\th)\Big(1-\cos^k\frac{\th}{2}\Big)\,d\sigma
		& =\int_0^{\pi/2}b(\cos\th)\Big(1-\big(1-\sin^2\frac{\th}{2}\big)^\frac{k}{2}\Big)\sin\th\,d\th \\
		& \approx\int_0^{1/\sqrt{2}}u^{-1-2s}\Big(1-(1-u^2)^\frac{k}{2}\Big)\,du \\
		& =:g(k).
	\end{align*}
	To obtain the estimate of $g(k)$, we take derivative with respect to $k$ to deduce
	\begin{align*}
		g'(k) & = -\frac{1}{2}\int_0^{1/\sqrt{2}}u^{-1-2s}(1-u^2)^{\frac{k}{2}}\ln(1-u^2)\,du \\
		& \approx \int_0^{1/\sqrt{2}}u^{1-2s}(1-u^2)^{\frac{k}{2}}\,du
		=C_{s,k}>0,
	\end{align*}
	where $C_{s,k}>0$ is a constant depending on $s,k$, and we have used $\ln(1-u^2)\approx -u^2$ near $u=0^+$.
	Thus, noticing $g(2)=\int_0^{1/\sqrt{2}}u^{1-2s}\,du=C_s>0$, we have
	\begin{align*}
		g(k)=g(2)+\int^k_2g'(l)\,dl=C_{k,s}>0.
	\end{align*}
	This implies \eqref{bcossink}.
	The estimate \eqref{bcossink2} for $l\ge 2$ follows from \eqref{bcossink}. For $l<2$, we have
	\begin{multline*}
		\Big|\int_{\S^2}b(\cos\th)\Big(1-\cos^l\frac{\th}{2}\Big)\,d\sigma\Big|
		\le \Big|\int_{\S^2}b(\cos\th)\Big(1-\cos^{4}\frac{\th}{2}\Big)\,d\sigma\Big|\\
		+ \Big|\int_{\S^2}b(\cos\th)\cos^{l}\frac{\th}{2}\Big(\cos^{4-l}\frac{\th}{2}-1\Big)\,d\sigma\Big|
		\le C_{l,s}.
	\end{multline*}
	This completes the proof of Lemma \ref{bcossinLem}.
\end{proof}

The following Lemma gives a method to overcome the non-integrability of $b(\cos\th)$.
\begin{Lem}\label{HHbcosLem}
	Suppose $H\in W^{2,\infty}_{loc}(\R^3)$, i.e. $\sup_{|v|\le R}|\na_v^kH(v)|\le C_R$ for any $R>0$ and $k=0,1,2$. Then for any $s\in(0,1)$ given in \eqref{ths}, we have
	\begin{multline}
		\label{HHbcos}
		\Big|\int_{\S^2}(H(v')-H(v))b(\cos\th)\,d\sigma\Big|\le C_s\Big(\sup_{|u|\le|v|+|v_*|}|H(u)||v-v_*|^{2s}\\
		+|\na_vH(v)|\big(|v_*-v|^{2s-1}\1_{|v-v_*|\ge\frac{2}{\pi}}+|v-v_*|\1_{|v-v_*|<\frac{2}{\pi}}\big)\\
		+\sup_{|u|\le|v|+|v_*|}|\na^2 H(u)||v-v_*|^{2s}\Big),
	\end{multline}
	for some constant $C_s>0$ depending only on $s$.
\end{Lem}
\begin{proof}
	By Taylor expansion, we have
	\begin{multline}\label{Taylor}
		H(v')-H(v)
		=\pa_t(H(v+t(v'-v)))|_{t=0}+\frac{1}{2}\int^1_0\pa_{tt}(H(v+t(v'-v)))\,dt\\
		=\na_vH(v)\cdot(v'-v)+\frac{1}{2}\int^1_0\na^2_vH(v+t(v'-v)):(v'-v)\otimes(v'-v)\,dt.
	\end{multline}
	Therefore, by decomposition $\1_{\th\in[0,\frac{\pi}{2}]}=\1_{\min\{\frac{\pi}{2},|v-v_*|^{-1}\}\le\th\le\frac{\pi}{2}}+\1_{0\le\th\le\min\{\frac{\pi}{2},|v-v_*|^{-1}\}}$, one has
	\begin{multline*}
		\Big|\int_{\S^2}(H(v')-H(v))b(\cos\th)\,d\sigma\Big|
		\le \Big|\int_{\S^2}\1_{\min\{\frac{\pi}{2},|v-v_*|^{-1}\}\le\th\le\frac{\pi}{2}}(H(v')-H(v))b(\cos\th)\,d\sigma\Big|\\
		+\Big|\int_{\S^2}\1_{0\le\th\le\min\{\frac{\pi}{2},|v-v_*|^{-1}\}}b(\cos\th)\na_vH(v)\cdot(v'-v)\,d\sigma\Big|\\
		+\frac{1}{2}\Big|\int_{\S^2}\int^1_0\1_{0\le\th\le\min\{\frac{\pi}{2},|v-v_*|^{-1}\}}\na^2_vH(v+t(v'-v)):(v'-v)\otimes(v'-v)\,dtd\sigma\Big|
		=: I_1+I_2+I_3.
	\end{multline*}
	For the term $I_1$, we have from \eqref{ths} that
	\begin{align*}
		I_1 & \le \Big|\int_{\S^2}\1_{\min\{\frac{\pi}{2},|v-v_*|^{-1}\}\le\th\le\frac{\pi}{2}}(H(v')-H(v))b(\cos\th)\,d\sigma\Big| \\
		& \notag\le C\sup_{|u|\le|v|+|v_*|}|H(u)|\int_{\min\{\frac{\pi}{2},|v-v_*|^{-1}\}}^{\frac{\pi}{2}}\th^{-1-2s}\,d\th \\
		& \le C\sup_{|u|\le|v|+|v_*|}|H(u)||v-v_*|^{2s}.
	\end{align*}
	For the term $I_2$, we apply \eqref{vprimeth} to deduce
	\begin{align}\label{831}
		I_2&\notag=\Big|\int_{\S^2}\1_{0\le\th\le\min\{\frac{\pi}{2},|v-v_*|^{-1}\}}\na_vH(v)\cdot(v_*-v)\sin^2\frac{\th}{2}b(\cos\th)\,d\sigma\\
		&\quad+\frac{1}{2}\int_{\S^2}\1_{0\le\th\le\min\{\frac{\pi}{2},|v-v_*|^{-1}\}}|v-v_*|\na_vH(v)\cdot\omega \sin\th b(\cos\th)\,d\sigma\Big|.
	\end{align}
	Choosing $\mathbf{k}$ given in \eqref{bfk} as the north pole, we write $\omega=(\cos\phi,\sin\phi,0)$ with $\phi\in[0,2\pi]$, then the second right-hand term of \eqref{831} vanishes by symmetry about $\phi$. Thus, by \eqref{bcossin}, we have
	\begin{align*}
		I_2 & \le C_s|\na_vH(v)|\int^{\min\{\frac{\pi}{2},|v-v_*|^{-1}\}}_0\th^{1-2s}|v-v_*|\,d\th \\
		& \le C_s|\na_vH(v)|\big(|v_*-v|^{2s-1}\1_{|v-v_*|\ge\frac{2}{\pi}}+|v-v_*|\1_{|v-v_*|<\frac{2}{\pi}}\big).
	\end{align*}
	For the term $I_3$, we apply \eqref{vpriminv} and \eqref{bcossin} to deduce
	\begin{align*}
		I_3 & \le C\sup_{|u|\le|v|+|v_*|}|\na^2 H(u)||v-v_*|^2\int_{\S^2}\1_{0\le\th\le\min\{\frac{\pi}{2},|v-v_*|^{-1}\}}\sin^2\frac{\th}{2}b(\cos\th)\,dtd\sigma \\
		& \le C_s\sup_{|u|\le|v|+|v_*|}|\na^2 H(u)||v-v_*|^{2s}.
	\end{align*}
	Combining the above estimates of $I_j$'s, we obtain \eqref{HHbcos} and conclude Lemma \ref{HHbcosLem}.
	
\end{proof}

The following Lemma is needed to estimate the collision terms.
\begin{Lem}\label{Lem21}
	For $\ga>-\frac{3}{2}$, we have
	\begin{align}
		\label{gafg}
		\int_{\R^6}|v-v_*|^{\ga}|f(v_*)g(v)|\,dv_*dv\le \|\<v\>^{2+{\ga}_+}f\|_{L^2_v}\|\<v\>^{{\ga}_+}g\|_{L^1_v},
	\end{align}
	and for any $a>0$,
	\begin{multline}
		\label{Gammaes1}
		\int_{\R^6}|v-v_*|^\ga |f(v)|\mu^{a}(v_*)|g(v_*)|\,dv_*dv\\
		\le C_a\min\big\{\|\<v\>^{-l}g\|_{L^2_v}\|\<v\>^{2\ga}f\|_{L^1_v},\|\<v\>^{-l}g\|_{L^\infty_v}\|\<v\>^{\ga}f\|_{L^1_v}\big\},
	\end{multline}
	where ${\ga}_+=\max\{0,\,\ga\}$.
\end{Lem}
\begin{proof}
	When $\ga\ge 0$, it's direct from $|v-v_*|^{\ga}\le \<v\>^{\ga}\<v_*\>^{\ga}$ that
	\begin{align*}
		\int_{\R^6}|v-v_*|^{\ga}|f(v_*)g(v)|\,dv_*dv\le\|\<v\>^{\ga}f\|_{L^1_v}\|\<v\>^{\ga}g\|_{L^1_v}\le\|\<v\>^{2+\ga}f\|_{L^2_v}\|\<v\>^{\ga}g\|_{L^1_v}.
	\end{align*}
	When $-\frac{3}{2}<\ga<0$, similar to \cite[Lemma 2.5]{Cao2022a}, we write
	\begin{align*}
		\int_{\R^3}|v-v_*|^{\ga}|f(v_*)|\,dv_*=\int_{|v-v_*|\le A}|v-v_*|^{\ga}|f(v_*)|\,dv_*+\int_{|v-v_*|>A}|v-v_*|^{\ga}|f(v_*)|\,dv_*,
	\end{align*}
	for some constant $A>0$.
	It's direct to calculate
	\begin{align*}
		\int_{|v-v_*|>A}|v-v_*|^{\ga}|f(v_*)|\,dv_*\le A^{\ga}\|f\|_{L^1_v},
	\end{align*}
	and
	\begin{align*}
		\int_{|v-v_*|\le A}|v-v_*|^{\ga}|f(v_*)|\,dv_*\le \Big(\int_{|v-v_*|\le A}|v-v_*|^{2\ga}\,dv_*\Big)^{\frac{1}{2}}\|f\|_{L^2_v}
		\le A^{\frac{3}{2}+\ga}\|f\|_{L^2_v}.
	\end{align*}
	Taking $A=\|f\|^{\frac{2}{3}}_{L^1_v}\|f\|^{-\frac{2}{3}}_{L^2_v}$, we have
	\begin{align*}
		\int_{\R^3}|v-v_*|^{\ga}|f(v_*)|\,dv_*\le \|f\|^{1+\frac{2\ga}{3}}_{L^1_v}\|f\|^{-\frac{2\ga}{3}}_{L^2_v},
	\end{align*}
	which implies
	\begin{align*}
		\int_{\R^6}|v-v_*|^{\ga}|f(v_*)g(v)|\,dv_*dv\le \|f\|^{1+\frac{2\ga}{3}}_{L^1_v}\|f\|^{-\frac{2\ga}{3}}_{L^2_v}\|g\|_{L^1_v}
		\le \|\<v\>^2f\|_{L^2_v}\|g\|_{L^1_v},
	\end{align*}
	where we used $-\frac{3}{2}<\ga<0$.
	For \eqref{Gammaes1}, by H\"{o}lder's inequality, we have
	\begin{align*}\notag
		& \int_{\R^6}|v-v_*|^\ga |f(v)|\mu^{a}(v_*)|g(v_*)|\,dv_*dv \\
		& \quad\notag\le \int_{\R^3}|f(v)|\Big(\int_{\R^3}|v-v_*|^{2\ga}\mu^{a}(v_*)\,dv_*\Big)^{\frac{1}{2}}
		\Big(\int_{\R^3}\mu^{a}(v_*)|g(v_*)|^2\,dv_*\Big)^{\frac{1}{2}}\,dv \\
		& \quad\le\|\<v\>^{-l}g\|_{L^2_v}\|\<v\>^{2\ga}f\|_{L^1_v},
	\end{align*}
	and
	\begin{align*}
		& \int_{\R^6}|v-v_*|^\ga |f(v)|\mu^{\frac{1}{2}}(v_*)|g(v_*)|\,dv_*dv \\
		& \quad\notag\le \|\<v\>^{-l}g\|_{L^\infty_v}\int_{\R^3}|f(v)||v-v_*|^{\ga} \<v\>^{-l}(v_*)\,dv_*dv \\
		& \quad\le\|\<v\>^{-l}g\|_{L^\infty_v}\|\<v\>^{\ga}f\|_{L^1_v}.
	\end{align*}
	This concludes Lemma \ref{Lem21}.
\end{proof}

\smallskip

\subsubsection{Regular change of variables}

In the following Lemma, we give the regular change of variable for $v'$ (and similarly for $v_*'$). Note that in this Lemma the function $b$ can depend on several independent variables.
\begin{Lem}\label{regularLem}
	Let $\ga>-3$ and $\mathbf{k}=\frac{v-v_*}{|v-v_*|}$ be given in \eqref{bfk}. For any functions $b$ and $f$ such that the integrals below are well-defined. Then we have the regular change of variables:
	\begin{multline}
		\label{regular1}
		\int_{\R^3}\int_{\S^2}|v-v_*|^\ga b\big(\sigma,\mathbf{k},\cos\th\big)f(v')\,d\sigma dv\\
		=\int_{\R^3}\int_{\S^1(\mathbf{k})}\int_0^{\frac{\pi}{2}}\frac{\sin\th|v-v_*|^\ga}{\cos^{3+\ga}\frac{\th}{2}}b\big(\cos\frac{\th}{2}\,\mathbf{k}+\sin\frac{\th}{2}\,\wt\omega,\cos\frac{\th}{2}\,\mathbf{k}-\sin\frac{\th}{2}\,\wt\omega,\cos\th\big)f(v)\,d\th d\wt\omega dv,
	\end{multline}
	where $\cos\th=\mathbf{k}\cdot\sigma$, $\wt\omega\in\S^1(\mathbf{k})$ with $\S^1(\mathbf{k})$ defined by \eqref{S1k}.
	Consequently, if $b$ depends only on $\cos\th$, then
	\begin{align}
		\label{regular}
		\int_{\R^3}\int_{\S^2}|v-v_*|^\ga b(\cos\th)f(v')\,d\sigma dv=\int_{\R^3}\int_{\S^2}\frac{|v-v_*|^\ga}{\cos^{3+\ga}\frac{\th}{2}}b(\cos\th)f(v)\,d\sigma dv,
	\end{align}
	Moreover, if we let $\omega=\frac{\sigma-(\mathbf{k}\cdot\sigma)\mathbf{k}}{|\sigma-(\mathbf{k}\cdot\sigma)\mathbf{k}|}$, and $b$ depends on $\omega$ and $\cos\th$, then
	\begin{multline}
		\label{regularomega}
		\int_{\R^3}\int_{\S^2}|v-v_*|^\ga b(\omega,\cos\th)f(v')\,d\sigma dv\\
		=\int_{\R^3}\int_{\S^1(\mathbf{k})}\int_0^{\frac{\pi}{2}}\frac{\sin\th|v-v_*|^\ga}{\cos^{3+\ga}\frac{\th}{2}} b\Big(\cos\frac{\th}{2}\,\wt\omega+\sin\frac{\th}{2}\,\mathbf{k},\cos\th\Big)f(v)\,d\th d\wt\omega dv,
	\end{multline}
	where $\cos\th=\mathbf{k}\cdot\sigma$, $\wt\omega\in\S^1(\mathbf{k})$.
\end{Lem}
\begin{proof}
The proof is similar to \cite[Lemma 1]{Alexandre2000} but we provide the proof with slight modification, since \eqref{regular1} and \eqref{regularomega} involve $\omega$ given in \eqref{sigma}.
	Recall from \eqref{vprime} that
	\begin{align}\label{vprsi}
		v'=\frac{v+v_*}{2}+\frac{|v-v_*|\sigma}{2}.
	\end{align}
	We will preform change of variable $v\mapsto v'$, which is well defined on the set $\{\cos\th>0\}$. By direct calculation, the Jacobian determinant is
	\begin{align*}
		\Big|\frac{\pa v'}{\pa v}\Big|=\Big|\frac{1}{2}I+\frac{1}{2}\mathbf{k}\otimes\sigma\Big|
		=\frac{1}{2^3}(1+\mathbf{k}\cdot\sigma)=\frac{1}{2^2}(\mathbf{k}'\cdot\sigma)^2.
	\end{align*}
	where the last identity can be deduced as
	\begin{align}\label{kkprime}
		1+\mathbf{k}\cdot\sigma=1+\cos\th=2\cos^2\frac{\th}{2}=2(\mathbf{k}'\cdot\sigma)^2,
	\end{align} which can also be seen from the geometry of binary collisions (as in \cite[Lemma 1]{Alexandre2000}). Here
	\begin{align*}
		\mathbf{k}'=\frac{v'-v_*}{|v'-v_*|}\ \text{ if }v'\ne v_*;\quad \mathbf{k}'=(1,0,0)\ \text{ if }v'=v_*,
	\end{align*}
	and
	\begin{align*}
		\mathbf{k}'\cdot\sigma=\cos^2\frac{\th}{2}\ge \frac{1}{\sqrt{2}}.
	\end{align*}
	Also, notice from \eqref{vprsi} and \eqref{vpriminv} that
	\begin{align*}
		\mathbf{k}=\frac{v-v_*}{|v-v_*|}=\frac{2(v'-v_*)-|v-v_*|\sigma}{|v-v_*|}
		=\frac{2(v'-v_*)\cos\frac{\th}{2}}{|v'-v_*|}-\sigma,
	\end{align*}
	and
	\begin{align*}
		|v-v_*|=\frac{|v'-v_*|}{\cos\frac{\th}{2}}=\frac{|v'-v_*|}{\mathbf{k}'\cdot\sigma}.
	\end{align*}
	Combining all the above preparation, we apply change of variable $v\mapsto v'$ to the left-hand side of \eqref{regular} and change notation $v'$ to $v$ to deduce
	\begin{align}\label{341a}\notag
		& \int_{\S^2}\int_{\R^3}|v-v_*|^\ga b\big(\sigma,\mathbf{k},\mathbf{k}\cdot\sigma\big)f(v')\,dvd\sigma \\
		& \notag\quad=\int_{\S^2}\int_{\mathbf{k}'\cdot\sigma\ge\frac{1}{\sqrt{2}}}\frac{|v'-v_*|^\ga}{(\mathbf{k}'\cdot\sigma)^\ga}b\big(\sigma,\frac{2(v'-v_*)\cos\frac{\th}{2}}{|v'-v_*|}-\sigma,2(\mathbf{k}'\cdot\sigma)^2-1\big)f(v')\frac{2^2}{(\mathbf{k}'\cdot\sigma)^2}\,dv'd\sigma \\
		& \quad=2^2\int_{\S^2}\int_{\mathbf{k}\cdot\sigma\ge\frac{1}{\sqrt{2}}}\frac{|v-v_*|^\ga}{(\mathbf{k}\cdot\sigma)^{2+\ga}}b\big(\sigma,2\cos\th\,\mathbf{k}-\sigma,2(\mathbf{k}\cdot\sigma)^2-1\big)f(v)\,dvd\sigma,
	\end{align}
	where $\cos\th$ is always defined by $\cos\th=\mathbf{k}\cdot\sigma$, and we used \eqref{kkprime} to find that, during the change of notation from $v'$ to $v$, $\cos\th$ is changed to $\cos(2\th)$.
	Applying spherical coordinate with $\mathbf{k}$,
	i.e. $$\sigma=\cos\th\,\mathbf{k}+\sin\th\,\wt\omega,\quad \wt\omega\in\S^1(\mathbf{k})$$
	with $\wt\omega\perp\mathbf{k}$ as in \eqref{sigma}, and then applying change of variable $\th\mapsto\frac{\th}{2}$, we have
	\begin{align*}
		& 2^2\int_{\mathbf{k}\cdot\sigma\ge\frac{1}{\sqrt{2}}}\frac{|v-v_*|^\ga}{(\mathbf{k}\cdot\sigma)^{2+\ga}}b\big(\sigma,2\cos\th\,\mathbf{k}-\sigma,2(\mathbf{k}\cdot\sigma)^2-1\big)f(v)\,d\sigma \\
		& \quad=2^2\int_{\S^1(\mathbf{k})}\int_0^{\frac{\pi}{4}}\frac{\sin\th|v-v_*|^\ga}{\cos^{2+\ga}\th}b\big(\cos\th\,\mathbf{k}+\sin\th\,\wt\omega,\cos\th\,\mathbf{k}-\sin\th\,\wt\omega,\cos(2\th)\big)f(v)\,d\th d\wt\omega \\
		& \quad=\int_{\S^1(\mathbf{k})}\int_0^{\frac{\pi}{2}}\frac{\sin\th|v-v_*|^\ga}{\cos^{3+\ga}\frac{\th}{2}}b\big(\cos\frac{\th}{2}\,\mathbf{k}+\sin\frac{\th}{2}\,\wt\omega,\cos\frac{\th}{2}\,\mathbf{k}-\sin\frac{\th}{2}\,\wt\omega,\cos\th\big)f(v)\,d\th d\wt\omega.
	\end{align*}
	This and \eqref{341a} imply \eqref{regular1}.
	Consequently, when $b$ depends only on $\cos\th$, we obtain \eqref{regular}.
	If we let $\omega=\frac{\sigma-(\mathbf{k}\cdot\sigma)\mathbf{k}}{|\sigma-(\mathbf{k}\cdot\sigma)\mathbf{k}|}$ as in \eqref{sigma}, then by \eqref{regular1}, the left hand side of \eqref{regularomega} is
	\begin{align*}
		& \int_{\R^3}\int_{\S^2}|v-v_*|^\ga b(\omega,\cos\th)f(v')\,d\sigma dv \\
		& =\int_{\R^3}\int_{\S^2}|v-v_*|^\ga b\big(\frac{\sigma-\cos\th\mathbf{k}}{\sin\th},\cos\th\big)f(v')\,d\sigma dv\\
		& =\int_{\R^3}\int_{\S^1(\mathbf{k})}\int_0^{\frac{\pi}{2}}\frac{\sin\th|v-v_*|^\ga}{\cos^{3+\ga}\frac{\th}{2}} b\Big(\frac{2\cos^2\frac{\th}{2}\sin\frac{\th}{2}\,\wt\omega
			+2\sin^2\frac{\th}{2}\cos\frac{\th}{2}\,\mathbf{k}}{\sin\th},\cos\th\Big)f(v)\,d\th d\wt\omega dv \\
		& =\int_{\R^3}\int_{\S^1(\mathbf{k})}\int_0^{\frac{\pi}{2}}\frac{\sin\th|v-v_*|^\ga}{\cos^{3+\ga}\frac{\th}{2}} b\Big(\cos\frac{\th}{2}\,\wt\omega+\sin\frac{\th}{2}\,\mathbf{k},\cos\th\Big)f(v)\,d\th d\wt\omega dv,
	\end{align*}
	which implies \eqref{regularomega}. This completes the proof of Lemma \ref{regularLem}.
	
\end{proof}

Furthermore, we have the following basic collision-type estimates. 
\begin{Lem}\label{FGLem}
	Suppose $H\in W^{2,\infty}$ and $-\frac{3}{2}<\ga\le 2$. Denote $\<v\>^{l}_\de$ as in \eqref{weight}. Assume $l\ge\ga+10$, $p\in(1,\infty)$ and $\frac{1}{p'}=1-\frac{1}{p}$. Then we have the following estimates for suitable functions $F,G,\Psi$.
\begin{multline}\label{Hpristar}
	\text{(a)}\ \ 
	\Big|\int_{\R^6\times\S^2}B(v-v_*,\sigma)F(v')\Psi(v_*)(H(v'_*)-H(v_*))\,d\sigma dv_*dv\Big|
	\\\le C\|\<v\>^{2+(\ga+2s)_+}\Psi\|_{L^2_v}\|\<v\>^{(\ga+2s)_+}F\|_{L^1_{v}}\|[H,\na_vH,\na^2_vH]\|_{L^\infty_v};
\end{multline}
\begin{multline}\label{Hpristar2}
	\text{(b)}\ \
	\Big|\int_{\R^6\times\S^2}B(v-v_*,\sigma)F(v')\Psi(v_*)(H(v)-H(v'))\,d\sigma dv_*dv\Big|
	\\\le C\|\<v\>^{2+(\ga+2s)_+}\Psi\|_{L^2_v}\|\<v\>^{(\ga+2s)_+}F\|_{L^1_{v}}\|[H,\na_vH,\na^2_vH]\|_{L^\infty_v};
\end{multline}
\begin{multline}\label{Hpristar1}
	\text{(c)}\ \ 
	\Big|\int_{\R^6\times\S^2}B(v-v_*,\sigma)F(v)\Psi(v_*)(H(v'_*)-H(v_*))\,d\sigma dv_*dv\Big|\\
	+\Big|\int_{\R^6\times\S^2}B(v-v_*,\sigma)F(v)\Psi(v_*)(H(v')-H(v))\,d\sigma dv_*dv\Big|
	\\\le C\|\<v\>^{2+(\ga+2s)_+}\Psi\|_{L^2_v}\|\<v\>^{(\ga+2s)_+}F\|_{L^1_v}\|[H,\na_vH,\na^2_vH]\|_{L^\infty_v};
\end{multline}
\begin{multline}
	\label{vpristar}
	\text{(d)}\ \ 
	\Big|\int_{\R^6\times\S^2}B(v-v_*,\sigma)F(v')\Psi(v_*)
	(\<v'\>^{-l}_\de-\<v\>^{-l}_\de)
	\,d\sigma dv_*dv\Big|\\
	\le C_{\de,l,\|n\|_{L^\infty}}\|\<v\>^{\frac{l}{2}+\ga+5}\Psi\|_{L^2_v}\|\<v\>^{-\frac{l}{2}+\ga+2}F\|_{L^1_v};
\end{multline}
The above norms on the right-hand side are taken over $\R^3_v$.
\end{Lem}
\begin{proof}
{\bf Estimate (a).}
Similar to \eqref{Taylor}, by Taylor expansion, we have
\begin{align}\notag\label{468}
	H(v'_*)-H(v_*)
	&=\na H(v_*)\cdot(v'_*-v_*)\\
	&\quad+\int^1_0(1-t)\na^2H(v_*+t(v'_*-v_*)):(v'_*-v_*)\otimes(v'_*-v_*)\,dt. 
\end{align}
Thus, using \eqref{vprimeth}, and decomposing the region of $\th$,
the left hand side of \eqref{Hpristar} is 
\begin{align*}
	&=\int_{\R^6\times\S^2}\1_{\min\{\frac{\pi}{2},|v-v_*|^{-1}\}\le\th\le\frac{\pi}{2}}|v-v_*|^{\ga}b(\cos\th)F(v')\Psi(v_*)(H(v'_*)-H(v_*))\,d\sigma dv_*dv\\
	&\quad+\int_{\R^6\times\S^2}\1_{0\le\th\le\min\{\frac{\pi}{2},|v-v_*|^{-1}\}}|v-v_*|^{\ga}b(\cos\th)\sin^2\frac{\th}{2}F(v')\Psi(v_*)\na H(v_*)\cdot(v-v_*)\,d\sigma dv_*dv\\
	&\quad-\frac{1}{2}\int_{\R^6\times\S^2}\1_{0\le\th\le\min\{\frac{\pi}{2},|v-v_*|^{-1}\}}|v-v_*|^{\ga+1}b(\cos\th)\sin\th F(v')\Psi(v_*)\na H(v_*)\cdot\omega\,d\sigma dv_*dv\\
	&\quad+\int_{\R^6\times\S^2}\int^1_0\1_{0\le\th\le\min\{\frac{\pi}{2},|v-v_*|^{-1}\}}|v-v_*|^{\ga}b(\cos\th)F(v')\Psi(v_*)\\
	&\qquad\qquad\qquad\times\na^2H(v_*+t(v'_*-v_*)):(v'_*-v_*)\otimes(v'_*-v_*)\,dtd\sigma dv_*dv\\
	&=:I_1+I_2+I_3+I_4.
\end{align*}
For the term $I_1$, by regular change of variable \eqref{regular} and estimate \eqref{gafg}, we have
\begin{align}\label{I11111}\notag
	|I_1| & \le 2\|H\|_{L^\infty_v}\int_{\R^6\times\S^2}\1_{\min\{\frac{\pi}{2},|v-v_*|^{-1}\}\le\th\le\frac{\pi}{2}}|v-v_*|^{\ga}b(\cos\th)|F(v')\Psi(v_*)|\,d\sigma dv_*dv\\
	& \notag\le C\|H\|_{L^\infty_v}\int_{\R^6}\int_{\min\{\frac{\pi}{2},|v-v_*|^{-1}\}}^{\frac{\pi}{2}}\th^{-1-2s}|v-v_*|^{\ga}|F(v)\Psi(v_*)|\,d\th dv_*dv \\
	& \notag\le C\|H\|_{L^\infty_v}\int_{\R^6}|v-v_*|^{\ga+2s}|F(v)\Psi(v_*)|\,dv_*dv \\
	& \le C\|H\|_{L^\infty_v}\|\<v\>^{2+(\ga+2s)_+}\Psi\|_{L^2_v}\|\<v\>^{(\ga+2s)_+}F\|_{L^1_v}.
\end{align}
For the term $I_3$, we apply regular change of variable \eqref{regularomega} for $\omega$ to deduce
\begin{multline}\label{I2a}
	I_3=-2\pi\int_{\R^6}\int_0^{\min\{\frac{\pi}{2},|v-v_*|^{-1}\}}
	\frac{|v-v_*|^{\ga+1} b(\cos\th)}{\cos^{4+\ga}\frac{\th}{2}}\Psi(v_*)F(v)\sin^2\th\sin\frac{\th}{2}
	\na H(v_*)\cdot\mathbf{k}\,d\th dv_*dv\\
	-\frac{1}{2}\int_{\R^6}\int_{\S^1(\mathbf{k})}\int_0^{\min\{\frac{\pi}{2},|v-v_*|^{-1}\}}\frac{|v-v_*|^{\ga+1} b(\cos\th)}{\cos^{4+\ga}\frac{\th}{2}}\Psi(v_*)F(v)\sin^2\th\na H(v_*)\cdot\wt\omega\,d\th d\wt\omega dv_*dv,
\end{multline}
where $\wt\omega\in\S^1(\mathbf{k})$ and $\mathbf{k}=\frac{v-v_*}{|v-v_*|}$ is given in \eqref{bfk}.
Then the second right-hand term of \eqref{I2a} vanishes due to symmetric about $\wt\omega$.
For the first right-hand term of \eqref{I2a}, $I_2$ and $I_4$, we apply \eqref{vpriminv}, regular change of variable \eqref{regular} and estimate \eqref{bcossin} to deduce
\begin{align}\label{471}\notag
	& |I_2|+|I_4|+\int_{\R^6}\int_0^{\min\{\frac{\pi}{2},|v-v_*|^{-1}\}}
	\frac{|v-v_*|^{\ga+1} b(\cos\th)}{\cos^{4+\ga}\frac{\th}{2}}|\Psi(v_*)F(v)|\sin^2\frac{\th}{2}
	|\na H(v_*)|\,d\th dv_*dv \\
	& \quad\le\notag C\int_{\R^6
	}\big(|v_*-v|^{\ga+2s-1}\1_{|v-v_*|\ge\frac{2}{\pi}}+|v-v_*|^{\ga+1}\1_{|v-v_*|<\frac{2}{\pi}}+|v-v_*|^{\ga+2s}\big)
	\\&\qquad\quad\quad\notag\times|\Psi(v_*)F(v)|\|[\na_vH,\na^2_vH]\|_{L^\infty_v}\,
	dv_*dv \\
	& \quad\le C\|\<v\>^{2+(\ga+2s)_+}\Psi\|_{L^2_v}\|\<v\>^{(\ga+2s)_+}F\|_{L^1_{v}}\|[\na_vH,\na^2_vH]\|_{L^\infty_v},
\end{align}
where we used \eqref{gafg} in the last inequality.
Therefore, combining the above estimates, we obtain \eqref{Hpristar}.

\medskip\noindent{\bf Estimate (b).}
To obtain \eqref{Hpristar2}, similar to \eqref{468}, we have
\begin{align*}
	H(v)-H(v')
	=\na H(v')\cdot(v-v')+\int^1_0(1-t)\na^2H(v'+t(v-v')):(v-v')\otimes(v-v')\,dt.
\end{align*}
Then using \eqref{vprimeth} and \eqref{vpriminv}, the left hand side of \eqref{Hpristar2} is
\begin{align}\label{wtI}
	&\notag=\int_{\R^6\times\S^2}\1_{\min\{\frac{\pi}{2},|v-v_*|^{-1}\}\le\th\le\frac{\pi}{2}}|v-v_*|^{\ga}b(\cos\th)F(v')\Psi(v_*)(H(v)-H(v'))\,d\sigma dv_*dv\\
	&\notag\quad+\int_{\R^6\times\S^2}\1_{0\le\th\le\min\{\frac{\pi}{2},|v-v_*|^{-1}\}}|v-v_*|^{\ga}b(\cos\th)\sin^2\frac{\th}{2}F(v')\Psi(v_*)\na H(v')\cdot(v-v_*)\,d\sigma dv_*dv\\
	&\notag\quad-\frac{1}{2}\int_{\R^6\times\S^2}\1_{0\le\th\le\min\{\frac{\pi}{2},|v-v_*|^{-1}\}}|v-v_*|^{\ga+1}b(\cos\th)\sin\th F(v')\Psi(v_*)\na H(v')\cdot\omega\,d\sigma dv_*dv\\
	&\notag\quad+\int_{\R^6\times\S^2}\int^1_0\1_{0\le\th\le\min\{\frac{\pi}{2},|v-v_*|^{-1}\}}|v-v_*|^{\ga}b(\cos\th)F(v')\Psi(v_*)\\
	&\notag\qquad\qquad\qquad\times\na^2H(v'+t(v-v')):(v-v')\otimes(v-v')\,dtd\sigma dv_*dv\\
	&=:\wt I_1+\wt I_2+\wt I_3+\wt I_4.
\end{align}
The estimate of $\wt I_1$ is the same as \eqref{I11111}, and we have
\begin{align*}
	|\wt I_1|\le C\|H\|_{L^\infty_v}\|\<v\>^{2+(\ga+2s)_+}\Psi\|_{L^2_v}\|\<v\>^{(\ga+2s)_+}F\|_{L^1_v}.
\end{align*}
For the term $\wt I_3$, similar to \eqref{I2a}, by regular change of variable \eqref{regularomega}, one has
\begin{multline*}
	\wt I_3=-2\pi\int_{\R^6}\int_0^{\min\{\frac{\pi}{2},|v-v_*|^{-1}\}}\frac{|v-v_*|^{\ga+1} b(\cos\th)}{\cos^{4+\ga}\frac{\th}{2}}\Psi(v_*)F(v)\sin^2\th\sin\frac{\th}{2}
	\na H(v)\cdot\mathbf{k}\,d\th dv_*dv\\
	-\frac{1}{2}\int_{\R^6}\int_{\S^1(\mathbf{k})}\int_0^{\min\{\frac{\pi}{2},|v-v_*|^{-1}\}}\frac{|v-v_*|^{\ga+1} b(\cos\th)}{\cos^{4+\ga}\frac{\th}{2}}\Psi(v_*)F(v)\sin^2\th\na H(v)\cdot\wt\omega\,d\th d\wt\omega dv_*dv,
\end{multline*}
and the second right-hand term vanishes due to symmetric about $\wt\omega$. The first right-hand term and $\wt I_2+\wt I_4$ in \eqref{wtI} can be estimated with the same method in \eqref{471}:
\begin{multline*}
	|\wt I_2|+|\wt I_4|+\int_{\R^6}\int_0^{\min\{\frac{\pi}{2},|v-v_*|^{-1}\}}\frac{|v-v_*|^{\ga+1} b(\cos\th)}{\cos^{4+\ga}\frac{\th}{2}}|\Psi(v_*)F(v)|\sin^2\th\sin\frac{\th}{2}
	|\na H(v)|\,d\th dv_*dv\\
	\le C\|\<v\>^{2+(\ga+2s)_+}\Psi\|_{L^2_v}\|\<v\>^{(\ga+2s)_+}F\|_{L^1_{v}}\|[\na_vH,\na^2_vH]\|_{L^\infty_v}.
\end{multline*}
Combining the above three estimates, we obtain \eqref{Hpristar2}.

\medskip\noindent{\bf Estimate (c).} The proof of \eqref{Hpristar1} is simpler. Applying \eqref{HHbcos}, we have
\begin{align*}
	& \Big|\int_{\R^6\times\S^2}B(v-v_*,\sigma)F(v)\Psi(v_*)(H(v'_*)-H(v_*))\,d\sigma dv_*dv\Big| \\
	& \quad\le C\|[H,\na_vH,\na^2_vH]\|_{L^\infty_v}\int_{\R^6}\big(|v-v_*|^{\ga}+|v-v_*|^{\ga+2s}\big)|F(v)\Psi(v_*)|\,dv_*dv \\
	& \quad\le C\|[H,\na_vH,\na^2_vH]\|_{L^\infty_v}\|\<v\>^{2+(\ga+2s)_+}\Psi\|_{L^2_v}\|\<v\>^{(\ga+2s)_+}F\|_{L^1_v},
\end{align*}
where we used \eqref{gafg} in the second inequality. The estimate of the second left-hand term of \eqref{Hpristar1} is similar and we omit the proof for brevity.

\medskip\noindent
{\bf Estimate (d).} 
Recall that we define the weight function $\<v\>^{l}_\de$ in \eqref{weight} for brevity. By Lemmas \ref{vldeLem1} and \ref{vldeLem}, and Taylor's expansion \eqref{Taylor}, we have 
\begin{align*}
	\<v'\>^{-l}_\de-\<v\>^{-l}_\de
	&=\<v'\>^{-\frac{l}{2}}_\de(\<v'\>^{-\frac{l}{2}}-\<v\>^{-\frac{l}{2}})+(\<v'\>^{-\frac{l}{2}}_\de-\<v\>^{-\frac{l}{2}}_\de)\<v\>^{-\frac{l}{2}}\\
	&=\Big(\<v'\>^{-\frac{l}{2}}_\de\na_v\<v\>^{-\frac{l}{2}}+\<v\>^{-\frac{l}{2}}\na_v\<v\>^{-\frac{l}{2}}_\de\Big)\cdot(v'-v)
	\\&\quad+\int^1_0(1-t)\<v'\>^{-\frac{l}{2}}_\de\na^2_u\<u\>^{-\frac{l}{2}}|_{u=v+t(v'-v)}:(v'-v)\otimes(v'-v)\,dt\\
	&\quad+\int^1_0(1-t)\<v\>^{-\frac{l}{2}}\na^2_u\<u\>^{-\frac{l}{2}}_\de|_{u=v+t(v'-v)}:(v'-v)\otimes(v'-v)\,dt. 
\end{align*}
Thus, using \eqref{vprimeth} and \eqref{vpriminv} to represent $v'-v$, and using Lemma \ref{vldeLem} to estimate the velocity derivatives of the weight, the left-hand side of \eqref{vpristar} is 
\begin{align}\label{J123}
	&\notag\le \Big|\int_{\R^6\times\S^2}B
	F(v')\Psi(v_*)\frac{|v-v_*|\sin\th}{2}\omega\cdot\Big(\<v'\>^{-\frac{l}{2}}_\de\na_v\<v\>^{-\frac{l}{2}}+\<v\>^{-\frac{l}{2}}\na_v\<v\>^{-\frac{l}{2}}_\de\Big)\,d\sigma dv_*dv\Big|\\
	&\notag\quad+\Big|\int_{\R^6\times\S^2}B
	F(v')\Psi(v_*)\frac{\sin^2\frac{\th}{2}}{2}(v_*-v)\cdot\Big(\<v'\>^{-\frac{l}{2}}_\de\na_v\<v\>^{-\frac{l}{2}}+\<v\>^{-\frac{l}{2}}\na_v\<v\>^{-\frac{l}{2}}_\de\Big)\,d\sigma dv_*dv\Big|\\
	&\quad+C_{\de,l,\|n\|_{L^\infty}}\int_{\R^6\times\S^2}B|F(v')||\Psi(v_*)|\big(\<v'\>^{-\frac{l}{2}}+\<v\>^{-\frac{l}{2}}\big)|v-v_*|^2\sin^2\frac{\th}{2}\,d\sigma dv_*dv\notag\\
	&\quad=:|J_1|+|J_2|+|J_3|.
\end{align}
For the term $J_1$, notice from \eqref{sigma} that $\omega\perp\mathbf{k}$ and hence,
$$v\cdot\omega=v_*\cdot\omega.$$
Thus, by Lemma \ref{vldeLem} (4), we estimate $J_1$ in \eqref{J123} as
\begin{align*}
	J_1&=\int_{\R^6\times\S^2}B
	F(v')\Psi(v_*)\frac{|v-v_*|\sin\th}{2}\omega\cdot\Big(\<v'\>^{-\frac{l}{2}}_\de\na_v\<v\>^{-\frac{l}{2}}+\<v\>^{-\frac{l}{2}}\na_v\<v\>^{-\frac{l}{2}}_\de\Big)\,d\sigma dv_*dv\\
	&=\int_{\R^6\times\S^2}B
	F(v')\Psi(v_*)\frac{|v-v_*|\sin\th}{2}
	\Big(-\frac{l}{2}\<v'\>^{-\frac{l}{2}}_\de\<v\>^{-\frac{l}{2}-2}\omega\cdot v_*\\
	&\qquad+
	\<v\>^{-\frac{l}{2}}\big(O(\de,l,\|n\|_{L^\infty})\<v\>^{-\frac{l}{2}-2}\omega\cdot v_*
	+O(\de,l,\|n\|_{L^\infty})\<v\>^{-\frac{l}{2}-2}\omega\cdot n(x)\big)\Big)\,d\sigma dv_*dv.  
\end{align*}
To apply regular change of variable \eqref{regularomega}, we rewrite the integrand as 
\begin{align*}
	&-\frac{l}{2}\<v'\>^{-\frac{l}{2}}_\de\<v\>^{-\frac{l}{2}-2}\omega\cdot v_*+
	O(\de,l,\|n\|_{L^\infty})\<v\>^{-l-2}\big(\omega\cdot v_*+\omega\cdot n(x)\big)\\
	&\quad
	=-\frac{l}{2}\<v'\>^{-\frac{l}{2}}_\de\big(\<v\>^{-\frac{l}{2}-2}-\<v'\>^{-\frac{l}{2}-2}\big)\omega\cdot v_*
	-\frac{l}{2}\<v'\>^{-\frac{l}{2}}_\de\<v'\>^{-\frac{l}{2}-2}\omega\cdot v_*
	\\&\qquad+O(\de,l,\|n\|_{L^\infty})\<v\>^{-\frac{l}{2}}\big(\<v\>^{-\frac{l}{2}-2}-\<v'\>^{-\frac{l}{2}-2}\big)\big(\omega\cdot v_*+\omega\cdot n(x)\big)
	\\&\qquad+O(\de,l,\|n\|_{L^\infty})\big(\<v\>^{-\frac{l}{2}}-\<v'\>^{-\frac{l}{2}}\big)\<v'\>^{-\frac{l}{2}-2}\big(\omega\cdot v_*+\omega\cdot n(x)\big)
	\\&\qquad+O(\de,l,\|n\|_{L^\infty})\<v'\>^{-l-2}\big(\omega\cdot v_*+\omega\cdot n(x)\big).
\end{align*}
For the term involving $\omega$ and $v'$, we apply regular change of variable \eqref{regularomega}, while for the commutator terms, we use Lemmas \ref{vldeLem1} and \ref{vldeLem} to obtain one more $|v'-v|$, change $\<v\>^{-\frac{l}{2}}$ into $\<v_*\>^{\frac{l}{2}}\<v'\>^{-\frac{l}{2}}$ and use \eqref{regular}. That is, 
\begin{align*}
	J_1
	&=\frac{l}{2}\int_{\R^6}\int_{\S^1(\mathbf{k})}\int_0^{\frac{\pi}{2}}\frac{\sin^2\th|v-v_*|^{\ga+1}}{\cos^{4+\ga}\frac{\th}{2}}b(\cos\th)
	F(v)\Psi(v_*)
	\big(\cos\frac{\th}{2}\,\wt\omega+\sin\frac{\th}{2}\,\mathbf{k}\big)\\
	&\qquad\quad\cdot\Big(\frac{-lv_*}{2}\<v\>^{-\frac{l}{2}}_\de\<v\>^{-\frac{l}{2}-2}
	+\big(v_*+n(x)\big)O(\de,l,\|n\|_{L^\infty})\<v\>^{-l-2}\Big)
	\,d\th d\wt\omega dv_*dv\\
	&\quad+C_{\de,l,\|n\|_{L^\infty}}\int_{\R^6\times\S^2}|v-v_*|^{\ga+2}b(\cos\th)\sin^2\frac{\th}{2}
	|F(v')\Psi(v_*)|\<v_*\>^{\frac{l}{2}+1}\<v'\>^{-\frac{l}{2}}\,d\sigma dv_*dv. 
\end{align*}
The term involving $\wt\omega\in\S^1(\mathbf{k})$ (given in \eqref{regularomega}) vanishes due to symmetric about $\wt\omega$. Thus, 
by \eqref{bcossin} and \eqref{gafg}, we obtain
\begin{align*}
	J_1 & \le C_{\de,l,\|n\|_{L^\infty}}\int_{\R^6}\big(|v-v_*|^{\ga+1}+|v-v_*|^{\ga+2}\big)
	|F(v)\Psi(v_*)|\<v_*\>^{\frac{l}{2}+1}\<v\>^{-\frac{l}{2}}\,dv_*dv\\
	& \le C_{\de,l,\|n\|_{L^\infty}}\|\<v\>^{\frac{l}{2}+\ga+5}\Psi\|_{L^2_v}\|\<v\>^{-\frac{l}{2}+\ga+2}F\|_{L^1_v}.
\end{align*}
The terms $J_2+J_3$ in \eqref{J123} can be estimated easily, since they have enough angular decay rate $\sin^2\frac{\th}{2}$. 
Thus, using Lemma \ref{vldeLem1} and \ref{vldeLem}, and the regular change of variable \eqref{regular}, we have
\begin{align*}
	J_2+J_3&\le C_{\de,l,\|n\|_{L^\infty}}\int_{\R^6}\big(|v-v_*|^{\ga+1}+|v-v_*|^{\ga+2}\big)
	|F(v)\Psi(v_*)|\<v_*\>^{\frac{l}{2}+1}\<v\>^{-\frac{l}{2}}
	\,dv_*dv\\
	&\le C_{\de,l,\|n\|_{L^\infty}}\|\<v\>^{\frac{l}{2}+\ga+5}\Psi\|_{L^2_v}\|\<v\>^{-\frac{l}{2}+\ga+2}F\|_{L^1_v}, 
\end{align*}
where we used \eqref{bcossin} and \eqref{gafg}.
Combining the above estimates of $J_1,J_2,J_3$, we obtain \eqref{vpristar}.
This completes the proof of Lemma \ref{FGLem}.
\end{proof}

\subsection{\texorpdfstring{$L^2$}{L2} estimate of the collision terms}
In this Subsection, we write the energy estimates for the collision term.
Note that in order to limit the highest order of $\Psi$ to $l$, i.e. $\|\<v\>^{l}\Psi\|_{L^\infty_xL^\infty_v}$, we will use a different approach than \cite{Alonso2022}. 

The following Lemma gives some basic estimates of the collision operator. 
The term $\Gamma(f,\mu^{\frac{1}{2}})$ has better behavior in upper bound than $\Gamma(\mu^{\frac{1}{2}},f)$, whose upper bound doesn't contain velocity regularity. Moreover, we can restrict all the derivatives to one function. 
\begin{Lem}\label{vpmu12Lem}
Assume $-\frac{3}{2}<\ga\le 2$, $s\in(0,1)$. Denote $\<v\>^l_{\de}$ as in \eqref{weight}. Then 
\begin{align}
	\label{vpmu12}
	&\text{(a)}\quad 
|(\Gamma(f,\mu^{\frac{1}{2}}),g)_{L^2_v}|\le C
\|\mu^{\frac{1}{76}}f\|_{L^2_v}\|\mu^{\frac{1}{76}}g\|_{L^1_v},
	\\
	\label{vpl12}
	&\text{(b)}\quad 
	|(\Gamma(f,\<v\>^{-l}_\de),g)_{L^2_v}|\le C_{\de,l}\|\<v\>^lf\|_{L^\infty_v}\|\<v\>^{-2}g\|_{L^1_v}, 
\end{align}
for any $l\ge \ga+10$, with some $C=C(a,\ga,s,l,\de)>0$. \eqref{vpl12} is the only estimate needed in level-function estimate, so we only consider the weight $\<v\>^{-l}_\de$ here.
To restrict regularity only to one collisional function, we have 
\begin{align}\label{Gammavp}
	&\text{(c)}\quad 
	|(\Gamma(f,g),h)_{L^2_v}|\le C\|[h,\,\na_vh,\,\na^2_vh]\|_{L^\infty_v}\|\<v\>^{2+(\ga+2s)_+}f\|_{L^2_v}\|\<v\>^{2+(\ga+2s)_+}g\|_{L^2_v},\\
	\label{Gammavp1}
	&\text{(d)}\quad 
	|(\Gamma(f,g),h)_{L^2_v}|\le C\|[g,\na_vg,\na^2_vg]\|_{L^\infty_v}\|\<v\>^{2+(\ga+2s)_+}f\|_{L^2_v}\|\<v\>^{(\ga+2s)_+}h\|_{L^1_v}.
\end{align}
If we use the $H^{2s}$ control on the third term, then for any $k\ge 0$, $l\in\R$, 
\begin{align}
	\label{64eq1}
	&\text{(e)}\quad 
	|(\Gamma(f,g),h)_{L^2_v}|\le C\|f\|_{L^2_v}\|\<v\>^{(l+\ga+2s)_+}g\|_{L^2_v}\|\<v\>^{-l}h\|_{H^{2s}}, \\
	\label{64eq2}
	&\text{(f)}\quad 
	|(\Gamma(f,g),\<v\>^{2k}h)_{L^2_v}|
	\le 
	C\|\<v\>^{\frac{(\ga+2s)_+}{2}}f\|_{L^2_v}\|\<v\>^{k+\frac{\ga}{2}}g\|_{L^2_v}\|\<v\>^{k+4}h\|_{H^{2s}_v}.
\end{align}
\end{Lem}
\begin{proof}
{\bf Estimating \eqref{vpmu12}.}
By pre-post change of variable, we write
\begin{align}\label{T3a}\notag
	(\Gamma(f,\mu^{\frac{1}{2}}),g)_{L^2_v} & =\int_{\R^6\times\S^2}Bf_*\mu^{\frac{1}{2}}(v)\big(g(v')\mu^{\frac{1}{2}}(v'_*)-g\mu^{\frac{1}{2}}(v_*)\big)\,d\sigma dv_*dv \\
	& \notag=\int_{\R^6\times\S^2}Bf_*\big(\mu^{\frac{1}{2}}(v)-\mu^{\frac{1}{2}}(v')\big)g(v')\big(\mu^{\frac{1}{2}}(v'_*)-\mu^{\frac{1}{2}}(v_*)\big)\,d\sigma dv_*dv \\
	& \quad\notag+\int_{\R^6\times\S^2}Bf_*\big(\mu^{\frac{1}{2}}(v)-\mu^{\frac{1}{2}}(v')\big)g(v')\mu^{\frac{1}{2}}(v_*)\,d\sigma dv_*dv \\
	& \quad\notag+\int_{\R^6\times\S^2}Bf_*\mu^{\frac{1}{2}}(v')g(v')\big(\mu^{\frac{1}{2}}(v'_*)-\mu^{\frac{1}{2}}(v_*)\big)\,d\sigma dv_*dv \\
	& \quad\notag+\int_{\R^6\times\S^2}Bf_*\big(\mu^{\frac{1}{2}}(v')g(v')-\mu^{\frac{1}{2}}(v)g\big)\mu^{\frac{1}{2}}_*\,d\sigma dv_*dv \\
	& =:T_{1,1}+T_{1,2}+T_{1,3}+T_{1,4}.
\end{align}
For the term $T_{1,1}$, notice that
\begin{align}\label{635}\begin{aligned}
		\mu^{\frac{1}{2}}(v_*)-\mu^{\frac{1}{2}}(v'_*)&=\big(\mu^{\frac{1}{4}}(v_*)-\mu^{\frac{1}{4}}(v'_*)\big)\big(\mu^{\frac{1}{4}}(v_*)+\mu^{\frac{1}{4}}(v'_*)\big), \\
		\mu^{\frac{1}{2}}(v)-\mu^{\frac{1}{2}}(v')&=\big(\mu^{\frac{1}{4}}(v)-\mu^{\frac{1}{4}}(v')\big)\big(\mu^{\frac{1}{4}}(v)+\mu^{\frac{1}{4}}(v')\big),
	\end{aligned}
\end{align}
and from 
\begin{align}\label{477}
	|\mu^{\frac{1}{4}}(v'_*)-\mu^{\frac{1}{4}}(v_*)|&\le|v'_*-v_*|\int^1_0|(v_*+t(v'_*-v_*))\cdot\na_v\mu^{\frac{1}{4}}(v_*+t(v'_*-v_*))|\,dt\notag\\
	&\le C|v-v_*|\sin\frac{\th}{2}
\end{align} that
\begin{align}\label{635a}
	\begin{aligned}
		|\mu^{\frac{1}{4}}(v'_*)-\mu^{\frac{1}{4}}(v_*)|\le C\min\big\{|v-v_*|\sin\frac{\th}{2},1\big\}, \\
		|\mu^{\frac{1}{4}}(v')-\mu^{\frac{1}{4}}(v)|\le C\min\big\{|v-v_*|\sin\frac{\th}{2},1\big\}.
	\end{aligned}
\end{align}
Then $T_{1,1}$ can be estimated as
\begin{align*}
	|T_{1,1}|&\le C\int_{\R^6\times\S^2}|v-v_*|^{\ga+2}b(\cos\th)\sin^2\frac{\th}{2}f(v_*)g(v')\\
	&\qquad\qquad\qquad\times\big(\mu^{\frac{1}{4}}(v)+\mu^{\frac{1}{4}}(v')\big)\big(\mu^{\frac{1}{4}}(v_*)+\mu^{\frac{1}{4}}(v'_*)\big)\,d\sigma dv_*dv.
\end{align*}
To obtain the large velocity decay, we apply \eqref{vstar1}, \eqref{vstar2}, and \eqref{vstar3} to obtain 
\begin{align}\begin{aligned}
		\label{muequiv}
		& \mu(v'_*)\mu(v)\le \big(\mu(v_*)\mu(v')\big)^{\frac{1}{18}}, \\
		& \mu(v_*)\mu(v)=\mu(v'_*)\mu(v')\le \big(\mu(v_*)\mu(v')\big)^{\frac{1}{16}}.
	\end{aligned}
\end{align}
Thus, by regular change of variable \eqref{regular}, estimates \eqref{bcossin} and \eqref{gafg}, $T_{1,1}$ is bounded above as
\begin{align*}
	|T_{1,1}|
	& \le C\int_{\R^6\times\S^2}|v-v_*|^{\ga+2}b(\cos\th)\sin^2\frac{\th}{2}|f_*|(\mu(v_*)\mu(v'))^{\frac{1}{72}}|g(v')|\,d\sigma dv_*dv \\
	& \le C\int_{\R^6}|v-v_*|^{\ga+2}f_*(\mu(v_*)\mu(v))^{\frac{1}{72}}g(v)\,dv_*dv\\
	& \le C\|\mu^{\frac{1}{76}}f\|_{L^2_v}\|\mu^{\frac{1}{76}}g\|_{L^1_{v}}.
\end{align*}
%
%
For the term $T_{1,2}$,
noticing $\mu^{\frac{1}{2}}(v)-\mu^{\frac{1}{2}}(v')=2\mu^{\frac{1}{4}}(v')\big(\mu^{\frac{1}{4}}(v)-\mu^{\frac{1}{4}}(v')\big)+\big(\mu^{\frac{1}{4}}(v)-\mu^{\frac{1}{4}}(v')\big)^2$, we write
\begin{align*}
	T_{1,2}&=2\int_{\R^6\times\S^2}Bf(v_*)\mu^{\frac{1}{4}}(v')\big(\mu^{\frac{1}{4}}(v)-\mu^{\frac{1}{4}}(v')\big)g(v')\mu^{\frac{1}{2}}(v_*)\,d\sigma dv_*dv\\
	&\quad+\int_{\R^6\times\S^2}Bf(v_*)\big(\mu^{\frac{1}{4}}(v)-\mu^{\frac{1}{4}}(v')\big)^2g(v')\mu^{\frac{1}{2}}(v_*)\,d\sigma dv_*dv
	=:T_{1,2,1}+T_{1,2,2}.
\end{align*}
For the term $T_{1,2,1}$, by \eqref{Hpristar2}, we have
\begin{align*}
	|T_{1,2,1}| & \le C\|\mu^{\frac{1}{4}}f\|_{L^2_v}\|\mu^{\frac{1}{8}}g\|_{L^1_v}.
\end{align*}
For the term $T_{1,2,2}$, by \eqref{635a}, 
\eqref{635}, regular change of variable \eqref{regular} and \eqref{muequiv},
\begin{align*}
	|T_{1,2,2}| & \le C\int_{\R^6\times\S^2}B\sin^2\frac{\th}{2}|v-v_*|^2|f_*|\big(\mu^{\frac{1}{4}}(v)+\mu^{\frac{1}{4}}(v')\big)g(v')\mu^{\frac{1}{2}}_*\,d\sigma dv_*dv \\
	& \le C\int_{\R^6}|f(v_*)|g(v)\mu^{\frac{1}{64}}(v_*)\mu^{\frac{1}{64}}(v)\,dv_*dv \\
	& \le C\|\mu^{\frac{1}{76}}f\|_{L^2_v}\|\mu^{\frac{1}{76}}g\|_{L^1_v},
\end{align*}
where we used \eqref{gafg}.
The above two estimates imply
\begin{align*}
	|T_{1,2}| & \le C\|\mu^{\frac{1}{76}}f\|_{L^2_v}\|\mu^{\frac{1}{76}}g\|_{L^1_v}.
\end{align*}
The estimate of $T_{1,3}$ can be carried out as the estimate of $T_{1,2}$ by decomposing $\mu^{\frac{1}{2}}(v'_*)-\mu^{\frac{1}{2}}(v_*)=2\mu^{\frac{1}{4}}(v_*)\big(\mu^{\frac{1}{4}}(v'_*)-\mu^{\frac{1}{4}}(v_*)\big)+\big(\mu^{\frac{1}{4}}(v_*)-\mu^{\frac{1}{4}}(v'_*)\big)^2$. Thus, by using the same method as in the estimates of $T_{1,2,1}$, $T_{1,2,2}$ and \eqref{muequiv}, we have
\begin{align*}
	|T_{1,3}| & \le C\|\mu^{\frac{1}{76}}f\|_{L^2_v}\|\mu^{\frac{1}{76}}g\|_{L^1_v}.
\end{align*}
For the term $T_{1,4}$ in \eqref{T3a}, we apply regular change of variable \eqref{regular} and \eqref{bcossink2} to deduce
\begin{align*}
	|T_{1,4}| & \le \int_{\R^6\times\S^2}|v-v_*|^{\ga}b(\cos\th)\frac{1-\cos^{3+\ga}\frac{\th}{2}}{\cos^{3+\ga}\frac{\th}{2}}\mu^{\frac{1}{2}}_*f_*\mu^{\frac{1}{2}}(v)g(v)\,d\sigma dv_*dv \\
	& \le C\|\mu^{\frac{1}{4}}f\|_{L^2_{v}}\|\mu^{\frac{1}{4}}g\|_{L^1_v},
\end{align*}
where we used \eqref{gafg} in the last inequality. Substituting the above estimates for $T_{1,j}$'s ($1\le j\le 4$) into \eqref{T3a}, we obtain \eqref{vpmu12}.

\smallskip\noindent{\bf Estimating \eqref{vpl12}.}
The proof is similar to step 1. The only difference is a worse decay in $\<v\>^{-l}_\de$ compared to $\mu^{\frac{1}{2}}$.
By pre-post change of variable $(v,v_*)\mapsto (v',v_*')$, we estimate the left-hand side of \eqref{vpl12} as 
\begin{align}\notag
	&\int_{\R^6\times\S^2}Bf(v_*)\<v\>^{-l}_\de\big(g(v')\mu^{\frac{1}{2}}(v'_*)-g(v)\mu^{\frac{1}{2}}(v_*)\big)\,d\sigma dv_*dv\\
	&\notag\quad=\int_{\R^6\times\S^2}Bf(v_*)g(v')\big(\mu^{\frac{1}{2}}(v'_*)-\mu^{\frac{1}{2}}(v_*)\big)\big(\<v\>^{-l}_\de-\<v'\>^{-l}_\de\big)\,d\sigma dv_*dv\\
	&\notag\quad+\int_{\R^6\times\S^2}Bf(v_*)g(v')\mu^{\frac{1}{2}}(v_*)\big(\<v\>^{-l}_\de-\<v'\>^{-l}_\de\big)\,d\sigma dv_*dv\\
	&\notag\quad+\int_{\R^6\times\S^2}Bf(v_*)g(v')\<v'\>^{-l}_\de\big(\mu^{\frac{1}{2}}(v'_*)-\mu^{\frac{1}{2}}(v_*)\big)\,d\sigma dv_*dv\\
	&\notag\quad+\int_{\R^6\times\S^2}B\mu^{\frac{1}{2}}(v_*)f(v_*)\big(g(v')\<v'\>^{-l}_\de-g(v)\<v\>^{-l}_\de\big)\,d\sigma dv_*dv\\
	&\quad=:T_{2,1}+T_{2,2}+T_{2,3}+T_{2,4}. \label{T2c}
\end{align}
By Lemmas \ref{vldeLem} and \ref{vldeLem1}, for $l\ge 2$, we have
\begin{align*}
	\begin{aligned}
		\Big|{\<v\>^{-\frac{l}{2}}}-{\<v'\>^{-\frac{l}{2}}}\Big|\le C\min\big\{|v-v_*|\sin\frac{\th}{2},1\big\}. 
	\end{aligned}
\end{align*}
Moreover, similar to \eqref{muequiv}, we have from \eqref{vstar1}, \eqref{vstar2} and \eqref{vstar3} that 
\begin{align*}\begin{aligned}
		\mu^{\frac{1}{4}}(v'_*)\<v\>^{-\frac{l}{2}}&\le C_l
		\<v'\>^{-\frac{l}{2}},\\
		\max\big\{\mu^{\frac{1}{4}}(v_*)\<v\>^{-\frac{l}{2}},\,\mu^{\frac{1}{4}}(v'_*)\<v'\>^{-\frac{l}{2}}\big\}&\le C_l
		\<v'\>^{-\frac{l}{2}}. 
	\end{aligned}
\end{align*}
Then by \eqref{bcossin}, \eqref{gafg} and Lemma \ref{vldeLem} (5), the term $T_{2,1}$ can be estimated as 
\begin{align*}
	|T_{2,1}| &\le
	C_{\de,l}\int_{\R^6\times\S^2}|v-v_*|^{\ga}b(\cos\th)\min\big\{|v-v_*|^2\sin^2\frac{\th}{2},1\big\}f(v_*)g(v')
	\<v'\>^{-\frac{l}{2}}\,d\sigma dv_*dv\\ 
	&\le C_{\de,l}\|\<v\>^{2+(\ga+2s)_+}f\|_{L^2_v}\|\<v\>^{-\frac{l}{2}+(\ga+2s)_+}g\|_{L^1_v}.
\end{align*}
Applying \eqref{vpristar} and \eqref{Hpristar} to the terms $T_{2,2},T_{2,3}$ respectively, we have 
\begin{align*}
	|T_{2,2}|+|T_{2,3}|&\le C\|\<v\>^{\frac{l}{2}+\ga+5}\mu^{\frac{1}{2}}f\|_{L^2_v}\|\<v\>^{-\frac{l}{2}+\ga+2}g\|_{L^1_v}+C\|\<v\>^{2+(\ga+2s)_+}f\|_{L^2_v}\|\<v\>^{-l+(\ga+2s)_+}g\|_{L^1_{v}}.
\end{align*}
For the term $T_{2,4}$, by regular change of variable \eqref{regular} and \eqref{bcossink2}, we have
\begin{align*}
	|T_{2,4}| & =\int_{\R^6\times\S^2}|v-v_*|^\ga b(\cos\th)\mu^{\frac{1}{2}}(v_*)f(v_*)\frac{g(v)}{\<v\>^l}\frac{1-\cos^{3+\ga}\frac{\th}{2}}{\cos^{3+\ga}\frac{\th}{2}}\,d\sigma dv_*dv \\
	& =C\int_{\R^6}|v-v_*|^\ga\mu^{\frac{1}{2}}(v_*)f(v_*)g(v)\<v\>^{-l}_\de\,dv_*dv \\
	& \le C_{\de,l}\|\mu^{\frac{1}{4}}f\|_{L^\infty_v}\|\<v\>^{-l+{\ga}_+}g\|_{L^1_v},
\end{align*}
where we used \eqref{gafg}.
Substituting the above estimates into \eqref{T2c} and assuming $l\ge\ga+10$ and $-\frac{3}{2}<\ga\le 2$, we obtain
\begin{align*}
	|T_{2}|\le C_{\de,l}\|\<v\>^lf\|_{L^\infty_v}\|\<v\>^{-2}g\|_{L^1_v}. 
\end{align*}

\smallskip\noindent{\bf Estimating \eqref{Gammavp}.}
By \eqref{Ga}, we can write
\begin{multline}\label{Gaest}
	(\Gamma(f,g),h)_{L^2_v}
	= \int_{\R^6}\int_{\S^2}B(v-v_*,\sigma)\mu^{\frac{1}{2}}(v_*)\big(f'_*g'-f_*g\big)h\,d\sigma dv_*dv\\
	=\int_{\R^6}\int_{\S^2}B(v-v_*,\sigma)\big(\mu^{\frac{1}{2}}(v'_*)h(v')-\mu^{\frac{1}{2}}(v_*)h(v)\big)f(v_*)g(v)\,d\sigma dv_*dv.
\end{multline}
Note that
\begin{align*}
	& \mu^{\frac{1}{2}}(v'_*)h(v')-\mu^{\frac{1}{2}}(v_*)h(v) \\
	& =\big(\mu^{\frac{1}{2}}(v'_*)-\mu^{\frac{1}{2}}(v_*)\big)h(v')+\mu^{\frac{1}{2}}(v_*)\big(h(v')-h(v)\big) \\
	& =\big(\mu^{\frac{1}{2}}(v'_*)-\mu^{\frac{1}{2}}(v_*)\big)\big(h(v')-h(v)\big)+\big(\mu^{\frac{1}{2}}(v'_*)-\mu^{\frac{1}{2}}(v_*)\big)h(v)+\mu^{\frac{1}{2}}(v_*)\big(h(v')-h(v)\big).
\end{align*}
Correspondingly, we write \eqref{Gaest} as
\begin{multline*}
	(\Gamma(f,g),h)_{L^2_v}=\int_{\R^6}\int_{\S^2}B\big(\mu^{\frac{1}{2}}(v'_*)-\mu^{\frac{1}{2}}(v_*)\big)\big(h(v')-h(v)\big)f(v_*)g(v)\,d\sigma dv_*dv\\
	+\int_{\R^6}\int_{\S^2}B\big(\mu^{\frac{1}{2}}(v'_*)-\mu^{\frac{1}{2}}(v_*)\big)h(v)f(v_*)g(v)\,d\sigma dv_*dv\\
	+\int_{\R^6}\int_{\S^2}B\mu^{\frac{1}{2}}(v_*)\big(h(v')-h(v)\big)f(v_*)g(v)\,d\sigma dv_*dv
	=\Gamma_1+\Gamma_2+\Gamma_3.
\end{multline*}
By estimates \eqref{HHbcos} and \eqref{gafg}, $\Gamma_2$ and $\Gamma_3$ can be estimated as
\begin{multline*}
	|\Gamma_2|+|\Gamma_3|
	\le \int_{\R^6}\big(|v_*-v|^{\ga+2s-1}\1_{|v-v_*|\ge\frac{2}{\pi}}+|v-v_*|^{\ga+1}\1_{|v-v_*|<\frac{2}{\pi}}+|v-v_*|^{\ga+2s}\big)\\
	\times\Big(\|[\mu^{\frac{1}{2}},\,\na_v\mu^{\frac{1}{2}},\,\na^2_v\mu^{\frac{1}{2}}]\|_{L^\infty_v}|h(v)|+\mu^{\frac{1}{2}}(v_*)\|[h,\,\na_vh,\,\na^2_vh]\|_{L^\infty_v}\Big)|f(v_*)g(v)|\,d\sigma dv_*dv\\
	\le C\|[h,\,\na_vh,\,\na^2_vh]\|_{L^\infty_v}\|\<v\>^{2+(\ga+2s)_+}f\|_{L^2_v}\|\<v\>^{2+(\ga+2s)_+}g\|_{L^2_v}.
\end{multline*}
For the term $\Gamma_1$, note from \eqref{vpriminv} that
\begin{align}\begin{aligned}
		\label{322}
		|\mu^{\frac{1}{2}}(v'_*)-\mu^{\frac{1}{2}}(v_*)| & \le \|\na_v\mu^{\frac{1}{2}}\|_{L^\infty_v}|v-v_*|\sin\frac{\th}{2}, \\
		|h(v')-h(v)| & \le\|\na_vh\|_{L^\infty_v}|v-v_*|\sin\frac{\th}{2}.
	\end{aligned}
\end{align}
Then by \eqref{bcossin} and \eqref{gafg}, we have
\begin{align*}
	|\Gamma_1| & \le\int_{\R^6}\int_{\S^2}|v-v_*|^{\ga}b(\cos\th)\min\big\{\sin^2\frac{\th}{2}|v-v_*|^2,1\big\}\|[h,\,\na_vh]\|_{L^\infty_v}|f(v_*)g(v)|\,d\sigma dv_*dv \\
	& \le C\|[h,\,\na_vh]\|_{L^\infty_v}\|\<v\>^{2+(\ga+2s)_+}f\|_{L^2_v}\|\<v\>^{2+(\ga+2s)_+}g\|_{L^2_v}.
\end{align*}
Combining the above estimates for $\Gamma_i$ $(i=1,2,3)$, we obtain \eqref{Gammavp}.


%

\smallskip\noindent{\bf Estimating \eqref{Gammavp1}.}
For the estimate \eqref{Gammavp1}, we write \eqref{Gaest} as
\begin{align*}
	(\Gamma(f,g),h)_{L^2_v}
	& =\int_{\R^6}\int_{\S^2}B\big(\mu^{\frac{1}{2}}(v'_*)h(v')-\mu^{\frac{1}{2}}(v_*)h(v)\big)f(v_*)g(v)\,d\sigma dv_*dv \\
	& =
	\int_{\R^6}\int_{\S^2}B\big(\mu^{\frac{1}{2}}(v'_*)-\mu^{\frac{1}{2}}(v_*)\big)f(v_*)\big(g(v)-g(v')\big)h(v')\,d\sigma dv_*dv \\
	& \quad+\int_{\R^6}\int_{\S^2}B\big(\mu^{\frac{1}{2}}(v'_*)-\mu^{\frac{1}{2}}(v_*)\big)f(v_*)g(v')h(v')\,d\sigma dv_*dv \\
	& \quad+\int_{\R^6}\int_{\S^2}B\mu^{\frac{1}{2}}(v_*)f(v_*)\big(g(v)-g(v'))h(v')\,d\sigma dv_*dv \\
	& \quad+\int_{\R^6}\int_{\S^2}B\mu^{\frac{1}{2}}(v_*)f(v_*)\big(g(v')h(v')-g(v)h(v)\big)\,d\sigma dv_*dv \\
	& =\ti\Ga_1+\ti\Ga_2+\ti\Ga_3+\ti\Ga_4.
\end{align*}
For the term $\ti\Ga_1$, we use \eqref{322} and regular change of variable \eqref{regular} to deduce
\begin{align*}
	|\ti\Ga_1| & \le \int_{\R^6}\int_{\S^2}\Big\{4\|g\|_{L^\infty_v},\,C|v-v_*|^2\sin^2\frac{\th}{2}\|\na_vg\|_{L^\infty_v}\Big\}|v-v_*|^{\ga}b(\cos\th)|f(v_*)h(v')|\,d\sigma dv_*dv \\
	& \le C\|[g,\na_vg]\|_{L^\infty_v}\int_{\R^6}|v-v_*|^{\ga+2s}|f(v_*)h(v)|\,dv_*dv \\
	& \le C\|[g,\na_vg]\|_{L^\infty_v}\|\<v\>^{2+(\ga+2s)_+}f\|_{L^2_v}\|\<v\>^{(\ga+2s)_+}h\|_{L^1_v},
\end{align*}
where we used \eqref{bcossin} and \eqref{gafg}.
For the terms $\ti\Ga_2$ and $\ti\Ga_3$, we apply \eqref{Hpristar} and \eqref{Hpristar2} to obtain
\begin{align*}
	|\ti\Ga_2|+|\ti\Ga_3|
	& \le C\|\<v\>^{2+(\ga+2s)_+}f\|_{L^2_v}\big(\|[g,\na_vg,\na^2_vg]\|_{L^\infty_v}\|\<v\>^{(\ga+2s)_+}h\|_{L^1_{v}}+\|\<v\>^{(\ga+2s)_+}gh\|_{L^1_{v}}\big) \\
	& \le C\|[g,\na_vg,\na^2_vg]\|_{L^\infty_v}\|\<v\>^{2+(\ga+2s)_+}f\|_{L^2_v}\|\<v\>^{(\ga+2s)_+}h\|_{L^1_{v}}.
\end{align*}
For the term $\ti\Ga_4$, we apply regular change of variable \eqref{regular} and \eqref{bcossink2} to deduce
\begin{align*}
	|\ti\Ga_4| & =\Big|\int_{\R^6}\int_{\S^2}|v-v_*|^\ga b(\cos\th)\mu^{\frac{1}{2}}(v_*)f(v_*)\big(\frac{1
	}{{\cos^{3+\ga}\frac{\th}{2}}}-1\big)g(v)h(v)\,d\sigma dv_*dv\Big| \\
	& \le C\int_{\R^6}|v-v_*|^\ga\mu^{\frac{1}{2}}(v_*)|f(v_*)||g(v)h(v)|\,dv_*dv\\
	& \le C\|\mu^{\frac{1}{4}}f\|_{L^2_v}\|\<v\>^{{\ga}_+}g\|_{L^\infty_v}\|\<v\>^{{\ga}_+}h\|_{L^1_v},
\end{align*}
where we used \eqref{gafg}.
The above estimates imply \eqref{Gammavp1}.

%
%
%

\smallskip\noindent{\bf Estimating \eqref{64eq1} and \eqref{64eq2}.}
It follows from \cite[Proposition 6.10]{Morimoto2016} that for any $l\in\R$, 
\begin{align}
	\label{64eq}
	|(\Gamma(f,g),h)_{L^2_v}|\le C\|f\|_{L^2_v}\|\<v\>^{(l+\ga+2s)_+}g\|_{L^2_v}\|\<v\>^{-l}h\|_{H^{2s}}, 
\end{align}
where $(l+\ga+2s)_+=\max\{l+\ga+2s,0\}$. This is \eqref{64eq1}. 
Using \eqref{25} and \eqref{64eq}, for any $k\ge 0$, we have 
\begin{multline*}
	|(\Gamma(f,g),\<v\>^{2k}h)_{L^2_v}|
	\le |(\Gamma(f,\<v\>^{k}g),\<v\>^{k}h)_{L^2_v}|+|(\Gamma(f,g),\<v\>^{2k}h)_{L^2_v}-(\Gamma(f,\<v\>^{k}g),\<v\>^{k}h)_{L^2_v}|\\
	\le 
	C\|f\|_{L^2_v}\|\<v\>^{k+(l+\ga+2s)_+}g\|_{L^2_v}\|\<v\>^{k-l}h\|_{H^{2s}_v}
	+C\|[f,\<v\>^{\frac{\ga+2s}{2}}f]\|_{L^2_v}\|\<v\>^{k+\frac{\ga}{2}}g\|_{L^2_v}\|\<v\>^{k}h\|_{L^2_D}. 
\end{multline*}
Set $l=-4$ and notice from \eqref{esD} that $\|\<v\>^{k}h\|_{L^2_D}\le \|\<v\>^{k+\frac{\ga+2s}{2}}h\|_{H^{2s}_v}\le\|\<v\>^{k+2}h\|_{H^{2s}_v}$, we obtain 
\begin{align*}
	|(\Gamma(f,g),\<v\>^{2k}h)_{L^2_v}|
	\le 
	C\|\<v\>^{\frac{(\ga+2s)_+}{2}}f\|_{L^2_v}\|\<v\>^{k+\frac{\ga}{2}}g\|_{L^2_v}\|\<v\>^{k+4}h\|_{H^{2s}_v}.
\end{align*}
This completes the proof of Lemma \ref{vpmu12Lem}.
\end{proof}

\subsection{\texorpdfstring{$L^2$}{L2} estimate for the strong singularity}\label{secStrong}
For the case of strong singularity $s\in[\frac{1}{2},1)$, we truncate the collision kernel as in \eqref{beta} and denote the truncated collision operator $\Ga_\eta$ as in \eqref{Gaeta}. After the truncation, the calculation for local-in-time existence is the same as in the case of weak singularity. 
To take the limit from the weak singularity to the strong singularity, we need the following $L^2$ estimate and convergence properties. 

\begin{Lem}\label{GaetaesLem}
	Let $\ga\in(-\frac{3}{2},2]$, $s\in[\frac{1}{2},1)$ and $\eta\in(0,1)$. Let $b_\eta(\cos\th)$ and $\Ga_\eta$ be defined in \eqref{beta} and \eqref{Gaeta} respectively. Then for suitable functions $f,g,h$ and $k\ge 0$, we have 
	\begin{align}\label{Gaetaes}
		\big|\big(\Ga_\eta(f,g),\<v\>^{2k}h\big)_{L^2_v}\big|\le C\|\<v\>^2f\|_{L^2_v}\|\<v\>^kg\|_{L^2_D}\|\<v\>^kh\|_{L^2_D},
	\end{align}
	and 
	\begin{align}\label{Gaetaes1}
		|(\Gamma_\eta(f,g),h)_{L^2_v}|\le C\|[h,\,\na_vh,\,\na^2_vh]\|_{L^\infty_v}\|\<v\>^{2+(\ga+2s)_+}f\|_{L^2_v}\|\<v\>^{2+(\ga+2s)_+}g\|_{L^2_v},
	\end{align}
	where $C>0$ is independent of $\eta$. 
	Moreover, for any fixed $h\in C^\infty_c(\R^7_{t,x,v})$ and $f,g$ satisfying $\|\<v\>^6[f,g]\|_{L^\infty_t([0,T])L^2_x(\Omega)L^2_v}<\infty$, we have 
	\begin{align}\label{Gaetalimit}
		\lim_{\eta\to0}\big(\Ga_\eta(f,g)-\Ga(f,g),h\big)_{L^2_t([0,T])L^2_x(\Omega)L^2_v}=0. 
	\end{align}
\end{Lem}
\begin{proof}
	Since $b_\eta(\cos\th)\le b(\cos\th)$, the same calculations for the upper bound of collision term $\Ga$ can be applied to $\Ga_\eta$. Therefore, using the method in \cite[Eq. (6.6), pp. 817]{Gressman2011} and \eqref{esD}, we can obtain the estimate \eqref{Gaetaes} for $\Ga_\eta$ when $k=0$, which is the same as $\Ga$; we omit the details for brevity. For the estimate \eqref{Gaetaes} with the case $k>0$ and estimate \eqref{Gaetaes1}, one can follow \eqref{214} and \eqref{Gammavp} respectively; note that only the upper bound of $b_\eta(\cos\th)\le b(\cos\th)$ is involved in these estimates. 
	
	\smallskip To prove the limit \eqref{Gaetalimit}, we use the pre-post change of variable and write 
	\begin{align}\label{821}\notag
		\big(\Ga_\eta(f,g)-\Ga(f,g),h&\big)_{L^2_t([0,T])L^2_x(\Omega)L^2_v}=\int_{[0,T]\times\Omega}\int_{\R^6}\int_{\S^2}|v-v_*|^\ga \big(b_\eta(\cos\th)-b(\cos\th)\big)\\&\qquad\times\big(\mu^{\frac{1}{2}}(v'_*)h(v')-\mu^{\frac{1}{2}}(v_*)h(v)\big)f(v_*)g(v)\,d\sigma dv_*dvdxdt.
	\end{align}
	We will make a rough estimate as follows. Notice that 
	\begin{align}\label{821a}\notag
		&\mu^{\frac{1}{2}}(v'_*)h(v')-\mu^{\frac{1}{2}}(v_*)h(v)\\
		&\notag\quad=\big(\mu^{\frac{1}{2}}(v'_*)-\mu^{\frac{1}{2}}(v_*)\big)\big(h(v')-h(v)\big)+\big(\mu^{\frac{1}{2}}(v'_*)-\mu^{\frac{1}{2}}(v_*)\big)h(v)+\mu^{\frac{1}{2}}(v_*)\big(h(v')-h(v)\big)\\
		&\notag\quad=\na_v\mu^{\frac{1}{2}}(\bar v_*)\cdot(v'_*-v_*)\na_vh(\bar v)\cdot(v'-v)\\
		&\notag\quad\quad+\na_v\mu^{\frac{1}{2}}(v_*)\cdot(v'_*-v_*)h(v)+\na_v^2\mu^{\frac{1}{2}}(\bar v'_*):(v'_*-v_*)\otimes(v'_*-v_*)h(v)\\
		&\quad\quad+\mu^{\frac{1}{2}}(v_*)\na_vh(v)\cdot(v'-v)+\mu^{\frac{1}{2}}(v_*)\na_v^2h(\bar v'):(v'-v)\otimes(v'-v),
	\end{align}
	for some $\bar v'_*,\bar v'$. 
	For the first-order terms, we use \eqref{vprimeth} to obtain 
	\begin{align}\label{821b}
		\begin{aligned}
			v'_*-v_*&=\sin^2\frac{\th}{2}(v-v_*)-\frac{1}{2}|v-v_*|\sin\th\omega,\\
			v'-v&=\sin^2\frac{\th}{2}(v_*-v)+\frac{1}{2}|v-v_*|\sin\th\omega.
		\end{aligned}
	\end{align}
	where $\omega\in\S^1(\mathbf{k})$ satisfies $\sigma=\cos\th\,\mathbf{k}+\sin\th\,\omega$ with $\mathbf{k}=\frac{v-v_*}{|v-v_*|}$. 
	By choosing $\mathbf{k}$ as the north pole, we can write $\omega=(\cos\phi,\sin\phi,0)$ with $\phi\in[0,2\pi]$ and hence, by the symmetric about $\phi$, the integrals involving $\omega$ vanish as 
	\begin{align*}
		\int_{\S^2}\big(b_\eta(\cos\th)-b(\cos\th)\big)\sin\th\omega\,d\sigma
		\int_{0}^{\frac{\pi}{2}}\int_{0}^{2\pi}\big(b_\eta(\cos\th)-b(\cos\th)\big)\sin\th(\cos\phi,\sin\phi,0)\,d\phi d\th=0.
	\end{align*}
	Therefore, using \eqref{vpriminv}, and combining \eqref{821a} and \eqref{821b}, the remaining terms in \eqref{821} satisfy 
	\begin{multline}\label{822}
		\big|\big(\Ga_\eta(f,g)-\Ga(f,g),h\big)_{L^2_t([0,T])L^2_x(\Omega)L^2_v}\big|
		\le C\int_{[0,T]\times\Omega}\int_{\R^6}\int_{\S^2}\|[h,\na_vh,\na_v^2h]\|_{L^\infty_v}\\
		\times|v-v_*|^{\ga+2}\sin^2\frac{\th}{2} \big|b_\eta(\cos\th)-b(\cos\th)\big||f(v_*)||g(v)|\,d\sigma dv_*dvdxdt.
	\end{multline}
	To apply the Dominated Convergence Theorem, since $b_\eta(\cos\th)\le b(\cos\th)$ (from \eqref{betab}), we estimate the integrand of the right-hand side of \eqref{822} as 
	\begin{align*}
		2|v-v_*|^{\ga+2}\sin^2\frac{\th}{2}b(\cos\th)f(v_*)g(v),
	\end{align*}
	which is independent of $\eta$. This function is in $L^1$ by using \eqref{bcossin}:
	\begin{align*}
		\int_{\R^6}\int_{\S^2}|v-v_*|^{\ga+2}\sin^2\frac{\th}{2}b(\cos\th)f(v_*)g(v)\,d\sigma dv_*dv
		&\le C_s\int_{\R^6}|v-v_*|^{\ga+2}f(v_*)g(v)\,dv_*dv\\
		&\le C_s\|\<v\>^4f\|_{L^1_v}\|\<v\>^4g\|_{L^1_v}.
	\end{align*}
	Thus, applying Dominated Convergence Theorem in \eqref{822} and using \eqref{beta}, we obtain 
	\begin{align*}
		\lim_{\eta\to0}\big|\big(\Ga_\eta(f,g)-\Ga(f,g),h\big)_{L^2_t([0,T])L^2_x(\Omega)L^2_v}\big|=0. 
	\end{align*}
	This completes the proof of Lemma \ref{GaetaesLem}. 
\end{proof}

\subsection{\texorpdfstring{$L^2$}{L2} estimate of the regularizing operator}
In this Subsection, we analyze the vanishing regularizing operator $V$ given in \eqref{Vf}. 
It's direct to obtain its $L^2$ estimate as the following. 

\begin{Lem}\label{LemRegu}
	Let $V$ be the operator defined by \eqref{Vf}. 
	For any $k\in\R$, by choosing $\wh{C}_0>0$ large enough (depending on $k$), we have 
	\begin{align*}
		(Vf,\<v\>^{2k}f)_{L^2_v}&\le
		-\|[\wh{C}^{}_0\<v\>^{{k+4}}f,\<v\>^{{k+2}}\na_vf]\|_{L^2_v}^2. 
	\end{align*}
	Additionally, if $f\ge 0$ and $\wh{C}_0>0$ is large enough (depending on $k$), then 
	\begin{align*}
		(V\<v\>^{-k},\<v\>^{2k}f)_{L^2_v}
		&\le -\wh{C}^2_0\|\<v\>^{k+{8}}f\|_{L^1_v}. 
	\end{align*}
\end{Lem}
\begin{proof}
	Taking $L^2$ inner product of $Vf$ with $\<v\>^{2k}f$ over $\R^3_v$, we have 
	\begin{align*}
		(Vf,\<v\>^{2k}f)_{L^2_v}&=\big(-2\wh{C}^2_0\<v\>^{{8}}f+2\na_v\cdot(\<v\>^{{4}}\na_v)f,\<v\>^{2k}f\big)_{L^2_v}\\
		&=-2\wh{C}^2_0\|\<v\>^{{k+4}}f\|_{L^2_v}^2-2\wh{C}_0\int_{\R^3}\<v\>^{{4}}\na_vf\cdot\big(\na_v\<v\>^{k}\<v\>^{k}f+\<v\>^{k}\na_v(\<v\>^{k}f)\big)\,dv.
	\end{align*}
	Notice that $\na_v\<v\>^k=kv\<v\>^{k-2}$.
	Then by Cauchy-Schwarz inequality and choosing $\wh{C}_0>0$ large enough (depending only on $k$), we have 
	\begin{align*}
		(Vf,\<v\>^{2k}f)_{L^2_v}&\le
		-\|[\wh{C}^{}_0\<v\>^{{k+4}}f,\<v\>^{{k+2}}\na_vf]\|_{L^2_v}^2.
	\end{align*} 
	If $f\ge 0$, 
	choosing $\wh{C}_0>0$ small enough, we have 
	\begin{align*}
		(V\<v\>^{-k},\<v\>^{2k}f)_{L^2_v}&=\big(-2\wh{C}^2_0\<v\>^{{8}}\<v\>^{-k}+2\na_v\cdot(\<v\>^{{4}}\na_v)\<v\>^{-k},\<v\>^{2k}f\big)_{L^2_v}\\
		&\le -\wh{C}^2_0\|\<v\>^{k+{8}}f\|_{L^1_v}. 
	\end{align*}
	This completes the proof of Lemma \ref{LemRegu}. 
\end{proof}

\subsection{Weak limit for collision term}
To obtain the existence of the nonlinear problem, we need to use the weak-$*$ limit to approximate the final solution. As a preparation, we need to derive the following limit. 
\begin{Lem}\label{claimLem}
	Fix $T>0$, $k>2$ and any function $\Phi\in C^\infty_c(\R^7)$. Assume that $f$ is the weak-$*$ limit of $f^n$ in the sense that
	\begin{align}\label{weaklimitz}
		\begin{aligned}
			& f^n\rightharpoonup f\quad \text{ weakly-$*$ in $L^2_x(\Omega)L^2_v$ for any $t\in(0,T]$},
		\end{aligned}
	\end{align}
	which satisfy
	\begin{multline}\label{84cg}
		\int^T_0\int_{\Omega}\|\Phi\|_{W^{2,\infty}_v}\|\<v\>^{k}[f^{n_j},f]\|_{L^2_v}\,dxdt+
		\sup_{0\le t\le T}\|[f,f^n]\|_{L^2_x(\Omega)L^2_v}^2
		\\+c_0\int^T_0\|[f,f^n]\|_{L^2_x(\Omega)L^2_D}^2\,dt
		\le C_\Phi<\infty,
	\end{multline}
	with some constant $C_\Phi>0$ uniformly in $n$, which can depend on $\Phi$. 
	Then there exists a subsequence $\{f^{n_j}\}\subset\{f^{n}\}$ such that 
	\begin{align}\label{86ab}
		\lim_{n_j\to\infty}\int^T_0\int_{\Omega}\|\Phi\|_{W^{2,\infty}_v}\|\<v\>^{k-2}
		(f^{n_j}-f)
		\|_{L^2_v}\,dxdt=0.
	\end{align}
\end{Lem}
\begin{proof}
	We use the method of Rellich-Kondrachov theorem to prove \eqref{86ab}.
	For any $\ve\in(0,1)$, 
	let $\rho(v)\in C^\infty_c(\R^3_v)$ be a positive smooth function with compact support in $\R^3_v$ such that $\|\rho\|_{L^1_v}=1$, and $\rho_\ve(v)=\ve^{-3}\rho(\ve^{-1}v)$. Also, we let $\chi_R\in C^\infty_c(\R^3_v)$ be a smooth cutoff function with $\chi_R(v)=1$ for $|v|\le R$ and $\chi_R(v)=0$ for $|v|\ge 2R$.
	Choosing $R=\ve^{-\frac{s}{2(k-2)}}$, we have
	\begin{align}\label{88}\notag
		\|\<v\>^{k-2}(f^{n}-f)\|_{L^2_v}
		&\le\|\<v\>^{k-2}(1-\chi_R)(f^{n}-f)\|_{L^2_v}+\ve^{-\frac{s}{2}}\|\chi_R(f^{n}-f)\|_{L^2_v}\\ 
		&\notag\le\ve^{\frac{s}{k-2}}\|\<v\>^{k}(f^{n}-f)\|_{L^2_v}+\ve^{-\frac{s}{2}}\|\chi_R(f^{n}-f)-\rho_\ve*(\chi_R(f^{n}-f))\|_{L^2_v}\\
		&\quad+\ve^{-\frac{s}{2}}\|\rho_\ve*(\chi_R(f^{n}-f))\|_{L^2_v},
	\end{align}
	where $\rho_\ve*(\chi_R(f^{n}-f))$ is the convolution defined by
	\begin{align*}
		\rho_\ve*(\chi_R(f^{n}-f))=\int_{\R^3}\rho_\ve(v_*)(\chi_R(f^{n}-f))(v-v_*)\,dv.
	\end{align*}
	For the second right hand term of \eqref{88}, by $\|\rho_\ve\|_{L^1_v}=1$, Minkowski's integral inequality and H\"{o}lder's inequality, we have
	\begin{align*}
		& \ve^{-\frac{s}{2}}\|\chi_R(f^{n}-f)-\rho_\ve*(\chi_R(f^{n}-f))\|_{L^2_v} \\
		& =\ve^{-\frac{s}{2}}\Big(\int_{\R^3}\Big|\int_{\R^3}\rho_\ve(v_*)\big[(\chi_R(f^{n}-f))(v)-(\chi_R(f^{n}-f))(v-v_*)\big]\,dv_*\Big|^2\,dv\Big)^{\frac{1}{2}} \\
		& \le\ve^{-\frac{s}{2}}\int_{\R^3}\rho_\ve(v_*)\Big(\int_{\R^3}|(\chi_R(f^{n}-f))(v)-(\chi_R(f^{n}-f))(v-v_*)|^2\,dv\Big)^{\frac{1}{2}}\,dv_* \\
		& \le\ve^{-\frac{s}{2}}\Big(\int_{\R^3}|\rho_\ve(v_*)|^2|v_*|^{3+2s}\,dv_*\Big)^{\frac{1}{2}}\Big(\int_{\R^6}\frac{|(\chi_R(f^{n}-f))(v)-(\chi_R(f^{n}-f))(v-v_*)|^2}{|v_*|^{3+2s}}\,dvdv_*\Big)^{\frac{1}{2}} \\
		& \le\ve^{\frac{s}{2}}C\|\<D_v\>^s(\chi_R(f^{n}-f))\|_{L^2_v(\R^3_v)}.
	\end{align*}
	Note that $\<D_v\>^s\chi_R$ can be regarded as a pseudo-differential operator with symbol in $S(\<v\>^{\frac{\ga}{2}}\<\eta\>^s)$; see \cite{Lerner2010} for more details. That is, $S(\<v\>^{\frac{\ga}{2}}\<\eta\>^s)$ consists of functions $a(v,\eta)$ such that for $\alpha,\beta\in \N^3$, $v,\eta\in\Rd$,
	\begin{align}\label{symbol}
		|\partial^\alpha_v\partial^\beta_\eta a(v,\eta)|\le C_{\alpha,\beta}\<v\>^{\frac{\ga}{2}}\<\eta\>^s,
	\end{align}
	and $\<D_v\>^s(\chi_Rf)$ can be written as the pseudo-differential operator in the form
	\begin{align*}
		\int_\Rd\int_\Rd e^{2\pi i (x-y)\cdot\xi}a\big(\frac{v+v_*}{2},\eta\big)f(v_*)\,dv_*d\eta,
	\end{align*}
	for some $a\in S(\<v\>^{\frac{\ga}{2}}\<\eta\>^s)$.
	Then by boundedness $\|\<D_v\>^s(\chi_R\vp)\|_{L^2_v}\le C\|\<v\>^{\frac{\ga}{2}}\<D_v\>^s\vp\|_{L^2_v}$ from \cite[Lemma 2.3]{Deng2020a} and \eqref{esD}, we have
	\begin{align}\label{819}\notag
		\ve^{-\frac{s}{2}}\|\chi_R(f^{n}-f)-\rho_\ve*(\chi_R(f^{n}-f))\|_{L^2_v}
		& \le \ve^{\frac{s}{2}}C\|\<D_v\>^s(\chi_R(f^{n}-f))\|_{L^2_v(\R^3_v)} \\
		& \le \ve^{\frac{s}{2}}C\|f^{n}-f\|_{L^2_D(\R^3_v)}.
	\end{align}
	For the third right-hand term of \eqref{88}, we know that
	$\rho_\ve*(\chi_R(f^{n}-f))$ is uniformly-in-$n$ bounded and equicontinuous for each fixed $\ve>0$, and almost every $(t,x)\in[0,T]\times\Omega$, since
	\begin{align*}
		\|\rho_\ve*(\chi_R(f^{n}-f))\|_{L^\infty_v}\le\|\rho_\ve\|_{L^\infty_v}\|\chi_R(f^{n}-f)\|_{L^1_v}\le C_{\ve,x},
	\end{align*}
	and
	\begin{align*}
		\|\na_v\{\rho_\ve*(\chi_R(f^{n}-f))\}\|_{L^\infty_v}\le\|\na_v\rho_\ve\|_{L^\infty_v}\|\chi_R(f^{n}-f)\|_{L^1_v}\le C_{\ve,x},
	\end{align*}
	where we used $\|\chi_R(f^{n}-f)\|_{L^1_v}\le \|f^{n}-f\|_{L^2_v}$ and \eqref{84cg} to obtain its almost-everywhere finiteness. Also, $\rho_\ve*(\chi_R(f^{n}-f))$ has uniformly-in-$n$ compact support in $v\in\R^3$ for any fixed $\ve>0$, since $\rho_\ve$ and $\chi_R$ have compact support that is independent of $n$ and $x$. Thus, by Arzel\`a–Ascoli Theorem, there exists a subsequence $\{f^{n_j}\}\subset\{f^n\}$ such that $\rho_\ve*(\chi_R(f^{n_j}-f))$ converges in $L^\infty_v$ as $n_j\to\infty$ for almost every $(t,x)\in[0,T]\times\Omega$. On the other hand, by the weak-$*$ convergence \eqref{weaklimitz}, we have
	\begin{multline*}
		\int^T_0\int_{\Omega}\vp(t,x)\rho_\ve*(\chi_R(f^{n}-f))(v)\,dxdt\\=\int^T_0\int_{\Omega}\int_{\R^3}\vp(t,x)\rho_\ve(v-v_*)\chi_R(v_*)(f^{n}-f)(v_*)\,dv_*dxdt\to 0,
	\end{multline*}
	as $n\to\infty$ for any test function $\vp\in C^\infty_c(\R^4_{t,x})$.
	Thus, by the uniqueness of the limit, for any fixed $\ve>0$, and almost every $(t,x)\in[0,T]\times\Omega$, 
	\begin{align*}
		\text{$\rho_\ve*(\chi_R(f^{n_j}-f))$ converges to $0$ in $L^\infty_v$ as $n_j\to\infty$.}
	\end{align*} 
	(One can prove this via contradiction by assuming it doesn't converge to $0$ for $(t,x)\in A$ with a non-zero measure set $A\subset[0,T]\times\Omega$). Therefore, since $\rho_\ve*(\chi_R(f^{n_j}-f))$ has a uniform-in-$n$ compact support, for any fixed $\ve>0$, and almost every $(t,x)\in[0,T]\times\Omega$, we have
	\begin{align}\label{820b}
		\|\rho_\ve*(\chi_R(f^{n_j}-f))\|_{L^2_v}
		\to 0,
	\end{align}
	as $n_j\to\infty$.
	Substituting \eqref{88} and \eqref{819} into the left-hand side of \eqref{86ab}, and choosing $n=n_j$ as the above, we have
	\begin{align}\label{820}
		&\notag\lim_{n_j\to\infty}\int^T_0\int_{\Omega}\|\Phi\|_{W^{2,\infty}_v}\|\<v\>^{k-2}(f^{n_j}-f)\|_{L^2_v}\,dxdt\\
		&\notag\le \lim_{n_j\to\infty}\int^T_0\int_{\Omega}\|\Phi\|_{W^{2,\infty}_v}
		\Big(\ve^{\frac{s}{k-2}}\|\<v\>^{k}(f^{n_j}-f)\|_{L^2_v}+\ve^{\frac{s}{2}}C\|f^{n_j}-f\|_{L^2_D}\\
		&\notag\quad+\ve^{-\frac{s}{2}}\|\rho_\ve*(\chi_R(f^{n_j}-f))\|_{L^2_v}\Big)\,dxdt\\
		&\notag\le \ve^{\frac{s}{k-2}}C_\Phi
		+\ve^{\frac{s}{2}}C\|\Phi\|_{L^2_tL^2_xW^{2,\infty}_v([0,T]\times\Omega\times\R^3_v)}\\
		&\quad+\ve^{-\frac{s}{2}}\lim_{n_j\to\infty}\int^T_0\int_{\Omega}\|\Phi\|_{W^{2,\infty}_v}\|\rho_\ve*(\chi_R(f^{n_j}-f))\|_{L^2_v}\,dxdt,
	\end{align}
	where we used \eqref{84cg} in the last inequality. For the last term, note from \eqref{820b} that for sufficiently large $n_j>1$ and fixed $\ve\in(0,1)$, 
	\begin{align*}
		\|\Phi\|_{W^{2,\infty}_v}\|\rho_\ve*(\chi_R(f^{n_j}-f))\|_{L^2_v}\le C\|\Phi\|_{W^{2,\infty}_v},
	\end{align*}
	where the right hand side is integrable over $[0,T]\times\Omega$. Then by Dominated Convergence Theorem and \eqref{820b},
	\begin{align*}
		\lim_{n_j\to\infty}\int^T_0\int_{\Omega}\|\Phi\|_{W^{2,\infty}_v}\|\rho_\ve*(\chi_R(f^{n_j}-f))\|_{L^2_v}\,dxdt=0,
	\end{align*}
	for any $\ve\in(0,1)$.
	Then \eqref{820} becomes
	\begin{align*}
		\lim_{n_j\to\infty}\int^T_0\int_{\Omega}\|\Phi\|_{W^{2,\infty}_v}\|\<v\>^{k-2}(f^{n_j}-f)\|_{L^2_v}\,dxdt
		\le \ve^{\frac{s}{k-2}}C_\Phi
		+\ve^{\frac{s}{2}}C\|\Phi\|_{L^2_tL^2_xW^{2,\infty}_v([0,T]\times\Omega\times\R^3_v)}.
	\end{align*}
	Since this inequality holds for any $\ve>0$, we let $\ve\to0$ and deduce \eqref{86ab}.
	This completes the proof of Lemma \ref{claimLem}.
\end{proof}

\subsection{Non-negativity of \texorpdfstring{$F$}{F}}
In this Subsection, we will prove the non-negativity of $F=\mu+\mu^{\frac{1}{2}}f$.
\begin{Thm}[Non-negativity]\label{positivity}
	Let $M,\vpi\ge 0$. Assume $F_0=\mu+\mu^{\frac{1}{2}}f_0\ge 0$. Assume that
	\begin{align}
		\label{4135}
		\sup_{0\le t\le T}\|\<v\>^4\psi(t)\|_{L^\infty_{x}(\Omega)L^\infty_v(\R^3_v)}\,dt & \le\de_0,
	\end{align}
	with sufficiently small $\de_0>0$.
	Fix any $\vpi\ge 0$ and $\ve\in[0,1)$. Denote operator $V$ by \eqref{Vf}
	\begin{align*}
		Vf=-2\wh{C}^2_0\<v\>^{{8}}f+2\na_v\cdot(\<v\>^{{4}}\na_v)f, 
	\end{align*}
	with some sufficiently large constant $\wh{C}_0>0$. 
	Let $f$ be the solution to
	\begin{align}\label{fLinearpo}
		\left\{
		\begin{aligned}
			& \pa_tf+ v\cdot\na_xf = \vpi Vf+ \Gamma(\mu^{\frac{1}{2}}+\psi,\mu^{\frac{1}{2}}+f)-Mf \quad \text{ in } (0,T]\times\Omega\times\R^3_v, \\
			& f(0,x,v)=f_0\quad \text{ in }\Omega\times\R^3_v,
		\end{aligned}\right.
	\end{align}
	with the inflow boundary condition \eqref{inflow} satisfying $G(t)=\mu+\mu^{\frac{1}{2}}g(t)\ge 0$ or the Maxwell-reflection boundary condition $f(t,x,v)|_{\Si_-}=(1-\ve)Rf$ with any fixed $0\le \ve<1$.
	Then $F(t)=\mu+\mu^{\frac{1}{2}}f(t)\ge 0$ in $\Omega\times\R^3_v$ for $t\in[0,T]$.
\end{Thm}
\begin{proof}
	Writing $F=\mu+\mu^{\frac{1}{2}}f$ and $\Psi=\mu^{\frac{1}{2}}+\psi$, we rewrite \eqref{fLinearpo} as
	\begin{align}\label{536}
		\left\{
		\begin{aligned}
			& \pa_tF+ v\cdot\na_xF = \mu^{\frac{1}{2}}\vpi V(\mu^{-\frac{1}{2}}F-\mu^{\frac{1}{2}}) + Q(\mu+\mu^{\frac{1}{2}}\psi,F)
			\\&\qquad\qquad\qquad-M(\mu^{-\frac{1}{2}}F-\mu^{\frac{1}{2}}) \quad \text{ in } (0,T]\times\Omega\times\R^3_v, \\
			& f(0,x,v)=f_0\quad \text{ in }\Omega\times\R^3_v.
		\end{aligned}\right.
	\end{align}
	Denote $F_+:=\max\{F,0\}$ and $F_-:=-\min\{F,0\}$.
	Then one has $F=F_+-F_-$, and 
	\begin{align*}
		\frac{d(x_-)^2}{dx} = -2x_-,\quad
		\na_{t,x}|F_-|^2=-2F_{-}\na_{t,x}F,
		\quad F_{-}|_{t=0}=0.
	\end{align*}
	Then it suffices to show that $F_-(t)=0$ for $t\in[0,T]$.
	Notice from $F=F_+-F_-$ that
	\begin{align*}
		& \int_{\R^3}Q(\mu^{\frac{1}{2}}\Psi,F)\mu^{-1}F_-(v)\,dv \\
		& \quad=-\int_{\R^3}Q(\mu^{\frac{1}{2}}\Psi,F_-)\mu^{-1}F_-(v)\,dv+\int_{\R^3}Q(\mu^{\frac{1}{2}}\Psi,F_+)\mu^{-1}F_-(v)\,dv \\
		& \quad=-\int_{\R^3}Q(\mu^{\frac{1}{2}}\Psi,F_-)\mu^{-1}F_-(v)\,dv+\int_{\R^3}\int_{\R^3}\int_{\S^2}B(\mu^{\frac{1}{2}}\Psi)_{*}'(F_+)'F_-(v)\mu^{-1}\,d\sigma dv_*dv
		\\&\quad\qquad
		-\int_{\R^3}\int_{\R^3}\int_{\S^2}B(\mu^{\frac{1}{2}}\Psi)_{*}F_+(v)F_-(v)\mu^{-1}\,d\sigma dv_*dv \\
		& \quad\ge -\int_{\R^3}{Q(\mu^{\frac{1}{2}}\Psi,F_-)\mu^{-1}F_-(v)\,dv},
	\end{align*}
	where we used $F_+\times F_-=0$ and $(F_+)'\times F_-\ge 0$.
	Moreover, it follows from $F=F_+-F_-$ and Lemma \ref{LemRegu} that 
	\begin{align*}
		&-2\vpi\int_{\R^3}V(\mu^{-\frac{1}{2}}F-\mu^{\frac{1}{2}})\mu^{-\frac{1}{2}}F_-\,dv\\
		&\quad=2\vpi\int_{\R^3}V(\mu^{-\frac{1}{2}}F_-+\mu^{\frac{1}{2}})\mu^{-\frac{1}{2}}F_-\,dv\\
		&\quad\le -\|[\wh{C}^{}_0\<v\>^{{4}}\mu^{-\frac{1}{2}}F_-,\na_v(\<v\>^{{2}}\mu^{-\frac{1}{2}}F_-)]\|_{L^2_v}^2
		-\wh{C}^2_0\|\<v\>^{{8}}F_-\|_{L^1_v}\\
		&\quad\le 0. 
	\end{align*}
	It's direct to obtain 
	\begin{align*}
		2\int_{\R^3}M(\mu^{-\frac{1}{2}}F-\mu^{\frac{1}{2}})\mu^{-1}F_-\,dv
		&=-2\int_{\R^3}M(\mu^{-\frac{1}{2}}F_-+\mu^{\frac{1}{2}})\mu^{-1}F_-\,dv\le 0. 
	\end{align*}
	Then combining the above three estimates, and taking $L^2$ inner product of \eqref{536} with $-2\mu^{-1}F_-$, we have
	\begin{multline*}
		\frac{d}{dt}\|\mu^{-\frac{1}{2}}F_-\|^2_{L^2_x(\Omega)L^2_v}+\int_{\pa\Omega}\int_{\R^3}v\cdot n|\mu^{-\frac{1}{2}}F_-|^2\,dvdS(x)
		\le-2\int_{\R^3}Q(\mu^{\frac{1}{2}}\Psi,F)\mu^{-1}F_-(v)\,dv \\ 
		-2\vpi\int_{\Omega\times\R^3}V(\mu^{-\frac{1}{2}}F-\mu^{\frac{1}{2}})\mu^{-\frac{1}{2}}F_-\,dvdx 
		+2\int_{\Omega\times\R^3}M(\mu^{-\frac{1}{2}}F-\mu^{\frac{1}{2}})\mu^{-1}F_-\,dvdx \\
		\le2\int_{\Omega\times\R^3}\Ga(\Psi,\mu^{-\frac{1}{2}}F_-)\mu^{-\frac{1}{2}}F_-(v)\,dvdx.
	\end{multline*}
	Applying \eqref{L1} and \eqref{Gaesweight} to the collision term, we have
	\begin{multline*}
		\frac{d}{dt}\|\mu^{-\frac{1}{2}}F_-\|^2_{L^2_x(\Omega)L^2_v}
		+\int_{\Si_+}|v\cdot n||\mu^{-\frac{1}{2}}F_-|^2\,dvdS(x)\\
		\le
		\int_{\Si_-}|v\cdot n||\mu^{-\frac{1}{2}}F_-|^2\,dvdS(x)+C\|\1_{|v|\le R_0}\mu^{-\frac{1}{2}}F_-\|_{L^2_x(\Omega)L^2_v}^2.
	\end{multline*}
	where we choose $\de_0>0$ in \eqref{4135} small enough.
	For the inflow case, we have $G\ge 0$ on $\Si_-$, and hence, $F_-=0$ on $\Si_-$. For the Maxwell-reflection case, we have from \eqref{reflect} that 
	\begin{align*}
		\frac{1}{1-\ve}&\int_{\Si_-}|v\cdot n||\mu^{-\frac{1}{2}}F_-|^2\,dvdS(x)\le 
		\int_{\Si_-}|v\cdot n|\Big((1-\al)^2|\mu^{-\frac{1}{2}}F_-(x,R_L(x)v)|^2\\
		&\quad+2(1-\al)\al c_\mu F_-(x,R_L(x)v)\int_{v'\cdot n(x)>0}\{v'\cdot n(x)\}F_-(t,x,v')(v')\,dv'\\
		&\quad+\al^2(c_\mu)^2\mu(v)\Big\{\int_{v'\cdot n(x)>0}\{v'\cdot n(x)\}F_-(t,x,v')(v')\,dv'\Big\}^2\Big)\,dvdS(x)\\
		&\le(1-\al)\int_{\Si_-}|v\cdot n||\mu^{-\frac{1}{2}}F_-(x,v)|^2\,dvdS(x)\\
		&\quad+\al c_\mu\int_{\pa\Omega}\Big(\int_{v'\cdot n(x)>0}\{v'\cdot n(x)\}F_-(t,x,v')(v')\,dv'\Big)^2\,dS(x)\\
		&\le\int_{\Si_+}|v\cdot n||\mu^{-\frac{1}{2}}F_-|^2\,dvdS(x).
	\end{align*}
	where we used change of variables $v\mapsto R_L(x)v:\Si_-\to\Si_+$, H\"{o}lder's inequality $\int_{v'\cdot n(x)>0}\{v'\cdot n(x)\}F_-(t,x,v')(v')\,dv'\le \big(\int_{v'\cdot n(x)>0}\{v'\cdot n(x)\}|\mu^{-\frac{1}{2}}F_-|^2\,dv'\big)\big(\int_{v'\cdot n(x)>0}\{v'\cdot n(x)\}\mu\,dv'\big)$ and \eqref{cmu}. 
	In both cases, we have
	\begin{align*}
		\frac{d}{dt}\|\mu^{-\frac{1}{2}}F_-\|^2_{L^2_x(\Omega)L^2_v}
		\le C\|\1_{|v|\le R_0}\mu^{-\frac{1}{2}}F_-\|_{L^2_x(\Omega)L^2_v}^2.
	\end{align*}
	Using Gr\"{o}nwall's inequality, we deduce
	\begin{align*}
		\sup_{0\le t\le T}\|\mu^{-\frac{1}{2}}F_-\|^2_{L^2_x(\Omega)L^2_v}
		\le e^{CT}\|\mu^{-\frac{1}{2}}F_-|_{t=0}\|^2_{L^2_x(\Omega)L^2_v}=0.
	\end{align*}
	We then conclude that $F(t)\ge 0$ in $[0,T]\times\Omega\times\R^3_v$. This completes the proof of Theorem \ref{positivity}.
\end{proof}

\section{Extension to the whole space and \texorpdfstring{$L^2$}{L2} local existence}\label{Sec4}
In this section, we will extend the boundary problem to a whole-space problem and analyze the $L^2$ estimates without level sets.
Assume that $\Omega\subset\R^3_x$ is an open bounded subset.
Then we can split $\Omega^c\times\R^3_v$ into the \emph{inflow} and \emph{outflow regions} $D_{in},D_{out}$ as in \eqref{Dinoutwt}. 


\subsection{Extension of the boundary value}
In order to use the Galerkin method for the derivation of a weak solution, we need the boundary value to vanish. We first extend the boundary value $g$ to $\R^3_x\times\R^3_v$ and mollify it to a smooth function. With such effort, we can consider $f-g$ in equation \eqref{L2existeq} below, which has the vanishing (inflow or outflow) boundary value. We begin with extending and approximating the boundary value $g$. 
\begin{Lem}\label{gextension}
	Let $0\le T_1<T_2$. 
	We can approximate the $L^2$ functions on $\Si_+$ or $\Si_-$ by a smooth function in $\R^3_x\times\R^3_v$ as follows. 
	\begin{enumerate}
		\item 
		Suppose $g$ is defined on $[T_1,T_2]\times\Si_-$ and satisfies 
		\begin{align}\label{eqg2}
			\|g\|_{L^2_{t,x,v}([T_1,T_2]\times\Si_-)}
			<\infty. 
		\end{align}
		Then there exists functions $g_j\in C^\infty_c(\R^7_{t,x,v})$ for $j\ge 1$ 
		such that its restriction on $[T_1,T_2]\times\Si_-$ satisfies 
		\begin{align}\label{eqg2limit}
			\int^{T_2}_{T_1}\int_{\Si_-}|v\cdot n||g_j-g|^2\,dS(x)dv\to 0, \ \ \text{ as }j\to \infty. 
		\end{align}
		\item Suppose $g$ is defined on $[T_1,T_2]\times\Si_+$ and satisfies 
		\begin{align*}
			\|g\|_{L^2_{t,x,v}([T_1,T_2]\times\Si_-)}
			<\infty. 
		\end{align*}
		Then there exists functions $g_j\in C^\infty_c(\R^7_{t,x,v})$ for $j\ge 1$ such that its restriction on $[T_1,T_2]\times\Si_+$ satisfies 
		\begin{align}\label{eqg2limit2}
			\int^{T_2}_{T_1}\int_{\Si_+}|v\cdot n||g_j-g|^2\,dS(x)dv\to 0, \ \ \text{ as }j\to \infty. 
		\end{align}
	\end{enumerate}
\end{Lem}
\begin{proof}
	We first straighten the boundary $\pa\Omega$ and then construct the approximate function by using approximate identities in $\R^2_x\times\R^3_v$. 
	From \eqref{Bnboundary}, we can define $C^3$ mapping $\Phi_k$ and its inverse $\Psi_k$ ($1\le k\le N$ with $N<+\infty$ given in \eqref{Bnboundary}) such that $\Phi_k$ ``straightens out $\pa\Omega$ in $B_k$". 
	
	\smallskip The straightening method is standard. Choosing $C^3$ functions $\rho_k$ as in \eqref{Bnboundary}, we define $y=\Phi_k(x)$ and $x=\Psi_k(y)$ by 
	\begin{align}\label{yPhi}
		\begin{aligned}
			&y_1=\Phi^1_k(x):=x_1,\quad y_2=\Phi^2_k(x):=x_2,\\
			&y_3=\Phi^3_k(x)=x_3-\rho_k(x_1,x_2),
		\end{aligned}
	\end{align}
	and 
	\begin{align}\label{xPsi}
		\begin{aligned}
			&x_1=\Psi^1_k(y):=y_1,\quad x_2=\Psi^2_k(y):=y_2,\\
			&x_3=\Psi^3_k(y)=y_3+\rho_k(y_1,y_2). 
		\end{aligned}
	\end{align}
	It follows that $\det\big(\frac{\pa\Phi(x)}{\pa x}\big)=\det\big(\frac{\pa\Psi(y)}{\pa y}\big)=1$. 
	From \eqref{Bnboundary}, we know that $\Phi$ is one-to-one onto mapping and satisfies 
	\begin{align*}
		&\Phi_k(\Omega\cap B_k)=\{y\in \Phi_k(B_k)\,:\,y_3<0\},\\
		&\Phi_k(\pa\Omega\cap B_k)=\{y\in \Phi_k(B_k)\,:\,y_3=0\},\\
		&\Phi_k(\ol\Omega^c\cap B_k)=\{y\in \Phi_k(B_k)\,:\,y_3>0\}. 
	\end{align*}
	Then for any suitable function $f$, the surface integral can be expressed as 
	\begin{align}\label{intebound}
		\int_{\pa\Omega\cap B_k}f(x)\,dS(x)=\int_{\{y\in \Phi_k(B_k):y_3=0\}}f(y)\sqrt{1+(\rho_{k,y_1})^2+(\rho_{k,y_2})^2}\,dy_1dy_2, 
	\end{align}
	where $\rho_{k,y_j}=\pa_{y_j}\rho_k$ ($j=1,2$), 
	and the interior integral can be expressed as 
	\begin{align*}
		\begin{aligned}
			&\int_{B_k}f(x)\,dx=\int_{y\in \Phi_k(B_k)}f(y)\,dy_1dy_2. 
		\end{aligned}
	\end{align*}
	Moreover, 
	since $\pa\Omega\subset\cup_{k=1}^NB_k$, we let $\{\zeta_k\}_{k=1}^N$ be an associated partition of unity ($\zeta_k$ are smooth function with compact support in $B_k$ and $\sum_{k=1}^N\zeta_k(x)=1$ for any $x$ in an open neighborhood of $\pa\Omega$). 
	
	\smallskip For the function $g$ defined on $[T_1,T_2]\times\Si_-$ satisfying \eqref{eqg2}, we first extend $g$ to $\Si_+$ by letting $g=0$ on $[T_1,T_2]\times(\Si_+\cup\Si_0)$ and choose $g_R=g\1_{|x|+|v|\le R}$. Since $\|g\|_{L^2_{t,x,v}([T_1,T_2]\times\Si_-)}<\infty$, we have from Dominated Convergence Theorem that 
	\begin{align}\label{limit1}
		\int^{T_2}_{T_1}\int_{\pa\Omega\times\R^3_v}|v\cdot n||g_R-g|^2\,dS(x)dv\to 0,\ \text{ as } R\to\infty. 
	\end{align}
	For any suitable function $h$, by surface integral \eqref{intebound}, we have 
	\begin{align*}
		&\int^{T_2}_{T_1}\int_{\pa\Omega\times\R^3_v}|v\cdot n||\zeta_k(h-g_R)(x)|^2\,dS(x)dv\\
		&\quad=\int^{T_2}_{T_1}\int_{\R^3_v}\int_{\{y\in \Phi_k(B_k):y_3=0\}}|v\cdot n(y)||\zeta_k(h-g_R)(x(y))|^2\sqrt{1+(\rho_{k,y_1})^2+(\rho_{k,y_2})^2}\,dy_1dy_2dvdt.
	\end{align*}
	Let $\eta_\de$ be the mollifier defined by $\eta_\de(Y)=\de^{-6}\eta(\de^{-1}Y)$ with $Y=(t,y_1,y_2,v)$, and a smooth function $\eta$ having compact support and satisfying $\int_{\R^6}\eta(Y)\,dY=1$. 
	Then we choose the smooth function
	\begin{align*}
		h_{\de,R}=\eta_\de*g_R:=\int_{\R^6}\eta_\de(Y-Y_*)g_R(Y_*)\,dY, 
	\end{align*}
	where $g_R(Y_*)$ is written in the $y$ coordinates by the transformation \eqref{yPhi}. 
	It follows from the properties of approximate identities $\eta_\de$ that (see for instance \cite[Theorem 1.2.19]{Grafakos2014})
	\begin{align}\label{limit2}\notag
		&\int^{T_2}_{T_1}\int_{\R^3_v}\int_{\{y\in \Phi_k(B_k):y_3=0\}}|v\cdot n(y)||\zeta_k(h_{\de,R}-g_R)(x(y))|^2\sqrt{1+(\rho_{k,y_1})^2+(\rho_{k,y_2})^2}\,dy_1dy_2dvdt\\&\notag\quad\le RC_k\int^{T_2}_{T_1}\int_{\R^3_v}\int_{\R^2}|\zeta_k(h_{\de,R}-g_R)(x(y_1,y_2,0))|^2\,dy_1dy_2dvdt\\
		&\quad\to 0, \ \text{ as }\de\to 0,
	\end{align}
	for any $1\le k\le N$, where $x=x(y)=\Psi(y)$ is defined by \eqref{xPsi}. 
	Since $\rho$ is $C^3$ function, $\rho_{k,y_1}$ and $\rho_{k,y_2}$ are bounded by a constant $C_k>0$ on the compact support of $\zeta_k$. Moreover, since $g_R$ and $\eta_\de$ have compact supports, we know that $h_{\de,R}(t,y_1,y_2,v)$ also has compact support. 
	
	\smallskip We next extend $h_{\de,R}$ from $\R^6$ to $\R^6\times\{y_3\in\R\}$ by simply multiplying a smooth function $\vp\in C^\infty_c(\R)$ which has compact support and satisfies $\vp(0)=1$. 
	Then the function $H_{\de,R}(t,y,v)=h_{\de,R}(t,y_1,y_2,v)\vp(y_3)$ is a smooth function with compact support and satisfies $H_{\de,R}(t,y_1,y_2,0,v)=h_{\de,R}(t,y_1,y_2,v)$. Denote $\wt H_{\de,R}(t,x,v)=H_{\de,R}(t,y(x),v)$ with $y(x)=\Phi(x)$ defined in \eqref{yPhi}. 
	Combining this with \eqref{limit1} and \eqref{limit2}, we know that for any $\ve>0$, there exists large $R=R(\ve)>0$ and small $\de=\de(R,\ve)>0$ such that 
	\begin{align*}
		\int^{T_2}_{T_1}\int_{\pa\Omega\times\R^3_v}|v\cdot n||g_R-g|^2\,dS(x)dv<\frac{\ve}{2}, 
	\end{align*}
	and 
	\begin{align*}
		&\int^{T_2}_{T_1}\int_{\pa\Omega\times\R^3_v}|v\cdot n||\wt H_{\de,R}(x)-g_R(x)|^2\,dS(x)dv\\
		&\quad=\sum_{k=1}^NC_k\int^{T_2}_{T_1}\int_{\R^3_v}\int_{\R^2}|v\cdot n(y)||\zeta_k(x(y_1,y_2,0))(h_{\de,R}(y_1,y_2)-g_R(x(y_1,y_2,0)))|^2\,dy_1dy_2dvdt\\
		&\quad< \frac{\ve}{2}. 
	\end{align*}
	Consequently, there exists $g_j\in C^\infty_c(\R^7_{t,x,v})$ (which is smooth and has compact support) such that its restriction on $[T_1,T_2]\times\Si_-$ satisfies 
	\begin{align}\label{555}
		\int^{T_2}_{T_1}\int_{\Si_-}|v\cdot n||g_j-g|^2\,dS(x)dv\to 0, \ \ \text{ as }j\to \infty. 
	\end{align}
	This implies \eqref{eqg2limit} and Lemma \ref{gextension} (1). 
	
	\smallskip The proof of the second assertion is the same since we can extend $g$ from $\Si_+$ to $\pa\Omega\times\R^3_v$ by letting $g=0$ on $\Si_-\cup\Si_0$. Then following the calculations from \eqref{limit1} to \eqref{555}, one can obtain Lemma \ref{gextension} (2) and \eqref{eqg2limit2}. This completes the proof of Lemma \ref{gextension}. 
\end{proof}

\subsection{\texorpdfstring{$L^2$}{L2} local existence of linear equation with inflow}
In this subsection, we will derive the $L^2$ existence of the modified linearized Boltzmann equation 
\begin{align}\label{L2existeq}
	\left\{
	\begin{aligned}
		& \pa_tf+ v\cdot\na_xf = \vpi Vf+\Gamma(\Psi,f)+\Gamma(\vp,\mu^{\frac{1}{2}})\\&\qquad\qquad\quad+\phi-N\<v\>^{{l-2}}f\quad \text{ in } (T_1,T_2]\times\Omega\times\R^3_v, \\
		& f(t,x,v)|_{\Si_-}=g\quad \text{ on }[T_1,T_2]\times\Si_-, \\
		& f(T_1,x,v)=f_{T_1}\quad \text{ in }\Omega\times\R^3_v, 
	\end{aligned}\right.
\end{align}
with given functions $\Psi=\mu^{\frac{1}{2}}+\psi\ge 0$, $\vp,\phi$, and inflow-boundary function $g$ on $[T_1,T_2]\times\Si_-$, and any $\vpi,N,M\ge 0$. 
Here, operator $V$ is given by \eqref{Vf}. 

\smallskip Using the approximation Lemma \ref{gextension}, we can derive the local existence of equation \eqref{L2existeq}. 
\begin{Thm}\label{weakfThm}
	Let $T_1\ge 0$, $s\in(0,1)$ and $\vpi,N,M\ge 0$. Assume time-dependent functions $\Psi=\mu^{\frac{1}{2}}+\psi\ge0$, $\vp$, $\phi$, inflow boundary value $g$, and initial data $f_{T_1}$ satisfy
	\begin{align}\label{41344}\begin{aligned}
			\|\<v\>^4\psi\|_{L^\infty_t([T_1,T_2])L^\infty_{x}(\Omega)L^\infty_v(\R^3_v)} & \le\de_0, \\
			\|[\vp,\phi]\|^2_{L^2_t([T_1,T_2])L^2_x(\Omega)L^2_v(\R^3_v)}& <\infty, \\
			\|f_{T_1}\|_{L^2_x(\Omega)L^2_v(\R^3_v)} +\|g\|_{L^2_tL^2_{x,v}(\Si_-)}& <\infty,
		\end{aligned}
	\end{align}
	with sufficiently small $\de_0>0$. 
	Then there exists a small time $T_2>T_1$ (depending on $\ti C$) and a weak solution $f$ to equation \eqref{L2existeq}
	in the sense that, for any function $\Phi\in C^\infty_c(\R_t\times\R^3_x\times\R^3_v)$, 
	\begin{multline}\label{weakf}
		(f(T_2),\Phi(T_2))_{L^2_x(\Omega)L^2_v}-(f_{T_1},\Phi(T_1))_{L^2_x(\Omega)L^2_v}
		-(f,(\pa_t+v\cdot\na_x)\Phi)_{L^2_t([T_1,T_2])L^2_x(\Omega)L^2_v}\\
		+(f,\Phi)_{L^2_t([T_1,T_2])L^2_{x,v(\Si_+)}}
		=(g,\Phi)_{L^2_t([T_1,T_2])L^2_{x,v(\Si_-)}}
		+(f,\vpi V\Phi)_{L^2_t([T_1,T_2])L^2_x(\Omega)L^2_v}\\
		+(\Gamma(\Psi,f)+\Gamma(\vp,\mu^{\frac{1}{2}})+\phi-N\<v\>^{{l-2}}f,\Phi)_{L^2_t([T_1,T_2])L^2_x(\Omega)L^2_v}.
	\end{multline}
	Moreover, for any weak solution $f$ to equation \eqref{L2existeq}, we have $L^2$ estimate: for any $k\ge 0$, 
	\begin{multline}\label{fLineares1}
		\pa_t\|\<v\>^kf(t)\|^2_{L^2_{x}(\Omega)L^2_v}+\|\<v\>^kf\|^2_{L^2_{x,v}(\Si_+)}+\vpi\|[\wh{C}_0\<v\>^{{k+4}}f,\<v\>^{{k+2}}\na_vf]\|^2_{L^2_{x}(\Omega)L^2_v}\\
		+c_0\|\<v\>^kf\|_{L^2_{x}(\Omega)L^2_D}^2
		+N\|\<v\>^{l-2+k}f\|_{L^2_{x}(\Omega)L^2_v}^2\\
		\le C\|\<v\>^kf(t)\|_{L^2_{x}(\Omega)L^2_v}^2+\|[\mu^{\frac{1}{10^4}}\vp,\<v\>^k\phi]\|_{L^2_{x}(\Omega)L^2_v}^2+\|\<v\>^kg\|^2_{L^2_{x,v}(\Si_-)}.
	\end{multline}
	
\end{Thm}
\begin{proof}
	By Lemma \ref{gextension}, there exists $g_j\in C^\infty_c(\R^7_{t,x,v})$ such that 
	\begin{align}\label{eqg2l2}
		\int^{T_2}_{T_1}\int_{\Si_-}|v\cdot n||g_j-g|^2\,dS(x)dv\to 0, \ \ \text{ as }j\to \infty. 
	\end{align}
	We begin with considering the inflow-boundary value $f|_{\Si_-}=g_j$ and the weak form of $h:=f-g_j$. That is, for any function $\Phi\in C^\infty_c(\R^7_{t,x,v})$, we search for an $L^2$ function $h$ satisfying 
	\begin{multline}\label{hheq}
		(h(T_2),\Phi(T_2))_{L^2_x(\Omega)L^2_v}-(f_{T_1}-g_j(T_1),\Phi(T_1))_{L^2_x(\Omega)L^2_v}
		-(h,(\pa_t+v\cdot\na_x)\Phi)_{L^2_t([T_1,T_2])L^2_x(\Omega)L^2_v}\\
		+(h,\Phi)_{L^2_t([T_1,T_2])L^2_{x,v(\Si_+)}}
		+\vpi\int^{T_2}_{T_1}\int_{\Omega\times\R^3_v}\Big(2\wh{C}^2_0\<v\>^{{8}}h\Phi+2\<v\>^{{4}}\na_vh\cdot\na_v\Phi\Big)\,dxdvdt\\
		=(\Gamma(\Psi,h)+\Gamma(\vp,\mu^{\frac{1}{2}})+\phi-N\<v\>^{{l-2}}h,\Phi)_{L^2_x(\Omega)L^2_v}\,dt\\
		-((\pa_t+v\cdot\na_x)g_j,\Phi)_{L^2_t([T_1,T_2])L^2_x(\Omega)L^2_v} 
		+(\vpi Vg_j+\Gamma(\Psi,g_j)-N\<v\>^{{l-2}}g_j,\Phi)_{L^2_t([T_1,T_2])L^2_x(\Omega)L^2_v}. 
	\end{multline}
	To further mollify the spatial and velocity variables, and eliminate the boundary effect, we directly consider the weak form $B_\ve[h,\Phi]:H^{2}_{x,v}(\<v\>^4)\times H^{2}_{x,v}(\<v\>^4)\to \R$ defined by 
	\begin{multline*}
		B_\ve[h,\Phi]=-\int_{\Omega\times\R^3_v}h v\cdot\na_x\Phi\,dxdv
		+\int_{\Si_+}|v\cdot n|h\Phi\,dS(x)dv\\
		+\ve\sum_{|\al|+|\beta|\le 2}\int_{\Omega\times\R^3_v}\<v\>^8\pa^\al_\beta h\pa^\al_\beta \Phi\,dxdv
		+\vpi\int_{\Omega\times\R^3_v}\Big(2\wh{C}^2_0\<v\>^{{8}}h\Phi+2\<v\>^{{4}}\na_vh\cdot\na_v\Phi\Big)\,dxdv\\
		-\int_{\Omega\times\R^3_v}\Big(\Gamma(\Psi,h)-N\<v\>^{{l-2}}h\Big)\Phi\,dxdv,
	\end{multline*}
	for any $\ve>0$, 
	where we denote $\pa^\al_\beta=\pa^{\al_1}_{x_1}\pa^{\al_1}_{x_1}\pa^{\al_2}_{x_2}\pa^{\al_3}_{x_3}\pa^{\beta_1}_{v_1}\pa^{\beta_2}_{v_2}\pa^{\beta_3}_{v_3}$ and linear space 
	\begin{align*}
		&H^{2}_{x,v}(\<v\>^4)=\{h\,:\,\|h\|_{H^{2}_{x,v}(\<v\>^4)}<\infty\},\ \text{ with }\\
		&\|h\|_{H^{2}_{x,v}(\<v\>^4)}^2:=\sum_{|\al|+|\beta|\le 2}\|\<v\>^4\pa^\al_\beta h\|_{L^2_x(\Omega)L^2_v}^2.
	\end{align*} 
	Note that $Vh=-2\wh{C}^2_0\<v\>^{{8}}h+2\na_v\cdot(\<v\>^{{4}}\na_v)h$ (defined in \eqref{Vf}). 
	Then we search for the solution $h\in H^{2}_{x,v}(\<v\>^4)$ such that for any $\Phi\in H^{2}_{x,v}(\<v\>^4)$, 
	\begin{align}\label{patfh}
		\left\{\begin{aligned}
			&\pa_t(h,\Phi)_{L^2_{x}(\Omega)L^2_v}-(h,\pa_t\Phi)_{L^2_{x}(\Omega)L^2_v}+B_\ve[h,\Phi]=l(\Phi),\\
			&h(T_1,x,v)=f_{T_1}(x,v)-g_j(T_1,x,v). 
		\end{aligned}\right.
	\end{align}
	where we denote $l(\Phi)$ by 
	\begin{align*}
		l(\Phi):&=\big(\Gamma(\vp,\mu^{\frac{1}{2}})+\phi,\Phi\big)_{L^2_{x}(\Omega)L^2_v}-\big((\pa_t+v\cdot\na_x)g_j,\Phi\big)_{L^2_x(\Omega)L^2_v}\\&\qquad
		+\big(\vpi Vg_j+\Gamma(\Psi,g_j)-N\<v\>^{{l-2}}g_j,\Phi\big)_{L^2_x(\Omega)L^2_v}. 
	\end{align*}
	Notice that this equation \eqref{patfh} is different from the weak form of the equation
	\begin{align*}
		\left\{
		\begin{aligned}
			& \pa_th+ v\cdot\na_xh = \ve\sum_{|\al|+|\beta|\le 2}(-1)^{|\al|+|\beta|}\pa^\al_\beta(\<v\>^8\pa^\al_\beta h)+\vpi Vh+\Gamma(\Psi,h)\\&\qquad\qquad\qquad+\Gamma(\vp,\mu^{\frac{1}{2}})+\phi-N\<v\>^{{l-2}}h-(\pa_t+v\cdot\na_x)g_j\\
			&\qquad\qquad\qquad+\vpi Vg_j+\Gamma(\Psi,g_j)-N\<v\>^{{l-2}}g_j\quad \text{ in } (T_1,T_2]\times\Omega\times\R^3_v, \\
			& h(t,x,v)|_{\Si_-}=0\quad \text{ on }[T_1,T_2]\times\Si_-, \\
			& h(T_1,x,v)=f_{T_1}(x,v)+g_j(T_1,x,v)\quad \text{ in }\Omega\times\R^3_v, 
		\end{aligned}\right.
	\end{align*}
	since the latter one contains boundary effect in $\pa^\al$ (by integration by parts, boundary terms occur). 
	
	\smallskip To solve equation \eqref{patfh}, we shall apply the Galerkin's method; see for instance \cite[III.5]{Ladyzhenskaia1995}. We will only give the sketch of the proof of the existence of \eqref{patfh} and use the \emph{a priori} arguments. 
	Let $\{w_k\}_{k=1}^\infty$ be an orthonormal basis in $H^{2}_{x,v}(\<v\>^4)$, i.e. $(w_k,w_l)_{H^{2}_{x,v}(\<v\>^4)}=1$. Then we construct a sequence 
	\begin{align*}
		h_n(t)=\sum_{k=1}^nc^k_n(t)w_k
	\end{align*}
	by solving $c^k_m(t)$ from the ordinary differential equations (ODEs):
	\begin{align}\label{patfh1}
		\left\{\begin{aligned}
			&(\pa_th_n,\Phi)_{L^2_{x}(\Omega)L^2_v}+B_\ve[h_n,\Phi]=l(\Phi),\\
			&\sum_{k=1}^n(c^k_m(T_1)w_k,\Phi)_{L^2_{x}(\Omega)L^2_v}=(f_{T_1}-g_j(T_1),\Phi)_{L^2_{x}(\Omega)L^2_v}, 
		\end{aligned}\right.
	\end{align}
	with $\Phi\in\{w_k\}_{k=1}^n$, which contains $n$ ODEs and can be solved locally in time. 
	To take the limit $n\to\infty$ in \eqref{patfh1}, we need to derive the uniform-in-$n$ estimate of $h_n$. Instead, we give the \emph{a priori} estimate as follows (similar to giving the uniform estimate of $h_n$). Let $\Phi=h$ in \eqref{patfh}, we have from \eqref{56} and \eqref{Gaesweight} that 
	\begin{align*}
		&\frac{1}{2}\pa_t\|h\|^2_{L^2_{x}(\Omega)L^2_v}+\frac{1}{2}\|h\|^2_{L^2_{x,v}(\Si_+\cup\Si_-)}
		+\ve\sum_{|\al|+|\beta|\le 2}\|\<v\>^4\pa^\al_\beta h\|^2_{L^2_{x}(\Omega)L^2_v}
		+2\vpi\wh{C}^2_0\|\<v\>^{{4}}h\|^2_{L^2_{x}(\Omega)L^2_v}\\
		&\quad+2\vpi\|\<v\>^{{2}}\na_vh\|^2_{L^2_{x}(\Omega)L^2_v}
		+\big(c_0-C\|\<v\>^{4}\psi\|_{L^\infty_{x}(\Omega)L^\infty_v}\big)\|h\|_{L^2_{x}(\Omega)L^2_D}^2
		+N\|\<v\>^{l-2}h\|_{L^2_{x}(\Omega)L^2_v}^2\\
		&\quad\le 
		C\|\1_{|v|\le R_0}h\|_{L^2_{x}(\Omega)L^2_v}^2
		+C\big(1+\|\<v\>^{4}\psi\|_{L^\infty_{x}(\Omega)L^\infty_v}\big)\|g_j\|_{L^2_{x}(\Omega)L^2_D}\|h\|_{L^2_{x}(\Omega)L^2_D}\\
		&\qquad+C\big(\|[\vp,\phi,(\pa_t+v\cdot\na_x)g_j]\|_{L^2_{x}(\Omega)L^2_v}
		+(\vpi+N)\|\<v\>^{l-2}g_j\|_{H^2_{x,v}(\<v\>^4)}\big)\|h\|_{L^2_{x}(\Omega)L^2_v}.
	\end{align*}
	Integrating over $t\in[T_1,T_2]$, using Gr\"{o}nwall's inequality, and letting $\de_0>0$ in \eqref{41344} small enough, we have 
	\begin{multline}\label{434}
		\frac{1}{2}\sup_{T_1\le t\le T_2}\|h\|^2_{L^2_{x}(\Omega)L^2_v}+\frac{1}{2}\int^{T_2}_{T_1}\|h\|^2_{L^2_{x,v}(\Si_+\cup\Si_-)}dt
		+{\ve}\sum_{|\al|+|\beta|\le 2}\int^{T_2}_{T_1}\|\<v\>^4\pa^\al_\beta h\|^2_{L^2_{x}(\Omega)L^2_v}\,dt\\
		+2\vpi\int^{T_2}_{T_1}\wh{C}^2_0\|\<v\>^{{4}}h\|^2_{L^2_{x}(\Omega)L^2_v}\,dt+2\vpi\int^{T_2}_{T_1}\|\<v\>^{{2}}\na_vh\|^2_{L^2_{x}(\Omega)L^2_v}\,dt
		+\frac{c_0}{2}\int^{T_2}_{T_1}\|h\|_{L^2_{x}(\Omega)L^2_D}^2\,dt\\
		+\int^{T_2}_{T_1}\Big(N\|\<v\>^{l-2}h\|_{L^2_{x}(\Omega)L^2_v}^2\Big)\,dt\\
		\le e^{C(T_2-T_1)}\Big(\|f_{T_1}-g_j(T_1)\|_{L^2_{x}(\Omega)L^2_v}^2+\int^{T_2}_{T_1}\|[\vp,\phi]\|_{L^2_{x}(\Omega)L^2_v}^2\,dt+(T_2-T_1)(1+\vpi+N)^2C_{g_j}\Big).
	\end{multline}
	Here, $C_{g_j}>0$ is some constant depending on some Sobolev norm of $g_j$, which is finite since $g_j\in C^\infty_c(\R^7_{t,x,v})$.
	The estimate \eqref{434} is uniform in $\ve$, and $h_n$ shares the same estimate uniformly-in-$(\ve,n)$ if one replace $h$ by $h_n$. 
	It follows from the Banach-Alaoglu Theorem that the sequence $\{h_n\}$ is weakly-$*$ compact in the sense that, up to a subsequence, 
	\begin{align*}
		&h_n\rightharpoonup h\ \text{ weakly-$*$ in $L^2_t([T_1,T_2])L^2_{x,v}(\Si_+\cup\Si_-)$, $L^2_t([T_1,T_2])H^2_{x,v}(\<v\>^4)$, 
			$L^2_t([T_1,T_2])L^2_{x}(\Omega)L^2_D$},\\
		&h_n\rightharpoonup h\ \text{ weakly-$*$ in $L^2_{x}(\Omega)L^2_v$ for any $t\in(T_1,T_2]$},
	\end{align*}
	as $n\to\infty$, with some $h$ satisfying \eqref{434}. 
	Therefore, taking limit $n\to\infty$ (up to a subsequence) in the weak form of \eqref{patfh1}, i.e. for any $\Phi\in \text{Span}\{w_k\}_{k=1}^n$, 
	\begin{multline*}
		(h_n(T_2),\Phi(T_2))_{L^2_{x}(\Omega)L^2_v}+(f_{T_1}-g_j(T_1),\Phi(T_1))_{L^2_{x}(\Omega)L^2_v}-\int^{T_2}_{T_1}(h_n,\pa_t\Phi)_{L^2_{x}(\Omega)L^2_v}\,dt\\+\int^{T_2}_{T_1}B_\ve[h_n,\Phi]\,dt=\int^{T_2}_{T_1}l(\Phi)\,dt, 
	\end{multline*}
	we obtain 
	\begin{multline}\label{435}
		(h(T_2),\Phi(T_2))_{L^2_{x}(\Omega)L^2_v}+(f_{T_1}-g_j(T_1),\Phi(T_1))_{L^2_{x}(\Omega)L^2_v}-\int^{T_2}_{T_1}(h,\pa_t\Phi)_{L^2_{x}(\Omega)L^2_v}\,dt\\+\int^{T_2}_{T_1}B_\ve[h,\Phi]\,dt=\int^{T_2}_{T_1}l(\Phi)\,dt, 
	\end{multline}
	which implies that $h$ is the weak solution to \eqref{patfh}. 
	Note that one can write $\big(\Gamma(\Psi,h_n),\Phi\big)_{L^2_{x}(\Omega)L^2_v}$ in the weak form as in \eqref{Gaest}, and use estimate \eqref{Gammavp} or \eqref{64eq1}. Moreover, the equation \eqref{435} is satisfied for all $\Phi\in C^\infty_c(\R^7_{t,x,v})$ since we have taken the limit $n\to\infty$. 
	
	\smallskip We further denote $h_\ve=h$ to illustrate its dependence on $\ve>0$. Then $h_\ve$ satisfies the estimate \eqref{434} uniformly in $\ve$ with $h$ replaced by $h_\ve$. By Banach-Alaoglu Theorem, the sequence $\{h_\ve\}$ is weakly-$*$ compact in the sense that, up to a subsequence, 
	\begin{align*}
		&h_\ve\rightharpoonup h\ \text{ weakly-$*$ in $L^2_t([T_1,T_2])L^2_{x,v}(\Si_+\cup\Si_-)$ and $L^2_t([T_1,T_2])L^2_{x}(\Omega)L^2_D$},\\
		&h_\ve\rightharpoonup h\ \text{ weakly-$*$ in $L^2_{x}(\Omega)L^2_v$ for any $t\in(T_1,T_2]$},
	\end{align*}
	as $\ve\to0$, with some $h$ satisfying 
	\begin{multline*}
		\frac{1}{2}\sup_{T_1\le t\le T_2}\|h\|^2_{L^2_{x}(\Omega)L^2_v}+\frac{1}{2}\int^{T_2}_{T_1}\|h\|^2_{L^2_{x,v}(\Si_+\cup\Si_-)}dt
		+2\vpi\int^{T_2}_{T_1}\wh{C}^2_0\|\<v\>^{{4}}h\|^2_{L^2_{x}(\Omega)L^2_v}\,dt\\+2\vpi\int^{T_2}_{T_1}\|\<v\>^{{2}}\na_vh\|^2_{L^2_{x}(\Omega)L^2_v}\,dt
		+\frac{c_0}{2}\int^{T_2}_{T_1}\|h\|_{L^2_{x}(\Omega)L^2_D}^2\,dt
		+\int^{T_2}_{T_1}\Big(N\|\<v\>^{l-2}h\|_{L^2_{x}(\Omega)L^2_v}^2\Big)\,dt\\
		\le e^{C(T_2-T_1)}\Big(\|f_{T_1}-g_j(T_1)\|_{L^2_{x}(\Omega)L^2_v}^2+\int^{T_2}_{T_1}\|[\vp,\phi]\|_{L^2_{x}(\Omega)L^2_v}^2\,dt+(T_2-T_1)(1+\vpi+N)^2C_{g_j}\Big).
	\end{multline*}
	Then we can further take the limit $\ve\to0$ in \eqref{435} with $h$ replaced by $h_\ve$ (up to a subsequence) to deduce that $h$ satisfies \eqref{hheq}. 
	
	\smallskip Consequently, if we let $f_j=h+g_j$ ($j\ge 1$) in $[T_1,T_2]\times\ol\Omega\times\R^3_v$, then $f_j$ is a weak solution to \eqref{L2existeq} with inflow-boundary value $g_j$ in the sense that for any $\Phi\in C^\infty_c(\R^7_{t,x,v})$, 
	\begin{multline}\label{weakfj}
		(f_j(T_2),\Phi(T_2))_{L^2_x(\Omega)L^2_v}-(f_{T_1},\Phi(T_1))_{L^2_x(\Omega)L^2_v}
		-\int^{T_2}_{T_1}(f_j,(\pa_t+v\cdot\na_x)\Phi)_{L^2_x(\Omega)L^2_v}\,dt\\
		+\int^{T_2}_{T_1}\int_{\Si_+}|v\cdot n|f_j\Phi\,dS(x)dvdt
		+\vpi\int^{T_2}_{T_1}\int_{\Omega\times\R^3_v}\Big(2\wh{C}^2_0\<v\>^{{8}}f_j\Phi+2\<v\>^{{4}}\na_vf_j\cdot\na_v\Phi\Big)\,dxdvdt\\
		=\int^{T_2}_{T_1}\int_{\Si_-}|v\cdot n|g_j\Phi\,dS(x)dvdt+\int^{T_2}_{T_1}(\Gamma(\Psi,f_j)+\Gamma(\vp,\mu^{\frac{1}{2}})+\phi-Nf_j,\Phi)_{L^2_x(\Omega)L^2_v}\,dt.
	\end{multline}
	Using \eqref{56}, the standard $L^2$ estimate of solution $f_j$ to equation \eqref{weakfj} yields 
	\begin{align}\label{weakfL2}
		&\notag\sup_{T_1\le t\le T_2}\|f_j\|^2_{L^2_x(\Omega)L^2_v}+\int^{T_2}_{T_1}\|f_j\|^2_{L^2_{x,v}(\Si_+)}dt
		+c_0\int^{T_2}_{T_1}\|f_j\|_{L^2_x(\Omega)L^2_D}^2\,dt\\
		&\notag\quad+\vpi\int^{T_2}_{T_1}\int_{\Omega\times\R^3_v}\Big(2\wh{C}^2_0\<v\>^{{8}}|f_j|^2+2\<v\>^{{4}}|\na_vf_j|^2\Big)\,dxdvdt
		+\int^{T_2}_{T_1}\Big(N\|\<v\>^{l-2}f_j\|_{L^2_{x}(\Omega)L^2_v}^2\Big)\,dt\\
		&\notag\qquad\le e^{C(T_2-T_1)}\Big(\|f_{T_1}\|_{L^2_{x}(\Omega)L^2_v}^2+\int^{T_2}_{T_1}\|[\vp,\phi]\|_{L^2_{x}(\Omega)L^2_v}^2\,dt+\int^{T_2}_{T_1}\|g_j\|^2_{L^2_{x,v}(\Si_-)}dt\Big)\\
		&\qquad\le e^{C(T_2-T_1)}\Big(\|f_{T_1}\|_{L^2_{x}(\Omega)L^2_v}^2+\int^{T_2}_{T_1}\|[\vp,\phi]\|_{L^2_{x}(\Omega)L^2_v}^2\,dt+2\int^{T_2}_{T_1}\|g\|^2_{L^2_{x,v}(\Si_-)}dt\Big),
	\end{align}
	with any sufficiently large $j\ge 1$, where $C>0$ is independent of $j$. 
	Consequently, the sequence $\{f_j\}$ is weakly-$*$ compact in the sense that, up to a subsequence, 
	\begin{align*}
		&f_j\rightharpoonup f\ \text{ weakly-$*$ in $L^2_t([T_1,T_2])L^2_{x,v}(\Si_+)$ and $L^2_t([T_1,T_2])L^2_{x}(\Omega)L^2_D$},\\
		&f_j\rightharpoonup f\ \text{ weakly-$*$ in $L^2_{x}(\Omega)L^2_v$ for any $t\in(T_1,T_2]$},
	\end{align*}
	as $j\to\infty$, where $f$ is some function satisfying estimate \eqref{weakfL2} with $f_j$ replaced by $f$. Together with the help of \eqref{eqg2l2}, we can take limit $j\to \infty$ in \eqref{weakfj} to deduce that $f$ satisfies \eqref{weakf}. 
	
	\smallskip Let $k\ge 0$ and $f$ be any weak solution to equation \eqref{L2existeq}. Similar to \eqref{434}, multiplying \eqref{L2existeq} by $\<v\>^{2k}f$, integrating over $\Omega\times\R^3_v$, and use \eqref{56} and Lemma \ref{LemRegu}, we can obtain the standard $L^2$ estimate:
	\begin{align*}
		&\pa_t\|\<v\>^kf(t)\|^2_{L^2_{x}(\Omega)L^2_v}+\|\<v\>^kf\|^2_{L^2_{x,v}(\Si_+)}
		+\vpi\wh{C}^2_0\|\<v\>^{{k+4}}f\|^2_{L^2_{x}(\Omega)L^2_v}\\
		&\quad
		+\vpi\|\<v\>^{{k+2}}\na_vf\|^2_{L^2_{x}(\Omega)L^2_v}+c_0\|\<v\>^kf\|_{L^2_{x}(\Omega)L^2_D}^2
		+N\|\<v\>^{l-2+k}f\|_{L^2_{x}(\Omega)L^2_v}^2\\
		&\qquad\le C\|\<v\>^kf(t)\|_{L^2_{x}(\Omega)L^2_v}^2+\|[\mu^{\frac{1}{10^4}}\vp,\<v\>^k\phi]\|_{L^2_{x}(\Omega)L^2_v}^2+\|\<v\>^kg\|^2_{L^2_{x,v}(\Si_-)}, 
	\end{align*}
	where we choose $\de_0>0$ in \eqref{41344} small enough. 
	This completes the proof of Lemma \ref{weakfThm}. 
\end{proof}

\subsection{\texorpdfstring{$L^2$}{L2} local Existence for linear equation with reflection}
Assume $\ve,\eta\in(0,1)$. 
 Given data $\Psi=\mu^{\frac{1}{2}}+\psi$ and $\vp,\phi$, we will derive the local-in-time $L^2$ existence to the linear regularized modified Boltzmann equation with modified reflection boundary and given source $\phi$, and dissipation $f$ and ${\eta\<v\>^l}^{}f$:
\begin{align}\label{linearre}
	\left\{
	\begin{aligned}
		 & \pa_tf+ v\cdot\na_xf = \vpi Vf+\Gamma(\Psi,f)+\Gamma(\vp,\mu^{\frac{1}{2}})\\
		 &\qquad\qquad\qquad+\phi-{\eta\<v\>^l}^{}f\quad \text{ in } [T_1,T_2]\times\Omega\times\R^3_v, \\
		 & f|_{\Si_-}=(1-\ve)Rf\quad \text{ on }[T_1,T_2]\times\Si_-, \\
		 & f(T_1,x,v)=f_{T_1}\quad \text{ in }\Omega\times\R^3_v,
	\end{aligned}\right.
\end{align}
for any small $\vpi\ge 0$, where $V$ is given by \eqref{Vf}.
To derive the solution to \eqref{linearre}, we use an iteration on the boundary condition: 
\begin{align}\label{lineariter1}
	\left\{
	\begin{aligned}
		 & \pa_tf^{n+1}+ v\cdot\na_xf^{n+1} = \vpi Vf^{n+1}+\Gamma(\Psi,f^{n+1})+\Gamma(\vp,\mu^{\frac{1}{2}})\\
		 &\qquad\qquad\qquad+\phi-{\eta\<v\>^l}^{}f^{n+1}
		 \quad \text{ in } [T_1,T_2]\times\Omega\times\R^3_v, \\
		 & f^{n+1}|_{\Si_-}=(1-\ve)Rf^{n}\quad \text{ on }[T_1,T_2]\times\Si_-, \\
		 & f^{n+1}(T_1,x,v)=f_{T_1}\quad \text{ in }\Omega\times\R^3_v.
	\end{aligned}\right.
\end{align}
with $f^0=0$. 
To take the limit $n\to\infty$, we need to derive its convergence as follows. 
\begin{Thm}[$L^2$ existence for linear equation]
	\label{LemLinRe}
	Assume that $\vpi,\eta\ge 0$ be small enough constants, $l\ge 0$, $\ve\in(0,1)$ and $0\le T_1<T_2$. 
	Suppose $\psi,\vp,\phi$ and $f_{T_1}$ satisfy
	\begin{align}\label{ve2}
		\begin{aligned}
		\|\<v\>^{4}\psi\|_{L^\infty_t([T_1,T_2])L^\infty_{x}(\Omega)L^\infty_v}& \le \de_0, \\
		\|[\psi,\vp,\<v\>^2\phi]\|^2_{L^2_t([T_1,T_2])L^2_{x}(\Omega)L^2_v}+
		\|f_{T_1}\|^2_{L^2_x(\Omega)L^2_v(\R^3_v)} & = \ti C.
		\end{aligned}
	\end{align}
	for some constant $\ti C>0$ and sufficiently small $\de_0>0$.
	Then there exists a unique solution $f$ to \eqref{linearre} in the sense that for any $\Phi\in C^\infty_c(\R^7_{t,x,v})$, 
	\begin{multline}\label{weakfre}
		(f(T_2),\Phi(T_2))_{L^2_x(\Omega)L^2_v}-(f,(\pa_t+v\cdot\na_x)\Phi)_{L^2_{t,x,v}([T_1,T_2]\times\Omega\times\R^3_v)}\\
		+(f,\Phi)_{L^2_t([T_1,T_2])L^2_{x,v}(\Si_+)}
		=(f_{T_1},\Phi(T_1))_{L^2_x(\Omega)L^2_v}
		+(1-\ve)(Rf,\Phi)_{L^2_t([T_1,T_2])L^2_{x,v}(\Si_-)}\\
		+\big(\vpi Vf+\Gamma(\Psi,f)+\Gamma(\vp,\mu^{\frac{1}{2}})+\phi-{\eta\<v\>^l}^{}f,\Phi\big)_{L^2_{t,x,v}([T_1,T_2]\times\Omega\times\R^3_v)},
	\end{multline}
	 that satisfies 
	\begin{multline}
		\label{913}
		\|\<v\>^kf\|_{L^\infty_t([T_1,T_2])L^2_x(\Omega)L^2_v}^2
		+c_\al\|\<v\>^kf\|_{L^2_t([T_1,T_2])L^2_{x,v}(\Si_+)}^2
		+c_0\|f\|_{L^2_t([T_1,T_2])L^2_x(\Omega)L^2_D}^2\\
		+\vpi\|[\wh{C}^{}_0\<v\>^{{k+4}}f,\<v\>^{{k+2}}\na_vf]\|_{L^2_t([T_1,T_2])L^2_x(\Omega)L^2_v}^2
		+\eta\|\<v\>^{k+\frac{l}{2}}f\|_{L^2_t([T_1,T_2])L^2_x(\Omega)L^2_v}^2\\
		\le
		C_{|T_2-T_1|}\big(\|\<v\>^kf(T_1)\|_{L^2_x(\Omega)L^2_v}^2+\|[\vp,\<v\>^k\phi]\|_{L^2_t([T_1,T_2])L^2_x(\Omega)L^2_v(\R^3_v)}^2\big),
	\end{multline}
	for some constant $C_{|T_2-T_1|}>0$ that is independent of $\vpi,\ve,\eta$. Note the underlying time interval is $[T_1,T_2]$. 
\end{Thm}
\begin{proof}
	Let $f^0=0$ and $f^n$ be the solution to equation \eqref{lineariter1}. 
	By the assumptions in \eqref{ve2}, we can apply Theorem \ref{weakfThm} to obtain the local solution $f$ to equation \eqref{lineariter1} with $n=0$:
	\begin{align*}
		\left\{
		\begin{aligned}
			 & \pa_tf^{1}+ v\cdot\na_xf^{1} = \vpi Vf^{1}+\Gamma(\Psi,f^{1})+\Gamma(\vp,\mu^{\frac{1}{2}})\\
			 &\qquad\qquad\qquad+\phi-{\eta\<v\>^l}^{}f^1\quad \text{ in } [T_1,T_2]\times\Omega\times\R^3_v, \\
			 & f^{1}|_{\Si_-}=0\quad\quad\quad \text{ on }[T_1,T_2]\times\Si_-, \\
			 & f^{1}(T_1,x,v)=f_{T_1}\quad \text{ in }\Omega\times\R^3_v, 
		\end{aligned}\right.
	\end{align*}
which satisfies 
	\begin{multline}\label{L2u}
	\|f^1\|_{L^\infty_tL^2_x(\Omega)L^2_v}^2
	+\|f^1\|_{L^2_tL^2_{x,v}(\Si_+)}^2
	+c_0\|f^1\|_{L^2_tL^2_x(\Omega)L^2_D}^2
	+\vpi\|[\wh{C}^{}_0\<v\>^{{4}}f^1,\<v\>^{{2}}\na_vf^1]\|_{L^2_tL^2_x(\Omega)L^2_v}^2
	\\+\eta\|\<v\>^{\frac{l}{2}}f^1\|_{L^2_tL^2_x(\Omega)L^2_v}^2
		\le
		e^{C(T_2-T_1)}\Big(\|f_{T_1}\|_{L^2_x(\Omega)L^2_v}^2+\|[\mu^{\frac{1}{10^4}}\vp,\<v\>^2\phi]\|_{L^2_tL^2_x(\Omega)L^2_v}^2\Big), 
	\end{multline}
	where we choose $\wh{C}_0>0$ sufficiently large.
	In order to obtain solution $f^{n+1}$ to equation \eqref{lineariter1}, we consider $h^n=f^{n+1}-f^n$ $(n\ge 0)$ that satisfies
	$h^0=f^1$, and
	for $n\ge 1$,
	\begin{align}\label{lineariter3}
		\left\{
		\begin{aligned}
			 & \pa_th^n+ v\cdot\na_xh^n =\vpi Vh^n+\Gamma(\Psi,h^n)\\&\qquad\qquad\qquad\qquad-{\eta\<v\>^l}^{}h^n
\quad \text{ in } [T_1,T_2]\times\Omega\times\R^3_v, \\
			 & h^n|_{\Si_-}=(1-\ve)Rh^{n-1}\quad \text{ on }[T_1,T_2]\times\Si_-, \\
			 & h^n(T_1,x,v)=0\quad \text{ in }\Omega\times\R^3_v.
		\end{aligned}\right.
	\end{align}
	Assume the iteration assumption
	\begin{align}
			\label{911}
	\int^{T_2}_{T_1}\int_{\Si_-}|v\cdot n||Rh^{n-1}|^2\,dS(x)dvdt & <\infty.
	\end{align}
	Then the proof of the existence of $h^n$ $(n\ge 1)$ to equation \eqref{lineariter3} is given by Theorem \ref{weakfThm}. 
	 By \eqref{L2u} and Lemma \ref{LemR}, we know that \eqref{911} is fulfilled when $n=1$. 
	Then for $n\ge 1$, by taking $L^2$ inner product of \eqref{lineariter3} with $2h^n$ over $[T_1,T_2]\times\Omega\times\R^3_v$ and using Lemma \ref{LemR} to control the boundary term, we obtain 
	\begin{align}\label{L2es5}\notag
		&\|h^n\|_{L^\infty_tL^2_x(\Omega)L^2_v}^2
		+\|h^n\|_{L^2_tL^2_{x,v}(\Si_+)}^2
		+c_0\|h^n\|_{L^2_tL^2_x(\Omega)L^2_D}^2
		+\vpi\|[\wh{C}^{}_0\<v\>^{{4}}h^n,\na_v(\<v\>^{{2}}h^n)]\|_{L^2_tL^2_x(\Omega)L^2_v}^2\\
		&\notag\quad+M\|h^n\|_{L^2_tL^2_x(\Omega)L^2_v}^2
		+\eta\|\<v\>^{\frac{l}{2}}h^n\|_{L^2_tL^2_x(\Omega)L^2_v}^2
		\\
		&\notag\qquad\le(1-\ve)^2\|Rh^{n-1}\|_{L^2_tL^2_{x,v}(\Si_+)}^2\le\cdots
		\le (1-\ve)^{2n}\|Rh^0\|_{L^2_tL^2_{x,v}(\Si_-)}^2\\
		&\qquad\le (1-\ve)^{2n}e^{C(T_2-T_1)}\Big(\|f_{T_1}\|_{L^2_x(\Omega)L^2_v}^2+\|\mu^{\frac{1}{80}}\vp\|_{L^2_tL^2_x(\Omega)L^2_v}^2\Big). 
	\end{align}
The estimate \eqref{L2es5} implies that the iteration assumption \eqref{911} is fulfilled for $n\ge 2$. 
	Since $f^{n+1}=\sum_{j=0}^nh^j$ and $h^n$ satisfies \eqref{lineariter3}, we know that $f^{n+1}$ solves equation \eqref{lineariter1}. From the estimate \eqref{L2es5}, we have 
	\begin{align}\label{5266}\notag
	&\|f^{n+1}\|_{L^\infty_tL^2_x(\Omega)L^2_v}^2
	+\|f^{n+1}\|_{L^2_tL^2_{x,v}(\Si_+)}^2
	+c_0\|f^{n+1}\|_{L^2_tL^2_x(\Omega)L^2_D}^2\\
		&\notag\quad+\vpi\|[\wh{C}^{}_0\<v\>^{{4}}f^{n+1},\<v\>^{{2}}\na_vf^{n+1}]\|_{L^2_tL^2_x(\Omega)L^2_v}^2+\eta\|\<v\>^{\frac{l}{2}}f^{n+1}\|_{L^2_tL^2_x(\Omega)L^2_v}^2\\
		&\qquad\le \ti Ce^{C(T_2-T_1)}\sum_{j=1}^{n}(1-\ve)^{2n}\le C_\ve\ti C e^{C(T_2-T_1)}.
	\end{align}
	On the other hand, recalling that $h^n=f^{n+1}-f^n$ and using \eqref{L2es5}, we know that $\{f^n\}$ is a Cauchy sequence in the corresponding spaces on the left-hand side of \eqref{5266}.
	By Lemma \ref{LemR} and boundary condition of \eqref{lineariter1}, $\{f^n\}$ is also a Cauchy sequence in $L^2_tL^2_{x,v}(\Si_-)$.
	Thus, there exists a function $f$ belonging to $L^\infty_tL^2_x(\Omega)L^2_v$, $L^2_tL^2_x(\Omega)L^2_D$, $L^2_tL^2_{x,v}(\Si_+)$ and $L^2_tL^2_{x,v}(\Si_-)$ such that
	\begin{align}\label{strongconver}\begin{aligned}
		f^n\to f & \text{ in $L^\infty_tL^2_{x,v}([T_1,T_2]\times\Omega\times\R^3_v)$, $L^2_{t,x}L^2_D([T_1,T_2]\times\Omega\times\R^3_v)$}, \\ &\quad\ \text{ $L^2_tL^2_{x,v}([T_1,T_2]\times\Si_+)$ and $L^2_tL^2_{x,v}([T_1,T_2]\times\Si_-)$, }
		\end{aligned}
	\end{align}
	as $n\to\infty$.
	Rewriting equation \eqref{lineariter1} in the weak form for $n\ge 2$: for any function $\Phi\in C^\infty_c(\R_t\times\R^3_x\times\R^3_v)$,
	\begin{multline}\label{weak3}
		(f^{n+1}(T_2),\Phi(T_2))_{L^2_{x}(\Omega)L^2_v}-(f^{n+1},(\pa_t+v\cdot\na_x)\Phi)_{L^2_tL^2_x(\Omega)L^2_v}\\
		+(f^{n+1},\Phi)_{L^2_tL^2_{x,v}(\Si_+)}
		=(f_{T_1},\Phi(T_1))_{L^2_x(\Omega)L^2_v}
		+(1-\ve)(Rf^{n},\Phi)_{L^2_tL^2_{x,v}(\Si_+)}\\
		+\big(\vpi Vf^{n+1}+\Gamma(\Psi,f^{n+1})+\Gamma(\vp,\mu^{\frac{1}{2}})+\phi-{\eta\<v\>^l}^{}f^{n+1},\Phi\big)_{L^2_tL^2_x(\Omega)L^2_v}.
	\end{multline}
	Using \eqref{Gammaes}, \eqref{strongconver} and passing the limit $n\to\infty$ in \eqref{weak3}, we obtain \eqref{weakfre}. 
	This shows that $f$ is the solution to \eqref{linearre}.
	Moreover, we give a short proof of the $L^2$ energy estimate; a detailed one will be given in Section \ref{Sec12}. By taking $L^2$ inner product of \eqref{linearre} with $f$ and $\<v\>^{2k}f$ over $\Omega\times\R^3_v$ for any $k\ge 0$, and using \eqref{56}, we have
	\begin{multline}\label{44enera}
		\pa_t\|f\|_{L^2_x(\Omega)L^2_v}^2+\|f\|^2_{L^2_{x,v}(\Si_+)}+2\vpi\|[\wh{C}^{}_0\<v\>^{{4}}f,\na_v(\<v\>^{{2}}f)]\|^2_{L^2_x(\Omega)L^2_v}+c_0\|f\|_{L^2_x(\Omega)L^2_D}^2\\
		\le \|[\vp,\phi]\|_{L^2_x(\Omega)L^2_v}^2
		+C\|f\|_{L^2_x(\Omega)L^2_v}^2
		+\|f\|^2_{L^2_{x,v}(\Si_-)}
		-2\eta\|\<v\>^{\frac{l}{2}}f\|_{L^2_x(\Omega)L^2_v}^2,
	\end{multline}
	and 
	\begin{multline}\label{44enerb}
		\pa_t\|\<v\>^{k}f\|_{L^2_x(\Omega)L^2_v}^2+\|\<v\>^{k}f\|^2_{L^2_{x,v}(\Si_+)}+\vpi\|[\wh{C}^{}_0\<v\>^{{k+4}}f,\na_v(\<v\>^{{k+2}}f)]\|^2_{L^2_x(\Omega)L^2_v}+c_0\|\<v\>^{k}f\|_{L^2_x(\Omega)L^2_D}^2\\
		\le \|[\vp,\<v\>^k\phi]\|_{L^2_x(\Omega)L^2_v}^2
		+C\|f\|_{L^2_x(\Omega)L^2_v}^2
		+\|\<v\>^{k}f\|^2_{L^2_{x,v}(\Si_-)},
	\end{multline}
	where we chose $\de_0>0$ in \eqref{ve2} small enough, with some constant $C>0$ that is independent of $\vpi,\ve,\eta$. 
	For the boundary terms, we use \eqref{101} and \eqref{101w} to obtain 
	\begin{align}\label{44bda}
		\begin{aligned}
			\|Rf\|^2_{L^2_{x,v}(\Si_-)}&=\|f\|^2_{L^2_{x,v}(\Si_+)}-\al\|f-R_Df\|^2_{L^2_{x,v}(\Si_+)},\\
			\|Rf\|^2_{L^2_{x,v}(\Si_-)}&\le\|f\|^2_{L^2_{x,v}(\Si_+)}-\frac{\al}{2}\|f\|_{L^2_{x,v}(\Si_+)}^2+\al\|R_Df\|^2,\\
			\|\<v\>^{k}Rf\|^2_{L^2_{x,v}(\Si_-)}&\le(1-\al)^2\|\<v\>^{k}f\|^2_{L^2_{x,v}(\Si_+)} + C_k\|f\|^2_{L^2_{x,v}(\Si_+)}. 
		\end{aligned}
	\end{align}
We plug these three boundary estimates into \eqref{44enera} and \eqref{44enerb} to obtain three energy estimates and take a proper combination.
For the term $\|R_D f\|_{L^2_{x,v}(\Si_+)}^2$, we let $\de>0$, denote $\chi^{+}_{\de}=\chi^{+}_{\de}(t,x,v;T_1+N\de^3)$ by \eqref{chide}. Then we rewrite  
\begin{align*}
	\int^{T_2}_{T_1}\|R_D f\|_{L^2_{x,v}(\Si_+)}^2\,dt
	&= \int^{T_2}_{T_1}\int_{\pa\Omega}c_\mu\Big|\int_{v'\cdot n(x)>0}\{v'\cdot n(x)\}f(v')\mu^{\frac{1}{2}}(v')\,dv'\Big|^2\,dS(x)dt\\
	&\le\Big(\int_{T_1+[(T_2-T_1)/\de^3]\de^3}^{T_2}+\sum_{N=0}^{[(T_2-T_1)/\de^3]-1}\int_{T_1+N\de^3}^{T_1+(N+1)\de^3}\Big)\big(\cdots\big)\,dt. 
\end{align*}
Splitting $f(v')=(1-\chi^{+}_{\de})f(v')+\chi^{+}_{\de} f(v')$ and applying trace Lemma \ref{diffboundLem}, i.e. \eqref{diffbound1}, \eqref{diffbound2j} and \eqref{273}, to each term, we have 
	\begin{align*}\notag
		&\int^{T_2}_{T_1}\|R_D f\|_{L^2_{x,v}(\Si_+)}^2\,dt
	\le C(\de^4+e^{-\de^{-1/2}})
		\|f\|^2_{L^2_tL^2_{x,v}(\Si_+)}\\
		&\quad\notag+2\sum_{N=0}^{[(T_2-T_1)/\de^3]}\bigg\{\int^{T_1+N\de^3}_{T_1}\big(\vpi Vf+\Gamma(\Psi,f)+\Gamma(\vp,\mu^{\frac{1}{2}})+\phi-\eta\<v\>^lf,f\big)_{L^2_x(\Omega)L^2_v}\,dt\\
		&\quad+\int^{T_2}_{T_1}
		\big(\vpi Vf+\Gamma(\Psi,f)+\Gamma(\vp,\mu^{\frac{1}{2}})+\phi-\eta\<v\>^lf,\chi^{+}_{\de} f\big)_{L^2_x(\Omega)L^2_v}\,dt\bigg\},
	\end{align*}
	where $\chi^{+}_{\de}$ depends on $N,\de$. 
	Then by collisional estimates \eqref{L1} and \eqref{L2}, with upper bound of $\chi^{+}_{\de}$ in \eqref{chideesti}, and Lemma \ref{GachideLem}), we continue it as 
	\begin{align}\label{44bdc}
		\int^{T_2}_{T_1}\|R_D f\|_{L^2_{x,v}(\Si_+)}^2\,dt&\le C(\de^4+e^{-\de^{-1/2}})\|f\|^2_{L^2_t([T_1,T_2])L^2_{x,v}(\Si_+)}+C_\de\|f\|^2_{L^2_t([T_1,T_2])L^2_x(\Omega)L^2_D}. 
	\end{align}
	Therefore, using Gr\"{o}nwall's inequality, integrating \eqref{44enera} and \eqref{44enerb} on $t\in[T_1,T_2]$ with estimates \eqref{44bda} and \eqref{44bdc}, and taking proper combination, we deduce that the solution $f$ to equation \eqref{linearre} satisfies the weighted $L^2$ energy estimate \eqref{913}. 
	
	\smallskip To prove the uniqueness, we assume that $f,g$ are two solutions to equation \eqref{linearre}. Then $h=f-g$ satisfies 
	\begin{align}\label{linearreh}
		\left\{
		\begin{aligned}
			& \pa_th+ v\cdot\na_xh = \vpi Vh+\Gamma(\Psi,h)-{\eta\<v\>^l}^{}h\quad \text{ in } [T_1,T_2]\times\Omega\times\R^3_v, \\
			& h|_{\Si_-}=(1-\ve)Rh\quad \text{ on }[T_1,T_2]\times\Si_-, \\
			& h(T_1,x,v)=0\quad \text{ in }\Omega\times\R^3_v.
		\end{aligned}\right.
	\end{align}
	Similar to \eqref{913}, taking $L^2$ inner product of \eqref{linearreh} with $2h$ over $[T_1,T_2]\times\Omega\times\R^3_v$, we have 
	\begin{align*}
		\|h\|_{L^\infty_t([T_1,T_2])L^2_x(\Omega)L^2_v}^2
		+c_0\|h\|_{L^2_tL^2_x(\Omega)L^2_D}^2\le 0.
	\end{align*}
	which implies $f=g$ and the uniqueness of equation \eqref{linearre}. 
	This completes the proof of Theorem \ref{LemLinRe}.
\end{proof}

\subsection{Forward-backward extension Lemma}
In this subsection, we give the forward-backward extension method as stated in Subsection \ref{Secexten} and equation \eqref{omegac}. 
Denote $D_{in},D_{out}$ as in \eqref{Dinoutwt}. 

\begin{Lem}[Forward-backward extension]
	\label{extendreverThm}
	let $0\le T_1<T_2$, $l>0$ be the largest polynomial-weight index, and $g$ be the given (inflow and outflow with respect to interior $\Omega$) boundary condition satisfying
	\begin{align*}
		\int^{T_2}_{T_1}\int_{\pa\Omega}|v\cdot n||g|^2\,dS(x)dvdt<\infty.
	\end{align*}
	Then there exists a weak solution $f$ to the equation
	\begin{align}\label{omegac1}
		\left\{
		\begin{aligned}
			&\pa_tf+v\cdot\na_xf+E\cdot\na_vf=P^2f& \text{ in } [T_1,T_2]\times D_{in},\\
			&\pa_tf+v\cdot\na_xf+E\cdot\na_vf=-P^2f& \text{ in } [T_1,T_2]\times D_{out},\\
			& f|_{\pa\Omega}=g \quad& \text{ on }[T_1,T_2]\times(\Si_+\cup\Si_-),\\
			&f(T_1,x,v)=0\quad& \text{ in }D_{out},\\
			&f(T_2,x,v)=0\quad& \text{ in }D_{in}, 
		\end{aligned}\right.
	\end{align}
	where $\{(x,v)\in\ol\Omega^c\times\R^3\,:\,v\cdot n(x)=0\}$ has normal vector $\ti n(x,v)\in\R^6$ that satisfies the vanishing boundary property: 
	\begin{align}\label{traceh}
		(v,E)\cdot \ti n(x,v)=0,\quad \text{ on }\ \{(x,v)\in\ol\Omega^c\times\R^3\,:\,v\cdot n(x)=0\}.
	\end{align}
	Moreover, the field $E=E(x,v)$ and positive function $P$ are given by 
	\begin{align}\label{EP}
		\begin{aligned}
			E(x,v)&=-v_i\pa_{x_j}n_i(x)v_j\frac{n(x)}{|n(x)|^2},\\
			P(x,v)&=\wh C_l\Big(\|[1,n,\na_xn]\|_{L^\infty_x}\frac{\<v\>^2}{|n(x)|}+1\Big). 
		\end{aligned}
	\end{align}
	Here, $\wh C_l$ is a sufficiently large constant depending on the largest weight index $l$ that will be chosen large later. 
	Implicit summation over repeated indices is taken hereafter. 
	The weak sense of solution means that for any function $\Phi\in C^\infty_c(\R_t\times\R^3_x\times\R^3_v)$,
	\begin{align}\label{hheq5}\notag
		&(f(T_2),\Phi(T_2))_{L^2_{x,v}(D_{out})}-(f(T_1),\Phi(T_1))_{L^2_{x,v}(D_{in})}
		-\int^{T_2}_{T_1}(f,(\pa_t+v\cdot\na_x)\Phi)_{L^2_{x,v}(\ol\Omega^c\times\R^3_v)}\,dt\\
		&\notag\qquad-\int^{T_2}_{T_1}(f,\na_v\cdot(E\Phi))_{L^2_{x,v}(\ol\Omega^c\times\R^3_v)}\,dt
		=
		\int^{T_2}_{T_1}(g,\Phi)_{L^2_{x,v}(\Si_+)}\,dt-\int^{T_2}_{T_1}(g,\Phi)_{L^2_{x,v}(\Si_-)}\,dt\\
		&\qquad\qquad\qquad+\int^{T_2}_{T_1}(P^2f,\Phi)_{L^2_{x,v}(D_{in})}\,dt-\int^{T_2}_{T_1}(P^2f,\Phi)_{L^2_{x,v}(D_{out})}\,dt. 
	\end{align}
	Note that we use the outward normal vector $n(x)$ of $\pa\Omega$ given in \eqref{naxi1}. 
	Moreover, for any weak solution $f$ to equation \eqref{omegac1}, $p\ge 2$ and $k\in[0,l]$, we have $L^p$ energy estimate:
	\begin{align}\label{hheqes5}
		\pa_t\|\<v\>^kf^{\frac{p}{2}}\|^2_{L^2_{x,v}(\ol\Omega^c\times\R^3_v)}
		+\frac{1}{2}\|\<v\>^{k}Pf^{\frac{p}{2}}\|^2_{L^2_{x,v}(\ol\Omega^c\times\R^3_v)}
		\le 2\|\<v\>^kg^{\frac{p}{2}}\|_{L^2(\Si_-\cup\Si_+)}^2.
	\end{align}
\end{Lem}

\begin{Rem}
	The vanishing property \eqref{traceh} means that the particles flowing starting at $\Si_-$ ($\Si_+$, resp.) within $D_{in}$ ($D_{out}$, resp.) along the trajectory will not cross the boundary between $D_{in}$ and $D_{out}$ and will be confined within $D_{in}$ ($D_{out}$, resp.). In fact, the trajectory is on the same level in the sense of \eqref{430} below. 
\end{Rem}

\begin{proof}
	We will consider the domains $D_{in}$ and $D_{out}$ separately, and use the method of characteristics. We focus on solving the equation within $[T_1, T_2]\times D_{in}$ while the part $D_{out}$ is similar and even simpler. 
	
	\smallskip\noindent{\bf Step 1. Characteristic setting.}
	By Lemma \ref{gextension}, there exists a smooth approximation $g_j\in C^\infty_c(\R^7_{t,x,v})$ $(j\ge 1)$ of $g$ such that \eqref{eqg2limit} is valid. 
	Write $h(t)=f(T_1+T_2-t)$ for $t\in[T_1,T_2]$ and use the approximation $g_{j}$ instead of $g$ as the inflow boundary data. Then $h(t)$ satisfies 
	\begin{align}\label{omegac2}
		\begin{cases}
				\pa_th-v\cdot\na_xh-E(x,v)\cdot\na_vh
				=-P^2(x,v)h,
			 &\text{ in }[T_1,T_2]\times D_{in},\\
			h(t,x,v)=g_{j}(T_1+T_2-t,x,v), & \text{ on }[T_1,T_2]\times\Si_-. 
		\end{cases}
	\end{align}
	Let $(t,x,v)\in[T_1, T_2]\times D_{in}$ be any point in the inflow region, and we construct the corresponding characteristic curve starting at $(t,x,v)$ by 
	\begin{align}\label{chara}
		(
			X(s),V(s)):=(
				X(s;t,x,v),V(s;t,x,v))
	\end{align}
	for $s\in[T_1,T_2]$, which is the solution of 
	\begin{align}\label{XVeq}\left\{\begin{aligned}
			&X'(s)=-V(s),
			\quad V'(s)=-E(X(s),V(s)),\\
			&
			X(t)=x,\quad V(t)=v. 
		\end{aligned}\right. 
	\end{align}
	Then $V(t)\cdot n(X(t))=v\cdot n(x)<0$, 
	\begin{align*}
			X(s)=x-\int^s_tV(r)
			\,dr,\quad V(s)=v-\int^s_tE(X(r),V(r))
			\,dr. 
	\end{align*}
	For any $s\ge t$, 
	\begin{align*}
		\frac{d}{ds}\big(V(s)\cdot n(X(s))\big)
		&=V'(s)\cdot n(X(s))+V(s)\cdot\pa_{x_j}n(X(s))X_j'(s)\\
		&=-E(X(s),V(s))\cdot n(X(s))-V_i(s)\pa_{x_j}n_i(X(s))V_j(s)\\
		&=V_i(s)\pa_{x_j}n_i(X(s))V_j(s)\frac{n(X(s))}{|n(X(s))|^2}\cdot n(X(s))
		-V_i(s)\pa_{x_j}n_i(X(s))V_j(s)
		\\&=0, 
	\end{align*}
	where we implicitly summed the repeated indices and made the \emph{a priori} assumption that $|n(X(r))|>0$ for any $r\in[T_1,T_2]$. 
	Thus, along the characteristic curve, $V(s)\cdot n(X(s))$ is constant and hence, 
	\begin{align}\label{430}
		V(s)\cdot n(X(s))=v\cdot n(x)<0, 
	\end{align}
	for any $s\ge t$. This further implies $|n(X(s))|>0$ and hence, closes the \emph{a priori} assumption $|n(X(r))|>0$ for any $r\in[T_1,T_2]$. 
	Therefore, the equation \eqref{XVeq} is always solvable for any $s$ until it hits the boundary $\pa\big([T_1,T_2]\times D_{in}\big)$, and any particles starting at $(t,x,v)\in [T_1, T_2]\times D_{in}$ will be confined in the region $D_{in}$ along the characteristic curve. 
	
	\medskip\noindent{\bf Step 2. Strong solution.}
	Next, we search for a strong solution and then pass the limit $\ve\to 0$. 
	We denote the backward stopping time $t_\b(t,x,v)$ as the time when the particle starting at $(t,x,v)\in (T_1, T_2]\times D_{in}$ first hits the boundary $\pa D_{in}$ along the characteristics \eqref{chara}:
	\begin{align*}
		t_\b(t,x,v):=\sup\big\{T_1\le\tau\le t\,:\,\big(X(\tau;t,x,v),V(\tau;t,x,v)\big)\in \pa D_{in}\big\},
		\quad \text{ if it exists};
	\end{align*}
	otherwise we set $t_\b(t,x,v)=T_1$ for the case that the particle never hits the boundary $\pa D_{in}$ within time $[T_1,T_2]$ (this means that it will hit $t=T_1$ backwardly).
	Then the solution $h$ to equation \eqref{omegac2} in $D_{in}$ is given by 
	\begin{align}\label{hhh}\notag
		h(t,x,v)\exp\Big(\int^t_{t_\b}P^2(X(r),V(r))\,dr\Big)&=h(t_\b,X(t_\b),V(t_\b))\\
		&=g_{j}(t_\b,X(t_\b),V(t_\b)),
	\end{align}
	Therefore, the solution $h$ is well defined in $[T_1,T_2]\times D_{in}$. Moreover, within open set $(T_1,T_2)\times D_{in}$, one can use \cite[Lemma 2]{Guo2009} to show that 
	$t_\b(t,x,v)$ is at least a $C^3_{t,x,v}$ function (since $\pa\Omega$ is $C^3$). 
	Then $h(t,x,v)$ is also $C^3_{t,x,v}$ in $(T_1,T_2)\times D_{in}$, and the standard characteristic method verifies that $h(t,x,v)$ given by \eqref{hhh} is indeed a strong solution to equation \eqref{omegac2} in $D_{in}$. 
	
	\smallskip\noindent{\bf Step 3. Confinement and energy estimate.}
	First, we have shown that the flow $h$ is confined in $D_{in}$ as in \eqref{430}. 
Moreover, we calculate the boundary measure on $\pa D_{in}\setminus\Si_-$, i.e. the boundary between $D_{in}$ and $D_{out}$. 
Note that $\pa D_{in}\setminus\Si_-$ is given by 
	\begin{align*}
		\pa D_{in}\setminus\Si_-=\{(x,v)\in\ol\Omega^c\times\R^3\,:\, v\cdot n(x)=0\},
	\end{align*}
	where $n(x)$ is given in \eqref{naxi1}. Therefore, we can choose the outward normal vector on the boundary $\pa D_{in}\setminus\Si_-$ by 
	\begin{align*}
		\ti n(x,v) = \frac{\na_{x,v}\big(v\cdot n(x)\big)}{\big|\na_{x,v}\big(v\cdot n(x)\big)\big|}
		=\frac{\big(v_i\na_xn_i(x),n(x)\big)}{\big|\na_{x,v}\big(v\cdot n(x)\big)\big|},
	\end{align*}
	where repeated indices are summed implicitly.
	This implies \eqref{traceh}, i.e.
	\begin{align*}
		(v,E)\cdot \ti n(x,v) = \frac{v_jv_i\pa_{x_j}n_i(x)+E\cdot n(x)}{\big|\na_{x,v}\big(v\cdot n(x)\big)\big|}=0. 
	\end{align*}
Then for any function $\Phi\in C^\infty_c(\R^7)$, $h$ satisfies the weak form 
\begin{align}\label{hheqa}\notag
	(h(T_2),\Phi(T_2))_{L^2_{x,v}(D_{in})}&+(h,v\cdot\na_x\Phi+\na_v\cdot(E\Phi))_{L^2_tL^2_{x,v}(D_{in})}\\&=(g_{j},\Phi)_{L^2_tL^2_{x,v}(\Si_-)}
	+(P^2h,\Phi)_{L^2_tL^2_{x,v}(D_{in})}, 
\end{align}
with $h=g_{j}$ on $\Si_-$. Then the standard $L^2$ estimate of equation \eqref{hheqa} gives 
\begin{align}\label{466}\notag
	&\sup_{T_1\le t\le T_2}\|h(t)\|^2_{L^2_{x,v}(D_{in})}
	+\frac{1}{2}\int^{T_2}_{T_1}\|Ph(t)\|^2_{L^2_{x,v}(D_{in})}\,dt\\
	&\quad\notag+\frac{1}{2}\int^{T_2}_{T_1}\int_{\pa D_{in}\setminus\Si_-}(v,E)\cdot \ti n(x,v)|h(t,x,v)|^2\,dxdvdt\\
	&\quad\le \int^{T_2}_{T_1}\int_{\Si_-}|v\cdot n||g_{j}|^2\,dS(x)dvdt,
\end{align}
where $(v,E)\cdot \ti n(x,v)=0$, and we used \begin{align*}
	|\na_v\cdot E|=
	\bigg|-\pa_{x_j}n_k(x)v_j\frac{n_k(x)}{|n(x)|^2}
	-v_i\pa_{x_k}n_i(x)\frac{n_k(x)}{|n(x)|^2}\bigg|
	\le \frac{\<v\>\|\na_xn\|_{L^\infty_x}}{|n(x)|^2}
	\le \frac{P^2}{4}. 
\end{align*} 
With the energy estimate \eqref{466} and noticing that \eqref{hheqa} is just a linear equation, by writing $h=h_{j}$ to emphasize its dependence on $g_{j}$ and using the strong convergence of $g_{j}$ in \eqref{eqg2limit}, it's direct to obtain the $L^2$ estimate of $h_{j_1}-h_{j_2}$: 
	\begin{align*}
		&\sup_{T_1\le t\le T_2}\|h_{j_1}-h_{j_2}\|^2_{L^2_{x,v}(D_{in})}
		+\frac{1}{2}\int^{T_2}_{T_1}\|Ph_{j_1}-Ph_{j_2}\|^2_{L^2_{x,v}(D_{in})}\,dt\\
		&\quad
		\le \int^{T_2}_{T_1}\int_{\Si_-}|v\cdot n||g_{\ve_1,j_1}-g_{\ve_2,j_2}|^2\,dS(x)dvdt \to 0, 
	\end{align*}
	as $j_1,j_2\to \infty$. 
	Therefore, the sequence $\{h_{j}\}$ possesses a strong limit, denoted by $h$, as $j\to\infty$.
	Such a strong limit satisfies 
	\begin{align*}
		\sup_{T_1\le t\le T_2}\|h(t)\|^2_{L^2_{x,v}(D_{in})}
		+\frac{1}{2}\int^{T_2}_{T_1}\|Ph(t)\|^2_{L^2_{x,v}(D_{in})}\,dt
		\le \int^{T_2}_{T_1}\int_{\Si_-}|v\cdot n||g|^2\,dS(x)dvdt. 
	\end{align*}
	By taking limit in \eqref{hheqa} and returning to original time coordinate $f(t):=h(T_1+T_2-t)$, $f$ solves equation \eqref{omegac1} in $D_{in}$. For the weighted energy estimate, the $L^p$ energy estimate of \eqref{omegac1} implies 
	\begin{align*}
		\pa_t\|\<v\>^kf^{\frac{p}{2}}\|^2_{L^2_{x,v}(D_{in})}
		+\frac{1}{2}\|\<v\>^{k}Pf^{\frac{p}{2}}\|^2_{L^2_{x,v}(D_{in})}
		\le 2\|\<v\>^kg^{\frac{p}{2}}\|_{L^2_{x,v}(\Si_-)}^2, 
	\end{align*}
	for any $k\in[0,l]$. 
	This completes the existence and estimates in $D_{in}$.

	\smallskip \noindent{\bf Step 4. Outflow region.}
	The existence and energy estimate in $D_{out}$ is similar and simpler, so we omit the details, but only construct the characteristic and verify the property of confinement. 
	
	\smallskip 
	For any point $(t,x,v)\in[T_1, T_2]\times D_{out}$ in the outflow region, we construct the characteristic curve starting at $(t,x,v)$ by 
	\begin{align*}
		(X(s),V(s)):=(X(s;t,x,v),V(s;t,x,v))
	\end{align*}
	for $s\in[T_1,T_2]$, which is the solution of 
	\begin{align*}\left\{\begin{aligned}
			&X'(s)=V(s),\quad V'(s)=E(X(s),V(s)),\\
			&X(t)=x,\quad V(t)=v. 
		\end{aligned}\right. 
	\end{align*}
	Then $V(t)\cdot n(X(t))=v\cdot n(x)>0$, 
	and for any $s\ge t$, 
	\begin{align*}
		\frac{d}{ds}\big(V(s)\cdot n(X(s))\big)
		&=\big(V'(s)\cdot n(X(s))+V(s)\cdot\pa_{x_j}n(X(s))X_j'(s)\big)\\
		&=E(X(s),V(s))\cdot n(X(s))+V_i(s)\pa_{x_j}n_i(X(s))V_j(s)\\
		&=0,
	\end{align*}
	where we have also made the \emph{a priori} assumption that $|n(X(r))|>0$ for any $r\in[T_1,T_2]$. 
	Thus, along the characteristic curve, we have 
	\begin{align*}
		V(s)\cdot n(X(s))=v\cdot n(x)>0, 
	\end{align*}
	which implies $|n(X(s))|>0$ and closes the \emph{a priori} assumption $|n(X(r))|>0$.
	Continuing the calculations in Steps 1--3, we can obtain the existence and $L^p$ estimates of $f$ to equation \eqref{omegac1} in $D_{out}$. 
	This completes the proof of Lemma \ref{extendreverThm}.
\end{proof}

With the above forward-backward extension Lemma \ref{extendreverThm}, we can extend the weak solution of the equation \eqref{L2existeq} to the whole space. 
\begin{Thm}\label{ThmExtend}
	Denote $E,P$ as in \eqref{EP}. Let $s\in(0,1)$, $\de>0$ and $\vpi,N\ge 0$. Assume that the inflow boundary value $g$, the initial data $f_{T_1}$, and the time-dependent functions $\Psi=\mu^{\frac{1}{2}}+\psi\ge0$, $\vp$, $\phi$ satisfy
	\begin{align*}\begin{aligned}
			\|\<v\>^4\psi\|_{L^\infty_t([T_1,T_2])L^\infty_{x}(\Omega)L^\infty_v(\R^3_v)}\,dt & \le\de_0, \\
			\|[\vp,\phi]\|^2_{L^2_t([T_1,T_2])L^2_x(\Omega)L^2_v(\R^3_v)} & <\infty, \\
			\|g\|^2_{L^2_t([T_1,T_2])L^2_{x,v}(\Si_-)}+
			\|f_{T_1}\|_{L^2_x(\Omega)L^2_v(\R^3_v)}& <\infty,
		\end{aligned}
	\end{align*}
	with sufficiently small $\de_0>0$.
	Then the weak solution $f$ to equation \eqref{L2existeq} within $\ol\Omega$ obtained in Theorem \eqref{weakfThm} can be extended to the whole space. That is, $f$ can be extended to $\R^3_x$ and solves 
	\begin{align*}
		\left\{
		\begin{aligned}
			& \pa_tf+ v\cdot\na_xf = \vpi Vf+\Gamma(\Psi,f)+\Gamma(\vp,\mu^{\frac{1}{2}})\\&\qquad\qquad\quad+\phi-N\<v\>^{{l-2}}f &\text{ in } [T_1,T_2]\times\Omega\times\R^3_v, \\
			& \pa_tf+ v\cdot\na_xf +E\cdot\na_vf=P^2f& \text{ in } [T_1,T_2]\times D_{in},\\
			&\pa_tf+v\cdot\na_xf+E\cdot\na_vf=-P^2f& \text{ in } [T_1,T_2]\times D_{out},\\
			& f(t,x,v)|_{\Si_-}=g\quad \text{ on }[T_1,T_2]\times\Si_-, \\
			& f(T_1,x,v)=f_{T_1}\quad \text{ in }\Omega\times\R^3_v, \\
			&f(T_1,x,v)=0\qquad \text{ in }D_{out},\\
			&f(T_2,x,v)=0\qquad \text{ in }D_{in},
		\end{aligned}\right.
	\end{align*}
	in the sense that for any function $\Phi\in C^\infty_c(\R_t\times\R^3_x\times\R^3_v)$, 
	\begin{multline}\label{weakfwhole}
		(f(T_2),\Phi(T_2))_{L^2_x(\Omega)L^2_v}-(f_{T_1},\Phi(T_1))_{L^2_x(\Omega)L^2_v}
		+(f(T_2),\Phi(T_2))_{L^2_{x,v}(D_{out})}\\
		-(f(T_1),\Phi(T_1))_{L^2_{x,v}(D_{in})}
		-\int^{T_2}_{T_1}(f,(\pa_t+v\cdot\na_x)\Phi)_{L^2_x(\R^3)L^2_v}\,dt
		-\int^{T_2}_{T_1}(f,\na_v\cdot(E\Phi))_{L^2_{x}(\ol\Omega^c)L^2_v}\,dt\\
		=
		\int^{T_2}_{T_1}(f,\vpi V\Phi)_{L^2_x(\Omega)L^2_v}\,dt+\int^{T_2}_{T_1}(\Gamma(\Psi,f)+\Gamma(\vp,\mu^{\frac{1}{2}})+\phi-N\<v\>^{{l-2}}f,\Phi)_{L^2_x(\Omega)L^2_v}\,dt\\
		+\int^{T_2}_{T_1}(P^2f,\Phi)_{L^2_{x,v}(D_{in})}\,dt-\int^{T_2}_{T_1}(P^2f,\Phi)_{L^2_{x,v}(D_{out})}\,dt.
	\end{multline}
	Moreover, for any weak solution $f$ satisfying \eqref{weakfwhole}, we have $L^2$ estimate: for any $k\ge 0$, 
	\begin{multline}\label{weakfes}
		\pa_t\|\<v\>^kf(t)\|^2_{L^2_{x}(\Omega)L^2_v}
		+\ka\pa_t\|\<v\>^kf(t)\|^2_{L^2_{x,v}(\ol\Omega^c\times\R^3_v)}
		+\frac{1}{2}\|\<v\>^kf(t)\|_{L^2_{x,v}(\Si_+)}^2
		+\vpi\wh{C}^2_0\|\<v\>^{{k+4}}f\|^2_{L^2_{x}(\Omega)L^2_v}\\
		+\vpi\|\<v\>^{{k+2}}\na_vf\|^2_{L^2_{x}(\Omega)L^2_v}
		+c_0\|\<v\>^kf\|_{L^2_{x}(\Omega)L^2_D}^2
		+N\|\<v\>^{k+l-2}f\|_{L^2_{x}(\Omega)L^2_v}^2\\
		\le 2\|\<v\>^kg(t)\|_{L^2_{x,v}(\Si_-)}^2+C\|\<v\>^kf(t)\|_{L^2_{x}(\Omega)L^2_v}^2+\|[\mu^{\frac{1}{10^4}}\vp,\<v\>^k\phi]\|_{L^2_{x}(\Omega)L^2_v}^2, 
	\end{multline} 
	where $\ka>0$ is a sufficiently small constant (depending on $k$). 	
\end{Thm}
\begin{proof}
We start with the weak solution $f$ to equation \eqref{L2existeq} within $\ol\Omega$ in the sense of \eqref{weakf}. With the boundary value $f|_{\Si_+\cup\Si_-}$, we apply Lemma \ref{extendreverThm} to obtain the solution $f$ to equation \eqref{omegac1} within $\ol\Omega^c$. 
Combining these two parts, we know that $f$ on $\R^3_x$ satisfies \eqref{weakfwhole}; that is, \eqref{weakfwhole} can be deduced from \eqref{weakf} and \eqref{hheq5}.
Moreover, combining \eqref{fLineares1} and \eqref{hheqes5}, we have the energy estimate \eqref{weakfes}. 
This completes the proof of Theorem \ref{ThmExtend}. 
\end{proof}

\section{\texorpdfstring{$L^p$}{Lp} collision estimate of level functions}
\label{Sec5}
In this Section, we will estimate the level functions of the extended equation in Theorem \ref{weakfThm}.
We fix $l\ge\ga+10$, $N\ge 1$, $\wh{C}_0>0$ (large as in Lemma \ref{LemRegu}), $\vpi\ge 0$ and $\de\in(0,1]$ with weight function $\<v\>^l_\de$ given by \eqref{weight}.

\subsection{\texorpdfstring{$L^1$}{L1} norm for level functions}
In this Subsection, we consider the estimate of level sets of the solution $f$ to the equation (the boundary conditions are not necessary in the Section)
\begin{align}\label{Bol1}
\left\{
\begin{aligned}
	& \pa_tf+v\cdot\na_xf= \left\{\begin{aligned} 
		&\vpi Vf+\Gamma(\Psi,f)+\Gamma(\vp,\mu^{\frac{1}{2}})\\&\qquad+N\<v\>^{{l-2}}\phi-\eta\<v\>^lf\quad \text{ in } [T_1,T_2]\times\Omega\times\R^3_v, \\
		& P^2f-E\cdot\na_vf \qquad \text{ in } [T_1,T_2]\times D_{in},\\
		& -P^2f-E\cdot\na_vf \qquad \text{ in } [T_1,T_2]\times D_{out},
	\end{aligned}\right. \\
	& f(T_1,x,v)=0\quad \text{ in }\Omega\times\R^3_v,\\
	&f(T_1,x,v)=0\qquad \text{ in }D_{out},\\
	&f(T_2,x,v)=0\qquad \text{ in }D_{in}.
\end{aligned}\right.
\end{align}
Write $x_+=\max\{x,0\}$. Then one has
\begin{align}\label{deri+}
\frac{d(x_+)^2}{dx} = 2x_+, \ \text{ and } \ \na_{t,x}|F_+|^2=2F_+\na_{t,x}F.
\end{align}
For convenience, we still denote level functions
\begin{align}\label{levelvl}
f^{(l)}_K := f-K\<v\>^{-l}_\de,\quad f^{(l)}_{K,+}=f^{(l)}_K\1_{f^{(l)}_K\ge 0},\quad f^{(l)}_{K,-}=f^{(l)}_K\1_{f^{(l)}_K<0}.
\end{align}
Multiplying \eqref{Bol1} by $\<v\>^lf^{(l)}_{K,+}$, and using \eqref{deri+}, we have
\begin{align}\label{Bollevel}
\frac{1}{2}\pa_t(f^{(l)}_{K,+})^2+ \frac{1}{2}v\cdot\na_x(f^{(l)}_{K,+})^2 = \G,
\end{align}
where $\G^{}$ is given by
\begin{align}\label{G}
\G=\left\{\begin{aligned}
	&\vpi f^{(l)}_{K,+}Vf+f^{(l)}_{K,+}\Big(\Gamma(\Psi,f)
	+\Gamma(\vp,\mu^{\frac{1}{2}})\Big)\\
	& \qquad+N\<v\>^{{l-2}}\phi f^{(l)}_{K,+}-\eta\<v\>^lf\,f^{(l)}_{K,+} \quad \text{ in } \Omega\times\R^3_v, \\
	& (P^2f-E\cdot\na_vf)f^{(l)}_{K,+}\qquad \text{ in } D_{in}, \\
	& (-P^2f-E\cdot\na_vf)f^{(l)}_{K,+} \quad \text{ in } D_{out}.
\end{aligned}\right.
\end{align}
Similar to the arguments in \cite[Lemma 3.5 and Proposition 3.7]{Alonso2022}, we have the $L^1$ estimate of $\G$.
\begin{Lem}\label{Qes1Lem}
Let $j\ge 0$, $T_2>T_1\ge0$, $K\ge 0$ and $N\ge 0$. Let $\wh{C}_0>0$ be large enough (it depends on $l,\de,\|n\|_{L^\infty}$). Denote weight function $\<v\>^{l}_\de$ as in \eqref{weight} with $\de\in(0,1]$. Assume
\begin{align*}
	-\frac{3}{2}<\ga\le 2,\quad \ka>2,\quad l\ge\ga+10.
\end{align*}
Suppose $\Psi=\mu^{\frac{1}{2}}+\psi\ge 0$, and $f$ is the solution of \eqref{Bol1}. Denote $\G^{}$ by \eqref{G}.
Then
\begin{align}\label{Qesz}\notag
	&\int^{T_2}_{T_1}\int_{\R^3}\big|\<v\>^j(1-\De_v)^{-\frac{\ka}{2}}\G^{}\big|\,dxdt
	\le
	C\|\<v\>^{\frac{j}{2}}f^{(l)}_{K,+}(T_1)\|_{L^2_{x,v}(\R^6)}^2\\
	&\quad\notag+C\|[\<v\>^{l}\Psi,\<v\>^{j+\ga+6}\Psi]\|_{L^\infty_tL^\infty_x(\Omega)L^\infty_v}\|\<v\>^{\frac{j+(\ga+2s)_+}{2}}f^{(l)}_{K,+}\|_{L^2_tL^2_x(\Omega)L^2_{v}}^2\\
	&\quad\notag+C\min\{\|\mu^{\frac{1}{80}}\vp\|_{L^\infty_tL^\infty_{x}(\Omega)L^\infty_v}\|\mu^{\frac{1}{80}}f^{(l)}_{K,+}\|_{L^1_tL^1_{x}(\Omega)L^1_v},
	\|\mu^{\frac{1}{80}}\vp\|_{L^2_tL^2_{x}(\Omega)L^2_v}\|\mu^{\frac{1}{80}}f^{(l)}_{K,+}\|_{L^2_tL^2_{x}(\Omega)L^2_v}\}\\
	&\quad\notag
	+KC\|\<v\>^l\Psi\|_{L^\infty_tL^\infty_x(\Omega)L^\infty_v}\|\<v\>^{j-2}f^{(l)}_{K,+}\|_{L^1_tL^1_{x}(\Omega)L^1_v}\\
	&\quad\notag+N\|\<v\>^l\phi\|_{L^\infty_tL^\infty_x(\Omega)L^\infty_v}\|\<v\>^{j-2}f^{(l)}_{K,+}\|_{L^1_tL^1_x(\Omega)L^1_v}
	+C\|\<v\>^{\frac{j}{2}}Pf^{(l)}_{K,+}\|_{L^2_tL^2_{x}(\ol\Omega^c)L^2_{v}}^2\\
	&\quad+KC\|\<v\>^{-l+j}P^2f^{(l)}_{K,+}\|_{L^1_tL^1_{x}(\ol\Omega^c)L^1_{v}}
	+{\vpi}C\|[\<v\>^{{\frac{j}{2}+3}}\na_vf^{(l)}_{K,+},\<v\>^{{\frac{j}{2}+1}}f^{(l)}_{K,+}]\|_{L^2_tL^2_x(\Omega)L^2_v}^2. 
\end{align}
where the underlying time interval is $[T_1,T_2]$ and the constant $C=C(l,s,\ga,\de,\|n\|_{L^\infty})$. Note:
\begin{itemize}[leftmargin=2em]
	\item Since the estimates in the proof only concern the upper bound of $b_\eta(\cos\th)\le b(\cos\th)$, the same estimate is valid when $\Ga$ in \eqref{G} is replaced by $\Ga_\eta$ (defined by \eqref{Gaeta}). 
	\item Since $-f$ satisfies a similar equation, $-f$ satisfies the same bound if $f^{(l)}_{K,+}$ is replaced by $(-f)^{(l)}_{K,+}$. 
\end{itemize}
\end{Lem}
\begin{proof}
Fix $v\in\R^3$, integrate \eqref{Bollevel} over $[T_1,T_2]\times\R^3_x$ to obtain
\begin{align}\label{319}
	\frac{1}{2}\int_{\R^3}(f^{(l)}_{K,+}(T_2))^2\,dx=\frac{1}{2}\int_{\R^3}(f^{(l)}_{K,+}(T_1))^2\,dx+\int^{T_2}_{T_1}\int_{\R^3}\G^{}\,dxdt.
\end{align}
Applying $\<v\>^j(1-\De_v)^{-\frac{\ka}{2}}$ to \eqref{319} and using the positivity of Bessel potential \eqref{BesselReal}, we have
\begin{multline*}
	0\le\frac{1}{2}\int_{\R^3}\<v\>^j(1-\De_v)^{-\frac{\ka}{2}}(f^{(l)}_{K,+}(T_2))^2\,dx=\frac{1}{2}\int_{\R^3}\<v\>^j(1-\De_v)^{-\frac{\ka}{2}}(f^{(l)}_{K,+}(T_1))^2\,dx\\
	+\int^{T_2}_{T_1}\int_{\R^3}\<v\>^j(1-\De_v)^{-\frac{\ka}{2}}\G^{}\,dxdt.
\end{multline*}
This implies
\begin{multline}\label{Qes1}
	\int^{T_2}_{T_1}\int_{\R^3}\big|\<v\>^j(1-\De_v)^{-\frac{\ka}{2}}\G^{}\big|\,dxdt
	\le\frac{1}{2}\int_{\R^3}\<v\>^j(1-\De_v)^{-\frac{\ka}{2}}(f^{(l)}_{K,+}(T_1))^2\,dx\\
	+2\int^{T_2}_{T_1}\int_{\R^3}\big(\<v\>^j(1-\De_v)^{-\frac{\ka}{2}}\G^{}\big)_+\,dxdt,
\end{multline}
where $(\cdot)_+$ is the positive part of the term.
For the first right-hand term of \eqref{Qes1}, since $\ka>2$, we have from \eqref{GsL1} and Young's convolution inequality that
\begin{align}\label{Qes1a}\notag
	&\int_{\R^3}\int_{\R^3}\<v\>^j(1-\De_v)^{-\frac{\ka}{2}}(f^{(l)}_{K,+}(T_1))^2\,dvdx \\
	&\quad \notag\le\int_{\R^3}\int_{\R^3}\int_{\R^3}\<u\>^j\<v-u\>^jG_\ka(u)(f^{(l)}_{K,+}(T_1))^2(v-u)\,dudvdx \\
	&\quad\notag\le C\|\<u\>^jG_\ka(u)\|_{L^1_u}\|\<v\>^{\frac{j}{2}}f^{(l)}_{K,+}(T_1)\|_{L^2_{x,v}(\R^6)}^2 \\
	&\quad\le C\|\<v\>^{\frac{j}{2}}f^{(l)}_{K,+}(T_1)\|_{L^2_{x,v}(\R^6)}^2.
\end{align}
For the second right-hand term of \eqref{Qes1}, denote
\begin{align*}
	A_K = \big\{(t,x,v)\in[T_1,T_2]\times\R^3_x\times\R^3_v\,:\,(1-\De_v)^{-\frac{\ka}{2}}\G^{}\ge 0\big\},
\end{align*}
and
\begin{align*}
	W_K=(1-\De_v)^{-\frac{\ka}{2}}(\<v\>^j\1_{A_K}).
\end{align*}
Then
\begin{align*}
	\int^{T_2}_{T_1}\int_{\R^3}\int_{\R^3}\big(\<v\>^j(1-\De_v)^{-\frac{\ka}{2}}\G^{}\big)_+\,dxdvdt
	& =\int^{T_2}_{T_1}\int_{\R^3}\int_{\R^3}\<v\>^j\1_{A_K}(1-\De_v)^{-\frac{\ka}{2}}\G^{}\,dxdvdt \\
	& =\int^{T_2}_{T_1}\int_{\R^3}\int_{\R^3}W_K\G^{}\,dxdvdt.
\end{align*}
For any $\ka>2$, we have from \cite[Eq. (3.38) and (3.39)]{Alonso2022}
that
\begin{align}\label{WKes}
	W_K(v)\ge 0, \quad |W_K(v)|+|\na_vW_K(v)|+|\na_v^2W_K(v)|\le C\<v\>^j,
\end{align}
with $C>0$ independent of $K$. This can be derived by estimating via the properties of Bessel potential. 
Next, we estimate
\begin{multline}\label{341}
	\int_{\R^3}\int_{\R^3}W_K\G^{}\,dxdv
	= 
	\vpi\int_{\R^3}\int_{\Omega}W_Kf^{(l)}_{K,+}Vf\,dxdv-\eta\int_{\R^3}\int_{\Omega}W_Kf\,f^{(l)}_{K,+}\,dxdv\\
	+\int_{\R^3}\int_{\Omega}W_Kf^{(l)}_{K,+}\Big(\Gamma(\Psi,f)
	+\Gamma(\vp,\mu^{\frac{1}{2}})\Big)\,dxdv
	+\int_{\R^3}\int_{\Omega}W_KN\<v\>^{{l-2}}\phi f^{(l)}_{K,+}\,dxdv \\
	+\int_{D_{in}}W_K(P^2f-E\cdot\na_vf)f^{(l)}_{K,+}\,dxdv
	+\int_{D_{out}}W_K(-P^2f-E\cdot\na_vf)f^{(l)}_{K,+}\,dxdv. 
\end{multline}

{{\smallskip}} \noindent{\bf Step 1. The third and fourth terms in \eqref{341}.} In this Step 1, the underlying domain is $\Omega$.
For the second right-hand term of \eqref{341}, we have
\begin{align}\notag\label{T1T2T3}
	& \int_{\R^3}\int_{\Omega}W_Kf^{(l)}_{K,+}\Big(\Gamma(\Psi,f)
	+\Gamma(\vp,\mu^{\frac{1}{2}})\Big)\,dxdv 
	= \int_{\R^3}\int_{\Omega}W_Kf^{(l)}_{K,+}\Gamma\big(\Psi,f-K\<v\>^{-l}_\de\big)\,dxdv \\
	& \quad\notag\qquad+ \int_{\R^3}\int_{\Omega}W_Kf^{(l)}_{K,+}\Gamma(\vp,\mu^{\frac{1}{2}})\,dxdv 
	+ \int_{\R^3}\int_{\Omega}W_Kf^{(l)}_{K,+}\Gamma\big(\Psi,K\<v\>^{-l}_\de\big)\,dxdv \\
	& \quad=: \int_{\Omega}(T_1+T_2+T_3)\,dx.
\end{align}
For the term $T_1$, we apply \eqref{Ga} and pre-post change of variable $(v,v_*)\mapsto (v',v_*')$ to deduce
\begin{align*}
	T_1 & =\int_{\R^3}\int_{\R^3}\int_{\S^2}B(v-v_*,\sigma)
	\Big(W_K(v')f^{(l)}_{K,+}(v')\mu^{\frac{1}{2}}(v'_*)-W_K(v)f^{(l)}_{K,+}(v)\mu^{\frac{1}{2}}(v_*)\Big) \\
	& \qquad\times\Psi(v_*)\big(f(v)-K\<v\>^{-l}_\de\big)\,d\sigma dv_*dv.
\end{align*}
Using Cauchy-Schwarz inequality, positivity of $\Psi$, and noticing that $f^{(l)}_Kf^{(l)}_{K,+}=(f^{(l)}_{K,+})^2$,
we deduce that
\begin{align}\label{T1+}
	T_1 & \le\int_{\R^6\times\S^2}B
	\Big(W_K(v')f^{(l)}_{K,+}(v')\mu^{\frac{1}{2}}(v'_*)-W_K(v)f^{(l)}_{K,+}(v)\mu^{\frac{1}{2}}(v_*)\Big)f^{(l)}_{K,+}(v)\Psi(v_*)\,d\sigma dv_*dv.
\end{align}
Using Cauchy-Schwarz inequality, we write the parts of integrand involving $f^{(l)}_{K,+}$ as
\begin{align*}
	\Big( & W_K(v')f^{(l)}_{K,+}(v')\mu^{\frac{1}{2}}(v'_*)-W_K(v)f^{(l)}_{K,+}(v)\mu^{\frac{1}{2}}(v_*)\Big)f^{(l)}_{K,+}(v)\\
	& \le
	\frac{1}{2}W_K(v')(f^{(l)}_{K,+}(v'))^2\mu^{\frac{1}{2}}(v'_*)
	+\frac{1}{2}W_K(v')(f^{(l)}_{K,+}(v))^2\mu^{\frac{1}{2}}(v'_*) \\
	& \quad-W_K(v)(f^{(l)}_{K,+}(v))^2\mu^{\frac{1}{2}}(v_*) \\
	& \le\frac{1}{2}W_K(v')(f^{(l)}_{K,+}(v'))^2(\mu^{\frac{1}{2}}(v'_*)-\mu^{\frac{1}{2}}(v_*)) \\
	& \quad+\frac{1}{2}\big(W_K(v')(f^{(l)}_{K,+}(v'))^2-W_K(v)(f^{(l)}_{K,+}(v))^2\big)\mu^{\frac{1}{2}}(v_*) \\
	& \quad+\frac{1}{2}(W_K(v')-W_K(v))(f^{(l)}_{K,+}(v))^2(\mu^{\frac{1}{2}}(v'_*)-\mu^{\frac{1}{2}}(v_*)) \\
	& \quad+\frac{1}{2}W_K(v)(f^{(l)}_{K,+}(v))^2(\mu^{\frac{1}{2}}(v'_*)-\mu^{\frac{1}{2}}(v_*)) \\
	& \quad+\frac{1}{2}(W_K(v')-W_K(v))(f^{(l)}_{K,+}(v))^2\mu^{\frac{1}{2}}(v_*).
\end{align*}
Correspondingly, we denote $T_1$ in \eqref{T1+} as
	$T_1=\sum_{j=1}^5T_{1,j}$.
For the term $T_{1,1}$, by \eqref{Hpristar} and \eqref{WKes}, we have
\begin{align}\label{T11}
	\notag
	|T_{1,1}| & =\frac{1}{2}\Big|\int_{\R^6\times\S^2}BW_K(v')(f^{(l)}_{K,+}(v'))^2(\mu^{\frac{1}{2}}(v'_*)-\mu^{\frac{1}{2}}(v_*))\Psi(v_*)\,d\sigma dv_*dv\Big| \\
	& \notag\le
	\|\<v\>^{2+(\ga+2s)_+}\Psi\|_{L^2_v}\|\<v\>^{(\ga+2s)_+}W_K(v)(f^{(l)}_{K,+}(v))^2\|_{L^1_{v}} \\
	& \le C\|\<v\>^{\ga+6}\Psi\|_{L^\infty_v}\|\<v\>^{\frac{j+(\ga+2s)_+}{2}}f^{(l)}_{K,+}\|^2_{L^2_v}.
\end{align}
For the term $T_{1,2}$, we apply regular change of variable \eqref{regular} and \eqref{Gammaes1} 
to deduce
\begin{multline}\label{T12}
	|T_{1,2}|=\frac{1}{2}\Big|\int_{\R^6\times\S^2}B
	(W_K(v')(f^{(l)}_{K,+}(v'))^2-W_K(v)(f^{(l)}_{K,+}(v))^2)\mu^{\frac{1}{2}}(v_*)\Psi(v_*)\,d\sigma dv_*dv\Big|\\
	=\frac{1}{2}\Big|\int_{\R^6\times\S^2}|v-v_*|^\ga b(\cos\th)
	\frac{1-\cos^{3+\ga}\frac{\th}{2}}{\cos^{3+\ga}\frac{\th}{2}}W_K(v)(f^{(l)}_{K,+}(v))^2\mu^{\frac{1}{2}}(v_*)\Psi(v_*)\,d\sigma dv_*dv\Big|\\
	\le C\|\mu^{\frac{1}{8}}\Psi\|_{L^\infty_v}\|\<v\>^{\frac{j+\ga}{2}}f^{(l)}_{K,+}\|_{L^2_{v}}^2,
\end{multline}
where we used \eqref{WKes} and \eqref{bcossink2}.
For the term $T_{1,3}$, by \eqref{WKes}, we have
\begin{align*}
	\big|W_K(v')-W_K(v)\big|
	& \le|v'-v|\int^1_0|\na_vW_K(v+t(v'-v))|\,dt \\
	& \le C|v-v_*|\sin\frac{\th}{2}(\<v\>^{j}+\<v_*\>^{j}),
\end{align*}
and from \eqref{477} that
	$|\mu^{\frac{1}{2}}(v'_*)-\mu^{\frac{1}{2}}(v_*)|
	\le C|v-v_*|\sin\frac{\th}{2}$.
Then we apply \eqref{bcossin} to deduce
\begin{align}
	\label{T13}\notag
	|T_{1,3}| & =\frac{1}{2}\Big|\int_{\R^6\times\S^2}B(W_K(v')-W_K(v))(f^{(l)}_{K,+}(v))^2(\mu^{\frac{1}{2}}(v'_*)-\mu^{\frac{1}{2}}(v_*))\Psi(v_*)\,d\sigma dv_*dv\Big| \\
	& \le\notag C_l\Big|\int_{\R^6\times\S^2}B\min\big\{\sin^2\frac{\th}{2}|v-v_*|^2,1\big\}(\<v\>^{j}+\<v_*\>^{j})(f^{(l)}_{K,+}(v))^2\Psi(v_*)\,d\sigma dv_*dv\Big| \\
	& \le\notag C_l\|\<v\>^{j+(\ga+2s)_+}\Psi\|_{L^1_v}\|\<v\>^{\frac{j+(\ga+2s)_+}{2}}f^{(l)}_{K,+}\|_{L^2_{v}}^2 \\
	& \le C_l\|\<v\>^{j+\ga+6}\Psi\|_{L^\infty_v}\|\<v\>^{\frac{j+(\ga+2s)_+}{2}}f^{(l)}_{K,+}\|_{L^2_{v}}^2,
\end{align}
where we used \eqref{gafg}.
For the term $T_{1,4}$, we apply \eqref{Hpristar1} and estimate \eqref{WKes} of $W_K$ to deduce
\begin{align}\label{T14}
	|T_{1,4}|&\notag=\frac{1}{2}\Big|\int_{\R^6\times\S^2}B(v-v_*,\sigma)
	(f^{(l)}_{K,+}(v))^2W_K(v)(\mu^{\frac{1}{2}}(v'_*)-\mu^{\frac{1}{2}}(v_*))\Psi(v_*)\,d\sigma dv_*dv\Big|\\
	&\notag\le C\|\<v\>^{2+(\ga+2s)_+}\Psi\|_{L^2_v}\|\<v\>^{(\ga+2s)_+}W_K(f^{(l)}_{K,+})^2\|_{L^1_v}\\
	&\le C\|\<v\>^{\ga+6}\Psi\|_{L^\infty_v}\|\<v\>^{\frac{j+(\ga+2s)_+}{2}}f^{(l)}_{K,+}\|_{L^2_v}^2.
\end{align}
{{\smallskip}}
For the term $T_{1,5}$, applying \eqref{HHbcos}, \eqref{WKes}, and estimate \eqref{gafg}, we have
\begin{align}\notag
	\label{T15}
	|T_{1,5}| & =\frac{1}{2}\Big|\int_{\R^6\times\S^2}B(W_K(v')-W_K(v))(f^{(l)}_{K,+}(v))^2\mu^{\frac{1}{2}}(v_*)\Psi(v_*)\,d\sigma dv_*dv\Big| \\
	& \le\notag C\Big|\int_{\R^6}\big(|v_*-v|^{\ga+2s-1}\1_{|v-v_*|\ge\frac{2}{\pi}}+|v-v_*|^{\ga+1}\1_{|v-v_*|<\frac{2}{\pi}}+|v-v_*|^{\ga+2s}\big) \\
	& \qquad\qquad\notag\times(\<v\>^j+\<v_*\>^j)(f^{(l)}_{K,+}(v))^2\mu^{\frac{1}{2}}(v_*)\Psi(v_*)\,d\sigma dv_*dv\Big| \\
	& \le C\|\mu^{\frac{1}{4}}\Psi\|_{L^\infty_v}\|\<v\>^{\frac{j+(\ga+2s)_+}{2}}f^{(l)}_{K,+}\|_{L^2_{v}}^2.
\end{align}
Therefore, combining estimates \eqref{T11}, \eqref{T12}, \eqref{T13}, \eqref{T14}, and \eqref{T15}, we obtain the estimate of $T_1$ given in \eqref{T1T2T3}:
\begin{align}\label{T1}
	|T_1|\le
	C\|\<v\>^{j+\ga+6}\Psi\|_{L^\infty_v}\|\<v\>^{\frac{j+(\ga+2s)_+}{2}}f^{(l)}_{K,+}\|_{L^2_{v}}^2. 
\end{align}
For the term $T_2$ in \eqref{T1T2T3}, it's direct from \eqref{vpmu12} that
\begin{align}
	\label{T3}
	|T_2| & \le C\min\{\|\mu^{\frac{1}{80}}\vp\|_{L^\infty_v}\|\mu^{\frac{1}{80}}f^{(l)}_{K,+}\|_{L^1_v},\|\mu^{\frac{1}{80}}\vp\|_{L^2_v}\|\mu^{\frac{1}{80}}f^{(l)}_{K,+}\|_{L^2_v}\}.
\end{align}
The term $T_3$ in \eqref{T1T2T3} can be estimated by \eqref{vpl12}:
\begin{align}\label{T2}
	T_3\le C\|\<v\>^l\Psi\|_{L^\infty_v}\|\<v\>^{-2}W_Kf^{(l)}_{K,+}\|_{L^1_v}
	\le\|\<v\>^l\Psi\|_{L^\infty_v}\|\<v\>^{j-2}f^{(l)}_{K,+}\|_{L^1_v}.
\end{align}

\smallskip
Plugging \eqref{T1}, \eqref{T3} and \eqref{T2} into \eqref{T1T2T3}, we obtain
\begin{align}
	\label{O1}
	&\notag\Big|\int_{\R^3}\int_{\Omega}W_Kf^{(l)}_{K,+}\Big(\Gamma(\Psi,f)
	+\Gamma(\vp,\mu^{\frac{1}{2}})\Big)\,dxdv\Big|\\
	&\notag\quad\le
	C\|[\<v\>^{l}\Psi,\<v\>^{j+\ga+6}\Psi]\|_{L^\infty_x(\Omega)L^\infty_v}\|\<v\>^{\frac{j+(\ga+2s)_+}{2}}f^{(l)}_{K,+}\|_{L^2_x(\Omega)L^2_{v}}^2\\
	&\quad+C\|\mu^{\frac{1}{80}}\vp\|_{L^\infty_{x}(\Omega)L^\infty_v}\|\mu^{\frac{1}{80}}f^{(l)}_{K,+}\|_{L^1_{x}(\Omega)L^1_v}
	+KC\|\<v\>^l\Psi\|_{L^\infty_x(\Omega)L^\infty_v}\|\<v\>^{j-2}f^{(l)}_{K,+}\|_{L^1_{x}(\Omega)L^1_v}.
\end{align}
For the third right-hand term of \eqref{341}, by \eqref{WKes}, it's direct to calculate:
\begin{align}\label{O1a}
	\int_{\R^3}\int_{\Omega}W_KN\<v\>^{{l-2}}\phi f^{(l)}_{K,+}\,dxdv
	\le N\|\<v\>^l\phi\|_{L^\infty_x(\Omega)L^\infty_v}\|\<v\>^{j-2}f^{(l)}_{K,+}\|_{L^1_x(\Omega)L^1_v}. 
\end{align}

{{\smallskip}} \noindent{\bf Step 2. The fifth and sixth terms in \eqref{341}.} 
For the last two terms in \eqref{341}, recalling \eqref{levelvl}, we write
\begin{multline*}
	\int_{D_{in}}W_K(P^2f-E\cdot\na_vf)f^{(l)}_{K,+}\,dxdv
	+\int_{D_{out}}W_K(-P^2f-E\cdot\na_vf)f^{(l)}_{K,+}\,dxdv\\
	=
	\int_{D_{in}}W_K\Big(P^2\big(f-K\<v\>^{-l}_\de\big)+KP^2\<v\>^{-l}_\de\Big)f^{(l)}_{K,+}\,dxdv\\
	+\int_{D_{out}}W_K\Big(-P^2\big(f-K\<v\>^{-l}_\de\big)-KP^2\<v\>^{-l}_\de\Big)f^{(l)}_{K,+}\,dxdv\\
	-\int_{\ol\Omega^c\times\R^3_v}W_K\Big(E\cdot\na_v\big(f-K\<v\>^{-l}_\de\big)+E\cdot\na_v\big(K\<v\>^{-l}_\de\big)\Big)f^{(l)}_{K,+}\,dxdv
	=:T_1'+T_2'+T_3'.
\end{multline*}
For the terms $T_1'$ and $T_2'$, using \eqref{WKes}, \eqref{EP} and Lemma \ref{vldeLem}, we have
\begin{align*}
	T'_1&=\int_{D_{in}}W_K\Big(P^2(f^{(l)}_{K,+})^2+KP^2\<v\>^{-l}_\de f^{(l)}_{K,+}\Big)\,dxdv\\
	&\le C\|\<v\>^{\frac{j}{2}}Pf^{(l)}_{K,+}\|_{L^2_{x,v}(D_{in})}^2
	+C_\de K\|\<v\>^{-l+j}P^2f^{(l)}_{K,+}\|_{L^1_{x,v}(D_{in})},
\end{align*}
and 
\begin{align*}
	T'_2&=\int_{D_{out}}W_K\Big(-P^2(f^{(l)}_{K,+})^2-KP^2\<v\>^{-l}_\de f^{(l)}_{K,+}\Big)\,dxdv
	\le 0.
\end{align*}
Moreover, by integration by parts and Lemma \ref{vldeLem}, the term $T_3'$ can be estimated as 
\begin{align*}
	T_3'&=\int_{\ol\Omega^c\times\R^3_v}W_K
	\Big(-\frac{1}{2}\na_v\cdot(EW_K)(f^{(l)}_{K,+})^2
	-KE\cdot\na_v\<v\>^{-l}_\de\Big)f^{(l)}_{K,+}\,dxdv\\
	&\le C\|\<v\>^{\frac{j}{2}}Pf^{(l)}_{K,+}\|_{L^2_{x}(\ol\Omega^c)L^2_{v}}^2
	+C_\de K\|\<v\>^{-l-2+j}P^2f^{(l)}_{K,+}\|_{L^1_{x}(\ol\Omega^c)L^1_{v}};
\end{align*}
recall that $E,P$ are given in \eqref{EP}. 
Combining the above three estimates, we obtain 
\begin{multline}
	\label{O2}
	\int_{D_{in}}W_K(P^2f-E\cdot\na_vf)f^{(l)}_{K,+}\,dxdv
	+\int_{D_{out}}W_K(-P^2f-E\cdot\na_vf)f^{(l)}_{K,+}\,dxdv\\
	\le C\|\<v\>^{\frac{j}{2}}Pf^{(l)}_{K,+}\|_{L^2_{x}(\ol\Omega^c)L^2_{v}}^2
	+C_\de K\|\<v\>^{-l+j}P^2f^{(l)}_{K,+}\|_{L^1_{x}(\ol\Omega^c)L^1_{v}}.
\end{multline}

{{\smallskip}}\noindent{\bf Step 3. The first and second terms in \eqref{341}.}
For the first right-hand term of \eqref{341}, by \eqref{Vf}, we write 
\begin{multline}\label{T12pripri}
	\vpi\int_{\R^3}\int_{\Omega}W_Kf^{(l)}_{K,+}Vf\,dxdv
	=
	\vpi\int_{\R^3\times\Omega}W_Kf^{(l)}_{K,+}\Big(-2\wh{C}^2_0\<v\>^{{8}}f^{(l)}_{K,+}+2\na_v\cdot(\<v\>^{{4}}\na_v)f^{(l)}_{K,+}\Big)\,dxdv\\
	+K\vpi\int_{\R^3\times\Omega}W_Kf^{(l)}_{K,+}\Big(-2\wh{C}^2_0\<v\>^{{8}}\<v\>^{-l}_\de+2\na_v\cdot(\<v\>^{{4}}\na_v)\<v\>^{-l}_\de\Big)\,dxdv
	=:\vpi(T_1''+T_2''). 
\end{multline}
For the term $T''_1$, by integration by parts and \eqref{WKes}, we have 
\begin{align*}
	T_1''
	&=-2\wh{C}^2_0\int_{\R^3\times\Omega}W_K\<v\>^{{8}}(f^{(l)}_{K,+})^2\,dxdv
	-2\int_{\R^3\times\Omega}\na_v(W_Kf^{(l)}_{K,+})\cdot(\<v\>^{{4}}\na_v)f^{(l)}_{K,+}\,dxdv\\
	&\le -2\wh{C}^2_0\|W_K^{\frac{1}{2}}\<v\>^{{4}}f^{(l)}_{K,+}\|_{L^2_x(\Omega)L^2_v}^2-2\|W_K^{\frac{1}{2}}\<v\>^{{2}}\na_vf^{(l)}_{K,+}\|_{L^2_x(\Omega)L^2_v}^2\\
	&\quad
	+C\|[\<v\>^{{\frac{j}{2}+3}}\na_vf^{(l)}_{K,+},\<v\>^{{\frac{j}{2}+1}}f^{(l)}_{K,+}]\|_{L^2_x(\Omega)L^2_v}^2.
\end{align*}
The negative terms are extra and can be dropped. 
For the term $T_{2}''$, by choosing $\wh{C}_0>0$ sufficiently large (it depends on $l,\de,\|n\|_{L^\infty}$), we have 
\begin{align*}
	T_{2}''\le -K\wh{C}_0^2\|W_K\<v\>^{{-l+8}}f^{(l)}_{K,+}\|_{L^1_x(\Omega)L^1_v}. 
\end{align*}
Substituting the above estimates into \eqref{T12pripri}, we obtain 
\begin{align}\label{T12pripries}
	\vpi\int_{\R^3}\int_{\Omega}W_Kf^{(l)}_{K,+}Vf\,dxdv
	\le
	C\vpi\|[\<v\>^{{\frac{j}{2}+3}}\na_vf^{(l)}_{K,+},\<v\>^{{\frac{j}{2}+1}}f^{(l)}_{K,+}]\|_{L^2_x(\Omega)L^2_v}^2.
\end{align}
For the second right-hand term in \eqref{341}, since $W_K\ge0$, we have 
\begin{align}\label{330aa}
	-\eta\int_{\R^3}\int_{\Omega}W_Kf\,f^{(l)}_{K,+}\,dxdv
	=-\eta\|W_K^{\frac{1}{2}}f^{(l)}_{K,+}\|_{L^2_x(\Omega)L^2_v}^2
	-\eta K\|\<v\>^{-l}W_Kf^{(l)}_{K,+}\|_{L^1_x(\Omega)L^1_v}\le 0. 
\end{align}

{{\smallskip}}\noindent{\bf Step 4. Combining estimates for \eqref{341}.}
Substituting estimates \eqref{O1}, \eqref{O1a}, \eqref{O2}, \eqref{T12pripries} and \eqref{330aa} into \eqref{341}, and then plugging the resultant estimate and \eqref{Qes1a} into \eqref{Qes1}, we deduce
\eqref{Qesz} and concludes Lemma \ref{Qes1Lem}.
\end{proof}

\subsection{\texorpdfstring{$L^2$}{L2} estimate for level functions}
Next, we derive the $L^2$ estimate of the collision term for $f^{(l)}_{K,+}$, whose proof is analogous to Lemma \ref{Qes1Lem}.
\begin{Lem}\label{CollLevelLem}
Assume the same conditions as in Lemma \ref{Qes1Lem}. 
Then we have
\begin{align}\label{GOmegaL2}\notag
	&\text{(a)}\ \int_{\Omega\times\R^3}f^{(l)}_{K,+}\Big(\Gamma\big(\Psi,f-K\<v\>^{-l}_\de\big)+\Gamma(\varphi,\mu^{\frac{1}{2}})\Big)\,dvdx\\
	&\qquad\qquad\quad\notag
	\le
	(-c_0+C\|\<v\>^4\psi\|_{L^\infty_x(\Omega)L^\infty_v})\|f^{(l)}_{K,+}\|_{L^2_x(\Omega)L^2_D}^2\\
	&\qquad\qquad\qquad\ +C\|\1_{|v|\le R_0}f^{(l)}_{K,+}\|_{L^2_x(\Omega)L^2_v}^2
	+C\|\mu^{\frac{1}{80}}\vp\|_{L^\infty_{x}(\Omega)L^\infty_v}\|\mu^{\frac{1}{80}}f^{(l)}_{K,+}\|_{L^1_{x}(\Omega)L^1_v},\\
	\label{GOmegaL2a}
	&\text{(b)}\ \int_{\Omega\times\R^3}\phi\Gamma\big(\Psi,K\<v\>^{-l}_\de\big)\,dvdx
	\le CK\|\<v\>^l\Psi\|_{L^\infty_x(\Omega)L^\infty_v}\|\<v\>^{-2}\phi\|_{L^1_{x}(\Omega)L^1_v},\\
	\label{GOmegaN}
	&\text{(c)}\ -\int_{\Omega\times\R^3}N\<v\>^{{l-2}}f^{(l)}_{K,+}f\,dvdx\le-N\|\<v\>^{\frac{l}{2}-1}f^{(l)}_{K,+}\|_{L^2_x(\Omega)L^2_v}^2-NK\|\<v\>^{{-2}}f^{(l)}_{K,+}\|_{L^1_x(\Omega)L^1_v},\\
	\label{O1r}
	&\text{(d)}\ \Big|\int_{\R^3}\int_{\Omega}N\<v\>^{{l-2}}\phi f^{(l)}_{K,+}\,dxdv\Big|
	\le N\|\<v\>^l\phi\|_{L^\infty_x(\Omega)L^\infty_v}\|\<v\>^{{-2}}f^{(l)}_{K,+}\|_{L^1_x(\Omega)L^1_v}, \\
	\label{651a}\notag
	&\text{(e)}\ \vpi\int_{\Omega\times\R^3}f^{(l)}_{K,+}Vf\,dxdv
	\le-2\vpi\|[\wh{C}^{}_0\<v\>^{{4}}f^{(l)}_{K,+},\<v\>^{{2}}\na_vf^{(l)}_{K,+}]\|^2_{L^2_{x}(\Omega)L^2_v}\\&\qquad\qquad\qquad\qquad\qquad\qquad-\vpi K\wh{C}^2_0\|\<v\>^{{-l+8}}f^{(l)}_{K,+}\|_{L^1_x(\Omega)L^1_v},\\
	\label{651in}
	&\text{(f)}\ \int_{D_{in}}f^{(l)}_{K,+}\big(P^2f-E\cdot\na_vf\big)\,dvdx
	\ge\frac{1}{2}\|Pf^{(l)}_{K,+}\|^2_{L^2_{x,v}(D_{in})}+\frac{K}{2}\|\<v\>^{-l}_\de P^2f^{(l)}_{K,+}\|_{L^1_{x,v}(D_{in})},\\ 
	\label{651out}
	&\text{(g)}\ \int_{D_{out}}f^{(l)}_{K,+}\big(P^2f+E\cdot\na_vf\big)\,dvdx
	\ge \frac{1}{2}\|Pf^{(l)}_{K,+}\|^2_{L^2_{x,v}(D_{out})}+\frac{K}{2}\|\<v\>^{-l}_\de P^2f^{(l)}_{K,+}\|_{L^1_{x,v}(D_{out})}, \\
	\label{651advec}
	&\text{(h)}\ |v\cdot\na_x\<v\>^{-l}_\de|\le C_{\de,\|n\|_{W^{1,\infty}}}\<v\>^{-l},\ \ \text{when }\de\in(0,1),  
\end{align}
with some constant $C=C(\ga,s,l,\de,\|n\|_{L^\infty})>0$. 
Note that, since $-f$ satisfies a similar equation, $-f$ satisfies the same bound if $f^{(l)}_{K,+}$ is replaced by $(-f)^{(l)}_{K,+}$. 

\end{Lem}
\begin{proof}
We will estimate \eqref{GOmegaL2}--\eqref{651out} one by one.

{{\smallskip}}\noindent{\bf Estimation of \eqref{GOmegaL2}.}
We write the left-hand side of \eqref{GOmegaL2} as
\begin{align}\label{T123c}
	\int_{\R^3}\int_{\Omega}f^{(l)}_{K,+}\Gamma\big(\Psi,f-K\<v\>^{-l}_\de\big)\,dxdv
	+ \int_{\R^3}\int_{\Omega}f^{(l)}_{K,+}\Gamma(\vp,\mu^{\frac{1}{2}})\,dxdv
	& =: T_1+T_2.
\end{align}
For the term $T_1$, notice that $f^{(l)}_K(v)f^{(l)}_{K,+}(v)=(f^{(l)}_{K,+}(v))^2$. Then we apply \eqref{Ga} and pre-post change of variable $(v,v_*)\mapsto (v',v_*')$ to deduce
\begin{align*}\notag
	T_1 & =\int_{\Omega}\int_{\R^3}\int_{\R^3}\int_{\S^2}B\Psi(v_*)
	\Big(f^{(l)}_{K,+}(v')\mu^{\frac{1}{2}}(v'_*)-f^{(l)}_{K,+}(v)\mu^{\frac{1}{2}}(v_*)\Big)
	\Big(f(v)-K\<v\>^{-l}_\de\Big)\,d\sigma dv_*dvdx\\
	& \le\int_{\Omega\times\R^6\times\S^2}B\Psi(v_*)
	\Big(f^{(l)}_{K,+}(v')\mu^{\frac{1}{2}}(v'_*)-f^{(l)}_{K,+}(v)\mu^{\frac{1}{2}}(v_*)\Big)f^{(l)}_{K,+}(v)\,d\sigma dv_*dvdx\\
	&\notag=\int_{\Omega\times\R^6\times\S^2}B\mu^{\frac{1}{2}}(v_*)\Big(\Psi(v'_*)f^{(l)}_{K,+}(v')-\Psi(v_*)f^{(l)}_{K,+}(v)\Big)f^{(l)}_{K,+}(v)\,d\sigma dv_*dvdx\\
	&=\Big(\Gamma\big(\Psi,f^{(l)}_{K,+}\big),f^{(l)}_{K,+}\Big)_{L^2_vL^2_x(\Omega)}.
\end{align*}
Applying \eqref{L1} and \eqref{Gaesweight}, we obtain
\begin{align}\label{T1es}\notag
	T_1 & \le -c_0\|f^{(l)}_{K,+}\|_{L^2_x(\Omega)L^2_D}^2+C\|\1_{|v|\le R_0}f^{(l)}_{K,+}\|_{L^2_x(\Omega)L^2_v}^2+C\|\<v\>^4\psi\|_{L^\infty_xL^\infty_v}\|f^{(l)}_{K,+}\|_{L^2_x(\Omega)L^2_D}^2 \\
	& \le (-c_0+C\|\<v\>^4\psi\|_{L^\infty_xL^\infty_v})\|f^{(l)}_{K,+}\|_{L^2_x(\Omega)L^2_D}^2+C\|\1_{|v|\le R_0}f^{(l)}_{K,+}\|_{L^2_x(\Omega)L^2_v}^2.
\end{align}
For the term $T_2$ in \eqref{T123c}, we apply \eqref{vpmu12} to deduce 
\begin{align}\label{T3es}
	|T_2| & \le C\|\mu^{\frac{1}{80}}\vp\|_{L^\infty_{x}(\Omega)L^\infty_v}\|\mu^{\frac{1}{80}}f^{(l)}_{K,+}\|_{L^1_{x}(\Omega)L^1_v}. 
\end{align}
Substituting \eqref{T1es} and \eqref{T3es} into \eqref{T123c}, we obtain \eqref{GOmegaL2}.

{{\smallskip}}\noindent{\bf Estimation of \eqref{GOmegaL2a}.} This is a direct consequence of the estimate \eqref{vpl12}.

{{\smallskip}}\noindent{\bf Estimation of \eqref{GOmegaN} and \eqref{O1r}.}
The estimate \eqref{GOmegaN} follows from direct calculation:
\begin{align*}
	& -\int_{\Omega\times\R^3}N\<v\>^{{l-2}}f^{(l)}_{K,+}f\,dvdx \\
	& \quad\le -N\int_{\Omega\times\R^3}\<v\>^{{l-2}}f^{(l)}_{K,+}(f-K\<v\>^{-l}_\de)\,dvdx-NK\int_{\Omega\times\R^3}\<v\>^{-2}_\de f^{(l)}_{K,+}\,dvdx \\
	& \quad\le-N\|\<v\>^{{\frac{l}{2}-1}}f^{(l)}_{K,+}\|_{L^2_x(\Omega)L^2_v}^2-NK\|\<v\>^{-2}_\de f^{(l)}_{K,+}\|_{L^1_x(\Omega)L^1_v},
\end{align*}
This implies \eqref{GOmegaN}. 
The estimate \eqref{O1r} follows directly from H\"older's inequality.

{{\smallskip}}\noindent{\bf Estimation of \eqref{651a}.}
From the definition of $Vf$ \eqref{Vf}, we have 
\begin{multline*}
	\vpi\int_{\Omega\times\R^3}f^{(l)}_{K,+}Vf\,dxdv
	=-\vpi\int_{\Omega\times\R^3}f^{(l)}_{K,+}\Big(2\wh{C}^2_0\<v\>^{{8}}f^{(l)}_{K,+}+2\na_v\cdot(\<v\>^{{4}}\na_v)f^{(l)}_{K,+}\Big)\,dxdv\\
	-\vpi K\int_{\Omega\times\R^3}f^{(l)}_{K,+}\Big(2\wh{C}^2_0\<v\>^{{8}}\<v\>^{-l}_\de +2\na_v\cdot(\<v\>^{{4}}\na_v)\<v\>^{-l}_\de \Big)\,dxdv
	=T_{1}''+T_2''. 
\end{multline*}
For the term $T_1''$, by integration by parts,
\begin{align*}
	T_1''\le-2\vpi\|[\wh{C}^{}_0\<v\>^{{4}}f^{(l)}_{K,+},\<v\>^{{2}}\na_vf^{(l)}_{K,+}]\|^2_{L^2_{x}(\Omega)L^2_v},
\end{align*}
For the term $T_2''$, by choosing $\wh{C}_0>0$ sufficiently large (which depends on $l,\de,\|n\|_{L^\infty}$), we obtain 
\begin{align*}
	T_2''\le -\vpi K\wh{C}^2_0\|\<v\>^{{-l+8}}f^{(l)}_{K,+}\|_{L^1_x(\Omega)L^1_v}. 
\end{align*}
Collecting the above two estimates, we obtain 
\eqref{651a}. 

{{\smallskip}}\noindent{\bf Estimation of \eqref{651in} and \eqref{651out}.}
For \eqref{651in}, we have 
\begin{multline*}
	\int_{D_{in}}f^{(l)}_{K,+}\big(P^2f-E\cdot\na_vf\big)\,dxdv
	=\int_{D_{in}}f^{(l)}_{K,+}
	\Big(P^2\big(f-K\<v\>^{-l}_\de\big)+P^2K\<v\>^{-l}_\de\Big)\,dxdv\\
	+\int_{D_{in}}f^{(l)}_{K,+}
	\Big(-E\cdot\na_v\big(f-K\<v\>^{-l}_\de\big)
	-KE\cdot\na_v\<v\>^{-l}_\de\Big)\,dxdv
	=:T_1'+T_2'.
\end{multline*}
For the term $T_1'$, noticing $f^{(l)}_{K,+}\big({f}-K\<v\>^{-l}\big)=(f^{(l)}_{K,+})^2$, we have
\begin{align*}
	T_1'=\|Pf^{(l)}_{K,+}\|^2_{L^2_{x,v}(D_{in})}+K\|\<v\>^{-l}_\de P^2f^{(l)}_{K,+}\|_{L^1_{x,v}(D_{in})}.
\end{align*}
For the term $T'_2$, recalling that $E,P$ are given by \eqref{EP}, using integration by parts, Lemma \ref{vldeLem} and choosing $\wh C_l=C(l,\de,\|n\|_{L^\infty})>0$ (given in \eqref{EP}) sufficiently large (this gives large $P$), 
\begin{align*}
	T_2'&=\int_{D_{in}}\Big(
	\frac{1}{2}\na_v\cdot E(f^{(l)}_{K,+})^2-Kf^{(l)}_{K,+}E\cdot\na_v\<v\>^{-l}_\de\Big)\,dxdv\\
	&\le \frac{1}{2}\|Pf^{(l)}_{K,+}\|^2_{L^2_{x,v}(D_{in})}+\frac{K}{2}\|\<v\>^{-l-1}_\de P^2f^{(l)}_{K,+}\|_{L^1_{x,v}(D_{in})}, 
\end{align*}
which can be absorbed by the dissipation in $T_1'$. Therefore, we obtain \eqref{651in}: 
\begin{align*}
	\int_{D_{in}}f^{(l)}_{K,+}\big(P^2f-E\cdot\na_vf\big)\,dxdv\ge
	\frac{1}{2}\|Pf^{(l)}_{K,+}\|^2_{L^2_{x,v}(D_{in})}+\frac{K}{2}\|\<v\>^{-l}_\de P^2f^{(l)}_{K,+}\|_{L^1_{x,v}(D_{in})}. 
\end{align*}
The estimate of \eqref{651out} in the domain $D_{out}$ can be deduced similarly to the above calculations by noting the minus sign in front of $P^2$, and we omit the details for brevity. 
Lastly, estimate \eqref{651advec} follows directly from Lemma \ref{vldeLem}. 
We then conclude Lemma \ref{CollLevelLem}. 
\end{proof}

\subsection{\texorpdfstring{$L^2$}{L2} estimate with cutoff}
If the $L^2$ inner product involves a weight or cutoff function, we have the following estimate on the Boltzmann collision term. The proof is similar to Lemmas \ref{Qes1Lem} and \ref{CollLevelLem}. 

\begin{Lem}\label{GachideLem}
Let $\chi_\de\in W^{2,\infty}(\R^7_{t,x,v})$ be a positive function satisfying $\|\chi_\de\|_{W^{2,\infty}}\le C_\de$ with any $\de>0$, $\vpi\ge 0$, $\eta>0$, $\ve\in(0,1)$, $-\frac{3}{2}<\ga\le 2$, $s\in(0,1)$, and $\vpi,N,M\ge 0$. 
Assume $\Psi=\mu^{\frac{1}{2}}+\psi\ge 0$. Then 
\begin{align}\label{colli2level}
	&\notag\big(\vpi Vf+\Gamma(\Psi,f)+\Gamma(\vp,\mu^{\frac{1}{2}})+N\phi-\eta\<v\>^{l}f,\chi_\de f^{(l)}_{K,+}\big)_{L^2_v}\\
	&\quad\notag\le C\|[\chi_\de,\na_v\chi_\de,\na_v^2\chi_\de]\|_{L^\infty_v}\|\<v\>^{\ga+6}\Psi\|_{L^\infty_v}\|\<v\>^{\frac{(\ga+2s)_+}{2}}f^{(l)}_{K,+}\|_{L^2_v}^2\\
	&\quad\notag+C\|\chi_\de\|_{L^\infty_v}\,\min\Big\{\|\mu^{\frac{1}{76}}\vp\|_{L^2_v}\|\mu^{\frac{1}{76}}f^{(l)}_{K,+}\|_{L^1_v},\,
	\|\mu^{\frac{1}{10^4}}\vp\|_{L^2_v}\|\mu^{\frac{1}{10^4}}f^{(l)}_{K,+}\|_{L^2_v}\Big\}\\
	&\quad\notag+CK\|\chi_\de\|_{L^\infty_v}\|\<v\>^l\psi\|_{L^\infty_v}\|\<v\>^{-2}f^{(l)}_{K,+}\|_{L^1_v}\\
	&\quad\notag+N\|\chi_\de\|_{L^\infty_v}\min\Big\{\|\phi\|_{L^\infty_v}\|f^{(l)}_{K,+}\|_{L^1_v},\,\|\phi\|_{L^2_v}\|f^{(l)}_{K,+}\|_{L^2_v}\Big\}\\
	&\quad+\vpi\|[\na_v\chi_\de,\na_v^2\chi_\de]\|_{L^\infty_v}\|\<v\>^{{2}}f^{(l)}_{K,+}\|_{L^2_v}\|\<v\>^{{2}}[f^{(l)}_{K,+},\na_vf^{(l)}_{K,+}]\|_{L^2_v},
\end{align}
and 
	\begin{align}\label{colli2nolevel}
		&\notag\big(\vpi Vf+\Gamma(\Psi,f)+\Gamma(\vp,\mu^{\frac{1}{2}})+N\phi-{\eta\<v\>^l}^{}f,\chi_\de f\big)_{L^2_v}\\
		&\quad\notag\le C\|[\chi_\de,\na_v\chi_\de,\na_v^2\chi_\de]\|_{L^\infty_v}\|\<v\>^{\ga+6}\Psi\|_{L^\infty_v}\|\<v\>^{\frac{(\ga+2s)_+}{2}}f\|_{L^2_v}^2\\
		&\quad\notag+C\|\chi_\de\|_{L^\infty_v}\min\Big\{\|\mu^{\frac{1}{76}}\vp\|_{L^2_v}\|\mu^{\frac{1}{76}}f\|_{L^1_v},\,\|\mu^{\frac{1}{80}}\vp\|_{L^2_v}\|\mu^{\frac{1}{80}}f\|_{L^2_v}\Big\}\\
		&\quad+N\|\chi_\de\|_{L^\infty_v}\min\Big\{\|\phi\|_{L^\infty_v}\|f\|_{L^1_v},\,\|\phi\|_{L^2_v}\|f\|_{L^2_v}\Big\},
	\end{align}
	where we let $\wh{C}_0=\wh{C}_0(\de,l,\|n\|_{L^\infty})>0$, given in \eqref{Vf}, be sufficiently large. 
	Here, the constant $C=C(\ga,s)>0$ is independent of $\vpi,\ve,N,\eta$. 
	Moreover, if $s\in[\frac{1}{2},1)$, the same estimates holds for $\Ga_\eta$ defined by \eqref{Gaeta}, which replaces $\Ga$, uniformly in $\eta$.
\end{Lem} 
(Note that we used the notation $\eta$ twice, but the vanishing dissipation will be used in Section \ref{Sec10} and the ``cut-off" $\Ga_\eta$ will be only used in Section \ref{Sec11}.)
\begin{proof}
	For the estimate \eqref{colli2level}, we write
	\begin{multline*}
		\big(\vpi Vf,\chi_\de f^{(l)}_{K,+}\big)_{L^2_v}+\big(\Gamma(\Psi,f-K\<v\>^{-l}_\de)f^{(l)}_{K,+},\chi_\de f^{(l)}_{K,+}\big)_{L^2_v}+\big(\Gamma(\vp,\mu^{\frac{1}{2}}),\chi_\de f^{(l)}_{K,+}\big)_{L^2_v}\\
		+\big(\Gamma(\Psi,K\<v\>^{-l}_\de),\chi_\de f^{(l)}_{K,+}\big)_{L^2_v}=:S_1+S_2+S_3+S_4.
	\end{multline*}
	By definition of $Vf$ \eqref{Vf} and integration by parts, we have 
	\begin{align*}
	S_1
	&\notag=\vpi\big(-2\wh{C}^2_0\<v\>^{{8}}f+2\na_v\cdot(\<v\>^{{4}}\na_v)f\big),\chi_\de f^{(l)}_{K,+}\big)_{L^2_v}\\
	&\notag\le \vpi\int_{\R^3_v}\Big(-2\wh{C}^2_0\<v\>^{{8}}\chi_\de |f^{(l)}_{K,+}|^2
	+2\chi_\de f^{(l)}_{K,+}\na_v\cdot(\<v\>^{{4}}\na_v)f^{(l)}_{K,+}\\
	&\notag\qquad\quad-K\wh{C}^2_0\<v\>^{{8}}\<v\>^{-l}_\de\chi_\de f^{(l)}_{K,+}
	+2K\chi_\de f^{(l)}_{K,+}\na_v\cdot(\<v\>^{{4}}\na_v)\<v\>^{-l}_\de
	\Big)\,dv\\
	&\le\vpi\|[\na_v\chi_\de,\na_v^2\chi_\de]\|_{L^\infty_v}\|\<v\>^{{2}}f^{(l)}_{K,+}\|_{L^2_v}\|\<v\>^{{2}}[f^{(l)}_{K,+},\na_vf^{(l)}_{K,+}]\|_{L^2_v},
	\end{align*}
	where we used $|\na_v\cdot(\<v\>^{{4}}\na_v)\<v\>^{-l}_\de|	\le C_{\de,\|n\|_{L^\infty}}\<v\>^{{-l+8}}$ and choose $\wh{C}_0=\wh{C}_0(\de,\|n\|_{L^\infty})>0$ sufficiently large. 
	For $S_2$, we have from \eqref{Ga} that
	\begin{align}\label{S1aaa}\notag
		S_2 
		& =\int_{\R^6}\int_{\S^2}B\mu^{\frac{1}{2}}(v_*)\big(\Psi'_*(f-K\<v\>^{-l}_\de)'-\Psi_*(f-K\<v\>^{-l}_\de)\big)\chi_\de(v)f^{(l)}_{K,+}(v)\,d\sigma dv_*dv \\
		& \le\int_{\R^6}\int_{\S^2}B\Psi_*f^{(l)}_{K,+}(v)\big(\mu^{\frac{1}{2}}(v'_*)\chi_\de(v')f^{(l)}_{K,+}(v')-\mu^{\frac{1}{2}}(v_*)\chi_\de(v)f^{(l)}_{K,+}(v)\big)\,d\sigma dv_*dv,
	\end{align}
	where we used $(f-K\<v\>^{-l}_\de)'\le f^{(l)}_{K,+}(v')$ and $(f-K\<v\>^{-l}_\de)f^{(l)}_{K,+}=|f^{(l)}_{K,+}|^2$, and pre-post change of variable. 
	The expression \eqref{S1aaa} is similar to \eqref{T1+} while $\chi_\de$ satisfies the same control as $W_K$ in \eqref{WKes} with $j=0$. Then one can apply similar calculations \eqref{T1+}--\eqref{T1} with $W_K$ replaced by $f^{(l)}_{K,+}$ to deduce 
	\begin{align*}
		S_2\le
		C\|[\chi_\de,\na_v\chi_\de,\na^2_v\chi_\de]\|_{L^\infty_v}\|\<v\>^{\ga+6}\Psi\|_{L^\infty_v}\|\<v\>^{\frac{(\ga+2s)_+}{2}}f^{(l)}_{K,+}\|_{L^2_{v}}^2. 
	\end{align*}
The estimation of $S_3$ and $S_4$ are given driectly by \eqref{vpmu12}, \eqref{L2} and \eqref{vpl12}
\begin{align*}
	S_3&\le C\|\chi_\de\|_{L^\infty_v}\,\min\Big\{\|\mu^{\frac{1}{76}}\vp\|_{L^2_v}\|\mu^{\frac{1}{76}}f^{(l)}_{K,+}\|_{L^1_v},\,
		\|\mu^{\frac{1}{10^4}}\vp\|_{L^2_v}\|\mu^{\frac{1}{10^4}}f^{(l)}_{K,+}\|_{L^2_v}\Big\},\\
	S_4&\le \big(\Gamma(\Psi,K\<v\>^{-l}_\de),\chi_\de f^{(l)}_{K,+}\big)_{L^2_v}
	\le CK\|\chi_\de\|_{L^\infty_v}\|\<v\>^l\psi\|_{L^\infty_v}\|\<v\>^{-2}f^{(l)}_{K,+}\|_{L^1_v}. 
\end{align*}
The estimate of the term $\phi$ is straightforward by using H\"{o}lder's inequality.

\smallskip The non-level estimate \eqref{colli2nolevel} is similar, and we only consider the terms $V$ and $\Ga(\Psi,f)$. 
	By definition of $Vf$ \eqref{Vf} and integration by parts, we have 
	\begin{align*}
	\big(\vpi Vf,\chi_\de f\big)_{L^2_v}&\notag=\vpi\big(-2\wh{C}^2_0\<v\>^{{8}}f+2\na_v\cdot(\<v\>^{{4}}\na_v)f\big),\chi_\de f\big)_{L^2_v}\\
	&\notag\le \vpi\int_{\R^3_v}\Big(-2\wh{C}^2_0\<v\>^{{8}}\chi_\de |f|^2
	-2\<v\>^{{4}}\chi_\de|\na_vf|^2+\<v\>^{4}\De_v\chi_\de|f|^2\Big)\,dv
	\le0,
	\end{align*}
provided $\vpi=\vpi(\de)\ge0$ is sufficiently small. 
	The term $\Ga(\Psi,f)$ is 
	\begin{align*}
		\big(\Gamma(\Psi,f)f,\chi_\de f\big)_{L^2_v} 
		& =\int_{\R^6}\int_{\S^2}B\Big(\mu^{\frac{1}{2}}(v'_*)\chi_\de(v')f(v')-\mu^{\frac{1}{2}}(v_*)\chi_\de(v)f(v)\Big)\Psi(v_*)f(v)\,d\sigma dv_*dv, 
	\end{align*}
	which is still similar to \eqref{T1+} if we replace $W_K$ by $\chi_\de$ and $f^{(l)}_{K,+}$ by $f$.
	Then one can apply similar calculations \eqref{T1+}--\eqref{T1} to deduce
	\begin{align*}
		\big(\Gamma(\Psi,f)f,\chi_\de f\big)_{L^2_v}\le
		C\|[\chi_\de,\na_v\chi_\de,\na^2_v\chi_\de]\|_{L^\infty_v}\|\<v\>^{\ga+6}\Psi\|_{L^\infty_v}\|\<v\>^{\frac{(\ga+2s)_+}{2}}f\|_{L^2_{v}}^2. 
	\end{align*}
	 Combining these estimates, we obtain \eqref{colli2level} and \eqref{colli2nolevel}. Since the above calculations only concern the upper bound of the collision kernel, the same technique can be applied to $b_\eta(\cos\th)$, which replaces $b(\cos\th)$ and has the upper bound $b(\cos\th)$. Thus, the same estimates hold for $\Ga_\eta$. This completes the proof of Lemma \ref{GachideLem}.
\end{proof}

\subsection{Besov regularity for level functions}
In this Subsection, we derive the time-space Besov regularity for level functions by using the velocity averaging lemma \ref{averThma}. 
\begin{Lem}\label{LemBesovreguLevel}
	Assume the same conditions as in Lemma \ref{Qes1Lem}. Let $f$ be a solution to \eqref{Bol1}.
	Moreover, 
\begin{align}\label{562x}
	\begin{aligned}
		\|[\<v\>^l\psi,\vp]\|_{L^\infty_{t,x,v}([T_1,T_2]\times\Omega\times\R^3_v)}+\|\vp\|_{L^2_{t,x,v}([T_1,T_2]\times\Omega\times\R^3_v)}
		&\le \de_0,\\ 
		\|\<v\>^l\phi\|_{L^\infty_{t,x,v}([T_1,T_2]\times\Omega\times\R^3_v)}&\le K_1,\\
		\|f\|_{L^\infty_{t,x,v}([T_1,T_2]\times\R^3_x\times\R^3_v)}&=C_\infty,
	\end{aligned}
\end{align}
for some $\de_0\in(0,1)$ and $K_1,C_\infty>0$. 
	Then for any $K\ge 0$, $f^{(l)}_{K,+}$ solves equation \eqref{Bollevel} and there exists small $s'\in(0,1)$ such that 
	\begin{align}\label{736}
		&\notag\Big\|\int_{\R^3}\1_{[T_1,T_2]}\<v\>^{-10}(f^{(l)}_{K,+})^2\,dv\Big\|_{B^{s',2}_p(\R^{4}_{t,x})}^p
		\le C\Big(\|\<v\>^{-2}[f^{(l)}_{K,+}(T_1),f^{(l)}_{K,+}(T_2)]\|_{L^2_{x,v}(\R^6)}^{2p}\\
		&\notag\quad+C_\infty^{2p-2}\|\1_{[T_1,T_2]}\<v\>^{-2p}f^{(l)}_{K,+}\|_{L^2_{t,x,v}(\R^{7})}^{2}
		+\|\<v\>^{\frac{(\ga+2s)_+}{2}}f^{(l)}_{K,+}\|^{2p}_{L^2_tL^2_x(\Omega)L^2_v}\\
		&\notag\quad
		+\min\{\|\mu^{\frac{1}{80}}f^{(l)}_{K,+}\|^{p}_{L^1_tL^1_x(\Omega)L^1_v},\|\mu^{\frac{1}{80}}f^{(l)}_{K,+}\|^{p}_{L^2_tL^2_x(\Omega)L^2_v}\}\\
		&\notag\quad
		+({K+NK_1})^p\|\<v\>^{-2}f^{(l)}_{K,+}\|^{p}_{L^1_tL^1_x(\Omega)L^1_v}
		+\vpi^p\|[\<v\>^{{3}}f^{(l)}_{K,+},\<v\>\na_vf^{(l)}_{K,+}]\|_{L^2_tL^2_x(\Omega)L^2_v}^{2p}\\
		&\quad+\|Pf^{(l)}_{K,+}\|_{L^2_tL^2_{x}(\ol\Omega^c)L^2_{v}}^{2p}
		+K^p\|\<v\>^{-l}P^2f^{(l)}_{K,+}\|^p_{L^1_tL^1_{x}(\ol\Omega^c)L^1_{v}}\Big), 
	\end{align}
	for some $C=C(l,\ga,s,p,\de)>0$.
	Note that we add the exponent $p$ in \eqref{736} and the energy function $\E_p$ \eqref{EpIntro}.
\end{Lem}
\begin{proof}
To obtain Besov regularity, we will apply velocity averaging Lemma \ref{averThma} to equation \eqref{Bollevel} with the following parameters.
Fix any $\si>2$ and choose small $\ka\in(0,1/p)$ such that $\ka+\si\le 3$. Then we set 
\begin{align*}
	d=3,\quad m=3\ge \ka+\si,\quad n=4
\end{align*}
and $1<p<2$ is chosen to be close to $1$ such that
\begin{align}\label{49}
	1<p\le p^{\#}, \quad \ka p<1,\quad 1<p<\frac{p}{2-p},
	\quad \ka p^*\equiv\frac{\ka p}{p-1}>5,
\end{align}
where $p^\#$ is given by \eqref{ppsharp}.
Then \eqref{lam1} gives regularity index 
\begin{align}\label{spri}
	s':=\frac{n(1-\ka)}{(1+2n)(1+m)}\big(1-\frac{1}{p}\big)\in(0,1). 
\end{align}
By Sobolev embedding (e.g. \cite[Theorem 4.12]{Adams2003} with necessary embedding to an integer Sobolev space first), the last condition in \eqref{49} implies
\begin{align*}
	W^{\kappa,\,p^*}_{t,x}([T_1,T_2]\times\Omega)W^{\kappa,\,p^*}_{v}(\R^3) \text{ is embedded in } L^\infty_{t,x,v}([T_1,T_2]\times\Omega\times\R^3_v),
\end{align*}
where $W^{\kappa,\,p^*}$ is the fractional Sobolev space. 
Hence, by duality,
\begin{align}\label{embeddsigma}
	L^1_{t,x,v}([T_1,T_2]\times\Omega\times\R^3_v) \text{ is embedded in } H^{-\kappa,p}_{t,x}H^{-\kappa,p}_{v}([T_1,T_2]\times\Omega\times\R^3_v).
\end{align}
On the other hand, it follows from equation \eqref{Bollevel} that
\begin{align}\label{linear1bb}
	\frac{1}{2}\pa_t\big(\<v\>^{-2}f^{(l)}_{K,+}\big)^2+ \frac{1}{2}v\cdot\na_x\big(\<v\>^{-2}f^{(l)}_{K,+}\big)^2 = \<v\>^{-4}\G.
\end{align}
Then we apply Lemma \ref{averThma} (i.e. estimate \eqref{avereq1} with averaging function $\psi=\<v\>^{-6}$ satisfying $\|\<v\>^n\<D_v\>^{m+1}\psi\|_{L^2_v}<\infty$ therein) to the equation \eqref{linear1bb} to derive
\begin{align}\label{100}\notag
	&\Big\|\int_{\R^3}\1_{[T_1,T_2]}\<v\>^{-10}(f^{(l)}_{K,+})^2\,dv\Big\|_{B^{s',2}_p(\R^{4}_{t,x})}\\
	&\quad\notag\le C
	\Big(\|(I-\De_{x})^{-\kappa/2}(I-\De_v)^{-(\ka+\si)/2}(\<v\>^{-2}f^{(l)}_{K,+}(T_1))^2\|_{L^p(\R^{6}_{x,v})}\\
	&\qquad\notag+\|(I-\De_{x})^{-\kappa/2}(I-\De_v)^{-(\ka+\si)/2}(\<v\>^{-2}f^{(l)}_{K,+}(T_2))^2\|_{L^p(\R^{6}_{x,v})}\\
	&\qquad\notag+\|\1_{[T_1,T_2]}(\<v\>^{-2}f^{(l)}_{K,+})^2\|_{L^p(\R^{7}_{t,x,v})}\\
	&\qquad+\|(I-\De_{t,x})^{-\kappa/2}(I-\De_v)^{-(\ka+\si)/2}(\1_{[T_1,T_2]}\<v\>^{-4}\G)\|_{L^p(\R^{7}_{t,x,v})}\Big), 
\end{align}
In the following, we bound each term on the right side of \eqref{100}.
For the first two right-hand terms of \eqref{100}, by embedding \eqref{embeddsigma}, 
we have 
\begin{align}\label{651t}\notag
	&\|(1-\De_x)^{-\frac{\kappa}{2}}(1-\Delta_v)^{-\frac{\ka+\sigma}{2}}[(\<v\>^{-2}f^{(l)}_{K,+}(T_1))^2,(\<v\>^{-2}f^{(l)}_{K,+}(T_2))^2]\|_{L^p_{x,v}(\R^6)}\\
	&\quad\le C
	\|\<v\>^{-2}[f^{(l)}_{K,+}(T_1),f^{(l)}_{K,+}(T_2)]\|_{L^2_{x,v}(\R^6)}^{2}. 
\end{align}
For the third right-hand term of \eqref{100}, applying H\"{o}lder's inequality and $L^\infty$ bound of $f$ in \eqref{562x},
\begin{align}\label{651s}
	\|\1_{[T_1,T_2]}\<v\>^{-4}(f^{(l)}_{K,+})^2\|_{L^p(\R^{7}_{t,x,v})}
			&\le\notag\Big(\int_{\R^{7}_{t,x,v}}\1_{[T_1,T_2]}\<v\>^{-4p}(f^{(l)}_{K,+})^{2p}\,dtdxdv\Big)^{\frac{1}{p}}\\
	&\le C_\infty^{2-2/p}\|\1_{[T_1,T_2]}\<v\>^{-2p}f^{(l)}_{K,+}\|_{L^2_{t,x,v}(\R^{7})}^{2/p}.
\end{align}
For the fourth right-hand term of \eqref{100}, using embedding \eqref{embeddsigma} and Lemma \ref{Qes1Lem} with $j=0$, 
\begin{align}\label{719}\notag
	&\quad \|(1-\De_{t,x})^{-\frac{\kappa}{2}}(1-\De_v)^{-\frac{\ka+\sigma}{2}}(\1_{[T_1,T_2]}\<v\>^{-4}\G)\|_{L^p_{t,x,v}(\R^7)}
	\le\|\1_{[T_1,T_2]}(1-\De_v)^{-\frac{\ka}{2}}\G\|_{L^1_{t,x,v}(\R^7)} \\
	& \notag\le C\bigg\{
	\|f^{(l)}_{K,+}(T_1)\|_{L^2_{x,v}(\R^6)}^2
	+\|[\<v\>^{l}\Psi,\<v\>^{\ga+6}\Psi]\|_{L^\infty_tL^\infty_x(\Omega)L^\infty_v}\|\<v\>^{\frac{(\ga+2s)_+}{2}}f^{(l)}_{K,+}\|_{L^2_tL^2_x(\Omega)L^2_{v}}^2\\
	&\quad\notag+\min\{\|\mu^{\frac{1}{80}}\vp\|_{L^\infty_tL^\infty_{x}(\Omega)L^\infty_v}\|\mu^{\frac{1}{80}}f^{(l)}_{K,+}\|_{L^1_tL^1_{x}(\Omega)L^1_v},
	\|\mu^{\frac{1}{80}}\vp\|_{L^2_tL^2_{x}(\Omega)L^2_v}\|\mu^{\frac{1}{80}}f^{(l)}_{K,+}\|_{L^2_tL^2_{x}(\Omega)L^2_v}\}\\
	&\quad\notag
	+K\|\<v\>^l\Psi\|_{L^\infty_tL^\infty_x(\Omega)L^\infty_v}\|\<v\>^{-2}f^{(l)}_{K,+}\|_{L^1_tL^1_{x}(\Omega)L^1_v}\\
	&\quad\notag+N\|\<v\>^l\phi\|_{L^\infty_tL^\infty_x(\Omega)L^\infty_v}\|\<v\>^{-2}f^{(l)}_{K,+}\|_{L^1_tL^1_x(\Omega)L^1_v}
	+\|Pf^{(l)}_{K,+}\|_{L^2_tL^2_{x}(\ol\Omega^c)L^2_{v}}^2\\
	&\quad\notag+K\|\<v\>^{-l}P^2f^{(l)}_{K,+}\|_{L^1_tL^1_{x}(\ol\Omega^c)L^1_{v}}
	+{\vpi}\|[\<v\>^{{3}}\na_vf^{(l)}_{K,+},\<v\>f^{(l)}_{K,+}]\|_{L^2_tL^2_x(\Omega)L^2_v}^2\bigg\}\\
	&\notag\le C\bigg\{\|f^{(l)}_{K,+}(T_1)\|_{L^2_{x,v}(\R^6)}^2
	+\|\<v\>^{\frac{(\ga+2s)_+}{2}}f^{(l)}_{K,+}\|^2_{L^2_tL^2_x(\Omega)L^2_v}\\	
	&\notag\quad+(1+{K+NK_1})\|\<v\>^{-2}f^{(l)}_{K,+}\|_{L^1_tL^1_x(\Omega)L^1_v}+\vpi \|[\<v\>^{{3}}f^{(l)}_{K,+},\<v\>\na_vf^{(l)}_{K,+}]\|_{L^2_tL^2_x(\Omega)L^2_v}^2\\	
	&\quad+\|Pf^{(l)}_{K,+}\|_{L^2_tL^2_{x}(\ol\Omega^c)L^2_{v}}^2
	+K\|\<v\>^{-l}P^2f^{(l)}_{K,+}\|_{L^1_tL^1_{x}(\ol\Omega^c)L^1_{v}}\bigg\},
\end{align}
where the time norm is lying in $[T_1,T_2]$ if not specified, and we used $\ga\le 2$, $l\ge\ga+10$, and \eqref{562x}.
Here the constant $C=C(l,s,\ga)>0$ depends on $l$. 
Substituting estimates \eqref{651t}, \eqref{651s} and \eqref{719} into \eqref{100} yields \eqref{736}. 
This completes the proof of Lemma \ref{LemBesovreguLevel}. 
\end{proof}

\section{\texorpdfstring{$L^\infty$}{L infty} estimate for inflow boundary}\label{SecLinfty}
In this Section, we will deduce the $L^\infty_{x,v}$ estimate of the Boltzmann equation, which is the crucial estimate for the non-cutoff Boltzmann equation in a domain $\Omega$ with boundary. Unless otherwise stated, the underlying time norm in this Section is always within $[T_1,T_2]$. 

\smallskip 
Let $0\le T_1<T_2$, $\vpi\ge 0$, and fix $l\ge \ga+10$.
In this Section, we denote the level functions $f^{(l)}_{K,+}$ as in \eqref{flK} with $\de=1$: 
\begin{align*}
	f^{(l)}_{K,+}=\big(f-K\<v\>^{-l}\big)_+. 
\end{align*}
We begin with splitting the linearized mollified equation by adding an extra dissipation term $\eta\<v\>^lf$ with any small $\eta>0$:
\begin{align}\label{linear1}
	\left\{
	\begin{aligned}
		& \pa_tf+ v\cdot\na_xf = \vpi Vf+ \Gamma(\Psi,f)+\Gamma(\vp,\mu^{\frac{1}{2}})-\eta\<v\>^lf\quad \text{ in } [T_1, T_2]\times\Omega\times\R^3_v, \\
		& f|_{\Si_-}=g\qquad \text{ on }[T_1, T_2]\times\Si_-, \\
		& f(T_1,x,v)=f_{T_1}\quad \text{ in }\Omega\times\R^3_v,
	\end{aligned}\right.
\end{align}
with given $\vp$ and $\Psi=\mu^{\frac{1}{2}}+\psi\ge 0$.
By adding an extra damping term, we can split the solution $f$ to \eqref{linear1} into $f=f_1+f_2$, where $f_1$ and $f_2$ solve
\begin{align}\label{linear1a}
	\left\{
	\begin{aligned}
		& \pa_tf_1+ v\cdot\na_xf_1 = \vpi Vf_1+\Gamma(\Psi,f_1)+\Gamma(\vp_1,\mu^{\frac{1}{2}})\\&\qquad\qquad\qquad\qquad-N\<v\>^{l-2}f_1-\eta\<v\>^lf_1\quad \text{ in } [T_1, T_2]\times\Omega\times\R^3_v, \\
		& f_1|_{\Si_-}=g\qquad \text{ on }[T_1, T_2]\times\Si_-, \\
		& f_1(T_1,x,v)=f_{T_1}\quad \text{ in }\Omega\times\R^3_v,
	\end{aligned}\right.
\end{align}
and
\begin{align}\label{linear1b}
	\left\{
	\begin{aligned}
		& \pa_tf_2+ v\cdot\na_xf_2 =
		\vpi Vf_2+\Gamma(\Psi,f_2)+\Gamma(\vp_2,\mu^{\frac{1}{2}}) \\&\qquad\qquad\qquad\qquad+N\<v\>^{l-2}f_1-\eta\<v\>^lf_2\quad \text{ in } [T_1, T_2]\times\Omega\times\R^3_v,\\
		& f_2|_{\Si_-}=0\qquad \text{ on }[T_1, T_2]\times\Si_-, \\
		& f_2(T_1,x,v)=0\quad \text{ in }\Omega\times\R^3_v,
	\end{aligned}\right.
\end{align}
respectively. Here $N>0$, $\Psi=\mu^{\frac{1}{2}}+\psi$ and $\vp=\vp_1+\vp_2$ are given. 
The $L^2$ existences of solutions $f,f_1,f_2$ to the above three equations are given in Theorem \ref{weakfThm}. 
In fact, one can obtain the solutions $f,f_1$ to equations \eqref{linear1} and \eqref{linear1a} by using Theorem \ref{weakfThm}, respectively, and let $f_2=f-f_1$ to obtain the solution to equation \eqref{linear1b}. 
We will do the $L^\infty$ estimation of $f_1$ and $f_2$ separately in the next two subsections.
Note that:
\begin{itemize}[leftmargin=2em]
	\item the extra dissipation term $\<v\>^lf$ is only for the initial $L^\infty$ bound of $f_2$, and it doesn't contribute to the improved $L^\infty$ estimate of $f_1$ (the one used for the final $L^2$--$L^\infty$ estimate);
	\item $f_2$ has vanishing initial and inflow-boundary conditions;
	\item at the end of this Section, we will let $\eta\to0$ and obtain the $L^\infty$ estimate of the ``original" equation with any $\vpi\ge 0$. 
\end{itemize}

\subsection{\texorpdfstring{$L^\infty$}{L infty} estimate with non-vanishing data}
In this subsection, we will do the estimation about the equation \eqref{linear1a}. By dropping $(\cdot)_1$, we rewrite it as 
\begin{align}\label{linear1a1}
	\left\{
	\begin{aligned}
		& \pa_tf+ v\cdot\na_xf = \vpi Vf+\Gamma(\Psi,f)+\Gamma(\vp,\mu^{\frac{1}{2}})\\&\qquad\qquad\qquad\qquad-N\<v\>^{l-2}f-\eta\<v\>^lf\quad \text{ in } [T_1, T_2]\times\Omega\times\R^3_v, \\
		& f|_{\Si_-}=g\quad \text{ on }[T_1, T_2]\times\Si_-, \\
		& f(T_1,x,v)=f_{T_1}\quad \text{ in }\Omega\times\R^3_v,
	\end{aligned}\right.
\end{align}
with given $K,N>0$, $\vpi\ge 0$, $\Psi=\mu^{\frac{1}{2}}+\psi$, $\vp$, and
fixed $l\ge\ga+10$. Assume the \emph{a priori} control on $\psi,\vp$:
\begin{align}
	\label{psies}
	\sup_{T_1\le t\le T_2}\|\<v\>^l[\psi,\vp]\|_{L^\infty_{x}(\Omega)L^\infty_v}\le \de_0,
\end{align}
with $\de_0>0$ sufficiently small which will be chosen in Lemmas \ref{LinftyLemNonVan} and \ref{LinftyLemVanish}. 
Then the $L^2$ existence of equation \eqref{linear1a1} is given in Theorem \ref{weakfThm}.
We next give the $L^\infty$ estimate of \eqref{linear1a1}.

\begin{Lem}[$L^\infty$ Estimate for linear equation with non-vanishing data]
	\label{LinftyLemNonVan}
	Assume the same conditions as in (the local existence) Theorem \ref{weakfThm}. 
	Let $N>0$ be a large constant depending on $\ga,s$. 
	Suppose further that $\psi,\vp$ satisfies \eqref{psies} with sufficiently small $\de_0>0$.
	Suppose $f$ solves \eqref{linear1a1}. 
	Then $f$ satisfies
	\begin{multline}\label{Linfty1} 
		\|\<v\>^lf\|_{L^\infty_{t,x,v}([T_1, T_2]\times\ol\Omega\times\R^3_v)}
		\le K_1\\
		\equiv
		\max\big\{\frac{1}{2}\|\vp\|_{L^\infty_{t,x,v}([T_1, T_2]\times\Omega\times\R^3_v)},\,\|\<v\>^lg\|_{L^\infty_{t,x,v}([T_1, T_2]\times\Si_-)},\,\|\<v\>^lf_{T_1}\|_{L^\infty_{x,v}(\Omega\times\R^3_v)}\big\}.
	\end{multline}
	whenever the right-hand side is bounded. 
	Note that the upper bound $K_1$ is independent of $\vpi$.
\end{Lem}
\begin{proof}
	Let $K>0$. 
	We multiply \eqref{linear1a1} by $f^{(l)}_{K,+}$ to obtain
	\begin{align}\label{linear1aa}\notag
			& \frac{1}{2}\pa_t(f^{(l)}_{K,+})^2+ \frac{1}{2}v\cdot\na_x(f^{(l)}_{K,+})^2 = f^{(l)}_{K,+}\big(\vpi Vf+\Gamma(\Psi,f)+\Gamma(\vp,\mu^{\frac{1}{2}})\big) \\
			& \qquad\qquad\qquad\qquad-N\<v\>^{{l-2}}f^{(l)}_{K,+}f-\eta\<v\>^lf^{(l)}_{K,+}f\quad \text{ in } [T_1, T_2]\times\Omega\times\R^3_v,
	\end{align}
	where we used \eqref{deri+}. By integrating \eqref{linear1aa} over $\Omega\times\R^3_v$ and using Lemma \ref{CollLevelLem}, i.e. estimates \eqref{GOmegaL2}, \eqref{GOmegaL2a}, \eqref{GOmegaN} and \eqref{651a}, we have
	\begin{align*}
		&\frac{1}{2}\pa_t\|f^{(l)}_{K,+}\|_{L^2_x(\Omega)L^2_v}^2+\frac{1}{2}\int_{\pa\Omega\times\R^3_v}v\cdot n(f^{(l)}_{K,+})^2\,dS(x)dv=(-c_0+C\|\<v\>^4\psi\|_{L^\infty_x(\Omega)L^\infty_v})\|f^{(l)}_{K,+}\|_{L^2_x(\Omega)L^2_D}^2\\
		&\qquad+C\|\1_{|v|\le R_0}f^{(l)}_{K,+}\|_{L^2_x(\Omega)L^2_v}^2
		+C\|\mu^{\frac{1}{80}}\vp\|_{L^\infty_{x}(\Omega)L^\infty_v}\|\mu^{\frac{1}{80}}f^{(l)}_{K,+}\|_{L^1_{x}(\Omega)L^1_v}\\
		&\qquad+CK\|\<v\>^{-2}f^{(l)}_{K,+}\|_{L^1_{x}(\Omega)L^1_v}-N\|\<v\>^{\frac{l}{2}-1}f^{(l)}_{K,+}\|_{L^2_x(\Omega)L^2_v}^2-NK\|\<v\>^{{-2}}f^{(l)}_{K,+}\|_{L^1_x(\Omega)L^1_v}\\
		&\qquad-\vpi\|[\wh{C}^{}_0\<v\>^{{4}}f^{(l)}_{K,+},\<v\>^{{2}}\na_vf^{(l)}_{K,+}]\|^2_{L^2_{x}(\Omega)L^2_v}
		-\vpi K\wh{C}^2_0\|\<v\>^{{-l+8}}f^{(l)}_{K,+}\|_{L^1_x(\Omega)L^1_v}.
	\end{align*}
	Choose $\de_0>0$ small enough and choose $N$ large enough such that $N>4C$ and set 
	\begin{align}\label{Kchoice}
		K_1:=\max\big\{\frac{1}{2}\|\vp\|_{L^\infty_{t,x,v}([T_1, T_2]\times\Omega\times\R^3_v)},\,\|\<v\>^lg\|_{L^\infty_{t,x,v}([T_1, T_2]\times\Si_-)},\,\|\<v\>^lf_{T_1}\|_{L^\infty_{x,v}(\Omega\times\R^3_v)}\big\}.
	\end{align}
	Then we have
	\begin{align*}
		\frac{1}{2}\pa_t\|f^{(l)}_{K,+}\|_{L^2_x(\Omega)L^2_v}^2+\frac{1}{2}\int_{\pa\Omega}\int_{\R^3_v}v\cdot n(f^{(l)}_{K,+})^2\,dS(x)dv=-\frac{c_0}{2}\|f^{(l)}_{K,+}\|_{L^2_x(\Omega)L^2_D}^2.
	\end{align*}
	Integrating over $t\in[T_1, T_2]$, we obtain
	\begin{multline}\label{57}
		\|f^{(l)}_{K,+}\|_{L^\infty_t([T_1,T_2])L^2_x(\Omega)L^2_v}^2+\|f^{(l)}_{K,+}(t)\|_{L^2_t([T_1,T_2])L^2_{x,v}(\Si_+)}\le2\|f^{(l)}_{K,+}(T_1)\|_{L^2_x(\Omega)L^2_v}^2\\
		+2\|g^{(l)}_{K,+}(t)\|_{L^2_t([T_1,T_2])L^2_{x,v}(\Si_-)}.
	\end{multline}
	By the choice of $K_1$ in \eqref{Kchoice}, the initial and inflow-boundary terms in \eqref{57} vanish, and hence,
	\begin{align*}
		\|f^{(l)}_{K,+}\|_{L^\infty_t([T_1,T_2])L^2_x(\Omega)L^2_v}^2+\|f^{(l)}_{K,+}(t)\|_{L^2_t([T_1,T_2])L^2_{x,v}(\Si_+)}=0.
	\end{align*}
	Recalling the definition of $f^{(l)}_{K,+}$, i.e. \eqref{levelvl}, we deduce 
	\begin{align*}
		\sup_{(t,x,v)\in[T_1, T_2]\times\ol\Omega\times\R^3}\<v\>^lf\le K_1. 
	\end{align*}
	This gives the upper $L^\infty$ estimate in \eqref{Linfty1}.
	For the lower $L^\infty$ estimate, we let $h=-f$ and take the multiplication of \eqref{linear1a1} by $h^{(l)}_{K,+}$. Similar arguments and estimates can be made for the term $h^{(l)}_{K,+}$ instead of $f^{(l)}_{K,+}$ since we have the same estimate of the collision terms for $h$ as for $f$ in Lemma \ref{CollLevelLem} with $\vp$ replaced by $-\vp$ therein. Thus, similar arguments imply the lower $L^\infty$ estimate in \eqref{Linfty1}, and conclude Lemma \ref{LinftyLemNonVan}.
\end{proof}

\subsection{Initial \texorpdfstring{$L^\infty$}{L infty} estimate for vanishing data}\label{Secinitial1}
First, we should derive an initial large $L^\infty$ bound of $f_2$, which depends only on the time interval and $\|\<v\>^lf_1\|_{L^\infty_{t,x,v}([T_1,T_2]\times\Omega\times\R^3_v)}$ (which further depends on initial data). 
Then, based on this bound, we derive the improved small $L^\infty$ bound of $f_2$. Here, by dropping the $(\cdot)_2$ in the equation \eqref{linear1b}, we consider
\begin{align}\label{linear1b1}
	\left\{
	\begin{aligned}
		& \pa_tf+ v\cdot\na_xf =
		\vpi Vf+\Gamma(\Psi,f)+\Gamma(\vp,\mu^{\frac{1}{2}}) \\&\qquad\qquad\qquad\quad+N\<v\>^{l-2}f_1-\eta\<v\>^lf\quad \text{ in } [T_1, T_2]\times\Omega\times\R^3_v,\\
		& f|_{\Si_-}=0\quad\quad \text{ on }[T_1, T_2]\times\Si_-, \\
		& f(T_1,x,v)=0\quad \text{ in }\Omega\times\R^3_v. 
	\end{aligned}\right.
\end{align}
\begin{Lem}\label{initialLinftyLem}
	Let $\vpi\ge 0$, $\eta>0$, $0\le T_1<T_2<\infty$ with $T_2-T_1\le 1$, and let
	$N=N(\ga,s)>0$ be a large constant chosen in Lemma \ref{LinftyLemNonVan}.
	Suppose $\Psi=\mu^{\frac{1}{2}}+\psi\ge 0$, $\vp$ satisfy
	\begin{align}
		\label{Asspsib}
		\sup_{T_1\le t\le T_2}\|[\<v\>^{l}\psi,\<v\>^{l}\vp]\|_{L^\infty_{x}(\Omega)L^\infty_v}\le \de_0,
	\end{align}
	with sufficiently small $0<\de_0<1$. 
	Assume that (the result in Lemma \ref{LinftyLemNonVan})
	\begin{align}\label{f1K1}
		\|\<v\>^lf_1\|_{L^\infty_{t,x,v}([T_1,T_2]\times\Omega\times\R^3_v)}=K_1<\infty. 
	\end{align}
	Let $f$ be the solution to \eqref{linear1b1}
	in the sense of \eqref{weakf}. 
	Then $f$ has an upper bound
	\begin{align}\label{initialLinfty}
		\|\<v\>^{l}f\|_{L^\infty_{x,v}(\ol\Omega\times\R^3_v)}\le e^{C_\eta(t-T_1)}(NK_1+1)<\infty. 
	\end{align}
	where $C_\eta>0$ is a constant that depends on $\eta,l,\ga,s$ but independent of $T_1,T_2$.
\end{Lem}
The $L^\infty$ bound in \eqref{initialLinfty} is not only large but also local in time. So it's not an appropriate bound for our analysis but only serves as an \emph{a priori} bound. 
\begin{proof}
	The proof is a simple application of De Giorgi's arguments. Here we give a proof of the upper bound of $f$, while the lower bound of $f$ shares a similar calculation, and the $L^\infty$ bound (of $|f|$) follows.

	\smallskip 
	To capture the necessary dissipation, we use a time-dependent function
	$K(t)\ge1$ with the level functions denoted by
	\begin{align*}
		f^{(l)}_{K(t)} := f-\frac{K(t)}{\<v\>^l},\quad 
		f^{(l)}_{K,+}=f^{(l)}_{K(t)}\1_{f^{(l)}_{K(t)}\ge 0}. 
	\end{align*}
	Unlike the rest of the analysis in this work, $K(t)$ is a time-dependent function for the derivation of the initial $L^\infty$ bound. 
	The function $K(t)$ will be chosen later. 
	By taking $L^2$ inner product of \eqref{linear1b1} with $f^{(l)}_{K,+}$ over $L^2_{x}(\Omega)L^2_v$, and using \eqref{deri+}, we have 
	\begin{align*}
		&\frac{1}{2}\pa_t\|f^{(l)}_{K,+}\|^2_{L^2_{x}(\Omega)L^2_v}
		+\pa_tK\|\<v\>^{-l}f^{(l)}_{K,+}\|_{L^1_{x}(\Omega)L^1_v}
		+\frac{1}{2}\|f^{(l)}_{K,+}\|^2_{L^2_{x,v}(\Si_+)}\\
		&\quad\le\frac{1}{2}\|f^{(l)}_{K,+}\|^2_{L^2_{x,v}(\Si_-)}
		-\eta(\<v\>^{l}f,f^{(l)}_{K,+})_{L^2_{x}(\Omega)L^2_v}\\
		&\qquad+\Big(\vpi Vf+ \Gamma(\Psi,f-K\<v\>^{-l})+\Gamma(\vp,\mu^{\frac{1}{2}})+ \Gamma(\Psi,K\<v\>^{-l})+N\<v\>^{l-2}f_1,f^{(l)}_{K,+}\Big)_{L^2_{x}(\Omega)L^2_v}.
	\end{align*}
	Applying the energy estimates of collision terms from Lemma \ref{CollLevelLem}, i.e. \eqref{GOmegaL2}, \eqref{GOmegaL2a}, \eqref{O1r} and \eqref{651a} (although the estimate in Lemma \ref{CollLevelLem} does not involve $K(t)$, the estimates remain the same since the collision terms do not depend on $t$), we deduce 
	\begin{multline}\label{256}
		\frac{1}{2}\pa_t\|f^{(l)}_{K,+}\|^2_{L^2_{x}(\Omega)L^2_v}
		+\pa_tK\|\<v\>^{-l}f^{(l)}_{K,+}\|_{L^1_{x}(\Omega)L^1_v}
		+\eta K\|f^{(l)}_{K,+}\|_{L^1_{x}(\Omega)L^1_v}
		+\frac{1}{2}\|f^{(l)}_{K,+}\|^2_{L^2_{x,v}(\Si_+)}\\
		\le \frac{1}{2}\|f^{(l)}_{K,+}\|^2_{L^2_{x,v}(\Si_-)}
		+C\|f^{(l)}_{K,+}\|_{L^2_x(\Omega)L^2_v}^2
		+(NK_1+1+K)C\|\<v\>^{-2}f^{(l)}_{K,+}\|_{L^1_{x}(\Omega)L^1_v}, 
	\end{multline}
	with sufficiently small $\de_0\in(0,1)$, where we used \eqref{f1K1} to control $f_1$ and simply dropped the good terms. 
	For the $L^1$ norm, we choose $K(t)\ge NK_1+1$ to deduce 
	\begin{align*}
		(NK_1+1+K)C\|\<v\>^{-2}f^{(l)}_{K,+}\|_{L^1_{x}(\Omega)L^1_v}
		\le KC\|\<v\>^{-2}f^{(l)}_{K,+}\|_{L^1_{x}(\Omega)L^1_v}. 
	\end{align*}
	with $C=C(\ga,s,l)>0$. 
	Moreover, using interpolation, we have 
	\begin{align*}
		\notag
		KC\|\<v\>^{-2}f^{(l)}_{K,+}\|_{L^1_{x}(\Omega)L^1_v}
		\le
		C_\eta K\|\<v\>^{-l}f^{(l)}_{K,+}\|_{L^1_{x}(\Omega)L^1_v}+\eta K\|f^{(l)}_{K,+}\|_{L^1_{x}(\Omega)L^1_v}, 
	\end{align*}
	for some constant $C_\eta=C(\eta,\ga,s,l)>0$. 
	Then, noticing the inflow boundary data vanishes, the right-hand side of \eqref{256} is 
	\begin{align}\label{276}
		\le
		C\|f^{(l)}_{K,+}\|_{L^2_x(\Omega)L^2_v}^2
		+C_\eta K\|\<v\>^{-l}f^{(l)}_{K,+}\|_{L^1_{x}(\Omega)L^1_v}
		+\eta K\|f^{(l)}_{K,+}\|_{L^1_{x}(\Omega)L^1_v}. 
	\end{align}
	Now, to control all these right-hand terms, we will suitably choose $K(t)$. 
	Also, we will fix the constant $C_\eta,C>0$ here until the end of this proof. 
	To eliminate the $L^1$ norms, we choose $K(t)\ge NK_1+1$ such that 
	\begin{align*}
		\pa_tK\ge C_\eta K,
	\end{align*}
	for which we simply let 
	\begin{align*}
		K(t)=e^{C_\eta(t-T_1)}(NK_1+1).
	\end{align*}
	Substituting such $K(t)$ into \eqref{256} and \eqref{276}, we have
	\begin{align*}
		\frac{1}{2}\pa_t\|f^{(l)}_{K,+}\|^2_{L^2_{x}(\Omega)L^2_v}
		+\frac{1}{2}\|f^{(l)}_{K,+}\|^2_{L^2_{x,v}(\Si_+)}
		\le \frac{1}{2}\|f^{(l)}_{K,+}\|^2_{L^2_{x,v}(\Si_-)}. 
	\end{align*} 
	Then by the Gr\"{o}nwall's inequality and noticing the initial data vanishes, we have 
	\begin{align*}
		\|f^{(l)}_{K,+}\|^2_{L^\infty_t([T_1,T_2])L^2_{x}(\Omega)L^2_v}+\|f^{(l)}_{K,+}\|^2_{L^2_t([T_1,T_2])L^2_{x,v}(\Si_+)}=0. 
	\end{align*}
		which implies $f\le K(t)\<v\>^{-l}$ in $[T_1,T_2]\times\ol\Omega\times\R^3_v$. 
	The lower bound can be deduced similarly and we conclude Lemma \ref{initialLinftyLem}. 
\end{proof}

\subsection{Energy inequality for level functions}
In this subsection, we will prepare some prior results for the $L^\infty$ estimation of the equation \eqref{linear1b1}. 
By Theorem \ref{ThmExtend}, we consider the solution $f$ that solves the extended equation to \eqref{linear1b1}:
\begin{align}\label{linear1b2}
	\left\{
	\begin{aligned}
		& \pa_tf+ v\cdot\na_xf = \left\{\begin{aligned}
			& \vpi Vf+\Gamma(\Psi,f)
			+\Gamma(\vp,\mu^{\frac{1}{2}})&\\&\quad\quad+N\<v\>^{{l-2}}f_1-\eta\<v\>^lf & \text{ in } [T_1,T_2]\times\Omega\times\R^3_v,\\
			& -E\cdot\na_vf+P^2f& \text{ in } [T_1,T_2]\times D_{in},\\
			& -E\cdot\na_vf-P^2f& \text{ in } [T_1,T_2]\times D_{out},\\
		\end{aligned}\right. \\
		& f|_{\Si_-}=0\quad \text{ on }[T_1,T_2]\times\Si_-, \\
		& f(T_1,x,v)=0\quad \text{ in }\Omega\times\R^3_v,\\
		&f(T_1,x,v)=0\qquad \text{ in }D_{out},\\
		&f(T_2,x,v)=0\qquad \text{ in }D_{in},
	\end{aligned}\right.
\end{align}
in the sense of \eqref{weakfwhole}, where $\vpi\ge 0$, and $N>0$ is a large constant chosen in Lemma \ref{LinftyLemNonVan}. 

To apply the iteration, we give the energy inequality including regular spatial and velocity variables on level functions.
\begin{Lem}[Energy inequality for level functions]
	\label{energyinterLem}
	Assume the same conditions as in Lemmas \ref{initialLinftyLem} and \ref{interLem}. 
	Let $s'\in(0,1)$ be a small constant depending on $p$ (chosen in \eqref{spri} below), and let 
	$C_0>0$ (given in \eqref{Ep}) be sufficiently large depending on $l,\ga,s,p$. 
	Assume $f$ solves \eqref{linear1b2} and satisfies 
	\begin{align}\notag
		&\|\<v\>^{l_0+l-2}f\|^2_{L^2_{t,x,v}([T_1,T_2]\times\Omega\times\R^3_v)}\le C_1<\infty,\\
		&\|\<v\>^{-2}f\|^2_{L^2_{t,x,v}([T_1,T_2]\times\Omega\times\R^3_v)}\le \de_0,\quad
		\|\<v\>^lf\|_{L^\infty_{t,x,v}([T_1,T_2]\times\Omega\times\R^3_v)}=C_\infty<\infty.\label{inftyboundf2}
	\end{align}
	with a large constant $l_0=l_0(l,s,s',p)>0$. 
	Suppose that 
	\begin{align*}
		\|\<v\>^lf_1\|_{L^\infty_t([T_1,T_2])L^\infty_{x}(\Omega)L^\infty_v}\le K_1<\infty,
	\end{align*}
	for some $K_1>0$.
	Then for any $0\le M< K$, we have
	\begin{align}\label{724}
		&\notag\|f^{(l)}_{K,+}\|_{L^\infty_t([T_1,T_2])L^2_{x,v}(\R^6)}^2+\|f^{(l)}_{K,+}\|^2_{L^2_tL^2_{x}(\Omega)L^2_D}
		+\vpi\|[\wh{C}^{}_0\<v\>^{{4}}f^{(l)}_{K,+},\<v\>^{{2}}\na_vf^{(l)}_{K,+}]\|^2_{L^2_tL^2_{x}(\Omega)L^2_v}\\
		&\notag\quad+\frac{1}{C_0\max\{C_\infty^{2p-2},1\}}
		\Big\|\int_{\R^3}\1_{[T_1,T_2]}\<v\>^{-10}(f^{(l)}_{K,+})^2\,dv\Big\|_{B^{s',2}_p(\R^{4}_{t,x})}^p\\
		&\quad\le C(1+C_1)^{C}(1+K_1)^p\sum_{i=1}^4\frac{\ga_i\E_p(M)^{\beta_i}}{(K-M)^{\al_i}}.
	\end{align}
	For any $K\ge 0$, the same left-hand side of \eqref{724} is also
	\begin{align}\notag\label{724x}
		&\le C\|f^{(l)}_{K,+}\|^2_{L^2_tL^2_x(\Omega)L^2_v}+C(1+K+K_1)\|\<v\>^{-2}f^{(l)}_{K,+}\|_{L^1_tL^1_x(\Omega)L^1_v}\\
	&\quad+\frac{1}{\max\{C_\infty^{2p-2},1\}}\Big(\|f^{(l)}_{K,+}\|_{L^2_tL^2_x(\Omega)L^2_v}^{2p}
	+(1+K+K_1)^p\|\<v\>^{-2}f^{(l)}_{K,+}\|^p_{L^1_tL^1_{x}(\Omega)L^1_v}\Big).
	\end{align}
	Here $C=C(s,s',p,\ga,l)>0$ is some large constant. The parameters $\beta_i>1$ and $\ga_i,\al_i>0$, depending on $s,s',p$, are given by \eqref{albe}. 
	Furthermore, the estimate \eqref{724} holds for $h:=-f$, with $f^{(l)}_{K,+}$ replaced by $(-f)^{(l)}_{K,+}$. The functional $\E_p$ is given by \eqref{Ep}. 
\end{Lem}
\begin{proof}

Note that if not specified, the underlying time norm in this proof is within $[T_1,T_2]$.
	
\smallskip\noindent{\bf Step 1. Regular velocity estimate.}
We estimate the first to third left-hand terms of \eqref{724}. 

\smallskip 
For the part in $\Omega$, similar to Lemma \ref{LinftyLemNonVan}, by taking $L^2$ inner product of \eqref{linear1b2} with $f^{(l)}_{K,+}$ 
over $[T_1,T_2]\times\Omega\times\R^3_v$, applying Lemma \ref{CollLevelLem} and noticing the initial-inflow boundary data vanish, we have 
\begin{align}\label{esffff2}\notag
	&\|f^{(l)}_{K,+}\|_{L^\infty_t([T_1,T_2])L^2_x(\Omega)L^2_v}^2
	+\frac{c_0}{2}\|f^{(l)}_{K,+}\|_{L^2_tL^2_x(\Omega)L^2_D}^2
	+\|f^{(l)}_{K,+}\|_{L^2_tL^2_{x,v}(\Si_+)}^2\\
	&\quad\notag+2\vpi\|[\wh{C}^{}_0\<v\>^{{4}}f^{(l)}_{K,+},\<v\>^{{2}}\na_vf^{(l)}_{K,+}]\|^2_{L^2_tL^2_{x}(\Omega)L^2_v}
	+\vpi K\wh{C}^2_0\|\<v\>^{{-l+8}}f^{(l)}_{K,+}\|_{L^1_tL^1_x(\Omega)L^1_v}\\
	&\quad\le
	C\|\1_{|v|\le R_0}f^{(l)}_{K,+}\|_{L^2_tL^2_x(\Omega)L^2_v}^2
	+(\de_0+K+NK_1)C\|\<v\>^{-2}f^{(l)}_{K,+}\|_{L^1_tL^1_{x}(\Omega)L^1_v},
\end{align}
where 
we choose $\de_0\in(0,1)$ in \eqref{Asspsib} sufficiently small. 

\smallskip For the part in $\ol\Omega^c$, as in the proof of Lemma \eqref{extendreverThm}, we denote 
\begin{align}\label{hftt}
	h(t)=\left\{\begin{aligned}
		&f(T_1+T_2-t)\quad &\text{ in }D_{in},\\
		&f(t)\quad&\text{ in } D_{out}. 
	\end{aligned}\right.
\end{align} 
Then it follows from \eqref{linear1b2} that 
\begin{align}\label{hlin}
	\left\{
	\begin{aligned}
		&\pa_th+v\cdot\na_x\big(h\1_{D_{out}}-h\1_{D_{in}}\big)&\\
		&\qquad\qquad+E\cdot\na_v\big(h\1_{D_{out}}-h\1_{D_{in}}\big)+P^2h=0
			\quad\text{ in } [T_1,T_2]\times\ol\Omega^c\times\R^3_v,\\
		&h|_{\Si_-}=0\qquad\qquad \text{ on }[T_1, T_2]\times\Si_-,\\
		&h|_{\Si_+}=f\qquad\qquad \text{ on }[T_1, T_2]\times\Si_+,\\
		&h(T_1,x,v)=0\qquad \text{ in }\ol\Omega^c\times\R^3_v.
	\end{aligned}\right.
\end{align}
Taking the inner product of \eqref{hlin} with $h^{(l)}_{K,+}$ over $[T_1,T_2]\times\ol\Omega^c\times\R^3_v$, using the last two estimates in Lemma \eqref{CollLevelLem}, and applying the vanishing boundary property in \eqref{traceh}, we have 
\begin{multline}\label{esffff1}
	\|h^{(l)}_{K,+}\|_{L^\infty_t([T_1,T_2])L^2_{x,v}(\ol\Omega^c\times\R^3_v)}^2
	+\frac{1}{2}\|Ph^{(l)}_{K,+}\|^2_{L^2_{t,x,v}([T_1,T_2]\times\ol\Omega^c\times\R^3_v)}\\
	+\frac{K}{2}\|\<v\>^{-l}P^2h^{(l)}_{K,+}\|_{L^1_{t,x,v}([T_1,T_2]\times\ol\Omega^c\times\R^3_v)}
	\le C_l\|f^{(l)}_{K,+}\|_{L^2_t([T_1,T_2])L^2_{x,v}(\Si_+)}^2. 
\end{multline}
It follows from \eqref{esffff1} that if we choose $K=C_\infty$ given by \eqref{inftyboundf2} in this step, then $$\|h^{(l)}_{K,+}\|_{L^\infty_t([T_1,T_2])L^2_{x,v}(\ol\Omega^c\times\R^3_v)}^2=0,$$ and hence, $f\le C_\infty\<v\>^{-l}$. The lower bound can be deduced similarly and thus 
\begin{align*}
	\|\<v\>^{l}f\|_{L^\infty_{t,x,v}([T_1,T_2]\times\R^3_x\times\R^3_v)}\le C_\infty,
\end{align*} 
while noting that we only assume the initial $L^\infty$ bound within $\Omega$. 
Taking combination $\eqref{esffff2} +\ka\times\eqref{esffff1}$ with sufficiently $\ka>0$ and changing $h$ back to $f$, we have 
\begin{align}\label{260}\notag
	&\|f^{(l)}_{K,+}\|_{L^\infty_t([T_1,T_2])L^2_{x,v}(\R^3_x\times\R^3_v)}^2
	+\frac{c_0}{2}\|f^{(l)}_{K,+}\|_{L^2_tL^2_x(\Omega)L^2_D}^2
	+\|f^{(l)}_{K,+}\|_{L^2_t([T_1,T_2])L^2_{x,v}(\Si_+)}^2\\
	&\notag\quad+2\vpi\|[\wh{C}^{}_0\<v\>^{{4}}f^{(l)}_{K,+},\<v\>^{{2}}\na_vf^{(l)}_{K,+}]\|^2_{L^2_tL^2_{x}(\Omega)L^2_v}
	+\vpi K\wh{C}^2_0\|\<v\>^{{-l+8}}f^{(l)}_{K,+}\|_{L^1_tL^1_x(\Omega)L^1_v}\\
	&\notag\quad+\frac{1}{2}\|Pf^{(l)}_{K,+}\|^2_{L^2_tL^2_{x}(\ol\Omega^c)L^2_v}
	+\frac{K}{2}\|\<v\>^{-l}P^2f^{(l)}_{K,+}\|_{L^1_tL^1_{x}(\ol\Omega^c)L^1_v}\\
	&\quad\le
	C\|\1_{|v|\le R_0}f^{(l)}_{K,+}\|_{L^2_tL^2_x(\Omega)L^2_v}^2
	+(\de_0+K+K_1)C\|\<v\>^{-2}f^{(l)}_{K,+}\|_{L^1_tL^1_{x}(\Omega)L^1_v},
\end{align}
where $C>0$ depends on $l$. Note also that $N=N(\ga,s)>0$ is already chosen in Lemma \ref{LinftyLemNonVan}. This is the main energy estimate for velocity regularity.

\smallskip\noindent{\bf Step 2. Regular time-space estimate.} 
The Besov regularity of the level functions is already given in Lemma \ref{LemBesovreguLevel}, i.e. estimate \eqref{736}. Moreover, note from \eqref{esD} that if $T_2-T_1\le1$, 
\begin{align*}
	&\|\<v\>^{\frac{(\ga+2s)_+}{2}}f^{(l)}_{K,+}\|^2_{L^2_tL^2_x(\Omega)L^2_v}\le \|f^{(l)}_{K,+}\|^2_{L^2_tL^2_x(\Omega)L^2_D}\text{ when $\ga+2s\ge 0$},\\
	&\|\<v\>^{\frac{(\ga+2s)_+}{2}}f^{(l)}_{K,+}\|^2_{L^2_tL^2_x(\Omega)L^2_v}\le \|f^{(l)}_{K,+}\|^2_{L^\infty_tL^2_x(\Omega)L^2_v}\text{ when $\ga+2s< 0$}.
\end{align*}
Then we can use $L^2$ energy estimate \eqref{260} to control the right-hand $p$-power terms of \eqref{736}. That is, 
\begin{align}\label{736x}
	\notag\Big\|\int_{\R^3}\1_{[T_1,T_2]}\<v\>^{-10}(f^{(l)}_{K,+})^2\,dv\Big\|_{B^{s',2}_p(\R^{4}_{t,x})}^p
	&\notag\le C\max\{C_\infty^{2p-2},1\}\|\<v\>^{-2p}f^{(l)}_{K,+}\|_{L^2_tL^2_x(\R^3_x)L^2_v}^{2}\\&\quad\notag
	+C\|\1_{|v|\le R_0}f^{(l)}_{K,+}\|_{L^2_tL^2_x(\Omega)L^2_v}^{2p}\\&\quad
	+C(1+K+K_1)^p\|\<v\>^{-2}f^{(l)}_{K,+}\|^p_{L^1_tL^1_{x}(\Omega)L^1_v}. 
\end{align}
Note again that $N=N(\ga,s)>0$ is chosen in Lemma \ref{LinftyLemNonVan}. The $L^2_x(\R^3_x)$ norm will be absorbed by \eqref{260}. 

\smallskip\noindent{\bf Step 3. Putting estimates together and controlling $L^1,L^2$ norms.}
Choose a large constant $C_0>0$ depending only on the constant $C>0$ in \eqref{736x}, which depends on $l,\ga,s,p$. 
Then the linear combination $\eqref{260}+\max\{C_\infty^{2p-2},1\}^{-1}C_0^{-1}\times\eqref{736x}$ gives 
\begin{align}\label{271}
	&\notag\|f^{(l)}_{K,+}\|_{L^\infty_t L^2_{x,v}(\R^6_{x,v})}^2
	+\frac{c_0}{2}\|f^{(l)}_{K,+}\|_{L^2_tL^2_x(\Omega)L^2_D}^2
	+\|f^{(l)}_{K,+}\|_{L^2_tL^2_{x,v}(\Si_+)}^2\\
	&\notag\quad+2\vpi\|[\wh{C}^{}_0\<v\>^{{4}}f^{(l)}_{K,+},\<v\>^{{2}}\na_vf^{(l)}_{K,+}]\|^2_{L^2_tL^2_{x}(\Omega)L^2_v}
	+\vpi K\wh{C}^2_0\|\<v\>^{{-l+8}}f^{(l)}_{K,+}\|_{L^1_tL^1_x(\Omega)L^1_v}\\
	&\notag\quad+\frac{1}{2}\|Pf^{(l)}_{K,+}\|^2_{L^2_tL^2_{x}(\ol\Omega^c)L^2_v}
	+\frac{K}{2}\|\<v\>^{-l}P^2f^{(l)}_{K,+}\|_{L^1_tL^1_{x}(\ol\Omega^c)L^1_v}\\
	&\notag\quad+\frac{1}{C_0\max\{C_\infty^{2p-2},1\}}\Big\|\int_{\R^3}\1_{[T_1,T_2]}\<v\>^{-10}(f^{(l)}_{K,+})^2\,dv\Big\|_{B^{s',2}_p(\R^{4}_{t,x})}^p\\
	&\notag\le 
	C\|f^{(l)}_{K,+}\|^2_{L^2_tL^2_x(\Omega)L^2_v}+C(1+K+K_1)\|\<v\>^{-2}f^{(l)}_{K,+}\|_{L^1_tL^1_x(\Omega)L^1_v}\\
	&\quad+\frac{1}{\max\{C_\infty^{2p-2},1\}}\Big(\|f^{(l)}_{K,+}\|_{L^2_tL^2_x(\Omega)L^2_v}^{2p}
	+(1+K+K_1)^p\|\<v\>^{-2}f^{(l)}_{K,+}\|^p_{L^1_tL^1_{x}(\Omega)L^1_v}\Big),
\end{align}
for any $K\ge0$, which implies \eqref{724x}. Now, for the $L^1$ and $L^2$ norms within $\Omega$, we let $0\le M<K$ and apply Lemma \ref{interLem} with $m=0$ to deduce 
\begin{align}\label{720a}
	\|f^{(l)}_{K,+}\|^2_{L^2_{t,x,v}([T_1,T_2]\times\Omega\times\R^3_v)}
	&\le\frac{C\big(\max\{C_\infty^{2p-2},1\}\big)^{\frac{(1-\si)\beta_*\xi_*}{2p}}C_1^{\frac{(1-\beta_*)\xi_*}{4}}(\E_p(M))^{r_*}}{(K-M)^{\xi_*-2}},
	%
\end{align}
and
\begin{align}\label{720b}
	\|\<v\>^{-2}f^{(l)}_{K,+}\|_{L^1_{t,x,v}([T_1,T_2]\times\Omega\times\R^3_v)}
	&\le \frac{C\big(\max\{C_\infty^{2p-2},1\}\big)^{\frac{(1-\si)\beta_*\xi_*}{2p}}C_1^{\frac{(1-\beta_*)\xi_*}{4}}(\E_p(M))^{r_*}}{(K-M)^{\xi_*-1}},
\end{align}
where $l_0>0$ is a sufficiently large constant depending on $l,s,s',p$ (given in Lemma \ref{interLem}) and we put the constant $C_0$ (which is determined in \eqref{271} and used in \eqref{Ep} and \eqref{77}) inside constant $C$. 
Then $C=C(s,s',p,\ga,l)>0$ here is independent of $C_1,C_\infty$. 
Moreover, 
\begin{align*}
	1\le \frac{K}{K-M}.
\end{align*}
	Thus, substituting \eqref{720a} and \eqref{720b} into \eqref{271}, 
	and choosing $\de_0\in(0,1)$ sufficiently small, we have 
	\begin{align*}
		&\|f^{(l)}_{K,+}\|_{L^\infty_tL^2_{x,v}(\R^6_{x,v})}^2
		+\frac{c_0}{2}\|f^{(l)}_{K,+}\|_{L^2_tL^2_x(\Omega)L^2_D}^2
		+\|f^{(l)}_{K,+}\|_{L^2_tL^2_{x,v}(\Si_+)}^2\\
		&\quad +\vpi\|[\wh{C}^{}_0\<v\>^{{4}}f^{(l)}_{K,+},\<v\>^{{2}}\na_vf^{(l)}_{K,+}]\|^2_{L^2_tL^2_{x}(\Omega)L^2_v}
		\\
		&\quad+\frac{1}{C_0\max\{C_\infty^{2p-2},1\}}\Big\|\int_{\R^3}\1_{[T_1,T_2]}\<v\>^{-10}(f^{(l)}_{K,+})^2\,dv\Big\|_{B^{s',2}_p(\R^{4}_{t,x})}^p\\
		&\quad\le C(1+C_1)^{C}(1+K_1)^p\sum_{i=1}^4\frac{\ga_i\E_p(M)^{\beta_i}}{(K-M)^{\al_i}},
	\end{align*}
	where we used \eqref{sibexi1}, i.e. $\frac{(1-\si)\beta_*\xi_*}{2p}<1$. Here, the parameters are given by 
	\begin{align}\label{albe}
		\begin{aligned}
			&\ga_1=\max\{C_\infty^{2p-2},1\}^{\frac{(1-\si)\beta_*\xi_*}{2p}},\quad \ga_2=\frac{\ga_1K}{K-M},\quad\ga_3=1,\quad\ga_4=\big(\frac{K}{K-M}\big)^p,\\
			& \beta_1=\beta_2=r_*,\quad\beta_3=\beta_4=pr_*,\\
			&\al_1=\al_2=\xi_*-2,
			\quad\al_3=\al_4=p(\xi_*-2).
		\end{aligned}
	\end{align}
	Here, $\beta_i>1$ and $\al_i>0$ ($1\le i\le 4$), which can be seen from Lemma \ref{interLem} (i.e. \eqref{rxistar}).
	This implies \eqref{724}. Moreover, $C_0$ used in \eqref{271} depends on $s,p,\ga,l$ and $C>0$ used in \eqref{271}--\eqref{720b} depends on $s,s',p,\ga,l$.

	\smallskip
	Since $(-f)^{(l)}_{K,+}$ satisfies the same bound as $f^{(l)}_{K,+}$ in Lemma \ref{Qes1Lem}
	and Lemma \ref{CollLevelLem}, we can obtain the same estimate \eqref{724} for $(-f)^{(l)}_{K,+}$.
	This completes the proof of Lemma \ref{energyinterLem}.
\end{proof}

\subsection{Improved \texorpdfstring{$L^\infty$}{L infty} estimate for vanishing data (De Giorgi iteration)}
Now we are ready for the proof of improved $L^\infty$ estimate for linear equation with vanishing initial-inflow data, which is given by the Lemma \ref{LinftyLemVanish} below.
 But before that, we need to give a control on the energy $\E_0:=\E_p(0)$ first, which uses $L^2$ energy estimates as follows. 
\begin{Lem}
	\label{E0Lem}
	Let $\vpi\ge 0$, $0\le T_1<T_2<\infty$ with $T_2-T_1\le 1$, 
	and fix $l\ge\ga+10$, $-\frac{3}{2}<\ga\le 2$ and $0<s<1$. Let $p^\#$ be given in \eqref{ppsharp} and suppose $1<p<p^\#$. 
	Let $s'\in(0,1)$ be a sufficiently small constant depending on $p$, 
	and $l_0=l_0(l,s,s',p)>0$ be a sufficiently large constant (which can be chosen in Lemma \ref{interLem} with $m=l-2$).
	Let $\psi$, $\vp$ and $f_1$ be given and satisfy
	\begin{align*}\begin{aligned}
			\|[\<v\>^l\psi,\<v\>^l\vp]\|_{L^\infty_t([T_1,T_2])L^\infty_{x}(\Omega)L^\infty_v}&\le \de_0,\\
			\|\<v\>^lf_1\|_{L^\infty_t([T_1,T_2])L^\infty_{x}(\Omega)L^\infty_v}\le K_1&<\infty,
		\end{aligned}
	\end{align*}
	with a sufficiently small $\de_0\in(0,1)$,
	and some $K_1>0$.
	Assume that $f$ solves \eqref{linear1b2} in the sense of \eqref{weakfwhole} and satisfies 
	\begin{align*}\begin{aligned}
			&\|\<v\>^{l_0+l-2}f\|^2_{L^2_{t,x,v}([T_1,T_2]\times\Omega\times\R^3_v)}\le C_1<\infty,\\
			&\|\<v\>^{-2}f\|^2_{L^2_{t,x,v}([T_1,T_2]\times\Omega\times\R^3_v)}\le \de_0,\quad
			\|\<v\>^lf\|_{L^\infty_{t,x,v}([T_1,T_2]\times\Omega\times\R^3_v)}\le C_\infty<\infty.
		\end{aligned}
	\end{align*}
	Let $\E_0=\E_p(0)$ be given in \eqref{Ep}. Then 
	\begin{align}\label{773}
		\E_0 \le C(1+K_1)^p(\DD^{\frac{1}{2}}+\DD^{p}),
	\end{align}
	where $C=C(T_2-T_1,|\Omega|,l,\ga,s,s',p)>0$, and $\DD$ denotes
	\begin{align}\label{DD}
		\DD:=\int^{T_2}_{T_1}\|[\mu^{\frac{1}{10^4}}\vp,\<v\>^{l-2}(f+f_1)]\|^2_{L^2_x(\Omega)L^2_v}\,dt. 
	\end{align}
\end{Lem}

\begin{Lem}[$L^\infty$ estimate for linear equation with vanishing data]
	\label{LinftyLemVanish}
	Suppose the same conditions as in Lemma \ref{LinftyLemVanish}. 
	Let $\E_0=\E_p(0)$ be given in \eqref{Ep}. 
		Then $f$ satisfies
	\begin{align*}
		\|\<v\>^lf\|_{L^\infty_t([T_1,T_2])L^\infty_{x,v}(\R^3_x\times\R^3_v)}
		\le 
		C(1+C_1+K_1)^{C}\max_{1\le i\le 4}(\lam_i)^{\frac{1}{\al_i}}(\DD^{\frac{1}{2}}+\DD^{p})^{\frac{\beta_i-1}{\al_i}},
	\end{align*}
	where parameters $\al_i,\beta_i,\lam_i$ are given in \eqref{albex}, depending only on $s,s',p$.
	Also, the constant $C=C(l,\ga,s,s',p)>0$ is independent of $T_1,T_2$. 
\end{Lem}
Then we give the proof of the above two Lemmas. 

\begin{proof}[Proof of Lemma \ref{E0Lem}]
The proof is the application of $L^2$ energy estimates.
First, it follows from the energy inequalities in Lemma \ref{energyinterLem} (estimate \eqref{724x}) that (with $K=0$)
\begin{align}\label{724a}\notag
	\E_p(0)&\le C\|f_{+}\|^2_{L^2_tL^2_x(\Omega)L^2_v}+C(1+K_1)\|\<v\>^{-2}f_{+}\|_{L^1_tL^1_x(\Omega)L^1_v}\\
	&\quad+\frac{1}{\max\{C_\infty^{2p-2},1\}}\Big(\|f_{+}\|_{L^2_tL^2_x(\Omega)L^2_v}^{2p}
	+(1+K_1)^p\|\<v\>^{-2}f_{+}\|^p_{L^1_tL^1_{x}(\Omega)L^1_v}\Big).
\end{align}
where $f_+=\max\{f,0\}$ and the underlying time interval is $[T_1,T_2]$. Note that these norms are taken within $\Omega$. 
For the $L^1$ and $L^2$ norms, we have 
\begin{align*}
	\|\<v\>^{-2}f^{}_+\|_{L^1_tL^1_x(\Omega)L^1_v}+\|f^{}_+\|_{L^2_tL^2_x(\Omega)L^2_v}
	\le C_{|\Omega|}\|f\|_{L^2_t([T_1,T_2])L^2_x(\Omega)L^2_v}, 
\end{align*}
with $T_2-T_1\le 1$ and some constant $C_{|\Omega|}>0$ depending on the measure of (bounded domain) $\Omega$. 
Moreover, by \eqref{weakfes}, we can obtain the $L^2$ energy estimate for $f$ (not level function $f_+$). That is, applying Theorem \eqref{ThmExtend} (with $\phi=N\<v\>^{l-2}(f+f_1)$ therein) to equation \eqref{linear1b2}, $f$ satisfies 
\begin{multline}\label{fes1}
\|\<v\>^kf\|^2_{L^\infty_t([T_1, T_2])L^2_{x}(\Omega)L^2_{v}}
+c_0\|\<v\>^kf\|_{L^2_t([T_1, T_2])L^2_{x}(\Omega)L^2_D}^2
+N\|\<v\>^{k+l-2}f\|_{L^2_t([T_1, T_2])L^2_{x}(\Omega)L^2_v}^2\\
\le \|[\mu^{\frac{1}{10^4}}\vp,N\<v\>^{k+l-2}(f+f_1)]\|_{L^2_t([T_1, T_2])L^2_{x}(\Omega)L^2_v}^2,
\end{multline}
for any $k\ge 0$, 
where we chose $N>0$ sufficiently large to absorb the $L^2$ energy of $f$. 
Therefore, substituting the above estimates into \eqref{724a} and recalling the functional $\DD$ given in \eqref{DD}, we obtain 
\begin{align*}
	\E_p(0)=C(1+K_1)^p(\DD^{\frac{1}{2}}+\DD^{p}), 
\end{align*}
with some $C=C(l,\ga,s,s',p,|\Omega|)>0$. 
This completes the proof of Lemma \ref{E0Lem}.
\end{proof}

To prove Lemma \ref{LinftyLemVanish}, we will apply the De Giorgi iteration scheme. Note that the assumptions allow us to use the preparation in Lemmas \ref{initialLinftyLem}, \ref{interLem}, and \ref{energyinterLem}.
\begin{proof}[Proof of Lemma \ref{LinftyLemVanish}]
	Fix $K_0>0$, which will be determined later in \eqref{K0111}. Denote the increasing levels $M_k$ by
	\begin{align*}
		M_k:=K_0\big(1-\frac{1}{2^k}\big), \quad k=0,1,2,\cdots.
	\end{align*}
	Notice that $M_0=0$, $\lim_{k\to\infty}M_k=K_0$, $M_k-M_{k-1}=K_02^{-k}>0$ and $\frac{M_k}{M_k-M_{k-1}}=2^k-1\le 2^k$. 
	Applying Lemma \ref{energyinterLem} with $(M,K)=(M_{k-1},M_k)$ and evaluating constants $\ga_i$ given in \eqref{albe}, 
	\begin{align*}
		&\notag\|f^{(l)}_{M_k,+}\|_{L^\infty_tL^2_{x,v}(\R^6)}^2+\|f^{(l)}_{M_k,+}\|^2_{L^2_tL^2_{x}(\Omega)L^2_D}
		+\vpi\|[\wh{C}^{}_0\<v\>^{{4}}f^{(l)}_{M_k,+},\<v\>^{{2}}\na_vf^{(l)}_{M_k,+}]\|^2_{L^2_tL^2_{x}(\Omega)L^2_v}\\
		&\notag\quad+\frac{1}{C_0\max\{C_\infty^{2p-2},1\}}
		\Big\|\int_{\R^3}\1_{[T_1,T_2]}\<v\>^{-10}(f^{(l)}_{M_k,+})^2\,dv\Big\|_{B^{s',2}_p(\R^{4}_{t,x})}^p\\
		&\quad\le C(1+C_1)^{C}(1+K_1)^p\sum_{i=1}^4\frac{\lam_i2^{k(\al_i+p)}\E_p(M)^{\beta_i}}{(K_0)^{\al_i}}, 
	\end{align*}
	where $C=C(s,s',p,\ga,l)>0$ and the parameters $\lam_i,\al_i>0$ and $\beta_i>1$ are given by 
	\begin{align}\label{albex}
		\begin{aligned}
			&\lam_1=\lam_2=\max\{C_\infty^{2p-2},1\}^{\frac{(1-\si)\beta_*\xi_*}{2p}},\quad\lam_3=\lam_4=1,\\
			& \beta_1=\beta_2=r_*,\quad\beta_3=\beta_4=pr_*,\\
			&\al_1=\al_2=\xi_*-2,
			\quad\al_3=\al_4=p(\xi_*-2).
		\end{aligned}
	\end{align}
	Thus, using functional $\E_p(M_k)$ from \eqref{Ep}, one has
	\begin{align}\label{524a}
		\E_p(M_k)\le C(1+C_1)^{C}(1+K_1)^p\sum_{i=1}^4\frac{\lam_i2^{k(\al_i+p)}\E_p(M_{k-1})^{\beta_i}}{(K_0)^{\al_i}},
	\end{align}
	for any $k\ge 1$. Then we can perform the De Giorgi iteration on the sequence $\{\E_p(M_k)\}$. 
	Noticing $\beta_i>1$, we write 
	\begin{equation}\label{Ekstar}
		Q_0=\max_{1\le i\le 4}2^{\frac{\al_i+p}{\beta_i-1}}>1,\quad \E^*_k=\frac{\E_0}{(Q_0)^{k}},\quad \text{ for }k=0,1,2,\dots,
	\end{equation}
	as an artificial sequence, and denote the upper bound by 
	\begin{equation}\label{K0111}
		K_0:= \max_{1\le i\le 4}\Big((4\lam_iC_2)^{\frac{1}{\al_i}}(\E_0)^{\frac{\beta_i-1}{\al_i}}(Q_0)^{\frac{\beta_i}{\al_i}}\Big),
	\end{equation}
	where $\E_0=\E_p(0)$ is given by \eqref{Ep}, 
	and $$C_2={C(1+C_1)^{C}(1+K_1)^p}>0$$ is the constant in \eqref{524a}.
	By \eqref{Ekstar} and \eqref{K0111}, we have $\E^*_0=\E_0$ and
	\begin{align}\label{Ekstar1}
		\E^*_k &\notag= \frac{\E_0}{(Q_0)^{k}}
		=\frac{1}{4}\sum_{i=1}^4\frac{(\E^*_{k-1})^{\beta_i}(K_0)^{\al_i}\E_0}{(\E^*_{k-1})^{\beta_i}(K_0)^{\al_i}(Q_0)^{k}}\\
		&\notag=\frac{1}{4}\sum_{i=1}^4\frac{(\E^*_{k-1})^{\beta_i}\max_{1\le j\le 4}\Big((4\lam_jC_2)^{\frac{1}{\al_j}}(\E_0)^{\frac{\beta_j-1}{\al_j}}(Q_0)^{\frac{\beta_j}{\al_j}}\Big)^{\al_i}\E_0}{\big(\frac{\E_0}{(Q_0)^{k-1}}\big)^{\beta_i}(K_0)^{\al_i}(Q_0)^{k}}\\
		&\ge C_2\sum_{i=1}^4\frac{(\E^*_{k-1})^{\beta_i}\lam_i(Q_0)^{k(\beta_i-1)}}{(K_0)^{\al_i}}
		\ge C_2\sum_{i=1}^4\frac{\lam_i2^{k(\al_i+p)}(\E^*_{k-1})^{\beta_i}}{(K_0)^{\al_i}}.
	\end{align}
	Comparing \eqref{Ekstar1} and \eqref{524a}, and using comparison principle (since $\E_0=\E_0^*=\E_p(M_0)$), 
	\begin{align*}
		\E_p(M_k)\le \E_k^*\to 0\text{ as }k\to\infty,
	\end{align*}
	since $Q_0>1$. Consequently, taking the limit $k\to\infty$, recall the functional $\E_p(M_k)$ in \eqref{Ep}, we deduce 
	\begin{align*}
		\|f^{(l)}_{K_0,+}\|^2_{L^\infty_t([T_1,T_2])L^2_{x,v}(\R^3_x\times\R^3_v)}=0,
	\end{align*}
	where $K_0$ is given by \eqref{K0111}. Thus,
	\begin{align*}
		\|(\<v\>^lf)_+\|_{L^\infty_{x,v}(\R^3_x\times\R^3_v)} & \le K_0
		\le
		C(1+C_1)^{C}(1+K_1)^p\max_{1\le i\le 4}(\lam_i)^{\frac{1}{\al_i}}(\E_0)^{\frac{\beta_i-1}{\al_i}},
	\end{align*}
	where the constant $C=C(s,s',p,\ga,l)>0$ and the parameters $\al_i,\beta_i,\lam_i$ are given in \eqref{albex}. 
Substituting estimate \eqref{773} for $\E_0$ to this, we have 
\begin{align*}
	\|(\<v\>^lf)_+\|_{L^\infty_{x,v}(\R^3_x\times\R^3_v)}
	&\le C(1+C_1+K_1)^{C}\max_{1\le i\le 4}(\lam_i)^{\frac{1}{\al_i}}(\DD^{\frac{1}{2}}+\DD^{p})^{\frac{\beta_i-1}{\al_i}}.
\end{align*}
Since the Lemmas \ref{interLem} and \ref{energyinterLem} have their corresponding counterparts for $-f$, a similar bound holds for $-f$ also in Lemma \ref{E0Lem}, i.e \eqref{773}. Thus, similarly, if we use $(-f)^{(l)}_{K,+}$ to replace $f^{(l)}_{K,+}$ in $\E_0$, then we have the same lower bound: 
	\begin{align*}
		\|(-\<v\>^lf)_+\|_{L^\infty_t([T_1,T_2])L^\infty_{x,v}(\R^3_x\times\R^3_v)}
		\le
		C(1+C_1+K_1)^{C}\max_{1\le i\le 4}(\lam_i)^{\frac{1}{\al_i}}(\DD^{\frac{1}{2}}+\DD^{p})^{\frac{\beta_i-1}{\al_i}}.
	\end{align*}
	This completes the proof of Lemma \ref{LinftyLemVanish}.
\end{proof}

\subsection{\texorpdfstring{$L^\infty$}{L infty} estimate of full linear equation}\label{Sec65}
In this Subsection, we will combine the $L^\infty$ estimate for non-vanishing data in Lemma \ref{LinftyLemNonVan} and improved $L^\infty$ estimate for vanishing data in Lemma \ref{LinftyLemVanish}. 
Recall that at the beginning of Section \ref{SecLinfty}, we split the modified linear equation \eqref{linear1} into \eqref{linear1a} and \eqref{linear1b}. Such splitting allows us to obtain the properties as in Remark \ref{coeff1}. Moreover, we will let $\eta\to0$ to recover the ``original" linear equation:
\begin{align}\label{linear1f}
	\left\{
	\begin{aligned}
		& \pa_tf+ v\cdot\na_xf = \vpi Vf+ \Gamma(\Psi,f)+\Gamma(\vp,\mu^{\frac{1}{2}})\quad \text{ in } [T_1, T_2]\times\Omega\times\R^3_v, \\
		& f|_{\Si_-}=g\qquad \text{ on }[T_1, T_2]\times\Si_-, \\
		& f(T_1,x,v)=f_{T_1}\quad \text{ in }\Omega\times\R^3_v,
	\end{aligned}\right.
\end{align}
\begin{Thm}[$L^\infty$ estimate for linear equation] \label{LinftyLinear}
	Fix $\vpi\ge 0$, $l\ge\ga+10$, $-\frac{3}{2}<\ga\le 2$ and $0<s<1$. Let $0\le T_1<T_2<\infty$ with $T_2-T_1\le 1$, and let $p^\#$ be given in \eqref{ppsharp} and fix any $1<p<p^\#$. Let $s'\in(0,1)$ be a sufficiently small constant depending on $p$, 
	and $l_0=l_0(l,s,s',p)>0$ be a sufficiently large constant (which can be chosen in Lemma \ref{interLem}), and $N=N(\ga,s)>0$ be a large constant chosen in Lemma \ref{LinftyLemNonVan}.
	
	Suppose $\psi$, $\vp=\vp_1+\vp_2$, $f_{T_1}$ and $g$ satisfy
	\begin{align}\label{ve1}
		\begin{aligned}
			\|\<v\>^l\max\{|\psi|,|\vp_1|,|\vp_2|\}\|_{L^\infty_t([T_1,T_2])L^\infty_{x}(\Omega)L^\infty_v}&= \de_0,\\
			\|\<v\>^{l}g\|_{L^\infty_{t,x,v}([T_1, T_2]\times\Si_-)}= \de_\infty,\quad\|\<v\>^{l}f_{T_1}\|_{L^\infty_{x,v}(\Omega\times\R^3_v)} & = \de_\infty', \\
			\|\<v\>^{l-2}g\|^2_{L^2_t([T_1,T_2])L^2_{x,v}(\Si_-)}+
			\|\<v\>^{l-2}f_{T_1}\|^2_{L^2_x(\Omega)L^2_v(\R^3_v)}&\\ +\|[\vp,\vp_1,\vp_2]\|^2_{L^\infty_t([T_1,T_2])L^2_{x}(\Omega)L^2_v}
			& = \de_1,\\
			\|\<v\>^{l_0+2l-2}g\|^2_{L^2_t([T_1,T_2])L^2_{x,v}(\Si_-)}+
			\|\<v\>^{l_0+2l-2}f_{T_1}\|^2_{L^2_x(\Omega)L^2_v}
			& = \ti C_1.
		\end{aligned}
	\end{align}
	with constants $\ti C_1>0$ and sufficiently small $\de_0,\de_1,\de_\infty,\de_\infty'\in(0,1)$.
	Then the solution $f$ to \eqref{linear1f} satisfies
	\begin{multline}\label{Linftyx}
		\|\<v\>^lf\|_{L^\infty_t([T_1,T_2])L^\infty_{x,v}(\ol\Omega\times\R^3_v)}
		\le\max\big\{\frac{1}{2}\|\vp_1\|_{L^\infty_{t,x,v}([T_1, T_2]\times\Omega\times\R^3_v)},\,\de_\infty,\,\de_\infty'\big\}\\
		+ C(1+\ti C_1+K_1)^{C}\Big((T_2-T_1)e^{C}\de_1\Big)^{\zeta},
	\end{multline}
	where $C=C(l,\ga,s,s',p)>0$ and $\zeta=\zeta(s)>0$ are independent of $T_1,T_2$. Here, $K_1$ is given by 
	\begin{align*}
		K_1=
		\max\big\{\frac{1}{2}\|\vp_1\|_{L^\infty_{t,x,v}([T_1, T_2]\times\Omega\times\R^3_v)},\,\de_\infty,\,\de_\infty'
		\big\}, 
	\end{align*}
	Furthermore, $f$ can be split as $f=f_1+f_2$, where $f_1$ and $f_2$ are the solutions to \eqref{linear1a} and \eqref{linear1b} respectively, and satisfy $L^\infty$ estimate \eqref{Linftyx} with $f$ replaced by $f_1$ and $f_2$ on the left-hand side.
\end{Thm}
\begin{proof}
	To find the estimate of $f$ to equation \eqref{linear1f}, we first consider the linear mollified Boltzmann equation \eqref{linear1} with dissipation $\eta\<v\>^lf$ and split $f=f_1+f_2$, where $f_1$ and $f_2$ solve the non-vanishing data equation \eqref{linear1a} and vanishing data equation \eqref{linear1b} respectively. 
	
	\smallskip 
	Until the end of this proof, if not specified, the underlying time interval is $[T_1,T_2]$. 
	Applying Lemma \ref{LinftyLemNonVan} to $f_1$, we have 
	\begin{align}\label{Linfty1z}
		\|\<v\>^lf_1\|_{L^\infty_tL^\infty_{x}(\ol\Omega)L^\infty_v}
		\le K_1\equiv\max\big\{\frac{1}{2}\|\vp_1\|_{L^\infty_tL^\infty_{x}(\Omega)L^\infty_v},\,\|\<v\>^lg\|_{L^\infty_tL^\infty_{x,v}(\Si_-)},\,
		\|\<v\>^lf_{T_1}\|_{L^\infty_{x}(\Omega)L^\infty_v}\big\}. 
	\end{align}
	On the other hand, 
	the $L^2$ estimate for $f_2$, for instance \eqref{fes1}, implies 
	\begin{align*}
		\|\<v\>^{l_0+l-2}f_2\|^2_{L^2_tL^2_x(\Omega)L^2_v}&\le C\|\<v\>^{l_0+l}f_2\|^2_{L^2_tL^2_x(\Omega)L^2_D}
		\le C\|[\mu^{\frac{1}{10^4}}\vp,N\<v\>^{l_0+2l-2}f]\|_{L^2_tL^2_{x}(\Omega)L^2_v}^2,\\
		\|\<v\>^{-2}f_2\|^2_{L^2_tL^2_x(\Omega)L^2_v}&\le C\|f_2\|^2_{L^2_tL^2_x(\Omega)L^2_D}\le C\|[\mu^{\frac{1}{10^4}}\vp,N\<v\>^{l-2}f]\|_{L^2_tL^2_{x}(\Omega)L^2_v}^2.
	\end{align*}
	For the $L^2$ norm of $f$, we have from \eqref{fLineares1} that, for any $k\ge 0$, 
	\begin{multline}\label{L2fff}
		\|\<v\>^kf\|_{L^\infty_tL^2_x(\Omega)L^2_v}^2		+\|\<v\>^kf\|^2_{L^2_tL^2_{x,v}(\Si_+)}+c_0\|\<v\>^kf\|_{L^2_tL^2_x(\Omega)L^2_D}^2\\
		\le
		e^{C(T_2-T_1)}\Big(\|\<v\>^kf(T_1)\|_{L^2_{x}(\Omega)L^2_v}^2+\|\vp\|_{L^2_tL^2_x(\Omega)L^2_v}^2+\|\<v\>^kg\|^2_{L^2_tL^2_{x,v}(\Si_-)}\Big). 
	\end{multline}
	Combining the above two estimates and assumption \eqref{ve1}, we have 
	\begin{align*}
		\begin{aligned}
			\|\<v\>^{l_0+l-2}f_2\|^2_{L^2_tL^2_x(\Omega)L^2_v}
			&\le \ti C_1,\\
			\|\<v\>^{-2}f_2\|^2_{L^2_tL^2_x(\Omega)L^2_v}
			&\le \de_1.
		\end{aligned}
	\end{align*}
	Applying Lemma \ref{initialLinftyLem} to $f_2$, we obtain the initial $L^\infty$ bound: 
	\begin{align}\label{inif2}
		\|\<v\>^{l}f_2\|_{L^\infty_t([T_1,T_2])L^\infty_{x}(\ol\Omega)L^\infty_v}\le e^{C_\eta(t-T_1)}(NK_1+1).
	\end{align}
	The problem is that the initial $L^\infty$ bound of $f_2$ in Theorem \ref{initialLinftyLem} depends on $\eta>0$; so it serves as the \emph{a priori} bound such that the following energy on the right-hand side is finite.
	For the improved $L^\infty$ bound of $f_2$, we denote it by 
	\begin{align}\label{488a}
		C_\infty=\|\<v\>^{l}f_2\|_{L^\infty_t([T_1,T_2])L^\infty_{x}(\ol\Omega)L^\infty_v}, 
	\end{align}
	which is finite due to \eqref{inif2} (such finiteness is essential). 
%
%
%
	Therefore, by Lemmas \ref{LinftyLemVanish}, and recalling parameters $\al_i,\beta_i,\lam_i$ given by \eqref{albex}, we have 
	\begin{align}\label{488x}
\|\<v\>^{l}f_2\|_{L^\infty_t([T_1,T_2])L^\infty_{x}(\ol\Omega)L^\infty_v}
		&\le {C(1+\ti C_1+K_1)^{C}}\max_{1\le i\le 4}(\lam_i)^{\frac{1}{\al_i}}(\DD^{\frac{1}{2}}+\DD^{p})^{\frac{\beta_i-1}{\al_i}},
	\end{align}
	where $\DD$ is given by \eqref{DD} (note that the $f$ in \eqref{DD} is now $f_2$ here), i.e.
	\begin{align*}
		\DD:=\int^{T_2}_{T_1}\|[\mu^{\frac{1}{10^4}}\vp,\<v\>^{l-2}f]\|^2_{L^2_x(\Omega)L^2_v}\,dt,
	\end{align*}
	 which, by
	using $L^2$ estimate \eqref{L2fff} and assumption \eqref{ve1}, satisfies
	\begin{align*}
		\DD\le (T_2-T_1)e^C\de_1<1.
	\end{align*}
	if we choose $\de_1\in(0,1)$ small (depending only on $\ga,s$). 
	Note that, we have fixed $p$, and the exponent $\frac{(1-\si)\beta_*\xi_*}{2p}$ in \eqref{albex} is the same the one in \eqref{sibexi1}, and thus 
	\begin{align}\label{627x}
		\frac{(1-\si)\beta_*\xi_*}{2p}<1, \text{ and }\xi_*>2+\frac{r(1)-2}{r(p^\#)}.
	\end{align}
Therefore, by \eqref{albex} and Lemma \ref{interLem}, we know that $\beta_i=\beta_i(s,p)>1$ are constants, and 
\begin{align*}
	&(\lam_1)^{\frac{1}{\al_1}}=(\lam_2)^{\frac{1}{\al_2}}
	=\max\{C_\infty^{2p-2},1\}^{\frac{(1-\si)\beta_*\xi_*}{2p(\xi_*-2)}},\\
	&(\lam_3)^{\frac{1}{\al_3}}=(\lam_4)^{\frac{1}{\al_4}}
	=1.
\end{align*}
Then we continue \eqref{488x} to deduce 
	\begin{align}\label{inftyf2aa}\notag
		\|\<v\>^{l}f_2\|_{L^\infty_t([T_1,T_2])L^\infty_{x}(\ol\Omega)L^\infty_v}&\le {C(1+\ti C_1+K_1)^{C}}\max\{C_\infty^{2p-2},1\}^{\frac{(1-\si)\beta_*\xi_*}{2p(\xi_*-2)}}\DD^{\zeta}\\
		&\le {C(1+\ti C_1+K_1)^{C}}\max\{C_\infty^{2p-2},1\}^{\frac{(1-\si)\beta_*\xi_*}{2p(\xi_*-2)}}((T_2-T_1)e^C\de_1)^{\zeta},
	\end{align}
	where $C=C(l,\ga,s,s',p)>0$ and $\zeta=\zeta(s,s',p)>0$ are some constants. 
	Therefore, there are two cases as below:
	\begin{enumerate}
		\item if $C_\infty<1$, then we obtain the upper bound 
		\begin{align*}
			\|\<v\>^lf_2\|_{L^\infty_t([T_1,T_2])L^\infty_x(\ol\Omega)L^\infty_{v}}
			\le {C(1+\ti C_1+K_1)^{C}}((T_2-T_1)e^C\de_1)^{\zeta};
		\end{align*}
		
		\smallskip
		\item if $C_\infty\ge 1$, then \eqref{inftyf2aa} implies 
		\begin{align}\label{750a}
			\|\<v\>^lf_2\|_{L^\infty_t([T_1,T_2])L^\infty_x(\ol\Omega)L^\infty_{v}}
			&\le {C(1+\ti C_1+K_1)^{C}}C_\infty^{\frac{(2p-2)(1-\si)\beta_*\xi_*}{2p(\xi_*-2)}}((T_2-T_1)e^C\de_1)^{\zeta}. 
		\end{align}
		From estimate \eqref{627x} (or \eqref{sibexi1}) and the choice of $p^\#$ given in \eqref{ppsharp}, we deduce that for any $p\in(1,p^\#)$, the exponent satisfies 
		\begin{align*}
			\frac{(1-\si)\beta_*\xi_*}{2p}\frac{2p-2}{\xi_*-2}<\frac{2p-2}{\xi_*-2}<1, 
		\end{align*}
		which is a fixed universal constant. (These parameters depend only on $s,p$ while $p$ is fixed).
		Therefore, we can absorb $C_\infty$ on the right-hand side of \eqref{750a} by the left hand due to its definition \eqref{488a}. Then we obtain \eqref{inftyf2aa1} with different constants $C,\zeta>1$. Further, if we choose $\de_1>0$ sufficiently small (depending on $\ga,s,l,|\Omega|$), then $C_\infty<1$, which reduces to the first case. 
	\end{enumerate}
	In summary, we obtain 
	\begin{align}\label{inftyf2aa1}
		\|\<v\>^lf_2\|_{L^\infty_tL^\infty_x(\ol\Omega)L^\infty_{v}}
		&\le {C(1+\ti C_1+K_1)^{C}}((T_2-T_1)e^C\de_1)^{\zeta}. 
	\end{align}
	Combining the $L^\infty$ estimates \eqref{Linfty1z} and \eqref{inftyf2aa1}, we see that the solution $f^\eta$ to the modified equation \eqref{linear1} satisfies \eqref{Linftyx}. Together with \eqref{L2fff} we know that $f^\eta$ has $L^2$ and $L^\infty$ energy estimates on $[T_1,T_2]$ uniformly in $\eta>0$, and thus has a subsequence which has a weak-$*$ limit $f$. Since the modified equation \eqref{linear1} is linear, it's standard to write it in the weak form and take the weak-$*$ limit to deduce that the limit $f$ satisfies the ``original'' linear equation \eqref{linear1f} (we will also consider the weak-$*$ limit for the \emph{nonlinear} case later in Section \ref{Sec8}, and one can refer to the details there). Moreover, the limit satisfies the same $L^\infty$ estimate \eqref{Linftyx}. 
	It's also direct from \eqref{Linfty1z} and \eqref{inftyf2aa1} that $f_1,f_2$ satisfy \eqref{Linftyx} with $f$ on the left-hand side of \eqref{Linftyx} replaced by $f_1,f_2$.
	This completes the proof of Theorem \ref{LinftyLinear}.
\end{proof}

\section{\texorpdfstring{$L^2$}{L2}--\texorpdfstring{$L^\infty$}{L infty} estimate for inflow boundary}\label{Sec8}
In this section, we will derive the existence of the nonlinear Boltzmann equation with the inflow-boundary condition in the bounded domain $\Omega$, by using the $L^2$--$L^\infty$ energy method.
As explained in Section \ref{Sec1non}, we added a regularizing term to solve the difficulties arising from nonlinearity. 
However, since $f^{(l)}_{K,+}$ is merely Lipschitz continuous, we can only insert the first-order derivative $\na_v$ into the regularizing term $Vf$ defined in \eqref{Vf}. Such a regularizing term can control the collision norm $\|f\|_{L^2_D}$ only if $s\in(0,\frac{1}{2})$. For the case $s\in[\frac{1}{2},1)$ we will truncate the cross section $b(\cos\th)$ into a weaker singular form, as in \cite{Alonso2022}. Unlike the torus case in \cite{Alonso2022}, one cannot have strong convergence for the boundary case because there is not enough spatial regularity and one cannot use the Sobolev compact embedding. 

\smallskip Let $-\frac{3}{2}<\ga\le 2$, $s\in(0,1)$, $0\le T_1<T_2<\infty$ with $T_2-T_1\le 1$ and let $l\ge\ga+10$ be a fixed constant. We prove the local-in-time and global-in-time nonlinear problems in the following two Subsections respectively. 

\subsection{Local nonlinear theory}
We begin with the construction of a local-in-time solution to the nonlinear Boltzmann equation \eqref{B1}.
For this purpose, we first consider the regularizing equation \eqref{sec1noneq}:
\begin{align}\label{liniter1}
	\left\{
	\begin{aligned}
		& \pa_tf^{\vpi}+ v\cdot\na_xf^{\vpi} =\vpi Vf^{\vpi}+\Gamma(\mu^{\frac{1}{2}}+f^{\vpi}\chi_{\de_0}(\<v\>^lf^\vpi),f^{\vpi})\\&\qquad\qquad\qquad\qquad\qquad+\Gamma(f^{\vpi}\chi_{\de_0}(\<v\>^lf^\vpi),\mu^{\frac{1}{2}})\quad \text{ in } [T_1,T_2]\times\Omega\times\R^3_v, \\
		& f^{\vpi}|_{\Si_-}=g\quad \text{ on }[T_1,T_2]\times\Si_-, \\
		& f^{\vpi}(T_1,x,v)=f_{T_1}\quad \text{ in }\Omega\times\R^3_v.
	\end{aligned}\right.
\end{align}
where $\de_0>0$ is a small constant. Here $\chi_{\de_0}(f)$ is a cutoff function given by 
\begin{align*}
	\chi_{\de_0}(f)=\left\{\begin{aligned}
		&0,\quad \text{ if }|f|>\de_0,\\
		&1,\quad \text{ if }|f|\le\de_0.
	\end{aligned}\right. 
\end{align*}
We add such a cutoff function to automatically obtain the $L^\infty$ bound when doing the approximation. 
To solve this equation \eqref{liniter1}, we let $\vpi>0$ be any small constant and $S:X\to X$ be the (weak-)solution operator of the equation:
\begin{align}\label{liniter2}
	\left\{
	\begin{aligned}
		& \pa_tf+ v\cdot\na_xf =\vpi Vf+\Gamma(\mu^{\frac{1}{2}}+\psi\chi_{\de_0}(\<v\>^l\psi),f)\\&\qquad\qquad\qquad\qquad+\Gamma(\psi\chi_{\de_0}(\<v\>^l\psi),\mu^{\frac{1}{2}})\quad \text{ in } [T_1,T_2]\times\Omega\times\R^3_v, \\
		& f|_{\Si_-}=g\quad \text{ on }[T_1,T_2]\times\Si_-, \\
		& f(T_1,x,v)=f_{T_1}\quad \text{ in }\Omega\times\R^3_v.
	\end{aligned}\right.
\end{align}
That is, for any $\psi\in X$, we set $S\psi = f$ to be the weak solution of \eqref{liniter2}, whose existence can be obtained from Theorem \ref{weakfThm}. Here $X$ is the normed space defined by 
\begin{multline}\label{Xdef}
	X:=\big\{f\in L^\infty_tL^2_{x,v}([T_1,T_2]\times\Omega\times\R^3_v)\,:\,
	\mu^{\frac{1}{2}}+f\ge 0,\ 
	\|f\|_{L^\infty_t([T_1,T_2])L^2_x(\Omega)L^2_v}^2 \le \de_0, \\
	\|\<v\>^lf\|_{L^\infty_t([T_1,T_2])L^\infty_{x,v}(\Omega\times\R^3_v)} \le \de_0\big\}, 
\end{multline}
equipped with norm $L^\infty_tL^2_{x,v}([T_1,T_2]\times\Omega\times\R^3_v)$, 
with some small $\de_0>0$ (which can be chosen in Theorems \ref{weakfThm} and \ref{LinftyLinear}). 
Next, we derive the local-in-time existence of the nonlinear equation by using the contraction mapping theorem and letting $\vpi\to 0$ with uniform estimates. 
\begin{Thm}[Local-in-time existence of nonlinear equation]\label{nonLocal}
	Fix $l\ge\ga+10$. 
	Let $\de_0\in(0,1)$ be a sufficiently small constant depending on $\ga,s$; and $l_0$ be a large constant depending on $l,\ga,s$ (both can be chosen in Theorem \ref{LinftyLinear}).
	Suppose $f_{T_1}$ and $g$ satisfy $F_0=\mu+\mu^{\frac{1}{2}}f_{T_1}\ge 0$ and 
	\begin{align}\begin{aligned}
			\label{ve0}
			\|\<v\>^lg\|_{L^\infty_{t,x,v}([T_1,T_2]\times\Si_-)}+\|\<v\>^lf_{T_1}\|_{L^\infty_{x,v}(\Omega\times\R^3_v)}& \le\ve_\infty, \\
			\|\<v\>^{l-2}g\|^2_{L^2_t([T_1,T_2])L^2_{x,v}(\Si_-)}+\|\<v\>^{l-2}f_{T_1}\|^2_{L^2_x(\Omega)L^2_v(\R^3_v)}& \le\ve_1,\\
			\|\<v\>^{l_0}g\|^2_{L^2_t([T_1,T_2])L^2_{x,v}(\Si_-)}+\|\<v\>^{l_0}f_{T_1}\|^2_{L^2_x(\Omega)L^2_v(\R^3_v)}&=\ti C_1, 
		\end{aligned}
	\end{align}
	with some constant $\ti C_1>0$, and a sufficiently small $0<\ve_\infty,\ve_1<1$.
	Then there exists a small time $T_2>T_1$ and a solution $f$ to the equation
	\begin{align}\label{non1}
		\left\{
		\begin{aligned}
			& \pa_tf+ v\cdot\na_xf = \Gamma(\mu^{\frac{1}{2}}+f,f)+\Gamma(f,\mu^{\frac{1}{2}})\quad \text{ in } [T_1,T_2]\times\Omega\times\R^3_v, \\
			& f|_{\Si_-}=g\qquad\quad \text{ on }[T_1,T_2]\times\Si_-, \\
			& f(T_1,x,v)=f_{T_1}\quad \text{ in }\Omega\times\R^3_v,
		\end{aligned}\right.
	\end{align}
	satisfying $F=\mu+\mu^{\frac{1}{2}}f\ge 0$, and 
	\begin{multline}\label{84cz}
		\|\<v\>^{k}f\|^2_{L^\infty_t([T_1,T_2])L^2_x(\Omega)L^2_v}
		+\|\<v\>^kf\|^2_{L^2_t([T_1,T_2])L^2_{x,v}(\Si_+)}
		+c_0\|\<v\>^kf\|_{L^2_t([T_1,T_2])L^2_x(\Omega)L^2_D}^2\\
		\le 2e^{C(T_2-T_1)}\|\<v\>^{k}f_{T_1}\|^2_{L^2_x(\Omega)L^2_v}+2e^{C(T_2-T_1)}\|\<v\>^kg\|^2_{L^2_t([T_1,T_2])L^2_{x,v}(\Si_-)},
	\end{multline}
	for any $k\ge 0$, 
	and
	\begin{align}\label{84bz}
		\|\<v\>^lf\|_{L^\infty_t([T_1,T_2])L^\infty_{x,v}(\ol\Omega\times\R^3_v)}\le \de_0, 
	\end{align}
	for some constant $C=C(k,\ga,s)>0$ that is independent of the time $T_1,T_2$.
\end{Thm}
Note that $\de_0>0$ is a given small constant, and $\ve_\infty,\ve_1>0$ are smaller constants, which implies that the solution obtained in this Theorem \ref{nonLocal} is not automatically global-in-time.
\begin{proof}
	We consider the fixed-point theorem for equations \eqref{liniter2} with the cases $s\in(0,\frac{1}{2})$ and $s\in[\frac{1}{2},1)$ in the first and second steps, respectively. Once we obtain the solution to the nonlinear equation, we pass the limit $\vpi\to0$ in the third step. 
	
	\smallskip\noindent{\bf Step 1. Contraction mapping for weak singularity.}
	Let $s\in(0,\frac{1}{2})$. Choose any sufficiently small $\vpi>0$ (which is less than a constant depending on $l$ that is chosen in Lemma \ref{LinftyLinear}). Then we let $S:X\to X$ be the weak solution operator of the equation \eqref{liniter2} by setting $S\psi = f$, whose existence is guaranteed by Theorem \ref{weakfThm}, with $\psi\equiv\vp$ replaced by $\psi\chi_{\de_0}(\<v\>^l\psi)$ therein. Here $X$ is the normed space given by \eqref{Xdef}. 
	Note that the space $X$ depends on $T_1,T_2$. 
	We next prove that $S:X\to X$ is a contraction mapping. 
	
	\smallskip For any $\psi\in X$, we begin by proving that $S\psi\in X$. Since 
	\begin{align*}
		\|\psi\|_{L^\infty_t([T_1,T_2])L^2_x(\Omega)L^2_v}^2 \le \de_0, 
		\quad 
		\|\<v\>^l\psi\|_{L^\infty_t([T_1,T_2])L^\infty_{x,v}(\Omega\times\R^3_v)} \le \de_0, 
	\end{align*}
	with some small constant $\de_0>0$ chosen in Theorems \ref{weakfThm} and \ref{LinftyLinear}. Using the assumption \eqref{ve0}, we can apply Theorems \ref{weakfThm} and \ref{LinftyLinear}, i.e. \eqref{fLineares1} and \eqref{Linftyx}, with $\psi=\vp=\vp_1$ replaced by $\psi\chi_{\de_0}(\<v\>^l\psi)$ and $\vp_2$ replaced by $0$ therein to obtain that the weak solution $f=S\psi$ to equation \eqref{liniter2} satisfies estimates
	\begin{multline*}
		\|f\|_{L^\infty_t([T_1,T_2])L^2_x(\Omega)L^2_v}^2+\|f\|^2_{L^2_t([T_1,T_2])L^2_{x,v}(\Si_+)}
		+c_0\|f\|_{L^2_t([T_1,T_2])L^2_x(\Omega)L^2_D}^2\\
		+\vpi\|[\wh{C}^{}_0\<v\>^{{4}}f,\<v\>^{{2}}\na_vf]\|^2_{L^2_t([T_1,T_2])L^2_x(\Omega)L^2_v}
		\le e^{C(T_2-T_1)}\big((T_2-T_1)\de_0+\ve_1\big),
	\end{multline*}
	and
	\begin{multline}\label{799a}
		\|\<v\>^lf\|_{L^\infty_t([T_1,T_2])L^\infty_{x,v}(\ol\Omega\times\R^3_v)}
		\le\max\big\{\frac{1}{2}\de_0,\,\ve_\infty\big\}\\
		+C\big(1+\ti C_1+\max\big\{\frac{1}{2}\de_0,\ve_\infty\big\}\big)^{C}\Big((T_2-T_1)e^{C}(\ve_1+\de_0)\Big)^{\zeta},
	\end{multline}
	where $C=C(l,\ga,s)>0$ is independent of $T_1,T_2$.
	Then we choose $T_2-T_1,\ve_\infty,\ve_1>0$ so small that 
	$(T_2-T_1)e^{C(T_2-T_1)}<\frac{1}{4}$, $e^{C(T_2-T_1)}\ve_1<\frac{\de_0}{4}$, and $\ve_\infty<\frac{\de_0}{2}$, $(T_2-T_1)e^{C}\big(\ve_1+1\big)\le \big(\frac{1}{2C(2+\ti C_1)^C}\de_0\big)^{\frac{1}{\zeta}}$ 
	and deduce
	\begin{align*}
		\begin{aligned}
			\|f\|_{L^\infty_t([T_1,T_2])L^2_x(\Omega)L^2_v}^2+\|f\|^2_{L^2_t([T_1,T_2])L^2_{x,v}(\Si_+)}+c_0\|f\|_{L^2_t([T_1,T_2])L^2_x(\Omega)L^2_D}^2 & \le \de_0, \\
			\|\<v\>^lf\|_{L^\infty_t([T_1,T_2])L^\infty_{x,v}(\ol\Omega\times\R^3_v)}
			\le \frac{1}{2}\de_0+C(2+\ti C_1)^C\Big((T_2-T_1)e^{C}\big(\ve_1+1\big)\Big)^{\zeta}& \le \de_0.
		\end{aligned}
	\end{align*}
	The non-negativity of $\mu^{\frac{1}{2}}+f$ can be derived from Theorem \ref{positivity}.
	These facts imply $f=S\psi\in X$. 
	
	\smallskip Next we prove that $S$ is a contraction map with small time $T_2-T_1>0$. Let $\psi,\vp\in X$ and $f=S\psi$, $h=S\vp$. Then, by \eqref{liniter2}, $f-h$ satisfies 
	\begin{align}\label{liniter3}
		\left\{
		\begin{aligned}
			& \pa_t(f-h)+ v\cdot\na_x(f-h) =\vpi V(f-h)+\Gamma(\mu^{\frac{1}{2}}+\psi\chi_{\de_0}(\<v\>^l\psi),f-h)
			\\&\qquad\qquad\qquad+\Gamma(\psi\chi_{\de_0}(\<v\>^l\psi)-\vp\chi_{\de_0}(\<v\>^l\vp),\mu^{\frac{1}{2}}+h)\quad \text{ in } [T_1,T_2]\times\Omega\times\R^3_v, \\
			& (f-h)|_{\Si_-}=0\quad \text{ on }[T_1,T_2]\times\Si_-, \\
			& (f-h)(T_1,x,v)=0\quad \text{ in }\Omega\times\R^3_v.
		\end{aligned}\right.
	\end{align}
	Taking $L^2$ inner product of \eqref{liniter3} with $f-h$ over $\Omega\times\R^3_v$ and noticing $\psi,\vp\in X$ has $L^\infty$ bound $\de_0$, we have 
	\begin{multline}\label{666}
		\frac{1}{2}\pa_t\|f-h\|_{L^2_x(\Omega)L^2_v}^2+\frac{1}{2}\|f-h\|^2_{L^2_{x,v}(\Si_+)}\\
		=\Big(\vpi V(f-h)+\Gamma(\mu^{\frac{1}{2}}+\psi,f-h)+\Gamma(\psi-\vp,h)+\Gamma(\psi-\vp,\mu^{\frac{1}{2}}),f-h\Big)_{L^2_x(\Omega)L^2_v}. 
	\end{multline}
	Applying Lemmas \ref{vpmu12Lem} and \ref{LemRegu} and estimate \eqref{56} for the right-hand side of \eqref{666}, we obtain 
	\begin{align*}
		&\frac{1}{2}\pa_t\|f-h\|_{L^2_x(\Omega)L^2_v}^2+\frac{1}{2}\|f-h\|^2_{L^2_{x,v}(\Si_+)}
		+\frac{\vpi}{C}\|\<v\>^{{2}}\<D_v\>(f-h)\|_{L^2_x(\Omega)L^2_v}^2\\
		&\quad\le \big(-c_0+C\|\<v\>^{4}\psi\|_{L^\infty_x(\Omega)L^\infty_v}\big)\|f-h\|_{L^2_x(\Omega)L^2_D}^2
		+C\|\1_{|v|\le R_0}(f-h)\|_{L^2_x(\Omega)L^2_v}^2\\
		&\qquad+C\|\psi-\vp\|_{L^2_x(\Omega)L^2_v}\|\<v\>^{2}h\|_{L^\infty_x(\Omega)L^2_v}\|\<v\>^{2}(f-h)\|_{L^2_x(\Omega)H^{2s}_v}\\
		&\qquad+C\|\mu^{\frac{1}{10^4}}(\psi-\vp)\|_{L^2_x(\Omega)L^2_v}\|\mu^{\frac{1}{10^4}}(f-h)\|_{L^2_x(\Omega)L^2_v}\\
		&\quad\le
		-\frac{c_0}{2}\|f-h\|_{L^2_x(\Omega)L^2_D}^2
		+C_\vpi\big(\|\<v\>^{2}h\|_{L^\infty_x(\Omega)L^\infty_v}^2+1\big)\|\psi-\vp\|_{L^2_x(\Omega)L^2_v}^2\\
		&\qquad+\frac{\vpi}{2C}\|\<v\>^{2}\<D_v\>(f-h)\|^2_{L^2_x(\Omega)L^{2}_v}+C\|f-h\|^2_{L^2_x(\Omega)L^2_v}, 
	\end{align*}
	since $s\in(0,\frac{1}{2})$. 
	Here we choose $\de_0>0$ in \eqref{Xdef} sufficiently small. Thus, we obtain 
	\begin{align*}
		\frac{1}{2}\pa_t\|f-h\|_{L^2_x(\Omega)L^2_v}^2\le C_\vpi\|\psi-\vp\|_{L^2_x(\Omega)L^2_v}^2+C\|f-h\|^2_{L^2_x(\Omega)L^2_v}.
	\end{align*}
	Using Gr\"{o}nwall's inequality and choosing $T_2=T_2(\vpi)>T_1$ sufficiently small, we have 
	\begin{align*}
		\|f-h\|_{L^\infty_t([T_1,T_2])L^2_x(\Omega)L^2_v}^2&\le (T_2-T_1)C_\vpi\|\psi-\vp\|_{L^\infty_t([T_1,T_2])L^2_x(\Omega)L^2_v}^2\\
		&\le \frac{1}{2}\|\psi-\vp\|_{L^\infty_t([T_1,T_2])L^2_x(\Omega)L^2_v}^2.
	\end{align*}
	This implies that $S:X\to X$ is a contraction map. Therefore, by Banach fixed point theorem, there exists $f=f^\vpi\in X$ such that 
	\begin{align*}
		\|f^\vpi\|_{L^\infty_t([T_1,T_2])L^2_x(\Omega)L^2_v}^2 \le \de_0, 
		\quad
		\|\<v\>^lf^\vpi\|_{L^\infty_t([T_1,T_2])L^\infty_{x,v}(\Omega\times\R^3_v)} \le \de_0,
	\end{align*}
	and it solves equation \eqref{liniter1} in the weak sense of \eqref{weakf}.

	\smallskip \noindent{\bf Step 2. Strong Singularity.} Let $s\in[\frac{1}{2},1)$ in this step. As in Subsection \ref{Sec1non}, we truncate the collision kernel $b(\cos\th)$ as
	\begin{align*}
		b_\eta(\cos\th):=\frac{b(\cos\th)\th^{2+2s}}{\th^{2+2s_*}(\th+\eta)^{2s-2s_*}}\approx \th^{-2-2s_*}, 
	\end{align*}
	with some fixed $s_*\in(0,\frac{1}{2})$. We also denote $\Ga_\eta$ by \eqref{Gaeta}. Since $b_\eta$ has weak singularity, we can apply the fixed-point arguments in Step 1 to obtain a small time $T_2=T_2(\vpi,\eta)>T_1$ and a weak solution $f_\eta$ to equation \eqref{liniter1} with $\Ga$ replaced by $\Ga_\eta$, i.e. $f_\eta$ satisfies $f_\eta|_{\Si_-}=g$ on $[T_1,T_2]\times\Si_-$ and $f_\eta(T_1,x,v)=f_{T_1}$ in $\Omega\times\R^3_v$, and solves 
	\begin{align}\label{liniter1aa}
		\begin{aligned}
			\pa_tf_\eta+ v\cdot\na_xf_\eta &=\vpi Vf_\eta+\Gamma_\eta(\mu^{\frac{1}{2}}+f_\eta\chi_{\de_0}(\<v\>^lf_\eta),f_\eta)\\&\qquad+\Gamma_\eta(f_\eta\chi_{\de_0}(\<v\>^lf_\eta),\mu^{\frac{1}{2}})\quad \text{ in } [T_1,T_2]\times\Omega\times\R^3_v.
		\end{aligned}
	\end{align}
	Taking $L^2$ inner product of \eqref{liniter1aa} with $\<v\>^{2k}f_\eta$ over $[T_1,T_2]\times\Omega\times\R^3_v$ with any $k\ge 0$, and using Lemma \ref{LemRegu} for regularizing term $Vf$ and Lemma \ref{GaetaesLem} for the collision terms, we obtain 
	\begin{align*}
		&\pa_t\|\<v\>^{k}f_\eta(t)\|_{L^2_x(\Omega)L^2_v}^2+\|\<v\>^{k}f_\eta\|^2_{L^2_{x,v}(\Si_+)}
		+\vpi\|[\wh{C}^{}_0\<v\>^{{k+4}}f_\eta,\<v\>^{{k+2}}\na_vf_\eta]\|_{L^2_x(\Omega)L^2_v}^2\\
		&\quad\le \|\<v\>^kg\|^2_{L^2_{x,v}(\Si_-)}
		+C(1+\de_0)\|\<v\>^{k}f_\eta\|_{L^2_{x}(\Omega)L^2_D}^2+C\|f_\eta\|_{L^2_{x}(\Omega)L^2_v}^2\\
		&\quad\le \|\<v\>^kg\|^2_{L^2_{x,v}(\Si_-)}
		+\frac{\vpi}{2}\|\<v\>^{k+2}f_\eta\|^2_{L^2_{x}(\Omega)H^1_v}+C_\vpi\|f_\eta\|_{L^2_{x}(\Omega)L^2_v}^2,
	\end{align*}
	where we used \eqref{esD}, $s\in[\frac{1}{2},1)$ and interpolation to deduce
	\begin{align}\label{kesD}
		\|\<v\>^{k}f\|^2_{L^2_D}\le C\|\<v\>^{k+\frac{\ga+2s}{2}}\<D_v\>^sf\|^2_{L^2_v}\le\frac{\vpi}{2C}\|\<v\>^{k+2}f\|^2_{H^1_v}+C_\vpi\|f\|^2_{L^2_v}.
	\end{align}
	The term $H^1_v$ can now be absorbed by the regularizing term. Therefore, integrating over $[T_1,T_2]$ and choosing $T_2=T_2(\vpi)>T_1$ sufficiently small, we have 
	\begin{multline}\label{713}
		\|\<v\>^{k}f_\eta\|_{L^\infty_tL^2_x(\Omega)L^2_v}^2+\|\<v\>^{k}f_\eta\|^2_{L^2_tL^2_{x,v}(\Si_+)}
		+\vpi\|[\wh{C}^{}_0\<v\>^{{k+4}}f_\eta,\<v\>^{{k+2}}\na_vf_\eta]\|_{L^2_tL^2_x(\Omega)L^2_v}^2\\
		\le 2\|\<v\>^{k}f_{T_1}\|_{L^2_x(\Omega)L^2_v}^2+2\|\<v\>^{k}f_\eta\|^2_{L^2_t([T_1,T_2])L^2_{x,v}(\Si_-)},
	\end{multline}
	which is uniform in $\eta$. This implies that the solution $f_\eta$ can be extended to a time $T_2=T_2(\vpi)>T_1$ which is independent of $\eta$. 
	
	\smallskip 
	For the $L^\infty_{t,x,v}$ estimate of $f_\eta$, we give a short proof for brevity; see also \cite[Section 7]{Alonso2022} or \cite[Section 8]{Cao2022b}. The main step to obtain the $L^\infty$ estimate is the estimate of level functions. In Lemma \ref{Qes1Lem}, the estimate is also valid for $\Ga_\eta$ replacing $\Ga$, where the constant is independent of $\eta$. In Lemma \ref{CollLevelLem}, we need to estimate the term $\int_{\Omega\times\R^3}f^{(l)}_{K,+}\Gamma_\eta\big(\Psi,f-K\<v\>^{-l}_\de\big)\,dvdx$, while the estimates for other terms remain the same. In fact, using \eqref{Gaetaes}, \eqref{kesD} and $\Psi\ge 0$, we have 
	\begin{align*}
		&\int_{\Omega\times\R^3}f^{(l)}_{K,+}\Gamma_\eta\big(\Psi,f-K\<v\>^{-l}_\de\big)\,dvdx\\
		&\quad\le \int_{\Omega\times\R^3}f^{(l)}_{K,+}\Gamma_\eta\big(\Psi,f^{(l)}_{K,+}\big)\,dvdx\\
		&\quad \le C\|\<v\>^4\Psi\|_{L^\infty_x(\Omega)L^\infty_v}\|f^{(l)}_{K,+}\|_{L^2_x(\Omega)L^2_D}^2\\
		&\quad\le \frac{\vpi}{2C}\|\<v\>^{k+2}f^{(l)}_{K,+}\|^2_{L^2_x(\Omega)H^1_v}+C_\vpi\|\<v\>^4\Psi\|_{L^\infty_x(\Omega)L^\infty_v}^2\|f^{(l)}_{K,+}\|^2_{L^2_x(\Omega)L^2_v}. 
	\end{align*}
	The $H^1_v$ term can be absorbed by the regularizing term in \eqref{651a}. Therefore, following the same calculations in Theorem \ref{LinftyLinear} (i.e. all the calculations in Section \ref{SecLinfty}), and using the energy functional 
	\begin{multline}\label{Eppri}
		\E'_p(K):=\|f^{(l)}_{K,+}\|^2_{L^\infty_tL^2_{x,v}([T_1,T_2]\times\R^6_{x,v})}+\vpi\|\<v\>^2f^{(l)}_{K,+}\|_{L^2_tL^2_xH^1_v([T_1,T_2]\times\Omega\times\R^3_v)}^2\\
		+\frac{1}{C_0\max\{C_\infty^{2p-2},1\}}\Big\|\int_{\R^3_v}\1_{[T_1,T_2]}\<v\>^{-10}(f^{(l)}_{K,+})^2\,dv\Big\|_{B^{s',2}_p(\R^{4}_{t,x})}^p
	\end{multline}
	instead of $\E_p(K)$ defined in \eqref{Ep} (we change the dissipation norm $\|f\|_{L^2_D}$ to a more regularized term but with a small constant $\vpi$), we can obtain the $L^\infty$ estimate of $f_\eta$ as in \eqref{799a}:
	\begin{multline*}
		\|\<v\>^lf_\eta\|_{L^\infty_t([T_1,T_2])L^\infty_{x,v}(\ol\Omega\times\R^3_v)}
		\le
		\max\big\{\frac{1}{2}\de_0,\,\ve_\infty\big\}\\
		+C\big(1+\ti C_1+\max\big\{\frac{1}{2}\de_0,\ve_\infty\big\}\big)^{C}\Big((T_2-T_1)e^{C}(\ve_1+\de_0)\Big)^{\zeta}, 
	\end{multline*}
	where $C=C(l,\ga,s,\vpi)>0$ is independent of $T_1,T_2$, and $\zeta=\zeta(s)>0$. 
	Note the constant $C>0$ depends on $\vpi>0$. Then we choose $T_2=T_2(\vpi)>T_1$ and $\ve_\infty,\ve_1>0$ so small that $\ve_\infty<\frac{\de_0}{2}$ and $(T_2-T_1)e^{C}\big(\ve_1+1\big)\le \big(\frac{1}{2C(2+\ti C_1)^C}\de_0\big)^{\frac{1}{\zeta}}$ to deduce
	\begin{align}\label{713a}
		\|\<v\>^lf_\eta\|_{L^\infty_t([T_1,T_2])L^\infty_{x,v}(\ol\Omega\times\R^3_v)}\le \de_0. 
	\end{align}

	\smallskip Therefore, applying Banach-Alaoglu Theorem, $f_\eta$ is weakly-$*$ compact in the corresponding spaces in \eqref{713} and \eqref{713a}, and there exists a subsequence (still denote it by $f_\eta$) such that 
	\begin{align}\label{weaklimiteta}
		\begin{aligned}
			& f_\eta\rightharpoonup f\quad
			\text{ weakly-$*$ in $L^2_{t,x,v}([T_1,T_2]\times\Si_+)$ and $L^2_{t,x}H^1_v([T_1,T_2]\times\Omega\times\R^3_v)$}, \\
			& f_\eta\rightharpoonup f\quad \text{ weakly-$*$ in $L^2_x(\Omega)L^2_v$ for any $t\in[T_1,T_2]$},\\
			& f_\eta\rightharpoonup f\quad
			\text{ weakly-$*$ in $L^\infty_{t,x,v}([T_1,T_2]\times\Omega\times\R^3_v)$},
		\end{aligned}
	\end{align}
	as $\eta\to 0$, with some function $f$ satisfying 
	\begin{multline}\label{713ab}
		\|\<v\>^{k}f\|_{L^\infty_tL^2_x(\Omega)L^2_v}^2+\|\<v\>^{k}f\|^2_{L^2_tL^2_{x,v}(\Si_+)}
		+\vpi\|[\wh{C}^{}_0\<v\>^{{k+4}}f,\<v\>^{{k+2}}\na_vf]\|_{L^2_tL^2_x(\Omega)L^2_v}^2\\
		\le 2\|\<v\>^{k}f_{T_1}\|_{L^2_x(\Omega)L^2_v}^2+2\|\<v\>^{k}g\|^2_{L^2_tL^2_{x,v}(\Si_-)},
	\end{multline}
	and 
	\begin{align}\label{713b}
		\|\<v\>^lf\|_{L^\infty_t([T_1,T_2])L^\infty_{x,v}(\ol\Omega\times\R^3_v)}\le \de_0. 
	\end{align}
	Rewriting equation \eqref{liniter1aa} in the weak form: for any function $\Phi\in C^\infty_c(\R_t\times\R^3_x\times\R^3_v)$,
	\begin{multline}\label{weaketa}
		(f_\eta(T_2),\Phi(T_2))_{L^2_{x}(\Omega)L^2_v}-(f_\eta,(\pa_t+v\cdot\na_x)\Phi)_{L^2_{t}([T_1,T_2])L^2_{x}(\Omega)L^2_v}+(f_\eta,\Phi)_{L^2_{t}([T_1,T_2])L^2_{x,v}(\Si_+)}\\
		=(f_{T_1},\Phi(T_1))_{L^2_{x}(\Omega)L^2_v}+(g,\Phi)_{L^2_{t}([T_1,T_2])L^2_{x,v}(\Si_-)}\\
		+\big(\vpi Vf_\eta+\Gamma_\eta(\mu^{\frac{1}{2}}+f_\eta\chi_{\de_0}(\<v\>^lf_\eta),f_\eta)+\Gamma_\eta(f_\eta\chi_{\de_0}(\<v\>^lf_\eta),\mu^{\frac{1}{2}}),\Phi\big)_{L^2_{t}([T_1,T_2])L^2_{x}(\Omega)L^2_v}.
	\end{multline}
	It suffices to obtain the limit for the collision terms. Using the upper bounds \eqref{713}, \eqref{713a}, \eqref{713ab}, \eqref{713b}, and estimate \eqref{Gaetaes1} for collision terms, we have 
	\begin{align}\label{collGaetaGa}\notag
		&\big(\Gamma_\eta(\mu^{\frac{1}{2}}+f_\eta\chi_{\de_0}(\<v\>^lf_\eta),\mu^{\frac{1}{2}}+f_\eta)-\Gamma(\mu^{\frac{1}{2}}+f\chi_{\de_0}(\<v\>^lf),\mu^{\frac{1}{2}}+f),\Phi\big)_{L^2_{t,x,v}([T_1,T_2]\times\Omega\times\R^3_v)}\\
		&\notag\quad=
		\big(\Gamma_\eta(f_\eta-f,\mu^{\frac{1}{2}}+f_\eta)
		+\Gamma_\eta(\mu^{\frac{1}{2}}+f,f_\eta-f)\\&\notag\quad\qquad
		+\Gamma_\eta(\mu^{\frac{1}{2}}+f,\mu^{\frac{1}{2}}+f)
		-\Gamma(\mu^{\frac{1}{2}}+f,\mu^{\frac{1}{2}}+f),\Phi\big)_{L^2_{t,x,v}([T_1,T_2]\times\Omega\times\R^3_v)}\\
		&\notag\quad\le C\int_{[T_1,T_2]\times\Omega}\|\Phi\|_{W^{2,\infty}_v}\|\<v\>^{\ga+4}(f_\eta-f)\|_{L^2_v}(1+\|\<v\>^{\ga+4}f_\eta\|_{L^\infty_v}+\|\<v\>^{\ga+4}f\|_{L^\infty_v})\,dxdt\\
		&\qquad\quad
		+\big(\Gamma_\eta(\mu^{\frac{1}{2}}+f,\mu^{\frac{1}{2}}+f)
		-\Gamma(\mu^{\frac{1}{2}}+f,\mu^{\frac{1}{2}}+f),\Phi\big)_{L^2_{t,x,v}([T_1,T_2]\times\Omega\times\R^3_v)}. 
	\end{align}
	Applying Lemma \ref{claimLem} to the first right-hand term of \eqref{collGaetaGa} and \eqref{Gaetalimit} to the second right-hand term, we obtain the limit: 
	\begin{align*}
		\lim_{\eta\to0}\big(\Gamma_\eta(\mu^{\frac{1}{2}}+f_\eta\chi_{\de_0}(\<v\>^lf_\eta),\mu^{\frac{1}{2}}+f_\eta)-\Gamma(\mu^{\frac{1}{2}}+f\chi_{\de_0}(\<v\>^lf),\mu^{\frac{1}{2}}+f),\Phi\big)_{L^2_{t,x,v}([T_1,T_2]\times\Omega\times\R^3_v)}=0. 
	\end{align*}
	Combining this with the weak-$*$ limit in \eqref{weaklimiteta}, we can take $\eta\to0$ in \eqref{weaketa} to deduce that $f$ is the weak solution to \eqref{liniter1} for the case $s\in[\frac{1}{2},1)$ in the sense of \eqref{weakf}.

	\smallskip\noindent{\bf Step 3. Convergence of $f^{\vpi}$.}
	Let $s\in(0,1)$. Notice that the solution $f^\vpi$ to \eqref{liniter1} obtained in Steps 1 and 2 only exists for a small time $T=T(\vpi)>0$ since we utilize the regularizing term $\vpi Vf$. In this Step, without utilizing the regularity from $\vpi Vf$, we will prove that the existence time $T>0$ can be independent of $\vpi>0$ and then one can pass the limit $\vpi\to 0$. 
	
	\smallskip Since the term $f^{\vpi}\chi_{\de_0}(\<v\>^lf^\vpi)$ satisfies 
	\begin{align*}
		\|\<v\>^lf^{\vpi}\chi_{\de_0}(\<v\>^lf^\vpi)\|_{L^\infty_t([T_1,T_2])L^\infty_{x,v}(\Omega\times\R^3_v)} \le \de_0, 
	\end{align*}
	we can apply Theorems \ref{weakfThm} and \ref{LinftyLinear}, i.e. \eqref{fLineares1} and \eqref{Linftyx}, with $\psi=\vp=\vp_1=f^{\vpi}\chi_{\de_0}(\<v\>^lf^\vpi)$ and $\vp_2=\phi=0$ to obtain that the solution $f^{\vpi}$ to equation \eqref{liniter1} satisfies 
	\begin{multline*}
		\|f^{\vpi}\|_{L^\infty_t([T_1,T_2])L^2_x(\Omega)L^2_v}^2+\|f^{\vpi}\|^2_{L^2_t([T_1,T_2])L^2_{x,v}(\Si_+)}
		+c_0\|f^{\vpi}\|_{L^2_t([T_1,T_2])L^2_x(\Omega)L^2_D}^2\\
		+\vpi\|[\wh{C}^{}_0\<v\>^{{4}}f^{\vpi},\<v\>^{{2}}\na_vf^{\vpi}]\|_{L^2_t([T_1,T_2])L^2_x(\Omega)L^2_v}^2
		\le e^{C(T_2-T_1)}\big((T_2-T_1)\de_0+\ve_1\big),
	\end{multline*}
	and, as in \eqref{799a}, 
	\begin{multline*}
		\|\<v\>^lf^{\vpi}\|_{L^\infty_t([T_1,T_2])L^\infty_{x,v}(\ol\Omega\times\R^3_v)}
		\le\max\big\{\frac{1}{2}\de_0,\,\ve_\infty\big\}\\
		+C\big(1+\ti C_1+\max\big\{\frac{\de_0}{2},\ve_\infty\big\}\big)^{C}\Big((T_2-T_1)e^{C}\big(\ve_1+\de_0\big)\Big)^{\zeta},
	\end{multline*}
	where $C=C(l,\ga,s)>0$ is independent of $T_1,T_2$, and $\zeta=\zeta(s)>0$. 
	Then we choose $T_2,\ve_\infty,\ve_1>0$ so small that $(T_2-T_1)e^{C(T_2-T_1)}<\frac{1}{4}$, $e^{C(T_2-T_1)}\ve_1<\frac{\de_0}{4}$, $\ve_\infty<\frac{\de_0}{2}$ and $(T_2-T_1)e^{C}\big(\ve_1+1\big)\le \big(\frac{1}{2C(2+\ti C_1)^C}\de_0\big)^{\frac{1}{\zeta}}$ to deduce 
	\begin{multline}\label{fn2}
		\|f^{\vpi}\|_{L^\infty_t([T_1,T_2])L^2_x(\Omega)L^2_v}^2+\|f^{\vpi}\|^2_{L^2_t([T_1,T_2])L^2_{x,v}(\Si_+)}
		+c_0\|f^{\vpi}\|_{L^2_t([T_1,T_2])L^2_x(\Omega)L^2_D}^2\\
		+\vpi\|[\wh{C}^{}_0\<v\>^{{4}}f^{\vpi},\<v\>^{{2}}\na_vf^{\vpi}]\|_{L^2_t([T_1,T_2])L^2_x(\Omega)L^2_v}^2\le\de_0, 
	\end{multline}
	and
	\begin{align}\label{fn3}
		\|\<v\>^lf^{\vpi}\|_{L^\infty_t([T_1,T_2])L^\infty_{x,v}(\ol\Omega\times\R^3_v)}
		\le \frac{1}{2}\de_0+C(2+\ti C_1)^C\Big((T_2-T_1)e^{C}\big(\ve_1+1\big)\Big)^{\zeta}& \le \de_0.
	\end{align}
	Note that the choice of $T_2-T_1,\ve_\infty,\ve_1>0$ here is independent of $\vpi>0$, and thus the existence time $T>0$ is independent of $\vpi>0$ (one can prove this by standard continuity argument).
	Therefore, the sequence $\{f^{\vpi}\}$ is bounded in the sense in \eqref{fn2} and \eqref{fn3}. By Banach-Alaoglu Theorem, there exists a subsequence $\{f^n\}\subset\{f^{\vpi}\}$ (for simplicity we can take $\vpi=\frac{1}{n}$) such that $\{f^n\}$ has a weak-$*$ limit $f$ as $n\to\infty$ satisfying
	\begin{align}\label{84c}
		\begin{aligned}
			&\|f\|_{L^\infty_t([T_1,T_2])L^2_x(\Omega)L^2_v}^2+\|f\|^2_{L^2_t([T_1,T_2])L^2_{x,v}(\Si_+)}
			+c_0\|f\|_{L^2_t([T_1,T_2])L^2_x(\Omega)L^2_D}^2
			\le \de_0,\\
			& \|\<v\>^lf\|_{L^\infty_t([T_1,T_2])L^\infty_{x,v}(\Omega\times\R^3_v)}\le \de_0,
		\end{aligned}
	\end{align}
	in the sense that
	\begin{align}\label{weaklimit}
		\begin{aligned}
			& f^n\rightharpoonup f\quad
			\text{ weakly-$*$ in $L^2_{t,x,v}([T_1,T_2]\times\Si_+)$ and $L^\infty_{t,x,v}([T_1,T_2]\times\Omega\times\R^3_v)$}, \\
			& f^n\rightharpoonup f\quad \text{ weakly-$*$ in $L^2_x(\Omega)L^2_v$ for any $t\in[T_1,T_2]$}.
		\end{aligned}
	\end{align}
	Notice from \eqref{fn3} that $f^n\chi_{\de_0}(\<v\>^lf^n)=f^n$. 
	We then write the solution $f^{n}$ to equation \eqref{liniter1} in the weak sense. That is, for any function $\Phi\in C^\infty_c(\R_t\times\R^3_x\times\R^3_v)$,
	\begin{multline}\label{weak1}
		(f^{n}(T_2),\Phi(T_2))_{L^2_x(\Omega)L^2_v}-(f^{n},(\pa_t+v\cdot\na_x)\Phi)_{L^2_{t,x,v}([T_1,T_2]\times\Omega\times\R^3_v)}
		+(f^{n},\Phi)_{L^2_{t}([T_1,T_2])L^2_{x,v}(\Si_+)}\\
		=(f_{T_1},\Phi(T_1))_{L^2_x(\Omega)L^2_v}+(g,\Phi)_{L^2_{t}([T_1,T_2])L^2_{x,v}(\Si_-)}\\
		+\Big(\frac{1}{n}Vf^{n}+\Gamma(\mu^{\frac{1}{2}}+f^n,f^{n})+\Gamma(f^n,\mu^{\frac{1}{2}}),\Phi\Big)_{L^2_{t,x,v}([T_1,T_2]\times\Omega\times\R^3_v)}.
	\end{multline}
	To take the weak-$*$ limit in \eqref{weak1}, it remains to find the limit of the regularizing term $V$ and collision terms $\Ga$. In fact, by definition of $V$ \eqref{Vf} and estimate \eqref{fn2}, we have 
	\begin{align}\label{716}
		\frac{1}{n}\big(Vf^{n},\Phi\big)_{L^2_{t,x,v}([T_1,T_2]\times\Omega\times\R^3_v)}
		\le-\frac{1}{n}\int
		\big(\wh{C}^2_0\<v\>^{{4}}|f^{n}||\Phi|+\<v\>^{{2}}|\na_vf^{n}||\na_v\Phi|\big)\,dvdxdt\to 0, 
	\end{align}
	as $n\to\infty$. 
	On the other hand, by \eqref{Gammavp}, we have 
	\begin{align}\label{86}
		&\notag\big|\big(\Gamma(\mu^{\frac{1}{2}}+f^n,f^{n})+\Gamma(f^n,\mu^{\frac{1}{2}})-\Gamma(\mu^{\frac{1}{2}}+f,f)-\Gamma(f,\mu^{\frac{1}{2}}),\Phi\big)_{L^2_{t,x,v}([T_1,T_2]\times\Omega\times\R^3_v)}\big|\\
		&\quad\notag=\big|\big(\Gamma(\mu^{\frac{1}{2}}+f,f^{n}-f)+\Gamma(f^n-f,f^{n}+\mu^{\frac{1}{2}}),\Phi\big)_{L^2_{t,x,v}([T_1,T_2]\times\Omega\times\R^3_v)}\big|\\
		&\quad\notag\le C\int^{T_2}_{T_1}\int_{\Omega}\|\Phi\|_{W^{2,\infty}_v}\Big((1+\|\<v\>^{\ga+6}f\|_{L^\infty_v})\|\<v\>^{\ga+4}(f^{n}-f)\|_{L^2_v}\\
		&\qquad\qquad\qquad\notag+\|\<v\>^{\ga+4}(f^n-f)\|_{L^2_v}(\|\<v\>^{\ga+6}f^{n}\|_{L^\infty_v}+1)\Big)\,dxdt\\
		&\quad\le C\int^{T_2}_{T_1}\int_{\Omega}\|\Phi\|_{W^{2,\infty}_v}
		\|\<v\>^{\ga+4}(f^{n}-f)\|_{L^2_v}
		\,dxdt,
	\end{align}
	where we used \eqref{fn3} and \eqref{84c} in the last inequality.
	Applying Lemma \ref{claimLem}, 
	there exists subsequence $\{f^{n_j}\}\subset\{f^{n}\}$ such that for any $\Phi\in C^\infty_c(\R^7)$,
	\begin{align}\label{86a}
		\lim_{n_j\to\infty}\int^{T_2}_{T_1}\int_{\Omega}\|\Phi\|_{W^{2,\infty}_v}\|\<v\>^{\ga+4}(f^{n_j}-f)\|_{L^2_v}\,dxdt=0.
	\end{align}
	Combining \eqref{716}, \eqref{86} and \eqref{86a}, and utilizing the weak-$*$ limit in \eqref{weaklimit}, we can take limit $n=n_j\to\infty$ in \eqref{weak1} to deduce that 
	for any $\Phi\in C^\infty_c(\R_t\times\R^3_x\times\R^3_v)$,
	\begin{multline*}
		(f(T_2),\Phi(T_2))_{L^2_x(\Omega)L^2_v}-(f,(\pa_t+v\cdot\na_x)\Phi)_{L^2_{t,x,v}([T_1,T_2]\times\Omega\times\R^3_v)}+(f,\Phi)_{L^2_{t}([T_1,T_2])L^2_{x,v}(\Si_+)}\\
		=(f_{T_1},\Phi(T_1))_{L^2_x(\Omega)L^2_v}+(g,\Phi)_{L^2_{t}([T_1,T_2])L^2_{x,v}(\Si_-)}\\
		+\big(\Gamma(\mu^{\frac{1}{2}}+f,f)+\Gamma(f,\mu^{\frac{1}{2}}),\Phi\big)_{L^2_{t,x,v}([T_1,T_2]\times\Omega\times\R^3_v)}.
	\end{multline*}
	That is, $f$ is a weak solution to \eqref{non1}. The estimates \eqref{84cz} and \eqref{84bz} follow from \eqref{fLineares1} and \eqref{84c}.
	The non-negativity of $F=\mu+\mu^{\frac{1}{2}}f$ can be derived from Theorem \ref{positivity}.
	This completes the proof of Theorem \ref{nonLocal}.
\end{proof}

	\subsection{Global nonlinear theory}
	In this subsection, we give proof of the global-in-time existence of the full nonlinear Boltzmann equation. 
	We will derive the global \emph{a priori} estimate from the Boltzmann equation so that one can repeat the existence result in the interval $[0,1]$, $[1,2]$, $\dots$, $[n,(n+1)]$ to derive the global existence. 
	\begin{Thm}[Global-in-time existence of nonlinear equation]\label{ThmInflow}
		Assume that $\Omega$ is bounded. 
		Let $-\frac{3}{2}<\ga\le 2$, $s\in(0,1)$ and fix $l\ge\ga+10$. Fix any sufficiently small $\de_0\in(0,1)$ (which can be given in Theorem \ref{LinftyLinear}), and let $l_0$ be a large constant depending on $l,s$ (which can be given in Theorem \ref{LinftyLemVanish}).
		There exists sufficiently small $\ve_\infty,\ve_1>0$ such that if $f_0$ and $g$ satisfy $F_0=\mu+\mu^{\frac{1}{2}}f_0\ge 0$ and 
		\begin{align}\begin{aligned}\label{ve0z} 
				\|\<v\>^lg\|_{L^\infty_{t,x,v}([0,\infty)\times\Si_-)}+\|\<v\>^lf_{0}\|_{L^\infty_{x,v}(\Omega\times\R^3_v)}& =\ve_\infty, \\
				\|e^{c_0t}\<v\>^{l-2}g\|^2_{L^2_t([0,\infty))L^2_{x,v}(\Si_-)}+\|\<v\>^{l-2}f_{0}\|^2_{L^2_x(\Omega)L^2_v(\R^3_v)}& =\ve_1,\\
				\|\<v\>^{l_0}g\|^2_{L^2_t([0,\infty))L^2_{x,v}(\Si_-)}+\|\<v\>^{l_0}f_{0}\|^2_{L^2_x(\Omega)L^2_v(\R^3_v)}&=\ti C, 
			\end{aligned}
		\end{align}
		for some constants $\ti C>0$ and $c_0>0$. 
		Then there exists a solution $f$ to the equation
		\begin{align}
			\label{non2}
			\left\{
			\begin{aligned}
				& \pa_tf+ v\cdot\na_xf = \Gamma(\mu^{\frac{1}{2}}+f,f)+\Gamma(f,\mu^{\frac{1}{2}})\quad \text{ in } (0,\infty)\times\Omega\times\R^3_v, \\
				& f|_{\Si_-}=g\quad \text{ on }[0,\infty)\times\Si_-, \\
				& f(0,x,v)=f_0\quad \text{ in }\Omega\times\R^3_v,
			\end{aligned}\right.
		\end{align}
		satisfying $F=\mu+\mu^{\frac{1}{2}}f\ge 0$, and for any $T\in(0,\infty)$ and $k\in[0,l_0]$, 
		\begin{multline}\label{84czz}
			\|\<v\>^kf\|^2_{L^\infty_t([0,T])L^2_x(\Omega)L^2_v}
			+\|\<v\>^kf\|^2_{L^2_t([0,T])L^2_{x,v}(\Si_+)}
			+c_0\|\<v\>^kf\|_{L^2_t([0,T])L^2_x(\Omega)L^2_D}^2\\
			+c_0\|\<v\>^kf\|_{L^2_t([0,T])L^2_x(\Omega)L^2_v}^2
			\le C\|\<v\>^kf_0\|^2_{L^2_x(\Omega)L^2_v}
			+C\|\<v\>^kg\|^2_{L^2_t([0,T])L^2_{x,v}(\Si_-)}, 
		\end{multline}
		and
		\begin{align}\label{84czy}
			\|\<v\>^lf(t)\|_{L^\infty_{t,x,v}([0,T]\times\ol\Omega\times\R^3_v)}\le\ve_\infty+C\ve_1^{\zeta}\le \de_0, 
		\end{align}
		with some constant $\zeta=\zeta(s)>0$. 
			%
	\end{Thm}
	\begin{proof}
		Since the local-in-time existence is already given in Theorem \ref{nonLocal}. we will use the \emph{a priori} arguments (for simplicity) and $L^2$--$L^\infty$ estimates to derive the global-in-time energy estimate. 
		Since $\Omega$ is a bounded domain, we can apply the global \emph{a priori} $L^2$ estimate in Section \ref{Sec12}. Using such a global $L^2$ decay, the local-in-time solution will be global-in-time. The non-negativity of $F=\mu+\mu^{\frac{1}{2}}f\ge 0$ can be derived from Theorem \ref{positivity}. 
		
		\medskip
		\noindent
		{\bf The \emph{a priori} estimate.}
		Let $n>0$ be any integer.
		Then we assume the \emph{a priori} assumption 
		\begin{align}\label{ve111}
			\|\<v\>^lf\|_{L^\infty_{t,x,v}([0,n]\times\ol\Omega\times\R^3_v)}\le \de_0.
		\end{align}
		%
		%
		Moreover, from the $L^2$ estimate in Theorem \ref{L2globalThm} (i.e. \eqref{L2es1a}), we have, for any $0\le j\le n-1$, 
		\begin{align}\label{733a}
			e^{c_0j}\|\<v\>^{l-2}f|_{t=j}\|^2_{L^2_x(\Omega)L^2_v}
			\le C\|\<v\>^{l-2}f_0\|^2_{L^2_x(\Omega)L^2_v}+C\|e^{c_0t}\<v\>^{l-2}g\|^2_{L^2_t([0,j])L^2_{x,v}(\Si_-)}, 
		\end{align}
		and 
		\begin{align*}
			\|\<v\>^{l_0}f|_{t=j}\|^2_{L^2_x(\Omega)L^2_v}
			\le C\|\<v\>^{l_0}f_0\|^2_{L^2_x(\Omega)L^2_v}+C\|\<v\>^{l_0}g\|^2_{L^2_t([0,j])L^2_{x,v}(\Si_-)}, 
		\end{align*}
		with some constant $c_0>0$. Moreover, by a simple $L^2$ energy estimate on $[j,j+1]$ (for instance \eqref{84cz}), we have 
		\begin{align}\label{733c}
			\|f\|^2_{L^\infty_t([j,j+1])L^2_x(\Omega)L^2_v}
			\le C\|f|_{t=j}\|^2_{L^2_x(\Omega)L^2_v}+C\|g\|^2_{L^2_t([j,j+1])L^2_{x,v}(\Si_-)}. 
		\end{align}
		Thus, combining the above three estimates \eqref{733a}--\eqref{733c}, by assumption \eqref{ve0z} and the \emph{a priori} assumption \eqref{ve111} for the current Theorem, the assumption \eqref{ve1} in Theorem \ref{LinftyLinear} (for linear equation) is satisfied with $[T_1,T_2]=[j,j+1]$, $\psi\equiv\vp\equiv\vp_2:=f$ and $\vp_1=0$:
		\begin{align*}
			\begin{aligned}
				\|\<v\>^lf\|_{L^\infty([j,j+1])L^\infty_{x}(\Omega)L^\infty_v}&\le\de_0,\\
				\|\<v\>^{l}g\|_{L^\infty_{t,x,v}([j,j+1]\times\Si_-)}\le \ve_\infty,\quad\|\<v\>^{l}f|_{t=j}\|_{L^\infty_{x,v}(\Omega\times\R^3_v)} &\le \de_0,\\
				\|\<v\>^{l-2}g\|^2_{L^2_t([j,j+1])L^2_{x,v}(\Si_-)}+
				\|\<v\>^{l-2}f|_{t=j}\|^2_{L^2_x(\Omega)L^2_v(\R^3_v)}+\|f\|^2_{L^\infty_t([j,j+1])L^2_{x}(\Omega)L^2_v}
				&\le \ve_1Ce^{-c_0j},\\
				\|\<v\>^{l_0}g\|^2_{L^2_t([j,j+1])L^2_{x,v}(\Si_-)}+
				\|\<v\>^{l_0}f|_{t=j}\|^2_{L^2_t([j,j+1])L^2_x(\Omega)L^2_v}
				&\le \ti CC.
			\end{aligned}
		\end{align*}
		Note that the constant $l_0$ in the current Theorem is larger than the one in Theorem \ref{LinftyLinear}. 
		Therefore, applying Theorem \ref{LinftyLinear}
		with $\de_\infty\le \ve_\infty$, $\de_\infty'=\|\<v\>^{l}f|_{t=j}\|_{L^\infty_{x,v}(\Omega\times\R^3_v)}$, $\de_1\le \ve_1Ce^{-c_0j}$ and $\ti C_1=\ti CC$ therein, 
		we deduce 
		the $L^\infty$ estimates in time interval $[j,j+1]$: 
		\begin{align}\label{ve222}
			\|\<v\>^lf\|_{L^\infty([j,j+1])L^\infty_{x,v}(\ol\Omega\times\R^3_v)}
			\le\max\big\{\ve_\infty,\,\|\<v\>^{l}f|_{t=j}\|_{L^\infty_{x,v}(\Omega\times\R^3_v)}\big\}
			+ C(2+\ti C)^{C}\ve_1^{\zeta}e^{-c_0j\zeta}, 
		\end{align}
		with some constant $C=C(l,\ga,s)>0$ which is independent of $j,n$, and $\zeta=\zeta(s)>0$. 
		Repeating such estimate on $[0,1]$, $[1,2]$, $\dots$, $[j,j+1]$ $(0\le j\le n-1)$, we have 
		\begin{align*}
			&\|\<v\>^lf\|_{L^\infty([j,j+1])L^\infty_{x,v}(\ol\Omega\times\R^3_v)}\\
			&\quad\le\max\big\{\ve_\infty,\,\|\<v\>^{l}f|_{t=j-1}\|_{L^\infty_{x,v}(\Omega\times\R^3_v)}\big\}
			+ C(2+\ti C)^{C}\ve_1^{\zeta}\big(e^{-c_0(j-1)\zeta}+e^{-c_0j\zeta}\big)\\
			&\quad\le \cdots\le \max\big\{\ve_\infty,\,\|\<v\>^{l}f|_{t=0}\|_{L^\infty_{x,v}(\Omega\times\R^3_v)}\big\}
			+ C(2+\ti C)^{C}\ve_1^{\zeta}\sum_{k=0}^je^{-c_0k\zeta}\\
			&\quad\le \ve_\infty+C(2+\ti C)^{C}\ve_1^{\zeta}. 
		\end{align*}
		Therefore, if we choose $\ve_\infty,\ve_1>0$ sufficiently small, which depends only on $l,\ga,s$ and is independent of $j,n$, we deduce 
		\begin{align*}
			\|\<v\>^lf\|_{L^\infty([0,n])L^\infty_{x,v}(\ol\Omega\times\R^3_v)}\le \ve_\infty+C\ve_1^{\zeta}\le\de_0, 
		\end{align*}
		which closes the \emph{a priori} assumption \eqref{ve111}. Moreover, the constants above are independent of $n$, and hence, one can let $n\to\infty$ to obtain \eqref{84czy}. The $L^2$ estimate \eqref{84czz} is directly from Theorem \ref{L2globalThm}. 
		
		\smallskip\noindent {\bf Remark.}
		One can also notice that the coefficient $``1"$ in the first right-hand term of \eqref{ve222} is essential, and this is the main purpose we split the equation into non-vanishing and vanishing initial-inflow data as stated in Section \ref{Sec65}. 
		Note that the initial local-in-time smallness of the \emph{a priori} assumption \eqref{ve111} can be given by \eqref{84bz}. Then it's standard to apply the continuity argument to obtain the global-in-time existence of the equation \eqref{non2}.
		In fact, the local-in-time existence of the equation \eqref{non2} is given in the Theorem \ref{nonLocal}. If there exists a solution $f$ for any time $t\in[0,T]$, then from \eqref{733a}--\eqref{733c}, we know that $f(T)$ and $g(t)$ ($t\in[T,\infty)$) satisfy the initial condition \eqref{ve0} for the local-in-time existence Theorem \ref{nonLocal} at $t=T$. Moreover, since the constants $L^2$--$L^\infty$ estimates \eqref{84czz} and \eqref{84czy} are independent of time, we can repeat the same procedure and finally deduce the global-in-time existence. 
		This completes the proof of Theorem \ref{ThmInflow}.
	\end{proof}
	
\subsection{Proof of Theorem \ref{ThmInflowMain}}
The main Theorem \ref{ThmInflowMain} follows from the 
global-in-time existence in Theorem \ref{ThmInflow}. 
Moreover, the large-time behavior \eqref{L2es1b} follows from the global $L^2$ decay \eqref{L2es1a} in Section \ref{Sec12}. (Note that Section \ref{Sec12} is self-consistent.)

\section{\texorpdfstring{$L^\infty$}{L infty} estimate for reflection boundary}\label{Sec10}
In this Section, we let $l\ge \ga+10$, $\vpi\ge 0$, $\eta,\ve\in(0,1)$, $0\le T_1<T_2\le T_1+1$, and fix $-\frac{3}{2}<\ga\le 2$ and $s\in(0,1)$. Moreover, we assume the same conditions as in Theorem \ref{LemLinRe}. 
If not specified, the underlying time interval in this Section is $[T_1,T_2]$; i.e.
	$L^q_t=L^q_t([T_1,T_2])$.
Then we consider the $L^\infty$ estimate of the linearized Boltzmann equation with \emph{Maxwell} reflection boundary condition \eqref{reflect}:
\begin{align*}
	Rf(x,v)=(1-\al)f(x,R_L(x)v)+\al R_Df(x,v),
\end{align*}
with some $\al\in(0,1)$, where $R_L(x)$ and $R_D$ are given in \eqref{RL1} and \eqref{RD1} respectively. Moreover,
we will consider level functions as in \eqref{flK} with constant $K>0$: 
\begin{align*}
	f^{(l)}_K :=f-K\<v\>^{-l}_\de,\quad f^{(l)}_{K,+}=f^{(l)}_K\1_{f^{(l)}_K\ge 0}.
\end{align*}
To find the solution to the nonlinear Boltzmann equation, we need to derive the $L^\infty$ estimate of the solution $f$ to \eqref{linearre} and use the modified Boltzmann equation as follows. 
As in Section \ref{SecLinfty}, we will add the extra dissipation term ${\eta\<v\>^l}^{}f$ to obtain the initial $L^\infty$ bound for $f$. 
Applying Theorem \ref{LemLinRe}, we can obtain the solution $f$ to the linear modified Boltzmann equation with reflection:
\begin{align}\label{diff}
	\left\{
	\begin{aligned}
 & \pa_tf+ v\cdot\na_xf =\vpi Vf+
		\Gamma(\Psi,f)+\Gamma(\vp,\mu^{\frac{1}{2}})\\&\qquad\qquad\qquad\qquad-{\eta\<v\>^l}^{}f \quad \text{ in } [T_1,T_2]\times\Omega\times\R^3_v, \\
 & f|_{\Si_-}=(1-\ve)Rf \quad \text{ on }[T_1,T_2]\times\Si_-, \\
 & f(T_1,x,v)=f_{T_1} \qquad \text{ in }\Omega\times\R^3_v.
	\end{aligned}\right.
\end{align}
To split equation \eqref{diff}, by splitting $\vp=\vp_1+\vp_2$ with $\vp_1,\vp_2\in L^2_{t,x,v}$, we apply Theorem \ref{LemLinRe} again to obtain the solution $f_1$ to equation (with reflection):
\begin{align}\label{diffa}
	\left\{
	\begin{aligned}
    & \pa_tf_1+ v\cdot\na_xf_1 =\vpi Vf_1+
		\Gamma(\Psi,f_1)+\Gamma(\vp_1,\mu^{\frac{1}{2}})
\\&\qquad\qquad\qquad\qquad
		-N\<v\>^{l-2}f_1
		-\eta\<v\>^lf_1
		\quad \text{ in } [T_1,T_2]\times\Omega\times\R^3_v,  \\
    & f_1|_{\Si_-}=(1-\ve)Rf_1\quad\quad \text{ on }[T_1,T_2]\times\Si_-, \\
    & f_1(T_1,x,v)=f_{T_1} \qquad\qquad \text{ in }\Omega\times\R^3_v,
	\end{aligned}\right.
\end{align}
which has non-vanishing initial data.
Let $f_2=f-f_1$ in $\ol\Omega$. We then obtain the equation of $f_2$:
\begin{align}\label{diffb}
	\left\{
	\begin{aligned}
 & \pa_tf_2+ v\cdot\na_xf_2 =\vpi Vf_2+
		\Gamma(\Psi,f_2)+\Gamma(\vp_2,\mu^{\frac{1}{2}})
		\\&\qquad\qquad\qquad\qquad
+N\<v\>^{l-2}f_1
		-\eta\<v\>^lf_2
		\quad \text{ in } [T_1,T_2]\times\Omega\times\R^3_v,  \\
 & f_2|_{\Si_-}=(1-\ve)Rf_2\quad\quad \text{ on }[T_1,T_2]\times\Si_-, \\
 & f_2(T_1,x,v)=0\qquad\qquad \text{ in }\Omega\times\R^3_v,
	\end{aligned}\right.
\end{align}
which has vanishing initial data.
Then we can find its extension to the whole space by using Theorem \ref{ThmExtend}: 
\begin{align}\label{diff1}
	\begin{cases}
		\pa_tf+v\cdot\na_xf = g,&\text{ in }[T_1,T_2]\times\R^3_x\times\R^3_v,\\
		f|_{\Si_-}=(1-\ve)Rf &\text{ on }[T_1,T_2]\times\Si_-,\\
		f(T_1,x,v)=0 &\text{ in }\Omega\times\R^3_v, \\
		f(T_1,x,v)=0 &\text{ in }D_{out}, \\
		f(T_2,x,v)=0 & \text{ in }D_{in},
	\end{cases}
\end{align}
for some initial data $f_{T_1}$, where 
\begin{align}
	\label{gDef}
	g=\begin{cases}
		\vpi Vf+\Gamma(\Psi,f)+\Gamma(\vp,\mu^{\frac{1}{2}})+N\<v\>^{l-2}f_1-{\eta\<v\>^l}^{}f&\text{ in }\Omega\times\R^3_v,\\
		-E\cdot\na_vf+P^2f&\text{ in }D_{in},\\
		-E\cdot\na_vf-P^2f&\text{ in }D_{out}.
	\end{cases}
\end{align}
Here $D_{in}$ and $D_{out}$ are given in \eqref{Dinoutwt}.
In the following Subsections, we will derive the $L^\infty$ estimates for equations \eqref{diffa} and \eqref{diff1}.

\subsection{Basic \texorpdfstring{$L^2$}{L2} level-function estimates}\label{Sec81}
Fix a $\de>0$ (used in the weight $\<v\>^{l}_{\de}$ and cutoff $\chi^-_\de$) to be determined in \eqref{87d}. 
Let $T_2=T_1+\de^3$ and $[T_1,T_2]$ be the underlying time interval.  
	Assume $\Psi=\mu^{\frac{1}{2}}+\psi\ge 0$ and $\vp$ satisfy
	\begin{align*}
		\|[\<v\>^l\psi,\<v\>^{l}\vp]\|_{L^\infty_t([T_1,T_2])L^\infty_x(\Omega)L^\infty_v}\le\de_0,
	\end{align*}
	with small $\de_0>0$. 
We consider a model equation for both \eqref{diffa} and \eqref{diff1}:
\begin{align}\label{diff1model}
\left\{
\begin{aligned}
	& \pa_tf+v\cdot\na_xf =
		\vpi Vf+\Gamma(\Psi,f)+\Gamma(\vp,\mu^{\frac{1}{2}})&\\&\qquad\qquad\qquad\quad-M\<v\>^{l-2}f+N\<v\>^{l-2}\phi-\eta\<v\>^lf 
		\quad\text{ in }[T_1,T_2]\times\Omega\times\R^3_v,\\
	&\pa_tf+ v\cdot\na_xf +E\cdot\na_vf=P^2f\qquad\qquad \text{ in } [T_1,T_2]\times D_{in},\\
	&\pa_tf+v\cdot\na_xf+E\cdot\na_vf=-P^2f\qquad\qquad \text{ in } [T_1,T_2]\times D_{out},\\
	& f|_{\Si_-}=(1-\ve)Rf \qquad\qquad \text{ on }[T_1,T_2]\times\Si_-,\\
	& f(T_1,x,v)=f_{T_1} \qquad\qquad \text{ in }\Omega\times\R^3_v,\\
	& f(T_1,x,v)=0 \qquad\qquad \text{ in }D_{out},\\
	& f(T_2,x,v)=0 \qquad\qquad \text{ in }D_{in},
\end{aligned}\right.
\end{align}
Here $D_{in}$ and $D_{out}$ are given in \eqref{Dinoutwt}.
Note that 
$$
v\cdot\na_xff^{(l)}_{K,+}=\frac{1}{2}v\cdot\na_x(f^{(l)}_{K,+})^2+Kv\cdot\na_x\<v\>^{-l}_\de f^{(l)}_{K,+}.
$$
Multiplying \eqref{diff1model} by $f^{(l)}_{K,+}$ in $\Omega$, 
\begin{align}\label{87c}\notag
	\pa_t|f^{(l)}_{K,+}|^2+v\cdot\na_x|f^{(l)}_{K,+}|^2 &=
		2\Big(\vpi Vf+\Gamma(\Psi,f-K\<v\>^l_\de)+\Gamma(\Psi,K\<v\>^l_\de)+\Gamma(\vp,\mu^{\frac{1}{2}})\\&\qquad-M\<v\>^{l-2}f+N\<v\>^{l-2}\phi-\eta\<v\>^lf-Kv\cdot\na_x\<v\>^{-l}_\de\Big)f^{(l)}_{K,+}.
\end{align}

\noindent{\bf Main energy in $\Omega$.}
For any $T\in[T_1,T_2]$, integrating \eqref{87c} over $(T_1,T)\times\Omega\times\R^3_v$, 
\begin{multline*}
	\|f^{(l)}_{K,+}(T)\|^2_{L^2_x(\Omega)L^2_v}-\|f^{(l)}_{K,+}(T_1)\|^2_{L^2_x(\Omega)L^2_v}
		+\|f^{(l)}_{K,+}\|^2_{L^2_t(T_1,T)L^2_{x,v}(\Si_+)}
		-\|f^{(l)}_{K,+}\|^2_{L^2_t(T_1,T)L^2_{x,v}(\Si_-)}\\
	=\int_{(T_1,T)\times\Omega\times\R^3_v}\Big(\pa_t|f^{(l)}_{K,+}|^2+v\cdot\na_x|f^{(l)}_{K,+}|^2\Big)\,dvdxdt. 
\end{multline*}
For the boundary term, we have from \eqref{101a} and \eqref{533} that 
\begin{align}\notag\label{87b}
	\|(Rf)^{(l)}_{K,+}\|^2_{L^2_t(T_1,T)L^2_{x,v}(\Si_-)}
	&\le 
	(1-\al+C\al\de^2)\|f^{(l)}_{K,+}\|^2_{L^2_t(T_1,T)L^2_{x,v}(\Si_+)}+\al\|f^{(l)}_{K,+}(T_1)\|^2_{L^2_x(\Omega)L^2_v}\\
	&\quad-\al\int_{T_1}^{T}\int_{\Omega\times\R^3_v}\chi^{-}_{\de}\big(\pa_t|f^{(l)}_{K,+}|^2+v\cdot\na_x|f^{(l)}_{K,+}|^2\big)\,dvdxdt. 
\end{align}
Combining the above estimates 
and choosing $\de=\de(\al)>0$ sufficiently small such that $1-\big(1-\al+C\al\de^2\big)=\frac{\al}{2}$, we have 
\begin{align}\label{87d}\notag
	&(1-\al)\|f^{(l)}_{K,+}(T)\|^2_{L^2_x(\Omega)L^2_v}-\|f^{(l)}_{K,+}(T_1)\|^2_{L^2_x(\Omega)L^2_v}
		+c_\al\|f^{(l)}_{K,+}\|^2_{L^2_t(T_1,T)L^2_{x,v}(\Si_+)}\\
	&\quad\le\int_{(T_1,T)\times\Omega\times\R^3_v}\Big(\al\big(1-\chi^{-}_{\de}\big)+(1-\al)\Big)\Big(\pa_t|f^{(l)}_{K,+}|^2+v\cdot\na_x|f^{(l)}_{K,+}|^2\Big)\,dvdxdt. 
\end{align}
for some constant $c_\al>0$ that depends only on $\al\in(0,1)$. 
Utilizing the right-hand side of \eqref{87c} and applying Lemmas \ref{GachideLem} and \ref{CollLevelLem} for level-function collisional estimate with and without cutoff $1-\chi^{-}_{\de}$, respectively, we deduce 
\begin{align*}
	&\notag\int_{(T_1,T)\times\Omega\times\R^3_v}\Big(\pa_t|f^{(l)}_{K,+}|^2+v\cdot\na_x|f^{(l)}_{K,+}|^2\Big)\,dvdxdt\\
	&\notag\quad+2M\|\<v\>^{\frac{l}{2}-1}f^{(l)}_{K,+}\|_{L^2_t(T_1,T)L^2_x(\Omega)L^2_v}^2
	+\frac{2MK}{C_\de}\|\<v\>^{-2}f^{(l)}_{K,+}\|_{L^1_t(T_1,T)L^1_x(\Omega)L^1_v}\\
	&\notag\quad+4\vpi\|[\wh{C}^{}_0\<v\>^{{4}}f^{(l)}_{K,+},\<v\>^{{2}}\na_vf^{(l)}_{K,+}]\|^2_{L^2_t(T_1,T)L^2_{x}(\Omega)L^2_v}
	+2\vpi K\wh{C}^2_0\|\<v\>^{{-l+8}}f^{(l)}_{K,+}\|_{L^1_t(T_1,T)L^1_x(\Omega)L^1_v}\\
	&\notag\quad+2(c_0-C\|\<v\>^{4}\psi\|_{L^\infty_t(T_1,T)L^\infty_x(\Omega)L^\infty_v})\|f^{(l)}_{K,+}\|_{L^2_t(T_1,T)L^2_x(\Omega)L^2_D}^2\\
	&\notag\quad+2\eta\|\<v\>^lf^{(l)}_{K,+}\|_{L^2_t(T_1,T)L^2_x(\Omega)L^2_v}^2+\frac{2\eta K}{C_\al}\|f^{(l)}_{K,+}\|_{L^1_t(T_1,T)L^1_x(\Omega)L^1_v}\\
	&\notag\quad\le C_{\de}K\|\<v\>^{-l}f^{(l)}_{K,+}\|_{L^1_t(T_1,T)L^1_x(\Omega)L^1_v}
	+C\|f^{(l)}_{K,+}\|_{L^2_t(T_1,T)L^2_x(\Omega)L^2_v}^2\\
	&\qquad+2\|[K\<v\>^l\psi,\vp,N\<v\>^{l}\phi]\|_{L^\infty_t(T_1,T)L^\infty_x(\Omega)L^\infty_v}\|\<v\>^{-2}f^{(l)}_{K,+}\|_{L^1_t(T_1,T)L^1_x(\Omega)L^1_v}, 
\end{align*}
and 
\begin{align}\label{8levelcollib}
	&\notag\int_{(T_1,T)\times\Omega\times\R^3_v}(1-\chi^{-}_{\de})\big(\pa_t|f^{(l)}_{K,+}|^2+v\cdot\na_x|f^{(l)}_{K,+}|^2\big)\,dvdxdt\\
	&\notag\quad\notag\le C_\de\|\<v\>^{\frac{(\ga+2s)_+}{2}}f^{(l)}_{K,+}\|_{L^2_t(T_1,T)L^2_x(\Omega)L^2_v}^2\\
	&\quad\notag+C_\de\|[K\<v\>^l\psi,\vp,N\<v\>^{l}\phi]\|_{L^\infty_t(T_1,T)L^\infty_x(\Omega)L^\infty_v}\|\<v\>^{-2}f^{(l)}_{K,+}\|_{{L^1_t(T_1,T)L^1_x(\Omega)L^1_v}}\\
	&\quad+\vpi C_\de\|\<v\>^{{2}}f^{(l)}_{K,+}\|_{L^2_t(T_1,T)L^2_x(\Omega)L^2_v}\|\<v\>^{{2}}[f^{(l)}_{K,+},\na_vf^{(l)}_{K,+}]\|_{L^2_t(T_1,T)L^2_x(\Omega)L^2_v},
\end{align}
Here, $1-\chi^-_\de\in[0,1]$, $\|\mu^{\frac{1}{10^4}}\vp\|_{L^2_t(T_1,T)L^2_x(\Omega)L^2_v}\le C_{|\Omega|}\de_0$, we choose $\de_0>0$ sufficiently small and use \eqref{chideesti} to control the norm of $\chi^-_\de$. 
Substituting these two collisional estimates into \eqref{87d} and choosing sufficiently large $\wh{C}_0=\wh{C}_0(\ga,s,\al)>0$ (note that $\de=\de(\al)>0$ is chosen), we obtain 
\begin{align}\label{8maininOmega}
	&\notag\|f^{(l)}_{K,+}(T)\|^2_{L^2_x(\Omega)L^2_v}
		+\frac{\al}{2(1-\al)}\|f^{(l)}_{K,+}\|^2_{L^2_t(T_1,T)L^2_{x,v}(\Si_+)}+c_0\|f^{(l)}_{K,+}\|_{L^2_t(T_1,T)L^2_x(\Omega)L^2_D}^2\\
	&\quad\notag+2M\|\<v\>^{\frac{l}{2}-1}f^{(l)}_{K,+}\|_{L^2_t(T_1,T)L^2_x(\Omega)L^2_v}^2
	+\frac{2MK}{C_\de}\|\<v\>^{-2}f^{(l)}_{K,+}\|_{L^1_t(T_1,T)L^1_x(\Omega)L^1_v}\\
	&\quad\notag+2\vpi\|[\wh{C}^{}_0\<v\>^{{4}}f^{(l)}_{K,+},\<v\>^{{2}}\na_vf^{(l)}_{K,+}]\|^2_{L^2_t(T_1,T)L^2_{x}(\Omega)L^2_v}
	+2\vpi K\wh{C}^2_0\|\<v\>^{{-l+8}}f^{(l)}_{K,+}\|_{L^1_t(T_1,T)L^1_x(\Omega)L^1_v}\\
	&\notag\quad+2\eta\|\<v\>^lf^{(l)}_{K,+}\|_{L^2_t(T_1,T)L^2_x(\Omega)L^2_v}^2+\frac{2\eta K}{C_\al}\|f^{(l)}_{K,+}\|_{L^1_t(T_1,T)L^1_x(\Omega)L^1_v}\\
	&\quad\notag\le C_\al\|f^{(l)}_{K,+}(T_1)\|^2_{L^2_x(\Omega)L^2_v}
	+C_{\al}\|\<v\>^{\frac{(\ga+2s)_+}{2}}f^{(l)}_{K,+}\|_{L^2_t(T_1,T)L^2_x(\Omega)L^2_v}^2\\
	&\qquad+C_{\al}\|[K\<v\>^l\psi,\vp,N\<v\>^{l}\phi]\|_{L^\infty_t(T_1,T)L^\infty_x(\Omega)L^\infty_v}\|\<v\>^{-2}f^{(l)}_{K,+}\|_{L^1_t(T_1,T)L^1_x(\Omega)L^1_v},  
\end{align}
for any $T\in[T_1,T_2]$ and any $T_2-T_1\le\de^3$. This is the main energy estimate in $\Omega$. 

\medskip
\noindent{\bf Main energy in $\ol\Omega^c$.} For the estimate in $\ol\Omega^c$, we follow the calculations for deriving \eqref{esffff1}, and denote $h$ as in \eqref{hftt}, i.e. 
	$h(t)=f(T_1+T_2-t)$ in $D_{in}$ and $h(t)=f(t)$ in $D_{out}$. 
Then it follows from \eqref{diff1} that 
\begin{align}\label{hrever}
	\left\{
	\begin{aligned}
		&\pa_th+v\cdot\na_x\big(h\1_{D_{out}}-h\1_{D_{in}}\big)+E\cdot\na_v\big(h\1_{D_{out}}-h\1_{D_{in}}\big)\\
		&\qquad\qquad+P^2h=0,
		\qquad\qquad\qquad\text{ in } [T_1,T_2]\times\ol\Omega^c\times\R^3_v,\\
		&h|_{\Si_-}=f(T_1+T_2-t)|_{\Si_-}\qquad\qquad \text{ on }[T_1,T_2]\times\Si_-,\\
		&h|_{\Si_+}=f(t)|_{\Si_+}\qquad\qquad \text{ on }[T_1,T_2]\times\Si_+,\\
		&h(s,x,v)=0\qquad\qquad \text{ in }\ol\Omega^c\times\R^3_v.
	\end{aligned}\right.
\end{align}
For any $T\in[T_1,T_2]$ and $T_2=T_1+\de^3$, taking the inner product of \eqref{hrever} with $h^{(l)}_{K,+}$ over $[T_1,T]\times\ol\Omega^c\times\R^3_v$, using Lemma \eqref{CollLevelLem} (i.e. \eqref{651in}, \eqref{651out} and \eqref{651advec}), and the vanishing boundary measure in \eqref{traceh}, we have
\begin{multline}\label{esffff12}
	\|h^{(l)}_{K,+}(T)\|_{L^2_{x,v}(\ol\Omega^c\times\R^3_v)}^2
	+\|Ph^{(l)}_{K,+}\|^2_{L^2_t(T_1,T)L^2_x(\ol\Omega^c)L^2_v}
	+2K\|\<v\>^{-l}_\de P^2h^{(l)}_{K,+}\|_{L^1_t(T_1,T)L^1_x(\ol\Omega^c)L^1_v}\\
	\le 
	\|f^{(l)}_{K,+}\|_{L^2_t(T_1,T)L^2_{x,v}(\Si_+)}^2
	+\|f^{(l)}_{K,+}\|_{L^2_t(T_1,T)L^2_{x,v}(\Si_-)}^2
	+C_{\de}K\|\<v\>^{-l}f^{(l)}_{K,+}\|_{L^1_t(T_1,T)L^1_x(\ol\Omega^c)L^1_v}. 
\end{multline}
Choosing the constant $\wh{C}_l=\wh{C}_l(\ga,s,\de,\|n\|_{W^{1,\infty}})>0$ in \eqref{EP} sufficiently large, the last term in \eqref{esffff12} can be absorbed. The boundary energy on $\Si_-$ can be estimated by \eqref{87b} and Lemma \ref{GachideLem} (similar to \eqref{8levelcollib}). Thus, for any $T\in[T_1,T_2]$ and any $T_2-T_1\le\de^3$, 
\begin{align}\label{8maininOmegac}\notag
	&\|h^{(l)}_{K,+}(T)\|_{L^2_{x,v}(\ol\Omega^c\times\R^3_v)}^2
	+\|Ph^{(l)}_{K,+}\|^2_{L^2_t(T_1,T)L^2_x(\ol\Omega^c)L^2_v}
	+K\|\<v\>^{-l}_\de P^2h^{(l)}_{K,+}\|_{L^1_t(T_1,T)L^1_x(\ol\Omega^c)L^1_v}\\
	&\notag\le 
	(2-\al+C\al\de^2)\|f^{(l)}_{K,+}\|^2_{L^2_t(T_1,T)L^2_{x,v}(\Si_+)}+\al\|f^{(l)}_{K,+}(T)\|^2_{L^2_x(\Omega)L^2_v}\\
	&\quad\notag
	+C_\al\|\<v\>^{\frac{(\ga+2s)_+}{2}}f^{(l)}_{K,+}\|_{L^2_t(T_1,T)L^2_x(\Omega)L^2_v}^2\\
	&\quad\notag+C\|[K\<v\>^l\psi,\vp,N\<v\>^{l}\phi]\|_{L^\infty_t(T_1,T)L^\infty_x(\Omega)L^\infty_v}\|f^{(l)}_{K,+}\|_{L^1_t(T_1,T)L^1_x(\Omega)L^1_v}\\
	&\quad+\vpi C_\de\|\<v\>^{{2}}f^{(l)}_{K,+}\|_{L^2_t(T_1,T)L^2_x(\Omega)L^2_v}\|\<v\>^{{2}}[f^{(l)}_{K,+},\na_vf^{(l)}_{K,+}]\|_{L^2_t(T_1,T)L^2_x(\Omega)L^2_v}. 
\end{align}

\subsection{\texorpdfstring{$L^\infty$}{L infty} estimate with non-vanishing data}
In this Subsection, we analyze the $L^\infty$ control on the solution $f$ to equation \eqref{diffa}, which can be rewritten as
\begin{align}\label{diff0}
	\left\{
	\begin{aligned}
    & \pa_tf+ v\cdot\na_xf =\vpi Vf+
	\Gamma(\Psi,f)+\Gamma(\vp,\mu^{\frac{1}{2}})
	\\&\qquad\qquad\qquad\qquad-N\<v\>^{l-2}f-\eta\<v\>^{l}f
	\quad\text{ in } [T_1,T_2]\times\Omega\times\R^3_v,  \\
    & f|_{\Si_-}=(1-\ve)Rf \qquad \text{ on }[T_1,T_2]\times\Si_-,\\
    & f(T_1,x,v)=f_{T_1}\qquad\ \text{ in }\Omega\times\R^3_v.
	\end{aligned}\right. 
\end{align}
\begin{Lem}\label{ReLinfty1Lem}
	Fix $\de=\de(\al)>0$ (determined in Subsection \ref{Sec81}). Let $l\ge\ga+10$, $0\le T_1<T_2=T_1+\de^3$ and let $[T_1,T_2]$ be the underlying time interval.  
	Assume $\Psi=\mu^{\frac{1}{2}}+\psi\ge 0$ and $\vp$ satisfy
	\begin{align*}
		\|[\<v\>^l\psi,\<v\>^{l}\vp]\|_{L^\infty_t([T_1,T_2])L^\infty_x(\Omega)L^\infty_v}=\de_0,
	\end{align*}
	with sufficiently small $\de_0\in(0,1)$. Let $N>0$ be a sufficiently large constant depending on $\ga,s,\al$. 
	If $f$ is the solution to \eqref{diff0}, then
	\begin{align}\label{ReLinfty1}
		\|\<v\>^{l}_{\de}f\|_{L^\infty_t([T_1,T_2])L^\infty_x(\ol\Omega)L^\infty_v}\le K_1\equiv
		\max\big\{\frac{1}{2}\|\<v\>^{l}_\de \vp\|_{L^\infty_t([T_1,T_2])L^\infty_x(\Omega)L^\infty_v},\,\|\<v\>^{l}_{\de}f_{T_1}\|_{L^\infty_x(\Omega)L^\infty_v}\big\}.
	\end{align}
\end{Lem}
\begin{proof}
The underlying time interval is $[T_1,T_2]$ is not specified. 
Using the energy estimate in Section \ref{Sec81} and taking supremum $T\in[T_1,T_2]$ in \eqref{8maininOmega}, we have 
\begin{align*}\notag
	&\|f^{(l)}_{K,+}\|^2_{L^\infty 
	_tL^2_x(\Omega)L^2_v}
	+\frac{\al}{2(1-\al)}\|f^{(l)}_{K,+}\|^2_{L^2_tL^2_{x,v}(\Si_+)}+c_0\|f^{(l)}_{K,+}\|_{L^2_tL^2_x(\Omega)L^2_D}^2\\
	&\quad\notag+2N\|\<v\>^{\frac{l}{2}-1}f^{(l)}_{K,+}\|_{L^2_tL^2_x(\Omega)L^2_v}^2
	+\frac{2NK}{C_\al}\|\<v\>^{-2}f^{(l)}_{K,+}\|_{L^1_tL^1_x(\Omega)L^1_v}\\
	&\quad\notag+2\vpi\|[\wh{C}^{}_0\<v\>^{{4}}f^{(l)}_{K,+},\<v\>^{{2}}\na_vf^{(l)}_{K,+}]\|^2_{L^2_tL^2_{x}(\Omega)L^2_v}
	+2\vpi K\wh{C}^2_0\|\<v\>^{{-l+8}}f^{(l)}_{K,+}\|_{L^1_tL^1_x(\Omega)L^1_v}\\
	&\quad\le C_\al\|f^{(l)}_{K,+}(T_1)\|^2_{L^2_x(\Omega)L^2_v}
	+C_{\al}\|\<v\>^{2}f^{(l)}_{K,+}\|_{L^2_tL^2_x(\Omega)L^2_v}^2\\
	&\qquad+C_{\al}\|[K\<v\>^l\psi,\vp]\|_{L^\infty_tL^\infty_x(\Omega)L^\infty_v}\|\<v\>^{-2}f^{(l)}_{K,+}\|_{L^1_tL^1_x(\Omega)L^1_v}. 
\end{align*}
Thus, choosing a large $N=N(\ga,s,\al,\|n\|_{W^{1,\infty}})>4C_\al>0$ and $K>\frac{\|\<v\>^{l}\vp\|_{L^\infty_tL^\infty_x(\Omega)L^\infty_v}}{C_\al}$, we have  
\begin{align*}
	&\|f^{(l)}_{K,+}\|^2_{L^\infty_tL^2_x(\Omega)L^2_v}
	+\frac{\al}{2(1-\al)}\|f^{(l)}_{K,+}\|^2_{L^2_tL^2_{x,v}(\Si_+)}+c_0\|f^{(l)}_{K,+}\|_{L^2_tL^2_x(\Omega)L^2_D}^2\le C_\al\|f^{(l)}_{K,+}(T_1)\|^2_{L^2_x(\Omega)L^2_v}. 
\end{align*}
Therefore, choosing 
\begin{align*}
	K=\max\big\{\frac{1}{2}\|\<v\>^{l}_\de \vp\|_{L^\infty_t([T_1,T_2])L^\infty_x(\Omega)L^\infty_v},\,\|\<v\>^{l}_{\de}f_{T_1}\|_{L^\infty_x(\Omega)L^\infty_v}\big\}, 
\end{align*}
we have $\|f^{(l)}_{K,+}(T_1)\|^2_{L^2_x(\Omega)L^2_v}=0$ and hence, $$\|f^{(l)}_{K,+}(t)\|^2_{L^\infty_t([T_1,T_2])L^2_x(\Omega)L^2_v}
=\|f^{(l)}_{K,+}\|^2_{L^2_t([T_1,T_2])L^2_{x,v}(\Si_+)}=0,$$
which implies $f\le K\<v\>^{-l}_\de$ in $[T_1,T_2]\times\ol\Omega\times\R^3_v$. 
Similarly, we can multiply \eqref{diff0} by $(-f)^{(l)}_{K,+}$ and follow the above arguments to deduce the lower bound: $f\ge -K\<v\>^{-l}_\de$ in $[T_1,T_2]\times\ol\Omega\times\R^3_v$. This implies \eqref{ReLinfty1} and concludes Lemma \ref{ReLinfty1Lem}. 
\end{proof}

\subsection{Initial \texorpdfstring{$L^\infty$}{L infty} bound for vanishing data}
In this and the following Subsections, we consider the vanishing-initial data equation \eqref{diff1}. 
As in Section \ref{Secinitial1}, we first derive a large initial $L^\infty$ bound of $f$ and then, based on this large bound, we derive the improved $L^\infty$ bound later in Section \ref{Sec97}. 
This initial $L^\infty$ bound will serve as an \emph{a priori} bound to imply the finiteness of $\|\<v\>^l_\de f\|_{L^\infty_{t,x,v}}$. 
\begin{Lem}\label{initialLinftyLem2}
	Let $0\le T_1<T_2=T_1+\de^3$, $\al\in(0,1)$ and $l\ge \ga+10$. Let $N=N(\al,\ga,s)>0$ be a large constant determined in Lemma \ref{ReLinfty1Lem}. 
	Suppose $\Psi=\mu^{\frac{1}{2}}+\psi\ge 0$, $\vp$ satisfy
	\begin{align*}
		\begin{aligned}
			\|[\<v\>^l\psi,\<v\>^{l}\vp]\|_{L^\infty_t([T_1,T_2])L^\infty_{x}(\Omega)L^\infty_v}&\le \de_0,\\
			\|\<v\>^lf_1\|_{L^\infty(T_1,T)L^\infty_x(\Omega)L^\infty_v}&=K_1,
		\end{aligned}
	\end{align*}
	with some $K_1>0$ and sufficiently small $0<\de_0<1$. 
	Let $f$ be the solution to \eqref{diff1} in the sense of \eqref{weakfwhole} with vanishing initial data. Fix $\de=\de(\al)>0$ as in Subsection \ref{Sec81}. Then $f$ has an $L^\infty$ bound 
	\begin{align}\label{inidiffLinfty1}
		\|\<v\>^l_\de f\|_{L^\infty_{t,x,v}([T_1,T_2]\times\R^3_x\times\R^3_v)}
		\le e^{C_{\al,\eta}\de^3}(1+\de_0+NK_1)<\infty, 
	\end{align}
	where $C_{\al,\eta}=C(\al,\eta,\ga,s,l)>0$ is a constant independent of $\vpi,\vp,T_1$.
\end{Lem}
\begin{proof}
	The proof is a simple application of De Giorgi's arguments and is similar to Lemma \ref{initialLinftyLem}. Also, we only give the proof of the upper bound of $f$, while the lower bound of $f$ shares a similar calculation which is omitted.

	\smallskip\noindent{\bf Initial $L^\infty$ bound in $\Omega$.}
	To capture the necessary dissipation, we use the time-dependent function
	$K(t)\ge 0$ to be chosen later with the level functions given by
	\begin{align*}
		f_{K} := f-{K(t)}\<v\>^{-l},\quad 
		f^{(l)}_{K,+}=f_{K(t)}\1_{f_{K(t)}\ge 0}. 
	\end{align*}
	Note that only in this step $K(t)$ is a time-dependent function, and 
\begin{align*}
	(\pa_t+v\cdot\na_x)ff^{(l)}_{K,+}=\frac{1}{2}(\pa_t+v\cdot\na_x)(f^{(l)}_{K,+})^2+(\pa_tK+v\cdot\na_x\<v\>^{-l}_\de) f^{(l)}_{K,+}.
\end{align*}
Multiplying \eqref{diff1} by $f^{(l)}_{K,+}$ and using \eqref{deri+}, we have 
\begin{align*}
	\pa_t|f^{(l)}_{K,+}|^2+v\cdot\na_x|f^{(l)}_{K,+}|^2+\pa_tK(t)f^{(l)}_{K,+}&=
		2\Big(\vpi Vf+\Gamma(\Psi,f)+\Gamma(\vp,\mu^{\frac{1}{2}})\\&\qquad+N\<v\>^{l-2}f_1-\eta\<v\>^lf-Kv\cdot\na_x\<v\>^{-l}_\de\Big)f^{(l)}_{K,+}, 
\end{align*}
which is a similar identity to \eqref{87c} except that the dissipation $\pa_tK(t)$ is present. Thus, using the same technique for deriving \eqref{8maininOmega}, we have 
\begin{align}\label{8maininOmega1}\notag
	&\|f^{(l)}_{K,+}(T)\|^2_{L^2_x(\Omega)L^2_v}
	+\frac{\al}{2(1-\al)}\|f^{(l)}_{K,+}\|^2_{L^2_t(T_1,T)L^2_{x,v}(\Si_+)}
	+c_0\|f^{(l)}_{K,+}\|_{L^2_t(T_1,T)L^2_x(\Omega)L^2_D}^2\\
	&\quad\notag
	+\|\pa_tK\<v\>^{-l}_\de f^{(l)}_{K,+}\|_{L^1_t(T_1,T)L^1_{x}(\Omega)L^1_v}+\frac{2\eta K}{C_\al}\|f^{(l)}_{K,+}\|_{L^1_t(T_1,T)L^1_x(\Omega)L^1_v}\\
	&\quad
	\le C_{\al}\|\<v\>^{\frac{(\ga+2s)_+}{2}}f^{(l)}_{K,+}\|_{L^2_t(T_1,T)L^2_x(\Omega)L^2_v}^2
	+C_{\al}\big(\de_0+K+NK_1\big)\|\<v\>^{-2}f^{(l)}_{K,+}\|_{L^1_t(T_1,T)L^1_x(\Omega)L^1_v},
\end{align}
for any $T\in[T_1,T_2]$, 
where we have vanishing initial data (we drop the unnecessary dissipation term). 
Take supremum over $T\in[T_1,T_2]$ and note that the term $\|\<v\>^{\frac{(\ga+2s)_+}{2}}f^{(l)}_{K,+}\|_{L^2_t(T_1,T)L^2_x(\Omega)L^2_v}^2$ can be absorbed by dissipation when $\ga+2s\ge 0$ or by instant energy when $\ga+2s<0$ with small $T_2-T_1\le\de^3>0$. 
For the $L^1$ norm, we choose $K(t)\ge \de_0+NK_1$, and apply interpolation to deduce 
	\begin{align*}
		C_\al(\de_0 C+K+NK_1)\|\<v\>^{-2}f^{(l)}_{K,+}\|_{L^1_{x}(\Omega)L^1_v}&\le C_\al K(t)\|\<v\>^{-2}f^{(l)}_{K,+}\|_{L^1_{x}(\Omega)L^1_v}\\
		&\le C_{\al,\eta}K\|\<v\>^{-l}f^{(l)}_{K,+}\|_{L^1_{x}(\Omega)L^1_v}+\frac{\eta K}{C_\al}\|f^{(l)}_{K,+}\|_{L^1_{x}(\Omega)L^1_v}, 
	\end{align*}
	for some constant $C_\eta=C(\eta,\ga,s,l)>0$. 
	To control the term $C_{\al,\eta}K\|\<v\>^{-l}f^{(l)}_{K,+}\|_{L^1_{x}(\Omega)L^1_v}$, we simply choose 
	\begin{align*}
		K(t)=e^{C_{\al,\eta}(t-T_1)}(1+\de_0+NK_1), \ \text{which satisfies }\pa_tK\ge C_{\al,\eta} K \text{ and } K(t)\ge \de_0+NK_1, 
	\end{align*}
	where we fix the constant $C_{\al,\eta}>0$ here until the end of this proof. 
	Substituting these into \eqref{8maininOmega1}, we have
	\begin{align*}
		\|f^{(l)}_{K,+}\|^2_{L^\infty_t(T_1,T_2)L^2_x(\Omega)L^2_v}
	+\frac{\al}{2(1-\al)}\|f^{(l)}_{K,+}\|^2_{L^2_t(T_1,T_2)L^2_{x,v}(\Si_+)}
	\le 0,  
	\end{align*} 
		which implies $f\le K(t)\<v\>^{-l}_\de$ in $[T_1,T_2]\times\ol\Omega\times\R^3_v$. 
	The lower bound can be deduced similarly. 

\medskip\noindent{\bf Initial $L^\infty$ bound in $\ol\Omega^c$.} 
For the upper bound in $\ol\Omega^c$, it follows from \eqref{8maininOmegac} that if we choose constant $$K=\|f\|_{L^\infty_{t,x,v}([T_1,T_2]\times\ol\Omega\times\R^3_v)},$$ which is finite due to the above step, then $$\|h^{(l)}_{K,+}\|_{L^\infty_t([T_1,T_2])L^2_{x,v}(\ol\Omega^c\times\R^3_v)}^2=0,$$ and hence, $f\le C_\infty\<v\>^{-l}_\de$ in $\ol\Omega^c$. The lower bound can be deduced similarly and we conclude Lemma \ref{initialLinftyLem2}. 
Here, the small constant $\de>0$ depends only on $\al$, and hence, all the constants $C_\al=C(\al,\ga,s)>0$ (depending on the fixed $\de=\de(\al)>0$ and $N=N(\al,\ga,s)>0$) are independent of $\vpi,\ve,T_1$. 
\end{proof}

\subsection{Energy inequality for level functions}
In this Subsection, we will derive the main energy estimates for level functions, including regular velocity and regular spatial-time estimates. The velocity regular energy can be given by \eqref{8maininOmega} while the time-space regularity is given in Lemma \ref{LemBesovreguLevel}. 
Then we mimic Lemma \ref{energyinterLem} to derive the energy interpolation inequality and denote the same energy functional $\E^{}_{p}(K)$ as in \eqref{Ep} by 
\begin{align}\label{Ep2}\notag
	\E^{}_{p}(K):&=
	\|f^{(l)}_{K,+}\|^2_{L^\infty_tL^2_{x}(\R^3_x)L^2_v}
	+\|f^{(l)}_{K,+}\|_{L^2_{t}L^2_x(\Omega)L^2_D}^2\\
	&\notag\quad
	+\vpi\|[\wh{C}^{}_0\<v\>^{{4}}f^{(l)}_{K,+},\<v\>^{{2}}\na_vf^{(l)}_{K,+}]\|^2_{L^2_{t,x,v}([T_1, T_2]\times\Omega\times\R^3_v)}\\
	&\quad+\frac{1
	}{C_0\max\{C_\infty^{2p-2},1\}}
	\Big\|\int_{\R^3_v}\1_{[T_1,T_2]}\<v\>^{-10}(f^{(l)}_{K,+})^\frac{2}{p}\,dv\Big\|^q_{B^{s',2}_p(\R^{1+3}_{t,x})}. 
\end{align}

\begin{Lem}[Energy inequality]
	\label{energyinterLem1}
	Assume the same conditions as in Lemmas \ref{initialLinftyLem2} and \ref{interLem}. 
Let $s'\in(0,1)$ be a small constant chosen in \eqref{LemBesovreguLevel}. Let $[T_1,T_2]$ be the underlying time interval. 
Suppose $\Psi=\mu^{\frac{1}{2}}+\psi\ge0$ and $\vp$ satisfy 
\begin{align*}
	\begin{aligned}
		\|[\<v\>^{l}\psi,\<v\>^l\vp]\|_{L^\infty_tL^\infty_x(\Omega)L^\infty_v}+\|\vp\|_{L^2_tL^2_{x}(\R^3_x)L^2_v}&\le \de_0,\\
		\|\<v\>^{l_0+l-2}f\|^2_{L^2_tL^2_{x}(\R^3_x)L^2_v}\le C_1,\quad
		\|\<v\>^l_\de f\|_{L^\infty_tL^\infty_x(\Omega)L^\infty_v}&=C_\infty,\\
		\|\<v\>^{l}_\de f_1\|_{L^\infty_tL^\infty_x(\Omega)L^\infty_v}&=K_1,
	\end{aligned}
\end{align*}
with some $C_1,C_\infty,K_1>0$ and sufficiently small $\de_0\in(0,1)$. Assume that $f$ solves equation \eqref{diff1}.
		Then 
	\begin{align}\label{724p}
		&\notag\|f^{(l)}_{K,+}\|_{L^\infty_tL^2_{x,v}(\R^6)}^2+\|f^{(l)}_{K,+}\|^2_{L^2_tL^2_{x}(\Omega)L^2_D}
		+\vpi\|[\wh{C}^{}_0\<v\>^{{4}}f^{(l)}_{K,+},\<v\>^{{2}}\na_vf^{(l)}_{K,+}]\|^2_{L^2_tL^2_{x}(\Omega)L^2_v}\\
		&\notag\quad+\frac{1}{C_0\max\{C_\infty^{2p-2},1\}}
		\Big\|\int_{\R^3}\1_{[T_1,T_2]}\<v\>^{-10}(f^{(l)}_{K,+})^2\,dv\Big\|_{B^{s',2}_p(\R^{4}_{t,x})}^p\\
		&\quad\le C(1+C_1+K_1)^{C}\sum_{i=1}^4\frac{\ga_i\E_p(M)^{\beta_i}}{(K-M)^{\al_i}}.
	\end{align}
	Here $C=C(\al,s,s',p,\ga,l)>0$ is some large constant independent of $T_1,T_2$. The parameters $\beta_i>1$ and $\ga_i,\al_i>0$, depending on $s,s',p$, are given by \eqref{albexp}. 

	Furthermore, the estimate \eqref{724p} holds for $h:=-f$, with $f^{(l)}_{K,+}$ replaced by $(-f)^{(l)}_{K,+}$. The functional $\E_p$ is given by \eqref{Ep2}. 

\end{Lem}
\begin{proof}
For the first three terms of \eqref{724p}, taking combination $ \eqref{8maininOmega}+\ka\times\eqref{8maininOmegac}$ with sufficiently small $\ka=\ka(\al)>0$ ($\ka$ is used to eliminate the right-hand terms of \eqref{8maininOmegac}), and taking supremum $T\in[T_1,T_2]$, we can obtain (the time norm is taken on $[T_1,T_2]$)
\begin{align}\label{8maininOmega2}
	&\notag\frac{1}{2}\|f^{(l)}_{K,+}\|^2_{L^\infty_tL^2_x(\Omega)L^2_v}
		+\frac{\al}{4(1-\al)}\|f^{(l)}_{K,+}\|^2_{L^2_tL^2_{x,v}(\Si_+)}+c_0\|f^{(l)}_{K,+}\|_{L^2_tL^2_x(\Omega)L^2_D}^2\\
		&\quad\notag+\ka\|f^{(l)}_{K,+}\|_{L^\infty_tL^2_x(\ol\Omega^c)L^2_v}^2
		+\ka\|Pf^{(l)}_{K,+}\|^2_{L^2_tL^2_x(\ol\Omega^c)L^2_v}
		+\ka K\|\<v\>^{-l}_\de P^2f^{(l)}_{K,+}\|_{L^1_tL^1_x(\ol\Omega^c)L^1_v}\\
	&\quad\notag+2\vpi\|[\wh{C}^{}_0\<v\>^{{4}}f^{(l)}_{K,+},\<v\>^{{2}}\na_vf^{(l)}_{K,+}]\|^2_{L^2_tL^2_{x}(\Omega)L^2_v}
	+2\vpi K\wh{C}^2_0\|\<v\>^{{-l+8}}f^{(l)}_{K,+}\|_{L^1_tL^1_x(\Omega)L^1_v}\\
	&\quad\le C_{\al}\|\<v\>^{2}f^{(l)}_{K,+}\|_{L^2_tL^2_x(\Omega)L^2_v}^2
	+C_{\al}\big(\de_0+K+NK_1\big)\|\<v\>^{-2}f^{(l)}_{K,+}\|_{L^1_tL^1_x(\Omega)L^1_v},  
\end{align}
Since \eqref{inidiffLinfty1} implies $C_\infty=\|\<v\>^l_\de f\|_{L^\infty_{t,x,v}([T_1,T_2]\times\R^3_x\times\R^3_v)}<\infty$, the time-space Besov regular estimate is given in \eqref{736}: 
\begin{align}\label{736xxp}
	&\notag\Big\|\int_{\R^3}\1_{[T_1,T_2]}\<v\>^{-10}(f^{(l)}_{K,+})^2\,dv\Big\|_{B^{s',2}_p(\R^{4}_{t,x})}^p
	\le C_\al\Big(\|\<v\>^{-2}[f^{(l)}_{K,+}(T_1),f^{(l)}_{K,+}(T_2)]\|_{L^2_{x,v}(\R^6)}^{2p}\\
	&\notag\quad+C_\infty^{2p-2}\|\1_{[T_1,T_2]}\<v\>^{-2p}f^{(l)}_{K,+}\|_{L^2_{t,x,v}(\R^{7})}^{2}
	+\|\<v\>^{\frac{(\ga+2s)_+}{2}}f^{(l)}_{K,+}\|^{2p}_{L^2_tL^2_x(\Omega)L^2_v}\\
	&\notag\quad+(1+K+K_1)^p\|\<v\>^{-2}f^{(l)}_{K,+}\|^{p}_{L^1_tL^1_x(\Omega)L^1_v}
	+\vpi^p\|[\<v\>^{{3}}f^{(l)}_{K,+},\<v\>\na_vf^{(l)}_{K,+}]\|_{L^2_tL^2_x(\Omega)L^2_v}^{2p}\\
	&\quad+\|Pf^{(l)}_{K,+}\|_{L^2_tL^2_{x}(\ol\Omega^c)L^2_{v}}^{2p}
	+K^p\|\<v\>^{-l}P^2f^{(l)}_{K,+}\|^p_{L^1_tL^1_{x}(\ol\Omega^c)L^1_{v}}\Big), 
\end{align}
for some $C=C(\al,\ga,s,l,p)>0$ ($\de=\de(\al)>0$ is already chosen depending on $\al$). 
Moreover, the extra terms with exponent $p$ in \eqref{736xxp} can be controlled by \eqref{8maininOmega2}:
\begin{multline}\label{736xxp2}
	\Big\|\int_{\R^3}\1_{[T_1,T_2]}\<v\>^{-10}(f^{(l)}_{K,+})^2\,dv\Big\|_{B^{s',2}_p(\R^{4}_{t,x})}^p
	\le C_\al\Big(\max\{C_\infty^{2p-2},1\}\|\1_{[T_1,T_2]}\<v\>^{-2p}f^{(l)}_{K,+}\|_{L^2_{t,x,v}(\R^{7})}^{2}\\
	+\|\<v\>^{2}f^{(l)}_{K,+}\|^{2p}_{L^2_tL^2_x(\Omega)L^2_v}+(1+K+K_1)^p\|\<v\>^{-2}f^{(l)}_{K,+}\|^{p}_{L^1_tL^1_x(\Omega)L^1_v}\Big). 
\end{multline}

\smallskip\noindent{\bf Putting estimates together and controlling $L^1,L^2$ norms.}
Choose a large constant $C_0>0$ depending only on the constant $C_\al=C(\al,\ga,s,l,p)>0$ in \eqref{736xxp2}. 
Then the linear combination $\eqref{8maininOmega2}+\max\{C_\infty^{2p-2},1\}^{-1}C_0^{-1}\times\eqref{736xxp2}$ gives 
\begin{align}\label{271xx}
	&\notag\|f^{(l)}_{K,+}\|^2_{L^\infty_tL^2_{x}(\Omega)L^2_v}
	+c_\al\|f^{(l)}_{K,+}\|_{L^\infty_tL^2_{x}(\ol\Omega^c)L^2_v}^2
	+c_0\|f^{(l)}_{K,+}\|_{L^2_tL^2_x(\Omega)L^2_D}^2
	+c_\al\|f^{(l)}_{K,+}\|_{L^2_tL^2_{x,v}(\Si_+)}^2\\
	&\notag\quad
	+2\vpi\|[\wh{C}^{}_0\<v\>^{{4}}f^{(l)}_{K,+},\<v\>^{{2}}\na_vf^{(l)}_{K,+}]\|_{L^2_tL^2_x(\Omega)L^2_v}^2
	+2\vpi K\wh{C}^2_0\|\<v\>^{{-l+8}}f^{(l)}_{K,+}\|_{L^1_tL^1_x(\Omega)L^1_v}\\
	&\notag\quad+c_\al\|Pf^{(l)}_{K,+}\|^2_{L^2_tL^2_x(\ol\Omega^c)L^2_v}
	+Kc_\al\|\<v\>^{-l}P^2f^{(l)}_{K,+}\|_{L^1_tL^1_x(\ol\Omega^c)L^1_v}\\
	&\notag\quad+\frac{1}{C_0\max\{C_\infty^{2p-2},1\}}\Big\|\int_{\R^3}\1_{[T_1,T_2]}\<v\>^{-10}(f^{(l)}_{K,+})^2\,dv\Big\|_{B^{s',2}_p(\R^{4}_{t,x})}^p\\
	&\notag\quad\le C\|\<v\>^2f^{(l)}_{K,+}\|_{L^2_tL^2_x(\Omega)L^2_v}^2+C(1+K+K_1)\|\<v\>^{-2}f^{(l)}_{K,+}\|_{L^1_x(\Omega)L^1_v}, \\
	&\quad+\frac{1}{\max\{C_\infty^{2p-2},1\}}\Big(\|\<v\>^2f^{(l)}_{K,+}\|_{L^2_tL^2_x(\Omega)L^2_v}^{2p}
	+(1+K+K_1)^p\|\<v\>^{-2}f^{(l)}_{K,+}\|^p_{L^1_tL^1_{x}(\Omega)L^1_v}\Big).
\end{align}
For the $L^1$ and $L^2$ norms within $\Omega$, we can apply Lemma \ref{interLem} with $m=4$ to deduce 
\begin{align}\label{720axx}
	\|\<v\>^2f^{(l)}_{K,+}\|^2_{L^2_{t,x,v}([T_1,T_2]\times\Omega\times\R^3_v)}
	&\le\frac{C\big(\max\{C_\infty^{2p-2},1\}\big)^{\frac{(1-\si)\beta_*\xi_*}{2p}}C_1^{\frac{(1-\beta_*)\xi_*}{4}}(\E_p(M))^{r_*}}{(K-M)^{\xi_*-2}},
	%
\end{align}
and
\begin{align}\label{720bxx}
	\|\<v\>^{-2}f^{(l)}_{K,+}\|_{L^1_{t,x,v}([T_1,T_2]\times\Omega\times\R^3_v)}
	&\le \frac{C\big(\max\{C_\infty^{2p-2},1\}\big)^{\frac{(1-\si)\beta_*\xi_*}{2p}}C_1^{\frac{(1-\beta_*)\xi_*}{4}}(\E_p(M))^{r_*}}{(K-M)^{\xi_*-1}},
\end{align}
where $l_0>0$ is a sufficiently large constant depending on $l,s,s',p$ (given in Lemma \ref{interLem}) and we put the constant $C_0$ (which is large but determined in \eqref{271xx}, and appeared in \eqref{Ep2} and \eqref{77}) inside constant $C$. 
Then $C=C(\al,s,s',p,\ga,l)>0$ here is independent of $C_1,C_\infty$. 
Moreover, since $0\le M<K$, we have 
\begin{align*}
	1\le \frac{K}{K-M}.
\end{align*}
	Thus, substituting \eqref{720axx} and \eqref{720bxx} into \eqref{271xx}, and choosing $\de_0\in(0,1)$ sufficiently small, 
	\begin{align*}
		&\notag\|f^{(l)}_{K,+}\|^2_{L^\infty_tL^2_{x}(\Omega)L^2_v}
	+c_\al\|f^{(l)}_{K,+}\|_{L^\infty_tL^2_{x}(\ol\Omega^c)L^2_v}^2
	+c_0\|f^{(l)}_{K,+}\|_{L^2_tL^2_x(\Omega)L^2_D}^2
	+c_\al\|f^{(l)}_{K,+}\|_{L^2_tL^2_{x,v}(\Si_+)}^2\\
	&\notag\quad+2\vpi K\wh{C}^2_0\|\<v\>^{{-l+8}}f^{(l)}_{K,+}\|_{L^1_tL^1_x(\Omega)L^1_v}
	+2\vpi\|[\wh{C}^{}_0\<v\>^{{4}}f^{(l)}_{K,+},\<v\>^{{2}}\na_vf^{(l)}_{K,+}]\|_{L^2_tL^2_x(\Omega)L^2_v}^2\\
	&\notag\quad+c_\al\|Pf^{(l)}_{K,+}\|^2_{L^2_tL^2_x(\ol\Omega^c)L^2_v}
	+Kc_\al\|\<v\>^{-l}P^2f^{(l)}_{K,+}\|_{L^1_tL^1_x(\ol\Omega^c)L^1_v}\\
	&\notag\quad+\frac{1}{C_0\max\{C_\infty^{2p-2},1\}}\Big\|\int_{\R^3}\1_{[T_1,T_2]}\<v\>^{-10}(f^{(l)}_{K,+})^2\,dv\Big\|_{B^{s',2}_p(\R^{4}_{t,x})}^p\\
		&\quad\le C(1+C_1+K_1)^{C}\sum_{i=1}^4\frac{\ga_i\E_p(M)^{\beta_i}}{(K-M)^{\al_i}},
	\end{align*}
	where we used \eqref{sibexi1}, i.e. $\frac{(1-\si)\beta_*\xi_*}{2p}<1$. Here, the parameters are given by 
	\begin{align}\label{albexp}
		\begin{aligned}
			&\ga_1=\max\{C_\infty^{2p-2},1\}^{\frac{(1-\si)\beta_*\xi_*}{2p}},\quad \ga_2=\frac{\ga_1K}{K-M},\quad\ga_3=1,\quad\ga_4=\big(\frac{K}{K-M}\big)^p,\\
			& \beta_1=\beta_2=r_*,\quad\beta_3=\beta_4=pr_*,\\
			&\al_1=\al_2=\xi_*-2,
			\quad\al_3=\al_4=p(\xi_*-2).
		\end{aligned}
	\end{align}
	Here, $\beta_i>1$ and $\al_i>0$ ($1\le i\le 4$), which can be seen from Lemma \ref{interLem} (i.e. \eqref{rxistar}).
	This implies \eqref{724p}. Moreover, the constants $C_0,C>0$ used above depends only on $s,s',p,\ga,l$.

Since $(-f)^{(l)}_{K,+}$ satisfies the same estimate as $f^{(l)}_{K,+}$ in Lemma \ref{Qes1Lem} and in Lemma \ref{interLem}, we can obtain the same estimates for $(-f)^{(l)}_{K,+}$, which concludes the proof of Lemma \ref{energyinterLem1}.
\end{proof}

\subsection{Improved \texorpdfstring{$L^\infty$}{L infty} estimate for vanishing data (De Giorgi iteration)}\label{Sec97}
With the above preparations on the regular time-space-velocity estimation of level functions, we are ready to use the De Giorgi method to deduce the improved $L^\infty$ estimate.
The Theorem \ref{ThmLinRe} below will give us the $L^\infty$ control of the solution $f$ to equation \eqref{diff1}. But before that, we need to give a control on the energy functional $\E_0:=\E_p(0)$ first. Until the end of this proof, if not specified, the underlying time interval is still $[T_1,T_2]$. 

\smallskip 
Note that $C_\infty:=\|\<v\>^l_\de f\|_{L^\infty_{t,x,v}([T_1,T_2]\times\Omega\times\R^3_v)}$ below is merely a notation, which is finite due to Lemma \ref{initialLinftyLem2}. 
Moreover, the following lemma and theorem are ``global" results that are independent on $T_1,T_2$ with assumption \eqref{Assfzp}.

\begin{Lem}
	\label{E0Lemp}
	Let $\al\in(0,1)$ and $\de=\de(\al)>0$ determined in Subsection \ref{Sec81}. 
	Let $\vpi\ge 0$, $0\le T_1<T_2=T_1+\de^3<\infty$, 
	and fix $l\ge\ga+10$, $-\frac{3}{2}<\ga\le 2$ and $0<s<1$. Let $p^\#$ be given in \eqref{ppsharp} and suppose $1<p<p^\#$. 
	Let $s'\in(0,1)$ be a sufficiently small constant depending on $p$, 
	and $l_0=l_0(l,s,s',p)>0$ be a sufficiently large constant (which can be chosen in Lemma \ref{interLem} with $m=4$).
	Suppose $\Psi=\mu^{\frac{1}{2}}+\psi\ge0$ and $\vp$ are given and satisfy 
	\begin{align*}
		\begin{aligned}
			&\|[\<v\>^l\psi,\<v\>^{l}\vp]\|_{L^\infty_{t,x,v}([T_1,T_2]\times\Omega\times\R^3_v)}\le \de_0,\\
			&\|\<v\>^{l}_\de f_1\|_{L^\infty_tL^\infty_x(\Omega)L^\infty_v}=K_1,\quad
			\|\vp\|_{L^2_tL^2_x(\Omega)L^2_v}=\ti C,\\
		\end{aligned}
	\end{align*}
	with some constant $\ti C>0$ and sufficiently small $\de_0\in(0,1)$.
	Assume that $f$ solves \eqref{diff1} with any $\eta\ge0$ in the sense of \eqref{weakfwhole} and satisfies 
	\begin{align}\label{Assfzp}\begin{aligned}
			&\|\<v\>^{l_0+l-2}f\|^2_{L^2_{t,x,v}([T_1,T_2]\times\Omega\times\R^3_v)}= C_1<\infty,\\
			&\|\<v\>^{2}f\|^2_{L^2_{t,x,v}([T_1,T_2]\times\Omega\times\R^3_v)}\le \de_0<1,\quad
			\|\<v\>^l_\de f\|_{L^\infty_{t,x,v}([T_1,T_2]\times\Omega\times\R^3_v)}=:C_\infty.
		\end{aligned}
	\end{align}
	Let $\E_0=\E_p(0)$ be given in \eqref{Ep2}. Then 
	\begin{align}\label{773pp}
		\E_0 \le C\DD+\DD^p,
	\end{align}
	where $C>0$ is a generic constant independent of $T_1,\de,\al,\vpi,\ve$, and $\DD$ is given by 
	\begin{align}\label{DDpp}\notag
		\DD:&=\|f\|_{L^\infty_tL^2_x(\R^3_x)L^2_v}^2
		+c_0\|f\|_{L^2_tL^2_x(\Omega)L^2_D}^2\\
		&\quad+\vpi\|[\wh{C}^{}_0\<v\>^{{4}}f,\<v\>^{{2}}\na_vf]\|_{L^2_tL^2_x(\Omega)L^2_v}^2+\|Pf\|^2_{L^2_tL^2_{x}(\ol\Omega^c)L^2_v}. 
	\end{align}
	The same estimate holds for $(-f)_{+}$ instead of $f_+$ within $\E_0$. 
\end{Lem}

\begin{Lem}[$L^\infty$ estimate for vanishing data]
	\label{Leminftypp}
	Suppose the same conditions as in Lemma \ref{E0Lemp}. 
	Let $\E_0=\E_p(0)$ be given in \eqref{Ep}. 
	For any $\eta>0$, there exists solution $f$ to equation \eqref{diff1}, which satisfies
	\begin{align}\label{Linfty2p}
		\|\<v\>^l_\de f\|_{L^\infty_tL^\infty_{x,v}(\R^6_{x,v})}
		&\le C(1+C_1+K_1)^{C}\max_{1\le i\le 4}(\lam_i)^{\frac{1}{\al_i}}\big(\DD+\DD^p\big)^{\frac{\beta_i-1}{\al_i}},
	\end{align}
	where $\DD$ denotes \eqref{DDpp}, $\al_i,\beta_i,\lam_i$ are given in \eqref{albexz}, and $\zeta=\zeta(\ga,s,p)$, $C=C(\al,l,\ga,s,s',p)>0$ are large constants independent of $\eta,T_1,\ve,\vpi$. 
\end{Lem}

Then we give the proof of the above three Lemmas and Theorem. 

\begin{proof}[Proof of Lemma \ref{E0Lemp}]
The proof is the application of $L^2$ energy estimates on $[T_1,T_2]$.

\smallskip \noindent{\bf Note that $K=0$ in $\E_p(0)$.}
We have 
	\begin{align*}
		&\notag\|f_{+}\|_{L^\infty_tL^2_x(\R^3_x)L^2_v}^2+\|f_{+}\|^2_{L^2_tL^2_{x}(\Omega)L^2_D}
			+\vpi\|[\wh{C}^{}_0\<v\>^{{4}}f_{+},\<v\>^{{2}}\na_vf_{+}]\|^2_{L^2_tL^2_{x}(\Omega)L^2_v}\\
		&\notag\quad\le C\|f\|_{L^\infty_tL^2_x(\R^3_x)L^2_v}^2
		+C\|f\|_{L^2_tL^2_x(\Omega)L^2_D}^2
		+\vpi\|[\wh{C}^{}_0\<v\>^{{4}}f,\<v\>^{{2}}\na_vf]\|_{L^2_tL^2_x(\Omega)L^2_v}^2\\
		&\quad\le C\DD, 
	\end{align*}
	with a generic constant $C>0$. 
	The Besov regular term is also given by \eqref{736} with $K=0$ while the extra $L^2$ energy can be controlled $\DD$ in \eqref{DDpp}:
	\begin{align*}
		&\notag\Big\|\int_{\R^3}\1_{[T_1,T_2]}\<v\>^{-10}(f_{+})^2\,dv\Big\|_{B^{s',2}_p(\R^{4}_{t,x})}^p
		\le C\Big(\|\<v\>^{-2}[f_{+}(T_1),f_{+}(T_2)]\|_{L^2_{x,v}(\R^6)}^{2p}\\
		&\notag\qquad+C_\infty^{2p-2}\|\1_{[T_1,T_2]}\<v\>^{-2p}f^{(l)}_{+}\|_{L^2_{t,x,v}(\R^{7})}^{2}
		+\|\<v\>^{\frac{(\ga+2s)_+}{2}}f_{+}\|^{2p}_{L^2_tL^2_x(\Omega)L^2_v}\\
		&\notag\qquad+\|\mu^{\frac{1}{80}}f_{+}\|^{p}_{L^2_tL^2_x(\Omega)L^2_v}
		+\vpi^p\|[\<v\>^{{3}}f_{+},\<v\>\na_vf_{+}]\|_{L^2_tL^2_x(\Omega)L^2_v}^{2p}+\|Pf_{+}\|_{L^2_tL^2_{x}(\ol\Omega^c)L^2_{v}}^{2p}\Big)\\
		&\quad\le C\max\{C_\infty^{2p-2},1\}(\DD+\DD^p).
	\end{align*}
	The term $\<v\>^{\frac{(\ga+2s)_+}{2}}f_{+}$ can be controlled by either instant energy or dissipation rate. 
	Combining the above two estimates and recalling the coefficient $\frac{1}{C_0\max\{C_\infty^{2p-2},1\}}$ in $\E_0$ \eqref{Ep} with a large constant $C_0=C_0(\al,\ga,s,l,p)$, we obtain 
\begin{align*}
	\E_p(0)&\le C\DD+\DD^p.
\end{align*}
This completes the proof of Lemma \ref{E0Lemp}.
\end{proof}

\begin{proof}[Proof of Lemma \ref{Leminftypp}]
	To prove Theorem \ref{Leminftypp}, similar to Lemma \ref{LinftyLemVanish}, we will use the De Giorgi iteration scheme. First, by Theorems \ref{LemLinRe} and \ref{ThmExtend}, equation \eqref{diff} has a solution in $\Omega$ and can be extended to equation \eqref{diff1} in the whole space. 
	Fix $K_0>0$ which is determined later in \eqref{K0111z}. Denote the increasing levels $M_k$ by
	\begin{align*}
		M_k:=K_0\big(1-\frac{1}{2^k}\big), \quad k=0,1,2,\cdots.
	\end{align*}
	Note that $M_0=0$, $\lim_{k\to\infty}M_k=K_0$, $M_k-M_{k-1}=K_02^{-k}>0$, and $\frac{M_k}{M_k-M_{k-1}}=2^k-1\le 2^k$. 
	Thus, applying Lemma \ref{energyinterLem1} with $(M,K)=(M_{k-1},M_k)$ and evaluating constants $\ga_i$ given in \eqref{albexp}, 
	\begin{align}\label{939xx}
		&\notag\|f^{(l)}_{M_k,+}\|_{L^\infty_tL^2_{x,v}(\R^6)}^2+\|f^{(l)}_{M_k,+}\|^2_{L^2_tL^2_{x}(\Omega)L^2_D}
		+\vpi\|[\wh{C}^{}_0\<v\>^{{4}}f^{(l)}_{M_k,+},\<v\>^{{2}}\na_vf^{(l)}_{M_k,+}]\|^2_{L^2_tL^2_{x}(\Omega)L^2_v}\\
		&\notag\quad+\frac{1}{C_0\max\{C_\infty^{2p-2},1\}}
		\Big\|\int_{\R^3}\1_{[T_1,T_2]}\<v\>^{-10}(f^{(l)}_{M_k,+})^2\,dv\Big\|_{B^{s',2}_p(\R^{4}_{t,x})}^p\\
		&\quad\le C(1+C_1+K_1)^{C}\sum_{i=1}^4\frac{\lam_i2^{k(\al_i+p)}\E_p(M_{k-1})^{\beta_i}}{(K_0)^{\al_i}}, 
	\end{align}
	where $C=C(\al,s,s',p,\ga,l)>0$ and the parameters $\lam_i,\al_i>0$ and $\beta_i>1$ are given by 
	\begin{align}\label{albexz}
		\begin{aligned}
			&\lam_1=\lam_2=\max\{C_\infty^{2p-2},1\}^{\frac{(1-\si)\beta_*\xi_*}{2p}},\quad\lam_3=\lam_4=1,\\
			& \beta_1=\beta_2=r_*,\quad\beta_3=\beta_4=pr_*,\\
			&\al_1=\al_2=\xi_*-2,
			\quad\al_3=\al_4=p(\xi_*-2).
		\end{aligned}
	\end{align}
	Then we can perform the De Giorgi iteration on sequence $\{\E_p(M_k)\}$. 
	Noting $\beta_i>1$, we write 
	\begin{equation}\label{Ekstarz}
		Q_0=\max_{1\le i\le 4}2^{\frac{\al_i+p}{\beta_i-1}}>1,\quad \E^*_k=\frac{\E_0}{(Q_0)^{k}},\quad \text{ for }k=0,1,2,\dots,
	\end{equation}
	as an artificial sequence, and denote the upper bound by 
	\begin{align}\label{K0111z}
		K_0:&= 
		\max_{1\le i\le 4}\Big((4\lam_iC_2)^{\frac{1}{\al_i}}(\E_0)^{\frac{\beta_i-1}{\al_i}}(Q_0)^{\frac{\beta_i}{\al_i}}\Big),
	\end{align}
	where 
	$\E_0=\E_p(0)$ is given by \eqref{Ep2}, 
	$g$ is given by \eqref{gDef}
	and $C_2={C(1+C_1)^{C}}>0$ is the constant in \eqref{939xx}. 
Then, by noting the left-hand of \eqref{939xx} is functional $\E_p(M_k)$ from \eqref{Ep2}, for any $k\ge 1$, one has
		\begin{align}\label{524az}
			\E_p(M_k)\le C_2\sum_{i=1}^4\frac{\lam_i2^{k(\al_i+p)}\E_p(M_{k-1})^{\beta_i}}{(K_0)^{\al_i}}. 
		\end{align}
	By \eqref{Ekstarz} and \eqref{K0111z}, we have $\E^*_0=\E_0$ and
	\begin{align}\label{Ekstar1z}
		\E^*_k &\notag= \frac{\E_0}{(Q_0)^{k}}
		=\frac{1}{4}\sum_{i=1}^4\frac{(\E^*_{k-1})^{\beta_i}(K_0)^{\al_i}\E_0}{(\E^*_{k-1})^{\beta_i}(K_0)^{\al_i}(Q_0)^{k}}\\
		&\notag\ge\frac{1}{4}\sum_{i=1}^4\frac{(\E^*_{k-1})^{\beta_i}\max_{1\le j\le 4}\Big((4\lam_jC_2)^{\frac{1}{\al_j}}(\E_0)^{\frac{\beta_j-1}{\al_j}}(Q_0)^{\frac{\beta_j}{\al_j}}\Big)^{\al_i}\E_0}{\big(\frac{\E_0}{(Q_0)^{k-1}}\big)^{\beta_i}(K_0)^{\al_i}(Q_0)^{k}}\\
		&\ge C_2\sum_{i=1}^4\frac{(\E^*_{k-1})^{\beta_i}\lam_i(Q_0)^{k(\beta_i-1)}}{(K_0)^{\al_i}}
		\ge C_2\sum_{i=1}^4\frac{\lam_i2^{k(\al_i+p)}(\E^*_{k-1})^{\beta_i}}{(K_0)^{\al_i}}.
	\end{align}
	By comparing \eqref{Ekstar1z} and \eqref{524az}, and comparison principle (note $\E_0=\E_0^*=\E_p(M_0)$ and $Q_0>1$), 
	\begin{align*}
		\E_p(M_k)\le \E_k^*\to 0\text{ as }k\to\infty. 
	\end{align*}
	Consequently, taking the limit $k\to\infty$, recall the functional $\E_p(M_k)$ in \eqref{Ep2}, we deduce 
	\begin{align*}
		\|f^{(l)}_{K_0,+}\|^2_{L^\infty_t([T_1,T_2])L^2_{x,v}(\R^3_x\times\R^3_v)}=0,
	\end{align*}
	where $K_0$ is given by \eqref{K0111z}. Thus, by \eqref{773pp}, on $[T_1,T_2]$, 
	\begin{align*}
		\|(\<v\>^l_\de f)_+\|_{L^\infty_{x,v}(\R^3_x\times\R^3_v)} \le K_0. 
	\end{align*}
	Since the Lemmas \ref{interLem}, \ref{energyinterLem1}, and \ref{E0Lemp} have their corresponding counterparts for $-f$, if we use $(-f)^{(l)}_{K,+}$ to replace $f^{(l)}_{K,+}$ in $\E_0$, then we have the same lower bound. In summary, using estimate \eqref{773pp} to control $\E_0$, we have 
\begin{align*}
	\|\<v\>^l_\de f\|_{L^\infty_tL^\infty_{x,v}(\R^6_{x,v})}
	&\le 
	C(1+C_1+K_1)^{C}\max_{1\le i\le 4}(\lam_i)^{\frac{1}{\al_i}}\big(\DD+\DD^p\big)^{\frac{\beta_i-1}{\al_i}}, 
\end{align*}
where the constants $C=C(\al,s,s',p,\ga,l)>0$ is independent of $T_1,\vpi,\eta,\ve$, and the parameters $\al_i,\beta_i,\lam_i$ are given in \eqref{albexz}. This implies \eqref{Linfty2p} and completes the proof of Lemma \ref{Leminftypp}.
\end{proof}

\subsection{\texorpdfstring{$L^\infty$}{L infty} estimate of full linear equation}
In this Subsection, we will combine the $L^\infty$ estimate for non-vanishing data and improved $L^\infty$ estimate for vanishing data to obtain the $L^\infty$ estimate of the modified linear equation \eqref{diff}. Moreover, by taking the limit $\eta\to0$, we obtain the existence of the ``original'' linear equation 
\begin{align}\label{linear1fp}
	\left\{
	\begin{aligned}
		& \pa_tf+ v\cdot\na_xf = \vpi Vf+ \Gamma(\Psi,f)+\Gamma(\vp,\mu^{\frac{1}{2}})\quad \text{ in } [T_1, T_2]\times\Omega\times\R^3_v, \\
		& f|_{\Si_-}=(1-\ve)Rf\qquad \text{ on }[T_1, T_2]\times\Si_-, \\
		& f(T_1,x,v)=0\qquad\qquad \text{ in }\Omega\times\R^3_v,
	\end{aligned}\right.
\end{align}

\begin{Thm}[$L^\infty$ estimate for linear equation] \label{ThmLinRe}
	Assume the same conditions as in Lemma \ref{E0Lemp}. 
Let $0\le T_1<T_2=T_1+\de^3$ and $N=N(\ga,s)>0$ be a large constant chosen in Lemma \ref{ReLinfty1Lem}. 
	Suppose $\Psi=\mu^{\frac{1}{2}}+\psi\ge 0$, $\vp=\vp_1+\vp_2$ and $f_{T_1}$ satisfy
\begin{align*}
\begin{aligned}
	\|[\<v\>^l\psi,\<v\>^l\vp_1,\<v\>^l\vp_2]\|_{L^\infty_t([T_1,T_2])L^\infty_{x}(\Omega)L^\infty_v}&= \de_0,\\
	\|\<v\>^{l}_\de f_{T_1}\|_{L^\infty_{x,v}(\Omega\times\R^3_v)} = \de_\infty,\quad
	\|\<v\>^{l_0+2l-2}f_{T_1}\|^2_{L^2_x(\Omega)L^2_v}
	& = \ti C_1,\\
	\|\<v\>^{l-2}f_{T_1}\|^2_{L^2_x(\Omega)L^2_v(\R^3_v)}+\|[\vp,\vp_1,\vp_2]\|^2_{L^\infty_t([T_1,T_2])L^2_{x}(\Omega)L^2_v}
	& = \de_1.
\end{aligned}
\end{align*}
	with constants $\ti C_1>0$ and sufficiently small $\de_0,\de_1,\de_\infty\in(0,1)$.
	Then the solution $f$ to \eqref{linear1fp} satisfies
	\begin{align}\label{Linftyxrefl}
		\|\<v\>^l_\de f\|_{L^\infty_t([T_1,T_2])L^\infty_{x,v}(\ol\Omega\times\R^3_v)}
		\le K_1
		+ C(1+\ti C_1+K_1)^{C}\Big((T_2-T_1)\de_1\Big)^{\zeta},
	\end{align}
	where $C=C(\al,l,\ga,s)>0$ and $\zeta=\zeta(s,\ga)>0$ are independent of $T_1$. Here, $K_1$ is given by 
	\begin{align*}
		K_1=
		\max\big\{\frac{1}{2}\|\<v\>^{l}_\de \vp_1\|_{L^\infty_t([T_1,T_2])L^\infty_x(\Omega)L^\infty_v},\,\|\<v\>^l_{\de}f_{T_1}\|_{L^\infty_{x}(\Omega)L^\infty_v}\big\}. 
	\end{align*}
\end{Thm}

\begin{proof}
	Until the end of this proof, if not specified, the underlying time interval is $[T_1,T_2]$. 	
	Note that 
	\begin{align}
		\label{T2T1minusde}
		\text{
			$T_2-T_1\le\de^3$\ and\ $\de=\de(\al)>0$ is a fixed small constant. }
	\end{align}
We begin with the modified equation \eqref{diff} with any $\eta>0$ and split $f=f_1+f_2$ as in \eqref{diffa} and \eqref{diffb}. 
	Applying Lemma \ref{ReLinfty1Lem} to $f_1$, we have 
	\begin{align}\label{Linftyf1a}
		\|\<v\>^l_\de f_1\|_{L^\infty_tL^\infty_{x}(\ol\Omega)L^\infty_v}
		\le K_1\equiv\max\big\{\frac{1}{2}\|\<v\>^{l}_\de \vp_1\|_{L^\infty_t([T_1,T_2])L^\infty_x(\Omega)L^\infty_v},\,
		\|\<v\>^l_{\de}f_{T_1}\|_{L^\infty_{x}(\Omega)L^\infty_v}\big\}. 
	\end{align}
	On the other hand, utilizing Theorem \ref{LemLinRe} (with $\phi=N\<v\>^{l-2}(f_1+f_2)$ therein) and Lemma \ref{extendreverThm} (for the part in $\Omega$) yields the $L^2$ estimate for $f_2$: for any $k\ge 0$, 
	\begin{multline}
		\label{L2reflesf2}
		\|\<v\>^kf_2\|_{L^\infty_tL^2_x(\R^3_x)L^2_v}^2
		+c_\al\|\<v\>^kf_2\|_{L^2_tL^2_{x,v}(\Si_+)}^2
		+c_0\|f_2\|_{L^2_tL^2_x(\Omega)L^2_D}^2
		+\eta\|\<v\>^{k+\frac{l}{2}}f_2\|_{L^2_tL^2_x(\Omega)L^2_v}^2\\
		+\vpi\|[\wh{C}^{}_0\<v\>^{{k+4}}f_2,\<v\>^{{k+2}}\na_vf_2]\|_{L^2_tL^2_x(\Omega)L^2_v}^2+\|\<v\>^kPf\|_{L^2_tL^2_x(\ol\Omega^c)L^2_v}^2\\
		\le
		C_{|T_2-T_1|}\|[\vp,N\<v\>^{k+l-2}f]\|_{L^2_tL^2_x(\Omega)L^2_v(\R^3_v)}^2,
	\end{multline}
Thus, 
\begin{align*}
	\|\<v\>^{l_0+l-2}f_2\|^2_{L^2_tL^2_x(\Omega)L^2_v}&\le C\|\<v\>^{l_0+l}f_2\|^2_{L^2_tL^2_x(\Omega)L^2_D}
	\le C\|[\vp,N\<v\>^{l_0+2l-2}f]\|_{L^2_tL^2_{x}(\Omega)L^2_v}^2,\\
	\|\<v\>^{-2}f_2\|^2_{L^2_tL^2_x(\Omega)L^2_v}&\le C\|f_2\|^2_{L^2_tL^2_x(\Omega)L^2_D}\le C\|[\vp,N\<v\>^{l-2}f]\|_{L^2_tL^2_{x}(\Omega)L^2_v}^2.
\end{align*}
For the $L^2$ norm of $f=f_1+f_2$, we have also from Theorem \ref{LemLinRe} that, 
\begin{multline}\label{L2reflesf}
	\|\<v\>^kf\|_{L^\infty_tL^2_x(\Omega)L^2_v}^2		+\|\<v\>^kf\|^2_{L^2_tL^2_{x,v}(\Si_+)}+\|\<v\>^kf\|_{L^2_tL^2_x(\Omega)L^2_D}^2\\
	\le
	C_{\al}\Big(\|\<v\>^kf(T_1)\|_{L^2_{x}(\Omega)L^2_v}^2+\|\vp\|_{L^2_tL^2_x(\Omega)L^2_v}^2\Big). 
\end{multline}
Combining the above two estimates and the assumption, we have 
\begin{align*}
		\|\<v\>^{l_0+l-2}f_2\|^2_{L^2_tL^2_x(\Omega)L^2_v}\le \ti C_1,\quad
		\|\<v\>^{-2}f_2\|^2_{L^2_tL^2_x(\Omega)L^2_v}\le \de_1.
\end{align*}
With these bounds, we can applying Lemma \ref{initialLinftyLem2} to $f_2$ and obtain the initial $L^\infty$ bound: 
	\begin{align*}
		\|\<v\>^{l}_\de f_2\|_{L^\infty_t([T_1,T_2])L^\infty_{x}(\ol\Omega)L^\infty_v}\le e^{C_{\al,\eta}\de^3}(NK_1+1).
	\end{align*}
	The problem is that the initial $L^\infty$ bound of $f_2$ in Theorem \ref{Leminftypp} depends on $\eta>0$; so it serves as the \emph{a priori} bound such that the following energy on the right-hand side is finite.
	For the improved $L^\infty$ bound of $f_2$, we denote it by 
	\begin{align*}
		C_\infty=\|\<v\>^{l}_\de f_2\|_{L^\infty_t([T_1,T_2])L^\infty_{x}(\ol\Omega)L^\infty_v}, 
	\end{align*}
	which is finite. 
	Then applying Lemma \ref{Leminftypp} to $f_2$, and recalling parameters $\al_i,\beta_i,\lam_i$ given by \eqref{albexz}, we have 
	\begin{align}\label{488xz}
\|\<v\>^{l}_\de f_2\|_{L^\infty_t([T_1,T_2])L^\infty_{x}(\ol\Omega)L^\infty_v}
		&\le {C(1+\ti C_1+K_1)^{C}}\max_{1\le i\le 4}(\lam_i)^{\frac{1}{\al_i}}(\DD^{\frac{1}{2}}+\DD^{p})^{\frac{\beta_i-1}{\al_i}},
	\end{align}
	where $\DD$ is given by \eqref{DDpp} (note that the $f$ in \eqref{DDpp} is now $f_2$ here), i.e.
	\begin{align*}
		\DD:&=\|f_2\|_{L^\infty_tL^2_x(\R^3_x)L^2_v}^2
		+c_0\|f_2\|_{L^2_tL^2_x(\Omega)L^2_D}^2\\
		&\quad+\vpi\|[\wh{C}^{}_0\<v\>^{{4}}f_2,\<v\>^{{2}}\na_vf_2]\|_{L^2_tL^2_x(\Omega)L^2_v}^2+\|Pf_2\|^2_{L^2_tL^2_{x}(\ol\Omega^c)L^2_v}, 
	\end{align*}
	 which, by using $L^2$ estimate \eqref{L2reflesf2} and \eqref{L2reflesf}, as well as \eqref{T2T1minusde}, satisfies
	\begin{align*}
		\DD&\le C_{\al}(T_2-T_1)\|[\vp,N\<v\>^{l-2}f]\|_{L^2_tL^2_x(\Omega)L^2_v(\R^3_v)}^2\\
		&\le C_{\al}(T_2-T_1)\big(\|\<v\>^{l-2}f(T_1)\|_{L^2_{x}(\Omega)L^2_v}^2+\|\vp\|_{L^\infty_tL^2_x(\Omega)L^2_v}^2\big)\\
		&\le C_\al(T_2-T_1)\de_1.
	\end{align*}
	Note that, we have fixed $p$, and the exponent $\frac{(1-\si)\beta_*\xi_*}{2p}$ in \eqref{albexz} is the same the one in \eqref{sibexi1}, and thus 
	\begin{align}\label{627xzw}
		\frac{(1-\si)\beta_*\xi_*}{2p}<1, \text{ and }\xi_*>2+\frac{r(1)-2}{r(p^\#)}.
	\end{align}
Therefore, by \eqref{albexz} and Lemma \ref{interLem}, we know that $\beta_i=\beta_i(s,p)>1$ are constants, and 
\begin{align*}
	&(\lam_1)^{\frac{1}{\al_1}}=(\lam_2)^{\frac{1}{\al_2}}
	=\max\{C_\infty^{2p-2},1\}^{\frac{(1-\si)\beta_*\xi_*}{2p(\xi_*-2)}},\\
	&(\lam_3)^{\frac{1}{\al_3}}=(\lam_4)^{\frac{1}{\al_4}}
	=1.
\end{align*}
Then we continue \eqref{488xz} to deduce 
	\begin{align}\label{inftyf2aaz}\notag
		\|\<v\>^{l}_\de f_2\|_{L^\infty_t([T_1,T_2])L^\infty_{x}(\ol\Omega)L^\infty_v}&\le {C(1+\ti C_1+K_1)^{C}}\max\{C_\infty^{2p-2},1\}^{\frac{(1-\si)\beta_*\xi_*}{2p(\xi_*-2)}}\DD^{\zeta}\\
		&\le {C(1+\ti C_1+K_1)^{C}}\max\{C_\infty^{2p-2},1\}^{\frac{(1-\si)\beta_*\xi_*}{2p(\xi_*-2)}}((T_2-T_1)\de_1)^{\zeta},
	\end{align}
	where $C=C(\al,l,\ga,s,s',p)>0$ and $\zeta=\zeta(s,s',p)>0$ are some constants. 
	Therefore, there are two cases as below:
	\begin{enumerate}[leftmargin=2em]
		\item if $C_\infty<1$, then we obtain the upper bound 
		\begin{align*}
			\|\<v\>^l_\de f_2\|_{L^\infty_t([T_1,T_2])L^\infty_x(\ol\Omega)L^\infty_{v}}
			\le {C(1+\ti C_1+K_1)^{C}}((T_2-T_1)\de_1)^{\zeta};
		\end{align*}
		
		\smallskip
		\item if $C_\infty\ge 1$, then \eqref{inftyf2aaz} implies 
		\begin{align}\label{750az}
			\|\<v\>^l_\de f_2\|_{L^\infty_t([T_1,T_2])L^\infty_x(\ol\Omega)L^\infty_{v}}
			&\le {C(1+\ti C_1+K_1)^{C}}C_\infty^{\frac{(2p-2)(1-\si)\beta_*\xi_*}{2p(\xi_*-2)}}((T_2-T_1)\de_1)^{\zeta}. 
		\end{align}
		From \eqref{627xzw} (or \eqref{sibexi1}) and the choice of $p^\#$ given in \eqref{ppsharp}, we deduce that for any $p\in(1,p^\#)$, the exponent satisfies 
		\begin{align*}
			\frac{(1-\si)\beta_*\xi_*}{2p}\frac{2p-2}{\xi_*-2}<\frac{2p-2}{\xi_*-2}<1, 
		\end{align*}
		which is a fixed universal constant. (These parameters depend only on fixed parameters $s,p$).
		Therefore, we can absorb $C_\infty$ on the right-hand side of \eqref{750az} by the left hand due to its definition. Then we obtain \eqref{inftyf2aa1uu} below with some different constants $C,\zeta>1$. Further, if we choose $\de_1>0$ sufficiently small (depending on $\al,\ga,s,l,|\Omega|$), then $C_\infty<1$, which reduces to the first case. 
	\end{enumerate}
	In summary, we obtain 
	\begin{align}\label{inftyf2aa1uu}
		\|\<v\>^l_\de f_2\|_{L^\infty_tL^\infty_x(\ol\Omega)L^\infty_{v}}
		&\le {C(1+\ti C_1+K_1)^{C}}((T_2-T_1)\de_1)^{\zeta}. 
	\end{align}
	Combining the $L^\infty$ estimates \eqref{Linftyf1a} and \eqref{inftyf2aa1uu}, we see that the solution $f^\eta$ to the modified equation \eqref{diff} satisfies \eqref{Linftyxrefl}. Together with \eqref{L2reflesf} we know that $f^\eta$ has $L^2$ and $L^\infty$ energy estimates on $[T_1,T_2]$ uniformly in $\eta>0$, and thus has a subsequence which has a weak-$*$ limit $f$. Since the modified equation \eqref{linear1fp} is linear, it's standard to write it in the weak form and take the weak-$*$ limit to deduce that the limit $f$ satisfies the ``original'' linear equation \eqref{linear1fp} (we will also consider the weak-$*$ limit for the \emph{nonlinear} case later in Section \ref{Sec11}, and one can refer to the details there). Moreover, the limit satisfies the same $L^\infty$ estimate \eqref{Linftyxrefl}. 
	This completes the proof of Theorem \ref{ThmLinRe}.
\end{proof}

\section{\texorpdfstring{$L^2$}{L2}--\texorpdfstring{$L^\infty$}{L infty} estimate for reflection boundary}\label{Sec11}
In Section \ref{Sec10} above, we obtained the $L^\infty$ estimate of the solution $f$ to linear Boltzmann equation with \emph{Maxwell} reflection boundary condition. Combining it with the $L^2$ estimate, we will derive the local and global existence of the nonlinear Boltzmann equation.

\subsection{Local nonlinear theory}
In this subsection, we will derive the local-in-time existence for the nonlinear Boltzmann equation in $\Omega$ with \emph{Maxwell} reflection boundary conditions.
For this purpose, we first consider the regularizing ($\vpi Vf$) equation 
\begin{align}\label{sec1Maxell2}
	\left\{
	\begin{aligned}
		& \pa_tf^\vpi+ v\cdot\na_xf^\vpi = \vpi Vf^\vpi+\Gamma(\mu^{\frac{1}{2}}+f^\vpi\chi_{\de_0}(\<v\>^lf^\vpi),f^\vpi)\\
		&\qquad\qquad\qquad+\Gamma(f^\vpi\chi_{\de_0}(\<v\>^lf^\vpi),\mu^{\frac{1}{2}})
\quad \text{ in } [T_1,T_2]\times\Omega\times\R^3_v, \\
		& f^\vpi(t,x,v)|_{\Si_-}=(1-\ve)Rf^\vpi\quad \text{ on }[T_1,T_2]\times\Si_-, \\
		& f^\vpi(T_1,x,v)=f_{T_1}\quad \text{ in }\Omega\times\R^3_v. 
	\end{aligned}\right.
\end{align}
where $\de_0>0$ is a given small constant and $l\ge \ga+10$ is fixed. Here $\chi_{\de_0}(f)$ is a cutoff function given by 
\begin{align*}
	\chi_{\de_0}(f)=\left\{\begin{aligned}
		&0\quad \text{ if }|f|>\de_0,\\
		&1\quad \text{ if }|f|\le\de_0.
	\end{aligned}\right. 
\end{align*}
We add such a cutoff function in order to automatically obtain the $L^\infty$ upper bound. 
To solve this equation \eqref{sec1Maxell2}, we let $\vpi>0$ be any small constant and $S:X\to X$ be the solution operator of equation:
\begin{align}\label{sec1Maxell3}
	\left\{
	\begin{aligned}
		& \pa_tf+ v\cdot\na_xf =\vpi Vf+\Gamma(\mu^{\frac{1}{2}}+\psi\chi_{\de_0}(\<v\>^l\psi),f)\\&\qquad\qquad\qquad+\Gamma(\psi\chi_{\de_0}(\<v\>^l\psi),\mu^{\frac{1}{2}})
		\quad \text{ in } [T_1,T_2]\times\Omega\times\R^3_v, \\
		& f|_{\Si_-}=(1-\ve)Rf\quad \text{ on }[T_1,T_2]\times\Si_-, \\
		& f(T_1,x,v)=f_{T_1}\quad \text{ in }\Omega\times\R^3_v.
	\end{aligned}\right.
\end{align}
That is, for any $\psi\in X$, we set $S\psi = f$. Here $X$ is the normed space defined by 
\begin{multline}\label{Xdef1}
	X:=\big\{f\in L^\infty_tL^2_{x,v}([T_1,T_2]\times\Omega\times\R^3_v)\,:\, \mu^{\frac{1}{2}}+f\ge 0,\ 
	\|f\|_{L^\infty_tL^2_{x,v}([T_1,T_2]\times\Omega\times\R^3_v)}\le \de_0, \\
	\|\<v\>^l_\de f\|_{L^\infty_t([T_1,T_2])L^\infty_{x,v}(\Omega\times\R^3_v)}\le \de_0
	\big\}, 
\end{multline}
equipped with norm $L^\infty_tL^2_{x,v}([T_1,T_2]\times\Omega\times\R^3_v)$, 
with some small $\de_0>0$.

\begin{Lem}\label{LemNonLocRe}
	Let $\al\in(0,1)$, $-\frac{3}{2}<\ga\le 2$, $s\in(0,1)$. Let $\de=\de(\al)>0$ be a small constant determined in Subsection \ref{Sec81} and $0\le T_1< T_2=T_1+\de^3$. 
	Let $\de_0>0$ be a sufficiently small constant (which can be chosen in Theorems \ref{LemLinRe} and \ref{ThmLinRe}). Let $l_0=l_0(l,\ga,s)>0$ be a large constant given in \eqref{ThmLinRe} (being $l_0+2l-2$ therein).
	Suppose $f_{T_1}$ satisfy $F_{T_1}=\mu+\mu^{\frac{1}{2}}f_{T_1}\ge 0$ and 
	\begin{align*}
		\begin{aligned}
			\|\<v\>^{l_0}f_{T_1}\|^2_{L^2_x(\Omega)L^2_v}=\ti C_1,\quad 
			\|\<v\>^l_\de f_{T_1}\|_{L^\infty_x(\Omega)L^\infty_v}=\ve_\infty,\quad\|\<v\>^{l-2}f_{T_1}\|^2_{L^2_x(\Omega)L^2_v}=\ve_1,
		\end{aligned}
	\end{align*}
	with sufficiently small $\ve_1,\ve_\infty\in(0,1)$ which depends only on $s,\de_0$ and is independent of $\ve$ (appeared in boundary condition), and a fixed $\ti C>0$.
	Then there exists a small $T_2>T_1$ and a solution $f$ to nonlinear equation
	\begin{align}\label{non4}
		\left\{
		\begin{aligned}
			 & \pa_tf+ v\cdot\na_xf = \Gamma(\mu^{\frac{1}{2}}+f,f)+\Gamma(f,\mu^{\frac{1}{2}})\quad \text{ in } [T_1,T_2]\times\Omega\times\R^3_v, \\
			 & f|_{\Si_-}=(1-\ve)Rf\quad \text{ on }[T_1,T_2]\times\Si_-, \\
			 & f(T_1,x,v)=f_{T_1}\quad \text{ in }\Omega\times\R^3_v,
		\end{aligned}\right.
	\end{align}
	in the sense of that, for any function $\Phi\in C^\infty_c(\R_t\times\R^3_x\times\R^3_v)$ satisfying $\Phi|_{\Si_+}=(1-\ve)R^*\Phi$ where $R^*$ is the dual reflection operator given by \eqref{reflectdual},
	\begin{multline*}
		(f(T_2),\Phi(T_2))_{L^2_x(\Omega)L^2_v}-(f,(\pa_t+v\cdot\na_x)\Phi)_{L^2_{t,x,v}([T_1,T_2]\times\Omega\times\R^3_v)}\\
		=(f_{T_1},\Phi(T_1))_{L^2_x(\Omega)L^2_v}
		+\big(\Gamma(\mu^{\frac{1}{2}}+f,f)+\Gamma(f,\mu^{\frac{1}{2}}),\Phi\big)_{L^2_{t,x,v}([T_1,T_2]\times\Omega\times\R^3_v)},
	\end{multline*}
	 which satisfies non-negativity $F=\mu+\mu^{\frac{1}{2}}f\ge 0$ and energy estimates
	\begin{align}\label{LinftyRelinear1}
		\begin{aligned}
			&\|f\|^2_{L^\infty_t([T_1,T_2])L^2_{x}(\Omega)L^2_v}+c_0\|f\|_{L^2_t([T_1,T_2])L^2_x(\Omega)L^2_D}^2\le \de_0,\\
			&\|\<v\>^l_\de f\|_{L^\infty_t([T_1,T_2])L^\infty_{x}(\Omega)L^\infty_v}\le \de_0.
		\end{aligned}
	\end{align}
\end{Lem}
\begin{proof}
	We consider $[T_1,T_2]$ as the underlying interval and use the fixed point theorem for equations \eqref{sec1Maxell3} with the cases $s\in(0,\frac{1}{2})$ and $s\in[\frac{1}{2},1)$ in the first and second steps, respectively. Once we obtain the solution to the nonlinear equation, we pass the limit $\vpi\to0$ in the third step. 
	The proof is similar to Theorem \ref{nonLocal}. 
	
	\smallskip\noindent{\bf Step 1. Contraction mapping.}
	Let $s\in(0,\frac{1}{2})$ in this step.
	For any $\vpi>0$, we let $S:X\to X$ be the solution operator of the equation \eqref{sec1Maxell3} by setting $S\psi = f$, whose existence is guaranteed by Theorem \ref{LemLinRe}.
	 Here $X$ is the normed space given by \eqref{Xdef1}. 
	We next prove that $S:X\to X$ is a contraction mapping. 
	
	\smallskip For any $\psi\in X$, we begin by proving that $S\psi \in X$. Since 
	\begin{align*}
		\|\psi\|_{L^\infty_tL^2_{x,v}([T_1,T_2]\times\Omega\times\R^3_v)} \le \de_0, \quad 
		\|\<v\>^l_\de \psi\|_{L^\infty_t([T_1,T_2])L^\infty_{x,v}(\Omega\times\R^3_v)} \le \de_0, 
	\end{align*}
	with some small constant $\de_0\in(0,\frac{1}{2})$, the existence of solution $f=S\psi$ to equation \eqref{Linfty2p} is given by Theorem \ref{LemLinRe}. Moreover, by $L^2$ estimate \eqref{913}, \eqref{hheqes5} and $L^\infty$ estimate \eqref{Linftyxrefl}, we obtain that, for any $k\ge 0$,
\begin{align}\notag\label{L2Lpenergy106}
	&\|\<v\>^kf\|_{L^\infty_tL^2_x(\R^3_x)L^2_v}^2
	+c_\al\|\<v\>^kf\|_{L^2_tL^2_{x,v}(\Si_+)}^2
	+\|\<v\>^kPf\|^2_{L^2_tL^2_{x}(\ol\Omega^c)L^2_v}\\
	&\quad\notag+c_0\|\<v\>^kf\|_{L^2_tL^2_x(\Omega)L^2_D}^2
	+\vpi\|[\wh{C}^{}_0\<v\>^{{k+4}}f,\<v\>^{{k+2}}\na_vf]\|_{L^2_tL^2_x(\Omega)L^2_v}^2\\
	&\qquad\le
	C_{k,T_2-T_1}\big(\|\<v\>^kf_{T_1}\|^2_{L^2_x(\Omega)L^2_v}+\|\vp\|_{L^2_tL^2_x(\Omega)L^2_v}^2\big),
\end{align} and 
\begin{align}\label{106zz}\notag
	\|\<v\>^l_\de f\|_{L^\infty_t([T_1,T_2])L^\infty_{x,v}(\ol\Omega\times\R^3_v)}
	&\le K_1
	+ C(1+\ti C_1+K_1)^{C}(T_2-T_1)^{\zeta}\\
	&\quad\quad\times\big(\|\<v\>^{l-2}f_{T_1}\|^2_{L^2_x(\Omega)L^2_v(\R^3_v)}+\|\psi\|^2_{L^\infty_t([T_1,T_2])L^2_{x}(\Omega)L^2_v}\big)^{\zeta},
\end{align}
with $C=C(\al,\ga,s,l)>0$ and $\zeta=\zeta(\ga,s)>0$ that are independent of $T_1,\vpi,\ve$. 
Here, $K_1$ is given by (by choosing $\vp_1=0$, $\vp=\vp_2$ in \eqref{Linftyxrefl})
\begin{align*}
	K_1=\|\<v\>^l_{\de}f_{T_1}\|_{L^\infty_{x}(\Omega)L^\infty_v}. 
\end{align*}
Therefore, choosing $T_2-T_1\le\de^3,\ve_1,\ve_\infty>0$ sufficiently small, depending on the constants in \eqref{L2Lpenergy106} and \eqref{106zz} (and hence, depending on $\al,\ga,s,l$), we deduce 
\begin{align*}
	&\|f\|_{L^\infty_tL^2_x(\R^3_x)L^2_v}
	+\|\<v\>^l_\de f\|_{L^\infty_tL^\infty_{x,v}(\R^6_{x,v})}\le \de_0,
\end{align*}
while the non-negativity of $\mu^{\frac{1}{2}}+f$ can be given by Theorem \ref{positivity}. These facts imply $f=S\psi\in X$. 

\smallskip Next we prove that $S$ is a contraction map with small time $T_2-T_1\le\de^3>0$. Let $\psi,\vp\in X$ and $f=S\psi$, $h=S\vp$. Then, by \eqref{sec1Maxell3}, $f-h$ satisfies 
\begin{align}\label{sec1Maxell4}
	\left\{
	\begin{aligned}
		& \pa_t(f-h)+ v\cdot\na_x(f-h) =\vpi V(f-h)+\Gamma(\mu^{\frac{1}{2}}+\psi\chi_{\de_0}(\<v\>^l\psi),f-h)
		\\&\qquad\quad+\Gamma(\psi\chi_{\de_0}(\<v\>^l\psi)-\vp\chi_{\de_0}(\<v\>^l\vp),\mu^{\frac{1}{2}}+h)
		\quad \text{ in } [T_1,T_2]\times\Omega\times\R^3_v, \\
		& (f-h)|_{\Si_-}=(1-\ve)R(f-h)\quad \text{ on }[T_1,T_2]\times\Si_-, \\
		& (f-h)(0,x,v)=0\quad \text{ in }\Omega\times\R^3_v.
	\end{aligned}\right.
\end{align}
Taking $L^2$ inner product of \eqref{sec1Maxell4} with $f-h$ over $\Omega\times\R^3_v$, we have 
\begin{multline}\label{666a}
	\frac{1}{2}\pa_t\|f-h\|_{L^2_x(\Omega)L^2_v}^2
	+\frac{1}{2}\|f-h\|_{L^2_{x,v}(\Si_+)}-\frac{1}{2}\|f-h\|_{L^2_{x,v}(\Si_-)}\\
	=\Big(\vpi V(f-h)+\Gamma(\mu^{\frac{1}{2}}+\psi,f-h)+\Gamma(\psi-\vp,h+\mu^{\frac{1}{2}})
	,f-h\Big)_{L^2_x(\Omega)L^2_v}. 
\end{multline}
Applying Lemma \ref{LemR}, \eqref{101} for the boundary term, Lemma \ref{LemRegu} for $\vpi Vf$, and estimates \eqref{64eq1} and \eqref{56} for the right-hand side of \eqref{666a}, we obtain 
\begin{align*}
	&\frac{1}{2}\pa_t\|f-h\|_{L^2_x(\Omega)L^2_v}^2+\frac{\ve}{2}\|f-h\|_{L^2_{x,v}(\Si_+)}
	+\frac{\vpi}{C}\|\<v\>^{{2}}\<D_v\>(f-h)\|_{L^2_x(\Omega)L^2_v}^2\\
	&\quad\le \big(-c_0+C\|\<v\>^{4}\psi\|_{L^\infty_x(\Omega)L^\infty_v}\big)\|f-h\|_{L^2_x(\Omega)L^2_D}^2
	+C\|\1_{|v|\le R_0}(f-h)\|_{L^2_x(\Omega)L^2_v}^2\\
	&\qquad+C\|\psi-\vp\|_{L^2_x(\Omega)L^2_v}\|\<v\>^{2}h\|_{L^\infty_x(\Omega)L^2_v}\|\<v\>^{2}(f-h)\|_{L^2_x(\Omega)H^{2s}_v}\\
	&\qquad+C\|\mu^{\frac{1}{10^4}}(\psi-\vp)\|_{L^2_x(\Omega)L^2_v}\|\mu^{\frac{1}{10^4}}(f-h)\|_{L^2_x(\Omega)L^2_v}\\
	&\quad\le
	-\frac{c_0}{2}\|f-h\|_{L^2_x(\Omega)L^2_D}^2
	+C_\vpi\|\<v\>^{4}h\|_{L^\infty_x(\Omega)L^\infty_v}^2\|\psi-\vp\|_{L^2_x(\Omega)L^2_v}^2\\
	&\qquad+\frac{\vpi}{2C}\|\<v\>^{2}\<D_v\>(f-h)\|^2_{L^2_x(\Omega)L^{2}_v}+C\|f-h\|^2_{L^2_x(\Omega)L^2_v},
\end{align*}
since $s<\frac{1}{2}$. 
Choosing $\de_0>0$ in \eqref{Xdef1} sufficiently small, we obtain 
\begin{align*}
	\frac{1}{2}\pa_t\|f-h\|_{L^2_x(\Omega)L^2_v}^2+\frac{\ve}{2}\|f-h\|^2_{L^2_{x,v}(\Si_+)}
	\le C_\vpi\|\psi-\vp\|_{L^2_x(\Omega)L^2_v}^2+C\|f-h\|^2_{L^2_x(\Omega)L^2_v}.
\end{align*}
Using Gr\"{o}nwall's inequality and choosing $T_2=T_2(\vpi)>T_1$ sufficiently small, we have 
\begin{align*}
	\|f-h\|_{L^\infty_t([T_1,T_2])L^2_x(\Omega)L^2_v}^2&\le (T_2-T_1)C_\vpi\|\psi-\vp\|_{L^\infty_t([T_1,T_2])L^2_x(\Omega)L^2_v}^2\\
	&\le \frac{1}{2}\|\psi-\vp\|_{L^\infty_t([T_1,T_2])L^2_x(\Omega)L^2_v}^2.
\end{align*}
This implies that $S:X\to X$ is a contraction map for a short time. Therefore, by Banach fixed point theorem, there exists $f=f^\vpi\in X$ such that 
\begin{align*}
	\|f^\vpi\|_{L^\infty_t([T_1,T_2])L^2_x(\Omega)L^2_v}^2 \le \de_0, \quad
	\|\<v\>^l_\de f^\vpi\|_{L^\infty_t([T_1,T_2])L^\infty_{x,v}(\Omega\times\R^3_v)} \le \de_0,
\end{align*}
and it satisfies equation \eqref{sec1Maxell2} in the sense of \eqref{weakfre}. 
The non-negativity of $\mu^{\frac{1}{2}}+f^\vpi$ can be derived from Theorem \ref{positivity}.

\smallskip \noindent{\bf Step 2. Strong Singularity.}
The proof is similar to the ``Step 2" in Theorem \ref{nonLocal}.
Let $s\in[\frac{1}{2},1)$ in this step. Truncate the collision kernel $b(\cos\th)$ as in \eqref{beta} and denote $\Ga_\eta$ by \eqref{Gaeta}.
Since $b_\eta$ has a weak singularity, we can apply the fixed-point arguments in Step 1 to obtain a small time $T=T(\vpi,\eta)>0$ and a weak solution $f_\eta$ to equation \eqref{sec1Maxell2} with $\Ga$ replaced by $\Ga_\eta$, i.e. $f_\eta$ satisfies $f_\eta(T_1,x,v)=f_{T_1}$ in $\Omega\times\R^3_v$, and solves 
\begin{align}\label{sec1Maxell2a}
	\left\{
	\begin{aligned}
		& \pa_tf_\eta+ v\cdot\na_xf_\eta = \vpi Vf_\eta+\Gamma(\mu^{\frac{1}{2}}+f_\eta\chi_{\de_0}(\<v\>^lf_\eta),f_\eta)\\
		&\qquad\qquad\qquad+\Gamma(f_\eta\chi_{\de_0}(\<v\>^lf_\eta),\mu^{\frac{1}{2}})
		\quad \text{ in } [T_1,T_2]\times\Omega\times\R^3_v, \\
		& f_\eta(t,x,v)|_{\Si_-}=(1-\ve)Rf_\eta\quad \text{ on }[T_1,T_2]\times\Si_-. 
	\end{aligned}\right.
\end{align}
Taking $L^2$ inner product of \eqref{sec1Maxell2a} with $f_\eta$ over $[T_1,T_2]\times\Omega\times\R^3_v$, and using Lemma \ref{LemRegu} for regularizing term $Vf$, \eqref{LemR} for boundary term, and Lemma \ref{GaetaesLem} for the collision terms, we obtain 
\begin{multline*}
	\pa_t\|f_\eta(t)\|_{L^2_x(\Omega)L^2_v}^2+\ve\|f_\eta\|_{L^2_{x,v}(\Si_+)}
	+\vpi\|[\wh{C}^{}_0\<v\>^{{4}}f_\eta,\na_v(\<v\>^{{2}}f_\eta)]\|_{L^2_x(\Omega)L^2_v}^2\\
	\le C(1+\de_0)\|f_\eta\|_{L^2_{x}(\Omega)L^2_D}^2+C\|f_\eta\|_{L^2_{x}(\Omega)L^2_v}^2
	\le \frac{\vpi}{2}\|\<v\>^{2}f_\eta\|^2_{L^2_{x}(\Omega)H^1_v}+C_\vpi\|f_\eta\|_{L^2_{x}(\Omega)L^2_v}^2,
\end{multline*}
where we used \eqref{kesD} to deduce 
\begin{align*}
	\|f\|^2_{L^2_D}\le \frac{\vpi}{2C}\|\<v\>^{2}f\|^2_{H^1_v}+C_\vpi\|f\|^2_{L^2_v}.
\end{align*} 
The term $H^1_v$ can now be absorbed by the regularizing term. Therefore, integrating over $[T_1,T_2]$ and choosing $T_2=T_2(\vpi)>T_1$ sufficiently small, we have 
\begin{align}\label{7132}
	\|f_\eta\|_{L^\infty_tL^2_x(\Omega)L^2_v}^2+\ve\|f_\eta\|^2_{L^2_tL^2_{x,v}(\Si_+)}
	+\vpi\|[\wh{C}^{}_0\<v\>^{{4}}f_\eta,\na_v(\<v\>^{{2}}f_\eta)]\|_{L^2_tL^2_x(\Omega)L^2_v}^2
	\le 2\|f_{T_1}\|_{L^2_x(\Omega)L^2_v}^2,
\end{align}
which is uniform in $\eta$. This implies that the solution $f_\eta$ can be extended to a time $T_2=T_2(\vpi)>T_1$ which is independent of $\eta$.

\smallskip 
For the $L^\infty_{t,x,v}$ estimate of $f_\eta$, we give a short proof for brevity as in the proof of Theorem \ref{nonLocal}; see also \cite[Section 7]{Alonso2022} or \cite[Section 8]{Cao2022b}. The main goal is to obtain an $L^\infty$ estimate of the level functions that is uniform in $\eta$ but depends on $\vpi$. (Note that, in Section \ref{Sec10}, the estimates are uniform in $\vpi$.)
In Lemmas \ref{GachideLem} and \ref{Qes1Lem}, the same estimates hold for $\Ga_\eta$, with constants independent of $\eta$. The modification of Lemma \ref{CollLevelLem} for $\Ga_\eta$ is already given in ``Step 2'' of Theorem \ref{nonLocal}. Therefore, using the same functional as in \eqref{Eppri} in Theorem \ref{nonLocal}
\begin{multline}\label{Epprizz}
	\E'_p(K):=\|f^{(l)}_{K,+}\|^2_{L^\infty_tL^2_{x,v}([T_1,T_2]\times\R^6_{x,v})}+\vpi\|\<v\>^2f^{(l)}_{K,+}\|_{L^2_tL^2_xH^1_v([T_1,T_2]\times\Omega\times\R^3_v)}^2\\
	+\frac{1}{C_0\max\{C_\infty^{2p-2},1\}}\Big\|\int_{\R^3_v}\1_{[T_1,T_2]}\<v\>^{-10}(f^{(l)}_{K,+})^2\,dv\Big\|_{B^{s',2}_p(\R^{4}_{t,x})}^p,
\end{multline}
and following the 
same calculations in Theorem \ref{ThmLinRe} (i.e. all the calculations in Section \ref{Sec10}),
we can obtain the $L^\infty$ estimate of $f_\eta$ (the same method for deriving \eqref{106zz}):
\begin{multline*}
	\|\<v\>^l_\de f\|_{L^\infty_tL^\infty_{x,v}(\R^6_{x,v})}
	\le K_1+ C(1+\ti C_1+K_1)^{C}(T_2-T_1)^{\zeta}\\
	\times\big(\|\<v\>^{l-2}f_{T_1}\|^2_{L^2_x(\Omega)L^2_v(\R^3_v)}+\|\psi\|^2_{L^\infty_t([T_1,T_2])L^2_{x}(\Omega)L^2_v}\big)^{\zeta},
\end{multline*}
for some $\zeta=\zeta(\ga,s)>0$. 
Note the constant depends on $\vpi>0$ because the energy functional \eqref{Epprizz} depends on $\vpi>0$. 
Then we choose $T_2=T_2(\vpi)>T_1$ and $\ve_1>0$ so small that 
\begin{align}\label{713a2}
	\|\<v\>^l_\de f_\eta\|_{L^\infty_t([T_1,T_2])L^\infty_{x,v}(\Omega\times\R^3_v)}\le \de_0. 
\end{align}

\smallskip Therefore, applying Banach-Alaoglu Theorem, $f_\eta$ is weakly-$*$ compact in the corresponding spaces in \eqref{7132} and \eqref{713a2}, and there exists a subsequence (still denote it by $f_\eta$) such that 
\begin{align}\label{weaklimiteta2}
	\begin{aligned}
		& f_\eta\rightharpoonup f\quad
		\text{ weakly-$*$ in $L^2_{t,x,v}([T_1,T_2]\times\Si_+)$ and $L^2_{t,x}H^1_v([T_1,T_2]\times\Omega\times\R^3_v)$}, \\
		& f_\eta\rightharpoonup f\quad \text{ weakly-$*$ in $L^2_x(\Omega)L^2_v$ for any $t\in[T_1,T_2]$},\\
		& f_\eta\rightharpoonup f\quad
		\text{ weakly-$*$ in $L^\infty_{t,x,v}([T_1,T_2]\times\Omega\times\R^3_v)$},
	\end{aligned}
\end{align}
as $\eta\to 0$, with some function $f$ satisfying \eqref{7132} and \eqref{713a2}. 
Rewriting equation \eqref{sec1Maxell2a} in the weak form: for any function $\Phi\in C^\infty_c(\R_t\times\R^3_x\times\R^3_v)$,
\begin{multline}\label{weaketa2}
	(f_\eta(T_2),\Phi(T_2))_{L^2_x(\Omega)L^2_v}-(f_\eta,(\pa_t+v\cdot\na_x)\Phi)_{L^2_{t,x,v}([T_1,T_2]\times\Omega\times\R^3_v)}+(f_\eta,\Phi)_{L^2_tL^2_{x,v}(\Si_+)}\\
	=(f_{T_1},\Phi(T_1))_{L^2_x(\Omega)L^2_v}
	+(1-\ve)(Rf_\eta,\Phi)_{L^2_tL^2_{x,v}(\Si_-)}\\
	+\big(\vpi Vf_\eta+\Gamma(\mu^{\frac{1}{2}}+f_\eta\chi_{\de_0}(\<v\>^lf_\eta),f_\eta)+\Gamma(f_\eta\chi_{\de_0}(\<v\>^lf_\eta),\mu^{\frac{1}{2}})
	,\Phi\big)_{L^2_{t,x,v}([T_1,T_2]\times\Omega\times\R^3_v)}.
\end{multline}
It suffices to obtain the limit for the collision terms, which follows from the estimate \eqref{collGaetaGa} in Theorem \ref{nonLocal}. 
That is 
\begin{align*}
	\lim_{\eta\to0}\big(\Gamma_\eta(\mu^{\frac{1}{2}}+f_\eta\chi_{\de_0}(\<v\>^lf_\eta),\mu^{\frac{1}{2}}+f_\eta)-\Gamma(\mu^{\frac{1}{2}}+f\chi_{\de_0}(\<v\>^lf),\mu^{\frac{1}{2}}+f),\Phi\big)_{L^2_{t,x,v}([T_1,T_2]\times\Omega\times\R^3_v)}=0. 
\end{align*}
Combining this with the weak-$*$ limit in \eqref{weaklimiteta2}, we can take $\eta\to0$ in \eqref{weaketa2} to deduce that $f$ is the weak solution to \eqref{sec1Maxell2} for the case $s\in[\frac{1}{2},1)$ in the sense of \eqref{weakfre}.

\smallskip\noindent{\bf Step 3. Convergence of $f^{\vpi}$.}
Let $s\in(0,1)$. Notice that the solution $f^\vpi$ to \eqref{sec1Maxell2} obtained in Steps 1 and 2 only exists for a small time $T_2=T_2(\vpi)>T_1$.
We will prove that the existence time $T_2>0$ can be independent of $\vpi>0$ and then one can pass the limit $\vpi\to 0$. 

\smallskip The term $f^{\vpi}\chi_{\de_0}(\<v\>^lf^\vpi)$ automatically satisfies 
\begin{align*}
	\|\<v\>^lf^{\vpi}\chi_{\de_0}(\<v\>^lf^\vpi)\|_{L^\infty_t([T_1,T_2])L^\infty_{x,v}(\Omega\times\R^3_v)} \le \de_0. 
\end{align*}
Thus, applying the $L^2$--$L^\infty$ estimates from \eqref{913}, \eqref{hheqes5}, and \eqref{Linftyxrefl}, we have the same estimates \eqref{L2Lpenergy106} and \eqref{106zz}, while the constants on the right-hand sides are independent of $\vpi>0$. 
Thus, we can choose $T_2-T_1\le\de^3,\ve_1,\ve_\infty>0$ so small (independent of $\vpi$) that
\begin{align}\label{fna2}\notag
		&\|f^{\vpi}\|_{L^\infty_t([T_1,T_2])L^2_x(\Omega)L^2_v}^2+\ve\|f^{\vpi}\|^2_{L^2_tL^2_{x,v}(\Si_+)}+c_0\|f^{\vpi}\|_{L^2_tL^2_x(\Omega)L^2_D}^2 \\
		&\quad+\vpi\|[\wh{C}^{}_0\<v\>^{{4}}f^{\vpi},\<v\>^{{2}}\na_vf^{\vpi}]\|_{L^2_tL^2_x(\Omega)L^2_v}^2
\le \de_0, 
\end{align}
and 
\begin{align}\label{fna3}
	\|\<v\>^lf^{\vpi}\|_{L^\infty_t([T_1,T_2])L^\infty_{x,v}(\Omega\times\R^3_v)}
	& \le \de_0.
\end{align}
Then it's the standard continuity argument to show that the existence time $T_2-T_1>0$ is independent of $\vpi>0$. 
The sequence $\{f^{\vpi}\}$ is bounded in the sense in \eqref{fna2} and \eqref{fna3}. By Banach-Alaoglu Theorem, there exists a subsequence $\{f^n\}\subset\{f^{\vpi}\}$ (for simplicity we can take $\vpi=\frac{1}{n}$) such that $\{f^n\}$ has a weak-$*$ limit $f$ as $n\to\infty$ satisfying
\begin{align}
	\label{iterass2}
	\begin{aligned}
			&\|f\|_{L^\infty_t([T_1,T_2])L^2_x(\Omega)L^2_v}^2+\ve\|f\|^2_{L^2_tL^2_{x,v}(\Si_+)}+c_0\|f\|_{L^2_tL^2_x(\Omega)L^2_D}^2\le \de_0, \\
			& \|\<v\>^l_\de f\|_{L^\infty_t([T_1,T_2])L^\infty_{x,v}(\Omega\times\R^3_v)}
		\le\de_0,
	\end{aligned}
\end{align}
where the weak-$*$ limit is taken in the sense that
\begin{align}\label{weaklimit1}
	\begin{aligned}
			& f^n\rightharpoonup f\quad \text{ weakly-$*$ in $L^2_{t,x}L^2_D([T_1,T_2]\times\Omega\times\R^3_v)$ and
		$L^\infty_{t,x,v}([T_1,T_2]\times\Omega\times\R^3_v)$}, \\
			& f^n\rightharpoonup f\quad \text{ weakly-$*$ in $L^2_x(\Omega)L^2_v$ for any $t\in[T_1,T_2]$}.
	\end{aligned}
\end{align}
Notice from \eqref{fna3} that $f^n\chi_{\de_0}(\<v\>^lf^n)=f^n$. 
Rewrite equation \eqref{sec1Maxell2} in the weak form: for any function $\Phi\in C^\infty_c(\R_t\times\R^3_x\times\R^3_v)$ satisfying $\Phi|_{\Si_+}=(1-\ve)R^*\Phi$,
where $R^*$ is the dual reflection operator given by \eqref{reflectdual},
the weak solution $f^{n}$ to equation \eqref{sec1Maxell2} satisfies
\begin{multline}\label{weak4}
	(f^{n}(T_2),\Phi(T_2))_{L^2_x(\Omega)L^2_v}-(f^{n},(\pa_t+v\cdot\na_x)\Phi)_{L^2_{t,x,v}([T_1,T_2]\times\Omega\times\R^3_v)}
	=(f_{T_1},\Phi(T_1))_{L^2_x(\Omega)L^2_v}\\
	+\big(\frac{1}{n}Vf^{n}+\Gamma(\mu^{\frac{1}{2}}+f^n,f^{n})+\Gamma(f^n,\mu^{\frac{1}{2}}),\Phi\big)_{L^2_{t,x,v}([T_1,T_2]\times\Omega\times\R^3_v)}. 
\end{multline}
To pass the weak-$*$ limit in \eqref{weak4}, we can apply the same estimates \eqref{716} and \eqref{86} in Theorem \ref{nonLocal} for $Vf$ and collision terms (using Lemma \ref{claimLem}).
Therefore, taking limit $n=n_j\to\infty$ in \eqref{weak4} and applying \eqref{weaklimit1}, we obtain
\begin{multline*}
	(f(T_2),\Phi(T_2))_{L^2_x(\Omega)L^2_v}-(f,(\pa_t+v\cdot\na_x)\Phi)_{L^2_{t,x,v}([T_1,T_2]\times\Omega\times\R^3_v)}\\
	=(f_{T_1},\Phi(T_1))_{L^2_x(\Omega)L^2_v}
	+\big(\Gamma(\mu^{\frac{1}{2}}+f,f)+\Gamma(f,\mu^{\frac{1}{2}}),\Phi\big)_{L^2_{t,x,v}([T_1,T_2]\times\Omega\times\R^3_v)}.
\end{multline*}
This implies that $f$ is the solution to \eqref{non4}. Estimate \eqref{iterass2} implies \eqref{LinftyRelinear1}.
The non-negativity of $F=\mu+\mu^{\frac{1}{2}}f$ can be derived from Theorem \ref{positivity}. This completes the proof of Lemma \ref{LemNonLocRe}.
\end{proof}

\subsection{Global nonlinear theory}
In this subsection, we are going to deduce the global-in-time existence of the full nonlinear Boltzmann equation with reflection boundary without $\ve$: 
\begin{align}\label{non5}
	\left\{
	\begin{aligned}
		 & \pa_tf+ v\cdot\na_xf = \Gamma(\mu^{\frac{1}{2}}+f,f)+\Gamma(f,\mu^{\frac{1}{2}})\quad \text{ in } (0,\infty)\times\Omega\times\R^3_v, \\
		 & f|_{\Si_-}=Rf\quad \text{ on }(0,\infty)\times\Si_-, \\
		 & f(0,x,v)=f_0\quad \text{ in }\Omega\times\R^3_v.
	\end{aligned}\right.
\end{align}
For this purpose, we begin with the nonlinear Boltzmann equation with a modified boundary:
\begin{align}\label{non6}
	\left\{
	\begin{aligned}
		 & \pa_tf+ v\cdot\na_xf = \Gamma(\mu^{\frac{1}{2}}+f,f)+\Gamma(f,\mu^{\frac{1}{2}})\quad \text{ in } (0,\infty)\times\Omega\times\R^3_v, \\
		 & f|_{\Si_-}=(1-\ve)Rf\quad \text{ on }(0,\infty)\times\Si_-, \\
		 & f(0,x,v)=f_{0}\quad \text{ in }\Omega\times\R^3_v.
	\end{aligned}\right.
\end{align}
The local-in-time existence of equation \eqref{non6} is given by Lemma \ref{LemNonLocRe}. Then we can apply Theorem \ref{ThmLinRe} with $\vpi=0$. 

\begin{Thm}\label{LemNonRe}
	Assume that $\Omega$ is bounded and $f_0$ satisfies the conservation law \eqref{conservationlaw}. Fix $\al\in(0,1)$, $-\frac{3}{2}<\ga\le 2$, $s\in(0,1)$, $l\ge \ga+10$, $\ti C>0$, and let $l_0=l_0(s,l)>0$ be a large constant. Fix a small $\de=\de(\al)>0$ determined in Subsection \ref{Sec81}. 
	Suppose the initial data $f_0$ satisfy $F_0=\mu+\mu^{\frac{1}{2}}f_0\ge 0$ and 
	\begin{align*}
			\|\<v\>^{l_0}f_0\|_{L^2_x(\Omega)L^2_v}=\ti C,\quad 
			\|\<v\>^l_\de f_0\|_{L^\infty_x(\Omega)L^\infty_v}=\ve_\infty,\quad\|\<v\>^{l-2}f_0\|_{L^2_x(\Omega)L^2_v}= \ve_1,
	\end{align*}
	with sufficiently small $\ve_1,\ve_\infty\in(0,1)$ depending on $\al,\ga,s,l,\ti C$.
	Then there exists a global-in-time solution $f$ to equation \eqref{non5} such that $F=\mu+\mu^{\frac{1}{2}}f\ge 0$ and
	\begin{align}
		\label{980}
		\|\<v\>^l_\de f\|_{L^\infty_tL^\infty_{x,v}(\R^6_{x,v})}
	&\le \ve_\infty+C(1+\ti C)^{C}\ve_1^{\zeta}, 
	\end{align}
	for some constants $C=C(\al,\ga,s,l,q)>0$ and $\zeta=\zeta(\ga,s)>0$.
\end{Thm}
\begin{proof}
The local-in-time existence is given by Lemma \ref{LemNonLocRe}. In this proof, we only give the \emph{a priori} estimates. Moreover, the non-negativity of $F=\mu+\mu^{\frac{1}{2}}f$ is given in Theorem \ref{positivity}. We use $[0,T]$ as the underlying time interval below. 

\medskip
\noindent{\bf Step 1. Global existence for $\ve> 0$.}
Let $f$ be the local-in-time solution to \eqref{non6}. 
Assume the \emph{a priori} assumption
\begin{align}\label{iterass16}
	\begin{aligned}
		\|\<v\>^l_\de f\|_{L^\infty_t(\R_t)L^\infty_{x}(\Omega)L^\infty_v}+\|\<v\>^{l-2}f\|^2_{L^\infty_t(\R_t)L^2_x(\Omega)L^2_v} & \le \de_0,
	\end{aligned}
\end{align}
with some fixed small $\de_0\in(0,1)$. 
Then, by the $L^2$ estimate in Theorem \ref{L2globalThm} (i.e. \eqref{L2esLarge} and \eqref{L2esLargea}), for any $T>s>0$ and $k\ge 0$, we have 
\begin{align*}
	\begin{aligned}
		e^{c_0t}\|\<v\>^{l-2}f(T)\|^2_{L^2_x(\Omega)L^2_v}
		&\le C\|\<v\>^{l-2}f_0\|^2_{L^2_x(\Omega)L^2_v},\\
		\|\<v\>^{l_0}f(T)\|^2_{L^2_x(\Omega)L^2_v}
		&\le C\|\<v\>^{l_0}f_0\|^2_{L^2_x(\Omega)L^2_v}, 
	\end{aligned}
\end{align*}
and 
\begin{align}\label{L2esglo}
	\|\<v\>^kf\|^2_{L^\infty_t([s,T])L^2_x(\Omega)L^2_v}
		+\|\<v\>^kf\|_{L^2_t([s,T])L^2_x(\Omega)L^2_v}^2
	\le C\|\<v\>^kf|_{t=s}\|^2_{L^2_x(\Omega)L^2_v}, 
\end{align}
with some constant $c_0>0$. 
Thus, the assumptions in Theorem \ref{ThmLinRe} (for linear equation) are satisfied with any $[T_1,T_2]\subset\R_t$ satisfying $T_2-T_1\le\de^3$ and $\psi\equiv\vp\equiv\vp_2:=f$ and $\vp_1=0$, i.e. 
\begin{align*}
\begin{aligned}
	\|\<v\>^l_\de f\|_{L^\infty_t([T_1,T_2])L^\infty_{x}(\Omega)L^\infty_v}&\le \de_0,\\
	\|\<v\>^{l}_\de f_{T_1}\|_{L^\infty_{x,v}(\Omega\times\R^3_v)} \le\de_0,\quad
	\|\<v\>^{l_0}f_{T_1}\|^2_{L^2_x(\Omega)L^2_v}
	& \le C\ti C,\\
	\|\<v\>^{l-2}f_{T_1}\|^2_{L^2_x(\Omega)L^2_v(\R^3_v)}+\|f\|^2_{L^\infty_t([T_1,T_2])L^2_{x}(\Omega)L^2_v}
	& \le \ve_1 Ce^{-c_0T_1}.
\end{aligned}
\end{align*}
Note that the current constant $l_0$ is greater than the one in Theorem \ref{ThmLinRe}. 
Therefore, applying Theorem \ref{ThmLinRe} with $\de_\infty\le \de_0$, $\de_1\le \ve_1Ce^{-c_0j}$ and $\ti C_1=C\ti C$ therein, 
we deduce the $L^\infty$ estimates in time interval $[T_1,T_2]\subset\R_t$ satisfying $T_2-T_1\le\de^3$ with a fixed constant $\de=\de(\al)>0$: 
\begin{align*}
	\|\<v\>^l_\de f\|_{L^\infty_t([T_1,T_2])L^\infty_{x,v}(\ol\Omega\times\R^3_v)}
		\le K_1
		+ C(1+\ti C_1+K_1)^{C}\Big(\de^3\ve_1 Ce^{-c_0T_1}\Big)^{\zeta},
	\end{align*}
	where $K_1$ is given by 
	\begin{align*}
		K_1=\|\<v\>^l_{\de}f_{T_1}\|_{L^\infty_{x}(\Omega)L^\infty_v}\le \de_0<1, 
	\end{align*}
with some constant $C=C(\al,l,\ga,s)>0$ and $\zeta=\zeta(\ga,s)>0$ that are independent $T_1$. 
Repeating such estimate on $[0,\de^3]$, $[\de^3,2\de^3]$, $\dots$, $[j\de^3,(j+1)\de^3]$ $(j\ge 0)$, we have 
\begin{align*}
	&\|\<v\>^l_\de f\|_{L^\infty([j\de^3,(j+1)\de^3])L^\infty_{x,v}(\ol\Omega\times\R^3_v)}\\
	&\quad\le\|\<v\>^{l}_\de f|_{t=j\de^3}\|_{L^\infty_{x}(\Omega)L^\infty_v}
	+ C(1+\ti C)^{C}\de^{3\zeta}\ve_1^{\zeta}e^{-c_0j\zeta}\\
	&\quad\le \cdots\le \|\<v\>^{l}_\de f|_{t=0}\|_{L^\infty_{x,v}(\Omega\times\R^3_v)}
	+ C(1+\ti C)^{C}\de^{3\zeta}\ve_1^{\zeta}\sum_{k=0}^je^{-c_0k\zeta}\\
	&\quad\le \ve_\infty+C(1+\ti C)^{C}\de^{3\zeta}\ve_1^{\zeta}. 
\end{align*}
Note that $\de=\de(\al)>0$ is a fixed constant depending only on the accommodation coefficient $\al\in(0,1)$. 
Therefore, if we choose $\ve_\infty,\ve_1>0$ sufficiently small, which depends only on $\al,l,\ga,s$ and is independent of time, we deduce 
\begin{align}\label{Linfty2zz}
	\|\<v\>^l_\de f\|_{L^\infty_t(\R)L^\infty_{x,v}(\ol\Omega\times\R^3_v)}\le \ve_\infty+C_{\al}(1+\ti C)^{C}\ve_1^{\zeta}\le\de_0, 
\end{align}
which, together with \eqref{L2esglo} and a small $\ve_1>0$, closes the \emph{a priori} assumption \eqref{iterass16}

\medskip\noindent{\bf Step 2. Passing the limit $\ve\to 0$.}
We write $f=f^{\ve}$ to be the solution to \eqref{non6}, showing its dependence on $\ve$. Then by the $L^2$--$L^\infty$ estimates in Step 1, i.e. \eqref{L2esglo} and \eqref{Linfty2zz} for $f^\ve$, and by applying the Banach-Alaoglu Theorem, there exists weak-$*$ limit $f$ satisfying \eqref{L2esglo} and \eqref{Linfty2zz} in the sense that
\begin{align}\label{weaklimit7}
	\begin{aligned}
			& f^{\ve}\rightharpoonup f\quad \text{ weakly-$*$ in 
			$L^2_{t,x,v}([0,T]\times\Omega\times\R^3_v)$, 
$L^2_{t,x}L^2_D([0,T]\times\Omega\times\R^3_v)$
} \\
			& \qquad\qquad\qquad\qquad\qquad\qquad\text{ and
		$L^\infty_{t,x,v}([0,T]\times\Omega\times\R^3_v)$}, \\
			& f^{\ve}\rightharpoonup f\quad \text{ weakly-$*$ in $L^2_x(\Omega)L^2_v$ for any $t\in[0,T]$}.
	\end{aligned}
\end{align}
We rewrite the equation \eqref{non6} in the weak form:
for any function $\Phi\in C^\infty_c(\R_t\times\R^3_x\times\R^3_v)$ satisfying $\Phi|_{\Si_+}=R^*\Phi$ with dual reflection operator $R^*$ given by \eqref{reflectdual}, the weak solution $f^{\ve}$ to equation \eqref{non6} satisfies
\begin{multline}\label{weak6}
	(f^{\ve}(T),\Phi(T))_{L^2_x(\Omega)L^2_v}-(f^{\ve},(\pa_t+v\cdot\na_x)\Phi)_{L^2_{t,x,v}([0,T]\times\Omega\times\R^3_v)}
	+\ve(f^{\ve},\Phi)_{L^2_tL^2_{x,v}(\Si_+)}\\
	=(f_{0},\Phi(0))_{L^2_x(\Omega)L^2_v}
	+\big(\Gamma(\mu^{\frac{1}{2}}+f^{\ve},f^{\ve})+\Gamma(f^{\ve},\mu^{\frac{1}{2}}),\Phi\big)_{L^2_{t,x,v}([0,T]\times\Omega\times\R^3_v)}.
\end{multline}
Similar to estimate \eqref{86}, we have 
\begin{align*}
	&\big|\big(\Gamma(\mu^{\frac{1}{2}}+f^\ve,f^\ve)+\Gamma(f^\ve,\mu^{\frac{1}{2}})-\Gamma(\mu^{\frac{1}{2}}+f,f)-\Gamma(f,\mu^{\frac{1}{2}}),\Phi\big)_{L^2_{t,x,v}([0,T]\times\Omega\times\R^3_v)}\big|\\
	&\quad=\big|\big(\Gamma(\mu^{\frac{1}{2}}+f,f^\ve-f)+\Gamma(f^\ve-f,f^\ve+\mu^{\frac{1}{2}}),\Phi\big)_{L^2_{t,x,v}([0,T]\times\Omega\times\R^3_v)}\big|\\
	&\quad\le C\int^{T}_{0}\int_{\Omega}\|\Phi\|_{W^{2,\infty}_v}(1+\|\<v\>^{\ga+6}f\|_{L^\infty_v}+\|\<v\>^{\ga+6}f^\ve\|_{L^\infty_v})\|\<v\>^{\ga+4}(f^\ve-f)\|_{L^2_v}\,dxdt\\
	&\quad\le C\int^{T}_{0}\int_{\Omega}\|\Phi\|_{W^{2,\infty}_v}\|\<v\>^{\ga+4}(f^\ve-f)\|_{L^2_v}\,dxdt.
\end{align*}
Then by Lemma \ref{claimLem}, there exists a subsequence $\{f^{\ve_j}\}$ of $\{f^{\ve}\}$ such that the collision terms in \eqref{weak6} satisfy
\begin{multline*}
	\lim_{\ve_j\to\infty}\big(\Gamma(\mu^{\frac{1}{2}}+f^{\ve_j},f^{\ve_j})+\Gamma(f^{\ve_j},\mu^{\frac{1}{2}}),\Phi\big)_{L^2_{t,x,v}([0,T]\times\Omega\times\R^3_v)}\\=\big(\Gamma(\mu^{\frac{1}{2}}+f,f)+\Gamma(f,\mu^{\frac{1}{2}}),\Phi\big)_{L^2_{t,x,v}([0,T]\times\Omega\times\R^3_v)}.
\end{multline*}
Also, the $L^2$ boundary norm is bounded as in \eqref{L2esglo}, i.e. $\|\<v\>^kf^\ve\|_{L^2_{t}L^2_{x,v}(\Si_+)}\le C\|\<v\>^kf_0\|^2_{L^2_x(\Omega)L^2_v}.$
Therefore, taking limit $\ve=\ve_j\to\infty$ in \eqref{weak6} and using \eqref{weaklimit7}, we obtain
\begin{multline*}
	(f(T),\Phi(T))_{L^2_x(\Omega)L^2_v}-(f,(\pa_t+v\cdot\na_x)\Phi)_{L^2_{t,x,v}([0,T]\times\Omega\times\R^3_v)}\\
	=(f_{0},\Phi(0))_{L^2_x(\Omega)L^2_v}
	+\big(\Gamma(\mu^{\frac{1}{2}}+f,f)+\Gamma(f,\mu^{\frac{1}{2}}),\Phi\big)_{L^2_{t,x,v}([0,T]\times\Omega\times\R^3_v)}.
\end{multline*}
This implies that $f$ is the solution to \eqref{non5}, which satisfies \eqref{980}. This completes the proof of Theorem \ref{LemNonRe}. 
\end{proof}
\subsection{Proof of Theorem \ref{ThmReflectionMain}}\label{SecMainMaxwell}
The main Theorem \ref{ThmReflectionMain} follows from the combination of local-in-time existence in Lemma \ref{LemNonLocRe} and global-in-time existence as well as $L^\infty$ estimate in Theorem \ref{LemNonRe}. 
The $L^2$ energy estimates can be found in Theorem \ref{L2globalThm}.

\section{The a priori \texorpdfstring{$L^2$}{L2} decay theory}\label{Sec12}
In this section, we will prove the global \emph{a priori} $L^2$ estimate of the nonlinear Boltzmann equation with large-time decay by assuming the $L^\infty$ bound of the solution is small.
Note that this Section \ref{Sec12} is {\bf self-consistent} without using the $L^2$--$L^\infty$ estimate.

\subsection{Global \texorpdfstring{$L^2$}{L2} estimate}
The next Theorem \ref{L2globalThm} gives the global \emph{a priori} $L^2$ estimate of the nonlinear Boltzmann equation.
\begin{Thm}\label{L2globalThm}
	Assume that $\Omega\subset\R^3_x$ is a bounded open subset, $-\frac{3}{2}<\ga\le 2$ and $s\in(0,1)$. 
	Let $T>0$ and $f$ be the solution to Boltzmann equation \eqref{B1} in $[0,T]$.
	Suppose
	\begin{align}
		\label{L2global}
		\sup_{0\le t\le T}\|\<v\>^{\ga+10}f\|_{L^\infty_{x,v}(\Omega\times\R^3_v)}\le\de_0,
	\end{align}
	with some sufficiently small $\de_0>0$,
	which is a constant depending only on $\ga,s$.
	
	{{\smallskip}}\noindent (1) Let $k\ge 0$. Suppose $f$ satisfies the inflow boundary conditions \eqref{inflow}. Then there exists small $c_0>0$ such that if
	\begin{align*}
		\int^T_0\int_{\Si_-}|v\cdot n|e^{2c_0t}|\<v\>^kg|^2\,dS(x)dvdt<\infty,
	\end{align*}
	then we have
	\begin{multline}\label{L2es1}
		\|\<v\>^{k}f\|^2_{L^\infty_t([0,T])L^2_x(\Omega)L^2_v}
		+\|\<v\>^kf\|_{L^2_t([0,T])L^2_{x,v}(\Si_+)}
		+c_0\|\<v\>^kf\|_{L^2_t([0,T])L^2_x(\Omega)L^2_D}^2\\
		+c_0\|\<v\>^kf\|_{L^2_t([0,T])L^2_x(\Omega)L^2_v}^2
		\le C\|\<v\>^{k}f_0\|^2_{L^2_x(\Omega)L^2_v}+C\|\<v\>^kg\|_{L^2_t([0,T])L^2_{x,v}(\Si_-)}, 
	\end{multline}
	for some $C>0$. Moreover, we have the large-time behavior: for any $t\in[0,T]$,
	\begin{multline}\label{L2es1a}
		e^{c_0t}\|\<v\>^{k}f(t)\|^2_{L^2_x(\Omega)L^2_v}+
		\|e^{c_0s}\<v\>^kf\|_{L^2_{s}([0,t])L^2_{x,v}(\Si_+)}\\
		\le C\Big(\|\<v\>^{k}f_0\|^2_{L^2_x(\Omega)L^2_v}+\|e^{c_0s}\<v\>^kg\|_{L^2_{s}([0,t])L^2_{x,v}(\Si_-)}\Big), 
	\end{multline}
	
	{{\smallskip}}\noindent (2) 
	Suppose $f$ satisfies the conservation of mass for the initial data as in \eqref{conservationlaw} and the {Maxwell} reflection boundary condition $f|_{\Si_-}=(1-\ve)Rf$ with any $\ve\in[0,1)$ and $R$ given in \eqref{reflect}. Then there exists small $c_0>0$ such that for any $k\ge 0$, 
	\begin{multline}\label{L2esLargea}
		\sup_{0\le t\le T}\|\<v\>^kf(t)\|^2_{L^2_x(\Omega)L^2_v}+c_{\al}\|\<v\>^kf\|_{L^2_{t}([0,T])L^2_{x,v}(\Si_+)}
		+c_0\|\<v\>^kf(t)\|_{L^2_{t}([0,T])L^2_x(\Omega)L^2_D}^2\\
		+c_0\|\<v\>^kf(t)\|_{L^2_{t}([0,T])L^2_x(\Omega)L^2_v}^2
		\le C\|\<v\>^kf_0\|^2_{L^2_x(\Omega)L^2_v}.
	\end{multline}
	for some $C>0$ and small $c_0,c_{\al}>0$.
	Moreover, we have the large-time behavior:
	\begin{align}
		\label{L2esLarge}
		e^{c_0 t}\|\<v\>^kf(t)\|^2_{L^2_x(\Omega)L^2_v}
			+\|e^{c_0s}\<v\>^kf\|^2_{L^2_s([0,t])L^2_{x,v}(\Si_+)}\le \|\<v\>^kf_0\|^2_{L^2_x(\Omega)L^2_v},
	\end{align}
	for any $t\ge 0$.
\end{Thm}
Note that in this Theorem \ref{L2globalThm}, we only give the \emph{a priori} estimate while the proof of its existence is given in Sections \ref{Sec8} and \ref{Sec11}. The microscopic estimate of $\{\I-\P\}f$ is given by \eqref{micro}. So we will focus on the macroscopic estimate of $\P f$ below, which is mainly the following:

\begin{Prop}\label{macroLem}
	Assume that $\Omega\subset\R^3_x$ is bounded open and let $-\frac{3}{2}<\ga\le 2$ and $s\in(0,1)$. 	
	Let $f$ be the solution to the Boltzmann equation \eqref{B1} for $t\in[0,1]$.
	Suppose
	\begin{align}
		\label{L2global1}
		\sup_{0\le t\le 1}\|\<v\>^{\ga+10}f\|_{L^\infty_{x,v}(\Omega\times\R^3_v)}\le\de_0,
	\end{align}
	with some sufficiently small $\de_0>0$, which is a constant depending only on $\ga,s$. Then we have 
	
	\smallskip
	\begin{enumerate}
		\item 	
		There exists $M>0$ such that for any solution $f(t,x,v)$ to the nonlinear Boltzmann equation \eqref{B1},
		\begin{align}\label{127}\notag
			\int^{1}_{0}\|\P f(t)\|_{L^2_x(\Omega)L^2_D}^2\,dt&\le M\int^{1}_{0}\|\{\I-\P\}f(t)\|_{L^2_x(\Omega)L^2_D}^2\,dt\\
			&\quad+M\int^1_0\int_{\pa\Omega\times\R^3_v}|v\cdot n(x)||f(t)|^2\,dS(x)dvdt.
		\end{align}
		
		\smallskip 
		\item Assume the conservation of mass for the initial data as in \eqref{conservationlaw}.
		There exists $M>0$ such that for any solution $f(t,x,v)$ to the nonlinear Boltzmann equation \eqref{B1} satisfying the Maxwell reflection boundary condition $f|_{\Si_-}=(1-\ve)Rf$ with $\ve\in[0,1)$ and $R$ given in \eqref{reflect},
		\begin{align}\label{127a}\notag
			&\int^{1}_{0}\|\P f(t)\|_{L^2_x(\Omega)L^2_D}^2\,dt\le M\int^{1}_{0}\|\{\I-\P\}f(t)\|_{L^2_x(\Omega)L^2_D}^2\,dt\\
			&\qquad\quad+M\int_{\Si_+}|v\cdot n|\big(\ve|f|^2+(1-\ve)\al|f-R_Df|^2\big)\,dS(x)dv,
		\end{align}
		where $R_D f(v)=c_\mu\mu^{\frac{1}{2}}(v)\int_{v'\cdot n(x)>0}\{v'\cdot n(x)\}f(v')\mu^{\frac{1}{2}}(v')\,dv'$ is given by \eqref{RD1}.
	\end{enumerate}
\end{Prop}
The idea of the proof of Proposition \ref{macroLem} is to use contradiction arguments and splitting the domain $\Omega\times\R^3_v$, which is similar to the methodology in \cite{Guo2002,Guo2009,Guo2021b}.
We first show that Proposition \ref{macroLem} implies Theorem \ref{L2globalThm}.

\begin{proof}[Proof of Theorem \ref{L2globalThm}]
	The proof is the standard energy argument by combining several weighted $L^2$ energy estimates. 
	For the solution $f$ to nonlinear Boltzmann equation \eqref{B1}, $e^{\lam t}f(t)$ satisfies
	\begin{align}\label{B13}
		(\pa_tf+ v\cdot\na_x)(e^{\lam t}f) = L(e^{\lam t}f)+\Gamma(f,e^{\lam t}f)+\lam e^{\lam t}f.
	\end{align}
	Let $0\le N\le t\le N+1$ with $N$ being an integer.
	With estimate \eqref{L2global}, for any $k\ge 0$, we take $L^2$ inner product of \eqref{B1} with $2\<v\>^{2k}f$ and $2f$ respectively over $[N,t]\times\Omega\times\R^3_v$ to deduce
	\begin{multline}\label{125}
		\|\<v\>^kf(t)\|^2_{L^2_x(\Omega)L^2_v} + \int^t_N\int_{\pa\Omega\times\R^3_v}v\cdot n|\<v\>^kf|^2\,dvdS(x)ds + (2c_0-C\de_0)\int^t_N\|\<v\>^kf\|^2_{L^2_x(\Omega)L^2_D}\,ds\\
		\le 2\|\<v\>^kf(N)\|^2_{L^2_x(\Omega)L^2_v}
		+C\int^t_N\|f\|^2_{L^2_x(\Omega)L^2_v}\,ds,
	\end{multline}
	(where we put the boundary integrations $\Si_+$ and $\Si_-$ together here and below) and
	\begin{multline}\label{125a}
		\|f(t)\|^2_{L^2_x(\Omega)L^2_v} + \int^t_N\int_{\pa\Omega\times\R^3_v}v\cdot n|f|^2\,dvdS(x)ds
		+(2c_0-C\de_0)\int^t_N\|f\|^2_{L^2_x(\Omega)L^2_D}\,ds\\
		\le 2\|f(N)\|^2_{L^2_x(\Omega)L^2_v}+C\int^t_N\|f\|^2_{L^2_x(\Omega)L^2_v}\,ds,
	\end{multline}
	for some $c_0>0$, where we used \eqref{micro}, \eqref{L1}, \eqref{L2} and \eqref{Gaesweight} for the collision terms.
	Choosing $\de_0>0$ small enough and using Gr\"{o}nwall's inequality to \eqref{125} and \eqref{125a}, we have
	\begin{multline}\label{125y}
		\|\<v\>^kf(t)\|^2_{L^2_x(\Omega)L^2_v} + \int^t_N\int_{\pa\Omega\times\R^3_v}v\cdot n|\<v\>^kf|^2\,dvdS(x)ds + c_0\int^t_N\|\<v\>^kf\|^2_{L^2_x(\Omega)L^2_D}\,ds\\
		\le e^{C(t-N)}\|\<v\>^kf(N)\|^2_{L^2_x(\Omega)L^2_v},
	\end{multline}
	and
	\begin{multline}\label{125ay}
		\|f(t)\|^2_{L^2_x(\Omega)L^2_v} + \int^t_N\int_{\pa\Omega\times\R^3_v}v\cdot n|f|^2\,dvdS(x)ds
		+c_0\int^t_N\|f\|^2_{L^2_x(\Omega)L^2_D}\,ds\\
		\le e^{C(t-N)}\|f(N)\|^2_{L^2_x(\Omega)L^2_v}.
	\end{multline}
	For the $L^2$ estimate on time interval $[0,N]$, we take $L^2$ inner product of \eqref{B13} with $e^{\lam t}\<v\>^{2k}f$ and $e^{\lam t}f$ respectively over $[0,N]\times\Omega\times\R^3_v$ to deduce
	\begin{multline}\label{125b}
		e^{2\lam N}\|\<v\>^kf(N)\|^2_{L^2_x(\Omega)L^2_v} + \int^N_0\int_{\pa\Omega\times\R^3_v}v\cdot ne^{2\lam s}|\<v\>^kf|^2\,dvdS(x)ds\\
		+ (2c_0-C\de_0)\int^N_0e^{2\lam s}\|\<v\>^kf\|^2_{L^2_x(\Omega)L^2_D}\,ds
		\le \|\<v\>^kf(0)\|^2_{L^2_x(\Omega)L^2_v} \\
		+2\lam\int^N_0e^{2\lam s}\|\<v\>^kf\|^2_{L^2_x(\Omega)L^2_v}\,ds
		+C\int^N_0e^{2\lam s}\|\<v\>^{\frac{\ga+2s}{2}}f\|^2_{L^2_x(\Omega)L^2_v}\,ds,
	\end{multline}
	and
	\begin{multline}\label{125c}
		e^{2\lam N}\|f(N)\|^2_{L^2_x(\Omega)L^2_v}
		+ \int^N_0\int_{\pa\Omega\times\R^3_v}v\cdot ne^{2\lam s}|f|^2\,dvdS(x)ds
		+ 2c_0\int^N_0e^{2\lam s}\|(\I-\P)f\|^2_{L^2_x(\Omega)L^2_D}\,ds\\
		\le 2\|f(0)\|^2_{L^2_x(\Omega)L^2_v}
		+2\lam \int^N_0e^{2\lam s}\|f\|^2_{L^2_x(\Omega)L^2_v}\,ds
		+C\de_0\int^N_0e^{2\lam s}\|f\|^2_{L^2_x(\Omega)L^2_D}\,ds.
	\end{multline}
	Dividing the time interval as $[0,N)=\cup^{N-1}_{n=0}[n,n+1)$ and writing $f_n(s,x,v)=f(n+s,x,v)$ for $n=0,1,2,\dots,N-1$, we have from \eqref{125c} that
	\begin{multline}\label{125d}
		e^{2\lam N}\|f(N)\|^2_{L^2_x(\Omega)L^2_v} + \sum_{n=0}^{N-1}\int^1_0\int_{\pa\Omega\times\R^3_v}v\cdot ne^{2\lam(n+s)}|f_n|^2\,dvdS(x)ds\\
		+ 2c_0\sum_{n=0}^{N-1}\int^1_0e^{2\lam(n+s)}\|(\I-\P)f_n\|^2_{L^2_x(\Omega)L^2_D}\,ds\
		\le 2\|f(0)\|^2_{L^2_x(\Omega)L^2_v} \\
		+2\lam \int^N_0e^{2\lam s}\|f\|^2_{L^2_x(\Omega)L^2_v}\,ds +C\de_0\int^N_0e^{2\lam s}\|f\|^2_{L^2_x(\Omega)L^2_D}\,ds.
	\end{multline}
	Notice that $f_n(s,x,v)$ satisfies the same Boltzmann equation \eqref{B1} for $s\in(0,1)$, which allows use to apply Proposition \ref{macroLem} to $f_n$. 

	\smallskip 
	In order to obtain the exponential decay for soft potential, we apply the space-velocity mixed weight introduced in \cite{Deng2022}. Let $W$ be given by \eqref{W}:
	\begin{align*}
		W = W(x,v) = \exp\Big(-q\frac{x\cdot v}{\<v\>}\Big),
	\end{align*}
	with some fixed $q\in(0,1)$.
	Note that we require the boundedness of $\Omega$ to obtain the boundedness of $W$ as in \eqref{Wbound}.
	Multiplying \eqref{B13} by $W$ and noticing \eqref{W1}, we have
	\begin{multline}\label{B14}
		(\pa_tf+ v\cdot\na_x)(e^{\lam t}Wf)+ q|v|^2\<v\>^{-1}e^{\lam t}Wf\\
		= 
		\Gamma(e^{\lam t}f,\mu^{\frac{1}{2}})+W\Gamma(\mu^{\frac{1}{2}}+f,e^{\lam t}f)
		+\lam e^{\lam t}Wf.
	\end{multline}
	Taking the $L^2$ inner product of \eqref{B14} with $\<v\>^{2k}Wf$ over $[0,t]\times\Omega\times\R^3_v$, and using \eqref{L2global}, \eqref{Gaesweight} and \eqref{L2} for the collision term yields
	\begin{multline}\label{125f}
		e^{2\lam t}\|\<v\>^{k}Wf(t)\|^2_{L^2_x(\Omega)L^2_v}
		+ \int^t_0\int_{\pa\Omega\times\R^3_v}v\cdot ne^{2\lam s}|\<v\>^{k}Wf|^2\,dS(x)dvds \\
		+ q\int^t_0e^{2\lam s}\||v|\<v\>^{k-\frac{1}{2}}Wf\|^2_{L^2_x(\Omega)L^2_v}\,ds
		\le 2\|\<v\>^{k}Wf(0)\|^2_{L^2_x(\Omega)L^2_v}\\
		+C\int^t_0e^{2\lam s}\|\<v\>^{k}f\|^2_{L^2_x(\Omega)L^2_D}\,ds
		+\lam \int^t_0e^{2\lam s}\|\<v\>^{k}Wf\|^2_{L^2_x(\Omega)L^2_v}\,ds.
	\end{multline}
	where we used the boundedness of $W$ from \eqref{Wbound}. We also used the control $\|W^2\<v\>^kf\|_{L^2_D}\le C\|(\ti a^{1/2})^w(W^2\<v\>^kf)\|_{L^2_v}\lesssim \|\<v\>^kf\|_{L^2_D}$; see for instance \cite[Lemma 2.3]{Deng2020a}. Here $\ti a$ is given in \eqref{tiaa} and $(\ti a^{1/2})^wW^2\<v\>^k$ can be regarded as a pseudo-differential operator in $v\in\R^3_v$ with symbol in $S(\ti a^{1/2}\<v\>^k)$ (symbol is defined in \eqref{symbol}).

	\smallskip
	Then we can apply Proposition \ref{macroLem} accordingly.

	{{\smallskip}}\noindent{\bf Inflow boundary condition.} For the case of inflow boundary condition \eqref{inflow}, we have from \eqref{127} that
	\begin{multline}\label{130}
		\frac{1}{M}\sum_{n=0}^{N-1}\int^1_0e^{2\lam(n+s)}\|\P f_n\|^2_{L^2_x(\Omega)L^2_D}\,ds
		\le \sum_{n=0}^{N-1}\int^1_0e^{2\lam(n+s)}\|(\I-\P)f_n\|^2_{L^2_x(\Omega)L^2_D}\,ds\\
		+\sum_{n=0}^{N-1}\int^1_0\int_{\pa\Omega\times\R^3_v}|v\cdot n|e^{2\lam(n+s)}|f_n|^2\,dvdS(x)ds.
	\end{multline}
	For the energy estimate in $[N,t]$, notice that $e^{2\lam t}\le e^{2\lam(t-N)}e^{2\lam s}$ for any $s\ge N$.
	Multiplying \eqref{125y} by $e^{2\lam t}$, we have
	\begin{align}\label{1221}
		&\notag e^{2\lam t}\|\<v\>^kf(t)\|^2_{L^2_x(\Omega)L^2_v} + \int^t_N\int_{\Si_+}|v\cdot n|e^{2\lam s}|\<v\>^kf|^2\,dvdS(x)ds + c_0\int^t_Ne^{2\lam s}\|\<v\>^kf\|^2_{L^2_x(\Omega)L^2_D}\,ds\\
		&\notag\quad\le e^{2\lam(t-N)+C(t-N)}e^{2\lam N}\|\<v\>^kf(N)\|^2_{L^2_x(\Omega)L^2_v}+e^{2\lam(t-N)}\int^t_N\int_{\Si_-}|v\cdot n|e^{2\lam s}|\<v\>^kf|^2\,dvdS(x)ds\\
		&\quad\le C_1e^{2\lam N}\|\<v\>^kf(N)\|^2_{L^2_x(\Omega)L^2_v}+C_1\int^t_N\int_{\Si_-}|v\cdot n|e^{2\lam s}|\<v\>^kg|^2\,dvdS(x)ds.
	\end{align}
	for some $C_1>1$,
	where we choose $\lam<1$ and used $t\le N+1$.
Thus, taking linear combination $\ka\times\eqref{130}+\eqref{125d}+\ka^2\times\eqref{125b}+\ka^3\times\eqref{125f}+\ka^2 C_1^{-1}\times\eqref{1221}$ with sufficiently small $\ka>0$, we have
	\begin{align}\label{125u}
		&\notag\ka^2C_1^{-1}e^{2\lam t}\|\<v\>^kf(t)\|^2_{L^2_x(\Omega)L^2_v}
		+\ka^2C_1^{-1}\int^t_N\int_{\Si_+}|v\cdot n|e^{2\lam s}|\<v\>^kf|^2\,dvdS(x)ds\\
		&\notag\quad+\int^N_0\int_{\Si_+}|v\cdot n|e^{2\lam s}\big(\frac{|f|^2}{2}+\ka^2|\<v\>^kf|^2\big)\,dvdS(x)ds\\
		&\notag\quad+\frac{\ka c_0}{2}\int^N_0e^{2\lam s}\|f\|^2_{L^2_x(\Omega)L^2_D}\,ds
		+\ka^2(2c_0-C\de_0)\int^N_0e^{2\lam s}\|\<v\>^kf\|^2_{L^2_x(\Omega)L^2_D}\,ds\\
		&\notag\quad+\ka^2C_1^{-1}c_0\int^t_Ne^{2\lam s}\|\<v\>^kf\|^2_{L^2_x(\Omega)L^2_D}\,ds
		+q\ka^2\int^t_0e^{2\lam s}\||v|\<v\>^{k-\frac{1}{2}}Wf\|^2_{L^2_x(\Omega)L^2_v}\,ds\\
		&\notag\quad\le
		C\|\<v\>^kf_0\|^2_{L^2_x(\Omega)L^2_v}
		+(C+\ka^2)\int^t_0\int_{\Si_-}|v\cdot n|e^{2\lam s}|\<v\>^{k}g|^2\,dvdS(x)ds\\
		&\notag\qquad+C\de_0\int^N_0e^{2\lam s}\|f\|^2_{L^2_x(\Omega)L^2_D}\,ds
		+\lam\int^N_0e^{2\lam s}\|f\|^2_{L^2_x(\Omega)L^2_v}\,ds\\
		&\qquad+\ka^2C\lam\int^t_0e^{2\lam s}\|\<v\>^kf\|^2_{L^2_x(\Omega)L^2_v}\,ds
		+\ka^3C\int^t_0e^{2\lam s}\|\<v\>^kf\|^2_{L^2_x(\Omega)L^2_D}\,ds, 
	\end{align}
	where we used $\|f\|_{L^2_D}\le\|\P f\|_{L^2_D}+\|\{\I-\P\}f\|_{L^2_D}$, $\|\<v\>^{\frac{\ga+2s}{2}}f\|^2_{L^2_v}\le C\|f\|_{L^2_D}$, the bound of $W$ from \eqref{Wbound} and interpolation 
	\begin{align}\label{interdamp}
		\|\<v\>^{k}f\|_{L^2_v}\le C\|\<v\>^{k}f\|_{L^2_D}+C\||v|\<v\>^{k-\frac{1}{2}}Wf\|^2_{L^2_v}
	\end{align}
	for both hard and soft potentials, which follows from \eqref{esD} and \eqref{Wbound}.
	With fixed constants $c_0,C,C_1>0$, we choose sufficiently small $0<\ka<\min\big\{1,\frac{c_0}{2C}\big\}$, $0<\de_0<\frac{c_0\ka}{2C}$, and $0\le\lam<\frac{c_0}{2CC_1}$ in \eqref{125u} to deduce
	\begin{multline*}
		\ka^2C_1^{-1}e^{2\lam t}\|\<v\>^kf(t)\|^2_{L^2_x(\Omega)L^2_v}
		+\ka^2C_1^{-1}\int^t_0\int_{\Si_+}|v\cdot n|e^{2\lam s}|\<v\>^kf|^2\,dvdS(x)ds \\
		+\ka^2C_1^{-1}c_0\int^t_0e^{2\lam s}\|\<v\>^kf\|^2_{L^2_x(\Omega)L^2_D}\,ds
		+q\ka^2\int^t_0e^{2\lam s}\||v|\<v\>^{k-\frac{1}{2}}Wf\|^2_{L^2_x(\Omega)L^2_v}\,ds\\
		\le
		C\|\<v\>^kf_0\|^2_{L^2_x(\Omega)L^2_v}
		+C\int^t_0\int_{\Si_-}|v\cdot n|e^{2\lam s}|\<v\>^{k}g|^2\,dvdS(x)ds, 
	\end{multline*}
	for some $C>0$.
	This implies \eqref{L2es1} and \eqref{L2es1a} for the case of inflow boundary condition by choosing $\lam=0$ and small $\lam>0$ respectively.

	{{\smallskip}}\noindent{\bf Maxwell reflection boundary condition.} The approach for this case is similar. The only difference is to use the weighted boundary estimate in Lemma \ref{LemR}, i.e. estimate \eqref{101} and \eqref{101w}, and trace lemma \ref{diffboundLem} to calculate the boundary terms. By \eqref{127a}, we have
	\begin{multline}\label{130a}
		\frac{1}{M}\sum_{n=0}^{N-1}\int^1_0e^{2\lam(n+s)}\|\P f_n\|^2_{L^2_x(\Omega)L^2_D}\,ds
		\le \sum_{n=0}^{N-1}\int^1_0e^{2\lam(n+s)}\|(\I-\P)f_n\|^2_{L^2_x(\Omega)L^2_D}\,ds\\
		+\sum_{n=0}^{N-1}e^{2\lam(n+s)}\big(\ve\|f_n\|^2_{L^2_{x,v}(\Si_+)}
		+(1-\ve)\al\|f_n-R_Df_n\|^2_{L^2_{x,v}(\Si_+)}\big). 
	\end{multline}
	Using \eqref{101} and \eqref{101w} for the case of Maxwell reflection boundary, we have 
\begin{align}\label{Maxbounda}
		\int_{\pa\Omega\times\R^3_v}v\cdot n|f|^2\,dvdS(x)
		= \ve\|f\|^2_{L^2_{x,v}(\Si_+)}
		+(1-\ve)\al\|f-R_Df\|^2_{L^2_{x,v}(\Si_+)},
\end{align}
	and 
\begin{align}\label{Maxboundb}
		\int_{\pa\Omega\times\R^3_v}v\cdot n|\<v\>^kf|^2\,dvdS(x)
		\ge \big(1-(1-\ve)(1-\al)\big)\|\<v\>^kf\|^2_{L^2_{x,v}(\Si_+)}
		-C_k\|f\|^2_{L^2_{x,v}(\Si_+)}.
\end{align}
The extra damping energy in \eqref{125f} can now be rewritten as 
\begin{align}\label{125fpp}\notag
	&e^{2\lam t}\|\<v\>^{k}Wf(t)\|^2_{L^2_x(\Omega)L^2_v}
	+ q\int^t_0e^{2\lam s}\||v|\<v\>^{k-\frac{1}{2}}Wf\|^2_{L^2_x(\Omega)L^2_v}\,ds
	\le C\|\<v\>^{k}f(0)\|^2_{L^2_x(\Omega)L^2_v}\\
	&\quad+C_k\int^t_0e^{2\lam s}\|f\|^2_{L^2_{x,v}(\Si_+)}\,ds+C\int^t_0e^{2\lam s}\|\<v\>^{k}f\|^2_{L^2_x(\Omega)L^2_D}\,ds
	+\lam \int^t_0e^{2\lam s}\|\<v\>^{k}f\|^2_{L^2_x(\Omega)L^2_v}\,ds.
\end{align}
	For the energy estimate in $[N,t]$, multiplying \eqref{125y} and \eqref{125ay} by $e^{2\lam t}$, and applying \eqref{Maxbounda} and \eqref{Maxboundb} to boundary terms (notice $\big(1-(1-\ve)(1-\al^2)\big)\ge 0$), we obtain
	\begin{align}\label{1225u}\notag
		&e^{2\lam t}\|\<v\>^kf(t)\|^2_{L^2_x(\Omega)L^2_v}
		+c_0e^{2\lam t}\int^t_N\|\<v\>^kf\|^2_{L^2_x(\Omega)L^2_D}\,ds
		+c_{\ve,\al}e^{2\lam t}\int^t_N\|\<v\>^kf\|^2_{L^2_{x,v}(\Si_+)}\,ds\\
		&\notag\quad\le e^{2\lam (t-N)+C(t-N)}e^{2\lam N}\|\<v\>^kf(N)\|^2_{L^2_x(\Omega)L^2_v}+C_ke^{2\lam t}\int^t_N\|f\|^2_{L^2_{x,v}(\Si_+)}\,ds\\
		&\quad\le C_1e^{2\lam N}\|\<v\>^kf(N)\|^2_{L^2_x(\Omega)L^2_v}+C_ke^{2\lam t}\int^t_N\|f\|^2_{L^2_{x,v}(\Si_+)}\,ds,
	\end{align}
and 
	\begin{align}\label{1225i}\notag
		e^{2\lam t}\|f(t)\|^2_{L^2_x(\Omega)L^2_v}
		+e^{2\lam t}\int^t_N\Big(\ve\|f\|^2_{L^2_{x,v}(\Si_+)}
		&+(1-\ve)\al\|f-R_Df\|^2_{L^2_{x,v}(\Si_+)}\Big)\,ds\\
		+c_0e^{2\lam t}\int^t_N\|f\|^2_{L^2_x(\Omega)L^2_D}\,ds
		&\notag\le e^{2\lam (t-N)+C(t-N)}e^{2\lam N}\|f(N)\|^2_{L^2_x(\Omega)L^2_v}\\
		&\le C_1e^{2\lam N}\|f(N)\|^2_{L^2_x(\Omega)L^2_v},
	\end{align}
	respectively, 
	for some $C_1>2$, where we used $t\le N+1$.
Combining $\eqref{1225u}+C_1\times\eqref{125b}$ and $\eqref{1225i}+C_1\times\eqref{125d}$, and using the similar boundary estimates for the interval $[0,N]$ again, we have 
\begin{align}\label{1225o}\notag
	&e^{2\lam t}\|\<v\>^kf(t)\|^2_{L^2_x(\Omega)L^2_v}
	+c_0\int^t_0e^{2\lam s}\|\<v\>^kf\|^2_{L^2_x(\Omega)L^2_D}\,ds
	+c_{\ve,\al}\int^t_0e^{2\lam s}\|\<v\>^kf\|^2_{L^2_{x,v}(\Si_+)}\,ds\\
	&\notag\quad\le C\|\<v\>^kf(0)\|^2_{L^2_x(\Omega)L^2_v}
	+C_k\int^t_0e^{2\lam s}\|f\|^2_{L^2_{x,v}(\Si_+)}\,ds\\
	&\quad
	+2\lam\int^N_0e^{2\lam s}\|\<v\>^kf\|^2_{L^2_x(\Omega)L^2_v}\,ds
	+C\int^N_0e^{2\lam s}\|f\|^2_{L^2_x(\Omega)L^2_D}\,ds, 
\end{align}
and 
\begin{align}\label{1225pz}\notag
	&e^{2\lam t}\|f(t)\|^2_{L^2_x(\Omega)L^2_v}
	+\int^t_Ne^{2\lam s}\Big(\ve\|f\|^2_{L^2_{x,v}(\Si_+)}
	+(1-\ve)\al\|f-R_Df\|^2_{L^2_{x,v}(\Si_+)}\Big)\,ds \\
	&\notag\quad
	+ C_1\sum_{n=0}^{N-1}\int^1_0e^{2\lam s}\Big(\ve\|f_n\|^2_{L^2_{x,v}(\Si_+)}
	+(1-\ve)\al\|f_n-R_Df_n\|^2_{L^2_{x,v}(\Si_+)}\Big)\,ds\\
	&\notag\quad+c_0\int^t_Ne^{2\lam s}\|f\|^2_{L^2_x(\Omega)L^2_D}\,ds+2C_1c_0\sum_{n=0}^{N-1}\int^1_0e^{2\lam(n+s)}\|(\I-\P)f_n\|^2_{L^2_x(\Omega)L^2_D}\,ds\\
	&\le C_1\|f(0)\|^2_{L^2_x(\Omega)L^2_v}
	+2C_1\lam\int^N_0e^{2\lam s}\|f\|^2_{L^2_x(\Omega)L^2_v}\,ds +C\de_0\int^N_0e^{2\lam s}\|f\|^2_{L^2_x(\Omega)L^2_D}\,ds.
\end{align}
respectively. The estimate \eqref{1225pz} can imply two types of energy estimates. First, combining \eqref{1225pz} and macroscopic estimate \eqref{130a} implies the dissipation rate without any boundary energy, and by choosing $\de_0>0$ small i.e. 
\begin{align}\label{1225p1}
	&e^{2\lam t}\|f(t)\|^2_{L^2_x(\Omega)L^2_v}+c_0\int^t_0e^{2\lam s}\|f\|^2_{L^2_x(\Omega)L^2_D}\,ds
	\le C\|f(0)\|^2_{L^2_x(\Omega)L^2_v}
	+2C_1\lam\int^N_0e^{2\lam s}\|f\|^2_{L^2_x(\Omega)L^2_v}\,ds.	
\end{align}
for some small generic constant $c_0>0$. 
Second, neglecting the dissipation rate and using $\|f-R_D f\|_{L^2_{x,v}(\Si_+)}^2\ge \frac{1}{2}\|f\|_{L^2_{x,v}(\Si_+)}^2-\|R_D f\|_{L^2_{x,v}(\Si_+)}^2$, we can obtain the boundary energy, 
\begin{align}\label{1225p}\notag
	&e^{2\lam t}\|f(t)\|^2_{L^2_x(\Omega)L^2_v}
	+\frac{1}{2}(1-\ve)\al\int^t_0e^{2\lam s}\|f\|_{L^2_{x,v}(\Si_+)}^2\,ds\le 
	C\int^t_0\|R_D f\|_{L^2_{x,v}(\Si_+)}^2\,ds\\
	&\quad
	+C\|f(0)\|^2_{L^2_x(\Omega)L^2_v}
	+2C_1\lam \int^N_0e^{2\lam s}\|f\|^2_{L^2_x(\Omega)L^2_v}\,ds +C\de_0\int^N_0e^{2\lam s}\|f\|^2_{L^2_x(\Omega)L^2_D}\,ds.
\end{align}
The term $\|R_D f\|_{L^2_{x,v}(\Si_+)}^2$ on the right-hand side will be controlled by the boundary energy $\|f\|_{L^2_{x,v}(\Si_+)}^2$ and the interior energy. 
For this, we shall apply the trace Lemma \ref{diffboundLem} to the boundary term $R_D f$ given in \eqref{RD1}, i.e. 
\begin{align*}
	R_D f(v)=c_\mu\mu^{\frac{1}{2}}(v)\int_{v'\cdot n(x)>0}\{v'\cdot n(x)\}f(v')\mu^{\frac{1}{2}}(v')\,dv'.
\end{align*}
For any $0\le T_1<T_2\le T_1+1$ and any $\de>0$, let $\chi^{+}_{\de}=\chi^{+}_{\de}(t,x,v;T_1+N\de^3)$ be defined in \eqref{chide}. Then we have 
\begin{align*}
	\int^{T_2}_{T_1}\|R_D f\|_{L^2_{x,v}(\Si_+)}^2\,dt
	&= \int^{T_2}_{T_1}\int_{\pa\Omega}c_\mu\Big|\int_{v'\cdot n(x)>0}\{v'\cdot n(x)\}f(v')\mu^{\frac{1}{2}}(v')\,dv'\Big|^2\,dS(x)dt\\
	&\le\Big(\int_{T_1+[(T_2-T_1)/\de^3]\de^3}^{T_2}+\sum_{N=0}^{[(T_2-T_1)/\de^3]-1}\int_{T_1+N\de^3}^{T_1+(N+1)\de^3}\Big)\big(\cdots\big)\,dt. 
\end{align*}
Splitting $f(v')=(1-\chi^{+}_{\de})f(v')+\chi^{+}_{\de} f(v')$ and applying trace Lemma \ref{diffboundLem} to each term, we have 
	\begin{align*}\notag
		\int^{T_2}_{T_1}\|R_D f\|_{L^2_{x,v}(\Si_+)}^2\,dt
	&\le C(\de^4+e^{-\de^{-1/2}})
		\|f\|^2_{L^2_tL^2_{x,v}(\Si_+)}\\
		&\quad\notag+2\sum_{N=0}^{[(T_2-T_1)/\de^3]}\bigg\{\int^{T_1+N\de^3}_{T_1}\big(\Gamma(\mu^{\frac{1}{2}}+f,f)+\Gamma(f,\mu^{\frac{1}{2}}),f\big)_{L^2_x(\Omega)L^2_v}\,dt\\
		&\quad+\int^{T_2}_{T_1}
		\big(\Gamma(\mu^{\frac{1}{2}}+f,f)+\Gamma(f,\mu^{\frac{1}{2}}),\chi^{+}_{\de} f\big)_{L^2_x(\Omega)L^2_v}\,dt\bigg\},
	\end{align*}
	where $\chi^{+}_{\de}$ depends on $N,\de$. 
	Then by the assumption \eqref{L2global}, collisional estimates \eqref{L1} and \eqref{L2}, with upper bound of $\chi^{+}_{\de}$ in \eqref{chideesti} (note that the commutator $[(\ti a^{1/2})^w,\chi^{+}_{\de}]$, between the pseudo-differential operator $\ti a^{1/2}$ given by \eqref{tiaa} and the good function $\chi^{+}_{\de}$, belong to symbol class $S(\ti a^{1/2})$; or one can simply apply Lemma \ref{GachideLem}),
	we continue it as 
	\begin{align}\label{boundaryPsi}
		\int^{T_2}_{T_1}\|R_D f\|_{L^2_{x,v}(\Si_+)}^2\,dt&\le C(\de^4+e^{-\de^{-1/2}})\|f\|^2_{L^2_t([T_1,T_2])L^2_{x,v}(\Si_+)}+C_\de\|f\|^2_{L^2_t([T_1,T_2])L^2_x(\Omega)L^2_D}. 
	\end{align}
	Therefore, substituting \eqref{boundaryPsi} into \eqref{1225p} with sufficiently small $\de=\de(\ve)>0$ (note that $\ve\in[0,1)$ is fixed within the assumptions), we have 
	\begin{align}\label{1225pp}\notag
		&e^{2\lam t}\|f(t)\|^2_{L^2_x(\Omega)L^2_v}
		+c_{\al,\ve}\int^t_0e^{2\lam s}\|f\|_{L^2_{x,v}(\Si_+)}^2\,ds
		\\
		&\quad\le 
		C\|f(0)\|^2_{L^2_x(\Omega)L^2_v}
		+2C_1\lam \int^N_0e^{2\lam s}\|f\|^2_{L^2_x(\Omega)L^2_v}\,ds +C\int^t_0e^{2\lam s}\|f\|^2_{L^2_x(\Omega)L^2_D}\,ds, 
	\end{align}
	for some small constant $c_{\al,\ve}>0$. To control the extra dissipation rate in \eqref{1225pp}, we may use \eqref{1225p1}. Therefore, taking combination $\eqref{1225p1}+\ka\times\eqref{1225pp}+\ka^2\times\eqref{1225o}+\ka^3\times\eqref{125fpp}$ with sufficiently small $\ka>0$, we have 
	\begin{align*}\notag
		&e^{2\lam t}\|f(t)\|^2_{L^2_x(\Omega)L^2_v}
		+\ka^2e^{2\lam t}\|\<v\>^kf(t)\|^2_{L^2_x(\Omega)L^2_v}
		+\frac{\ka c_{\al,\ve}}{2}\int^t_0e^{2\lam s}\|f\|_{L^2_{x,v}(\Si_+)}^2\,ds\\
		&\notag\quad
		+\ka^2c_{\ve,\al}\int^t_0e^{2\lam s}\|\<v\>^kf\|^2_{L^2_{x,v}(\Si_+)}\,ds
		+\frac{c_0}{2}\int^t_0e^{2\lam s}\|f\|^2_{L^2_x(\Omega)L^2_D}\,ds\\
		&\notag\quad
		+\frac{\ka^2c_0}{2}\int^t_Ne^{2\lam s}\|\<v\>^kf\|^2_{L^2_x(\Omega)L^2_D}\,ds
+ \ka^3q\int^t_0e^{2\lam s}\||v|\<v\>^{k-\frac{1}{2}}Wf\|^2_{L^2_x(\Omega)L^2_v}\,ds\\
&\quad
\le 
C\lam \int^t_0e^{2\lam s}\|\<v\>^{k}f\|^2_{L^2_x(\Omega)L^2_v}\,ds
+C\|\<v\>^{k}f(0)\|^2_{L^2_x(\Omega)L^2_v}, 
	\end{align*}
	for some constant $\ka>0$.
Therefore, usign interpolation \eqref{interdamp} to obtain the extra damping and choosing $0\le\lam<\frac{c_0\ka^3q}{C}$ small, we have 
\begin{align*}
	&
	\ka^2e^{2\lam t}\|\<v\>^kf(t)\|^2_{L^2_x(\Omega)L^2_v}
	+\ka^2c_{\ve,\al}\int^t_0e^{2\lam s}\|\<v\>^kf\|^2_{L^2_{x,v}(\Si_+)}\,ds
	+\frac{\ka^2c_0}{2}\int^t_Ne^{2\lam s}\|\<v\>^kf\|^2_{L^2_x(\Omega)L^2_D}\,ds\\
	&\notag\quad
	+ \ka^3q\int^t_0e^{2\lam s}\||v|\<v\>^{k-\frac{1}{2}}Wf\|^2_{L^2_x(\Omega)L^2_v}\,ds+\frac{\ka^3q}{C}\int^t_0e^{2\lam s}\|\<v\>^kf\|^2_{L^2_x(\Omega)L^2_v}\,ds\\
	&\quad
	\le C\|\<v\>^{k}f(0)\|^2_{L^2_x(\Omega)L^2_v}.
\end{align*}
This implies \eqref{L2esLargea} and \eqref{L2esLarge} with $\lam=0$ and small $\lam>0$ respectively. We then conclude the global \emph{a priori} $L^2$ estimate in Theorem \ref{L2globalThm}.
\end{proof}

The remaining of this Section \ref{Sec12} is devoted to the proof of Proposition \ref{macroLem}.

\subsection{Macroscopic estimate}\label{Sec121}
Assume that $\Omega\subset\R^3_x$ is a bounded open subset.
We will prove Proposition \ref{macroLem} by contradiction. If Proposition \ref{macroLem} is false, then no $M$ exists as in Proposition \ref{macroLem} for every solution to the nonlinear Boltzmann equation. Hence, for any $k\ge 1$, there exists a sequence of non-zero solutions $f_k(t,x,v)$ to the nonlinear Boltzmann equation \eqref{B1} that satisfy \eqref{L2global1}:
\begin{align}
	\label{L2global2}
	\sup_{0\le t\le 1}\|\<v\>^{\ga+10}f_k\|_{L^\infty_{x,v}(\Omega\times\R^3_v)}\le \frac{1}{k},
\end{align}
and one of the following:

{{\smallskip}}
\begin{enumerate}
	\item \noindent{\bf In the inflow boundary case:} $f_k$ satisfies inflow boundary condition \eqref{inflow} and
	\begin{align*}\notag
		\int^{1}_{0}\|\P f_k(t)\|_{L^2_x(\Omega)L^2_D}^2\,dt & \ge k\int^{1}_{0}\|\{\I-\P\}f_k(t)\|_{L^2_x(\Omega)L^2_D}^2\,dt \\
		& \quad+k\int^1_0\int_{\pa\Omega\times\R^3_v}|v\cdot n(x)||f_k(t)|^2\,dS(x)dvdt.
	\end{align*}
	Equivalently, by normalization
	\begin{align}
		\label{norm}
		Z_k(t,x,v)=\frac{f_k(t,x,v)}{\Big(\int^{1}_{0}\|\P f_k(t)\|_{L^2_x(\Omega)L^2_D}^2\,dt\Big)^{\frac{1}{2}}},
	\end{align}
	we have
	\begin{align}\label{normalPf}
		\int^{1}_{0}\|\P Z_k(t)\|_{L^2_x(\Omega)L^2_D}^2\,dt=1,
	\end{align}
	and
	\begin{align}\label{140}
		\hspace{3em}\int^{1}_{0}\|\{\I-\P\}Z_k(t)\|_{L^2_x(\Omega)L^2_D}^2\,dt
		+\int^1_0\int_{\pa\Omega\times\R^3_v}|v\cdot n(x)||Z_k(t)|^2\,dS(x)dvdt\le\frac{1}{k}.
	\end{align}
	
	{{\smallskip}}
	\item \noindent{\bf In the Maxwell reflection boundary case:} $f_k$ satisfies Maxwell reflection boundary condition \eqref{reflect} and
	\begin{align*}
		\int^{1}_{0}\|\P f_k(t)\|_{L^2_x(\Omega)L^2_D}^2\,dt & \ge k\int^{1}_{0}\|\{\I-\P\}f_k(t)\|_{L^2_x(\Omega)L^2_D}^2\,dt \\
		+k\int_{\Si_+} & |v\cdot n|\big(\ve|f_k|^2+(1-\ve)\al|f_k-R_Df_k|^2\big)\,dS(x)dv.
	\end{align*}
	The normalized $Z_k$ given in \eqref{norm} satisfies $Z_k|_{\Si_-}=(1-\ve)RZ_k$, \eqref{normalPf} and
	\begin{align}\label{140a}\notag
		& \int^{1}_{0}\|\{\I-\P\}Z_k(t)\|_{L^2_x(\Omega)L^2_D}^2\,dt \\
		& \quad+\int_{\Si_+}|v\cdot n|\big(\ve|Z_k|^2+(1-\ve)\al|Z_k-R_DZ_k|^2\big)\,dS(x)dv\le \frac{1}{k}.
	\end{align}
\end{enumerate}

{{\smallskip}}
In both cases above, we know that
\begin{align}\label{ZkD}
	\int^{1}_{0}\|Z_k(t)\|_{L^2_x(\Omega)L^2_D}^2\,dt\le 2
\end{align}
is bounded. By Banach-Alaoglu theorem, there exists $Z(t,x,v)$ satisfying
\begin{align}\label{ZD}
	\int^{1}_{0}\|Z(t)\|_{L^2_x(\Omega)L^2_D}^2\,dt\le 2
\end{align}
such that
\begin{align*}
	Z_k\rightharpoonup Z\ \text{ weakly in }\ \int^{1}_{0}\|\cdot\|_{L^2_x(\Omega)L^2_D}^2\,dt.
\end{align*}
By \eqref{140} and \eqref{140a}, we have
\begin{align}\label{IPZk0}
	\int^{1}_{0}\|\{\I-\P\}Z_k(t)\|_{L^2_x(\Omega)L^2_D}^2\,dt\to 0,
\end{align}
and hence,
\begin{align*}
	\P Z_k\rightharpoonup\P Z\ \text{ weakly in }\ \int^{1}_{0}\|\cdot\|_{L^2_x(\Omega)L^2_D}^2\,dt.
\end{align*}
Moreover, we have from the nonlinear Boltzmann equation $(\pa_t+v\cdot\na_x)f_k=Lf_k+\Gamma(f_k,f_k)$ that
\begin{align}
	\label{1235}
	(\pa_t+v\cdot\na_x)Z_k=LZ_k+\Gamma(f_k,Z_k).
\end{align}
Rewriting this in the weak form: for any smooth compactly-support function $\Phi\in C^\infty_c((0,1)\times\Omega\times\R^3_v)$,
\begin{align}\label{1233}
	\big(Z_k,(\pa_t+v\cdot\na_x)\Phi\big)_{L^2_{t,x,v}((0,1)\times\Omega\times\R^3_v)}=\big(LZ_k+\Gamma(f_k,Z_k),\Phi\big)_{L^2_{t,x,v}((0,1)\times\Omega\times\R^3_v)}.
\end{align}
Notice from \eqref{L2global2}, \eqref{normalPf}, \eqref{140} and \eqref{140a} that
\begin{align*}
	(LZ_k,\Phi)_{L^2_{t,x,v}}=(L(\{\I-\P\}Z_k),\Phi)_{L^2_{t,x,v}}\le \|\{\I-\P\}Z_k\|_{L^2_{t,x}L^2_D}\|\Phi\|_{L^2_{t,x}L^2_D}\to 0,
\end{align*}
and
\begin{align*}
	\big(\Gamma(f_k,Z_k),\Phi\big)_{L^2_{t,x,v}}\le\|f_k\|_{L^\infty_{t,x,v}}\|Z_k\|_{L^2_{t,x}L^2_D}\|\Phi\|_{L^2_{t,x}L^2_D}\to 0,
\end{align*}
as $k\to\infty$. Then we take limit $k\to\infty$ in \eqref{1233} to obtain
\begin{align}\label{panaxZ}
	(\pa_t+v\cdot\na_x)Z=0,
\end{align}
in the sense of distribution.
From \eqref{IPZk0}, we have $\P Z=0$, and hence, by \cite[Lemma 6, pp. 736]{Guo2009}, we have
\begin{Lem}[\cite{Guo2009}, Lemma 6]
	There exist constants $a_0,c_0,c_1,c_2$, and constant vectors $b_0,b_1$ and $\varpi$ such that $Z(t,x,v)$ takes the form:
	\begin{multline}\label{499}
		Z(t,x,v)=\bigg(\Big\{\frac{c_0}{2}|x|^2-b_0\cdot x+a_0\Big\}+\big\{-c_0tx-c_1x+\varpi\times x+b_0t+b_1\big\}\times v\\
		+\Big\{\frac{c_0t^2}{2}+c_1t+c_2\Big\}|v|^2\bigg)\sqrt\mu.
	\end{multline}
	Moreover, these constants are finite:
	\begin{align*}
		|a_0|+|c_0|+|c_1|+|c_2|+|b_0|+|b_1|+|\varpi|<\infty.
	\end{align*}
\end{Lem}
%
%
The following subsections are devoted to proving the following Lemma, which leads to a contradiction.
\begin{Lem}\label{Lem123}
	Assume Proposition \ref{macroLem} is false, and let $Z_k,Z$ be defined as the above in Subsection \ref{Sec121}. Then $Z_k$ converge strongly to $Z$ in the sense that
	\begin{align}\label{1238}
		\int^1_0\|Z_k-Z\|_{L^2_x(\Omega)L^2_D}^2\,dt\to 0,
	\end{align}
	as $k\to\infty$. Moreover,
	\begin{align}\label{ZC}
		\int^1_0\|Z\|_{L^2_x(\Omega)L^2_D}^2\ge C>0,
	\end{align}
	for some $C>0$. Furthermore,
	\begin{enumerate}
		\item for the inflow boundary case, $Z(t,x,v)=0$ for $(t,x,v)\in[0,1]\times\pa\Omega\times\R^3_v$;
		{{\smallskip}} \item for the Maxwell reflection boundary case, $Z(t,x,v)=R_D Z(t,x,v)$ for $(t,x,v)\in[0,1]\times\Si_+$. Moreover, if $\ve>0$, then $Z=0$ on $\Si_+$. If $\ve=0$, then for $t\ge 0$,
		\begin{align}\label{conver}
			\int_{\Omega\times\R^3_v}Z(t,x,v)\mu^{\frac{1}{2}}(v)\,dxdv=0.
		\end{align}
	\end{enumerate}
	
\end{Lem}
We first show that Lemma \ref{Lem123} implies Proposition \ref{macroLem}.

\begin{proof}[Proof of Proposition \ref{macroLem}]
	Assume that Proposition \ref{macroLem} is false. Then by the above construction of $Z_k,Z$ in Subsection \ref{Sec121} and Lemma \ref{Lem123}, we have the following:
	
	{{\smallskip}}\noindent{\bf The case of inflow boundary.} In this case, we have from Lemma \ref{Lem123} that $Z=0$ on $[0,1]\times\pa\Omega\times\R^3_v$. Then by \eqref{499}, and comparing the coefficients in front of the polynomials of $v$, we have
	\begin{align*}
		\frac{c_0}{2}|x|^2-b_0\cdot x+a_0
		& \equiv -c_0tx-c_1x+\varpi\times x+b_0t+b_1 \\
		& \equiv\frac{c_0t^2}{2}+c_1t+c_2\equiv 0,
	\end{align*}
	for any $(t,x,v)\in[0,1]\times\pa\Omega\times\R^3_v$. Therefore, $c_0=c_1=c_2=0$ and $a_0=b_0=0$. Then $\varpi\times x+b_1\equiv 0$, and equivalently,
	\begin{align}\label{1256}
		\varpi^2x_3-\varpi^3x_2+b^1_1=\varpi^3x_1-\varpi^1x_3+b^2_1=\varpi^1x_2-\varpi^2x_1+b^3_1=0,
	\end{align}
	for any $x\in\pa\Omega$. Since $\pa\Omega=\{x\,:\,\xi(x)=0\}$ is two dimensional surface, we know that $(x_1,x_2)$ are locally independent in some subsets of $\pa\Omega$. Hence, $\varpi^1=\varpi^2=b^3_1=0$ and then $\varpi^3=b^2_1=0$. Therefore, we obtain $Z=0$, which contradicts to \eqref{ZC}.

	{{\smallskip}}\noindent{\bf The case of Maxwell reflection boundary.}
	In this case, for any $\ve\in[0,1)$ and $\al\in(0,1)$, we have from Lemma \ref{Lem123} that
	\begin{align*}
		Z(t,x,v)=c_\mu\mu^{\frac{1}{2}}(v)\int_{v'\cdot n(x)>0}\{v'\cdot n(x)\}Z(t,x,v')\mu^{\frac{1}{2}}(v')\,dv',
	\end{align*}
	on $\Si_+$. Therefore, comparing the coefficient with \eqref{499} which is in the form $(1,v,|v|^2)\mu^{\frac{1}{2}}$, for any $(t,x)\in[0,1]\times\pa\Omega$, we have
	\begin{align*}
			& -c_0tx-c_1x+\varpi\times x+b_0t+b_1 \equiv \frac{c_0t^2}{2}+c_1t+c_2\equiv 0.
	\end{align*}
	which implies $c_0=c_1=c_2=b_0=0$, and hence, $\varpi=b_1=0$ as in \eqref{1256}.
	Thus, $Z(t,x,v)=a_0\mu^{\frac{1}{2}}(v)$.
	
	{{\smallskip}} If $\ve>0$, then we have from Lemma \ref{Lem123} that $Z=0$ on $\Si_+$, which implies $a_0=0$ and hence, $Z=0$ in $\Omega$.
	
	{{\smallskip}} If $\ve=0$, then by \eqref{conver} and $Z(t,x,v)=a_0\mu^{\frac{1}{2}}(v)$, we have $a_0=0$ and $Z=0$. In both cases, we have $Z\equiv0$, which contradicts to \eqref{ZC}.
	This completes the proof of Proposition \ref{macroLem}.
\end{proof}

The following subsections are devoted to the proof of Lemma \ref{Lem123}. 

\subsection{Decomposing the integrating domain}
We split the integrating domain in \eqref{1238} into several subsets. That is, we write
\begin{align*}
	[0,1]\times\Omega\times\R^3_v=\cup_{j=1}^5D_j,
\end{align*}
where
\begin{align}\label{D1234}
	\begin{aligned}
		D_1 & =\big([0,\de]\cup[1-\de,1]\big)\times\Omega\times\R^3_v, \\
		D_2 & =(\de,1-\de)\times\{x\in\Omega\,:\,\xi(x)<-2\de^6\}\times\{v\,:\,|v|\le\de^{-\frac{1}{4}}\}, \\
		D_3 & =(\de,1-\de)\times\{x\in\Omega\,:\,\xi(x)<-2\de^6\}\times\{v\,:\,|v|>\de^{-\frac{1}{4}}\}, \\
		D_4 & =(\de,1-\de)\times\Big\{(x,v)\,:\,0>\xi(x)\ge-2\de^6,\,\Big[|v|>2\de^{-\frac{1}{4}} \text{ or }|v\cdot n(x)|<3\de^2\Big]\Big\}, \\
		D_5 & =(\de,1-\de)\times\Big\{(x,v)\,:\,0>\xi(x)\ge-2\de^6,\,\Big[|v|\le2\de^{-\frac{1}{4}} \text{ and }|v\cdot n(x)|\ge3\de^2\Big]\Big\}, 
	\end{aligned}
\end{align}
where $\xi(x)$ is given in \eqref{Omega}. 
The subsets $D_4$ and $D_5$ correspond to grazing and non-grazing sets, respectively.
To prove the strong convergence \eqref{1238}, noticing the microscopic estimate \eqref{IPZk0}, it suffices to show that
\begin{align*}
	\int^1_0\|\P(Z_k-Z)\|_{L^2_x(\Omega)L^2_D}^2\,dt\to 0, \
	\text{ as $k\to\infty$},
\end{align*} which is equivalent to
\begin{align}\label{ZkZ}
	\sum_{j=1}^5\int^1_0\int_{\Omega}\Big|\int_{\R^3_v}(Z_k-Z)e_j\,dv\Big|^2\,dxdt\to 0, \
	\text{ as $k\to\infty$},
\end{align}
where $\{e_j\}$ is the orthonormal basis in $L^2_v$:
\begin{align*}
	\{e_j\}_{j=1}^5 = \big\{\mu^{\frac{1}{2}},v\mu^{\frac{1}{2}},\frac{|v|^2-3}{6}\mu^{\frac{1}{2}}\big\}.
\end{align*}

\subsubsection{Near the time boundary}

We \emph{claim} that there exists $K>0$ such that for $k\ge K$,
\begin{align}\label{timeboundary}
	\sup_{0\le t\le 1}\|Z_k(t)\|_{L^2_x(\Omega)L^2_v}\le C<\infty,
\end{align}
for some $C>0$ which is independent of $k$.
Then by the Banach-Alaoglu Theorem and the uniqueness of weak limit, we have
\begin{align*}
	\sup_{0\le t\le 1}\|Z(t)\|_{L^2_x(\Omega)L^2_v}\le C<\infty.
\end{align*}
Thus, the left-hand side of \eqref{ZkZ} within the domain $D_1$ given in \eqref{D1234} can be estimated as
\begin{align}\label{D1}
	\Big(\int^\de_0+\int_{1-\de}^1\Big)\int_{\Omega}|(Z_k-Z,e_j)_{L^2_v}|^2\,dxdt
	\le \de C\sup_{0\le t\le 1,\, k\ge 1}\|[Z_k(t),Z(t)]\|^2_{L^2_x(\Omega)L^2_v}
	\le \de C.
\end{align}

{{\smallskip}} We next prove the \emph{claim} \eqref{timeboundary}. For $T\in[0,1]$, by taking $L^2$ inner product of \eqref{1235} with $Z_k$ over $[0,T]\times\Omega\times\R^3_v$, we have
\begin{multline}\label{ZkT}
	\|Z_k(T)\|_{L^2_x(\Omega)L^2_v}^2
	+\int^T_0\int_{\Si_+}|v\cdot n||Z_k|^2\,dS(x)dvdt
	=\|Z_k(0)\|_{L^2_x(\Omega)L^2_v}^2\\
	+\int^T_0\int_{\Si_-}|v\cdot n||Z_k|^2\,dS(x)dvdt+2\int^T_0\big(LZ_k+\Gamma(f_k,Z_k),Z_k\big)_{L^2_x(\Omega)L^2_v}\,dt.
\end{multline}
By estimates \eqref{L1}, \eqref{L2} and \eqref{Gaesweight} for the collision terms, we have
\begin{multline*}
	\|Z_k(T)\|_{L^2_x(\Omega)L^2_v}^2
	+\int^T_0\int_{\Si_+}|v\cdot n||Z_k|^2\,dS(x)dvdt
	\le\|Z_k(0)\|_{L^2_x(\Omega)L^2_v}^2\\
	+\int^T_0\int_{\Si_-}|v\cdot n||Z_k|^2\,dS(x)dvdt+C\int^T_0\|Z_k\|_{L^2_x(\Omega)L^2_v}^2\,dt\\
	+2\big(-c_0+C\sup_{0\le t\le 1}\|\<v\>^{4}f_k\|_{L^\infty_x(\Omega)L^\infty_v}\big)\int^T_0\|Z_k\|_{L^2_x(\Omega)L^2_D}^2\,dt.
\end{multline*}
Using assumption \eqref{L2global2} to choose $K>0$ large enough that $\sup_{0\le t\le 1}\|\<v\>^4f_k\|_{L^\infty_x(\Omega)\L^\infty_v}\le\frac{c_0}{2C}$ and using Gr\"{o}nwall's inequality, we obtain
\begin{multline}\label{ZkT0}
	\|Z_k(T)\|_{L^2_x(\Omega)L^2_v}^2
	+\int^T_0\int_{\Si_+}|v\cdot n|e^{C(T-t)}|Z_k|^2\,dS(x)dvdt
	\le C\|Z_k(0)\|_{L^2_x(\Omega)L^2_v}^2\\
	+\int^T_0\int_{\Si_-}|v\cdot n|e^{C(T-t)}|Z_k|^2\,dS(x)dvdt,
\end{multline}
where we used $T\le 1$. On the other hand, for the term $\|Z_k(0)\|_{L^2_x(\Omega)L^2_v}^2$ in \eqref{ZkT0}, we have from \eqref{ZkT}, \eqref{L2global2} and collisional estimates \eqref{L2}, \eqref{Gaesweight} that
\begin{multline}\label{ZkT1}
	\|Z_k(0)\|_{L^2_x(\Omega)L^2_v}^2
	\le
	\|Z_k(T)\|_{L^2_x(\Omega)L^2_v}^2
	+\int^T_0\int_{\Si_+}|v\cdot n||Z_k|^2\,dS(x)dvdt\\-\int^T_0\int_{\Si_-}|v\cdot n||Z_k|^2\,dS(x)dvdt
	+C\int^T_0\|Z_k\|_{L^2_x(\Omega)L^2_D}^2\,dt.
\end{multline}
Integrating \eqref{ZkT1} over $T\in[0,1]$, and using \eqref{ZkD} and \eqref{esD} with $\ga+2s\ge 0$, one has
\begin{multline}\label{ZkT2}
	\|Z_k(0)\|_{L^2_x(\Omega)L^2_v}^2
	\le
	2
	+\int^1_0\int^T_0\int_{\Si_+}|v\cdot n||Z_k|^2\,dS(x)dvdtdT\\
	-\int^1_0\int^T_0\int_{\Si_-}|v\cdot n||Z_k|^2\,dS(x)dvdtdT.
\end{multline}

\smallskip\noindent{\bf The inflow boundary case.} In this case, by \eqref{140}, we have from \eqref{ZkT0} and \eqref{ZkT2} that
\begin{align*}
	\|Z_k(T)\|_{L^2_x(\Omega)L^2_v}^2
	\le C\|Z_k(0)\|_{L^2_x(\Omega)L^2_v}^2+\frac{1}{k}\le C,
\end{align*}
for any $T\in[0,1]$, which implies \eqref{timeboundary} in the inflow case.

{{\smallskip}}\noindent{\bf The reflection boundary case.} In this case, since $Z_k|_{\Si_-}=(1-\ve)RZ_k$, using boundary estimate \eqref{101} and \eqref{140a}, we have from \eqref{ZkT0} and \eqref{ZkT2} that
\begin{align*}
	\|Z_k(T)\|_{L^2_x(\Omega)L^2_v}^2
	& \le C\|Z_k(0)\|_{L^2_x(\Omega)L^2_v}^2 \\
	& \le C
	+C\ve\int^1_0\int^T_0\int_{\Si_+}|v\cdot n||Z_k|^2\,dS(x)dvdtdT
	\\&\qquad+C(1-\ve)\al\int^1_0\int^T_0\int_{\Si_+}|v\cdot n||f-R_Df|^2\,dS(x)dvdtdT\\
	& \le C+\frac{C}{k}.
\end{align*}
This implies \eqref{timeboundary} in the \emph{Maxwell} reflection case, and we conclude the \emph{claim} \eqref{timeboundary}.

\subsubsection{The interior set}
For the domains $D_2$ and $D_3$, we will use smooth cutoff functions to represent them. 
Let $\chi_0$ be the smooth cutoff function such that
\begin{align}\label{chi000}
	\chi_0(t,x,v)=\left\{\begin{aligned}
		& 1,\quad \text{ if }t\in[\de,1-\de]\text{ and }\xi(x)\le -2\de^6\text{ and }|v|\le2\de^{-\frac{1}{4}}, \\
		& 0,\quad \text{ if }t\in\big[0,\frac{\de}{2}\big]\cup\big[1-\frac{\de}{2},1\big]\text{ or }\xi(x)\ge-\de^6\text{ or }|v|\ge3\de^{-\frac{1}{4}},
	\end{aligned}\right.
\end{align}
which satisfies
\begin{align*}
	|\na_{t,x,v}\chi_0|\le C\de^{-6}.
\end{align*}
Then the left-hand side \eqref{ZkZ} within domain $D_2\cup D_3$ can be written as
\begin{multline}\label{D222}
	\int^{1-\de}_\de\int_{\xi(x)<-2\de^6}\Big|\int_{\R^3}(Z_k-Z)e_j\,dv\Big|^2\,dxdt\\
	\le 2\int^{1}_0\int_{\xi(x)<-2\de^6}\Big\{\Big|\int_{\R^3_v}\chi_0(Z_k-Z)e_j\,dv\Big|^2+\Big|\int_{\R^3_v}(1-\chi_0)(Z_k-Z)e_j\,dv\Big|^2\Big\}\,dxdt.
\end{multline}
Note that $\chi_0Z_k$ is supported in
\begin{align*}
	\{(t,x,v)\,:\,t\in\big[\frac{\de}{2},1-\frac{\de}{2}\big],\ \xi(x)\le-\de^6,\ |v|\le 3\de^{-\frac{1}{4}}\},
\end{align*}
which is a compact subset in the open bounded set $\Omega$. Moreover, by \eqref{1235}, $\chi_0Z_k$ satisfies the equation
\begin{align*}
	(\pa_t+v\cdot\na_x)(\chi_0Z_k) = Z_k(\pa_t+v\cdot\na_x)\chi_0 +\chi_0LZ_k+\chi_0\Gamma(f_k,Z_k).
\end{align*}
From \eqref{ZkD}, we have $\int^1_0\|Z_k\|^2_{L^2_x(\Omega)L^2_D}\,ds\le 2<\infty$, and hence,
\begin{align*}
	\chi_0Z_k & \in L^2_{t,x,v}(\R_t\times\R^3_x\times\R^3), \\
	\{[\pa_t+v\cdot \na_x]\chi_0\}Z_k & \in L^2_{t,x,v}(\R_t\times\R^3_x\times\R^3)), \\
	\chi_0LZ_k+\chi_0\Gamma(f_k,Z_k) & \in L^2_{t,x}H^{-1}_v(\R_t\times\R^3_x\times\R^3)).
\end{align*}
where we used \eqref{GaesweightDvs} with $s\in(0,1)$; note from \eqref{Gaesweight}, \eqref{L2} and \eqref{esD} that 
\begin{align*}
	\|\chi_0LZ_k+\chi_0\Gamma(f_k,Z_k)\|_{H^{-1}_v}&=\sup_{\|\phi\|_{H^1_v}\le 1}\big(\chi_0LZ_k+\chi_0\Gamma(f_k,Z_k),\phi\big)_{L^2_v}\\
	&\le C\|\<v\>^4f_k\|_{L^\infty_v}\|Z_k\|_{L^2_D}\sup_{\|\phi\|_{H^1_v}\le 1}\|\chi_0\phi\|_{L^2_D}\le C. 
\end{align*}
Then we deduce from the averaging lemma \cite[Theorem 5]{Diperna1989a} (see also \cite{Golse1988} for earlier compactness result) that
\begin{align*}
	\int_{\R^3}\chi_0Z_k(v)\phi(v)\,dv\in H^{\frac{1}{4}}(\R_t\times\R^3_x)
\end{align*}
uniformly in $k$ for any smooth function $\phi(v)$ with compact support. Since $\chi_0$ has compact support in $v$, it then follows that the sequence
\begin{align*}
	\int_{\R^3}\chi_0Z_k(v)e_j\,dv
\end{align*}
is compact in $L^2(\R_t\times\R^3_x)$.
Therefore, by uniqueness of the weak limit, up to a subsequence, the first right-hand term of \eqref{D222} can be estimated as 
\begin{align}
	\label{D2}
	\int^{1-\de}_\de\int_{\xi(x)<-2\de^6}\Big|\int_{\R^3}\chi_0(Z_k-Z)e_j\,dv\Big|^2\,dxdt
	\to 0,
\end{align}
as $k\to\infty$ for any $j$ and fixed $\de>0$.

\subsubsection{The large-velocity set}
For second right-hand term of \eqref{D222} with large-velocity, we have from \eqref{chi000} that 
\begin{align}\label{D3}\notag
	& \int^{1-\de}_\de\int_{\xi(x)<-2\de^6}\Big|\int_{\R^3}(1-\chi_0)(Z_k-Z)e_j\,dv\Big|^2\,dxdt \\
	& \notag\quad\le C\int^1_0\int_{\Omega}\Big|\int_{|v|>2\de^{-\frac{1}{4}}}\<v\>^2|Z_k-Z|\mu^{\frac{1}{4}}e^{-\frac{\de^{-1/2}}{2}}\,dv\Big|^2 \\
	& \notag\quad\le Ce^{-\frac{\de^{-1/2}}{2}}\int^{1}_{0}\|[Z_k(t),Z(t)]\|_{L^2_x(\Omega)L^2_D}^2\,dt \\
	& \quad \le Ce^{-\frac{\de^{-1/2}}{2}},
\end{align}
where we used \eqref{ZkD} and \eqref{ZD}.

\subsubsection{The grazing set}
For the estimate \eqref{ZkZ} within the grazing set $D_4$ defined in \eqref{D1234}, by H\"{o}lder's inequality, we have
\begin{align*}
	& \int^{1-\de}_{\de}\int_{0>\xi(x)\ge-2\de^6}\Big|\int_{|v|>2\de^{-\frac{1}{4}} \text{ or }|v\cdot n(x)|<3\de^2}(Z_k-Z)e_j\,dv\Big|^2\,dxdt \\
	& \quad\le\int^1_0\int_{\Omega}\|Z_k-Z\|_{L^2_D}^2\,dxdt\,\sup_{x\in\Omega}\int_{|v|>2\de^{-\frac{1}{4}} \text{ or }|v\cdot n(x)|<3\de^2}\<v\>^2\mu^{\frac{1}{2}}\,dv \\
	& \quad\le C\sup_{x\in\Omega}\int_{|v|>2\de^{-\frac{1}{4}} \text{ or }|v\cdot n(x)|<3\de^2}\<v\>^2\mu^{\frac{1}{2}}\,dv,
\end{align*}
where we used \eqref{esD}, \eqref{ZkD} and \eqref{ZD}.
For the large-velocity region, we have
\begin{align}\label{smallmu}
	\int_{|v|>2\de^{-\frac{1}{4}}}\<v\>^2\mu^{\frac{1}{2}}\,dv
	\le Ce^{-\frac{\de^{-1/2}}{2}}\int_{\R^3_v}\<v\>^2\mu^{\frac{1}{4}}\,dv\le Ce^{-\frac{\de^{-1/2}}{2}}.
\end{align}
For the grazing region, we use a rotation $v\mapsto \ti Rv$ with $\ti R^Tn=(1,0,0)$ and $|v|=|\ti Rv|$ ($\ti R$ is a orthogonal matrix and $\ti R^T$ is the transpose of $\ti R$) to deduce
\begin{align}\label{smallmu1}
	\int_{|v\cdot n(x)|<3\de^2}\<v\>^2\mu^{\frac{1}{2}}\,dv
	& \le \int_{|v_1|<3\de^2}\<v\>^2\mu^{\frac{1}{2}}\,dv
	\le C\de^2.
\end{align}
Collecting the above three estimates, we obtain
\begin{align}\label{D4}
	\int^{1-\de}_{\de}\int_{0>\xi(x)\ge-2\de^6}\Big|\int_{|v|>2\de^{-\frac{1}{4}} \text{ or }|v\cdot n(x)|<3\de^2}(Z_k-Z)e_j\,dv\Big|^2\,dxdt\le C\big(\de^2+e^{-\frac{\de^{-1/2}}{2}}\big).
\end{align}

\subsubsection{The non-grazing set}
In order to prove the convergence of \eqref{ZkZ} in domain $D_5$ given by \eqref{D1234}, we use the smooth cutoff functions as before. 

Recall that we assume $\Omega$ is a bounded open set defined by
\begin{align*}
	\Omega=\{x\in\R^3_x\,:\,\xi(x)<0\}.
\end{align*}
For any $\de>0$, we denote $\chi_1:\R\to[0,1]$, $\chi_2:\R^3\to[0,1]$ and $\chi_3:\R^3\to[0,1]$ as smooth cutoff functions satisfying
\begin{align}\label{chi123}
	\begin{aligned}
		& \chi_1(r)=\left\{\begin{aligned}
			& 1\quad\text{ if } r\ge3\de^2, \\
			& 0\quad\text{ if } r<2\de^2,
		\end{aligned}\right.
		\quad \chi_2(v) = \left\{\begin{aligned}
			& 1\quad\text{ if } |v|\le 2\de^{-\frac{1}{4}}, \\
			& 0\quad\text{ if } |v|>4\de^{-\frac{1}{4}},
		\end{aligned}\right. \\
		& \chi_3(x) = \left\{\begin{aligned}
			& 1\quad\text{ if } 2\de^6\ge\xi(x)\ge -2\de^6, \\
			& 0\quad\text{ if } \xi(x)< -3\de^6\text{ or }\xi(x)>3\de^6,
		\end{aligned}\right.
	\end{aligned}
\end{align}x
and
\begin{align*}
	\begin{aligned}
		& |\chi_1'(r)|\le C\de^{-2},\quad |\chi_1''(r)|\le C\de^{-4}, \quad |\na_v\chi_2(v)|\le C\de^{\frac{1}{4}}, \\
		& |\na^2_v\chi_2(v)|\le C\de^{\frac{1}{2}} \quad |\na_x\chi_3(x)|\le C\de^{-6}, \quad |\na^2_v\chi_3(x)|\le C\de^{-12}.
	\end{aligned}
\end{align*}
Recall that we assume that the outward unit normal vector $n=n(x)$ on $\pa\Omega$ has an extension to $\R^3_x$ as in \eqref{naxi1} such that
\begin{align}\label{nW24}
	n(x)\in W^{2,\infty}(\R^3_x).
\end{align} 
Then we construct the backward/forward smooth cutoff function as
\begin{align}\label{chipm}
	\begin{aligned}
		\chi^\de_+(t,x,v;{T}) & =\chi_1\big(v\cdot n(x-v\{t-{T}\})\big)\chi_2(v)\chi_3(x-v\{t-{T}\}),\ \text{ for }0\le T\le t, \\
		\chi^\de_-(t,x,v;{T}) & =\chi_1\big(-v\cdot n(x-v\{t-{T}\})\big)\chi_2(v)\chi_3(x-v\{t-{T}\}),\ \text{ for }0\le t\le T,
	\end{aligned}
\end{align}
which satisfies
\begin{align}\label{pachi1}
	\pa_t\chi_\pm^\de+v\cdot\na_x\chi_\pm^\de=0.
\end{align}
We conclude the properties of cutoff functions $\chi_\pm^\de$ in the following Lemma.
\begin{Lem}\label{chipmLem}
	Let $t,T\in[0,1]$ and $\chi_\pm^\de$ be defined in \eqref{chipm}. Fix a sufficiently small $\de>0$. Then
	
	\begin{enumerate}
		\item Let $|t-T|\le \de^3$ and $x\in\Omega$. If $\chi^\de_+(t,x,v;T)\neq 0$ then $v\cdot n(x)>\de^2>0$. If $\chi^\de_-(t,x,v;T)\neq 0$ then $v\cdot n(x)<-\de^2<0$.
		
		{{\smallskip}}
		\item Let $x\in\Omega$ such that $0>\xi(x)\ge-3\de^6$. Then
		\begin{align*}
			\chi^\de_+(T-\de^3,x,v;T)=0\ \text{ and }\ \chi^\de_-(T+\de^3,x,v;T)=0.
		\end{align*}
		
		\smallskip
		\item Let $(x,v)\in\ol\Omega\times\R^3_v$ such that $|v|\le 2\de^{-\frac{1}{4}}$ and $0\ge\xi(x)\ge-2\de^6$. If $v\cdot n(x)\ge3\de^2$, then $\chi^\de_+(T,x,v;T)=1$. If $v\cdot n(x)\le-3\de^2$, then $\chi^\de_-(T,x,v;T)=1$.
		
		{{\smallskip}} \item We have
		\begin{align}\label{colli3x}
			\|[\na_v\chi_\pm^\de,\na_v^2\chi_\pm^\de](t,x,v;T)\|_{L^\infty_t([0,1])L^\infty_x(\Omega)L^\infty_v(\R^3_v)}
			\le C\de^{-12},
		\end{align}
		for some $C>0$ depends only on $\|n\|_{W^{2,\infty}(\Omega)}$ and is independent of $T$.
	\end{enumerate}
\end{Lem}
Note that we use smooth cutoff functions $\chi_\pm^\de$ whereas \cite[Lemma 10]{Guo2009} used indicator functions. 
\begin{proof}
	To prove (1), we have from \eqref{chi123} and \eqref{chipm} that for $(t,x,v)\in[0,1]\times\Omega\times\R^3_v$, if $\chi^\de_+(t,x,v)\neq 0$, then
	\begin{align*}
		v\cdot n(x-v\{t-{T}\})\ge2\de^2,\quad |v|\le 4\de^{-\frac{1}{4}},\quad \xi(x-v\{t-{T}\})\ge -3\de^6.
	\end{align*}
	Thus, by \eqref{nW24}, for $|t-T|\le\de^3$,
	\begin{align*}
		v\cdot n(x) & =v\cdot n(x-v\{t-{T}\})+\big[v\cdot n(x)-v\cdot n(x-v\{t-{T}\})\big] \\
		& \ge2\de^2-\sup_{\th\in[0,1]}|\na n(x-\th v\{t-T\})|\times|t-T||v|^2 \\
		& \ge2\de^2-C_n\de^{3-\frac{1}{2}}
		\ge \de^2,
	\end{align*}
	if we choose $\de>0$ sufficiently small, which depends only on $n$, i.e. depends on $\Omega$.
	Similarly, for the forward cutoff function $\chi^\de_-$, if $\chi^\de_-(t,x,v)\neq 0$, then
	\begin{align*}
		-v\cdot n(x-v\{t-{T}\})\ge2\de^2,\quad |v|\le 4\de^{-\frac{1}{4}},\quad \xi(x-v\{t-{T}\})\ge -3\de^6.
	\end{align*}
	Thus, for $|t-T|\le\de^3$,
	\begin{align*}
		v\cdot n(x) & =v\cdot n(x-v\{t-{T}\})+\big[v\cdot n(x)-v\cdot n(x-v\{t-{T}\})\big] \\
		& \le -2\de^2+C_n\de^{3-\frac{1}{2}}\le -\de^2,
	\end{align*}
	with small $\de>0$.
	
	{{\smallskip}} To prove (2), let $x\in\Omega$ such that $0>\xi(x)\ge-3\de^6$. If $\chi^\de_+(T-\de^3,x,v;T)\neq 0$, then we have
	\begin{align}\label{1246}
		v\cdot n(x+v\de^3)\ge2\de^2,\quad |v|\le 4\de^{-\frac{1}{4}},\quad 3\de^6\ge\xi(x+v\de^3)\ge -3\de^6.
	\end{align}
	However,
	\begin{align*}
		\xi(x+v\de^3) = \xi(x)+\de^3v\cdot\na\xi(x)+\de^3v\cdot\na^2\xi(\bar x)\cdot\de^3v,
	\end{align*}
	for some $\bar x$ is between $x$ and $x+v\de^3$. Since $n=\frac{\na\xi(x)}{|\na\xi(x)|}$ as in \eqref{nW24}, we have from the first assertion (1), i.e. $v\cdot n(x)\ge \de^2$, that 
	\begin{align*}
		\de^3v\cdot\na\xi(x)=\de^3|\na\xi(x)|v\cdot n(x)\ge \de^5c_\xi,
	\end{align*}
	where $c_\xi$ is the lower bound of $|\na\xi(x)|$ in $\{0>\xi(x)\ge-3\de^6\}$. Therefore,
	\begin{align*}
		\xi(x+v\de^3) & \ge -3\de^6+\de^5c_\xi-C_\xi\de^6|v|^2 \\
		& \ge -3\de^6+\de^5c_\xi-C_\xi\de^{6-\frac{1}{2}}\ge\frac{\de^5c_\xi}{2}>3\de^6,
	\end{align*}
	with the upper bound $C_\xi>0$ of $|\na^2\xi(x)|$ in $\{0>\xi(x)\ge-3\de^6\}$, where $\de>0$ is chosen to be small enough. This contradicts \eqref{1246} and hence, $\chi^\de_+(T-\de^3,x,v;T)=0$.
	
	\smallskip 
	Similarly, let $0>\xi(x)\ge-3\de^6$. If $\chi^\de_-(T+\de^3,x,v;T)\neq 0$, then
	\begin{align}\label{1246a}
		v\cdot n(x-v\de^3)\le-2\de^2,\quad |v|\le 4\de^{-\frac{1}{4}},\quad 3\de^6\ge\xi(x-v\de^3)\ge -3\de^6.
	\end{align}
	However, we have from (1) that $v\cdot n(x)<-\de^2$, and hence, by choosing $\de>0$ small enough,
	\begin{align*}
		\xi(x-v\de^3) & = \xi(x)-\de^3v\cdot\na\xi(x)+\de^3v\cdot\na^2\xi(\bar x)\cdot\de^3v \\
		& \ge -3\de^6+\de^5c_\xi-C_\xi\de^{6-\frac{1}{2}}> 3\de^6,
	\end{align*}
	which contradicts \eqref{1246a}. Thus, $\chi^\de_-(T+\de^3,x,v;T)=0$ for any $x$ satisfying $0>\xi(x)\ge-3\de^6$.

	{{\smallskip}} To prove (3), letting $(t,x,v)$ such that
	\begin{align*}
		t=T,\ v\cdot n(x)\ge3\de^2,\ |v|\le 2\de^{-\frac{1}{4}},\ 0>\xi(x)\ge-2\de^6,
	\end{align*}
	we have from \eqref{chi123} that
	then $\chi^\de_+(T,x,v;T)=1$. Similarly, for $(x,v)$ satisfying $v\cdot n(x)\le-3\de^2$, $|v|\le 2\de^{-\frac{1}{4}}$ and $0>\xi(x)\ge-2\de^6$, we have $\chi^\de_-(T,x,v;T)=1$.
	
	{{\smallskip}} To prove (4), similar to the proof of \eqref{chideesti} at the end of Subsection \ref{Sec103}, it's direct to calculate the derivatives of $\chi_\pm^\de$ and deduce
	\begin{align*}
		|\na_v\chi_\pm^\de| & \le C\big(1+\|n\|_{L^\infty_x}+\|\na_xn\|_{L^\infty_x}\big)\big(\|[\chi_1,\chi_2,\chi_3,\na\chi_1,\chi'_2,\na\chi_3]\|_{L^\infty}\big) \\
		& \le C\de^{-6}\big(1+\|n\|_{L^\infty_x}+\|\na_xn\|_{L^\infty_x}\big).
	\end{align*}
	and
	\begin{align*}
		|\na^2_v\chi_\pm^\de| & \le C\big(1+\|n\|_{W^{2,\infty}_x}\big)\big(\|[\chi_1,\chi_2,\chi_3,\na\chi_1,\chi'_2,\na\chi_3,\na^2\chi_1,\chi''_2,\na^2\chi_3]\|_{L^\infty}\big) \\
		& \le C\de^{-12}\big(1+\|n\|_{W^{2,\infty}_x}\big).
	\end{align*}
	for some generic constant $C>0$.
	This implies \eqref{colli3x} and completes the proof of Lemma \ref{chipmLem}.
	
\end{proof}

\smallskip 
Let $0\le T-\de^3\le T\le T+\de^3\le 1$.
Denote the smooth cutoff functions $\chi_\pm^\de(t,x,v;T)$ as in \eqref{chipm}. 
By \eqref{1235}, \eqref{panaxZ} and \eqref{pachi1}, $\chi_\pm^\de(Z_k-Z)$ satisfies the equation
\begin{align}\label{1371}
	\big(\pa_t+v\cdot\na_x\big)\big(\chi_\pm^\de(Z_k-Z)\big) = \chi_\pm^\de LZ_k+\chi_\pm^\de\Gamma(f_k,Z_k).
\end{align}
We denote the inner domain
\begin{align*}
	\Omega_\de=\{x\in\R^3_x\:\,\xi(x)<-2\de^6\}.
\end{align*}
By our construct of $n(x)$ on $\pa\Omega_\de$ in \eqref{nW24}, we know that $n(x)$ is also the outward normal unit vector on boundary
\begin{align}\label{paOmegade}
	\pa\Omega_\de=\{x\in\R^3_x\:\,\xi(x)=-2\de^6\}.
\end{align}
Then we denote the corresponding incoming and outgoing sets as
\begin{align*}
	\begin{aligned}
		\Si^\de_+ & =\{(x,v)\in\pa\Omega_\de\times\R^3_v\,:\, v\cdot n(x)>0\}, \\
		\Si^\de_- & =\{(x,v)\in\pa\Omega_\de\times\R^3_v\,:\, v\cdot n(x)<0\}, \\
		\Si^\de_0 & =\{(x,v)\in\pa\Omega_\de\times\R^3_v\,:\, v\cdot n(x)=0\}.
	\end{aligned}
\end{align*}
For $\chi_\pm^\de=\chi^\de_+$, taking $L^2$ inner product of \eqref{1371} with $\chi^\de_+(Z_k-Z)$ over $[T-\de^3,T]\times\big(\Omega\setminus\Omega_\de\big)\times\R^3_v$, we have
\begin{multline*}
	\|\chi^\de_+(Z_k-Z)(T)\|_{L^2_x(\Omega\setminus\Omega_\de)L^2_v}^2+\int_{T-\de^3}^T\int_{\Si_+}|v\cdot n(x)|\big(\chi^\de_+(Z_k-Z)\big)^2\,dS(x)dvdt\\-\int_{T-\de^3}^T\int_{\Si_-}|v\cdot n(x)|\big(\chi^\de_+(Z_k-Z)\big)^2\,dS(x)dvdt\\
	= \|\chi^\de_+(Z_k-Z)(T-\de^3)\|_{L^2_x(\Omega\setminus\Omega_\de)L^2_v}^2+\int_{T-\de^3}^T\int_{\Si^\de_+}|v\cdot n(x)|\big(\chi^\de_+(Z_k-Z)\big)^2\,dS(x)dvdt\\-\int_{T-\de^3}^T\int_{\Si^\de_-}|v\cdot n(x)|\big(\chi^\de_+(Z_k-Z)\big)^2\,dS(x)dvdt\\
	+\big(\chi^\de_+ LZ_k+\chi^\de_+\Gamma(f_k,Z_k),\chi^\de_+(Z_k-Z)\big)_{L^2_{t,x,v}([T-\de^3,T]\times(\Omega\setminus\Omega_\de)\times\R^3_v)}.
\end{multline*}
Using Lemma \ref{chipmLem} (1) and (2), we have
\begin{multline}\label{chi+}
	\|\chi^\de_+(Z_k-Z)(T)\|_{L^2_x(\Omega\setminus\Omega_\de)L^2_v}^2+\int_{T-\de^3}^T\int_{\Si_+}|v\cdot n(x)|\big(\chi^\de_+(Z_k-Z)\big)^2\,dS(x)dvdt\\
	=\int_{T-\de^3}^T\int_{\Si^\de_+}|v\cdot n(x)|\big(\chi^\de_+(Z_k-Z)\big)^2\,dS(x)dvdt\\
	+\big(\chi^\de_+LZ_k+\chi^\de_+\Gamma(f_k,Z_k),\chi^\de_+(Z_k-Z)\big)_{L^2_{t,x,v}([T-\de^3,T]\times(\Omega\setminus\Omega_\de)\times\R^3_v)}.
\end{multline}
Similar calculations can be carried out for $\chi^\de_-$ by taking $L^2$ inner product of \eqref{1371} with $\chi^\de_-(Z_k-Z)$ over $[T-\de^3,T]\times\big(\Omega\setminus\Omega_\de\big)\times\R^3_v$:
\begin{multline*}
	\|\chi^\de_-(Z_k-Z)(T+\de^3)\|_{L^2_x(\Omega\setminus\Omega_\de)L^2_v}^2+\int^{T+\de^3}_T\int_{\Si_+}|v\cdot n(x)|\big(\chi^\de_-(Z_k-Z)\big)^2\,dS(x)dvdt\\-\int^{T+\de^3}_T\int_{\Si_-}|v\cdot n(x)|\big(\chi^\de_-(Z_k-Z)\big)^2\,dS(x)dvdt\\
	= \|\chi^\de_-(Z_k-Z)(T)\|_{L^2_x(\Omega\setminus\Omega_\de)L^2_v}^2+\int^{T+\de^3}_T\int_{\Si^\de_+}|v\cdot n(x)|\big(\chi^\de_-(Z_k-Z)\big)^2\,dS(x)dvdt\\-\int^{T+\de^3}_T\int_{\Si^\de_-}|v\cdot n(x)|\big(\chi^\de_-(Z_k-Z)\big)^2\,dS(x)dvdt\\
	+\big(\chi^\de_- LZ_k+\chi^\de_-\Gamma(f_k,Z_k),\chi^\de_-(Z_k-Z)\big)_{L^2_{t,x,v}([T,T+\de^3]\times(\Omega\setminus\Omega_\de)\times\R^3_v)}.
\end{multline*}
Then one has from Lemma \ref{chipmLem} (1) and (2) that
\begin{multline}\label{chi-}
	\|\chi^\de_-(Z_k-Z)(T)\|_{L^2_x(\Omega\setminus\Omega_\de)L^2_v}^2+\int^{T+\de^3}_T\int_{\Si_-}|v\cdot n(x)|\big(\chi^\de_-(Z_k-Z)\big)^2\,dS(x)dvdt\\
	= \int^{T+\de^3}_T\int_{\Si^\de_-}|v\cdot n(x)|\big(\chi^\de_-(Z_k-Z)\big)^2\,dS(x)dvdt\\
	+\big(\chi^\de_- LZ_k+\chi^\de_-\Gamma(f_k,Z_k),\chi^\de_-(Z_k-Z)\big)_{L^2_{t,x,v}([T-\de^3,T]\times(\Omega\setminus\Omega_\de)\times\R^3_v)}.
\end{multline}
By \eqref{colli3x} and \eqref{Gaesweight}, we know that
\begin{align}\label{chi+1}
	& \notag\big(\chi^\de_+ L\{\I-\P\}Z_k+\chi^\de_+\Gamma(f_k,Z_k),\chi^\de_+(Z_k-Z)\big)_{L^2_{t,x,v}([T-\de^3,T]\times(\Omega\setminus\Omega_\de)\times\R^3_v)} \\
	& \notag\quad\le C\int_{T-\de^3}^T\Big(\|\{\I-\P\}Z_k\|_{L^2_x(\Omega\setminus\Omega_\de)L^2_D}+\Big\{\sup_{0\le t\le 1}\|\<v\>^{4}f_k\|_{L^\infty_{x,v}(\Omega\times\R^3_v)}\Big\}\|Z_k\|_{L^2_x(\Omega\setminus\Omega_\de)L^2_D}\Big) \\
	& \notag\qquad\qquad\times\|(\chi^\de_+)^2(Z_k-Z)\|_{L^2_x(\Omega\setminus\Omega_\de)L^2_D}\,dt \\
	& \quad\le \frac{C_\de}{k},
\end{align}
for some constant $C_\de>0$ depending on $\de>0$,
where we used \eqref{L2global2}, \eqref{140}, \eqref{140a} to control the corresponding quantities. 
Note that we also used $\|Z_k\|_{L^2_D}=\|(\ti a^{\frac{1}{2}})^wZ_k\|_{L^2_v}$ from \eqref{tia}, and \cite[Lemma 2.3]{Deng2020a} and \eqref{colli3x} to deduce
\begin{align}\label{chi+D}
	\|(\chi^\de_+)^2(Z_k-Z)\|_{L^2_D}\le C_\de\|Z_k-Z\|_{L^2_D}, 
\end{align}
since $\chi^\de_+$ is a smooth cutoff function in $v\in\R^3$ with compact support. 
Similarly,
\begin{align}\label{chi-1}
	& \big(\chi^\de_- L\{\I-\P\}Z_k+\chi^\de_-\Gamma(f_k,Z_k),\chi^\de_-(Z_k-Z)\big)_{L^2_{t,x,v}([T,T+\de^3]\times(\Omega\setminus\Omega_\de)\times\R^3_v)}\le \frac{C_\de}{k}.
\end{align}
Noticing $\chi^\de_+(T)+\chi^\de_-(T)=1$ for $|v|\le 2\de^{-\frac{1}{4}}$ and $|v\cdot n(x)|\ge3\de^2$, and substituting \eqref{chi+1} and \eqref{chi-1} into \eqref{chi+} and \eqref{chi-} respectively, we have
\begin{multline}\label{ZkZT}
	\|(Z_k-Z)(T)\|_{L^2_{x,v}(\{x\in\Omega\setminus\Omega_\de,\,|v|\le 2\de^{-\frac{1}{4}},\,|v\cdot n(x)|\ge3\de^2\})}^2+\int_{T-\de^3}^T\int_{\Si_+}|v\cdot n(x)|\big(\chi^\de_+(Z_k-Z)\big)^2\,dS(x)dvdt\\
	+\int^{T+\de^3}_T\int_{\Si_-}|v\cdot n(x)|\big(\chi^\de_-(Z_k-Z)\big)^2\,dS(x)dvdt
	\le\int_{T-\de^3}^T\int_{\Si^\de_+}|v\cdot n(x)|\big(\chi^\de_+(Z_k-Z)\big)^2\,dS(x)dvdt\\
	+\int^{T+\de^3}_T\int_{\Si^\de_-}|v\cdot n(x)|\big(\chi^\de_-(Z_k-Z)\big)^2\,dS(x)dvdt
	+\frac{C_\de}{k}.
\end{multline}

We next deal with the boundary terms in \eqref{ZkZT} by showing that they are further bounded via the interior compactness inside $\Omega_\de$. For this, we need a trace theorem on the non-grazing set similar to Lemma \ref{diffboundLem}. By \eqref{chi123}, we denote a smooth cutoff function
\begin{align*}
	\begin{aligned}
		\ol\chi^\de_+(t,x,v;{T}) & =\chi^\de_+(t,x,v;T)\chi_3(x),\ \text{ for }0\le T\le t, \\
		\ol\chi^\de_-(t,x,v;{T}) & =\chi^\de_-(t,x,v;T)\chi_3(x),\ \text{ for }0\le t\le T,
	\end{aligned}
\end{align*}
Then 
\begin{align}\label{supported}
	\ol\chi_\pm^\de\text{ are supported on }\{x\,:\,3\de^6\ge\xi(x)\ge -3\de^6\}.
\end{align}
Using \eqref{supported}, and Lemma \ref{chipmLem} (1) and (2), we have that for $(x,v)\in\Omega\times\R^3_v$ and $|t-T|\le \de^3$,
\begin{align}\label{support1}\begin{aligned}
		& \text{ if\ $\chi^\de_+(t,x,v;T)\neq 0$\ then\ $v\cdot n(x)>\de^2>0$;} \\
		& \text{ if\ $\chi^\de_-(t,x,v;T)\neq 0$\ then\ $v\cdot n(x)<-\de^2<0$;} \\
		& \ol\chi^\de_+(T-\de^3,x,v)=0\ \text{ and }\ \ol\chi^\de_-(T+\de^3,x,v;T)=0.
	\end{aligned}
\end{align}
Since by \eqref{chi123} and \eqref{paOmegade}, $\chi_3(x)=1$ on $\pa\Omega_\de$, we know that
\begin{align*}
	\begin{aligned}
		\big(\chi^\de_+(Z_k-Z)\big)^2|_{\Si^\de_+}= \big(\ol\chi^\de_+(Z_k-Z)\big)^2|_{\Si^\de_+}, \\
		\big(\chi^\de_-(Z_k-Z)\big)^2|_{\Si^\de_-}= \big(\ol\chi^\de_-(Z_k-Z)\big)^2|_{\Si^\de_-}.
	\end{aligned}
\end{align*}
We then have from \eqref{1371} that
\begin{align}\label{1371a}
	\big(\pa_t+v\cdot\na_x\big)\big(\ol\chi_\pm^\de(Z_k-Z)\big) = \chi_\pm^\de(Z_k-Z)v\cdot\na_x\chi_3(x)+ \ol\chi_\pm^\de LZ_k+\ol\chi_\pm^\de\Gamma(f_k,Z_k).
\end{align}
For $\ol\chi^\de_+$, taking $L^2$ inner product of \eqref{1371a} with $\ol\chi_\pm^\de(Z_k-Z)$ over $[T-\de^3,T]\times\Omega_\de\times\R^3_v$ and using \eqref{support1} to deduce
\begin{multline}\label{1380}
	\|\ol\chi^\de_+(Z_k-Z)(T)\|_{L^2_x(\Omega_\de)L^2_v}^2+\int_{T-\de^3}^T\int_{\Si^\de_+}|v\cdot n(x)|\big(\ol\chi^\de_+(Z_k-Z)\big)^2\,dS(x)dvdt\\
	=\big(\chi^\de_+(Z_k-Z)v\cdot\na_x\chi_3(x)+\ol\chi^\de_+ LZ_k+\ol\chi^\de_+\Gamma(f_k,Z_k),\ol\chi^\de_+(Z_k-Z)\big)_{L^2_{t,x,v}([T-\de^3,T]\times\Omega_\de\times\R^3_v)}.
\end{multline}
Similarly, for $\ol\chi^\de_-$ in \eqref{1371a}, we have
\begin{multline}\label{1380a}
	\|\ol\chi^\de_-(Z_k-Z)(T)\|_{L^2_x(\Omega\setminus\Omega_\de)L^2_v}^2+\int^{T+\de^3}_T\int_{\Si^\de_-}|v\cdot n(x)|\big(\chi^\de_-(Z_k-Z)\big)^2\,dS(x)dvdt\\
	=
	-\big(\chi^\de_-(Z_k-Z)v\cdot\na_x\chi_3(x)+\ol\chi^\de_- LZ_k+\ol\chi^\de_-\Gamma(f_k,Z_k),\ol\chi^\de_-(Z_k-Z)\big)_{L^2_{t,x,v}([T,T+\de^3]\times\Omega_\de\times\R^3_v)}.
\end{multline}
For the last terms in \eqref{1380} and \eqref{1380a}, as in \eqref{chi+1} and \eqref{chi-1}, we have from \eqref{colli3x}, \eqref{Gaesweight} and \eqref{chi+D} that
\begin{align*}
	& \big(\chi^\de_+(Z_k-Z)v\cdot\na_x\chi_3(x)+\ol\chi^\de_+ LZ_k+\ol\chi^\de_+\Gamma(f_k,Z_k),\ol\chi^\de_+(Z_k-Z)\big)_{L^2_{t,x,v}([T-\de^3,T]\times\Omega_\de\times\R^3_v)} \\
	& \quad\le C_\de\int^{1-\de^3}_{\de^3}\|\chi^\de_+(Z_k-Z)\|_{L^2_x(\Omega_\de)L^2_v(|v|\le4\de^{-\frac{1}{4}})} \\
	& \qquad+C\int_{T-\de^3}^T\Big(\|\{\I-\P\}Z_k\|_{L^2_x(\Omega\setminus\Omega_\de)L^2_D}+\Big\{\sup_{0\le t\le 1}\|\<v\>^{4}f_k\|_{L^\infty_{x,v}(\Omega\times\R^3_v)}\Big\}\|Z_k\|_{L^2_x(\Omega\setminus\Omega_\de)L^2_D}\Big) \\
	& \notag\qquad\qquad\times\|(\chi^\de_+)^2(Z_k-Z)\|_{L^2_x(\Omega\setminus\Omega_\de)L^2_D}\,dt \\
	& \quad\le C_\de\int^{1-\de^3}_{\de^3}\|\chi^\de_+(Z_k-Z)\|_{L^2_x(\Omega_\de)L^2_v(|v|\le4\de^{-\frac{1}{4}})}\,dt+\frac{C_\de}{k}.
\end{align*}
Similar calculations can be derived for the $\chi^\de_-$ parts in \eqref{1380a}.
Using \eqref{140} and \eqref{140a} for the $\{\I-\P\}f$ parts, and interior compactness \eqref{D2} for the $\P f$ parts, we deduce the limit:
\begin{align*}
	& \big(\chi_\pm^\de(Z_k-Z)v\cdot\na_x\chi_3(x)+\ol\chi_\pm^\de LZ_k+\ol\chi_\pm^\de\Gamma(f_k,Z_k),\ol\chi_\pm^\de(Z_k-Z)\big)_{L^2_{t,x,v}([T-\de^3,T]\times\Omega_\de\times\R^3_v)} \\
	& \qquad\to 0\quad \text{ uniformly in $T$ as }k\to\infty.
\end{align*}
Substituting this limit into \eqref{1380} and \eqref{1380a}, and then \eqref{ZkZT}, we have, as $k\to \infty$, 
\begin{multline}\label{D5}
	\sup_{T\in[\de^3,1-\de^3]}\Big(\|(Z_k-Z)(T)\|_{L^2_{x,v}(\{x\in\Omega\setminus\Omega_\de,\,|v|\le 2\de^{-\frac{1}{4}},\,|v\cdot n(x)|\ge3\de^2\})}^2
	\\
	+\int_{T-\de^3}^T\int_{\Si_+}|v\cdot n(x)|\big(\chi^\de_+(Z_k-Z)\big)^2\,dS(x)dvdt\\
	+\int^{T+\de^3}_T\int_{\Si_-}|v\cdot n(x)|\big(\chi^\de_-(Z_k-Z)\big)^2\,dS(x)dvdt\Big)\to 0. 
\end{multline}

\subsection{Strong convergence and non-zero \texorpdfstring{$\P Z$}{PZ}}
In this subsection, we will conclude Lemma \ref{Lem123}.
Recall that we decompose $[0,1]\times\Omega\times\R^3_v$ in \eqref{D1234}. Substituting estimates \eqref{D1}, \eqref{D2}, \eqref{D3}, \eqref{D4} and \eqref{D5} (notice \eqref{D222}) into the left-hand side \eqref{ZkZ}, and letting $k\to\infty$, we have
\begin{align*}
	\lim_{k\to\infty}\sum_{j=1}^5\int^1_0\int_{\Omega}\Big|\int_{\R^3_v}(Z_k-Z)e_j\,dv\Big|^2\,dxdt
	\le C\big(\de^2+\de+e^{-\frac{\de^{-1/2}}{2}}\big). 
\end{align*}
Letting $\de\to 0$, we obtain \eqref{ZkZ}:
\begin{align*}
	\lim_{k\to\infty}\sum_{j=1}^5\int^1_0\int_{\Omega}\Big|\int_{\R^3_v}(Z_k-Z)e_j\,dv\Big|^2\,dxdt= 0.
\end{align*}
Together with \eqref{140} and \eqref{140a}, we obtain the strong convergence \eqref{1238}.

{{\smallskip}} To prove \eqref{ZC}, noticing \eqref{IPZk0} and \eqref{normalPf}, we deduce from the strong convergence \eqref{1238} that
\begin{align*}
	\int^1_0\|Z\|_{L^2_x(\Omega)L^2_D}^2=\lim_{k\to0}\int^1_0\|\P Z_k\|_{L^2_x(\Omega)L^2_D}^2=1.
\end{align*}

{{\smallskip}} We next derive the boundary condition for $Z$. By \eqref{D5} and Lemma \ref{chipmLem} (3), we obtain that 
\begin{align}\label{Zklimit}
	\int_{\de^3}^{1-\de^3}\int_{\pa\Omega}\int_{|v|\le 2\de^{-\frac{1}{4}}\text{ and }|v\cdot n(x)|\ge 3\de^2}|v\cdot n(x)|(Z_k-Z)^2\,dS(x)dvdt
	\to 0.
\end{align}
For the inflow boundary case, we have from \eqref{140} that
\begin{multline*}
	\int_{\de^3}^{1-\de^3}\int_{\pa\Omega}\int_{|v|\le 2\de^{-\frac{1}{4}}\text{ and }|v\cdot n(x)|\ge 3\de^2}|v\cdot n(x)||Z(t)|^2\,dS(x)dvdt\\
	\le \frac{1}{k}+\int_{\de^3}^{1-\de^3}\int_{\pa\Omega}\int_{|v|\le 2\de^{-\frac{1}{4}}\text{ and }|v\cdot n(x)|\ge 3\de^2}|v\cdot n(x)||Z_k-Z|^2\,dS(x)dvdt\to 0,
\end{multline*}
as $k\to\infty$. Since $\de>0$ is arbitrary, we obtain $Z(t,x,v)=0$ for $(t,x,v)\in[0,1]\times\pa\Omega\times\R^3_v$.

\smallskip
For the Maxwell reflection boundary case, since $\ve\in[0,1)$, we denote
\begin{align}\label{bara}
	\bar a_k(t,x)=c_\mu\int_{v'\cdot n(x)>0}\{v'\cdot n(x)\}Z_k(t,x,v')\mu^{\frac{1}{2}}(v')\,dv'.
\end{align}
Then we have from \eqref{140a} and $\al\in(0,1)$ that
\begin{align}\label{140b}
	\int_{\Si_+}|v\cdot n|\big(\ve|Z_k|^2+(1-\ve)|Z_k-R_DZ_k|^2\big)\,dS(x)dv\le \frac{C}{k},
\end{align}
for some $C=C(\al)>0$.
Thus, for any sufficiently large $k>0$, 
\begin{align*}
	& \int_{\de^3}^{1-\de^3}\int_{\pa\Omega}\int_{|v|\le 2\de^{-\frac{1}{4}}\text{ and }v\cdot n(x)\ge 3\de^2}|v\cdot n||R_DZ_k|^2\,dvdS(x)dt \\
	& \quad\le 2\int_{\de^3}^{1-\de^3}\int_{\pa\Omega}\int_{|v|\le 2\de^{-\frac{1}{4}}\text{ and }v\cdot n(x)\ge 3\de^2}|v\cdot n|\big(|Z_k-R_DZ_k|^2+|Z-Z_k|^2+|Z|^2\big)\,dvdS(x)dt \\
	& \quad\le \frac{C}{k}+C\le C,
\end{align*}
where we used \eqref{Zklimit} to control the $Z-Z_k$ term and \eqref{499} to control the $Z$ term.
Together with \eqref{bara}, we use rotation $v\mapsto \ti R$ satisfying $\ti R^Tn=(1,0,0)$ to deduce
\begin{align*}\notag
	C & \ge \int_{\de^3}^{1-\de^3}\int_{\pa\Omega}\int_{|v|\le 2\de^{-\frac{1}{4}}\text{ and }v\cdot n(x)\ge 3\de^2}|v\cdot n|\mu(v)|\bar a_k(t,x)|^2\,dvdS(x)dt \\
	& \notag\ge \int_{\de^3}^{1-\de^3}\int_{\pa\Omega}\int_{|v|\le 2\text{ and }v_1\ge 3}|v_1|\mu(v)|\bar a_k(t,x)|^2\,dvdS(x)dt \\
	& \ge \frac{1}{C}\int_{\de^3}^{1-\de^3}\int_{\pa\Omega}|\bar a_k(t,x)|^2dS(x)dt,
\end{align*}
for some $C>0$, where we choose $\de\in(0,1)$ small.
Therefore,
\begin{align*}
	\int_{\de^3}^{1-\de^3}\int_{\pa\Omega\times\R^3_v}|v\cdot n||R_D Z_k|^2\,dvdS(x)dt\le C
\end{align*}
is uniformly-in-$k$ bounded, and hence, by \eqref{140b},
\begin{align}\label{PSiZk}
	\int_{\de^3}^{1-\de^3}\int_{\pa\Omega}\int_{\Si_+}|v\cdot n||Z_k|^2\,dS(x)dvdt\le C,
\end{align}
is also uniformly-in-$k$ bounded.
Notice from \eqref{140b} that
\begin{multline}\label{1390}
	\int_{\de^3}^{1-\de^3}\int_{\pa\Omega}\int_{|v|\le 2\de^{-\frac{1}{4}}\text{ and }v\cdot n(x)\ge 3\de^2}|v\cdot n(x)||Z-R_D Z|^2\,dS(x)dvdt\\
	=\frac{C_{\al,\ve}}{k}+\int_{\de^3}^{1-\de^3}\int_{\pa\Omega}\int_{|v|\le 2\de^{-\frac{1}{4}}\text{ and }v\cdot n(x)\ge 3\de^2}|v\cdot n(x)||Z-Z_k|^2\,dS(x)dvdt\\
	+\int_{\de^3}^{1-\de^3}\int_{\pa\Omega}\int_{|v|\le 2\de^{-\frac{1}{4}}\text{ and }v\cdot n(x)\ge 3\de^2}|v\cdot n(x)||R_D(Z-Z_k)|^2\,dS(x)dvdt.
\end{multline}
By \eqref{Zklimit}, the second right-hand term of \eqref{1390} converges to $0$. For the last term in \eqref{1390}, we use \eqref{499} to control $Z$ term and \eqref{PSiZk} to control $Z_k$ term:
\begin{multline}\label{1390a}
	\int_{\de^3}^{1-\de^3}\int_{\pa\Omega}\int_{|v|\le 2\de^{-\frac{1}{4}}\text{ and }v\cdot n(x)\ge 3\de^2}|v\cdot n(x)||R_D(Z-Z_k)|^2\,dS(x)dvdt\\
	\le C\int_{\de^3}^{1-\de^3}\int_{\pa\Omega}\Big|\int_{|v|\le 2\de^{-\frac{1}{4}}\text{ and }v\cdot n(x)\ge 3\de^2}\{v\cdot n(x)\}\big(Z(v)-Z_k(v)\big)\mu^{\frac{1}{2}}(v)\,dv\Big|^2\,dS(x)dt\\
	+C\int_{\de^3}^{1-\de^3}\int_{\pa\Omega}\Big|\int_{|v|> 2\de^{-\frac{1}{4}}\text{ or }0<v\cdot n(x)< 3\de^2}\{v\cdot n(x)\}\big(Z(v)-Z_k(v)\big)\mu^{\frac{1}{2}}(v)\,dv\Big|^2\,dS(x)dt\\
	\le
	C\int_{\de^3}^{1-\de^3}\int_{\pa\Omega}\Big|\int_{|v|\le 2\de^{-\frac{1}{4}}\text{ and }v\cdot n(x)\ge 3\de^2}\{v\cdot n(x)\}|Z(v)-Z_k(v)|^2\,dv\\
	+
	C\int_{|v|> 2\de^{-\frac{1}{4}}\text{ or }0<v_1< 3\de^2}v_1\mu(v)\,dv,
\end{multline}
where we used Cauchy-Schwarz inequality and rotation $v\mapsto \ti Rv$ with $\ti R^Tn=(1,0,0)$ in the last inequality.
Using \eqref{smallmu} and \eqref{smallmu1} for the last term in \eqref{1390a}, and taking limit $k\to\infty$ in \eqref{1390} with the help of \eqref{Zklimit}, we deduce
\begin{align*}
	\int_{\de^3}^{1-\de^3}\int_{\pa\Omega}\int_{|v|\le 2\de^{-\frac{1}{4}}\text{ and }v\cdot n(x)\ge 3\de^2}|v\cdot n(x)||Z-R_D Z|^2\,dS(x)dvdt
	\le C\big(\de^2+e^{-\frac{\de^{-1/2}}{2}}\big).
\end{align*}
Since $\de>0$ is arbitrary small, we obtain $Z|_{\Si_+}=R_D Z$ on $[0,1]\times\Si_+$.

\smallskip For the case $\ve>0$, we have from \eqref{140b} that
\begin{align*}
	\int_{\Si_+}|v\cdot n||Z_k|^2\,dS(x)dv\le \frac{C_\ve}{k}\to 0,
\end{align*}
as $k\to\infty$. Consequently, $Z=0$ on $\Si_+$.

\smallskip For the case $\ve=0$, we have $Z_k|_{\Si_-}=RZ_k$. We take $L^2$ inner product of \eqref{1235} with $\mu^{\frac{1}{2}}$ over $\Omega\times\R^3_v$ to deduce
\begin{align}\label{converva}
	\pa_t\int_{\Omega\times\R^3_v}Z_k(t,x,v)\mu^{\frac{1}{2}}(v)\,dxdv+\int_{\pa\Omega\times\R^3_v}\{v\cdot n(x)\}Z_k(t,x,v)\mu^{\frac{1}{2}}(v)\,dS(x)dv=0.
\end{align}
For the boundary term, as in \eqref{101}, using Maxwell reflection condition $Z_k|_{\Si_-}=RZ_k$ with $R$ given by \eqref{reflect}, we take the change of variable $v\mapsto R_L(x)v\,:\,\Si_-\to \Si_+$ to deduce
\begin{multline*}
	\int_{\pa\Omega\times\R^3_v}\{v\cdot n(x)\}Z_k(t,x,v)\mu^{\frac{1}{2}}(v)\,dS(x)dv
	=\int_{\Si_+}\{v\cdot n(x)\}Z_k(t,x,v)\mu^{\frac{1}{2}}(v)\,dS(x)dv\\
	+\int_{\Si_-}\{v\cdot n(x)\}\big((1-\al)Z_k(x,R_L(x)v)+\al R_DZ_k(x,v)\big)\mu^{\frac{1}{2}}(v)\,dS(x)dv
	=0.
\end{multline*}
Thus, using \eqref{conservationlaw}, \eqref{converva} implies the conservation law of mass:
\begin{align*}
	\int_{\Omega\times\R^3_v}Z_k(t,x,v)\mu^{\frac{1}{2}}(v)\,dxdv=\int_{\Omega\times\R^3_v}f_0(x,v)\mu^{\frac{1}{2}}(v)\,dxdv=0.
\end{align*}
Taking the limit $k\to\infty$ and using strong convergence \eqref{1238}, we have
\begin{align*}
	\int_{\Omega\times\R^3_v}Z(t,x,v)\mu^{\frac{1}{2}}(v)\,dxdv=0.
\end{align*}
This completes the proof of Lemma \ref{Lem123}.

\section{Appendix: Proof of Velocity Averaging Lemma}
\label{AppAver}
Here, we provide the proof of Velocity Averaging Lemma \ref{averThma} and the precise bound. 
The proof in \cite{Diperna1991a} considered the differential operator $\De_x$ but not $\De_{t,x}$, and used the Littlewood-Paley decomposition in dyadic cubes. However, we don't impose any smooth cutoff-from-infinity function $\phi$ in the velocity averages. Instead, we only use a sufficiently smooth weight function $\phi$. 
Because of this, we need to use some new cutoff functions instead of the classic method and moreover, we utilize the crucial change of variable on $\tau+v\cdot\xi$ from \cite{Bezard1994}. 
\begin{proof}[Proof of Theorem \ref{averThma}]
	
	\smallskip {\bf Part (1).}
	We begin with the proof of \eqref{avereq} and use the same notations for the Besov space in \eqref{Besov}. 
	Then for any $f\in L^p$, we have 
	\begin{align*}
		f=\sum^\infty_{j=0}f_j,\quad \text{ where } f_j=\De_j^2f. 
	\end{align*}
	By writing velocity averaging as 
	\begin{align}\label{olfj}
		\ol{f_j}=\int_{\R^d_v}f_j(v)\psi(v)\,dv,\quad j\ge 0, 
	\end{align}
 we have an estimate: 	
	\begin{Lem}\label{AB12}
		For any $n>0$ and $p\in[1,\infty]$, let $\beta=\frac{1-\ka}{1+m}\in(0,1)$ and $\frac{1}{p'}=1-\frac{1}{p}$. Then 
\begin{align}\label{AjBjes}
	\begin{aligned}
		\|\ol{f_0}\|_{L^p_{t,x}}&\le C\|\<D_v\>\psi\|_{L^{2}_v}\|f\|_{L^p(\R^{1+2d}_{t,x,v})},\\
		\|\ol{f_j}\|_{L^{p}_{t,x}}&\le C_m 2^{-\frac{n\beta j}{1+2n}\frac{1}{\max\{p,p'\}}}\|\<v\>^n\<D_v\>^{m+1}\psi\|_{L^2_v}\|(\De_jf,\De_jG)\|_{L^p_{t,x,v}}.
	\end{aligned}
\end{align}
	\end{Lem}
	Assuming Lemma \ref{AB12} is valid, we conclude the proof of Theorem \ref{averThma} as follows. 
	
	\smallskip\noindent$\bullet$ For the case $p\in[1,2]$, by Besov norm definition \eqref{Besov}, we need to estimate 
	\begin{align}\label{olf1x}
		\|\ol f\|_{B^{\al,2}_p}
		&=\|\ol{f_0}\|_{L^p_{t,x}}
		+\Big(\sum^\infty_{j=1}(2^{j\al}\|\ol{f_j}\|_{L^p_{t,x}})^2\Big)^{1/2}. 
	\end{align}
	The first right-hand term is estimated by \eqref{AjBjes}. 
	Applying \eqref{AjBjes} to the last term in \eqref{olf1x}, we have  
	\begin{align*}
		\|\ol{f_j}\|_{L^p_{t,x}}\le C_m2^{-\frac{n\beta j}{1+2n}\frac{1}{\max\{p,p'\}}}\|\<v\>^{n}\<D_v\>^{m+1}\psi\|_{L^{2}_{v}}\big(\|\De_jf\|_{L^p_{t,x,v}}+\|\De_jG\|_{L^p_{t,x,v}}\big), 
	\end{align*}
	where $C$ is a constant independent of $j,f,G$. 
	Then we choose $\al=\frac{n\beta}{1+2n}\frac{1}{\max\{p,p'\}}$ to eliminate the frequency coefficient in \eqref{olf1x}. 
Further, by Minkowski's inequality with $1< p\le 2$, we have 
	\begin{align*}
		\Big(\sum_{j\ge 1}\|\De_ju\|^2_{L^p(\R^{1+2d}_{t,x,v})}\Big)^{1/2}
		\le\Big\|\Big(\sum_{j\ge 1}|\De_ju|^2\Big)^{1/2}\Big\|_{L^p(\R^{1+2d}_{t,x,v})}.
	\end{align*}
	Using the Littlewood-Paley Theorem, e.g. \cite[Theorem 6.1.2]{Grafakos2014a}, one has
	\begin{align*}
		\big\|\big(\sum_{j\ge 1}|\De_ju|^2\big)^{1/2}\big\|_{L^p(\R^{1+d}_{t,x})}\le C_{d,p}\|u\|_{L^p(\R^{1+d}_{t,x})}. 
	\end{align*}
Finally, substituting the above estimates into \eqref{olf1x} yields 
	\begin{align*}\notag
		\|\ol f\|_{B^{\al,2}_p}
		&\le C_{m,d,p}\|\<v\>^{n}\<D_v\>^{m+1}\psi\|_{L^{2}_{v}}\big(\|f\|_{L^p(\R^{1+2d}_{t,x,v})}+\|G\|_{L^p(\R^{1+2d}_{t,x,v})}\big). 
	\end{align*}

	\smallskip \noindent$\bullet$
	For the case $p\in [2,\infty)$, we consider $B^{\al,p}_p$ norm instead: 
	\begin{align}\label{127x}
		\|\ol f\|_{B^{\al,p}_p}
		&=\|\ol{f_0}\|_{L^p_{t,x}}
		+\Big(\sum^\infty_{j=1}(2^{j\al}\|\ol{f_j}\|_{L^p_{t,x}})^{p}\Big)^{1/p}. 
	\end{align}
Similarly, by Lemma \ref{AB12}, we have 
	\begin{align*}
		\|\ol{f_0}\|_{L^p_{t,x}}
		&\le C\|\<D_v\>\psi\|_{L^2_v}\|f\|_{L^p_{t,x,v}},\\
		\|\ol{f_j}\|_{L^p_{t,x}}&\le C_m2^{-\frac{n\beta j}{1+2n}\frac{1}{\max\{p,p'\}}}\|\<v\>^{n}\<D_v\>^{m+1}\psi\|_{L^{2}_{v}}\big(\|\De_jf\|_{L^p_{t,x,v}}+\|\De_jG\|_{L^p_{t,x,v}}\big).  
	\end{align*}
	Then we choose $\al=\frac{n\beta}{1+2n}\frac{1}{\max\{p,p'\}}$. 
	Also, by Minkowski's inequality, Littlewood-Paley Theorem,
	\begin{align*}
		\Big(\sum_{j\ge 1}\|\De_jf\|^{p}_{L^p_{t,x,v}}\Big)^{1/p}
		\le\Big\|\Big(\sum_{j\ge 1}|\De_jf(v)|^{p}\Big)^{1/p}\Big\|_{L^p_{t,x,v}}
		\le C_{d,p}\|f\|_{L^p_{t,x,v}}.
	\end{align*}
Substituting into \eqref{127x} yields 
	\begin{align*}\notag
		\|\ol f\|_{B^{\al,p}_p}
		&\le C_{m,d,p,q}\|\<v\>^{n}\<D_v\>^{m+1}\psi\|_{L^{2}_{v}}\big(\|f\|_{L^p_{t,x,v}}+\|G\|_{L^p_{t,x,v}}\big). 
	\end{align*}
 This completes the proof of (1) in Theorem \ref{averThma}.

	\medskip\noindent {\bf Part (2).} For the proof of \eqref{avereq1} in bounded time interval $[T_1,T_2]$ with $p\in(1,2]$, we follow \cite[Proposition 2.14]{Alonso2022} and sketch the proof for brevity. Multiplying \eqref{transG1} by $\1_{[T_1,T_2]}(t)$ we arrive at 
	\begin{align}\label{1230}
		\pa_t\wt f+v\cdot\na_x\wt f = f(T_1)\de(t-T_1)-f(T_2)\de(t-T_2)+\wt G,
	\end{align}
	which is valid for all $t\in(-\infty,\infty)$,
	where 
	\begin{align*}
		\wt f=\1_{[T_1,T_2]}f,\quad 
		\wt G=\1_{[T_1,T_2]}G. 
	\end{align*}
	Then the velocity averaging estimate \eqref{avereq} can be applied to \eqref{1230}. To overcome the delta function about time, we choose $0\le\ka<\wt\ka\le 1$ such that $\wt\ka>1-\frac{1}{p}+\ka$ which implies that we should require $\ka<\frac{1}{p}$. Thus, as in \cite[Proposition 2.14]{Alonso2022}, using the asymptotic behavior of Bessel kernel (see \eqref{GsL1} for the fast decay for large $t$ and slow growth near the singular point), we have 
	\begin{align*}
		&\|(I-\De_{t,x})^{-\wt\ka/2}(I-\De_v)^{-\frac{m}{2}}(f(T_1)\de(t-T_1))\|_{L^p(\R^{1+2d}_{t,x,v})}\\
		&\quad\le\|G_{\wt\ka-\ka}(t-T_1)(I-\De_{x})^{-\frac{\ka}{2}}(I-\De_v)^{-\frac{m}{2}}f(T_1)\|_{L^p(\R^{1+2d}_{t,x,v})}\\
		&\quad\le C_{d,\wt\ka,\ka,p}\|(I-\De_{x})^{-\frac{\ka}{2}}(I-\De_v)^{-\frac{m}{2}}f(T_1)\|_{L^p(\R^{1+d}_{x,v})}. 
	\end{align*}
	The term $f(T_2)\de(t-T_2)$ can be deduced similarly and one can obtain \eqref{avereq1}.  	
This completes the proof of Theorem \ref{averThma}, provided Lemma \ref{AB12} is valid. 
\end{proof}

{{\smallskip}} Now, it remains to prove Lemma \ref{AB12}. 

\begin{proof}[Proof of Lemma \ref{AB12}]
The estimate of $\ol{f_0}$ in \eqref{olfj} can be calculated directly by using Young's convolution inequality: for any $p\in[1,2]$, 
\begin{align*}
	\|\ol{f_0}\|_{L^p(\R^{1+d}_{t,x})}
	\le \|\psi\|_{L^{p'}_v}\|\De_0^2f\|_{L^p(\R^{1+2d}_{t,x,v})}
	\le C\|\psi\|_{L^{p'}_v}\|f\|_{L^p(\R^{1+2d}_{t,x,v})}. 
\end{align*}
To find the estimate of $f_j$ ($j\ge 1$), we split the region in frequency space. Denote the cut-off $\vp_0\in C^\infty_c(\R)$ such that $\vp_0=1$ in $[-1,1]$ and $\vp_0=0$ in $(-\infty,-2)\cup(2,\infty)$, and we set $\vp_1=1-\vp_0$. 
Then we decompose $\wh f$ as 
\begin{align}\label{whf1}
	\wh{f_j}(\tau,\xi,v)=
	\vp_0\Big(2^{\beta j+4}\frac{\tau+\xi\cdot v}{|(\tau,\xi)|}\Big)\wh{f_j}(\tau,\xi,v)
	+\vp_1\Big(2^{\beta j+4}\frac{\tau+\xi\cdot v}{|(\tau,\xi)|}\Big)\wh{f_j}(\tau,\xi,v), 
\end{align}
for some constant $\beta>0$ to be chosen. 
We have from \eqref{transG} that $\wh{f_j}=\frac{(1+|\tau|^2+|\xi|^2)^{\ka/2}}{i(\tau+\xi\cdot v)}(I-\De_v)^{m/2}\wh{G_j}$. 
Thus, we can rewrite \eqref{whf1} as 
\begin{align*}
	\wh{f_j}=
	\vp_0\Big(2^{\beta j+4}\frac{\tau+\xi\cdot v}{|(\tau,\xi)|}\Big)\wh{f_j}
	+\vp_1\Big(2^{\beta j+4}\frac{\tau+\xi\cdot v}{|(\tau,\xi)|}\Big)\frac{(1+|\tau|^2+|\xi|^2)^{\ka/2}(I-\De_v)^{m/2}\wh{G_j}}{i(\tau+\xi\cdot v)}. 
\end{align*}
Integrating with $\psi(v)dv$ and taking inverse Fourier transform $\F^{-1}_{t,x}$ about $(t,x)$, we have 
\begin{align}\label{olfj111}\notag
	\ol{f_j}(t,x)&=\F^{-1}_{t,x}\Big\{\int_{\R^d}\vp_0\Big(2^{\beta j+4}\frac{\tau+\xi\cdot v}{|(\tau,\xi)|}\Big)\psi(v)\wh\Psi\Big(\frac{\tau}{2^j},\frac{\xi}{2^j}\Big)\wh{\De_jf}\,dv\\
	&\qquad-i\int_{\R^d}(I-\De_v)^{m/2}\Big[\frac{\vp_1(2^{\beta j+4}\frac{\tau+\xi\cdot v}{|(\tau,\xi)|})\psi(v)}{\tau+\xi\cdot v}\Big](1+|\tau|^2+|\xi|^2)^{\ka/2}\wh\Psi\Big(\frac{\tau}{2^j},\frac{\xi}{2^j}\Big)\wh{\De_jG}\,dv\Big\}.  
\end{align}
Correspondingly, by writing $\vp_2(r)=\frac{\vp_1(r)}{r}$, we denote 
\begin{align*}
	\begin{aligned}
		A_jf&=\F^{-1}_{t,x}\Big\{\int_{\R^d}\vp_0\Big(2^{\beta j+4}\frac{\tau+\xi\cdot v}{|(\tau,\xi)|}\Big)\wh\Psi\Big(\frac{\tau}{2^j},\frac{\xi}{2^j}\Big)\psi(v)\wh f\,dv\Big\},\\
		B_jG
		&=\F^{-1}_{t,x}\Big\{\int_{\R^d}(I-\De_v)^{m/2}\Big[\frac{2^{\beta j+4}\vp_2\big(2^{\beta j+4}\frac{\tau+\xi\cdot v}{|(\tau,\xi)|}\big)\psi(v)}{|(\tau,\xi)|}\Big](1+|\tau|^2+|\xi|^2)^{\ka/2}\wh\Psi\Big(\frac{\tau}{2^j},\frac{\xi}{2^j}\Big)\wh G\,dv\Big\}. 
	\end{aligned}
\end{align*}
Here, we only consider even number $m\ge 0$ while the other cases can be obtained by interpolation later. Then the velocity derivative term is 
\begin{align*}
	(I-\De_v)^{m/2}\Big[\vp_2\big(2^{\beta j+4}\frac{\tau+\xi\cdot v}{|(\tau,\xi)|}\big)\psi(v)\Big]
	&=\sum_{|\al_1|+|\al_2|\le m}\frac{C_{\al_1,\al_2}2^{|\al_1|\beta j}\xi^{\al_1}}{|(\tau,\xi)|^{|\al_1|}}
	\\&\qquad\qquad\qquad\times\vp^{(|\al_1|)}_2\big(2^{\beta j+4}\frac{\tau+\xi\cdot v}{|(\tau,\xi)|}\big)\pa^{\al_2}_v\psi(v), 
\end{align*} 
where $\vp^{(n)}_2(r):=\frac{d^{n}\vp_2}{dr^n}$ and $\al_1,\al_2\in\N^d$ are multi-indices.

\medskip\noindent
{\bf The $L^2$ estimate.}
Fix any $\al\in(0,1)$. 
To find the $L^2$ estimates of $\ol{f_j}$ with regularity, we need to design the appropriate smooth cutoff in frequency $|\xi|$ and $|\tau|$ as follows. Denote the cut-off $\vp_0,\vp_1$ as above. Then we consider decomposition
\begin{align*}
	\begin{aligned}
		\vp_a&=\vp_1(2^{-\al j+2}|\xi|),\\
		\vp_b&=\vp_0(2^{-\al j+2}|\xi|)\vp_0(2^{-j+2}|\tau|),\\
		\vp_c&=\vp_0(2^{-\al j+2}|\xi|)\vp_1(2^{-j+2}|\tau|)\vp_0(2^{-\al j}|v|^{-1}|\tau|),\\
		\vp_d&=\vp_0(2^{-\al j+2}|\xi|)\vp_1(2^{-j+2}|\tau|)\vp_1(2^{-\al j}|v|^{-1}|\tau|), 
	\end{aligned}
\end{align*}
which satisfies $\vp_a+\vp_b+\vp_c+\vp_d=1$. 
On the corresponding support of these functions, we have 
\begin{align}\label{support118}
\begin{cases}
	|\xi|>2^{\al j-2}, &\text{ on the support of }\vp_a,\\
	|\xi|\le 2^{\al j-1},\ |\tau|\le 2^{j-1}, &\text{ on the support of }\vp_b,\\
	|\xi|\le 2^{\al j-1},\ |\tau|>2^{j-2},\ |\tau|\le 2^{\al j+1}|v|, &\text{ on the support of }\vp_c,\\
	|\xi|\le 2^{\al j-1},\ |\tau|>2^{j-2},\ |\tau|> 2^{\al j}|v|, &\text{ on the support of }\vp_d.
\end{cases}
\end{align}
Corresponding to these cutoff functions, we the $A_jf,B_jG$ as 
\begin{align*}
	A_jf=A_{j,a}f+A_{j,b}f+A_{j,c}f+A_{j,d}f,\quad B_jG=B_{j,a}G+B_{j,b}G+B_{j,c}G+B_{j,d}G. 
\end{align*}
First, for the terms $A_{j,b}f,B_{j,b}G$, since $|(\tau,\xi)|<2^{j}$ on the support of $\vp_b$ while $|(\tau,\xi)|\ge 2^j$ on the support of $\wh\Psi(\frac{\tau}{2^j},\frac{\xi}{2^j})$, these terms vanish, i.e. $\vp_b\wh\Psi(\frac{\tau}{2^j},\frac{\xi}{2^j})=0$ and hence, 
\begin{align}\label{esfjb}
	A_{j,b}f=B_{j,b}G=0. 
\end{align}
Moreover, on the support of $\vp_d$, since $|\xi|\le 2^{\al j-2}$, $|\tau|>2^{j-2}$, $|\tau|> 2^{\al j}|v|$, we have 
\begin{align}
	\label{suppvpd}
|\tau+\xi\cdot v|\ge |\tau|-|\xi||v|\ge \frac{|\tau|}{2}>2^{j-3}.
\end{align}
For the term $A_{j,a}f$, by Plancherel's identity, we have  
\begin{align*}
	\|A_{j,a}f\|_{L^2_{t,x}}&=\Big\|\int_{\R^d}\vp_a\vp_0\Big(2^{\beta j+4}\frac{\tau+\xi\cdot v}{|(\tau,\xi)|}\Big)\wh\Psi\Big(\frac{\tau}{2^j},\frac{\xi}{2^j}\Big)\psi(v)\wh f\,dv\Big\|_{L^2_{\tau,\xi}}\\
	&\le \Big\|\vp_a\big\|\vp_0\Big(2^{\beta j+4}\frac{\tau+\xi\cdot v}{|(\tau,\xi)|}\Big)\psi(v)\big\|_{L^2_v}\wh\Psi\big(\frac{\tau}{2^j},\frac{\xi}{2^j}\big)\|\wh f\|_{L^2_v}\Big\|_{L^2_{\tau,\xi}}. 
\end{align*}
To evaluate this, we can use a rotating change of variable: $v\mapsto u=A^{-1}v$ with an orthogonal matrix $A$ satisfying $A^{-1}=A^T$ and $A^T\frac{\xi}{|\xi|}=e_1\equiv (1,0,\dots,0)$. For example, we choose $A^T=\bigl(\frac{\xi}{|\xi|},A_2,\dots,A_d\bigr)^T$ with some unit vectors $A_i\in\R^d$ orthogonal to $\frac{\xi}{|\xi|}$. Then $\pa_{u_1}(Au)=\frac{\xi}{|\xi|}$. By such a rotation and using embedding $\|\cdot\|_{L^\infty_{v_1}(\R)}\le C\|\cdot\|_{H^1_{v_1}(\R)}$, we have 
\begin{align*}\notag
	\Big\|\vp_0\Big(2^{\beta j+4}\frac{\tau+\xi\cdot v}{|(\tau,\xi)|}\Big)\psi(v)\Big\|_{L^2_v}
	&\le \Big\|\vp_0\Big(2^{\beta j+4}\frac{\tau+v_1|\xi|}{|(\tau,\xi)|}\Big)\psi(A(v-\tau|\xi|^{-1}e_1))\Big\|_{L^2_v}\\
	&\notag\le \Big\|\vp_0\Big(2^{\beta j+4}\frac{\tau+v_1|\xi|}{|(\tau,\xi)|}\Big)\Big\|_{L^2_{v_1}}\|\psi(A(v-\tau|\xi|^{-1}e_1))\|_{L^\infty_{v_1}L^2_{v_2,\dots,v_d}}\\
	&\le C2^{-\frac{\beta j}{2}}|(\tau,\xi)|^{\frac{1}{2}}|\xi|^{-\frac{1}{2}}\|\<D_v\>\psi\|_{L^2_v}. 
\end{align*}
Thus, by the support in \eqref{support118}, 
\begin{align}\label{esAja}
	\|A_{j,a}f\|_{L^2_{t,x}}&\notag\le C\|\psi\|_{H^1_v}\|2^{-\frac{\beta j}{2}}|(\tau,\xi)|^{\frac{1}{2}}|\xi|^{-\frac{1}{2}}\1_{|\xi|>2^{\al j-2}}\1_{2^j\le|[\tau,\xi]|\le3\cdot 2^j}\wh{f}\|_{L^2_{\tau,\xi,v}}\\
	&\le C2^{-\frac{\beta j}{2}+\frac{(1-\al)j}{2}}\|\psi\|_{H^1_v}\|f\|_{L^2_{t,x,v}}. 
\end{align}
The term $B_{j,a}G$ is similar. Using the support of $\wh\Psi(\frac{\tau}{2^j},\frac{\xi}{2^j})$, one has $2^j\le|(\tau,\xi)|\le C2^j$ and hence, 
\begin{align}\label{esBja}\notag
	\|B_{j,a}G\|_{L^2_{t,x}}
	&\le \sum_{|\al_1|+|\al_2|\le m}\frac{C2^{(m+1)\beta j+\ka j}}{2^j}\Big\|\vp_a\big\|\vp^{(|\al_1|)}_2\big(2^{\beta j+4}\frac{\tau+\xi\cdot v}{|(\tau,\xi)|}\big)\pa^{\al_2}_v\psi(v)\big\|_{L^2_{v}}\|\wh G\|_{L^2_v}\Big\|_{L^2_{\tau,\xi}}\\
	&\notag\le \frac{C2^{(m+1)\beta j+\ka j}}{2^j}\|\<D_v\>^{m+1}\psi\|_{L^2_v}\||(\tau,\xi)|^{\frac{1}{2}}|\xi|^{-\frac{1}{2}}\1_{|\xi|>2^{\al j-2}}\1_{2^j\le|[\tau,\xi]|\le3\cdot 2^j}\wh{G}\|_{L^2_{\tau,\xi,v}}\\
	&\le C2^{(m+1)\beta j+\ka j-j-\frac{\beta j}{2}+\frac{(1-\al)j}{2}}\|\<D_v\>^{m+1}\psi\|_{L^2_v}\|G\|_{L^2_{t,x,v}}. 
\end{align}
For the term $A_{j,c}f,B_{j,c}G$, since $|v|\ge 2^{-\al j-1}|\tau|\ge 2^{(1-\al)j-3}$ on the support of $\vp_c$, for any $n\ge 0$, 
\begin{align}\label{esAjc}\notag
	\|A_{j,c}f\|_{L^2_{t,x}}
	&\le\Big\|\int_{\R^d_v}\vp_c\wh\Psi\Big(\frac{\tau}{2^j},\frac{\xi}{2^j}\Big)|\wh{f}(v)\psi(v)|\,dv\Big\|_{L^2_{\tau,\xi}}\\
	&\notag\le \Big\|\|\vp_c\<v\>^{-n}\wh{f}(v)\|_{L^2_v}\|\<v\>^n\psi\|_{L^2_v}\Big\|_{L^2_{\tau,\xi}}\\
	&\le C_n2^{-n(1-\al)j}\|\<v\>^n\psi\|_{L^2_v}\|f\|_{L^2_{t,x,v}},  
\end{align}
and 
\begin{align}\label{esBjc}
	\|B_{j,c}f\|_{L^2_{t,x}}
	&\notag\le \frac{C2^{\beta j+m\beta j+\ka j}}{2^j}\Big\|\|\vp_c\<v\>^{-n}\wh{G}(v)\|_{L^2_v}\|\<v\>^n\<D_v\>^m\psi\|_{L^2_v}\Big\|_{L^2_{\tau,\xi}}\\
	&\le C_n2^{(1+m)\beta j+\ka j-j-n(1-\al)j}\|\<v\>^n\psi\|_{L^2_v}\|G\|_{L^2_{t,x,v}}. 
\end{align}
For the terms $A_{j,d}f,B_{j,d}G$, by \eqref{suppvpd} and the support of $\wh\Psi(\frac{\tau}{2^j},\frac{\xi}{2^j})$, i.e. $2^j\le|(\tau,\xi)|\le 3\times2^j$, one has 
\begin{align*}
	2^{\beta j+4}\frac{|\tau+\xi\cdot v|}{|(\tau,\xi)|}
	&\ge 2^{\beta j+4}\frac{2^{j-3}}{2^j}\ge 2,
\end{align*}
which, together with the support of $\vp_0$, implies that on the support of $\vp_d$,
\begin{align*}
	\vp_0(2^{\beta j+4}\frac{\tau+\xi\cdot v}{|(\tau,\xi)|})=0,\qquad
	\vp_1(2^{\beta j+4}\frac{\tau+\xi\cdot v}{|(\tau,\xi)|})=1,
\end{align*}
and hence, $A_{j,d}f=0$ and 
\begin{align}\label{esBjd}\notag
	\|B_{j,d}G\|_{L^2_{t,x}}
	&\le C_m\Big\|\int_{\R^d}\sum_{|\al_1|+|\al_2|\le m}\frac{2^{|\al_1|\beta j}\xi^{\al_1}}{|(\tau,\xi)|^{|\al_1|}}\frac{\big|2^{\beta j+4}\frac{\tau+\xi\cdot v}{|(\tau,\xi)|}\big|^{-|\al_1|}}{|\tau+\xi\cdot v|}|\pa^{\al_2}_v\psi(v)|
	2^{\ka j}\wh\Psi\Big(\frac{\tau}{2^j},\frac{\xi}{2^j}\Big)\wh G\,dv\Big\|_{L^2_{\tau,\xi}}\\
	&\le C_m2^{\ka j-j}\|\<D_v\>^{m}\psi\|_{L^2_v}\|G\|_{L^2_{t,x,v}},
\end{align}
where $\vp_2(r)=\vp_1(r)/r$. 
Combining the above estimates \eqref{esfjb}, \eqref{esAja}, \eqref{esBja}, \eqref{esAjc}, \eqref{esBjc} and \eqref{esBjd}, we obtain 
\begin{align}\label{AjBjesL2}
	\begin{aligned}
		\|A_jf\|_{L^2_{t,x}}&\le C\big(2^{-\frac{\beta j}{2}+\frac{(1-\al)j}{2}}+2^{-n(1-\al)j}\big)\|\<v\>^n\<D_v\>\psi\|_{L^2_v}\|f\|_{L^2_{t,x,v}},\\
		\|B_jG\|_{L^2_{t,x}}&\le C\Big(2^{(1+m)\beta j+\ka j-j}\big(2^{-\frac{\beta j}{2}+\frac{(1-\al)j}{2}}+2^{-n(1-\al)j}\big)+2^{\ka j-j}\Big)\|\<v\>^n\<D_v\>^{m+1}\psi\|_{L^2_v}\|G\|_{L^2_{t,x,v}},	
	\end{aligned}
\end{align}
for any $n\ge 0$ and even number $m\ge 0$. Then, by interpolation, estimate \eqref{AjBjesL2} holds for any real number $m\ge 0$. Now, we choose
\begin{align*}
	n>0,\quad
	\beta=\frac{1-\ka}{1+m}\in(0,1),\quad
	\al=1-\frac{\beta}{1+2n}\in(0,1).
\end{align*}
Thus, it follows from \eqref{AjBjesL2} that 
\begin{align}\label{AjBjesL2z}
	\begin{aligned}
		\|A_jf\|_{L^2_{t,x}}&\le C2^{-\frac{n\beta j}{1+2n}}\|\<v\>^n\<D_v\>\psi\|_{L^2_v}\|f\|_{L^2_{t,x,v}},\\
		\|B_jG\|_{L^2_{t,x}}&\le C2^{-\frac{n\beta j}{1+2n}}\|\<v\>^n\<D_v\>^{m+1}\psi\|_{L^2_v}\|G\|_{L^2_{t,x,v}}.	
	\end{aligned}
\end{align}

\medskip\noindent
{\bf The $L^p$ estimate.}
To derive the $L^p$ estimate, we consider the $L^1$ and $L^\infty$ Theorem by showing the multiplier is the Fourier transform of a $L^1$ function, i.e. if a bounded $C^{d+1}$ function $m$ on $\R^{d+1}$ satisfies that it inverse Fourier transform $\F^{-1}m\in L^1$, 
then $m$ is an $L^1$ and $L^\infty$ multiplier since, by Young's convolution inequality, 
\begin{align}\label{1120}
	\big\|\big(m\wh f\big)^{\vee}\big\|_{L^p(\R^{d+1})}
	=\|\F^{-1}m*f\|_{L^p(\R^{d+1})}\le \|\F^{-1}m\|_{L^1(\R^{d+1})}\|f\|_{L^p(\R^{d+1})},  
\end{align} 
for any $p\in[1,\infty]$. 
To calculate the $L^p$ estimate of $A_jf,B_jG$, by taking integration by parts about $(I-\De_v)^{m/2}$, we consider multipliers 
\begin{align*}
	\begin{aligned}
		&m_{A,j}:=\vp_0\Big(2^{\beta j+4}\frac{\tau+\xi\cdot v}{|(\tau,\xi)|}\Big)\wh\Psi\Big(\frac{\tau}{2^j},\frac{\xi}{2^j}\Big)\psi(v),\\
		&m_{B,j}:=\frac{2^{(1+|\al_1|)\beta j}\xi^{\al_1}}{|(\tau,\xi)|^{1+|\al_1|}}\vp^{(|\al_1|)}_2\Big(2^{\beta j+4}\frac{\tau+\xi\cdot v}{|(\tau,\xi)|}\Big)\pa^{\al_2}_v\psi(v)(1+|\tau|^2+|\xi|^2)^{\ka/2}\wh\Psi\Big(\frac{\tau}{2^j},\frac{\xi}{2^j}\Big),
	\end{aligned}
\end{align*}
for any multi-indices $|\al_1|+|\al_2|\le m$.
For time-dependent averaging lemma, we also need to utilize the crucial change of variable from \cite{Bezard1994}. In the frequency variables, we consider matrix $T$ given by  
\begin{align*}
	T(\tau,\xi_1,\xi'):=\big(\frac{\ti\tau+\ti\xi_1}{\sqrt{2}},\frac{-\ti\tau+\ti\xi_1}{\sqrt{2}},\ti\xi'\big),
\end{align*}
where $\xi'=(\xi'_2,\dots,\xi'_d)$. The corresponding Jacobian determinant is $\big|\frac{\pa(\tau,\xi_1)}{\pa(\ti\tau,\ti\xi_1)}\big|=1$. Applying this change of variable $T$, we have 
\begin{align*}
	\tau+\xi\cdot v=(\ti\tau,\ti\xi)\cdot\Big(\frac{1-v_1}{\sqrt{2}},\frac{v_1+1}{\sqrt{2}},v'\Big), 
	\quad|\tau|^2+|\xi|^2=|\ti\tau|^2+|\ti\xi|^2.
\end{align*}
Moreover, noting $\big|\big(\frac{1-v_1}{\sqrt{2}},\frac{v_1+1}{\sqrt{2}},v'\big)\big|=\<v\>$, we consider a rotating change of variable: $\xi\mapsto \ti\xi=R^{-1}\xi$ with an orthogonal matrix $R$ satisfying $R^{-1}=R^T$ and $R^T\big(\frac{1-v_1}{\sqrt{2}},\frac{v_1+1}{\sqrt{2}},v'\big)=\<v\>e_0\equiv (\<v\>,0,\dots,0)$. 
Moreover, for orthogonal matrices $T$, $R$ and dilation $(\tau,\xi)\to 2^{-j}(\tau,\xi)$, these change of variables won't change the multiplier norm; see for instance \cite[Proposition 2.5.14]{Grafakos2014}. Then it suffices to consider multipliers 
\begin{align*}
	\begin{aligned}
		&\ti m_{A,j}=\vp_0\Big(2^{\beta j+4}\frac{\tau\<v\>}{|(\tau,\xi)|}\Big)\wh\Psi(\tau,\xi)\psi(v),\\
		&\ti m_{B,j}=\frac{2^{(1+|\al_1|)\beta j}\xi^{\al_1}}{|(\tau,\xi)|^{1+|\al_1|}}\vp^{(|\al_1|)}_2\Big(2^{\beta j+4}\frac{\tau\<v\>}{|(\tau,\xi)|}\Big)\pa^{\al_2}_v\psi(v)(1+|\tau|^2+|\xi|^2)^{\ka/2}\wh\Psi(\tau,\xi). 
	\end{aligned}
\end{align*}
Then, by direct calculations and noting $\wh\Psi(\tau,\xi)$ is supported in $\{2^j\le|(\tau,\xi)|\le 3\times 2^j\}$ and $\beta=\frac{1-\ka}{1+m}$, we have 
\begin{align*}
	\|\ti m_{A,j}\|_{L^\infty_{\tau,\xi}}\le \wh\Psi(\tau,\xi)|\psi(v)|,\qquad
	\|\ti m_{B,j}\|_{L^\infty_{\tau,\xi}}
	\le \wh\Psi(\tau,\xi)|\pa^{\al_2}_v\psi(v)|,
\end{align*}
and, for any multi-index $\si\in\N^{d+1}$, utilizing the support of $\vp^{(|\si'|)}_0,\vp^{(|\si'|)}_1$, we have (consider $\tau$ and $\xi$ derivatives separately)
\begin{align*}
	\big|\pa^\si_{\tau,\xi}\ti m_{A,j}\big|&\le C_m\frac{|\psi(v)|}{|(\tau,\xi)^{\si}|}\sum_{\si'\le\si}|\pa^{\si'}\wh\Psi(\tau,\xi)|,\\
	\big|\pa^\si_{\tau,\xi}\ti m_{B,j}\big|&
	\le C_m\frac{|\pa^{\al_2}_v\psi(v)|}{|(\tau,\xi)^{\si}|}\sum_{\si'\le\si}|\pa^{\si'}\wh\Psi(\tau,\xi)|. 
\end{align*}
Therefore, $\ti m_{A,j},\ti m_{B,j}$ are the Fourier transform of some $L^1_{t,x}$ functions, i.e. by integration by parts, for any $N>\frac{d+1}{2}$, 
\begin{align*}
	\|\F_{t,x}^{-1}\ti m_{A,j}\|_{L^1_{t,x}}
	&=\Big\|\int_{\R^{d+1}_{t,x}}\ti m_{A,j}(\tau,\xi)\frac{(I-\De_{\tau,\xi})^{N}e^{2\pi i(\tau,\xi)\cdot(t,x)}}{\<(t,x)\>^{2N}}\,d\tau d\xi\Big\|_{L^1_{t,x}}\\
	&\le C_N\Big\|\int_{\R^{d+1}_{t,x}}\sum_{|\si|\le 2N}\frac{|\psi(v)|}{|(\tau,\xi)^{\si}|}\sum_{|\si|\le 2N}|\pa^{\si}\wh\Psi(\tau,\xi)|\<(t,x)\>^{-2N}\,d\tau d\xi\Big\|_{L^1_{t,x}}\\
	&\le C_N|\psi(v)|, 
\end{align*}
and similarly, 
\begin{align*}
	\|\F_{t,x}^{-1}\ti m_{B,j}\|_{L^1_{t,x}}
	&\le C_{m,N}|\pa^{\al_2}_v\psi(v)|. 
\end{align*}
It follows from \eqref{1120} that $\ti m_{A,j},\ti m_{B,j}$ are $L^1$ and $L^\infty$ multipliers and hence, 
\begin{align*}
	\|A_jf\|_{L^p_{t,x}}&\le C\int_{\R^d_v}|\psi(v)|\|f\|_{L^p_{t,x}}\,dv\le C_p\|\<D_v\>\psi\|_{L^{2}_v}\|f\|_{L^p_{t,x,v}},\\
	\|B_jG\|_{L^p_{t,x}}&\le C_m\int_{\R^d_v}\sum_{|\al_2|\le m}|\pa^{\al_2}_v\psi(v)|\|G\|_{L^p_{t,x}}\,dv\le C_m\|\<D_v\>^{m+1}\psi\|_{L^{2}_v}\|G\|_{L^p_{t,x,v}}, 
\end{align*}
for any $p\in[1,\infty]$. 
Therefore, together with $L^2$ estimate \eqref{AjBjesL2z} and using real interpolation, one has 
\begin{align*}
	\begin{aligned}
		\|A_jf\|_{L^2_{t,x}}&\le C2^{-\frac{n\beta j}{1+2n}\frac{1}{\max\{p,p'\}}}\|\<v\>^n\<D_v\>\psi\|_{L^2_v}\|f\|_{L^p_{t,x,v}},\\
		\|B_jG\|_{L^2_{t,x}}&\le C2^{-\frac{n\beta j}{1+2n}\frac{1}{\max\{p,p'\}}}\|\<v\>^n\<D_v\>^{m+1}\psi\|_{L^2_v}\|G\|_{L^p_{t,x,v}}.	
	\end{aligned}
\end{align*}
Substituting these estimates into \eqref{olfj111}, we complete the proof of Lemma \ref{AB12}. 
\end{proof}

{{\smallskip}}
\noindent {\bf Acknowledgements.}
The author would like to thank Prof. Yan Guo for his valuable comments and suggestions on the manuscript. The author would also like to thank Prof. Renjun Duan and Prof. Weiran Sun for the valuable discussion. 
The author also appreciates the support of Prof. Donghyun Lee and Prof. Jin Woo Jang. 
This work is supported by the National Research Foundation of Korea (NRF) grant funded by the Korea government (MSIT) No. RS-2023-00210484 and No. RS-2023-00212304.
This work is also supported by the Basic Science Research Institute Fund, whose NRF grant number is 2021R1A6A1A10042944.

\bibliography{../1}
\bibliographystyle{amsplain-url}

\end{document}